\documentclass[12pt]{article}

\usepackage{amsmath, amsthm, amssymb, accents}
\usepackage{enumitem}
\usepackage{pdflscape}
\usepackage{caption}

\usepackage{ifpdf}
\ifpdf
\usepackage[pdftex]{graphicx}
\else
\usepackage[dvips]{graphicx}
\fi
\usepackage{tikz}
 	 \usetikzlibrary{arrows,backgrounds}
\usepackage[all]{xy}

\usepackage{multicol}
\usepackage{tocvsec2}

\input xy
\xyoption{all}

\usepackage[pdftex,plainpages=false,hypertexnames=false,pdfpagelabels]{hyperref}
\newcommand{\arxiv}[1]{\href{http://arxiv.org/abs/#1}{\tt arXiv:\nolinkurl{#1}}}
\newcommand{\arXiv}[1]{\href{http://arxiv.org/abs/#1}{\tt arXiv:\nolinkurl{#1}}}
\newcommand{\doi}[1]{\href{http://dx.doi.org/#1}{{\tt DOI:#1}}}

\newcommand{\googlebooks}[1]{(preview at \href{http://books.google.com/books?id=#1}{google books})}

\usepackage{xcolor}
\definecolor{dark-red}{rgb}{0.7,0.25,0.25}
\definecolor{dark-blue}{rgb}{0.15,0.15,0.55}
\definecolor{medium-blue}{rgb}{0,0,.8}
\definecolor{DarkGreen}{RGB}{0,150,0}
\definecolor{rho}{named}{red}
\hypersetup{
   colorlinks, linkcolor={purple},
   citecolor={medium-blue}, urlcolor={medium-blue}
}
\definecolor{wString}{named}{orange}
\definecolor{xString}{named}{red}
\definecolor{yString}{named}{blue}
\definecolor{zString}{named}{DarkGreen}

\usepackage{longtable}
\usepackage{fullpage}

\theoremstyle{plain}
\newtheorem{thm}{Theorem}[section]
\newtheorem*{thm*}{Theorem}
\newtheorem{thmalpha}{Theorem}

\newtheorem{mainthm}[thm]{Main Theorem}

\newtheorem*{cor*}{Corollary}

\newtheorem*{conj*}{Conjecture}
\newtheorem{lem}[thm]{Lemma}
\newtheorem{prop}[thm]{Proposition}

\newtheorem*{quest*}{Question}
\newtheorem*{claim*}{Claim}

\theoremstyle{definition}

\newtheorem{defn}[thm]{Definition}
\newtheorem{alg}[thm]{Algorithm}

\newtheorem{ex}[thm]{Example}
\newtheorem{sub-ex}[thm]{Sub-Example}
\newtheorem{rem}[thm]{Remark}
\newtheorem{remark}[thm]{Remark}
\newtheorem{rems}[thm]{Remarks}

\DeclareMathOperator{\coev}{coev}

\DeclareMathOperator{\ev}{ev}
\DeclareMathOperator{\Hom}{Hom}

\DeclareMathOperator{\id}{id}

\DeclareMathOperator{\Tr}{Tr}

\newcommand{\comment}[1]{}

\newcommand{\be}{\begin{enumerate}[label=(\arabic*)]}
\newcommand{\ee}{\end{enumerate}}

\def\semicolon{;}
\def\applytolist#1{
    \expandafter\def\csname multi#1\endcsname##1{
        \def\multiack{##1}\ifx\multiack\semicolon
            \def\next{\relax}
        \else
            \csname #1\endcsname{##1}
            \def\next{\csname multi#1\endcsname}
        \fi
        \next}
    \csname multi#1\endcsname}

\def\calc#1{\expandafter\def\csname c#1\endcsname{{\mathcal #1}}}
\applytolist{calc}QWERTYUIOPLKJHGFDSAZXCVBNM;
\def\bbc#1{\expandafter\def\csname bb#1\endcsname{{\mathbb #1}}}
\applytolist{bbc}QWERTYUIOPLKJHGFDSAZXCVBNM;
\def\bfc#1{\expandafter\def\csname bf#1\endcsname{{\mathbf #1}}}
\applytolist{bfc}QWERTYUIOPLKJHGFDSAZXCVBNM;
\def\sfc#1{\expandafter\def\csname s#1\endcsname{{\sf #1}}}
\applytolist{sfc}QWERTYUIOPLKJHGFDSAZXCVBNM;

\newcommand{\Mod}{{\sf ModTens}}
\newcommand{\APA}{{\sf APA}}
\renewcommand{\Vec}{{\sf Vec}}

\newcommand{\noshow}[1]{}
\newcommand{\MR}[1]{}

\def\anchor{\tikz[baseline=0, scale=.1]{\fill[red, line width=.1](-.3,-1.35) arc(-180:0:.15) arc(-99:-15:.73) -- ++(.15,-.01) -- ++(-.05,.4) -- ++(-.3,-.2) -- ++(.05,-.07) arc(-15:-90:.6) -- ++(-.12,.12) -- ++(0,1.73) -- ++(.5,0) -- ++(0,.13) -- ++(-.35,0) arc(0:180:.27) -- ++(-.35,0) -- ++(0,-.13) -- ++(.5,0) -- ++(0,-1.73) -- ++(-.12,-.12) arc(270:180+15:.6) -- ++(.05,.07) -- ++(-.3,.2) -- ++(-.05,-.4) -- ++(.15,.01) arc(180+15:270:.73) -- cycle;}}

\usetikzlibrary{shapes}
\usetikzlibrary{backgrounds}
\usetikzlibrary{decorations,decorations.pathreplacing,decorations.markings}
\usetikzlibrary{fit,calc,through}
\usetikzlibrary{external}
\tikzset{
	super thick/.style={line width=3pt}
}
\tikzstyle{shaded}=[fill=red!10!blue!20!gray!30!white]
\tikzstyle{unshaded}=[fill=white]
\tikzstyle{empty box}=[circle, draw, thick, fill=white, opaque, inner sep=2mm]
\tikzstyle{annular}=[scale=.7, inner sep=1mm, baseline]
\tikzstyle{rectangular}=[scale=.75, inner sep=1mm, baseline=-.1cm]
\tikzstyle{mid>}=[decoration={markings, mark=at position 0.5 with {\arrow{>}}}, postaction={decorate}]
\tikzstyle{mid<}=[decoration={markings, mark=at position 0.5 with {\arrow{<}}}, postaction={decorate}]
\tikzstyle{over}=[double, draw=white, super thick, double=]

\newcommand{\roundNbox}[6]{
	\draw[rounded corners=5pt, very thick, #1] ($#2+(-#3,-#3)+(-#4,0)$) rectangle ($#2+(#3,#3)+(#5,0)$);
	\coordinate (ZZa) at ($#2+(-#4,0)$);
	\coordinate (ZZb) at ($#2+(#5,0)$);
	\node at ($1/2*(ZZa)+1/2*(ZZb)$) {#6};
}

\newcommand{\nbox}[6]{
	\draw[very thick, #1] ($#2+(-#3,-#3)+(-#4,0)$) rectangle ($#2+(#3,#3)+(#5,0)$);
	\coordinate (ZZa) at ($#2+(-#4,0)$);
	\coordinate (ZZb) at ($#2+(#5,0)$);
	\node at ($1/2*(ZZa)+1/2*(ZZb)$) {#6};
}

\newcommand{\ncircle}[5]{
	\draw[very thick, #1] #2 circle (#3);
	\node at #2 {#5};
	\filldraw[red] ($#2+(#4:#3cm)$) circle (.05cm);
}

\newcommand{\Mbox}[4]{
	\pgfmathsetmacro{\planeWidth}{#2};
	\pgfmathsetmacro{\planeDepth}{#3};
	
	\draw[very thick, unshaded] ($ #1 + (-\planeDepth,\planeDepth) $) -- #1 -- ($ #1 + (\planeWidth,0) $) -- ($ #1 + (\planeWidth,0) + (-\planeDepth,\planeDepth) $) -- ($ #1 + (-\planeDepth,\planeDepth) $);

	\node at ($#1+1/2*(#2,0)+1/2*(-#3,#3)$) {\rotatebox{-75}{#4}};
}

\newcommand{\CMbox}[6]{
	\coordinate (#1) at #2;
	\pgfmathsetmacro{\boxWidth}{#3};
	\pgfmathsetmacro{\boxHeight}{#4};
	\pgfmathsetmacro{\boxDepth}{#5};
	\draw[unshaded, very thick] ($(#1) + (\boxWidth,\boxHeight) $) -- ($(#1) + (\boxWidth,\boxHeight) - (\boxDepth,-\boxDepth) $) -- ($(#1) + (0,\boxHeight) - (\boxDepth,-\boxDepth) $) -- ($(#1) - (\boxDepth,-\boxDepth) $);
	\draw[unshaded, very thick] ($(#1) + (0,\boxHeight) $)  -- ($(#1) + (\boxWidth,\boxHeight) $) -- ($(#1) + (\boxWidth,0) $) -- (#1) -- ($(#1) - (\boxDepth,-\boxDepth) $);
	\draw[very thick] ($(#1) + (0,\boxHeight) - (\boxDepth,-\boxDepth) $) -- ($(#1) + (0,\boxHeight) $) -- (#1);
	\node at ($(#1) + 1/2*(\boxWidth,\boxHeight) $) {#6}; 
}

\newcommand{\halfDottedEllipse}[3]{
	\draw[thick] #1 arc(-180:0:{#2} and {#3});
	\draw[thick, dotted] ($ #1 + 2*(#2,0)$) arc(0:180:{#2} and {#3});
}

\newcommand{\curvedTubeNoString}[4]{
	\filldraw[white, super thick, fill=white] #1 arc (-90:90:.2cm) arc (270:180:{#2 - .2}) -- ($#1 -(#2,0) - (.2,0) + (0,#2) + (0,.2) $) arc (180:270:{#2+.2});
	\filldraw[unshaded, thick] #1 arc (-90:90:.2cm) arc (270:180:{#2 - .2}) -- ($#1 -(#2,0) - (.2,0) + (0,#2) + (0,.2) $) arc (180:270:{#2+.2});
	\draw ($ #1 + (.18,.3) $) -- ($ #1 + (0,.3) $) arc (270:180:{#2 - .1});
	\draw ($ #1+ (.18,.1) $) -- ($ #1+ (0,.1) $) arc (270:180:{#2+.1});
	\roundNbox{unshaded}{($#1 - (#2,0) + (0,#2) + (0,.6) $)}{.4}{#3}{#3}{#4};
}

\newcommand{\straightTubeNoString}[3]{
	\coordinate (ZZq) at #1;
	\pgfmathsetmacro{\tubeLength}{#3};
	\pgfmathsetmacro{\tubeWidth}{#2};
	\pgfmathsetmacro{\buffer}{.05};	

	\fill[unshaded] ($ (ZZq) + (-\tubeLength,0) + 2*(0,-\buffer) $) -- ($ (ZZq) + 2*(0,-\buffer) $) arc(-90:90:{\tubeWidth+\buffer} and {2*(\tubeWidth+\buffer)}) -- ($ (ZZq) + (-\tubeLength,0) + 4*(0,\tubeWidth) + 2*(0,\buffer) $) arc(90:270:{\tubeWidth+\buffer} and {2*(\tubeWidth+\buffer)}) ;
	\draw[unshaded, thick]  ($ (ZZq) + (-\tubeLength,0) $) -- (ZZq) arc(-90:90:{\tubeWidth} and {2*\tubeWidth}) -- ($ (ZZq) + (-\tubeLength,0) + 4*(0,\tubeWidth) $) ;
	\draw[thick] ($ (ZZq) + (-\tubeLength,0) $) arc(-90:90:{\tubeWidth} and {2*\tubeWidth}) arc(90:270:{\tubeWidth} and {2*\tubeWidth});
}

\newcommand{\straightCappedTube}[3]{
	\coordinate (ZZq) at #1;
	\pgfmathsetmacro{\tubeLength}{#3};
	\pgfmathsetmacro{\tubeWidth}{#2};
	\pgfmathsetmacro{\buffer}{.05};	

	\fill[unshaded] ($ (ZZq) + (-\tubeLength,0) + 2*(0,-\buffer) $) -- ($ (ZZq) + 2*(0,-\buffer) $) arc(-90:90:{\tubeWidth+\buffer} and {2*(\tubeWidth+\buffer)}) -- ($ (ZZq) + (-\tubeLength,0) + 4*(0,\tubeWidth) + 2*(0,\buffer) $) arc(90:270:{\tubeWidth+\buffer} and {2*(\tubeWidth+\buffer)}) ;
	\draw[unshaded, thick]  ($ (ZZq) + (-\tubeLength,0) $) -- (ZZq) arc(-90:90:{\tubeWidth} and {2*\tubeWidth}) -- ($ (ZZq) + (-\tubeLength,0) + 4*(0,\tubeWidth) $) ;
	\draw[thick] ($ (ZZq) + (-\tubeLength,0) $) arc(270:90:{2*\tubeWidth});
}

\newcommand{\straightTubeWithString}[4]{
	\coordinate (ZZq) at #1;
	\pgfmathsetmacro{\tubeLength}{#3};
	\pgfmathsetmacro{\tubeWidth}{#2};
	\pgfmathsetmacro{\buffer}{.05};	

	\fill[unshaded] ($ (ZZq) + (-\tubeLength,0) + 2*(0,-\buffer) $) -- ($ (ZZq) + 2*(0,-\buffer) $) arc(-90:90:{\tubeWidth+\buffer} and {2*(\tubeWidth+\buffer)}) -- ($ (ZZq) + (-\tubeLength,0) + 4*(0,\tubeWidth) + 2*(0,\buffer) $) arc(90:270:{\tubeWidth+\buffer} and {2*(\tubeWidth+\buffer)}) ;
	\draw[unshaded, thick]  ($ (ZZq) + (-\tubeLength,0) $) -- (ZZq) arc(-90:90:{\tubeWidth} and {2*\tubeWidth}) -- ($ (ZZq) + (-\tubeLength,0) + 4*(0,\tubeWidth) $) ;
	\draw[thick] ($ (ZZq) + (-\tubeLength,0) $) arc(-90:90:{\tubeWidth} and {2*\tubeWidth}) arc(90:270:{\tubeWidth} and {2*\tubeWidth});
	\draw[thick, #4] ($(ZZq) + (\tubeWidth,0) + 2*(0,\tubeWidth) $) -- ($ (ZZq) + (-\tubeLength,0) + 2*(0,\tubeWidth) + (\tubeWidth,0)$);
}

\newcommand{\straightTubeTwoStrings}[4]{
	\fill[unshaded] ($ #1 + (-#2,-.1) $) -- ($ #1 + (0,-.1) $) arc(-90:90:.15cm and .3cm) -- ($ #1 + (-#2,.5) $) ;
	\draw[unshaded, thick]  ($ #1 + (-#2,0) $) -- #1 arc(-90:90:.1cm and .2cm) -- ($ #1 + (-#2,.4) $) ;
	\draw[thick] ($ #1 + (-#2,0) $) arc(-90:90:.1cm and .2cm) arc(90:270:.1cm and .2cm);
	\draw[thick, #3] ($#1 + (.08,.27) $) -- ($ #1 + (-#2,.27) + (.08,0)$);
	\draw[thick, #4] ($#1 + (.08,.13) $) -- ($ #1 + (-#2,.13) + (.08,0)$);
}

\newcommand{\pairOfPants}[2]{
	\draw[thick] #1 .. controls ++(90:.8cm) and ++(270:.8cm) .. ($ #1 + (.7,1.5) $);
	\draw[thick] ($ #1 + (2,0) $) .. controls ++(90:.8cm) and ++(270:.8cm) .. ($ #1 + (2,0) + (-.7,1.5) $);
	\draw[thick] ($ #1 + (.6,0) $).. controls ++(90:.8cm) and ++(90:.8cm) .. ($ #1 + (1.4,0) $); 
	\halfDottedEllipse{($ #1 + (.7,1.5) $)}{.3}{.1}
	\halfDottedEllipse{#1}{.3}{.1}
	\halfDottedEllipse{($ #1 + (1.4,0) $)}{.3}{.1}
}

\newcommand{\topPairOfPants}[2]{
	\draw[thick] #1 .. controls ++(90:.8cm) and ++(270:.8cm) .. ($ #1 + (.7,1.5) $);
	\draw[thick] ($ #1 + (2,0) $) .. controls ++(90:.8cm) and ++(270:.8cm) .. ($ #1 + (2,0) + (-.7,1.5) $);
	\draw[thick] ($ #1 + (.6,0) $).. controls ++(90:.8cm) and ++(90:.8cm) .. ($ #1 + (1.4,0) $); 
	\draw[thick] ($ #1 + (1,1.5) $) ellipse (.3cm and .1cm);
	\halfDottedEllipse{#1}{.3}{.1}
	\halfDottedEllipse{($ #1 + (1.4,0) $)}{.3}{.1}

}

\newcommand{\pairOfPantsNoCircles}[2]{
	\draw[thick] #1 .. controls ++(90:.8cm) and ++(270:.8cm) .. ($ #1 + (.7,1.5) $);
	\draw[thick] ($ #1 + (2,0) $) .. controls ++(90:.8cm) and ++(270:.8cm) .. ($ #1 + (2,0) + (-.7,1.5) $);
	\draw[thick] ($ #1 + (.6,0) $).. controls ++(90:.8cm) and ++(90:.8cm) .. ($ #1 + (1.4,0) $); 
}

\newcommand{\invertedPairOfPants}[2]{
	\draw[thick] #1 .. controls ++(270:.8cm) and ++(90:.8cm) .. ($ #1 + (.7,-1.5) $);
	\draw[thick] ($ #1 + (2,0) $) .. controls ++(270:.8cm) and ++(90:.8cm) .. ($ #1 + (2,0) + (-.7,-1.5) $);
	\draw[thick] ($ #1 + (.6,0) $).. controls ++(270:.8cm) and ++(270:.8cm) .. ($ #1 + (1.4,0) $); 
	\halfDottedEllipse{($ #1 + (.7,-1.5) $)}{.3}{.1}

}

\newcommand{\topCylinder}[3]{
	\draw[thick] #1 -- ($ #1 + (0,#3) $);
	\draw[thick] ($ #1 + 2*(#2,0) $) -- ($ #1 + (0,#3) + 2*(#2,0) $);
	\draw[thick] ($ #1 + (#2,#3) $) ellipse (#2 and {1/3*#2});
}

\newcommand{\bottomCylinder}[3]{
	\draw[thick] #1 -- ($ #1 + (0,#3) $);
	\draw[thick] ($ #1 + 2*(#2,0) $) -- ($ #1 + 2*(#2,0) + (0,#3) $);
	\halfDottedEllipse{#1}{#2}{{1/3*#2}}	
}

\newcommand{\emptyCylinder}[3]{
	\draw[thick] #1 -- ($ #1 + (0,#3) $);
	\draw[thick] ($ #1 + 2*(#2,0) $) -- ($ #1 + 2*(#2,0) + (0,#3) $);	
}

\newcommand{\RightSlantCylinder}[2]{
	\draw[thick] #1 .. controls ++(90:.8cm) and ++(270:.8cm) .. ($ #1 + (.7,1.5) $);
	\draw[thick] ($ #1 + (.6,0) $).. controls ++(90:.8cm) and ++(270:.8cm) .. ($ #1 + (1.3,1.5) $); 
	\halfDottedEllipse{($ #1 + (.7,1.5) $)}{.3}{.1}
	\halfDottedEllipse{#1}{.3}{.1}
}

\newcommand{\LeftSlantCylinder}[2]{
	\draw[thick] #1 .. controls ++(90:.8cm) and ++(270:.8cm) .. ($ #1 + (-.7,1.5) $);
	\draw[thick] ($ #1 + (.6,0) $).. controls ++(90:.8cm) and ++(270:.8cm) .. ($ #1 + (-.1,1.5) $); 
	\halfDottedEllipse{($ #1 + (-.7,1.5) $)}{.3}{.1}
	\halfDottedEllipse{#1}{.3}{.1}
}

\newcommand{\braid}[3]{

	\coordinate (ZZz) at #1;
	\pgfmathsetmacro{\tubeHeight}{#3};
	\pgfmathsetmacro{\tubeRadius}{#2};
	\pgfmathsetmacro{\buffer}{.2};	

	\draw[thick] ($ (ZZz) +(\tubeHeight,0) $) .. controls ++(90:.7cm) and ++(270:.7cm) .. ($ (ZZz) + 2*(\tubeRadius,0) + (0,\tubeHeight) $);
	\draw[thick] ($ (ZZz) + (\tubeHeight,0) - 2*(\tubeRadius,0) $) .. controls ++(90:.7cm) and ++(270:.7cm) .. ($ (ZZz) + (0,\tubeHeight) $);

	\fill[unshaded] ($ (ZZz) - (\buffer,0) $) .. controls ++(90:.8cm) and ++(270:.8cm) .. ($ (ZZz) + (\tubeHeight,\tubeHeight) - 2*(\tubeRadius,0) - (\buffer,0) $) -- ($ (ZZz) + (\tubeHeight,\tubeHeight) + (\buffer,0) $)  .. controls ++(270:.8cm) and ++(90:.8cm) .. ($ (ZZz) + 2*(\tubeRadius,0) + (\buffer,0) $);
	\draw[thick] (ZZz) .. controls ++(90:.7cm) and ++(270:.7cm) .. ($ (ZZz) + (\tubeHeight,\tubeHeight) - 2*(\tubeRadius,0) $);
	\draw[thick] ($ (ZZz) + 2*(\tubeRadius,0) $) .. controls ++(90:.7cm) and ++(270:.7cm) .. ($ (ZZz) + (\tubeHeight,\tubeHeight) $);
}

\newcommand{\plane}[3]{
	\pgfmathsetmacro{\planeWidth}{#2};
	\pgfmathsetmacro{\planeDepth}{#3};
	
	\draw[thick] ($ #1 + (-\planeDepth,\planeDepth) $) -- #1 -- ($ #1 + (\planeWidth,0) $) -- ($ #1 + (\planeWidth,0) + (-\planeDepth,\planeDepth) $) -- ($ #1 + (-\planeDepth,\planeDepth) $);
}

\newcommand{\multiplication}[5]{
	\draw[rounded corners=5pt, very thick, unshaded] ($ #1 - (#2,#2) $) rectangle ($ #1 + (#2,#2) $);
	\draw ($ #1 + 5/6*(0,#2) $) -- ($ #1 - 5/6*(0,#2) $);
	\draw[thick, red] ($ #1 + 1/3*(0,#2) - 1/5*(#2,0) $) -- ($ #1 - 2/3*(#2,0) $);
	\draw[thick, red] ($ #1 - 1/3*(0,#2) - 1/5*(#2,0) $) -- ($ #1 - 2/3*(#2,0) $);
	\draw[very thick] #1 ellipse ({2/3*#2} and {5/6*#2});
	\filldraw[very thick, unshaded] ($ #1 + 1/3*(0,#2) $) circle (1/5*#2);
	\filldraw[very thick, unshaded] ($ #1 - 1/3*(0,#2) $) circle (1/5*#2);
	\node at ($ #1 + (.2,0) + 5/8*(0,#2)$) {\scriptsize{$#5$}};
	\node at ($ #1 + (.2,0) $) {\scriptsize{$#4$}};
	\node at ($ #1 + (.2,0) - 5/8*(0,#2)$) {\scriptsize{$#3$}};
}

\newcommand{\tensor}[6]{
	\draw[rounded corners=5pt, very thick, unshaded] ($ #1 - (#2,#2) $) rectangle ($ #1 + (#2,#2) $);
	\draw ($ #1 + 14/23*(0,#2) - 1/3*(#2,0) $) -- ($ #1 - 14/23*(0,#2) - 1/3*(#2,0) $);
	\draw ($ #1 + 14/23*(0,#2) + 1/3*(#2,0) $) -- ($ #1 - 14/23*(0,#2) + 1/3*(#2,0) $);
	\draw[thick, red] ($ #1 - 1/3*(#2,0) - 1/5*(#2,0) $) -- ($ #1 - 5/6*(#2,0) $);
	\draw[thick, red] ($ #1 + 1/3*(#2,0) - 1/5*(#2,0) $) .. controls ++(180:.2cm) and ++(0:.2cm) .. ($ #1 - 1/3*(#2,0) + 2/5*(0,#2) $) .. controls ++(180:.2cm) and ++(0:.2cm) .. ($ #1 - 5/6*(#2,0) $);
	\draw[very thick] #1 ellipse ( {5/6*#2} and {2/3*#2});
	\filldraw[very thick, unshaded] ($ #1 + 1/3*(#2,0) $) circle (1/5*#2);
	\filldraw[very thick, unshaded] ($ #1 - 1/3*(#2,0) $) circle (1/5*#2);
	\node at ($ #1 + (.2,0) - 1/3*(#2,0) - 1/3*(0,#2) $) {\scriptsize{$#3$}};
	\node at ($ #1 + (.2,0) - 1/3*(#2,0) + 1/3*(0,#2) $) {\scriptsize{$#4$}};
	\node at ($ #1 + (.2,0) + 1/3*(#2,0) - 1/3*(0,#2) $) {\scriptsize{$#5$}};
	\node at ($ #1 + (.2,0) + 1/3*(#2,0) + 1/3*(0,#2) $) {\scriptsize{$#6$}};
}

\newcommand{\tensorRightId}[5]{
	\draw[rounded corners=5pt, very thick, unshaded] ($ #1 - (#2,#2) $) rectangle ($ #1 + (#2,#2) $);
	\draw ($ #1 + 14/23*(0,#2) - 1/3*(#2,0) $) -- ($ #1 - 14/23*(0,#2) - 1/3*(#2,0) $);
	\draw ($ #1 + 14/23*(0,#2) + 1/3*(#2,0) $) -- ($ #1 - 14/23*(0,#2) + 1/3*(#2,0) $);
	\draw[thick, red] ($ #1 - 1/3*(#2,0) - 1/5*(#2,0) $) -- ($ #1 - 5/6*(#2,0) $);
	\draw[very thick] #1 ellipse ( {5/6*#2} and {2/3*#2});
	\filldraw[very thick, unshaded] ($ #1 - 1/3*(#2,0) $) circle (1/5*#2);
	\node at ($ #1 + (.2,0) - 1/3*(#2,0) - 1/3*(0,#2) $) {\scriptsize{$#3$}};
	\node at ($ #1 + (.2,0) - 1/3*(#2,0) + 1/3*(0,#2) $) {\scriptsize{$#4$}};
	\node at ($ #1 + (.2,0) + 1/3*(#2,0) $) {\scriptsize{$#5$}};
}

\newcommand{\tensorRightIdZero}[4]{
	\draw[rounded corners=5pt, very thick, unshaded] ($ #1 - (#2,#2) $) rectangle ($ #1 + (#2,#2) $);
	\draw ($ #1 + 14/23*(0,#2) - 1/3*(#2,0) $) -- ($ #1 - 14/23*(0,#2) - 1/3*(#2,0) $);
	\draw[thick, red] ($ #1 - 1/3*(#2,0) - 1/5*(#2,0) $) -- ($ #1 - 5/6*(#2,0) $);
	\draw[very thick] #1 ellipse ( {5/6*#2} and {2/3*#2});
	\filldraw[very thick, unshaded] ($ #1 - 1/3*(#2,0) $) circle (1/5*#2);
	\node at ($ #1 + (.2,0) - 1/3*(#2,0) - 1/3*(0,#2) $) {\scriptsize{$#3$}};
	\node at ($ #1 + (.2,0) - 1/3*(#2,0) + 1/3*(0,#2) $) {\scriptsize{$#4$}};
}

\newcommand{\tensorRightIdZeroZero}[2]{
	\draw[rounded corners=5pt, very thick, unshaded] ($ #1 - (#2,#2) $) rectangle ($ #1 + (#2,#2) $);
	\draw[thick, red] ($ #1 - 1/3*(#2,0) - 1/5*(#2,0) $) -- ($ #1 - 5/6*(#2,0) $);
	\draw[very thick] #1 ellipse ( {5/6*#2} and {2/3*#2});
	\filldraw[very thick, unshaded] ($ #1 - 1/3*(#2,0) $) circle (1/5*#2);
}

\newcommand{\tensorLeftRightId}[2]{
	\draw[rounded corners=5pt, very thick, unshaded] ($ #1 - (#2,#2) $) rectangle ($ #1 + (#2,#2) $);
	\draw ($ #1 + 7/13*(0,#2) - 1/2*(#2,0) $) -- ($ #1 - 7/13*(0,#2) - 1/2*(#2,0) $);
	\draw ($ #1 + 2/3*(0,#2) $) -- ($ #1 - 2/3*(0,#2) $);
	\draw ($ #1 + 7/13*(0,#2) + 1/2*(#2,0) $) -- ($ #1 - 7/13*(0,#2) + 1/2*(#2,0) $);
	\draw[thick, red] ($ #1 - 1/5*(#2,0) $) -- ($ #1 - 5/6*(#2,0) $);
	\draw[very thick] #1 ellipse ( {5/6*#2} and {2/3*#2});
	\filldraw[very thick, unshaded] ($ #1 $) circle (1/5*#2);
}

\newcommand{\tensorLeftRightIdEv}[2]{
	\draw[rounded corners=5pt, very thick, unshaded] ($ #1 - (#2,#2) $) rectangle ($ #1 + (#2,#2) $);
	\draw ($ #1 + 7/13*(0,#2) - 1/2*(#2,0) $) -- ($ #1 - 7/13*(0,#2) - 1/2*(#2,0) $);
	\draw ($ #1 $) -- ($ #1 - 2/3*(0,#2) $);
	\draw ($ #1 + 7/13*(0,#2) + 1/2*(#2,0) $) -- ($ #1 - 7/13*(0,#2) + 1/2*(#2,0) $);
	\draw[thick, red] ($ #1 - 1/5*(#2,0) $) -- ($ #1 - 5/6*(#2,0) $);
	\draw[very thick] #1 ellipse ( {5/6*#2} and {2/3*#2});
	\filldraw[very thick, unshaded] ($ #1 $) circle (1/5*#2);
}

\newcommand{\tensorLeftRightIdCoev}[2]{
	\draw[rounded corners=5pt, very thick, unshaded] ($ #1 - (#2,#2) $) rectangle ($ #1 + (#2,#2) $);
	\draw ($ #1 + 7/13*(0,#2) - 1/2*(#2,0) $) -- ($ #1 - 7/13*(0,#2) - 1/2*(#2,0) $);
	\draw ($ #1 + 2/3*(0,#2) $) -- ($ #1 $);
	\draw ($ #1 + 7/13*(0,#2) + 1/2*(#2,0) $) -- ($ #1 - 7/13*(0,#2) + 1/2*(#2,0) $);
	\draw[thick, red] ($ #1 - 1/5*(#2,0) $) -- ($ #1 - 5/6*(#2,0) $);
	\draw[very thick] #1 ellipse ( {5/6*#2} and {2/3*#2});
	\filldraw[very thick, unshaded] ($ #1 $) circle (1/5*#2);
}

\newcommand{\tensorRightIdEv}[4]{
	\draw[rounded corners=5pt, very thick, unshaded] ($ #1 - (#2,#2) $) rectangle ($ #1 + (#2,#2) $);
	\draw ($ #1 - 1/3*(#2,0) $) -- ($ #1 - 14/23*(0,#2) - 1/3*(#2,0) $);
	\draw ($ #1 + 14/23*(0,#2) + 1/3*(#2,0) $) -- ($ #1 - 14/23*(0,#2) + 1/3*(#2,0) $);
	\draw[thick, red] ($ #1 - 1/3*(#2,0) - 1/5*(#2,0) $) -- ($ #1 - 5/6*(#2,0) $);
	\draw[very thick] #1 ellipse ( {5/6*#2} and {2/3*#2});
	\filldraw[very thick, unshaded] ($ #1 - 1/3*(#2,0) $) circle (1/5*#2);
	\node at ($ #1 + (.2,0) - 1/3*(#2,0) - 1/3*(0,#2) $) {\scriptsize{$#3$}};
	\node at ($ #1 + (.2,0) + 1/3*(#2,0) $) {\scriptsize{$#4$}};
}

\newcommand{\tensorRightIdCoev}[4]{
	\draw[rounded corners=5pt, very thick, unshaded] ($ #1 - (#2,#2) $) rectangle ($ #1 + (#2,#2) $);
	\draw ($ #1 + 14/23*(0,#2) - 1/3*(#2,0) $) -- ($ #1 - 1/3*(#2,0) $);
	\draw ($ #1 + 14/23*(0,#2) + 1/3*(#2,0) $) -- ($ #1 - 14/23*(0,#2) + 1/3*(#2,0) $);
	\draw[thick, red] ($ #1 - 1/3*(#2,0) - 1/5*(#2,0) $) -- ($ #1 - 5/6*(#2,0) $);
	\draw[very thick] #1 ellipse ( {5/6*#2} and {2/3*#2});
	\filldraw[very thick, unshaded] ($ #1 - 1/3*(#2,0) $) circle (1/5*#2);
	\node at ($ #1 + (.2,0) - 1/3*(#2,0) + 1/3*(0,#2) $) {\scriptsize{$#3$}};
	\node at ($ #1 + (.2,0) + 1/3*(#2,0) $) {\scriptsize{$#4$}};
}

\newcommand{\tensorLeftId}[5]{
	\draw[rounded corners=5pt, very thick, unshaded] ($ #1 - (#2,#2) $) rectangle ($ #1 + (#2,#2) $);
	\draw ($ #1 + 14/23*(0,#2) - 1/3*(#2,0) $) -- ($ #1 - 14/23*(0,#2) - 1/3*(#2,0) $);
	\draw ($ #1 + 14/23*(0,#2) + 1/3*(#2,0) $) -- ($ #1 - 14/23*(0,#2) + 1/3*(#2,0) $);
	\draw[thick, red] ($ #1 + 1/3*(#2,0) - 1/5*(#2,0) $) -- ($ #1 - 5/6*(#2,0) $);
	\draw[very thick] #1 ellipse ( {5/6*#2} and {2/3*#2});
	\filldraw[very thick, unshaded] ($ #1 + 1/3*(#2,0) $) circle (1/5*#2);
	\node at ($ #1 - (.2,0) - 1/3*(#2,0) + (0,.2) $) {\scriptsize{$#3$}};
	\node at ($ #1 - (.2,0) + 1/3*(#2,0) - 1/3*(0,#2) $) {\scriptsize{$#4$}};
	\node at ($ #1 - (.2,0) + 1/3*(#2,0) + 1/3*(0,#2) $) {\scriptsize{$#5$}};
}

\newcommand{\tensorLeftIdZero}[4]{
	\draw[rounded corners=5pt, very thick, unshaded] ($ #1 - (#2,#2) $) rectangle ($ #1 + (#2,#2) $);
	\draw ($ #1 + 14/23*(0,#2) + 1/3*(#2,0) $) -- ($ #1 - 14/23*(0,#2) + 1/3*(#2,0) $);
	\draw[thick, red] ($ #1 + 1/3*(#2,0) - 1/5*(#2,0) $) -- ($ #1 - 5/6*(#2,0) $);
	\draw[very thick] #1 ellipse ( {5/6*#2} and {2/3*#2});
	\filldraw[very thick, unshaded] ($ #1 + 1/3*(#2,0) $) circle (1/5*#2);
	\node at ($ #1 - (.2,0) + 1/3*(#2,0) - 1/3*(0,#2) $) {\scriptsize{$#3$}};
	\node at ($ #1 - (.2,0) + 1/3*(#2,0) + 1/3*(0,#2) $) {\scriptsize{$#4$}};
}

\newcommand{\tensorLeftIdZeroZero}[2]{
	\draw[rounded corners=5pt, very thick, unshaded] ($ #1 - (#2,#2) $) rectangle ($ #1 + (#2,#2) $);
	\draw[thick, red] ($ #1 + 1/3*(#2,0) - 1/5*(#2,0) $) -- ($ #1 - 5/6*(#2,0) $);
	\draw[very thick] #1 ellipse ( {5/6*#2} and {2/3*#2});
	\filldraw[very thick, unshaded] ($ #1 + 1/3*(#2,0) $) circle (1/5*#2);
}

\newcommand{\tensorLeftIdEv}[4]{
	\draw[rounded corners=5pt, very thick, unshaded] ($ #1 - (#2,#2) $) rectangle ($ #1 + (#2,#2) $);
	\draw ($ #1 + 14/23*(0,#2) - 1/3*(#2,0) $) -- ($ #1 - 14/23*(0,#2) - 1/3*(#2,0) $);
	\draw ($ #1 + 1/3*(#2,0) $) -- ($ #1 - 14/23*(0,#2) + 1/3*(#2,0) $);
	\draw[thick, red] ($ #1 + 1/3*(#2,0) - 1/5*(#2,0) $) -- ($ #1 - 5/6*(#2,0) $);
	\draw[very thick] #1 ellipse ( {5/6*#2} and {2/3*#2});
	\filldraw[very thick, unshaded] ($ #1 + 1/3*(#2,0) $) circle (1/5*#2);
	\node at ($ #1 - (.2,0) - 1/3*(#2,0) + (0,.2) $) {\scriptsize{$#3$}};
	\node at ($ #1 - (.2,0) + 1/3*(#2,0) - 1/3*(0,#2) $) {\scriptsize{$#4$}};
}

\newcommand{\tensorLeftIdCoev}[4]{
	\draw[rounded corners=5pt, very thick, unshaded] ($ #1 - (#2,#2) $) rectangle ($ #1 + (#2,#2) $);
	\draw ($ #1 + 14/23*(0,#2) - 1/3*(#2,0) $) -- ($ #1 - 14/23*(0,#2) - 1/3*(#2,0) $);
	\draw ($ #1 + 14/23*(0,#2) + 1/3*(#2,0) $) -- ($ #1 + 1/3*(#2,0) $);
	\draw[thick, red] ($ #1 + 1/3*(#2,0) - 1/5*(#2,0) $) -- ($ #1 - 5/6*(#2,0) $);
	\draw[very thick] #1 ellipse ( {5/6*#2} and {2/3*#2});
	\filldraw[very thick, unshaded] ($ #1 + 1/3*(#2,0) $) circle (1/5*#2);
	\node at ($ #1 - (.2,0) - 1/3*(#2,0) + (0,.2) $) {\scriptsize{$#3$}};
	\node at ($ #1 - (.2,0) + 1/3*(#2,0) + 1/3*(0,#2) $) {\scriptsize{$#4$}};
}

\newcommand{\identityMap}[3]{
	\draw[rounded corners=5pt, very thick, unshaded] ($ #1 - (#2,#2) $) rectangle ($ #1 + (#2,#2) $);
	\draw ($ #1 + 5/6*(0,#2) $) -- ($ #1 - 5/6*(0,#2) $);
	\draw[very thick] #1 circle ({5/6*#2});
	\filldraw[red] ($ #1 - 5/6*(#2,0) $) circle ({1/10*#2});
	\node at ($ #1 + (.2,0) $) {\scriptsize{$#3$}};
}

\newcommand{\emptyMap}[2]{
	\draw[rounded corners=5pt, very thick, unshaded] ($ #1 - (#2,#2) $) rectangle ($ #1 + (#2,#2) $);
	\draw[very thick] #1 circle ({2/3*#2});
	\filldraw[red] ($ #1 - 2/3*(#2,0) $) circle (.05cm);
}

\newcommand{\twist}[1]{
	\fill[unshaded] ($ #1 - (.1,.2) $) rectangle ($ #1 + (.1,.2) $);
	\draw ($ #1 - (0,.2) $) .. controls ++(135:.1cm) and ++(-135:.2cm) .. #1 .. controls ++(45:.1cm) and ++(-45:.2cm) .. ($ #1 + (0,.2) $);
}

\newcommand{\twistInverse}[1]{
	\fill[unshaded] ($ #1 - (.1,.2) $) rectangle ($ #1 + (.1,.2) $);
	\draw ($ #1 - (0,.2) $) .. controls ++(45:.1cm) and ++(-45:.2cm) .. #1 .. controls ++(135:.1cm) and ++(-135:.2cm) .. ($ #1 + (0,.2) $);
}

\newcommand{\twistHorizontal}[1]{
	\fill[unshaded] ($ #1 - (.2,.1) $) rectangle ($ #1 + (.2,.1) $);
	\draw ($ #1 - (.2,0) $) .. controls ++(45:.1cm) and ++(135:.2cm) .. #1 .. controls ++(-45:.1cm) and ++(-135:.2cm) .. ($ #1 + (.2,0) $);
}

\newcommand{\twistHorizontalInverse}[1]{
	\fill[unshaded] ($ #1 - (.2,.1) $) rectangle ($ #1 + (.2,.1) $);
	\draw ($ #1 - (.2,0) $) .. controls ++(-45:.1cm) and ++(-135:.2cm) .. #1 .. controls ++(45:.1cm) and ++(135:.2cm) .. ($ #1 + (.2,0) $);
}

\newcommand{\loopIso}[1]{
	\fill[unshaded] ($ #1 - (.3,.3) $) rectangle ($ #1 + (.1,.3) $);
	\draw ($ #1 + (-.3,.2) $) arc (90:270:.2cm);
	\draw ($ #1 + (0,-.3) $)  .. controls ++(90:.2cm) and ++(0:.2cm) .. ($ #1 + (-.3,.2) $);
	\draw[super thick, white] ($ #1 + (0,.3) $)  .. controls ++(270:.2cm) and ++(0:.2cm) .. ($ #1 + (-.3,-.2) $);
	\draw ($ #1 + (0,.3) $)  .. controls ++(270:.2cm) and ++(0:.2cm) .. ($ #1 + (-.3,-.2) $);
}

\newcommand{\loopIsoReverse}[1]{
	\fill[unshaded] ($ #1 - (.3,.3) $) rectangle ($ #1 + (.1,.3) $);
	\draw ($ #1 + (-.3,.2) $) arc (90:270:.2cm);
	\draw ($ #1 + (0,.3) $)  .. controls ++(270:.2cm) and ++(0:.2cm) .. ($ #1 + (-.3,-.2) $);
	\draw[super thick, white]	 ($ #1 + (0,-.3) $)  .. controls ++(90:.2cm) and ++(0:.2cm) .. ($ #1 + (-.3,.2) $);
	\draw ($ #1 + (0,-.3) $)  .. controls ++(90:.2cm) and ++(0:.2cm) .. ($ #1 + (-.3,.2) $);
}

\newcommand{\loopIsoInverseReverse}[1]{
	\fill[unshaded] ($ #1 - (.1,.3) $) rectangle ($ #1 + (.3,.3) $);
	\draw ($ #1 + (.3,.2) $) arc (90:-90:.2cm);
	\draw ($ #1 + (0,.3) $)  .. controls ++(270:.2cm) and ++(180:.2cm) .. ($ #1 + (.3,-.2) $);
	\draw[super thick, white] ($ #1 + (0,-.3) $)  .. controls ++(90:.2cm) and ++(180:.2cm) .. ($ #1 + (.3,.2) $);
	\draw ($ #1 + (0,-.3) $)  .. controls ++(90:.2cm) and ++(180:.2cm) .. ($ #1 + (.3,.2) $);
}

\newcommand{\braiding}[3]{
	\fill[unshaded] ($ #1 - (.1,#3) $) rectangle ($ #1 + (.1,#3) + (#2,0) $);
	\draw ($ #1 + (#2,-#3) $) to[out=90, in=-90] ($ #1 + (0,#3) $);
	\draw[white, line width=4] ($ #1 + (0,-#3) $) to[out=90, in=-90] ($ #1 + (#2,#3) $);
	\draw ($ #1 + (0,-#3) $) to[out=90, in=-90] ($ #1 + (#2,#3) $);
}

\newcommand{\evaluationMap}[3]{
	\draw[rounded corners=5pt, very thick, unshaded] ($ #1 - (#2,#2) $) rectangle ($ #1 + (#2,#2) $);
	\draw ($ #1 - 2/3*(#2,0) - 1/2*(0,#2) $) .. controls ++(90:{#2}) and ++(90:{#2}) .. ($ #1 + 2/3*(#2,0) - 1/2*(0,#2) $);
	\draw[very thick] #1 circle ({5/6*#2});
	\filldraw[red] ($ #1 - 5/6*(#2,0) $) circle ({1/10*#2});
	\node at ($ #1 - (0,.1) $) {\scriptsize{$#3$}};
}

\newcommand{\coevaluationMap}[3]{
	\draw[rounded corners=5pt, very thick, unshaded] ($ #1 - (#2,#2) $) rectangle ($ #1 + (#2,#2) $);
	\draw ($ #1 - 2/3*(#2,0) + 1/2*(0,#2) $) .. controls ++(270:{#2}) and ++(270:{#2}) .. ($ #1 + 2/3*(#2,0) + 1/2*(0,#2) $);
	\draw[very thick] #1 circle ({5/6*#2});
	\filldraw[red] ($ #1 - 5/6*(#2,0) $) circle ({1/10*#2});
	\node at ($ #1 + (0,.1) $) {\scriptsize{$#3$}};
}

\newcommand{\identityTangle}[4]{
\	\draw[rounded corners=5pt, very thick, unshaded] ($ #1 - (#2,#2) $) rectangle ($ #1 + (#2,#2) $);
	\draw ($ #1 + 5/6*(0,#2) $) -- ($ #1 - 5/6*(0,#2) $);
	\draw[thick, red] ($ #1 - 1/5*(#2,0) $) -- ($ #1 - 5/6*(#2,0) $);
	\draw[very thick] #1 circle ({5/6*#2});
	\draw[very thick, unshaded] #1 circle ({1/5*#2});
	\node at ($ #1 + (.2,0) - 1/2*(0,#2)$) {\scriptsize{$#3$}};
	\node at ($ #1 + (.2,0) + 1/2*(0,#2)$) {\scriptsize{$#4$}};
}

\makeatletter
\newcommand{\hashdef}[2]{\@namedef{#1}{#2}}
\newcommand{\hashlookup}[1]{\@nameuse{#1}}
\makeatother

\begin{document}
\title{Planar algebras in braided tensor categories}
\author{Andr\'{e} Henriques, David Penneys, and James Tener}
\date{\today}
\maketitle

\begin{abstract}
We generalize Jones' planar algebras by internalising the notion to a pivotal braided tensor category $\cC$.
To formulate the notion, the planar tangles are now equipped with additional `anchor lines' which connect the inner circles to the outer circle.
We call the resulting notion an \emph{anchored planar algebra}.
If we restrict to the case when $\cC$ is the category of vector spaces, then we recover the usual notion of a planar algebra.
\\\indent
Building on our previous work on categorified traces,
we prove that there is an equivalence of categories between anchored planar algebras in $\cC$ and pivotal module tensor categories over $\cC$ equipped with a chosen self-dual generator.
Even in the case of usual planar algebras, the precise formulation of this theorem, as an equivalence of categories, has not appeared in the literature.
Using our theorem, we describe many examples of anchored planar algebras.
\end{abstract}

\[
\,\,\tikz{
\draw[ultra thick] (0,0) circle (4);
\def \rA {.7}\def \rB {.7}\def \rC {.6}\def \rD {.7}
\path (60:1.75) coordinate (a)
(150:1.9) coordinate (b)
(-100:1.2) coordinate (c)
(-20:2) coordinate (d);
\draw (a) + (-70:\rA+.1) circle (.4);
\draw (a) + (-15:\rA+1.2) circle (.3);
\draw (b) + (-135:\rB+.1) circle (.58 and .6);
\draw (b) to[out=-30,in=100,looseness=1.2] (c);
\draw (d) + (-.1,-.1) to[out=150,in=40,looseness=1.2] (c);
\draw (d) + (-.1,-.1) to[out=-130,in=-40,looseness=1.2] (c);
\draw (a) to[out=155,in=55,looseness=1.2] (b);
\draw (c) + (-\rC,0) to[out=180,in=90] ($(c) + (-150:2.2*\rC)$) to[out=-90,in=90] (-100:4);
\draw (-70:4) arc (210:40:1);
\draw (42:4) to[bend left=15] ($(a) + (.2,.2)$);
\draw (70:4) to[bend right=20] ($(a) + (.1,.2)$);
\draw (137:4) to[bend left=10] ($(b) + (-.1,.2)$);
\draw[fill=white, ultra thick] (a) circle (\rA) (b) circle (\rB) (c) circle (\rC) (d) circle (\rD);
\fill[red] (-137:4) coordinate (Z) circle (.08);
\fill[red] (a) + (-140:\rA) coordinate (A) circle (.08);
\fill[red] (b) + (-140:\rB) coordinate (B) circle (.08);
\fill[red] (c) + (-110:\rC) coordinate (C) circle (.08);
\fill[red] (d) + (-80:\rD) coordinate (D) circle (.08);
\node[scale=1.6] at (-137:4.45) {\anchor};
\draw[very thick, red]
(Z) ..controls ++(50:1) and ++(-130:1).. ($(A)!.50!(Z) + (-.25,.25)$) ..controls ++(50:1) and ++(-130:1).. (A)
(Z) ..controls ++(60:1) and ++(-130:1.5).. (B)
(Z) ..controls ++(39:1.5) and ++(-130:.9).. (C)
(Z) ..controls ++(30:1.1) and ++(-90:1.5).. (D);
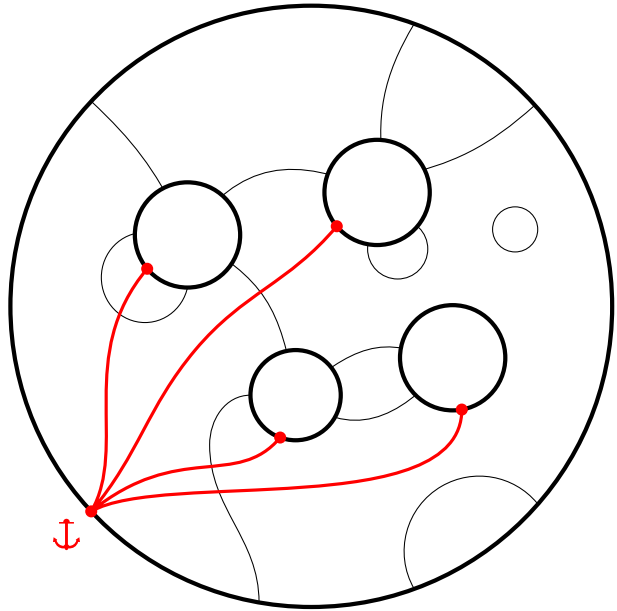
\node at (0,-4.9) {Figure 1: An anchored planar tangle};
}
\]
\addtocounter{figure}{1}
\label{That's where Figure 1 is}

\newpage

\setcounter{tocdepth}{2}
\tableofcontents
\newpage


\settocdepth{section}

\section{Introduction}\label{sec:Introduction}

In \cite{math.QA/9909027}, Jones introduced the notion of a {\it planar algebra} as an axiomatization of the standard invariant of a finite index subfactor. 
A planar algebra (in vector spaces) is a sequence of vector spaces $\cP[0]$, $\cP[1]$, $\cP[2], \ldots$ called the {\it box spaces} of the planar algebra, along with an action of planar tangles\footnote{While Jones worked with shaded planar tangles in his study of subfactors, in this article we work with unshaded planar tangles.}, i.e., for every planar tangle $T$
we have a linear map 
$Z(T): \cP[k_1] \otimes \cdots \otimes \cP[k_r] \to \cP[k_0]$. For example:
$$
Z\left(
\begin{tikzpicture}[baseline =-.1cm]
	\coordinate (a) at (0,0);
	\coordinate (b) at ($ (a) + (1.4,1) $);
	\coordinate (c) at ($ (a) + (.6,-.6) $);
	\coordinate (d) at ($ (a) + (-.6,.6) $);
	\coordinate (e) at ($ (a) + (-.8,-.6) $);
	
	\ncircle{}{(a)}{1.6}{85}{}
		
	\draw (60:1.6cm) arc (150:300:.4cm);
	\draw ($ (c) + (0,.4) $) arc (0:90:.8cm);
	\draw ($ (c) + (-.4,0) $) circle (.25cm);
	\draw ($ (d) + (0,.88) $) -- (d) -- ($ (d) + (-.88,0) $);
	\draw ($ (c) + (0,-.88) $) -- (c) -- ($ (c) + (.88,0) $);
	\ncircle{unshaded}{(d)}{.4}{235}{}
	\ncircle{unshaded}{(c)}{.4}{235}{}
	\node[blue] at (c) {\small 2};
	\node[blue] at (d) {\small 1};
\end{tikzpicture}
\right)\,\,:\,\,\,\cP[3] \otimes \cP[5] \to \cP[6]
$$
The operation of sticking a tangle inside another is required, by the axioms of a planar algebra, to correspond to the composition of multilinear maps.

In this paper, we internalize the notion of planar algebra to a braided pivotal tensor category $\cC$.
The resulting notion is called an \emph{anchored planar algebra}. 
The box spaces $\cP[k]$ are now objects of $\cC$, and the planar tangles are replaced by {\it anchored planar tangles} (see Figure~\pageref{That's where Figure 1 is} for an example).
An anchored planar algebra associates to each anchored planar tangle $T$ a morphism in $\cC$:
\begin{equation*}
Z\left(
\begin{tikzpicture}[baseline =-.1cm]
	\coordinate (a) at (0,0);
	\coordinate (b) at ($ (a) + (1.4,1) $);
	\coordinate (c) at ($ (a) + (.6,-.6) $);
	\coordinate (d) at ($ (a) + (-.6,.6) $);
	\coordinate (e) at ($ (a) + (-.8,-.6) $);
	
	\ncircle{}{(a)}{1.6}{85}{}
	\draw[thick, red] (c) arc (0:-180:.4cm) arc (180:0:.7cm) .. controls ++(270:.15cm) and ++(45:.15cm) .. ($ (a) + (-45:1.45cm) $) arc (315:90:1.45cm) arc (-90:0:.125cm);
	\draw[thick, red] (d) .. controls ++(225:1.2cm) and ++(270:2.6cm) .. ($ (a) + (85:1.6) $);
		
	\draw (60:1.6cm) arc (150:300:.4cm);
	\draw ($ (c) + (0,.4) $) arc (0:90:.8cm);
	\draw ($ (c) + (-.4,0) $) circle (.25cm);
	\draw ($ (d) + (0,.88) $) -- (d) -- ($ (d) + (-.88,0) $);
	\draw ($ (c) + (0,-.88) $) -- (c) -- ($ (c) + (.88,0) $);
	\ncircle{unshaded}{(d)}{.4}{235}{}
	\ncircle{unshaded}{(c)}{.4}{235}{}
\end{tikzpicture}
\right)
\,\,:\,\,\,\mathcal{P}[3]\otimes \mathcal{P}[5] \,\to\, \mathcal{P}[6].
\end{equation*}
In general, this is a morphism
$Z(T) \in \Hom_\cC(\cP[k_1] \otimes \cdots \otimes \cP[k_r], \cP[k_0])$,
where the order of the tensor factors $\cP[k_1], \ldots, \cP[k_r]$ is determined by the anchor lines (the red lines in the picture).

We recall that a tensor category is called \emph{pivotal} if every object is equipped with an isomorphism to its double dual, satisfying certain axioms.
There is a well-known algebraic classification of planar algebras \cite{MR2559686,1207.1923,MR3405915} which goes as follows:
\begin{equation}
\label{eq:PlanarAlgebraClassification}
\left\{\,\text{\rm Planar algebras}\left\}
\,\,\,\,\longleftrightarrow\,\,
\left\{\,\parbox{8.3cm}{\rm Pairs $(\mathcal{D},X)$, where $\cD$ is a pivotal category and $X\in\cD$ is a symmetrically self-dual generator}\,\right\}\right.\right.\!\!.
\end{equation}
Here, an object $X$ is called symmetrically self-dual if it is equipped with an isomorphism to its dual, subject to a certain symmetry condition.
For a pair $(\cD, X)$ as above, the $k$-th box space of the associated planar algebra $\cP$ is given by the invariants in $X^{\otimes k}$:
\begin{equation*}
\cP[k]:=\Hom_\cD(1_\cD, X^{\otimes k}).
\end{equation*}
Conversely, the action of planar tangles is given by the graphical calculus of $\cD$.

The main goal of our paper is to generalise the correspondence \eqref{eq:PlanarAlgebraClassification} to the case of planar algebras internal to $\cC$, i.e., anchored planar algebras.
Our classification result is formulated as an equivalence of categories.
It establishes a correspondence between anchored planar algebras in $\cC$, and a certain type of {\it module tensor categories} for $\cC$ (Definition~\ref{def: def of ModTC}).

Here, a module tensor category is a simultaneous generalisation of the notion of $\cC$-module category and of the notion of tensor category (a monoidal category which is linear over some field).
First and foremost, a module tensor category $\cM$ is a tensor category.
In addition to being a tensor category, it comes equipped with a tensor functor $\Phi:\cC\to \cM$ which gives it with the structure of a $\cC$-module category: $c \cdot m = \Phi(c) \otimes m$.
Now, because $\cC$ is braided, there is another left action of $\cC$ on $\cM$, given by $c \cdot m = m \otimes \Phi(c)$.
In a module tensor category, those two actions are isomorphic, i.e., we are provided with isomorphisms $e_{\Phi(c), m}:\Phi(c) \otimes m\to m \otimes \Phi(c)$.
The latter are subject to the following three axioms:
\begin{align*}
e_{\Phi(c), x \otimes y} &= (\id_x \otimes e_{\Phi(c),y}) \circ (e_{\Phi(c),x} \otimes \id_y)
\\
e_{\Phi(c \otimes d),x} &= (e_{\Phi(c),x} \otimes \id_{\Phi(d)}) \circ (\id_{\Phi(c)} \otimes e_{\Phi(d),x})
\\
e_{\Phi(c),\Phi(d)} &= \Phi(\beta_{c,d})
\end{align*}
where $\beta$ is the braiding in $\cC$.
The above conditions can be re-packaged into the single data of a braided functor
$\Phi^{\scriptscriptstyle \cZ}: \cC \to \cZ(\cM)$ from the category $\cC$ into the Drinfel'd center of $\cM$.

We say that a module tensor category is \emph{pivotal} if both $\cM$ and $\Phi^{\scriptscriptstyle \cZ}$ are pivotal,
and \emph{pointed} if it comes equipped with a symmetrically self-dual object $m \in \cM$ which generates it as a module tensor category.
Our main theorem is:

\begin{thmalpha}\label{thm:EquivalenceOfCategories}
There is an equivalence of categories
\[
\left\{\,\text{\rm Anchored planar algebras in $\cC$}\left\}
\,\,\,\,\cong\,\,
\left\{\,\parbox{4.7cm}{\rm \centerline{Pointed pivotal module} \centerline{tensor categories over $\cC$}}\,\right\}\right.\right.\!\!
\]
\end{thmalpha}

We warn the reader that equipping the collection of all pointed pivotal module tensor categories with the structure of a category is not totally obvious (they are more naturally a $2$-category).
The precise version of our theorem is stated as Theorem~\ref{thm:EquivalenceOfCategories2} in the body of this paper.

Given a pointed pivotal module tensor category $(\cM,m)$,
the $n$-th box object of the associated anchored planar algebra is given by  the formula
\[
\cP[n]=\Tr_\cC(m^{\otimes n}),
\]
where $\Tr_\cC: \cM \to \cC$ is the right adjoint of $\Phi$.
If one removes the condition that the generator $m\in\cM$ is symmetrically self-dual, then one obtains a classification of {\it oriented} planar algebras (i.e., planar algebras where the strands are oriented); for simplicity we only consider the unoriented case. 
We note that, even when $\cC=\Vec$, our theorem yields a version of the equivalence \eqref{eq:PlanarAlgebraClassification} which is more precise than has previously appeared in the literature.

In our previous article \cite{1509.02937}, we showed that the functor $\Tr_\cC$ admits a `calculus of strings on tubes' (see Section~\ref{sec:TubeRelations} for an overview).
As a corollary of our main theorem, we can now prove that this calculus of strings on tubes is invariant under all 3-dimensional isotopies (Appendix~\ref{sec:TubeCalculus}).

Examples of anchored planar algebras have already appeared in the literature, in the work of Jaffe-Liu on parafermions and reflection positivity \cite{1602.02662,1602.02671}.
In their work, a notion of `planar para algebra' is presented, which is equivalent to that of an anchored planar algebra in the braided tensor category $\Vec(\bbZ/N\bbZ)$ of $\bbZ/N\bbZ$-graded vector spaces (Example \ref{ex:Planar para algebras}).
By our main theorem (Theorem~\ref{thm:EquivalenceOfCategories}), the notion of a planar para algebra is equivalent to that of a module tensor category over $\Vec(\bbZ/N\bbZ)$.
The parafermion planar para algebras constructed in \cite{1602.02662} then correspond, under the equivalence, to Tambara-Yamagami categories associated to $\bbZ/N\bbZ$.
They lie in the larger family of Tambara-Yamagami module tensor categories over $\Vec(A)$, where $A$ is an abelian group (Example \ref{ex:TambaraYamagami}).

In the other direction, the algebraic classification given in Theorem \ref{thm:EquivalenceOfCategories} allows us to construct many examples of anchored planar algebras, including examples from near group categories (Section \ref{sec: Near group examples}), and from $ADE$ module tensor categories over Temperley-Lieb-Jones categories (Section \ref{sec: TLJ examples}).
We explicitly compute the box objects $\cP[k]$ in all our examples.

Our paper is structured as follows. 
In Section~\ref{sec: Anchored planar algebras}, we review material on planar algebras and introduce the notion of anchored planar algebra.
In Section~\ref{sec: The main theorem}, we review the notion of module tensor category, and state our main theorem (Theorem~\ref{thm:EquivalenceOfCategories2}).
Using our theorem, we then provide a number of examples of anchored planar algebras.
In Section~\ref{sec:Constructing anchored planar algebras}, we explain how to construct anchored planar algebras via generators and relations.
In Section~\ref{sec:APAfromMTC}, we use the categorified trace associated to a module tensor category \cite{1509.02937} to construct a functor $\Lambda:\Mod_* \to \APA$ from the category of pointed $\cC$-module tensor categories to the category of anchored planar algebras in $\cC$.
In Section~\ref{sec:MTCfromAPA}, we construct a functor $\Delta: \APA \to \Mod_*$ going the other way.
Finally, in Section~\ref{sec:Equivalence} we complete the proof of our main theorem and show that the two functors $\Lambda$ and $\Delta$ witness an equivalence of categories.

\paragraph{Acknowledgements.}
The authors would like to thank
Bruce Bartlett,
Vaughan Jones,
Zhengwei Liu,
Scott Morrison,
Mathew Pugh,
Noah Snyder, and
Kevin Walker
for helpful discussions.

Andr\'e Henriques gratefully acknowledges the Leverhulme trust and the EPSRC grant ``Quantum Mathematics and Computation'' for financing his visiting position in Oxford.
Andr\'e Henriques has received funding from the European Research Council (ERC) under the European Union's Horizon 2020 research and innovation programme (grant agreement No 674978).
David Penneys was supported by the NSF DMS grant 1500387.
James Tener would like to thank the Max Planck Institute for Mathematics for support during the course of this work.
David Penneys and James Tener were partially supported by NSF DMS grant 0856316.

\settocdepth{subsection}


\section{Anchored planar algebras}
\label{sec: Anchored planar algebras}

\subsection{Planar algebras}

Planar algebras were first introduced by Jones \cite{math.QA/9909027} in the context of his program to understand and classify subfactors
\cite{MR2679382,MR2979509,MR3166042,1509.00038}.
More generally, in the spirit of the tangle hypothesis \cite{MR2555928}, 
one can expect objects similar to planar algebras to be relevant whenever dealing with dualizable 1-morphisms in a 2-category.
We start with a brief definition, following \cite{MR2559686,MR3405915}.\medskip

Let $\mathbb D:=\{z\in\mathbb C:|z|\le 1\}$ be the closed unit disc.
We fix a base point $q \in \partial \mathbb D$ on its boundary.
For each $n\in\mathbb N$, we also fix a set $\mu_n=(\mu_n^{1},\ldots,\mu_n^{n})$ consisting of $n$ germs of arcs that intersect $\partial \mathbb D$ transversely, not at $q$:
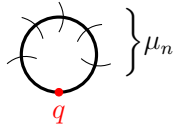
\[
\tikz{
	\ncircle{}{(0,0)}{.5}{-90}{}
\draw (-10:.3) to[bend right=30] (-10:.7);
\draw (35:.3) to[bend left=30] (35:.7);
\draw (90:.3) to[bend right=30] (90:.7);
\draw (135:.3) to[bend right=30] (135:.7);
\draw (190:.3) to[bend right=30] (190:.7);
\node[red, scale=.9] at (-90:.8) {$q$};
\node[scale=.9] at (8:1.2) {$\bigg\}\mu_n$};
}
\]

\begin{defn}\label{def:planar tangle}
A \emph{planar tangle} $T=(T,X,Q)$ consists of:
\begin{itemize}
\item
A subset $T\subseteq\mathbb D$ of the form
\[
\qquad T=\mathbb D\setminus (\mathring D_1\cup \ldots \cup \mathring D_r),\qquad r\ge 0,
\]
where $D_i=\{z\in\mathbb D:|z-a_i| \le r_i\}\subset \mathring {\mathbb D}$ are disjoint closed discs contained in the interior of $\mathbb D$.
We let $D_0:=\mathbb D$ and write $\partial^i T:=\partial D_i$, so that $$\partial T\,=\,\bigcup_{i=0}^r\partial^iT.$$
The boundary components $\partial^1T,\ldots,\partial^rT$ are called \emph{input circles} (the input circles are ordered), and $\partial^0T=\partial\mathbb D$ is called the \emph{output circle} of $T$.
\item
A closed $1$-dimensional submanifold $X\subset T$ satisfying $\partial X = X\cap \partial T$.
The components of $X$ are called \emph{strands} and the points of $\partial X$ are called \emph{boundary points}.
\item
A set $Q=\{q_0,\ldots q_r\}$, $q_0=q$,
$q_i\in\partial^i T\setminus X$, called the \emph{anchor points}.
\end{itemize} 
If $X\cap \partial^iT$ has cardinality $k$ and $f:\mathbb C\to \mathbb C$ is the unique affine linear map that sends $\partial \mathbb D$ to $\partial^iT$ and $q$ to $q_i$ ($f$ is the identity if $i=0$),
we demand that the strands $X$ be tangent to $f(\mu_k)$ to infinite order.
\end{defn}

A planar tangle $T$ is \emph{of type} $(k_1,\ldots,k_r;k_0)$ if it has $r$ input circles, and for each $i\in\{0,\ldots,r\}$ we have $|X\cap\partial^iT|=k_i$.
The collection of all tangles (for fixed $q$ and $\mu_n$) admits the structure of a \emph{topological coloured operad} \cite{MR0420609}. 
This amounts to the following structure.
If $S=(S,Y,P)$ and $T=(T,X,Q)$ are tangles of types $(k_1,\ldots,k_r;k_0)$ and $(\ell_1,\ldots,\ell_s;\ell_0)$,
and if $\ell_0=k_i$ for some $i\ge 1$, then we can form a new tangle
\[
T\circ_i S\,\,\,\,\,
\text{of type}\,\,\,\,\, (k_1,\ldots,k_{i-1},\ell_1,\ldots, \ell_s,k_{i+1}\ldots,k_r\,;\,k_0),
\]
called the \emph{$i$-th operadic composition} of $T$ and $S$.
This tangle is defined as follows.
Letting $f:\mathbb C\to \mathbb C$ be the unique affine linear map that sends $\partial^0S$ to $\partial^iT$ and $q$ to $q_i$,
it is given by
\begin{equation}\label{eq: compose tangles no anchor lines}
T\circ_i S \,:=\, \big(T\cup f(S),X\cup f(Y), Q\cup f(P)\setminus \{q_i\}\big).
\end{equation}
Our conventions about the behaviour of strands near their boundary ensure that this operation is well defined, and in particular that $X\cup f(Y)$ is a smooth manifold.

\begin{ex}
Here is an instance of operadic composition of tangles:
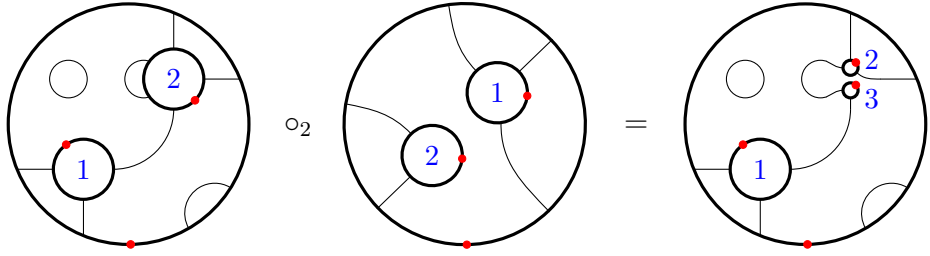
\label{ex:TangleComposition}
$$
\begin{tikzpicture}[baseline =-.1cm, yscale = -1]
	\coordinate (a) at (0,0);
	\coordinate (b) at ($ (a) + (1.4,1) $);
	\coordinate (c) at ($ (a) + (.6,-.6) $);
	\coordinate (d) at ($ (a) + (-.6,.6) $);
	\coordinate (e) at ($ (a) + (-.8,-.6) $);
	\ncircle{}{(a)}{1.6}{89}{}
	\draw (60:1.6cm) arc (150:300:.4cm);
	\draw ($ (c) + (0,.4) $) arc (0:90:.8cm);
	\draw ($ (c) + (-.4,0) $) circle (.25cm);
	\draw ($ (d) + (0,.88) $) -- (d) -- ($ (d) + (-.88,0) $);
	\draw ($ (c) + (0,-.88) $) -- (c) -- ($ (c) + (.88,0) $);
	\draw (e) circle (.25cm);
	\ncircle{unshaded}{(d)}{.4}{235}{}
	\ncircle{unshaded}{(c)}{.4}{225+180}{}
	\node[blue] at (c) {\small 2};
	\node[blue] at (d) {\small 1};
\end{tikzpicture}
\,\,\,\,\,\circ_2\,\,
\begin{tikzpicture}[baseline =-.16cm, yscale = -1]
\pgftransformscale{4.01}
\pgftransformrotate{46}
	\draw[very thick] (.6,-.6) circle (.4);
	\filldraw[red] (.6,-.6) + (225+178:.4) circle (.012); 
	\draw[very thick] (.6,-.75) circle (.1);
	\filldraw[red] (.6,-.75) + (-40:.1) circle (.012); 
	\draw[very thick] (.6,-.45) circle (.1);
	\filldraw[red] (.6,-.45) + (-40:.1) circle (.012); 
	\draw (.6,-1.005) -- (.6,-.85) (.6,-.2+.005) -- (.6,-.35);
	\draw (.68,-.69) to[in=180, out=45] (1.005,-.6);
	\draw (.275,-.839) to[in=180, out=35] (.5,-.75);
	\draw (.275,-.6+0.239) to[in=180, out=-35] (.5,-.45);
	\node[blue] at (.6,-.75) {\small 1};
	\node[blue] at (.6,-.45) {\small 2};
\end{tikzpicture}
\,\,\,=\,\,\,\,\,
\begin{tikzpicture}[baseline =-.1cm, yscale = -1]
	\coordinate (a) at (0,0);
	\coordinate (b) at ($ (a) + (1.4,1) $);
	\coordinate (c) at ($ (a) + (.6,-.6) $);
	\coordinate (d) at ($ (a) + (-.6,.6) $);
	\coordinate (e) at ($ (a) + (-.8,-.6) $);
	\ncircle{}{(a)}{1.6}{89}{}
	\draw (60:1.6cm) arc (150:300:.4cm);
	\draw ($ (c) + (0,.4) $) arc (0:90:.8cm);
	\draw ($ (c) + (-.4,0) $) circle (.25cm);
	\draw ($ (d) + (0,.88) $) -- (d) -- ($ (d) + (-.88,0) $);
	\draw ($ (c) + (0,-.88) $) -- (c) -- ($ (c) + (.88,0) $);
	\draw (e) circle (.25cm);
	\ncircle{unshaded}{(d)}{.4}{235}{}
	\fill[white] (c) circle (.4);
	\ncircle{unshaded}{(.6,-.75)}{.1}{-45}{}
	\ncircle{unshaded}{(.6,-.45)}{.1}{-45}{}
	\draw (.6,-1.005) -- (.6,-.85) (.6,-.2+.005) -- (.6,-.35);
	\draw (.68,-.69) to[in=180, out=45] (1.005,-.6);
	\draw (.275,-.839) to[in=180, out=25] (.5,-.75);
	\draw (.275,-.6+0.239) to[in=180, out=-25] (.5,-.45);
	\node[blue, xshift=8, yshift=3] at (.6,-.75) {\small 2};
	\node[blue, xshift=8, yshift=-3] at (.6,-.45) {\small 3};
	\node[blue] at (d) {\small 1};
\end{tikzpicture}
$$
\end{ex}

The \emph{planar operad} is the coloured operad with set of colours given by the natural numbers,
and operations of type $(k_1,\ldots,k_r;k_0)$ given by isotopy classes of planar tangles of type $(k_1,\ldots,k_r;k_0)$.

We note that the planar operad is independent, up to canonical isomorphism, of the choices of basepoint $q\in \partial \mathbb D$ and germs of arcs $\mu_n$.

A \emph{planar algebra} is an algebra over the planar operad \cite{math.QA/9909027}:

\begin{defn}
A planar algebra in $\Vec$ is a sequence of vector spaces $\mathcal{P}=(\mathcal{P}[n])_{n\geq 0}$ together with an action of the planar operad, i.e., for every isotopy class of planar tangle $T$ 
of type $(k_1,\ldots,k_r;k_0)$, a linear map 
$Z(T):\mathcal{P}[k_1]\otimes \cdots \otimes \mathcal{P}[k_r]\to \mathcal{P}[k_0]$. 
For example,
$$
Z\left(
\begin{tikzpicture}[baseline =-.1cm, yscale = -1]
	\coordinate (a) at (0,0);
	\coordinate (b) at ($ (a) + (1.4,1) $);
	\coordinate (c) at ($ (a) + (.6,-.6) $);
	\coordinate (d) at ($ (a) + (-.6,.6) $);
	\coordinate (e) at ($ (a) + (-.8,-.6) $);
	\ncircle{}{(a)}{1.6}{88}{}
	\draw (60:1.6cm) arc (150:300:.4cm);
	\draw ($ (c) + (0,.4) $) arc (0:90:.8cm);
	\draw ($ (c) + (-.4,0) $) circle (.25cm);
	\draw ($ (d) + (0,.88) $) -- (d) -- ($ (d) + (-.88,0) $);
	\draw ($ (c) + (0,-.88) $) -- (c) -- ($ (c) + (.88,0) $);
	\draw (e) circle (.25cm);
	\ncircle{unshaded}{(d)}{.4}{235}{}
	\ncircle{unshaded}{(c)}{.4}{235}{}
	\node[blue] at (c) {\small 2};
	\node[blue] at (d) {\small 1};
\end{tikzpicture}
\right)
\,\,:\,\,\,\mathcal{P}[3]\otimes \mathcal{P}[5] \,\to\, \mathcal{P}[6].
$$
This data must satisfy the following axioms:
\begin{itemize}
\item 
(identity) the identity tangle (which only has radial strings and no rotation between marked points) acts as the identity transformation.
$$
Z\left(
\begin{tikzpicture}[baseline=-.1cm]
	\ncircle{unshaded}{(0,0)}{.7}{-90}{}
	\ncircle{unshaded}{(0,0)}{.2}{-90}{}
	\draw (90:.2cm) -- (90:.7cm);
	\node at (105:.45cm) {\scriptsize{$n$}};
\end{tikzpicture}
\right)
=
\id_{\cP[n]}
$$
Here, we draw an $n$ next to a string to indicate $n$ parallel strings.
\item 
(composition) operadic composition of tangles corresponds to composition of maps.
Specifically, if $S$ is of type $(k_1,\ldots,k_r;k_0)$, $T$ is of type $(\ell_1,\ldots,\ell_s;\ell_0)$, and $\ell_0=k_i$, then
\[
Z(T\circ_i S)=Z(T)\circ (\id_{\cP[k_1]\otimes\ldots\otimes \cP[k_{i-1}]}\otimes Z(S)\otimes\id_{\cP[k_{i+1}]\otimes\ldots\otimes \cP[k_r]}).
\]
For example, the composite tangle $T\circ_2 S$ from Example \ref{ex:TangleComposition} yields the 
equation
$$
Z(T\circ_2 S)= Z(T)\circ (\id_{\mathcal P[3]}\otimes Z(S))\,:\, \mathcal P[3]\otimes \mathcal P[3]\otimes \mathcal P[2]\,\to\,\mathcal P[6]
$$
\item(symmetry) renumbering the input discs amounts to precomposing by a permutation.
Namely, if $T$ is of type $(k_1,\ldots,k_r;k_0)$, $\sigma\in \mathfrak S_r$ is a permutation, and $T^\sigma$ is the tangle of type 
$(k_{\sigma(1)},\ldots,k_{\sigma(r)};k_0)$ obtained by renumbering the input circles of $T$, then
\[
Z(T^\sigma)=Z(T)\circ \sigma:\mathcal{P}[k_{\sigma(1)}]\otimes \cdots \otimes \mathcal{P}[k_{\sigma(r)}]\to\mathcal{P}[k_1]\otimes \cdots \otimes \mathcal{P}[k_r] \to \mathcal{P}[k_0].
\]
\end{itemize}
\end{defn}

\noindent
The vector space $\cP[n]$ is called the $n$\textsuperscript{th} \emph{box space} of the planar algebra $\cP$.
A planar algebra is called \emph{connected} if its zeroth box space $\cP[0]$ is one dimensional.
In that case, the map
$$
Z\left(\begin{tikzpicture}[baseline = -.1cm]
	\draw[very thick] (0,0) circle (.4cm);
	\filldraw[red] (0,-.4) circle (.05cm);
\end{tikzpicture}\right):\bbC\to \cP[0]
$$
is necessarily an isomorphism.
If $\cP$ is a connected planar algebra, then closed loops count for a multiplicative factor $\delta$, for some $\delta\in\bbC$.

Even though it has never been phrased in this precise form, the following theorem is essentially well known  \cite{MR2559686,1207.1923,MR3405915}:

\begin{thm}
There is an equivalence of categories\;\!\footnote{Strictly speaking, such pairs $(\mathcal{C},X)$ form a $2$-category. However, this $2$-category is equivalent to a usual category.
See Lemma~\ref{lem: it's secretly a 1-category} and the paragraph thereafter for a discussion of this subtle point.}
\[
\left\{\,\text{\rm Planar algebras}\left\}
\,\,\,\,\longleftrightarrow\,\,
\left\{\,\parbox{8.3cm}{\rm Pairs $(\mathcal{C},X)$, where $\cC$ is a pivotal category and $X\in\cC$ is a symmetrically self-dual generator}\,\right\}\right.\right.\!\!,
\]
where for an object $X$ to generate a tensor category means that every other object is isomorphic to a direct summand of a direct sum of powers of $X$.
\end{thm}

The main result of our paper (Theorem \ref{thm:EquivalenceOfCategories}) is a generalisation of the above result.

\subsection{Anchored planar algebras}\label{sec:DefineAnchored}

Anchored planar tangles are a variant of planar tangles where we add \emph{anchor lines} connecting the anchor points.
Anchored planar algebras are very similar to planar algebras, with the notable difference that the symmetry axiom gets replaced by a braiding axiom.

\begin{defn}
\label{defn:AnchoredPlanarTangle}
An \emph{anchored planar tangle} is a quadruple $T=(T,X,Q,A)$, where $(T,X,Q)$ is a planar tangle (Definition \ref{def:planar tangle}), and $A$ is a system of anchor lines.
Here, a system of anchor lines is a collection of paths in $T$ that connect the output anchor point $q_0$ to the input anchor points $q_i$ and which,
unlike the strings, are only considered up to isotopy.
The anchor lines (always drawn in red in our pictures) may intersect the strings, but should not intersect each other.
We also insist that the numbering of the input circles corresponds to the clockwise numbering of the anchor lines as they approach $q_0$ (indicated by the little blue arrow in the picture):
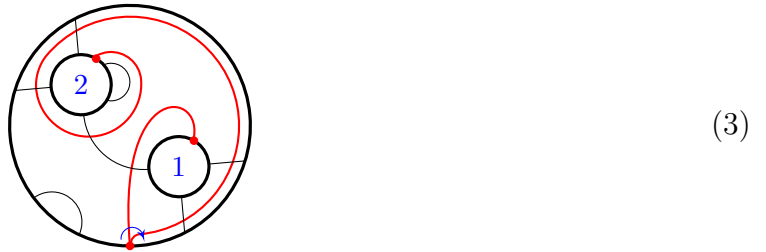
\begin{equation}\label{aflfmwbjsnnd}
\begin{matrix}
\begin{tikzpicture}[baseline =-.1cm]
\pgftransformrotate{185}
	\coordinate (a) at (0,0);
	\coordinate (b) at ($ (a) + (1.4,1) $);
	\coordinate (c) at ($ (a) + (.6,-.6) $);
	\coordinate (d) at ($ (a) + (-.6,.6) $);
	\coordinate (e) at ($ (a) + (-.8,-.6) $);
	
	\ncircle{}{(a)}{1.6}{85}{}
	\draw[thick, red] (c) arc (0:-180:.4cm) arc (180:0:.7cm) .. controls ++(270:.15cm) and ++(45:.15cm) .. ($ (a) + (-45:1.45cm) $) arc (315:90:1.45cm) arc (-90:0:.125cm);
	\draw[thick, red] (d) .. controls ++(225:1.2cm) and ++(270:2.6cm) .. ($ (a) + (85:1.6) $);
		
	\draw (60:1.6cm) arc (150:300:.4cm);
	\draw ($ (c) + (0,.4) $) arc (0:90:.8cm);
	\draw ($ (c) + (-.4,0) $) circle (.25cm);
	\draw ($ (d) + (0,.88) $) -- (d) -- ($ (d) + (-.88,0) $);
	\draw ($ (c) + (0,-.88) $) -- (c) -- ($ (c) + (.88,0) $);
	\ncircle{unshaded}{(d)}{.4}{235}{}
	\ncircle{unshaded}{(c)}{.4}{235}{}
	\node[blue] at (c) {\small 2};
	\node[blue] at (d) {\small 1};
\draw[blue,-stealth] (.25,1.5) arc(0:-183:.15) -- ++(105:.001);
\end{tikzpicture}
\end{matrix}
\end{equation}
This means that, when considering anchored planar tangles, we may drop the numbering of the input circles, as it is entirely determined by the topology of the anchor lines.
\end{defn}

Let $S$ and $T$ be anchored planar tangles of types $(k_1,\ldots,k_r;k_0)$ and $(\ell_1,\ldots,\ell_s;\ell_0)$, with systems of anchor lines $A_S$ and $A_T$.
If $\ell_0=k_i$ for some $i\ge 1$, then their $i$-th operadic composition $T\circ_iS$ is defined as follows.
The underlying planar tangle is as in (\ref{eq: compose tangles no anchor lines}).
It is then equipped with the system of anchor lines obtained
from $A_T\cup f(A_S)$ ($f$ as in (\ref{eq: compose tangles no anchor lines}))
by replacing the line that connects $q_0$ to $q_i$ with $s$ parallel lines.
We illustrate this process by an example:
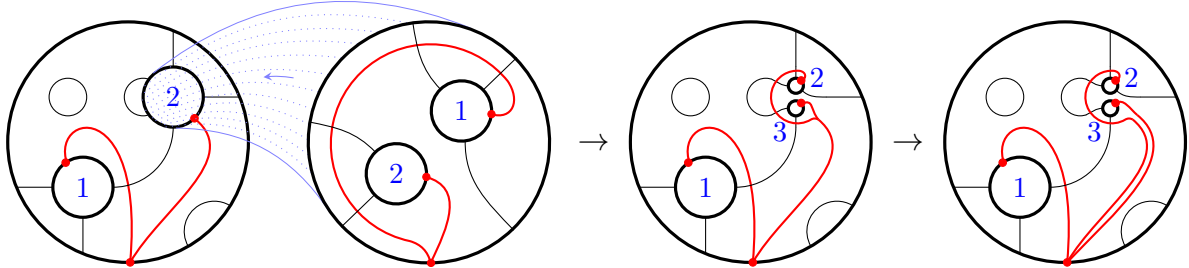
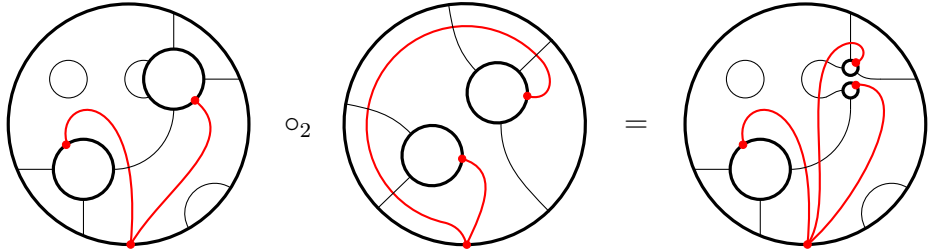
\[
\begin{tikzpicture}[baseline =-.1cm]
\pgftransformyscale{-1}
	\coordinate (a) at (0,0);
	\coordinate (b) at ($ (a) + (1.4,1) $);
	\coordinate (c) at ($ (a) + (.6,-.6) $);
	\coordinate (d) at ($ (a) + (-.6,.6) $);
	\coordinate (e) at ($ (a) + (-.8,-.6) $);
	\ncircle{}{(a)}{1.6}{89}{}
	\draw[thick, red] (d) .. controls ++(225:1.2cm) and ++(275:2.6cm) .. ($ (a) + (89:1.6) $);
	\draw[thick, red] ($ (a) + (89:1.6) $) to[out=-60, in=40] (.9,-.3);
	\draw (60:1.6cm) arc (150:300:.4cm);
	\draw ($ (c) + (0,.4) $) arc (0:90:.8cm);
	\draw ($ (c) + (-.4,0) $) circle (.25cm);
	\draw ($ (d) + (0,.88) $) -- (d) -- ($ (d) + (-.88,0) $);
	\draw ($ (c) + (0,-.88) $) -- (c) -- ($ (c) + (.88,0) $);
	\draw (e) circle (.25cm);
	\ncircle{unshaded}{(d)}{.4}{235}{}
	\ncircle{unshaded}{(c)}{.4}{225+180}{}
	\node[blue] at (c) {\small 2};
	\node[blue] at (d) {\small 1};
\pgftransformyscale{-1}
\foreach \x/\s in {2.1/25, 3.1/22, 3.85/21, 4.55/20, 5.25/21, 5.9/23, 6.5/25, 7/26, 7.5/26, 8/27, 8.6/27}
{\draw[dotted, blue!50] ($(4,0)+(70+14.5*\x:1.615)$) to[bend right=\s] ($(.6,.6) + (130+14.5*\x:.415)$);}
\foreach \x/\s in {0/30,10/28}
{\draw[blue!50, very thin] ($(4,0)+(70+14.5*\x:1.615)$) to[bend right=\s] ($(.6,.6) + (130+14.5*\x:.415)$);}
\draw[blue!50, -stealth] (2.2,.855) to[bend right=7] (1.8,.865);
\pgftransformyscale{-1}
\pgftransformyshift{-1.6}
\pgftransformxshift{17}
\pgftransformscale{4.01}
\pgftransformrotate{46}
	\draw[very thick] (.6,-.6) circle (.4);
	\filldraw[red] (.6,-.6) + (225+178:.4) coordinate(CIRC1) circle (.012); 
	\draw[very thick] (.6,-.75) circle (.1);
	\filldraw[red] (.6,-.75) + (-40:.1) coordinate(CIRC2) circle (.012); 
	\draw[very thick] (.6,-.45) circle (.1);
	\filldraw[red] (.6,-.45) + (-40:.1) coordinate(CIRC3) circle (.012); 
	\draw[thick, red] (CIRC1) to[out=-150, in=-30] ++(180:.1) arc(56:275:.33) to[out=10, in=-30, looseness=1.6] (CIRC2);
	\draw[thick, red] (CIRC1) to[out=-110, in=-30, looseness=1.2] (CIRC3);
	\draw (.6,-1.005) -- (.6,-.85) (.6,-.2+.005) -- (.6,-.35);
	\draw (.68,-.69) to[in=180, out=45] (1.005,-.6);
	\draw (.275,-.839) to[in=180, out=35] (.5,-.75);
	\draw (.275,-.6+0.239) to[in=180, out=-35] (.5,-.45);
	\node[blue] at (.6,-.75) {\small 1};
	\node[blue] at (.6,-.45) {\small 2};
\end{tikzpicture}
\,\to\,\,
\begin{tikzpicture}[baseline =-.1cm, yscale = -1]
	\coordinate (a) at (0,0);
	\coordinate (b) at ($ (a) + (1.4,1) $);
	\coordinate (c) at ($ (a) + (.6,-.6) $);
	\coordinate (d) at ($ (a) + (-.6,.6) $);
	\coordinate (e) at ($ (a) + (-.8,-.6) $);
	\ncircle{}{(a)}{1.6}{89}{}
	\draw[thick, red] (d) .. controls ++(225:1.2cm) and ++(275:2.6cm) .. ($ (a) + (89:1.6) $);
	\draw[thick, red] ($ (a) + (89:1.6) $) to[out=-60, in=40] (.9,-.3) coordinate(CIRC1);
	\draw (60:1.6cm) arc (150:300:.4cm);
	\draw ($ (c) + (0,.4) $) arc (0:90:.8cm);
	\draw ($ (c) + (-.4,0) $) circle (.25cm);
	\draw ($ (d) + (0,.88) $) -- (d) -- ($ (d) + (-.88,0) $);
	\draw ($ (c) + (0,-.88) $) -- (c) -- ($ (c) + (.88,0) $);
	\draw (e) circle (.25cm);
	\ncircle{unshaded}{(d)}{.4}{235}{}
	\fill[white] (c) circle (.4);
	\ncircle{unshaded}{(.6,-.75)}{.1}{-45}{}
	\ncircle{unshaded}{(.6,-.45)}{.1}{-45}{}
	\draw (.6,-1.005) -- (.6,-.85) (.6,-.2+.005) -- (.6,-.35);
	\draw (.68,-.69) to[in=180, out=45] (1.005,-.6);
	\draw (.275,-.839) to[in=180, out=25] (.5,-.75);
	\draw (.275,-.6+0.239) to[in=180, out=-25] (.5,-.45);
	\node[blue, xshift=8, yshift=3] at (.6,-.75) {\small 2};
	\node[blue, xshift=-6, yshift=-9] at (.6,-.45) {\small 3};
	\node[blue] at (d) {\small 1};
	\path (.6,-.75) + (-40:.1) coordinate(CIRC2); 
	\path (.6,-.45) + (-40:.1) coordinate(CIRC3); 
	\draw[thick, red] (CIRC1) to[out=-150, in=-30] ++(180:.1) arc(59:275:.35) to[out=10, in=-30, looseness=1.6] (CIRC2);
	\draw[thick, red] (CIRC1) to[out=-130, in=-20, looseness=1.5] (CIRC3);
\end{tikzpicture}
\,\,\to\,\,
\begin{tikzpicture}[baseline =-.1cm, yscale = -1]
	\coordinate (a) at (0,0);
	\coordinate (b) at ($ (a) + (1.4,1) $);
	\coordinate (c) at ($ (a) + (.6,-.6) $);
	\coordinate (d) at ($ (a) + (-.6,.6) $);
	\coordinate (e) at ($ (a) + (-.8,-.6) $);
	\ncircle{}{(a)}{1.6}{89}{}
	\draw[thick, red] (d) .. controls ++(225:1.2cm) and ++(275:2.6cm) .. ($ (a) + (89:1.6) $);
	\draw[thick, red] ($ (a) + (89:1.6) $) to[out=-60, in=40] (.9,-.3) coordinate(CIRC1);
	\draw[thick, red] ($ (a) + (89:1.6) $) to[out=-47, in=39] (.96,-.36) coordinate(CIRC1');
	\draw (60:1.6cm) arc (150:300:.4cm);
	\draw ($ (c) + (0,.4) $) arc (0:90:.8cm);
	\draw ($ (c) + (-.4,0) $) circle (.25cm);
	\draw ($ (d) + (0,.88) $) -- (d) -- ($ (d) + (-.88,0) $);
	\draw ($ (c) + (0,-.88) $) -- (c) -- ($ (c) + (.88,0) $);
	\draw (e) circle (.25cm);
	\ncircle{unshaded}{(d)}{.4}{235}{}
	\fill[white] (c) circle (.4);
	\ncircle{unshaded}{(.6,-.75)}{.1}{-45}{}
	\ncircle{unshaded}{(.6,-.45)}{.1}{-45}{}
	\draw (.6,-1.005) -- (.6,-.85) (.6,-.2+.005) -- (.6,-.35);
	\draw (.68,-.69) to[in=180, out=45] (1.005,-.6);
	\draw (.275,-.839) to[in=180, out=25] (.5,-.75);
	\draw (.275,-.6+0.239) to[in=180, out=-25] (.5,-.45);
	\node[blue, xshift=8, yshift=3] at (.6,-.75) {\small 2};
	\node[blue, xshift=-6, yshift=-9] at (.6,-.45) {\small 3};
	\node[blue] at (d) {\small 1};
	\path (.6,-.75) + (-40:.1) coordinate(CIRC2); 
	\path (.6,-.45) + (-40:.1) coordinate(CIRC3); 
	\draw[thick, red] (CIRC1) to[out=-150, in=-30] ++(180:.1) arc(59:275:.35) to[out=10, in=-30, looseness=1.6] (CIRC2);
	\draw[thick, red] (CIRC1') to[out=-140, in=-40] (CIRC3);
\end{tikzpicture}
\]
So we have:
$$
\begin{tikzpicture}[baseline =-.1cm, yscale = -1]
	\coordinate (a) at (0,0);
	\coordinate (b) at ($ (a) + (1.4,1) $);
	\coordinate (c) at ($ (a) + (.6,-.6) $);
	\coordinate (d) at ($ (a) + (-.6,.6) $);
	\coordinate (e) at ($ (a) + (-.8,-.6) $);
	\ncircle{}{(a)}{1.6}{89}{}
	\draw[thick, red] (d) .. controls ++(225:1.2cm) and ++(275:2.6cm) .. ($ (a) + (89:1.6) $);
	\draw[thick, red] ($ (a) + (89:1.6) $) to[out=-60, in=40] (.9,-.3);
	\draw (60:1.6cm) arc (150:300:.4cm);
	\draw ($ (c) + (0,.4) $) arc (0:90:.8cm);
	\draw ($ (c) + (-.4,0) $) circle (.25cm);
	\draw ($ (d) + (0,.88) $) -- (d) -- ($ (d) + (-.88,0) $);
	\draw ($ (c) + (0,-.88) $) -- (c) -- ($ (c) + (.88,0) $);
	\draw (e) circle (.25cm);
	\ncircle{unshaded}{(d)}{.4}{235}{}
	\ncircle{unshaded}{(c)}{.4}{225+180}{}
\end{tikzpicture}
\,\,\,\,\,\circ_2\,\,
\begin{tikzpicture}[baseline =-.16cm, yscale = -1]
\pgftransformscale{4.01}
\pgftransformrotate{46}
	\draw[very thick] (.6,-.6) circle (.4);
	\filldraw[red] (.6,-.6) + (225+178:.4) coordinate(CIRC1) circle (.012); 
	\draw[very thick] (.6,-.75) circle (.1);
	\filldraw[red] (.6,-.75) + (-40:.1) coordinate(CIRC2) circle (.012); 
	\draw[very thick] (.6,-.45) circle (.1);
	\filldraw[red] (.6,-.45) + (-40:.1) coordinate(CIRC3) circle (.012); 
	\draw[thick, red] (CIRC1) to[out=-150, in=-30] ++(180:.1) arc(56:275:.33) to[out=10, in=-30, looseness=1.6] (CIRC2);
	\draw[thick, red] (CIRC1) to[out=-110, in=-30, looseness=1.2] (CIRC3);
	\draw (.6,-1.005) -- (.6,-.85) (.6,-.2+.005) -- (.6,-.35);
	\draw (.68,-.69) to[in=180, out=45] (1.005,-.6);
	\draw (.275,-.839) to[in=180, out=35] (.5,-.75);
	\draw (.275,-.6+0.239) to[in=180, out=-35] (.5,-.45);
\end{tikzpicture}
\,\,\,=\,\,\,\,\,
\begin{tikzpicture}[baseline =-.1cm, yscale = -1]
	\coordinate (a) at (0,0);
	\coordinate (b) at ($ (a) + (1.4,1) $);
	\coordinate (c) at ($ (a) + (.6,-.6) $);
	\coordinate (d) at ($ (a) + (-.6,.6) $);
	\coordinate (e) at ($ (a) + (-.8,-.6) $);
	\ncircle{}{(a)}{1.6}{89}{}
	\draw[thick, red] (d) .. controls ++(225:1.2cm) and ++(275:2.6cm) .. ($ (a) + (89:1.6) $);
	\draw (60:1.6cm) arc (150:300:.4cm);
	\draw ($ (c) + (0,.4) $) arc (0:90:.8cm);
	\draw ($ (c) + (-.4,0) $) circle (.25cm);
	\draw ($ (d) + (0,.88) $) -- (d) -- ($ (d) + (-.88,0) $);
	\draw ($ (c) + (0,-.88) $) -- (c) -- ($ (c) + (.88,0) $);
	\draw (e) circle (.25cm);
	\ncircle{unshaded}{(d)}{.4}{235}{}
	\fill[white] (c) circle (.4);
	\ncircle{unshaded}{(.6,-.75)}{.1}{-45}{}
	\ncircle{unshaded}{(.6,-.45)}{.1}{-45}{}
	\draw (.6,-1.005) -- (.6,-.85) (.6,-.2+.005) -- (.6,-.35);
	\draw (.68,-.69) to[in=180, out=45] (1.005,-.6);
	\draw (.275,-.839) to[in=180, out=25] (.5,-.75);
	\draw (.275,-.6+0.239) to[in=180, out=-25] (.5,-.45);
	\path (.6,-.45) ++ (-40:.1) ++(0,-.015) coordinate(CIRC3); 
	\draw[thick, red] ($ (a) + (89:1.6) $) to[out=-40, in=0] (CIRC3);

	\path (.6,-.75) + (-40:.1) coordinate(CIRC2); 
	\draw[thick, red] ($ (a) + (89:1.6) $) to[out=-55, in=160, looseness=.7] ($(CIRC2) + (-.15,-.25)$) to[out=-20, in=-30, looseness=2] (CIRC2);
\end{tikzpicture}
$$

The collection of all isotopy classes of anchored planar tangles, together with the operation of anchored tangle composition
forms the \emph{anchored planar operad}.

Note that, unlike the planar operad, the anchored planar operad does not admit an action of the symmetric group $\mathfrak S_n$ on its set of $n$-ary operations.
It is therefore not an operad, but rather a non-$\Sigma$ operad \cite[Def.\,3.12]{MR0420610}.
It carries however actions of the ribbon braid groups (see Section \ref{sec:RibbonBraidGroup}), and so it is what one might call a colored \emph{ribbon braided operad}.
To our knowledge, this particular variant of the notion of operad has not appeared in the literature.
It is a straightforward modification of the notion of a braided operad \cite[Def.\,3.2]{SymmetricBar}.

As is well known, operads can act on objects of a symmetric monoidal category (such as $\Vec$), while
non-$\Sigma$ operads can act on the objects of a monoidal category.
On the other hand, braided operads are designed so as to be able to act on the objects of a braided monoidal category.
Correspondingly, ribbon braided operads can act on the objects of a braided pivotal category.

Recall that a pivotal category is a rigid tensor category (= tensor category with left and right duals) equipped with a monoidal natural isomorphism $\varphi_a:a\to a^{**}$ between the identity functor and the double dual functor.
\begin{defn}
A \emph{braided pivotal} category is a tensor category which is both braided and pivotal (no compatibility between the two structures);
this is a slight weakening of the notion of ribbon category.
\end{defn}
Given an object $a$ of a braided pivotal category, we define its twist $\theta_a: a\to a$ by the formula
\begin{equation}\label{def:theta1}
\theta_{a}
:= 
(\id_a\otimes \ev_{a^*})
\circ
(\beta_{a^{**},a}\otimes\id_{a^*})
\circ
(\id_{a^{**}}\otimes \coev_a)
\circ
\varphi_a
=
\begin{tikzpicture}[rotate=180, baseline=.35cm]
	\draw (0,-1.6) -- (0,.8);
	\loopIso{(0,-1)}
	\roundNbox{unshaded}{(0,0)}{.35}{0}{0}{$\varphi_a$}	
	\node at (-.2+.4,.6) {\scriptsize{$a$}};
	\node at (-.2+.4,-1.4) {\scriptsize{$a$}};
	\node at (-.3+.55,-.55) {\scriptsize{$a^{**}$}};
\end{tikzpicture}
\end{equation}
where $\beta_{a,b} : a\otimes b \to b\otimes a$ denotes the braiding.
It satifsies $\theta_{a\otimes b}=\beta_{b,a}\circ\beta_{a,b}\circ(\theta_a\otimes\theta_b)$.
We refer the reader to \cite[\S2]{1509.02937} for an extended discussion of braided pivotal categories.

An anchored planar algebra is an algebra over the anchored planar operad:

\begin{defn}\label{def: anchored planar algebra}
Let $\cC$ be a braided pivotal category.
An \emph{anchored planar algebra} in $\cC$ is a sequence $\cP=(\cP[n])_{n\ge 0}$ of objects of $\cC$, along with operations
$$Z(T):\cP[k_1]\otimes\ldots\otimes \cP[k_r]\to \cP[k_0]$$
for every isotopy class of anchored planar tangle $T$ of type $(k_1,\ldots,k_r;k_0)$,
subject to the following axioms:
\begin{itemize}
\item
(identity) the identity anchored tangle acts as the identity morphism
\item
(composition) if $S$ and $T$ are anchored planar tangles of type $(k_1,\ldots,k_r;k_0)$ and $(\ell_1,\ldots,\ell_s;\ell_0)$, and if $\ell_0=k_i$, then
\begin{equation}\label{eq: composition of tangles}
Z(T\circ_i S)=Z(T)\circ (\id_{\cP[k_1]\otimes\ldots\otimes \cP[k_{i-1}]}\otimes Z(S)\otimes\id_{\cP[k_{i+1}]\otimes\ldots\otimes \cP[k_r]})
\end{equation}
\item
(anchor dependence) the following relations hold:
\begin{itemize}
\item
(braiding)\hspace{.7cm}
$
Z\left(
\begin{tikzpicture}[baseline = -.1cm]
	\draw (-.6,-.2) -- (-.6,1);
	\draw (0,1) -- (0,.6);
	\draw (.6,-.2) -- (.6,1);
	\draw[thick, red] (0,-.6) -- (0,-1);
	\draw[thick, red] (0,.2) arc (0:-90:.2cm) -- (-.6,0) arc (90:270:.4cm) -- (-.2,-.8) arc (90:0:.2cm);
	\roundNbox{}{(0,0)}{1}{.2}{.2}{}
	\roundNbox{unshaded}{(0,-.4)}{.2}{.6}{.6}{}
	\roundNbox{unshaded}{(0,.4)}{.2}{.2}{.2}{}
	\node at (-.8,.8) {\scriptsize{$i$}};
	\node at (-.2,.8) {\scriptsize{$j$}};
	\node at (.8,.8) {\scriptsize{$k$}};
	\fill[red] (0,.2) circle (.05)  (0,-.6) circle (.05)  (0,-1) circle (.05);
\end{tikzpicture}
\right)
=
Z
\left(
\begin{tikzpicture}[baseline = -.1cm]
	\draw (-.6,-.2) -- (-.6,1);
	\draw (0,1) -- (0,.6);
	\draw (.6,-.2) -- (.6,1);
	\draw[thick, red] (0,-.6) -- (0,-1);
	\draw[thick, red] (0,.2) arc (180:270:.2cm) -- (.6,0) arc (90:-90:.4cm) -- (.2,-.8) arc (90:180:.2cm);
	\roundNbox{}{(0,0)}{1}{.2}{.2}{}
	\roundNbox{unshaded}{(0,-.4)}{.2}{.6}{.6}{}
	\roundNbox{unshaded}{(0,.4)}{.2}{.2}{.2}{}
	\node at (-.8,.8) {\scriptsize{$i$}};
	\node at (-.2,.8) {\scriptsize{$j$}};
	\node at (.8,.8) {\scriptsize{$k$}};
	\fill[red] (0,.2) circle (.05)  (0,-.6) circle (.05)  (0,-1) circle (.05);
\end{tikzpicture}
\right)
\,
\circ \beta_{\cP[j],\cP[i+k]}
$
\item
(twist)\hspace{1.3cm}
$
Z\left(
\begin{tikzpicture}[baseline=-.1cm]
	\ncircle{unshaded}{(0,0)}{1}{270}{}
	\ncircle{unshaded}{(0,0)}{.3}{270}{}
	\draw (90:.3cm) -- (90:1cm);
	\draw[thick, red] (-90:.3cm) .. controls ++(270:.3cm) and ++(270:.5cm) .. (0:.5cm) .. controls ++(90:.8cm) and ++(90:.8cm) .. (180:.7cm) .. controls ++(270:.6cm) and ++(90:.4cm) .. (270:1cm);
	\node at (100:.8cm) {\scriptsize{$n$}};
\end{tikzpicture}
\right)
=\theta_{\cP[n]}
$
\end{itemize}
\end{itemize}
(Recall that a little number $n$ next to a string to indicates $n$ parallel strings.)
\end{defn}

$\cP[n]$ is called the $n$\textsuperscript{th} \emph{box object} of the anchored planar algebra $\cP$.

\begin{rem}
For any anchored planar algebra in $\cC$, the anchor dependence axiom implies that $\cP[0]$ has trivial twist, since the following tangles are isotopic:
$$
\theta_{\cP[0]}
=
Z
\left(
\begin{tikzpicture}[baseline=-.1cm]
	\ncircle{unshaded}{(0,0)}{.7}{270}{}
	\draw[thick, red] (-90:.2cm) .. controls ++(270:.25cm) and ++(270:.4cm) .. (0:.4cm) .. controls ++(90:.5cm) and ++(90:.5cm) .. (180:.4cm) .. controls ++(270:.3cm) and ++(90:.2cm) .. (270:.7cm);
	\ncircle{unshaded}{(0,0)}{.2}{270}{}
\end{tikzpicture}
\right)
=
Z\left(
\begin{tikzpicture}[baseline=-.1cm]
	\ncircle{unshaded}{(0,0)}{.7}{270}{}
	\ncircle{unshaded}{(0,0)}{.2}{270}{}
	\draw[thick, red] (-90:.2cm) -- (270:.7cm);
\end{tikzpicture}
\right)
=
\id_{\cP[0]}.
$$
\end{rem}

\begin{rem}\label{rem:operadic composition of morphisms}
If we write
\[
f\circ_i g:=f\circ(\id_{b_1\otimes\ldots\otimes b_{i-1}}\otimes g\otimes \id_{b_{i+1}\otimes\ldots\otimes b_n}):b_1\otimes\ldots\otimes b_{i-1}\otimes a_1\otimes\ldots\otimes a_m\otimes b_{i+1}\otimes\ldots\otimes b_n\to c
\]
for the operadic composition of morphisms $f:b_1\otimes\ldots\otimes b_n\to c$ and $g:a_1\otimes\ldots\otimes a_m\to b_i$ in $\cC$,
then axiom (\ref{eq: composition of tangles}) can be elegantly rephrased as
\[
Z(T\circ_i S)=Z(T)\circ_i Z(S).
\]
\end{rem}

\subsection{Generators for the anchored planar operad}
\label{sec:AnchoredTangleGeneratorsAndRelations}

In this section, we present a collection of elements of the anchored planar operad, along with a number of relations that they satisfy.
The tangles that we present below generate the anchored planar algebra operad as a non-$\Sigma$ operad, but we will not need this result.
What is more relevant is that these tangles generate the anchored planar algebra operad as a ribbon braided operad,
a fact which will become evident from the role that they will play further down, in Algorithm \ref{alg:AssignMap}.
However, as was have decided not to discuss ribbon braided operads, we will not attempt to formalize this statement.

Throughout this section, the word `tangle' means `isotopy class of anchored planar tangle':

\begin{defn}[generating tangles]
\label{defn:GeneratingTangles}
\be
\item
The unit tangle $u$ has no input disks and no strings:\,\,
$
u = \,
\begin{tikzpicture}[baseline = -.1cm, scale=.8]
	\draw[very thick] (0,0) circle (.4cm);
	\filldraw[red] (0,-.4) circle (.05cm);
\end{tikzpicture}
$\,.

\item
For every $n\in\bbN_{\geq 0}$, the identity tangle $\id_n$ has one input disk, $n$ radial strings, and a radial anchor line:\,\,
$
\id_n=\,
\begin{tikzpicture}[baseline=-.1cm, scale=.8]
	\ncircle{unshaded}{(0,0)}{.7}{270}{}
	\ncircle{unshaded}{(0,0)}{.25}{270}{}
	\draw (90:.25cm) -- (90:.7cm);
	\draw[thick, red] (270:.25cm) -- (270:.7cm);
	\node at (109:.477cm) {\scriptsize{$n$}};
\end{tikzpicture}
$\,.

\item
For every $n\in\bbN_{\geq 0}$ and every $0\leq i\leq n$,  we let $a_i$ be the tangle of type $(n+2;n)$ which caps the $(i+1)$-st and $(i+2)$-nd strings of the input circle,
and we let $\bar a_i$ be tangle of type $(n;n+2)$ which caps the $(i+1)$-st and $(i+2)$-nd strings of the output circle.

\item
Finally, for every $n\in\bbN_{\geq 0}$ and every $0\leq i\leq n$ and $j\ge 0$, we let $p_{i,j}$ be the tangle of type $(n,j;n+j)$ drawn here below:\vspace{-.2cm}
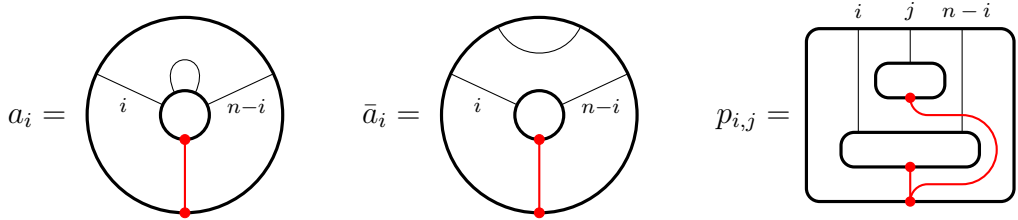
$$
a_i=\,\,
\begin{tikzpicture}[baseline = -.1cm, scale=1.3]
	\ncircle{unshaded}{(0,0)}{1}{270}{}
	\ncircle{unshaded}{(0,0)}{.25}{270}{}
	\draw[thick, red] (-90:1cm) -- (-90:.25cm);
	\draw (115:.25cm) .. controls ++(115:.5cm) and ++(65:.5cm) .. (65:.25cm);
	\draw (155:.25cm) -- (155:1cm);
	\draw (25:.25cm) -- (25:1cm);
\node at (-.63,.1) {$\scriptstyle i$};
\node at (.63,.08) {$\scriptstyle n-i$};
\end{tikzpicture}
\qquad\,\,\,
\bar a_i=\,\,
\begin{tikzpicture}[baseline = -.1cm, scale=1.3]
	\ncircle{unshaded}{(0,0)}{1}{270}{}
	\ncircle{unshaded}{(0,0)}{.25}{270}{}
	\draw[thick, red] (-90:1cm) -- (-90:.25cm);
	\draw (115:1cm) .. controls ++(-65:.4cm) and ++(245:.4cm) .. (65:1cm);
	\draw (155:.25cm) -- (155:1cm);
	\draw (25:.25cm) -- (25:1cm);
\node at (-.63,.1) {$\scriptstyle i$};
\node at (.63,.08) {$\scriptstyle n-i$};
\end{tikzpicture}
\qquad\,\,\,
p_{i,j}=\,
\begin{tikzpicture}[baseline = -.1cm, scale=1.15]
	\draw (-.6,-.2) -- (-.6,1);
	\node at (-.6,1.2) {\scriptsize{$i$}};
	\draw (0,1) -- (0,.6);
	\node at (0,1.2) {\scriptsize{$j$}};
	\draw (.6,-.2) -- (.6,1);
	\node at (.65,1.2) {\scriptsize{$n-i$}};
	\draw[thick, red] (0,-.6) -- (0,-1);
	\draw[thick, red] (0,.2) arc (180:270:.2cm) -- (.6,0) arc (90:-90:.4cm) -- (.2,-.8) arc (90:180:.2cm);
	\roundNbox{}{(0,0)}{1}{.2}{.2}{}
	\roundNbox{unshaded}{(0,-.4)}{.2}{.6}{.6}{}
	\roundNbox{unshaded}{(0,.4)}{.2}{.2}{.2}{}
	\fill[red] (0,.2) circle (.06)  (0,-.6) circle (.06)  (0,-1) circle (.06);
\end{tikzpicture}\qquad
\vspace{.2cm}
$$
\ee
\end{defn}

We now list a number of relations satisfied by the above tangles, similar to the ones which appear in \cite{math.QA/9909027,MR2903179}:

\begin{prop}\label{prop:AnchoredRelations}
The following relations hold amongst the tangles defined above.
Here, we write $\circ$ in place of $\circ_1$ when there is no possible confusion.
\begin{enumerate}[label={\rm(A\arabic*)}]
\item
\label{reln:id}
$p_{0,j}\circ_1 u = \id_j$,\, $p_{i,0}\circ_2 u = \id_n$
\item
\label{reln:Caps}
$a_i \circ a_j = a_{j-2}\circ a_i$ for $i+1< j$
\item
\label{reln:Cups}
$\bar a_i \circ \bar a_j = \bar a_{j+2}\circ \bar a_i$ for $i\leq j$
\\
\item
\label{reln:CapCup}
$\displaystyle a_i \circ \bar a_j = 
\begin{cases}
\bar a_{j-2}\circ a_i & \text{for } i<j-1
\\
\id_n & \text{for } i=j\pm1
\\
\delta \id_n & \text{for } i=j\qquad\,\,\,\, \text{where $\delta$ refers to a free-floating circular strand}
\\
\bar a_j \circ a_{i-2} & \text{for } i>j+1
\end{cases}
$
\item
\label{reln:CapQuadratic}
$\displaystyle
a_i \circ p_{j,k} = 
\begin{cases}
p_{j-2,k} \circ_1 a_{i} & \text{for }  i+1<j 
\\
p_{j,k-2}\circ_2 a_{i-j} &\text{for }j<i+1<j+k
\\
p_{j,k} \circ_1 a_{i-k} & \text{for } i+1> j+k
\end{cases}
$
\\
\item
\label{reln:CupQuadratic}
$\displaystyle
\bar a_i \circ p_{j,k} = 
\begin{cases}
p_{j+2,k}\circ_1 \bar a_i & \text{for } i\leq j
\\
p_{j,k+2} \circ_2 \bar a_{i-j} & \text{for } j\leq i\leq j+k
\\
p_{j,k} \circ_1 \bar a_{i-k} & \text{for } i\geq j+k
\end{cases}
$
\\
\item
\label{reln:EasyQuadratic}
$p_{i+j,k} \circ_1 p_{i, j+\ell} = p_{i, j+k+\ell} \circ_2 p_{j,k}$
$\displaystyle
$
\end{enumerate}
\end{prop}
\begin{proof}
Immediate by drawing pictures.
\end{proof}

\begin{remark}
\label{rem:Quadratics}
The relation $p_{i+j+k,\ell} \circ_1 p_{i,j} = p_{i,j} \circ_1 p_{i+k,\ell}$ holds in the absence of anchor lines (\cite[Lemma 4.1.20, (i)]{math.QA/9909027}), but not in their presence.
These tangles are given by:
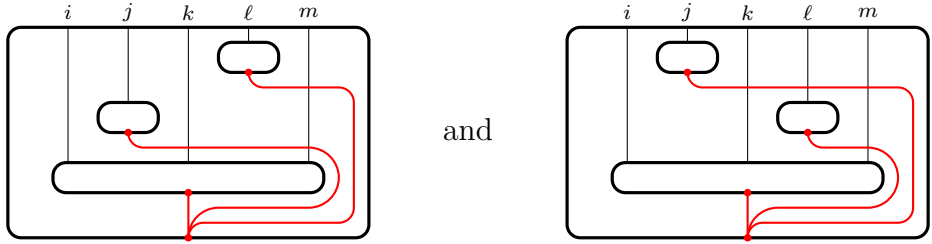
\[
\begin{tikzpicture}[baseline = -.1cm]
	\draw (-1.6,-.4) -- (-1.6,1.4);
	\node at (-1.6,1.6) {\scriptsize{$i$}};
	\draw (-.8,.4) -- (-.8,1.4);
	\node at (-.8,1.6) {\scriptsize{$j$}};
	\draw (0,-.4) -- (0,1.4);
	\node at (0,1.6) {\scriptsize{$k$}};
	\draw (.8,1) -- (.8,1.4);
	\node at (.8,1.6) {\scriptsize{$\ell$}};
	\draw (1.6,-.4) -- (1.6,1.4);
	\node at (1.6,1.6) {\scriptsize{$m$}};
	\draw[thick, red] (0,-.5) -- (0,-1.4);
	\draw[thick, red] (-.8,0) arc (180:270:.2cm) -- (1.6,-.2) arc (90:-90:.4cm) -- (.4,-1) arc (90:180:.4cm);
	\draw[thick, red] (.8,.8) arc (180:270:.2cm) -- (2,.6) arc (90:0:.2cm) -- (2.2,-1) arc (0:-90:.2cm) -- (.2,-1.2) arc (90:180:.2cm);
	\roundNbox{}{(0,0)}{1.4}{1}{1}{}
	\roundNbox{unshaded}{(0,-.6)}{.2}{1.6}{1.6}{}
	\roundNbox{unshaded}{(-.8,.2)}{.2}{.2}{.2}{}
	\roundNbox{unshaded}{(.8,1)}{.2}{.2}{.2}{}
	\fill[red] (0,-1.4) circle (.05)  (0,-.8) circle (.05)  (-.8,0) circle (.05)  (.8,.8) circle (.05);
\end{tikzpicture}
\qquad
\text{ and }
\qquad
\begin{tikzpicture}[baseline = -.1cm]
	\draw (-1.6,-.4) -- (-1.6,1.4);
	\node at (-1.6,1.6) {\scriptsize{$i$}};
	\draw (-.8,1) -- (-.8,1.4);
	\node at (-.8,1.6) {\scriptsize{$j$}};
	\draw (0,-.4) -- (0,1.4);
	\node at (0,1.6) {\scriptsize{$k$}};
	\draw (.8,.4) -- (.8,1.4);
	\node at (.8,1.6) {\scriptsize{$\ell$}};
	\draw (1.6,-.4) -- (1.6,1.4);
	\node at (1.6,1.6) {\scriptsize{$m$}};
	\draw[thick, red] (0,-.5) -- (0,-1.4);
	\draw[thick, red] (.8,0) arc (180:270:.2cm) -- (1.6,-.2) arc (90:-90:.4cm) -- (.4,-1) arc (90:180:.4cm);
	\draw[thick, red] (-.8,.8) arc (180:270:.2cm) -- (2,.6) arc (90:0:.2cm) -- (2.2,-1) arc (0:-90:.2cm) -- (.2,-1.2) arc (90:180:.2cm);
	\roundNbox{}{(0,0)}{1.4}{1}{1}{}
	\roundNbox{unshaded}{(0,-.6)}{.2}{1.6}{1.6}{}
	\roundNbox{unshaded}{(.8,.2)}{.2}{.2}{.2}{}
	\roundNbox{unshaded}{(-.8,1)}{.2}{.2}{.2}{}
	\fill[red] (0,-1.4) circle (.05)  (0,-.8) circle (.05)  (.8,0) circle (.05)  (-.8,.8) circle (.05);
\end{tikzpicture}
\]
The underlying planar tangles are isotopic, but the anchor lines are not the same.
\end{remark}


\section{The main theorem and examples}
\label{sec: The main theorem}

Given a braided pivotal category $\cC$, there is the notion of a \emph{module tensor category} over $\cC$.
This is a tensor category $\cM$ equipped with a multiplicative version of an action of $\cC$ (see Definition \ref{def: def of ModTC}).
A module tensor category is called \emph{pointed} if it equipped with a chosen self-dual object $m$ that generates $\cM$ as a $\cC$-module tensor category (see Definition~\ref{def: pointed} for a precise definition).

Our main theorem (Theorem \ref{thm:EquivalenceOfCategories2}) says that there is an equivalence between the notions of
anchored planar algebra in $\cC$ and of pointed module tensor category over $\cC$.
This will the be used later, in Section~\ref{sec: examples of APAs}, to construct many examples of anchored planar algebras.

\subsection{Module tensor categories}
\label{sec:InternalTrace}

The notion of module tensor category over a braided pivotal category was studied in \cite{1509.02937}:
\begin{defn}\label{def: def of ModTC}
Let $\cC$ be a braided tensor category.
A \emph{module tensor category} over $\cC$ is a tensor category $\cM$ together with a braided tensor functor $\Phi^{\scriptscriptstyle \cZ}:\cC\to \cZ(\cM)$ from $\cC$ to the Drinfel'd center of $\cM$. 
If $\cC$ and $\cM$ are moreover pivotal and if $\Phi^{\scriptscriptstyle \cZ}$ is a pivotal functor, then we call $\cM$ a pivotal module tensor category.
\end{defn}

From now on, we shall use the term ``module tensor category'' to mean ``pivotal module tensor category''.
The composite 
\[
\Phi:= F\circ \Phi^{\scriptscriptstyle \cZ} :\, \cC\to \cM
\]
of $\Phi^{\scriptscriptstyle \cZ}$ with the forgetful functor $F: \cZ(\cM)\to \cM$ is called the \emph{action functor}.
We write $e_{\Phi(c)}$ for the half-braiding associated to $\Phi^{\scriptscriptstyle \cZ}(c)$, so that $\Phi^{\scriptscriptstyle \cZ}(c)=(\Phi(c),e_{\Phi(c)})$.
If the action functor $\Phi$ admits a right adjoint, then we denote it by
\begin{equation}\label{eq: def of Tr}
\Tr_\cC:\cM\to \cC
\end{equation}
and call it a \emph{categorified trace}.
It comes equipped with \emph{traciator} isomorphisms
\[
\tau_{x,y}: \Tr_\cC(x \otimes y) \to \Tr_\cC(y \otimes x)
\]
satisfying $\tau_{x, y \otimes z}=\tau_{z \otimes x, y} \circ \tau_{x \otimes y, z}$.
For more details, we refer the reader to \cite{1509.02937}.

\begin{defn}
\label{defn:ModuleTensorCategoryFunctor}
Suppose $(\cM_1,\Phi^{\scriptscriptstyle \cZ}_1)$ and $(\cM_2,\Phi^{\scriptscriptstyle \cZ}_2)$ are module tensor categories over $\cC$.
A functor of module tensor categories $(G,\gamma): (\cM_1,\Phi^{\scriptscriptstyle \cZ}_1) \to (\cM_2,\Phi^{\scriptscriptstyle \cZ}_2)$ consists of the following data:
\begin{itemize}
\item
a pivotal tensor functor $G=(G,\nu,i): \cM_1\to \cM_2$.
Here, the structure maps $\nu_{x,y}: G(x)\otimes G(y)\to G(x\otimes y)$ and $i:1_{\cM_2}\to G(1_{\cM_1})$ satisfy the obvious
naturality, associativity and unitality axioms, along with the following compatibility with the pivotal structure:
\[
G(\varphi_x)=\delta_{x^*}^{-1}\circ \delta_x^*\circ \varphi_{G(x)}
\]
where $\delta_x:=((G(\ev_x)\circ\nu_{x^*,x})\otimes\id_{G(x)^*})\circ(\id_{G(x^*)}\otimes\coev_{G(x)}):G(x^*)\to G(x)^*$.
\item
an \emph{action coherence} monoidal natural isomorphism $\gamma: \Phi_2 \Rightarrow G\Phi_1$.
This consists of isomorphisms $\gamma_c : \Phi_2(c) \to G(\Phi_1(c))$, natural in $c$, making the following diagram commute:
\[
\xymatrix{
\Phi_2(c)\otimes \Phi_2(d)\ar[d]^{\cong}
\ar[rr]^(.43){\gamma_c\otimes \gamma_d}
&&
G(\Phi_1(c)) \otimes G(\Phi_1(d))
\ar[d]^{\cong} 
\\
\Phi_2(c\otimes d) 
\ar[r]^(.43){\gamma_{c\otimes d}}
&
G(\Phi_1(c\otimes d))
\ar[r]^(.45){\cong}
&
G(\Phi_1(c)\otimes\Phi_1(d)) 
}
\]
and such that $\gamma_{1_\cC}$ is equal to the obvious composite of unit coherences of $\Phi_i$ and $G$.
\end{itemize}
Moreover, for every $x\in\cM_1$ and $c\in\cC$,
the above pieces of data should satisfy the following compatibility axiom with half-braidings:
\begin{equation}
\label{eqn:GammaAndHalfBriadings}
\begin{matrix}\xymatrix{
\Phi_2(c) \otimes G(x) 
\ar[d]^{e_{\Phi_2(c), G(x)}}
\ar[rr]^(.45){\gamma_c\otimes \id}
&&
G(\Phi_1(c)) \otimes G(x)
\ar[rr]^(.52){\cong}
&&
G(\Phi_1(c) \otimes x)
\ar[d]^{G(e_{\Phi_1(c),x})}
\\
G(x) \otimes \Phi_2(c)
\ar[rr]^(.45){\id\otimes \gamma_c}
&&
G(x)\otimes G(\Phi_1(c))
\ar[rr]^(.52){\cong}
&&
G(x\otimes \Phi_1(c))
}\end{matrix}
\end{equation}
\end{defn}

\begin{defn}\label{Def: natural transformation for functors of module tensor categories}
Given two functors of module tensor categories $(G,\gamma^G)$ and $(H,\gamma^H)$ from $(\cM_1, \Phi^{\scriptscriptstyle \cZ}_1)$ to $(\cM_2, \Phi^{\scriptscriptstyle \cZ}_2)$,
a natural transformation $\kappa : (G,\gamma^G)\Rightarrow (H,\gamma^H)$ is a family of
isomorphisms $\kappa_x : G(x)\to H(x)$ for $x\in \cM_1$, natural in $x$.
They are monoidal in the sense that $\kappa_{1_{\cM_1}}\circ i_G=i_H$ and that the following diagram commutes:
\begin{equation}\label{eq: def nat transf 1}
\begin{matrix}
\xymatrix{
F(x\otimes y)  
\ar[rr]^(.45){\cong} 
\ar[d]^{\kappa_{x\otimes y}}
&&
F(x)\otimes F(y)
\ar[d]^{\kappa_x\otimes \kappa_y}
\\
G(x\otimes y) 
\ar[rr]^(.45){\cong}
&&
G(x)\otimes G(y),\!
}
\end{matrix}
\end{equation}
and they are subject to the following compatibility axiom with the action coherence isomorphisms:
\begin{equation}\label{eq: def nat transf 2}
\begin{matrix}
\xymatrix@C=.5cm{
\Phi_2(c)
\ar[rrrr]^(.47){\gamma^G_c}
\ar[drr]_{\gamma^H_c}
&&&&
G(\Phi_1(c))\,\,\,
\ar[dll]^{\kappa_{\Phi_1(c)}}
\\
&&
H(\Phi_1(c))
}
\end{matrix}
\end{equation}
\end{defn}

An object $m$ in a pivotal category is called \emph{symmetrically self-dual} if it comes equipped with an isomorphism $\psi:m\to m^*$ satisfying $\psi^*\circ\varphi_m=\psi$.
Here, $\varphi_m:m\to m^{**}$ is the isomorphism provided by the pivotal structure and 
\begin{equation}\label{eq: psi** def}
\psi^* = (\ev_{m^*}\otimes \id_m) \circ (\id_{m^{**}}\otimes\psi\otimes\id_{m^*}) \circ (\id_{m^{**}}\otimes \coev_m):m^{**}\to m^*
\end{equation}
is the dual of $\psi$.
Upon identifying $m$ with $m^*$ via $\psi$, we write 
\begin{equation}\label{eq: ev bar and coev bar}
\begin{split}
\bar\ev_m&:=\ev_m\circ(\psi\otimes\id_m):m\otimes m \to 1\quad\,\,\,\,\, \text{and}\\
\bar\coev_m&:=(\id_m\otimes\psi^{-1})\circ\coev_m:1 \to m \otimes m
\end{split}
\end{equation}
for the evaluation and coevaluation maps.
Symmetrically self-dual objects admit a graphical calculus of unoriented strands where all isotopies are allowed,
and where the caps and cups are given by the above modified evaluation and coevaluation maps.

\begin{defn}\label{def: pointed}
A module tensor category $\cM$ is called \emph{pointed} if it comes equipped with a chosen symmetrically self-dual object $m\in \cM$
that generates it as a module tensor category over $\cC$.
Here, we say that $m$ generates $\cM$ as a module tensor category if every object of $\cM$ is a direct summand of a sum of objects of the form 
$\Phi(c)\otimes m^{\otimes n}$, for some $c\in \cC$.

A functor of pointed module tensor categories is a functor $G: \cM_1 \to \cM_2$ of module tensor categories that sends the distinguished object $m_1\in\cM_1$ to the distinguished object $m_2\in\cM_2$
and satisfies $\psi_2=\delta_{m_1}\circ G(\psi_1)$.

A natural transformation between two functors $G,H:\cM_1\to\cM_2$ of pointed module tensor category is a natural transformation $\kappa : G \Rightarrow H$ that satisfies $\kappa_{m_1} = \id_{m_2}$.
\end{defn}

From now on, we restrict our attention to (pivotal) module tensor categories which admit direct sums, are idempotent complete, and whose action functor $\Phi:\cC\to\cM$ admits a right adjoint.
The collection of all such module tensor categories over a fixed braided pivotal category $\cC$ forms a $2$-category, which we denote $\Mod$ (we omit $\cC$ from the notation).
Similarly, the collection of all pointed module tensor categories forms a $2$-category $\Mod_*$.

\begin{lem}\label{lem: it's secretly a 1-category}
Let be $G,H:\cM_1\to\cM_2$ be two $1$-morphisms in $\Mod_*$.
Then there exists at most one $2$-morphism from $G$ to $H$.
Moreover, any $2$-morphism in $\Mod_*$ is invertible.
\end{lem}
\begin{proof}
By definition, a natural transformation $\kappa : G \Rightarrow H$ of functors between pointed module tensor categories satisfies $\kappa_{m_1} = \id_{m_2}$.
By the commutativity of the diagram \eqref{eq: def nat transf 1}, this completely determines $\kappa$ on objects of the form $m_1^{\otimes n}$.
By the same diagram, we also see that $\kappa_{m_1^{\otimes n}}$ is invertible.
Similarly, the commutativity of \eqref{eq: def nat transf 2} completely determines $\kappa$ on objects of the form $\Phi_1(c)$ for $c\in\cC$, and $\kappa_{\Phi_1(c)}$ is invertible.
By one more application of \eqref{eq: def nat transf 1}, we see that $\kappa$ is completely determined and invertible on objects of the form $\Phi_1(c)\otimes m_1^{\otimes n}$.
At last, $\kappa$ is completely determined and invertible on any direct summand of a direct sum of objects of the form $\Phi_1(c)\otimes m_1^{\otimes n}$, which is all of $\cM_1$ by our assumption that $m_1$ generates $\cM_1$.
\end{proof}

Consider the equivalence relation on $1$-morphisms of $\Mod_*$ according to which two functors $G,H:\cM_1\to\cM_2$
are equivalent if there exists a (necessarily invertible) natural transformation $\kappa : G \Rightarrow H$.
We denote $\tau_{\le 1}(\Mod_*)$ the $1$-category with same objects as $\Mod_*$, and with $1$-morphisms given by equivalence classes of $1$-morphisms of $\Mod_*$ with respect to the above equivalence relation.
By the previous lemma, the projection functor
\[
\Mod_*\to \tau_{\le 1}(\Mod_*)
\]
is an equivalence of $2$-categories (where $\Mod_*$ is treated as a $2$-category with only identity $2$-morphisms).
The moral of the above discussion is that, even though $\Mod_*$ is not a $1$-category,
it is harmless to treat it as if it were one.

\subsection{The main theorem}

Fix a braided pivotal tensor category $\cC$ which admits finite direct sums and is idempotent complete.
Let $\APA$ be the category of anchored planar algebras in $\cC$ (we omit $\cC$ from the notation), and let $\Mod_*$ and  $\tau_{\le 1}(\Mod_*)$ be as in the previous section.
The following result generalizes the well known result according to which the data of a planar algebra (in vector spaces) is equivalent to that of a pair $(\cD,X)$, where $\cD$ is a pivotal tensor category and $X\in \cD$ is a generating symmetrically self-dual object \cite{MR2559686,MR3405915,1207.1923}:

\begin{mainthm}[Theorem \ref{thm:EquivalenceOfCategories}]
\label{thm:EquivalenceOfCategories2}
There exists an equivalence of categories
\[
\left\{\,\text{\rm Anchored planar algebras in $\cC$}\left\}
\,\,\,\,\cong\,\,
\left\{\,\parbox{4.7cm}{\rm \centerline{Pointed pivotal module} \centerline{tensor categories over $\cC$}}\,\right\}\right.\right.\!\!
\]
between the category $\APA$ of anchored planar algebras in $\cC$, and the category $\tau_{\le 1}(\Mod_*)$ of pointed pivotal module tensor categories over $\cC$ whose action functor admits a right adjoint.
\end{mainthm}

\begin{proof}
We construct functors $\Lambda: \tau_{\le 1}(\Mod_*) \to \APA$ and $\Delta:\APA\to \tau_{\le 1}(\Mod_*)$ in Sections \ref{sec:APAfromMTC} and \ref{sec:MTCfromAPA}, respectively.
In Sections \ref{sec:CMGtoAPAtoCMG} and \ref{sec:CMGtoAPAtoCMG_Naturality}, we construct a natural isomorphism $\Delta\Lambda \Rightarrow \id_{\tau_{\le 1}(\Mod_*)}$.
Finally, in Section \ref{sec:APAtoCMGtoAPA}, we show that the equality $\Lambda\Delta = \id_\APA$ holds on the nose.
\end{proof}

We quickly sketch the construction of the above two functors.
The functor $\Lambda$ sends a pointed module tensor category $(\cM,m)$
to the anchored planar algebra $\cP$ with $n$th box object
\begin{equation}\label{eq: def box object}
\cP[n] := \Tr_\cC(m^{\otimes n}),
\end{equation}
and action of the generating tangles (Definition \ref{defn:GeneratingTangles}) given by Theorem~\ref{thm: construct P from M and m}.

The functor $\Delta$ sends an anchored planar algebra $\cP$ to the category $\cM$ constructed as follows.
We first construct $\cM_0$, whose objects are formal expressions ``$\Phi(c)\otimes m^{\otimes n}$''
and whose morphisms are given by
\[
\cM_0\big(\Phi(c)\otimes m^{\otimes n_1},\Phi(d)\otimes m^{\otimes n_2}\big) \,:=\, \cC(c, d\otimes \cP[n_2 + n_1]).
\]
$\cM$ is then obtained from $\cM_0$ by formally adding direct sums, and idempotent completing.
The composition of morphisms, tensor structure, and structure of module tensor category are described in Section~\ref{sec:MTCfromAPA}.

\subsubsection{Generalisations}
\label{sec:Generalisations}
There exist a number of variations of Theorem~\ref{thm:EquivalenceOfCategories2}, which we list below without proof.
We believe that these more general statements can be proved in a way essentially identical to our main theorem:

\begin{itemize}
\item
If the generating object $m\in\cM$ is not required to be self-dual, then 
the notion of anchored planar algebra needs to be modified to include orientations on the strands.
We call the resulting notion an \emph{oriented anchored planar algebra}.
In an oriented anchored planar algebra, the box objects are no longer indexed by numbers.
They are indexed by sequences of $+$ and $-$, encoding the orientations of the strands.
For example, in an oriented anchored planar algebra, we have:
\[
Z\left(
\begin{tikzpicture}[scale=.8, baseline =-.1cm]
	\coordinate (a) at (0,0);
	\coordinate (c) at (.6,-.6);         
	\coordinate (d) at (-.6,.6);         
	
	\ncircle{}{(a)}{1.6}{-115}{}
	\draw[thick, red] (c)+(235:.4) .. controls ++(230:.5cm) and ++(45:.4cm) .. (-115:1.6);
	\draw[thick, red] (d)+(235:.4) .. controls ++(250:.6cm) and ++(70:.4cm) .. (-115:1.6);
			
	\draw (60:1.6cm) arc (150:300:.4cm);
	\draw ($ (c) + (0,.4) $) arc (0:90:.8cm);
	\draw ($ (c) + (-.4,0) $) circle (.25cm);
	\draw ($ (d) + (0,.88) $) -- (d) -- ($ (d) + (-.88,0) $);
	\draw ($ (c) + (0,-.88) $) -- (c) -- ($ (c) + (.88,0) $);
	\ncircle{unshaded}{(d)}{.4}{235}{}
	\ncircle{unshaded}{(c)}{.4}{235}{}
\draw[<-] (-.6,1.2) -- +(90:.01);
\draw[->] (-1.25,.6) -- +(180:.01);
\draw[<-] (-.05,-.62) -- +(87:.01);
\draw[<-] (.365,.365) -- +(135:.01);
\draw[->] (.85,.91) -- +(139:.01);
\draw[<-] (.6,-1.28) -- +(90:.01);
\draw[<-] (1.28,-.6) -- +(180:.01);
\end{tikzpicture}
\right)
\,:\,\,\,\mathcal{P}[\raisebox{1pt}{$\scriptstyle +,-,+$}]\otimes \mathcal{P}[\raisebox{1pt}{$\scriptstyle -,+,-,+,+$}] \,\to\, \mathcal{P}[\raisebox{1pt}{$\scriptstyle +,-,+,-,+,+$}].
\]

There is an equivalence of categories:
\[
\left\{\,\parbox{3.8cm}{\rm \centerline{Oriented anchored} \centerline{planar algebras in $\cC$}}
\left\}
\,\,\,\,\cong\,\,
\left\{\,\parbox{7.3cm}{\rm \centerline{Pivotal module tensor categories over $\cC$} \centerline{equipped with a chosen generator}
}\,\right\}\right.\right.\!\!
\]
Here, the condition that $m$ generates $\cM$ means that every object of the category is isomorphic to
a direct summand of a direct sum of objects of the form
$\Phi(c) \otimes m^{\otimes n_1} \otimes (m^*)^{\otimes n_2} \otimes m^{\otimes n_3}\otimes (m^*)^{\otimes n_4} \otimes \ldots$

\item
If instead of requiring that the generator $m\in\cM$ is symmetrically self-dual, we take it anti-symetrically self-dual, i.e., if one replaces the condition $\psi^*\circ\varphi_m=\psi$ in the paragraph above \eqref{eq: psi** def} by $\psi^*\circ\varphi_m=-\psi$, then the strands need to be equipped with a ``disorientation'' in the sense of \cite[Fig.\,3]{MR2496052}.
Here, a disorientation is the data of an orientation, along with a finite collection of orientation-reversing points, drawn as
$\tikz[baseline=-2.5]{\draw (0,0) -- (1,0); \draw[->] (.22,0) -- +(.01,0);\draw[->] (.78,0) -- +(-.01,0); \draw[red] (.43,0) +(0,.08) arc (90:-90:.08);}$
and
$\tikz[baseline=-2.5]{\draw (0,0) -- (1,0); \draw[<-] (.2,0) -- +(.01,0);\draw[<-] (.8,0) -- +(-.01,0); \draw[red] (.45,0) +(0,.08) arc (90:-90:.08);}$.
These are subject to the local relations
$Z(\tikz[baseline=-2.5]{\draw (0,0) -- (1.75,0); \draw[->] (.22,0) -- +(.01,0);\draw[->] (.78,0) -- +(-.01,0); \draw[->] (1.5,0) -- +(.01,0); \draw[red] (.43,0) +(0,.08) arc (90:-90:.08);\draw[red] (1.2,0) +(0,.08) arc (90:270:.08);})
=Z(\tikz[baseline=-2.5]{\draw (0,0) -- (.75,0); \draw[->] (.38,0) -- +(.01,0);})
$
and 
$Z(\tikz[baseline=-2.5]{\draw (0,0) -- (1,0); \draw[->] (.22,0) -- +(.01,0);\draw[->] (.78,0) -- +(-.01,0); \draw[red] (.43,0) +(0,.08) arc (90:-90:.08);})
=-Z(\tikz[baseline=-2.5]{\draw (0,0) -- (1,0); \draw[->] (.22,0) -- +(.01,0);\draw[->] (.78,0) -- +(-.01,0); \draw[red] (.57,0) +(0,.08) arc (90:270:.08);})$.
We call the resulting notion a \emph{disoriented anchored planar algebra}.
It is convenient (but not necessary) to insist that the strands only meet the boundary circles in the positive orientation:
in that way, the box object are indexed by $\mathbb N_{\ge 0}$.

We have an equivalence of categories:
\[
\left\{\,\parbox{3.9cm}{\rm \centerline{Disoriented anchored} \centerline{planar algebras in $\cC$}}
\left\}
\,\,\,\cong\,
\left\{\,\parbox{9.1cm}{\rm \centerline{Pivotal module tensor categories over $\cC$ equipped} \centerline{with a chosen anti-symetrically self-dual generator}
}\,\right\}\right.\right.\!
\]
The crucial step which gets affected by the presence of disorientations is the last equality in \eqref{eq: It's symmetrically self-dual!}, which now acquires a minus sign.

\item
Instead of a single generator, we may consider a collection of symmetrically self-dual objects $\{m_a\}_{a\in \Gamma}$ that jointly generate $\cM$.
The appropriate generalisation of anchored planar algebras then includes multiple colours of strings, one for each element of $\Gamma$.
We call the resulting notion a \emph{colored anchored planar algebra}.
The box objects of a colored anchored planar algebra are now indexed by words in the alphabet $\Gamma$.
For example:
\[
\qquad
Z\left(
\begin{tikzpicture}[scale=.8, baseline =-.1cm]
	\coordinate (a) at (0,0);
	\coordinate (c) at (.6,-.6);         
	\coordinate (d) at (-.6,.6);         
	
	\ncircle{}{(a)}{1.6}{-117}{}
	\draw[thick, red] (c)+(235:.4) .. controls ++(230:.5cm) and ++(45:.4cm) .. (-117:1.6);
	\draw[thick, red] (d)+(235:.4) .. controls ++(250:.6cm) and ++(70:.4cm) .. (-117:1.6);
			
	\draw[red] (60:1.6cm) arc (150:300:.4cm);
	\draw[DarkGreen] ($ (c) + (0,.4) $) arc (0:90:.8cm);
	\draw[blue] ($ (c) + (-.4,0) $) circle (.25cm);
	\draw[blue] ($ (d) + (0,.88) $) -- (d) -- ($ (d) + (-.88,0) $);
	\draw[DarkGreen] ($ (c) + (0,-.88) $) -- (c);
	\draw[red] (c) -- ($ (c) + (.88,0) $);
	\ncircle{unshaded}{(d)}{.4}{235}{}
	\ncircle{unshaded}{(c)}{.4}{235}{}
\node[right] at (-.7,1.25) {$\scriptstyle a$};
\node[below] at (-1.21,.68) {$\scriptstyle a$};
\node[left] at (.05,-.55) {$\scriptstyle a$};
\node[left] at (.83,1.05) {$\scriptstyle c$};
\node at (.62,.42) {$\scriptstyle b$};
\node[right] at (.52,-1.2) {$\scriptstyle b$};
\node[above] at (1.23,-.65) {$\scriptstyle c$};
\useasboundingbox;
\node[scale=.7] at (-120:1.82) {\anchor};
\end{tikzpicture}
\right)
:\,\mathcal{P}[aab]\otimes \mathcal{P}[aabcb] \,\to\, \mathcal{P}[aacccb], \quad\qquad \textcolor{blue}{a},\textcolor{DarkGreen}{b},\textcolor{red}{c}\in \Gamma.
\]

For every set $\Gamma$, we then get an equivalence of categories:
\[
\left\{\,\parbox{4cm}{\rm \centerline{$\Gamma$-colored anchored} \centerline{planar algebras in $\cC$}}
\left\}
\,\,\,\,\,\cong\,\,\,
\left\{\,\parbox{8.05cm}{\rm \centerline{Pivotal module tensor categories over $\cC$} \centerline{equipped with a collection of symmetrically} \centerline{self-dual generators indexed by $\Gamma$}
}\,\right\}\right.\right.\!
\]

\item
Dropping the self-dual condition, we may also consider collections of arbitrary objects that jointly generate $\cM$.
In that case, the appropriate generalisation of planar algebra is a \emph{colored oriented anchored planar algebra}.

There is an equivalence of categories:
\[
\left\{\,\parbox{5.1cm}{\rm \centerline{$\Gamma$-colored oriented anchored} \centerline{planar algebras in $\cC$}}
\left\}
\,\,\,\cong\,
\left\{\,\parbox{8.1cm}{\rm \centerline{Pivotal module tensor categories over $\cC$} \centerline{with a collection of generators indexed by $\Gamma$}
}\,\right\}\right.\right.\!
\]
We note that in a $\Gamma$-colored oriented anchored planar algebra, the box objects are now indexed by words in $\Gamma\cup \Gamma^*$.

\end{itemize}

\subsection{Examples}\label{sec: examples of APAs}

Using Theorem \ref{thm:EquivalenceOfCategories2}, we now provide a list of examples of anchored planar algebras.
For each anchored planar algebra, we compute the box objects by means of the formula \eqref{eq: def box object}.

\subsubsection{Near group categories}
\label{sec: Near group examples}
Let $A$ be an abelian group and let $\cC:=\Vec(A)$ be the category of $A$-graded vector spaces.
Given a symmetric bicharacter $\langle \,\cdot\,,\,\cdot\,\rangle: A\times A \to U(1)$ (which we will later take to be non-degenerate), there is a braided structure on $\cC$
with braiding $\beta_{a,b}: a\otimes b \to b\otimes a$ given by $\langle a,b\rangle\id_{a+b}$.
Note that two such categories are braided equivalent iff the associated quadratic forms $a\mapsto\langle a,a\rangle$ are equal \cite[Prop.\,2.14]{MR2076134}.
A braided category as above has a preferred pivotal structure, characterised by the fact that all the quantum dimensions of the simple objects are 1.
Its twist \eqref{def:theta1} are given by $\theta_{a} = \langle a, a\rangle\id_a$.

\begin{ex}[Tambara-Yamagami]
\label{ex:TambaraYamagami}
Let $A$ be a finite abelian group, and let $\langle \,\cdot\,,\cdot\,\rangle$ be a non-degenerate symmetric bicharacter. Let $N:=|A|$, and let $\tau$ be a square root of $N^{-1}$.
The Tambara-Yamagami fusion category $\cM := \cT\cY(A, \langle \,\cdot\,,\cdot\,\rangle, \tau)$ \cite{MR1659954}
has distinguished representatives of simple objects indexed by the set $A\cup\{m\}$.
We follow the conventions in \cite[\S3]{MR2774703}.
The monoidal structure on $A\subset \cM$ is given by group multiplication, with trivial associator.
The remaining fusion rules are:
\[
a\otimes m \,:=\, m, \qquad  m\otimes a\,:=\,m, \qquad\,\,\, m\otimes m \,:=\, \bigoplus_{a\in A} a.
\]
The associator $\alpha_{a,m,b}:(a\otimes m)\otimes b \rightarrow a\otimes(m\otimes b)$ is $\langle a,b\rangle\id_m$
(we refer to \cite[\S3]{MR2774703} for the remaining associators).
We work with the pivotal structure on $\cM$ given by $\varphi_a=\id_a$ for $a\in A$ (so that $\dim(a)=1$) and $\varphi_m=\id_m$ (so that $\dim(m)=\tau^{-1}$).
Note that, with respect to this pivotal structure, the object $m$ is symmetrically self-dual.

The center $\cZ(\cM)$ contains $2N$ invertible objects, labeled by pairs $(a, \varepsilon)$ where $a\in A$ and $\varepsilon^2 = \langle a, a\rangle$ \cite{MR1832764}. 
The half-braiding $e_{(a,\varepsilon)}$ is given by $e_{(a,\varepsilon),b} = \langle a,b\rangle\id_{a+b}$ and $e_{(a,\varepsilon), m} = \varepsilon \id_m$,
from which it follows that $\theta_{(a, \varepsilon)} = \langle a, a\rangle\id_{(a,\varepsilon)}$.
Using the formula for $\alpha_{a,m,b}$, one easily computes that 
$(a, \varepsilon)\otimes (b, \eta) = (a+b, \varepsilon \eta \langle a,b\rangle)$ in $\cZ(\cM)$.

We now specialize to $A=\bbZ/N\bbZ$, with bicharacter $\langle i, j\rangle = q^{ij}$, $q=\exp(2\pi i/N)$.
There is a distinguished copy of $A$ inside $\cZ(\cM)$, given by $\varepsilon_a = q^{a^2 (N+1)/2}$.
%
The thus obtained inclusion functor $\Phi^{\scriptscriptstyle \cZ}:\cC\to\cZ(\cM)$ endows $\cM$ with the structure of a module tensor category over $\cC=\Vec(\bbZ/N\bbZ)$
(with braiding given by the above bicharacter, and equipped with its canonical pivotal structure).
The action functor $\Phi: \cC\to \cM$ is fully faithful, and its adjoint $\Tr_\cC:\cM\to \cC$ is the projection functor from $\cM$ onto its full $\Vec(\bbZ/N\bbZ)$ subcategory.
In other words, $\Tr_\cC(m) = 0$, and $\Tr_\cC(a) = a$ for all $a\in \bbZ/N\bbZ$.
Applying formula \eqref{eq: def box object} to this situation,
we see that the anchored planar algebra $\cP$ associated to $\cM$ and $m$ is given by:
\[
\cP[2k] = \Tr_\cC(m^{\otimes 2k})\cong \bbC[\bbZ/N\bbZ]^{\otimes k}\qquad\,\, \text{and}\qquad\,\, \cP[2k+1] = 0.
\]
\end{ex}

\begin{ex}[Planar para algebras]
\label{ex:Planar para algebras}
Let $\cC=\Vec(\bbZ/N\bbZ)$ be as in the previous example.
The recently introduced planar para algebras \cite[Def.\,2.3]{1602.02662} are the same thing as anchored planar algebras in $\Vec(\bbZ/N\bbZ)$:
the para isotopy axiom (iv) is our braid anchor dependence, and the rotation axiom (v) is our twist anchor dependence. 

Applying the functor $\Delta:\APA\to \Mod_*$ (see Section \ref{sec:MTCfromAPA} below) to their main example of the parafermion planar para algebra, one recovers the Tambara-Yamagami category $\cT\cY(A, \langle \,\cdot\,,\cdot\,\rangle, \tau)$ associated to the positive square root of $N^{-1}$.
Indeed, if $m$ is the strand in the parafermion planar para algebra, then $m\otimes m \cong \bbC[\bbZ/N\bbZ]$.
Thus, the parafermion planar para algebra is given, as in Example \ref{ex:TambaraYamagami}, by a braided functor $\Vec(\bbZ/N\bbZ) \to \cZ(\cM)$ to the center of a Tambara-Yamagami fusion category.
\end{ex}

\begin{ex}[Near groups]
\label{ex: Near groups}
Let $A$ be an abelian group of order $N$, and let $\cM$ be a near group category with fusion rules 
\[
a\otimes m \,=\, m, \qquad  m\otimes a\,=\,m, \qquad\,\,\, m\otimes m \,=\, Nm \oplus \bigoplus_{a\in A} a.
\]
Such categories are known to exist for many groups (up to order 13) \cite[Table 2]{MR3167494}, and are conjectured to exist for all finite cyclic groups.

By \cite{MR1832764}, \cite[Thm.\,4]{MR3167494}, there is a non-degenerate symmetric bicharacter $\langle \,\cdot\, , \cdot\,\rangle: A\times A \to U(1)$
which is an invariant of $\cM$.
Now, by \cite[\S4.3]{MR3167494}, each invertible object $a\in A$ lifts uniquely to the center $\cZ(\cM)$, 
and the braidings and twists on $A\subset \cZ(\cM)$ are given by the bicharacter: $\beta_{a,b} = \langle a,b\rangle$ and $\theta_a = \langle a, a\rangle$.
Hence, $\cZ(\cM)$ admits a braided pivotal functor from the corresponding ribbon category $\Vec(A)$.

Let $\cP$ be the anchored planar algebra associated to $(\cM, m)$.
The action functor $\Phi:\cC\to \cM$ is full, and its adjoint $\Tr_\cC:\cM\to\cC$ is the projection onto the copy of $\Vec(A)$ inside $\cM$.
For small $k$, the box objects $\cP[k]=\Tr_\cC(m^{\otimes k})$ are given by $\cP[0]=1$, $\cP[1]=0$, and $\cP[2]=\bigoplus_{a\in A} a$.
For larger $k\geq 1$, we have $\cP[k] = f(k) \bigoplus_{a\in A} a$, where the function $f$ is given by the recursive formula
$$
f(1)=0\quad\,\,\,\, f(2)=1 \qquad  f(k+1) = N(f(k) + f(k-1)).
$$
We note that the quantity $f(k)$ depends polynomially on $N$, with coefficients given by diagonals in Pascal's triangle:
$$
f(k) = \sum_{i=2}^{\lfloor k/2\rfloor +1} {{k-i}\choose {i-2}} N^{k-i}.
$$
\end{ex}

\subsubsection{Temperley-Lieb-Jones module categories}
\label{sec: TLJ examples}

Let $\cC$ be non-degenerately braided pivotal fusion category, 
let $A\in\cC$ be a connected \'{e}tale algebra (an algebra is called \emph{\'etale} is it is both commutative and separable), and let $\cM = {\sf Mod}_\cC(A)$ be the category of right $A$-modules.
By \cite[Cor.\,3.30]{MR3039775}, we then have
\[
\cZ(\cM) \,\cong\, \cC \boxtimes \cD,
\]
where $\cD$ is the full subcategory of $\cM$ consisting of dyslectic $A$-modules, equipped with the braiding opposite to the one inherited from $\cC$.
Thus, we see that $\cM$ becomes a module tensor category for both $\cC$ and $\cD$:
\[
\Phi^{\scriptscriptstyle\cZ}_\cC:\cC\to\cZ(\cM)\quad\qquad\text{and}\quad\qquad \Phi^{\scriptscriptstyle\cZ}_\cD:\cD\to\cZ(\cM).
\]
The functor $\Phi_\cC:\cC\to \cM$ is dominant, while the functor $\Phi_\cD:\cD\to \cM$ is fully faithful 
(so that its adjoint $\Tr_\cD:\cM\to\cD$ is the projection onto a full subcategory of $\cM$ equivalent to~$\cD$).

We now specialize to the case where $\cC$ is the Temperley-Lieb-Jones (TLJ) $A_n$ category.
Let $q:=\exp(\pi i/(n+1))$ and $\delta:=q+q^{-1}$.
The TLJ $A_n$ category is the quotient of the free pivotal idempotent complete category on a single symmetrically self-dual object $m$ of dimension $\delta$
by the ideal of negligible morphisms (morphisms $f:x\to y$ such that $tr(fg)=0$ for all $g:y\to x$).
More concretely, it is obtained by taking the category whose objects are the natural numbers and whose morphisms are linear combinations of Kauffman diagrams diagrams
 (things like this: $\tikz[scale=.4, baseline=2]{\draw[thick] (0,0) rectangle (3.3,1);
\draw (.5,-.12) -- (.5,0) arc (180:0:.4) -- +(0,-.12);
\draw (.73,-.12) -- (.73,0) arc (180:0:.17) -- +(0,-.12);
\draw (1.7,-.12) -- (1.7,0) arc (180:0:.2) -- +(0,-.12);
\draw (2,1.12) -- (2,1) arc (-180:0:.25) -- +(0,.12);
\draw (2.5,-.12) -- (2.5,0) to[out=90,in=-90] (1.5,1) -- +(0,.12);
\draw (.6,1.12) -- (.6,1) arc (-180:0:.25) -- +(0,.12);
\draw (2.9,-.12) -- (2.9,0) to[out=90,in=-90] (2.8,1) -- +(0,.12);
}$\;), 
equating free floating loops with $\delta$,
taking the idempotent completion, and then modding out by the negligible morphisms.
This category has $n$ simple objects, denoted $\mathbf{1},\mathbf{2},\dots, \mathbf{n}$, called the Jones-Wenzl idempotents.
The object $\mathbf{1}$ is the unit, and the object $\mathbf{2}$ is the generator $m$ (the strand).
We also write $\mathbf g$ for the object $\mathbf{n}$.
We equip $\cC$ with the braiding given by
\begin{equation}
\label{eq:BraidingInTL}
\begin{tikzpicture}[baseline=-.1cm]
	\braiding{(0,0)}{.6}{.4}
\end{tikzpicture}
\,:=\,
iq^{1/2}\,
\begin{tikzpicture}[baseline=-.1cm]
	\nbox{unshaded}{(0,0)}{.4}{0}{0}{}
	\draw (.2,-.6)--(.2,.6);
	\draw (-.2,-.6)--(-.2,.6);
\end{tikzpicture}
\,-
iq^{-1/2}\,
\begin{tikzpicture}[baseline=-.1cm]
	\draw (.2,-.6)--(.2,.6);
	\draw (-.2,-.6)--(-.2,.6);
	\nbox{unshaded}{(0,0)}{.4}{0}{0}{}
	\draw (-.2,-.4) arc (180:0:.2cm);
	\draw (-.2,.4) arc (-180:0:.2cm);
\end{tikzpicture}
\end{equation}
on the generator.
Note that this is not the only possible braiding, and that the classification of \'etale algebras depends on the choice braiding \cite[Rem.\,3.14]{1509.02937}.

\begin{ex}
\label{sub-ex:TypeA}
Let $\cC$ be the TLJ $A_n$ category, and
let $(\cM,m)=(A_n, \mathbf{2})$ with its obvious structure of module tensor category over $\cC=A_n$.
The functor $\Tr_\cC:\cM\to\cC$ is the identity functor, and
the associated anchored planar algebra $\cP$ is given by $\cP[k]=\underline{\Hom}_\cC(1, m^{\otimes k}) = m^{\otimes k}$ (the object represented by $k$ parallel strands).
\end{ex}

\begin{ex}[$D_{2n}$]
\label{sub-ex:TypeD}
When $\cC$ is the TLJ $A_{4n-3}$ category, $A:=\mathbf{1}\oplus \mathbf{g}\in \cC$ is a commutative algebra object, and its category of modules $\cM$ is a $D_{2n}$ tensor category.
We write $\underline{1}$, $\underline{2},\dots, \underline{2n\hspace{.715cm}}\hspace{-.715cm}-2$, $\underline{2n\hspace{.715cm}}\hspace{-.715cm}-1$, $\underline{2n\hspace{.715cm}}\hspace{-.715cm}-1'$
for the simples of $\cM$.

If we view $(\cM,m)=(D_{2n},\underline{2})$ as a pointed module tensor category over $\cC$, then the box objects of the associated anchored planar algebra can be computed as follows:
\[
\cP[k] = \Tr_\cC(\underline{2}^{\otimes k}) = \Tr_\cC(\Phi(\mathbf{2}^{\otimes k}))
=
(\mathbf{1}\oplus \mathbf{g})\otimes \mathbf{2}^{\otimes k}.
\]
Here, $\Phi : A_{4n-3} \to D_{2n} = {\sf Mod}_{A_{4n-3}}(A)$ is the free $A$-module functor, and its adjoint $\Tr_\cC : D_{2n} \to A_{4n-3}$ is the functor that forgets the module structure (see \cite[Ex.\,3.12]{1509.02937}).
\end{ex}

\begin{ex}[$D_{2n}$]
Let $\cM=D_{2n}$ be as in the previous example.
By \cite[Cor.\,4.9]{MR1815993}, the center of $\cM$ is given by $\cZ(D_{2n})\cong \cC\boxtimes \cD = A_{4n-3} \boxtimes \frac{1}{2}\overline{D}_{2n}$, where the bar on the second factor indicates that it is equipped with the opposite of the natural braiding.

If we view $(\cM,m)=(D_{2n},\underline{2})$ as a pointed module tensor category over $\cD:=\frac{1}{2}\overline{D}_{2n}$, then
$\Tr_\cD:\cM\to \cD$ is the projection functor onto the copy of $\cD$ inside $D_{2n}$. 
The associated anchored planar algebra in $\cD$ is given by
$\cP[k]=\underline{2}^{\otimes k}$ when $k$ is even, and $\cP[k]=0$ when $k$ is odd.
\end{ex}

\begin{ex}[$D_4$]
\label{sub-ex:D4}
Let $(\cM,m)=(D_4,\underline{2})$ be as in the previous examples (note that it is also a Tamabra-Yamagami category),
and let
\[
\cC := \cZ(\cM) = A_5\boxtimes \tfrac12\overline{D}_4 = A_5\boxtimes \Vec(\bbZ/3\bbZ).
\]
As before, we write $\mathbf{1}$, $\mathbf{2},\ldots,\mathbf{5}$ for the objects of $A_5$, and
$\underline{1}$, $\underline{2}$, $\underline{3}$, $\underline{3}'$ for the object of $D_4$.
The functor $\Phi:\cC\to\cM$ is given on pairs $(p,q)\in\cZ(D_4)$ by $\Phi(p,q) = p\otimes q$:
\begin{align*}
(\mathbf1,\underline1) &\mapsto \underline1 
& 
(\mathbf2,\underline1) &\mapsto \underline2
& 
(\mathbf3,\underline1) &\mapsto \underline3\oplus \underline3'
& 
(\mathbf4,\underline1) &\mapsto \underline2
&
(\mathbf5,\underline1) &\mapsto \underline1
\\
(\mathbf1,\underline3) &\mapsto \underline3
&
(\mathbf2,\underline3) &\mapsto \underline2
& 
(\mathbf3,\underline3) &\mapsto \underline1\oplus \underline3'
& 
(\mathbf4,\underline3) &\mapsto \underline2
&
(\mathbf5,\underline3) &\mapsto \underline3
\\
(\mathbf1,\underline3')\! &\mapsto \underline3'
&
(\mathbf2,\underline3')\! &\mapsto \underline2
& 
(\mathbf3,\underline3')\! &\mapsto \underline1\oplus \underline3
& 
(\mathbf4,\underline3')\! &\mapsto \underline2
&
(\mathbf5,\underline3')\! &\mapsto \underline3'
\end{align*}
The box objects $\cP[k] = \Tr_\cC(\underline{2}^{\otimes k})$ of the associated anchored planar algebra are then given by
\begin{align*}
\cP[0] &= (\mathbf1,\underline1) + (\mathbf3,\underline3) + (\mathbf3,\underline3') + (\mathbf5,\underline1)
\\
\cP[1] &= (\mathbf2,\underline1) + (\mathbf2,\underline3) + (\mathbf2,\underline3') + (\mathbf4,\underline1) + (\mathbf4, \underline3) + (\mathbf4,\underline3')
\\
\cP[2] &= (\mathbf1,\underline1) + (\mathbf1,\underline3) + (\mathbf1,\underline3') + 2(\mathbf3,\underline1) + 2(\mathbf3,\underline3) + 2(\mathbf3,\underline3') + (\mathbf5,\underline1) + (\mathbf5,\underline3) + (\mathbf5,\underline3')
\\
\cP[k] &= 3 \cP[k-2] \text{ for }k\geq 3.
\end{align*}
\end{ex}

\begin{ex}[$E_6$]
\label{sub-ex:E6}
Let $\cC$ be the TLJ $A_{11}$ category, with simple objects $\mathbf{1}$, $\mathbf{2}$, $\dots$, $\mathbf{11}$, and let $\cM$ be its $E_6$ module tensor category.
We write
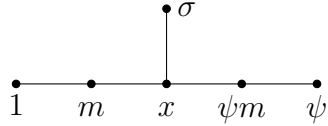
\begin{equation*}
\begin{tikzpicture}
	\filldraw (0,0) circle (.05cm) node [below] {$1$};
	\filldraw (1,0) circle (.05cm) node [below, yshift=-3] {$m$};
	\filldraw (2,0) circle (.05cm) node [below, yshift=-3] {$x$};
	\filldraw (3,0) circle (.05cm) node [below] {$\psi m$};
	\filldraw (4,0) circle (.05cm) node [below] {$\psi$};
	\filldraw (2,1) circle (.05cm) node [right] {$\sigma$};
	\draw (0,0) -- (4,0);
	\draw (2,0) -- (2,1);
\end{tikzpicture} 
\end{equation*}
for the simple objects of $\cM$.
By \cite[Cor.\,4.9]{MR1815993}, $\cZ(E_6)\cong \cC\boxtimes \cD = A_{11} \boxtimes \overline{A}_3$,
where $\cD=\overline{A}_3=\langle 1, \sigma, \psi\rangle\subset\cM$.
Using the $E_6$ fusion rules $m\otimes y=\bigoplus_{\substack{z\,:\,\text{neighbour}\\ \text{of $y$ in $\begin{tikzpicture}[scale=.14]
	\filldraw (0,0) circle (.1cm);
	\filldraw (1,0) circle (.1cm);
	\filldraw (2,0) circle (.1cm);
	\filldraw (3,0) circle (.1cm);
	\filldraw (4,0) circle (.1cm);
	\filldraw (2,1) circle (.1cm);
	\draw[very thin] (0,0) -- (4,0);
	\draw[very thin] (2,0) -- (2,1);
\end{tikzpicture}$}}} \,z$,\vspace{-.15cm} one easily computes\vspace{-.15cm}
\begin{equation}\label{eq: powers of m}
\begin{split}
m^{\otimes 2k}\,\,\;\,\, &= a_{k} 1 \,+\, b_{k} x \,+\, (a_{k}-1) \psi
\\
m^{\otimes 2k+1} &= a_{k+1} m \,+\, b_{k} \sigma \,+\, (a_{k+1}-1) \psi m
\end{split}
\end{equation}
where $a_{k}$ and $b_{k}$ are given by the recursive formula
$a_{k}= a_{k-1}+b_{k-1}$, $b_{k} = b_{k-1}+2a_{k}-1$,
$a_0=1$, $b_0 = 0$.
The functor $\Tr_\cC:\cM\to\cC$ is given by
\begin{equation}\label{eq: tables of Tr_C}
\begin{tabular}{lll}
$\Tr_\cC(1) \,=\, \mathbf{1} + \mathbf{7}$
&
$\Tr_\cC(m) = \mathbf{2} + \mathbf{6} + \mathbf{8}$
&
$\Tr_\cC(x) = \mathbf{3} + \mathbf{5} + \mathbf{7} + \mathbf{9}$
\\[.5mm]
$\Tr_\cC(\psi m) = \mathbf{4} + \mathbf{6} + \mathbf{10}$
&
$\Tr_\cC(\psi)\, = \mathbf{5} + \mathbf{11}$
&
$\Tr_\cC(\sigma) = \mathbf{4} + \mathbf{8}$
\\[.5mm]
\end{tabular}
\end{equation}
(see \cite[Sub-Ex.\,3.13 \& Rem.\,3.14]{1509.02937}).
Combining \eqref{eq: powers of m} and \eqref{eq: tables of Tr_C}, we obtain the following explicit formulas for the box objects of the associated anchored planar algebra:
\begin{align*}
\cP[2k] &= \Tr_\cC\big(a_{k}1 + b_{k}x + (a_{k}-1)\psi\big)
         \\&=  a_{k}(\mathbf{1}{+}\mathbf{7}) + b_{k}(\mathbf{3}{+}\mathbf{5}{+}\mathbf{7}{+}\mathbf{9}) + (a_{k}{-}1)(\mathbf{5}{+}\mathbf{11})
         \\&=  a_{k}\mathbf{1}  +  b_{k}\mathbf{3}  +  (a_{k}{+}b_{k}{-}1)\mathbf{5}  +  (a_{k}{+}b_{k})\mathbf{7}  +  b_{k}\mathbf{9}  +  (a_{k}{-}1)\mathbf{11}
\\
\cP[2k{+}1] &= \Tr_\cC\big(a_{k+1} m + b_{k} \sigma + (a_{k+1}-1) \psi m\big)
           \\&= a_{k+1}(\mathbf{2}{+}\mathbf{6}{+}\mathbf{8}) + b_{k}(\mathbf{4}{+}\mathbf{8}) + (a_{k+1}{-}1)(\mathbf{4}{+}\mathbf{6}{+}\mathbf{10})
           \\&= a_{k+1}\mathbf{2}  +  (a_{k+1}{+}b_{k}{-}1)\mathbf{4}  + (2a_{k+1}{-}1)\mathbf{6}  +  (a_{k+1}{+}b_{k})\mathbf{8}  +  (a_{k+1}{-}1)\mathbf{10}
\end{align*}
\end{ex}

\begin{ex}[$E_6$]
\label{sub-ex:E6 bis}
Let us now view $\cM=E_6$ as a module tensor category over $\cD=\overline{A}_3\subset \cZ(\cM)$.
The functor $\Phi_\cD:\cD\to \cM$ is fully faithful, so its adjoint $\Tr_\cD:\cM\to \cD$ is the projection functor onto the subcategory $\langle 1,\sigma,\psi\rangle\subset \cM$.
Using formula \eqref{eq: powers of m} for the powers of $m$, the box objects of the associated anchored planar algebra are easily computed:
\[
\cP[2k] = \Tr_\cD(m^{\otimes 2k})= a_{k}1+ (a_{k}-1)\psi,\qquad\,\,\,\,
\cP[2k+1] = \Tr_\cD(m^{\otimes 2k+1}) =  b_{k}\sigma.
\]
\end{ex}

\begin{ex}[$D_{10}{-}E_7$]
Let $\cC$ be the TLJ $A_{17}$ category, and let $A = \mathbf1+\mathbf9+\mathbf{17}$ be the algebra object whose category of modules is $E_7$.
Let $\cM$ be the category of $A$-$A$-bimodules.
This category is Morita equivalent to $\cC$, with ${\sf Mod}_\cC(A)$ as the equivalence bimodule.
Its center is thus given by $\cZ(\cM)=\cZ(\cC)=A_{17}\boxtimes \overline{A}_{17}$,
where the second $A_{17}$ is equipped with the opposite braiding \cite[Thm.\,7.10]{MR1966525}.

By identifying $A$-$A$-bimodules with $A\otimes A^{op}$-modules (braided tensor product), and computing $A\otimes A^{op}\in\cC$ (this object has a unique algebra structure), one can show that $\cM$ is equivalent to $D_{10}\oplus E_7$ as a $\cC$-module category (and also as a $\overline\cC$-module category).
Finally, by playing around with quantum dimensions and using the existence of a dominant functor $A_{17}\boxtimes \overline{A}_{17}\to \cM$, one obtains the following structure for $\cM$:
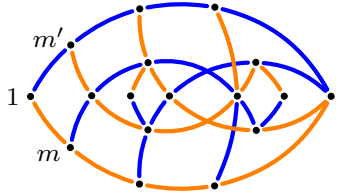
$$
\begin{tikzpicture}[scale=2, baseline = -3]
	\useasboundingbox (-.2,-.9) rectangle (2.2,.9);
	\coordinate (x) at (0,0);
	\node (A) [circle through=(x)] at (1,.5) {};
	\node (B) [circle through=(x)] at (1,-.5) {};
	\coordinate (y) at (intersection 2 of A and B);
	\node (C) [circle through=(y)] at (1.46,.51) {};
	\node (D) [circle through=(y)] at (1.46,-.51) {};
	\coordinate (z) at (intersection 1 of C and D);
	\coordinate (w) at ($(y)!.58!(z)$);
	\node (E) [circle through=(w)] at (.89,.4) {};
	\node (F) [circle through=(w)] at (.89,-.4) {};
	\coordinate (v) at (intersection 1 of E and F);
	\coordinate (u) at ($(v)!.5!(z)$);
	\coordinate (s) at (intersection 1 of C and F);
	\coordinate (t) at (intersection 2 of D and E);
	\coordinate (r) at ($(w)!.5!(y)$);
	\coordinate (p) at (intersection 1 of D and E);
	\coordinate (q) at (intersection 2 of C and F);
	\coordinate (a) at (intersection 1 of E and B);
	\coordinate (b) at (intersection 1 of C and B);
	\coordinate (c) at (intersection 2 of F and A);
	\coordinate (d) at (intersection 2 of D and A);
	\node[circle, inner sep=20] (W1) at (b) {};
	\node[circle, inner sep=20] (W2) at (d) {};
	\coordinate (m) at (intersection 1 of W2 and A);
	\coordinate (n) at (intersection 2 of W1 and B);
	\foreach \x in {a,b,c,d,n,m,p,q,r,s,t,u,v,w,x,y,z}
	{\filldraw (\x) circle (.02);}
\begin{scope}[blue, ultra thick, line cap=round, shorten >=3, shorten <=3]
	\draw(x)node[left, black, xshift=1.5, scale=1.2]{$\scriptstyle 1$} to[bend left=9] (a);
	\draw (a)node[left, black, xshift=3.5, yshift=3, scale=1.2]{$\scriptstyle m'$} to[bend left=10] (b);
	\draw (b) to[bend left=9] (n);
	\draw (n) to[bend left=23] (y);
	\draw (y) to[bend right=19] (p);
	\draw (p) to[bend right=19] (z);
	\draw (z) to[bend right=6] (t);
	\draw (t) to[bend left=6] (u);
	\draw (t) to[bend right=10] (d);
	\draw(c) to[bend left=11] (v);
	\draw(v) to[bend left=15] (s);
	\draw(s) to[bend left=27] (w);
	\draw(w) to[bend left=9] (q);
	\draw(r) to[bend left=9] (q);
	\draw(w) to[bend left=10] (m);
\end{scope}
\begin{scope}[orange, ultra thick, line cap=round, shorten >=3, shorten <=3]
	\draw(x) to[bend left=-9] (c);
	\draw (c)node[left, black, xshift=2, yshift=-3, scale=1.2]{$\scriptstyle m$} to[bend left=-10] (d);
	\draw (d) to[bend left=-9] (m);
	\draw (m) to[bend left=-23] (y);
	\draw (y) to[bend right=-19] (q);
	\draw (q) to[bend right=-19] (z);
	\draw (z) to[bend right=-6] (s);
	\draw (s) to[bend left=-6] (u);
	\draw (s) to[bend right=-10] (b);
	\draw(a) to[bend left=-11] (v);
	\draw(v) to[bend left=-15] (t);
	\draw(t) to[bend left=-27] (w);
	\draw(w) to[bend left=-9] (p);
	\draw(r) to[bend left=-9] (p);
	\draw(w) to[bend left=-10] (n);
\end{scope}
\end{tikzpicture}
$$
This category has $17$ simple objects, represented by the vertices of the above graph.
It is generated by two symmetrically self-dual objects $m,m'\in\cM$, and
the operations of fusion with $m$ and with $m'$ are encoded by the orange and blue edges, respectively.
This fusion category was studied in \cite[\S5.3.5]{MR2670999}, where it was realised as the category of $K$-$K$ bi-unitary connections for $K$ the graph $E_7$.

The inclusion $\cC\to\cC\,\boxtimes\,\cC=\cZ(\cM)$ of the first copy of $\cC$ (the one corresponding to the blue edges) equips $\cM$ with the structure of a $\cC$-module tensor category.
Let $\cP$ be the anchored planar algebra in $\cC$ associated to the pair $(\cM,m)$.
Its odd box objects $\cP[2k+1]$ are all zero, and its even box objects start as follows:
\[
\begin{split}
&\cP[0] = \mathbf1 + \mathbf{17}
\\&\cP[2] = \mathbf1 + \mathbf9 + \mathbf{17}
\\&\cP[4] = 2{\cdot}\mathbf1 + \mathbf5 + 3{\cdot}\mathbf9 + \mathbf{13} + 2{\cdot}\mathbf{17}
\\&\cP[6] = 5{\cdot}\mathbf1 + 5{\cdot}\mathbf5 + \mathbf7 + 9{\cdot}\mathbf9 + \mathbf{11} + 5{\cdot}\mathbf{13} + 5{\cdot}\mathbf{17}
\\&\cP[8] = 14{\cdot}\mathbf1 + \mathbf3 + 20{\cdot}\mathbf5 + 7{\cdot}\mathbf7 + 29{\cdot}\mathbf9 + 7{\cdot}\mathbf{11} + 20{\cdot}\mathbf{13} + \mathbf{15} + 14{\cdot}\mathbf{17}
\\&\cP[10] = 42{\cdot}\mathbf1 + 9{\cdot}\mathbf3 + 75{\cdot}\mathbf5 + 36{\cdot}\mathbf7 + 99{\cdot}\mathbf9 + 36{\cdot}\mathbf{11} + 75{\cdot}\mathbf{13} + 9{\cdot}\mathbf{15} + 42{\cdot}\mathbf{17}
\end{split}
\]
The latter are computed by taking the various powers of $m$:
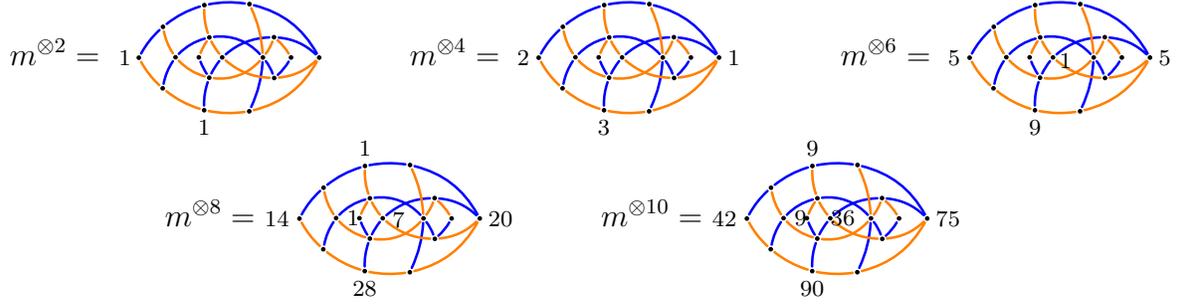
\[
\begin{split}
&
m^{\otimes 2} =
\begin{tikzpicture}[scale=2, baseline = -3] \useasboundingbox (-.2,-.6) rectangle (1.8,.4); \begin{scope}[transform canvas={scale=.6}] \coordinate (x) at (0,0); \node (A) [circle through=(x)] at (1,.5) {}; \node (B) [circle through=(x)] at (1,-.5) {}; \coordinate (y) at (intersection 2 of A and B); \node (C) [circle through=(y)] at (1.46,.51) {}; \node (D) [circle through=(y)] at (1.46,-.51) {}; \coordinate (z) at (intersection 1 of C and D); \coordinate (w) at ($(y)!.58!(z)$); \node (E) [circle through=(w)] at (.89,.4) {}; \node (F) [circle through=(w)] at (.89,-.4) {}; \coordinate (v) at (intersection 1 of E and F); \coordinate (u) at ($(v)!.5!(z)$); \coordinate (s) at (intersection 1 of C and F); \coordinate (t) at (intersection 2 of D and E); \coordinate (r) at ($(w)!.5!(y)$); \coordinate (p) at (intersection 1 of D and E); \coordinate (q) at (intersection 2 of C and F); \coordinate (a) at (intersection 1 of E and B); \coordinate (b) at (intersection 1 of C and B); \coordinate (c) at (intersection 2 of F and A); \coordinate (d) at (intersection 2 of D and A); \node[circle, inner sep=20] (W1) at (b) {}; \node[circle, inner sep=20] (W2) at (d) {}; \coordinate (m) at (intersection 1 of W2 and A); \coordinate (n) at (intersection 2 of W1 and B);
\foreach \x in {a,b,c,d,n,m,p,q,r,s,t,u,v,w,x,y,z} {\filldraw (\x) circle (.02);} \begin{scope}[blue, ultra thick, line cap=round, shorten >=3, shorten <=3] \draw(x) to[bend left=9] (a); \draw (a) to[bend left=10] (b); \draw (b) to[bend left=9] (n); \draw (n) to[bend left=23] (y); \draw (y) to[bend right=19] (p); \draw (p) to[bend right=19] (z); \draw (z) to[bend right=6] (t); \draw (t) to[bend left=6] (u); \draw (t) to[bend right=10] (d); \draw(c) to[bend left=11] (v); \draw(v) to[bend left=15] (s); \draw(s) to[bend left=27] (w); \draw(w) to[bend left=9] (q); \draw(r) to[bend left=9] (q); \draw(w) to[bend left=10] (m);
\end{scope}
\begin{scope}[orange, ultra thick, line cap=round, shorten >=3, shorten <=3]
	\draw(x)node[left, black, xshift=3.5, scale=1.8]{$\scriptstyle 1$} to[bend left=-9] (c);
	\draw (c) to[bend left=-10] (d);
	\draw (d)node[below, black, yshift=2, scale=1.8]{$\scriptstyle 1$} to[bend left=-9] (m);
	\draw (m) to[bend left=-23] (y);
	\draw (y) to[bend right=-19] (q);
	\draw (q) to[bend right=-19] (z);
	\draw (z) to[bend right=-6] (s);
	\draw (s) to[bend left=-6] (u);
	\draw (s) to[bend right=-10] (b);
	\draw(a) to[bend left=-11] (v);
	\draw(v) to[bend left=-15] (t);
	\draw(t) to[bend left=-27] (w);
	\draw(w) to[bend left=-9] (p);
	\draw(r) to[bend left=-9] (p);
	\draw(w) to[bend left=-10] (n);
\end{scope}\end{scope}
\end{tikzpicture}
m^{\otimes 4} =
\begin{tikzpicture}[scale=2, baseline = -3] \useasboundingbox (-.2,-.6) rectangle (1.8,.4); \begin{scope}[transform canvas={scale=.6}]
\coordinate (x) at (0,0); \node (A) [circle through=(x)] at (1,.5) {}; \node (B) [circle through=(x)] at (1,-.5) {}; \coordinate (y) at (intersection 2 of A and B); \node (C) [circle through=(y)] at (1.46,.51) {}; \node (D) [circle through=(y)] at (1.46,-.51) {}; \coordinate (z) at (intersection 1 of C and D); \coordinate (w) at ($(y)!.58!(z)$); \node (E) [circle through=(w)] at (.89,.4) {}; \node (F) [circle through=(w)] at (.89,-.4) {}; \coordinate (v) at (intersection 1 of E and F); \coordinate (u) at ($(v)!.5!(z)$); \coordinate (s) at (intersection 1 of C and F); \coordinate (t) at (intersection 2 of D and E); \coordinate (r) at ($(w)!.5!(y)$); \coordinate (p) at (intersection 1 of D and E); \coordinate (q) at (intersection 2 of C and F); \coordinate (a) at (intersection 1 of E and B); \coordinate (b) at (intersection 1 of C and B); \coordinate (c) at (intersection 2 of F and A); \coordinate (d) at (intersection 2 of D and A); \node[circle, inner sep=20] (W1) at (b) {}; \node[circle, inner sep=20] (W2) at (d) {}; \coordinate (m) at (intersection 1 of W2 and A); \coordinate (n) at (intersection 2 of W1 and B);
\foreach \x in {a,b,c,d,n,m,p,q,r,s,t,u,v,w,x,y,z} {\filldraw (\x) circle (.02);} \begin{scope}[blue, ultra thick, line cap=round, shorten >=3, shorten <=3] \draw(x) to[bend left=9] (a); \draw (a) to[bend left=10] (b); \draw (b) to[bend left=9] (n); \draw (n) to[bend left=23] (y); \draw (y) to[bend right=19] (p); \draw (p) to[bend right=19] (z); \draw (z) to[bend right=6] (t); \draw (t) to[bend left=6] (u); \draw (t) to[bend right=10] (d); \draw(c) to[bend left=11] (v); \draw(v) to[bend left=15] (s); \draw(s) to[bend left=27] (w); \draw(w) to[bend left=9] (q); \draw(r) to[bend left=9] (q); \draw(w) to[bend left=10] (m);
\end{scope}
\begin{scope}[orange, ultra thick, line cap=round, shorten >=3, shorten <=3]
	\draw(x)node[left, black, xshift=2.5, scale=1.8]{$\scriptstyle 2$} to[bend left=-9] (c);
	\draw (c) to[bend left=-10] (d);
	\draw (d)node[below, black, yshift=2, scale=1.8]{$\scriptstyle 3$} to[bend left=-9] (m);
	\draw (m) to[bend left=-23] (y);
	\draw (y) node[right, black, xshift=-3, scale=1.8]{$\scriptstyle 1$} to[bend right=-19] (q);
	\draw (q) to[bend right=-19] (z);
	\draw (z) to[bend right=-6] (s);
	\draw (s) to[bend left=-6] (u);
	\draw (s) to[bend right=-10] (b);
	\draw(a) to[bend left=-11] (v);
	\draw(v) to[bend left=-15] (t);
	\draw(t) to[bend left=-27] (w);
	\draw(w) to[bend left=-9] (p);
	\draw(r) to[bend left=-9] (p);
	\draw(w) to[bend left=-10] (n);
\end{scope}\end{scope}
\end{tikzpicture}
\quad
m^{\otimes 6} =
\begin{tikzpicture}[scale=2, baseline = -3] \useasboundingbox (-.2,-.6) rectangle (1.4,.4); \begin{scope}[transform canvas={scale=.6}] \coordinate (x) at (0,0); \node (A) [circle through=(x)] at (1,.5) {}; \node (B) [circle through=(x)] at (1,-.5) {}; \coordinate (y) at (intersection 2 of A and B); \node (C) [circle through=(y)] at (1.46,.51) {}; \node (D) [circle through=(y)] at (1.46,-.51) {}; \coordinate (z) at (intersection 1 of C and D); \coordinate (w) at ($(y)!.58!(z)$); \node (E) [circle through=(w)] at (.89,.4) {}; \node (F) [circle through=(w)] at (.89,-.4) {}; \coordinate (v) at (intersection 1 of E and F); \coordinate (u) at ($(v)!.5!(z)$); \coordinate (s) at (intersection 1 of C and F); \coordinate (t) at (intersection 2 of D and E); \coordinate (r) at ($(w)!.5!(y)$); \coordinate (p) at (intersection 1 of D and E); \coordinate (q) at (intersection 2 of C and F); \coordinate (a) at (intersection 1 of E and B); \coordinate (b) at (intersection 1 of C and B); \coordinate (c) at (intersection 2 of F and A); \coordinate (d) at (intersection 2 of D and A); \node[circle, inner sep=20] (W1) at (b) {}; \node[circle, inner sep=20] (W2) at (d) {}; \coordinate (m) at (intersection 1 of W2 and A); \coordinate (n) at (intersection 2 of W1 and B);
\foreach \x in {a,b,c,d,n,m,p,q,r,s,t,u,v,w,x,y,z} {\filldraw (\x) circle (.02);} \begin{scope}[blue, ultra thick, line cap=round, shorten >=3, shorten <=3] \draw(x) to[bend left=9] (a); \draw (a) to[bend left=10] (b); \draw (b) to[bend left=9] (n); \draw (n) to[bend left=23] (y); \draw (y) to[bend right=19] (p); \draw (p) to[bend right=19] (z); \draw (z) to[bend right=6] (t); \draw (t) to[bend left=6] (u); \draw (t) to[bend right=10] (d); \draw(c) to[bend left=11] (v); \draw(v) to[bend left=15] (s); \draw(s) to[bend left=27] (w); \draw(w) to[bend left=9] (q); \draw(r) to[bend left=9] (q); \draw(w) to[bend left=10] (m);
\end{scope}
\begin{scope}[orange, ultra thick, line cap=round, shorten >=3, shorten <=3]
	\draw(x)node[left, black, xshift=2.5, scale=1.8]{$\scriptstyle 5$} to[bend left=-9] (c);
	\draw (c) to[bend left=-10] (d);
	\draw (d)node[below, black, yshift=2, scale=1.8]{$\scriptstyle 9$} to[bend left=-9] (m);
	\draw (m) to[bend left=-23] (y);
	\draw (y) node[right, black, xshift=-3, scale=1.8]{$\scriptstyle 5$} to[bend right=-19] (q);
	\draw (q) to[bend right=-19] (z);
	\draw (z) node[black, xshift=8, yshift=-2, scale=1.8]{$\scriptstyle 1$} to[bend right=-6] (s);
	\draw (s) to[bend left=-6] (u);
	\draw (s) to[bend right=-10] (b);
	\draw(a) to[bend left=-11] (v);
	\draw(v) to[bend left=-15] (t);
	\draw(t) to[bend left=-27] (w);
	\draw(w) to[bend left=-9] (p);
	\draw(r) to[bend left=-9] (p);
	\draw(w) to[bend left=-10] (n);
\end{scope}\end{scope}
\end{tikzpicture}
\\&\qquad\qquad\quad
m^{\otimes 8} =\,
\begin{tikzpicture}[scale=2, baseline = -3] \useasboundingbox (-.2,-.6) rectangle (1.8,.4); \begin{scope}[transform canvas={scale=.6}] \coordinate (x) at (0,0); \node (A) [circle through=(x)] at (1,.5) {}; \node (B) [circle through=(x)] at (1,-.5) {}; \coordinate (y) at (intersection 2 of A and B); \node (C) [circle through=(y)] at (1.46,.51) {}; \node (D) [circle through=(y)] at (1.46,-.51) {}; \coordinate (z) at (intersection 1 of C and D); \coordinate (w) at ($(y)!.58!(z)$); \node (E) [circle through=(w)] at (.89,.4) {}; \node (F) [circle through=(w)] at (.89,-.4) {}; \coordinate (v) at (intersection 1 of E and F); \coordinate (u) at ($(v)!.5!(z)$); \coordinate (s) at (intersection 1 of C and F); \coordinate (t) at (intersection 2 of D and E); \coordinate (r) at ($(w)!.5!(y)$); \coordinate (p) at (intersection 1 of D and E); \coordinate (q) at (intersection 2 of C and F); \coordinate (a) at (intersection 1 of E and B); \coordinate (b) at (intersection 1 of C and B); \coordinate (c) at (intersection 2 of F and A); \coordinate (d) at (intersection 2 of D and A); \node[circle, inner sep=20] (W1) at (b) {}; \node[circle, inner sep=20] (W2) at (d) {}; \coordinate (m) at (intersection 1 of W2 and A); \coordinate (n) at (intersection 2 of W1 and B);
\foreach \x in {a,b,c,d,n,m,p,q,r,s,t,u,v,w,x,y,z} {\filldraw (\x) circle (.02);} \begin{scope}[blue, ultra thick, line cap=round, shorten >=3, shorten <=3] \draw(x) to[bend left=9] (a); \draw (a) to[bend left=10] (b); \draw (b) to[bend left=9] (n); \draw (n) to[bend left=23] (y); \draw (y) to[bend right=19] (p); \draw (p) to[bend right=19] (z); \draw (z) to[bend right=6] (t); \draw (t) to[bend left=6] (u); \draw (t) to[bend right=10] (d); \draw(c) to[bend left=11] (v); \draw(v) to[bend left=15] (s); \draw(s) to[bend left=27] (w); \draw(w) to[bend left=9] (q); \draw(r) to[bend left=9] (q); \draw(w) to[bend left=10] (m);
\end{scope}
\begin{scope}[orange, ultra thick, line cap=round, shorten >=3, shorten <=3]
	\draw(x)node[left, black, xshift=2, scale=1.8]{$\scriptstyle 14$} to[bend left=-9] (c);
	\draw (c) to[bend left=-10] (d);
	\draw (d)node[below, black, yshift=2, scale=1.8]{$\scriptstyle 28$} to[bend left=-9] (m);
	\draw (m) to[bend left=-23] (y);
	\draw (y) node[right, black, xshift=-3, scale=1.8]{$\scriptstyle 20$} to[bend right=-19] (q);
	\draw (q) to[bend right=-19] (z);
	\draw (z) node[black, xshift=10, yshift=-1, scale=1.8]{$\scriptstyle 7$} to[bend right=-6] (s);
	\draw (s) to[bend left=-6] (u) node[black, xshift=-4, scale=1.8]{$\scriptstyle 1$};
	\draw (s) to[bend right=-10] (b) node[above, black, yshift=-2, scale=1.8]{$\scriptstyle 1$};
	\draw(a) to[bend left=-11] (v);
	\draw(v) to[bend left=-15] (t);
	\draw(t) to[bend left=-27] (w);
	\draw(w) to[bend left=-9] (p);
	\draw(r) to[bend left=-9] (p);
	\draw(w) to[bend left=-10] (n);
\end{scope}\end{scope}
\end{tikzpicture}
\quad
m^{\otimes 10} =\,
\begin{tikzpicture}[scale=2, baseline = -3] \useasboundingbox (-.2,-.6) rectangle (1.8,.4); \begin{scope}[transform canvas={scale=.6}] \coordinate (x) at (0,0); \node (A) [circle through=(x)] at (1,.5) {}; \node (B) [circle through=(x)] at (1,-.5) {}; \coordinate (y) at (intersection 2 of A and B); \node (C) [circle through=(y)] at (1.46,.51) {}; \node (D) [circle through=(y)] at (1.46,-.51) {}; \coordinate (z) at (intersection 1 of C and D); \coordinate (w) at ($(y)!.58!(z)$); \node (E) [circle through=(w)] at (.89,.4) {}; \node (F) [circle through=(w)] at (.89,-.4) {}; \coordinate (v) at (intersection 1 of E and F); \coordinate (u) at ($(v)!.5!(z)$); \coordinate (s) at (intersection 1 of C and F); \coordinate (t) at (intersection 2 of D and E); \coordinate (r) at ($(w)!.5!(y)$); \coordinate (p) at (intersection 1 of D and E); \coordinate (q) at (intersection 2 of C and F); \coordinate (a) at (intersection 1 of E and B); \coordinate (b) at (intersection 1 of C and B); \coordinate (c) at (intersection 2 of F and A); \coordinate (d) at (intersection 2 of D and A); \node[circle, inner sep=20] (W1) at (b) {}; \node[circle, inner sep=20] (W2) at (d) {}; \coordinate (m) at (intersection 1 of W2 and A); \coordinate (n) at (intersection 2 of W1 and B);
\foreach \x in {a,b,c,d,n,m,p,q,r,s,t,u,v,w,x,y,z} {\filldraw (\x) circle (.02);} \begin{scope}[blue, ultra thick, line cap=round, shorten >=3, shorten <=3] \draw(x) to[bend left=9] (a); \draw (a) to[bend left=10] (b); \draw (b) to[bend left=9] (n); \draw (n) to[bend left=23] (y); \draw (y) to[bend right=19] (p); \draw (p) to[bend right=19] (z); \draw (z) to[bend right=6] (t); \draw (t) to[bend left=6] (u); \draw (t) to[bend right=10] (d); \draw(c) to[bend left=11] (v); \draw(v) to[bend left=15] (s); \draw(s) to[bend left=27] (w); \draw(w) to[bend left=9] (q); \draw(r) to[bend left=9] (q); \draw(w) to[bend left=10] (m);
\end{scope}
\begin{scope}[orange, ultra thick, line cap=round, shorten >=3, shorten <=3]
	\draw(x)node[left, black, xshift=2, scale=1.8]{$\scriptstyle 42$} to[bend left=-9] (c);
	\draw (c) to[bend left=-10] (d);
	\draw (d)node[below, black, yshift=2, scale=1.8]{$\scriptstyle 90$} to[bend left=-9] (m);
	\draw (m) to[bend left=-23] (y);
	\draw (y) node[right, black, xshift=-3, scale=1.8]{$\scriptstyle 75$} to[bend right=-19] (q);
	\draw (q) to[bend right=-19] (z);
	\draw (z) node[black, xshift=8, scale=1.8]{$\scriptstyle 36$} to[bend right=-6] (s);
	\draw (s) to[bend left=-6] (u) node[black, xshift=-4, scale=1.8]{$\scriptstyle 9$};
	\draw (s) to[bend right=-10] (b) node[above, black, yshift=-2, scale=1.8]{$\scriptstyle 9$};
	\draw(a) to[bend left=-11] (v);
	\draw(v) to[bend left=-15] (t);
	\draw(t) to[bend left=-27] (w);
	\draw(w) to[bend left=-9] (p);
	\draw(r) to[bend left=-9] (p);
	\draw(w) to[bend left=-10] (n);
\end{scope}\end{scope}
\end{tikzpicture}
\end{split}
\]
and applying the functor $\Tr_\cC:\cM\to \cC$.

\end{ex}



\section{Constructing anchored planar algebras}
\label{sec:Constructing anchored planar algebras}

In this section we show that, in order to construct an anchored planar algebra, it is enough to provide the data associated to a certain small set of tangles (the ones which appear in Definition \ref{defn:GeneratingTangles}),
and to verify a certain list of relations.

\begin{thm}
\label{thm:ConstructAPA}
Let $\cC$ be a braided pivotal category and let $\cP=(\cP[n])_{n\ge 0}$ be a sequence of objects of $\cC$.
Let 
\[
\eta:1\to\cP[0]\quad\,\,\, \alpha_i:\cP[n+2]\to\cP[n]\quad\,\,\,\bar\alpha_i:\cP[n]\to\cP[n+2]\quad\,\,\, \varpi_{i,j}:\cP[n]\otimes\cP[j]\to\cP[n+j]
\]
be morphisms (numbering and indices as in Definition \ref{defn:GeneratingTangles}) satisfying the following relations (the first seven are similar to \ref{reln:id} -- \ref{reln:EasyQuadratic}; the last two 
are related to the two anchor dependence axioms, and are discussed in Remark~\ref{rem: explanation of 7th and 8th rels} below):

\begin{enumerate}[label={\rm(C\arabic*)}]
\item
\label{reln:UnitMap}
$\varpi_{0,j}\circ (\eta\otimes \id_j ) = \id_j$, $\varpi_{i,0}\circ (\id_n\otimes \eta) = \id_n$
\item
\label{reln:CapMaps}
$\alpha_i \circ \alpha_j = \alpha_{j-2}\circ \alpha_i$ for $i+1< j$
\item
\label{reln:CupMaps}
$\bar \alpha_i \circ \bar \alpha_j = \bar \alpha_{j+2}\circ \bar \alpha_i$ for $i\leq j$
\item
\label{reln:CapCupMaps}
$\displaystyle \alpha_i \circ \bar \alpha_j = 
\begin{cases}
\bar \alpha_{j-2}\circ \alpha_i & \text{for } i<j-1
\\
\id_n & \text{for } i=j\pm1
\\
\bar \alpha_j \circ \alpha_{i-2} & \text{for } i>j+1
\end{cases}
$
\item
\label{reln:CapQuadraticMaps}
$\displaystyle
\alpha_i \circ \varpi_{j,k} = 
\begin{cases}
\varpi_{j-2,k}\circ(\alpha_{i}\otimes\id_j) & \text{for }  i+1<j 
\\
\varpi_{j,k-2}\circ(\id_n\otimes\alpha_{i-j}) &\text{for }j<i+1<j+k
\\
\varpi_{j,k}\circ(\alpha_{i-k}\otimes \id_j) & \text{for } i+1> j+k
\end{cases}
$
\item
\label{reln:CupQuadraticMaps}
$\displaystyle
\bar \alpha_i \circ \varpi_{j,k} = 
\begin{cases}
\varpi_{j+2,k}\circ(\bar \alpha_i\otimes\id_j) & \text{for } i\leq j
\\
\varpi_{j,k+2}\circ(\id_n\otimes\bar \alpha_{i-j}) & \text{for } j\leq i\leq j+k
\\
\varpi_{j,k}\circ(\bar \alpha_{i-k}\otimes\id_j) & \text{for } i\geq j+k
\end{cases}
$
\item
\label{reln:EasyQuadraticMaps}
$\varpi_{i+j,k} \circ(\varpi_{i, j+\ell}\otimes \id_k) = \varpi_{i, j+k+\ell} \circ(\id_{i+m}\otimes \varpi_{j,k})$
\item
\label{reln:HardQuadraticMaps}
$\varpi_{i+j+k,\ell}\circ(\varpi_{i,j} \otimes \id_{\ell}) = \varpi_{i,j}\circ(\varpi_{i+k,\ell} \otimes \id_{j})\circ(\id_{i+k+m} \otimes \beta_{j,\ell})$
\item
\label{reln:RotationThetaMaps}
$\alpha_{n}\circ \cdots \circ \alpha_{2n-1} \circ \varpi_{n,n}\circ[(\bar\alpha_{n-1}\circ \cdots \circ \bar\alpha_0 \circ\eta)\otimes \id_n ] = \theta_{\cP[n]}$,
\end{enumerate}
where the twist $\theta$ is as in~\eqref{def:theta1}, $\id_n$ stands for $\id_{\cP[n]}$, and $\beta_{j,\ell}$ stands for $\beta_{\cP[j],\cP[\ell]}$.

Then the assignment $T\mapsto Z(T)$ given in Algorithm~\ref{alg:AssignMap} below endows $\cP$ with the structure of an anchored planar algebra in $\cC$.
\end{thm}

The proof of Theorem \ref{thm:ConstructAPA} will be presented at the end of Section~\ref{sec:WellDefined}.

\begin{rems}\label{rem: explanation of 7th and 8th rels}
\mbox{}
\begin{enumerate}[label=(\arabic*)]
\item
The relation \ref{reln:HardQuadraticMaps} above is a variant of the first anchor dependence relation in Definition~\ref{def: anchored planar algebra}.
Consider the following non-generic tangle:
$$
\begin{tikzpicture}
	\pgfmathsetmacro{\high}{1.2};
	\pgfmathsetmacro{\hoffset}{.8};
	\pgfmathsetmacro{\low}{-.5};
	\pgfmathsetmacro{\medium}{.5};
	
	\draw ($ {-2}*(\hoffset,0)+(0,\low)+(0,.2)$) -- ($ {-2}*(\hoffset,0)+(0,\high) $);
	\node at ($ {-2}*(\hoffset,0)+(0,\high) + (0,.2) $) {\scriptsize{$i$}};
	\draw ($ {-1}*(\hoffset,0)+(0,\medium)+(0,.2)$) -- ($ {-1}*(\hoffset,0)+(0,\high) $);
	\node at ($ {-1}*(\hoffset,0)+(0,\high) + (0,.2) $) {\scriptsize{$j$}};
	\draw ($ {0}*(\hoffset,0)+(0,\low)+(0,.2)$) -- ($ {0}*(\hoffset,0)+(0,\high) $);
	\node at ($ {0}*(\hoffset,0)+(0,\high) + (0,.2) $) {\scriptsize{$k$}};
	\draw ($ {1}*(\hoffset,0)+(0,\medium)+(0,.2)$) -- ($ {1}*(\hoffset,0)+(0,\high) $);
	\node at ($ {1}*(\hoffset,0)+(0,\high) + (0,.2) $) {\scriptsize{$\ell$}};
	\draw ($ {2}*(\hoffset,0)+(0,\low)+(0,.2)$) -- ($ {2}*(\hoffset,0)+(0,\high) $);
	\node at ($ {2}*(\hoffset,0)+(0,\high) + (0,.2) $) {\scriptsize{$m$}};
	\draw[thick, red] ($ (0,\low) + (0,-.2) $) -- (0,-\high);
	\draw[thick, red] ($ (\hoffset, \medium) + (0,-.2) $) arc (180:270:.2cm) -- ($ 2*(\hoffset,0) + (0,\medium) + (0,-.4) $) arc (90:-90:.6cm) -- ($ (.1,.1) + (0,-\high) $) arc (90:180:.1cm);
	\draw[thick, red] ($ (-\hoffset, \medium) + (0,-.2) $) arc (180:270:.4cm) -- ($ 2*(\hoffset,0) + (0,\medium) + (0,-.6) $) arc (90:-90:.4cm) -- ($ (.3,.3) + (0,-\high) $) arc (90:180:.3cm);
	\roundNbox{unshaded}{(0,\low)}{.2}{{2*\hoffset}}{{2*\hoffset}}{}
	\roundNbox{unshaded}{(-\hoffset,\medium)}{.2}{.2}{.2}{}
	\roundNbox{unshaded}{(\hoffset,\medium)}{.2}{.2}{.2}{}
	\roundNbox{}{(0,0)}{\high}{\high}{\high}{}
\end{tikzpicture}
$$
which we can make generic in the following two ways:
$$
\begin{tikzpicture}[baseline=-.1cm]
	\pgfmathsetmacro{\high}{1.2};
	\pgfmathsetmacro{\hoffset}{.8};
	\pgfmathsetmacro{\low}{-.5};
	\pgfmathsetmacro{\medium}{.3};
	\pgfmathsetmacro{\mediumhigh}{.7};
	
	\draw ($ {-2}*(\hoffset,0)+(0,\low)+(0,.2)$) -- ($ {-2}*(\hoffset,0)+(0,\high) $);
	\node at ($ {-2}*(\hoffset,0)+(0,\high) + (0,.2) $) {\scriptsize{$i$}};
	\draw ($ {-1}*(\hoffset,0)+(0,\medium)+(0,.2)$) -- ($ {-1}*(\hoffset,0)+(0,\high) $);
	\node at ($ {-1}*(\hoffset,0)+(0,\high) + (0,.2) $) {\scriptsize{$j$}};
	\draw ($ {0}*(\hoffset,0)+(0,\low)+(0,.2)$) -- ($ {0}*(\hoffset,0)+(0,\high) $);
	\node at ($ {0}*(\hoffset,0)+(0,\high) + (0,.2) $) {\scriptsize{$k$}};
	\draw ($ {1}*(\hoffset,0)+(0,\mediumhigh)+(0,.2)$) -- ($ {1}*(\hoffset,0)+(0,\high) $);
	\node at ($ {1}*(\hoffset,0)+(0,\high) + (0,.2) $) {\scriptsize{$\ell$}};
	\draw ($ {2}*(\hoffset,0)+(0,\low)+(0,.2)$) -- ($ {2}*(\hoffset,0)+(0,\high) $);
	\node at ($ {2}*(\hoffset,0)+(0,\high) + (0,.2) $) {\scriptsize{$m$}};
	\draw[thick, red] ($ (0,\low) + (0,-.2) $) -- (0,-\high);
	\draw[thick, red] ($ (\hoffset, \mediumhigh) + (0,-.2) $) arc (180:270:.2cm) -- ($ 2*(\hoffset,0) + (0,\mediumhigh) + (0,-.4) $) arc (90:0:.6cm) -- ($ 2*(\hoffset,0) + (.6,\low) $) arc (0:-90:.6cm) -- ($ (.1,.1) + (0,-\high) $) arc (90:180:.1cm);
	\draw[thick, red] ($ (-\hoffset, \medium) + (0,-.2) $) arc (180:270:.2cm) -- ($ 2*(\hoffset,0) + (0,\medium) + (0,-.4) $) arc (90:-90:.4cm) -- ($ (.3,.3) + (0,-\high) $) arc (90:180:.3cm);
	\roundNbox{unshaded}{(0,\low)}{.2}{{2*\hoffset}}{{2*\hoffset}}{}
	\roundNbox{unshaded}{(-\hoffset,\medium)}{.2}{.2}{.2}{}
	\roundNbox{unshaded}{(\hoffset,\mediumhigh)}{.2}{.2}{.2}{}
	\roundNbox{}{(0,0)}{\high}{\high}{\high}{}
\end{tikzpicture}\
\,\,\,\text{ and }\,\,\,
\begin{tikzpicture}[baseline=-.1cm]
	\pgfmathsetmacro{\high}{1.2};
	\pgfmathsetmacro{\hoffset}{.8};
	\pgfmathsetmacro{\low}{-.5};
	\pgfmathsetmacro{\medium}{.7};
	\pgfmathsetmacro{\mediumhigh}{.3};
	
	\draw ($ {-2}*(\hoffset,0)+(0,\low)+(0,.2)$) -- ($ {-2}*(\hoffset,0)+(0,\high) $);
	\node at ($ {-2}*(\hoffset,0)+(0,\high) + (0,.2) $) {\scriptsize{$i$}};
	\draw ($ {-1}*(\hoffset,0)+(0,\medium)+(0,.2)$) -- ($ {-1}*(\hoffset,0)+(0,\high) $);
	\node at ($ {-1}*(\hoffset,0)+(0,\high) + (0,.2) $) {\scriptsize{$j$}};
	\draw ($ {0}*(\hoffset,0)+(0,\low)+(0,.2)$) -- ($ {0}*(\hoffset,0)+(0,\high) $);
	\node at ($ {0}*(\hoffset,0)+(0,\high) + (0,.2) $) {\scriptsize{$k$}};
	\draw ($ {1}*(\hoffset,0)+(0,\mediumhigh)+(0,.2)$) -- ($ {1}*(\hoffset,0)+(0,\high) $);
	\node at ($ {1}*(\hoffset,0)+(0,\high) + (0,.2) $) {\scriptsize{$\ell$}};
	\draw ($ {2}*(\hoffset,0)+(0,\low)+(0,.2)$) -- ($ {2}*(\hoffset,0)+(0,\high) $);
	\node at ($ {2}*(\hoffset,0)+(0,\high) + (0,.2) $) {\scriptsize{$m$}};
	\draw[thick, red] ($ (0,\low) + (0,-.2) $) -- (0,-\high);
	\draw[thick, red] ($ (\hoffset, \mediumhigh) + (0,-.2) $) arc (180:270:.1cm) -- ($ 2*(\hoffset,0) + (0,\mediumhigh) + (0,-.3) $) arc (90:-90:.5cm) -- ($ (.2,.2) + (0,-\high) $) arc (90:180:.2cm);
	\draw[thick, red] ($ (-\hoffset, \medium) + (0,-.2) $) -- ($  (-\hoffset, \mediumhigh) + (0,-.2) $) arc (180:270:.2cm) -- ($ 2*(\hoffset,0) + (0,\mediumhigh) + (0,-.4) $) arc (90:-90:.4cm) -- ($ (.3,.3) + (0,-\high) $) arc (90:180:.3cm);
	\roundNbox{unshaded}{(0,\low)}{.2}{{2*\hoffset}}{{2*\hoffset}}{}
	\roundNbox{unshaded}{(-\hoffset,\medium)}{.2}{.2}{.2}{}
	\roundNbox{unshaded}{(\hoffset,\mediumhigh)}{.2}{.2}{.2}{}
	\roundNbox{}{(0,0)}{\high}{\high}{\high}{}
\end{tikzpicture}
$$
The left and right tangles corresponds to the left and right hand sides of \ref{reln:HardQuadraticMaps}.
\item
The relation \ref{reln:RotationThetaMaps} above is a variant of the twist anchor dependence axiom:
$$
Z\left(
\begin{tikzpicture}[baseline=-.1cm, yscale = -1]
	\draw [thick, red] (0,0) -- (90:1cm);
	\draw[] (-90:.3cm) .. controls ++(270:.3cm) and ++(270:.5cm) .. (0:.5cm) .. controls ++(90:.8cm) and ++(90:.8cm) .. (180:.7cm) .. controls ++(270:.6cm) and ++(90:.4cm) .. (270:1cm);
	\draw[very thick] (0,0) circle (1cm);
	\draw[unshaded, very thick] (0,0) circle (.3cm);
	\node at (-80:.8cm) {\scriptsize{$n$}};
\end{tikzpicture}
\right)
=\theta_{\cP[n]}
$$
\end{enumerate}
\end{rems}

\subsection{Generic tangles and tangles in standard form}
\label{sec:StandardForm}

The moduli space of anchored planar tangles is an infinite dimensional manifold $\mathbb M$ modelled on a Fr\'echet space.
In this section, we define a stratification of $\mathbb M$ and study the associated notion of \emph{generic} anchored planar tangle.
We then prove that any two generic anchored planar tangles can be connected by a finite sequence of certain moves.
Later on, we define a notion of anchored planar tangle with \emph{underlying tangle in standard form},
and describe a procedure which takes a generic anchored planar tangle and brings it into standard form.

Recall that, before stating Definitions \ref{def:planar tangle} and \ref{defn:AnchoredPlanarTangle} (the definitions of planar tangles and of anchored planar tangles), we had assumed the data of a point $q\in\partial \mathbb D$,
and of certain collections of germs of arcs $\mu_n=(\mu_n^{1},\ldots,\mu_n^{n})$.
It will be convenient to make a number of technical restrictions to the above data:
we assume that $q$ is at the bottom of $\mathbb D$ (angle $-\pi/2$),
that the arcs $\mu_n^i$ cross the boundary of $\mathbb D$ orthogonally, that they have everywhere non-zero curvature,
and that they intersect $\partial\mathbb D$ only in the top hemisphere (angles strictly between $0$ and $\pi$).

\begin{defn}\label{def: generic APT}
An anchored planar tangle is called \emph{generic} if it satisfies:
\begin{enumerate}[label=(G\arabic*)]
\item
\label{G:1}
the $y$-coordinates of the local maxima/minima of strings as well as those of the centers of the input disks are all distinct.
\item
\label{G:2}
if a string has a horizontal tangent at some point, then the second derivative at that point does not vanish.
\item
\label{G:3}
no boundary point $p\in\partial X$ is on the equator of a circle (angle 0 or $\pi$).
\item
\label{G:4}
no anchor point $q_i$ is at the top of a circle (angle $\pi/2$).
\end{enumerate}
\end{defn}

We let $\mathbb M_0\subset \mathbb M$ denote the open dense subset consisting of anchored planar tangles that satisfy \ref{G:1} -- \ref{G:4}.
The simplest way in which a tangle can fail to be generic is if:

\begin{enumerate}[label=(S\arabic*)]
\item
\label{S:Fail1}
the $y$-coordinates of two centers of circles and/or critical points of strands coincide; the tangle otherwise satisfies 
\ref{G:2}, \ref{G:3}, \ref{G:4}.
\item
\label{S:Fail2}
a strand has a horizontal inflection point (and the third derivative at that point does not vanish); the tangle otherwise satisfies 
\ref{G:1}, \ref{G:3}, \ref{G:4}.
\item
\label{S:Fail3}
a strand is attached on the equator of a circle; the tangle otherwise satisfies
\ref{G:1}, \ref{G:2}, \ref{G:4}.
\item
\label{S:Fail4}
an anchor point is on the top of its circle; the tangle otherwise satisfies
\ref{G:1}, \ref{G:2}, \ref{G:3}.
\end{enumerate}

Each of the above conditions \ref{S:Fail1} -- \ref{S:Fail4} determines a submanifold of $\mathbb M$ of codimension one (it is locally determined by one equation).
Let $\mathbb M_1$ denote their union.
Its boundary $\mathbb M_{\ge2}:=\partial \mathbb M_1\subset \mathbb M$ has codimension $\geq 2$, leading to a stratification
\[
\mathbb M \,=\, \mathbb M_0\cup \mathbb M_1\cup \mathbb M_{\ge 2}
\]
of the moduli space of anchored planar tangles.

Let $T_1,T_2\in\mathbb M_0$ be two generic tangles which are isotopic.
An isotopy between $T_1$ and $T_2$ is a path $[0,1]\to\mathbb M$.
By transversality, such a path can be isotoped to one which avoids $\mathbb M_{\ge 2}$ and which intersects $\mathbb M_1$ transversely.
Each transversal intersection correspond to one the following moves:
\begin{enumerate}[label=(M\arabic*)]
\item
\label{M:1}
exchange the relative order of the $y$-coordinates of the centers of two input discs.
\item
\label{M:2}
exchange the order of the $y$-coordinates of the center of a disc and of a critical point of a string
(two cases, depending on whether the critical point is a cup  or a cap).
\item
\label{M:3}
exchange the relative order of the heights of two critical points of strings
(four sub-cases: cup/cup, cup/cap, cap/cup, cap/cap).
\item
\label{M:4}
replace a cup followed by a cap by a string that goes straight down : $\tikz[scale=.6, baseline=-4]{\draw (-.1,-.5) to[out=90,in=-90] (-.15,-.1) to[out=90,in=115, looseness=2] (0,0) to[out=-65,in=-90, looseness=2] (.15,.1) to[out=90,in=-90] (.1,.5);} \to \tikz[scale=.6, baseline=-4]{\draw (0,-.5) -- (0,.5);}$\,\,
(this move comes in two flavors which are each other's mirror images).
\item
\label{M:5}
move a strand's boundary point past the equator of an input circle
(two variants, depending on whether the boundary point is on the left or on the right of the circle).
\item
\label{M:6}
swing an anchor point from slightly to the right to slightly to the left of the north pole.
\end{enumerate}

\noindent The moves \ref{M:1} -- \ref{M:3} correspond to crossing a component of the stratum \ref{S:Fail1},
and the moves \ref{M:4} -- \ref{M:6} correspond to crossing the strata \ref{S:Fail2} -- \ref{S:Fail4}, respectively.

The condition about the arcs $\mu_n$ having non-zero curvature is needed in the following technical lemma:
\begin{lem}
Let $\gamma:[0,1]\to\mathbb M$ be a smooth path which avoids $\mathbb M_{\ge 2}$ and intersects $\mathbb M_1$ transversely.
Then $\gamma$ intersects $\mathbb M_1$ in finitely many points.
\end{lem}
\begin{proof}
First of all, there can be at most finitely many \ref{M:1} events,
as any accumulation point of such events would be a point where $\gamma$ fails to be transverse to the submanifold $\mathbb M_1$.
Similarly, there can be at most finitely many occurrences of the moves \ref{M:5} and \ref{M:6}.

The situation for the moves \ref{M:2} -- \ref{M:4} is different, and the lemma would be false in the absence of the curvature condition on the arcs $\mu_n$.

Write $(T_t,X_t,Q_t,A_t)$ for the tangle $\gamma(t)$.
Let $\underline X\subset [0,1]\times\mathbb D$ be the $2$-manifold which is the union of all the $\{t\}\times X_t$, and let
\[
\partial \underline X:=\bigcup_{t\in[0,1]}\{t\}\times \partial X_t\qquad\qquad
\underline {\mathring X}:=\underline X\setminus \partial \underline X=\bigcup_{t\in[0,1]}\{t\}\times \mathring X_t.
\]
By Morse theory, the cups and caps form a $1$-dimensional submanifold $\mathring C\subset\underline {\mathring X}$.
The projection map $\mathring C\to [0,1]$ is then itself a Morse function, whose critical points are the \ref{M:4} events.
Let $C$ be the closure of $\mathring C$ in $\underline X$.
By our condition that the $\mu_n$ cross $\partial \mathbb D$ orthogonally, one can identify the points of $C\cap \partial \underline X$ with the \ref{M:5} events.
Moreover, because of our condition on the curvature of the $\mu_n$, $C$ is a submanifold of $\underline X$ (in the sense of manifolds with boundary),
and $C\cap \partial \underline X$ consists of regular points for the projection to $[0,1]$.
The projection $C\to [0,1]$ is therefore Morse on all of $C$.

Since $C$ is a manifold and since it is compact, the projection map $C\to [0,1]$ has at most finitely critical points.
Our path $\gamma$ thererfore sees at most finitely many \ref{M:4} events.

Finally, there can be at most finitely many \ref{M:2} -- \ref{M:3} events.
An accumulation point of such events would either be a point of $\mathring C$ where 
the transversality condition with $\mathbb M_1$ fails, or a point of $\partial C$ where the condition \ref{G:1}
fails in the definition of \ref{S:Fail3}\,$\subset \mathbb M_1$.
\end{proof}

We have proved:

\begin{prop}
Any two generic anchored planar tangles can be connected by a finite sequence of the moves \ref{M:1} -- \ref{M:6}.\hfill $\square$
\end{prop}

Generic tangles can be brought to a certain standard form by an algorithm that we now describe.

\begin{defn}\label{def: anchored planar tangle : underlying tangle in standard form}
A generic anchored planar tangle is said to have \emph{underlying tangle in standard form} if:
\begin{itemize}
\item
all the anchor points are on the bottom of the disks (by our conditions on the $\mu_n$,
this implies that all the strings are attached to the top halves of the input disks).
\item
the projections onto the $y$-axis of all the input disks, and of all the local maxima/minima of the strings are disjoint.
\end{itemize}
\end{defn}

\begin{figure}[!ht]
$$
\begin{tikzpicture}[scale=.55]
	\clip (-6.15,-6.15) rectangle (8.15,6.15);
	\roundNbox{unshaded}{(0,0)}{6}{0}{0}{};
	\draw (-3.5,6)--(-3.5,4);
	\draw (-3,4) -- (-3,5) arc (180:0:.5cm) -- (-2,4);
	\draw (-1.5,4) -- (-1.5,4.5) arc (180:0:.5cm) --(-.5,4) .. controls ++(270:1cm) and ++(90:1cm) .. (0.5,2);
	\draw (-2,-3.5)--(-2,-3) arc (0:180:1cm) -- (-4,-3) .. controls ++(270:5cm) and ++(270:5cm) .. (5,0) arc (0:180:.5cm) -- (4,0); 
	\draw (1.5,2) -- (1.5,6);
	\draw (3.3,0) -- (3.3,2) arc (0:180:.5cm);
	\draw (-2.5, -3.5)--(-2.5,-3) arc(0:180:.5cm) .. controls ++(270:3cm) and ++(270:3cm) .. (2,-2.5) arc (0:180:1cm)  -- (0,-3.5);	
	\draw (-5,6) .. controls ++(270:7cm) and ++(90:7cm) .. (-1,-3.5);
	\roundNbox{unshaded}{(-2.5,4)}{.4}{1}{1}{};
	\roundNbox{unshaded}{(-1,-3.5)}{.4}{1.5}{1.5}{};
	\roundNbox{unshaded}{(3.3,-.4)}{.4}{1}{1}{};
	\roundNbox{unshaded}{(1.5,1.6)}{.4}{1}{1}{};
	\draw[thick, red] (1.5,1.2) .. controls ++(270:1cm) and ++(-90:2cm) ..
	                          (-.8,1.9) .. controls ++(90:1cm) and ++(-90:1cm) .. 
	                          (-.7,3.5) .. controls ++(90:1cm) and ++(0:1cm) ..
	                          (-2.4,5)  .. controls ++(180:3.2cm) and ++(90:1cm) .. (-5,1.5) --
	                          (-5,-1.5) .. controls ++(270:4.5cm) and ++(120:1.5cm) .. (0,-6);
	\draw[thick, red] (-1,-3.9) .. controls ++(270:1cm) and ++(90:1cm) .. (0,-6);
	\draw[thick, red] (3.3,-.8) .. controls ++(270:4cm) and ++(60:1cm) .. (0,-6);
	\draw[thick, red] (-2.5,3.6)  .. controls ++(270:5cm) and ++(90:4cm) .. (1.5,-3) .. controls ++(270:2cm) and ++(70:1cm) .. (0,-6);
	\filldraw[red] (-2.5,3.6) circle (.1cm);
	\filldraw[red] (-1,-3.9) circle (.1cm);
	\filldraw[red] (3.3,-.8) circle (.1cm);
	\filldraw[red] (1.5,1.2) circle (.1cm);
	\filldraw[red] (0,-6) circle (.1cm);
\draw (8.07,-6) -- (8.07,6); \draw[blue, dashed, very thin] (8,5.5) -- +(-10.5,0); \fill[blue] (8,5.5) circle (.07); \draw[blue, dashed, very thin] (8,5) -- +(-9,0); \fill[blue] (8,5) circle (.07); \filldraw[blue, dashed, very thin] (8,4.4) -- +(-9.3,0); \filldraw[blue, dashed, very thin] (8,3.6) -- +(-9.3,0); \draw[blue, line width=2.2, line cap=round] (8,3.6) -- (8,4.4); \draw[blue, dashed, very thin] (8,2.5) -- +(-5.3,0); \fill[blue] (8,2.5) circle (.07); \filldraw[blue, dashed, very thin] (8,2) -- +(-5.3,0); \filldraw[blue, dashed, very thin] (8,1.2) -- +(-5.3,0); \draw[blue, line width=2.2, line cap=round] (8,2) -- (8,1.2); \draw[blue, dashed, very thin] (8,0.5) -- +(-3.5,0); \fill[blue] (8,0.5) circle (.07); \filldraw[blue, dashed, very thin] (8,0) -- +(-3.5,0); \filldraw[blue, dashed, very thin] (8,-.8) -- +(-3.5,0); \draw[blue, line width=2.2, line cap=round] (8,0) -- (8,-.8); \draw[blue, dashed, very thin] (8,-1.5) -- +(-7,0); \fill[blue] (8,-1.5) circle (.07); \draw[blue, dashed, very thin] (8,-2) -- +(-11,0); \fill[blue] (8,-2) circle (.07); \draw[blue, dashed, very thin] (8,-2.5) -- +(-11,0); \fill[blue] (8,-2.5) circle (.07); \filldraw[blue, dashed, very thin] (8,-3.1) -- +(-7.5,0); \filldraw[blue, dashed, very thin] (8,-3.9) -- +(-7.5,0); \draw[blue, line width=2.2, line cap=round] (8,-3.1) -- (8,-3.9); \draw[blue, dashed, very thin] (8,-5.01) -- +(-9,0); \fill[blue] (8,-5.01) circle (.07); \draw[blue, dashed, very thin] (8,-5.57) -- +(-9.3,0); \fill[blue] (8,-5.57) circle (.07); \end{tikzpicture}
$$
\caption{An example of an anchored planar tangle with underlying tangle in standard form (except that the circles are drawn more like rectangles).
The $y$-projections of the input circles and of the maxima/minima of strings are indicated in blue. They are all disjoint, as required by the definition.
}
\label{fig:StandardForm}
\end{figure}
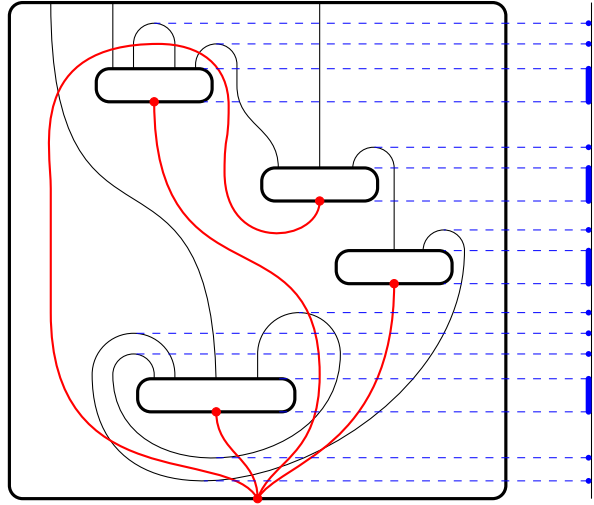

The following algorithm takes as input a generic anchored planar tangle (Definition \ref{def: generic APT}) and produces as output an anchored planar tangle whose underlying tangle is in standard form (Definition \ref{def: anchored planar tangle : underlying tangle in standard form}).
The output of the algorithm is well defined up to isotopy within the space of anchored planar tangles whose underlying tangle is in standard form.

\begin{alg}[Standard form]
\label{alg:StandardForm}
Let $T$ be a generic anchored planar tangle.
\begin{enumerate}[label=(\arabic*)]
\item
First, we shrink all input circles so that their $y$-projections do not overlap with those of the other input disks, nor with
those of the local maxima/minima of the strings.
The strings attached to these circles are extended by straight lines:

$$
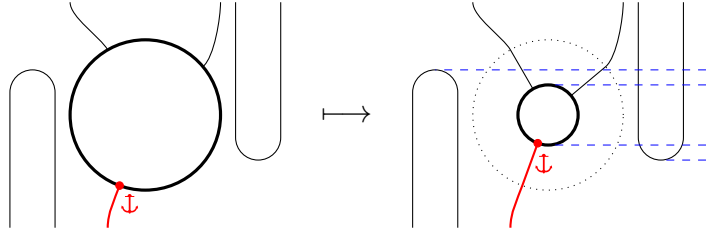
\begin{tikzpicture}[baseline = -.1cm]
	\draw (40:.5cm) -- (40:1cm) .. controls ++(40:.2cm) and ++(270:.2cm) .. (1,1.5);
	\draw (120:.5cm) -- (120:1cm) .. controls ++(120:.2cm) and ++(270:.2cm) .. (-1,1.5);
\pgftransformrotate{180}
	\draw[thick, red] (70:.5cm) -- (70:1cm) .. controls ++(70:.2cm) and ++(270:.2cm) .. (.5,1.5);
\pgftransformrotate{180}
	\draw[very thick, unshaded] (0,0) circle (1cm);
\pgftransformrotate{180}
	\filldraw[red] (70:1cm) circle (.05cm);
	\node at (80:1.2cm) {\anchor};
\pgftransformrotate{180}
	\draw (-1.8,-1.5) -- (-1.8,.3) arc (180:0:.3cm) -- (-1.2,-1.5);
	\draw (1.8,1.5) -- (1.8,-.3) arc (0:-180:.3cm) -- (1.2,1.5);
\end{tikzpicture}
\quad\longmapsto\quad
\begin{tikzpicture}[baseline = -.1cm]
	\draw (40:.4cm) -- (40:1cm) .. controls ++(40:.2cm) and ++(270:.2cm) .. (1,1.5);
	\draw (120:.4cm) -- (120:1cm) .. controls ++(120:.2cm) and ++(270:.2cm) .. (-1,1.5);
\pgftransformrotate{180}
	\draw[thick, red] (70:.4cm) -- (70:1cm) .. controls ++(70:.2cm) and ++(270:.2cm) .. (.5,1.5);
\pgftransformrotate{180}
	\draw[very thick, unshaded] (0,0) circle (.4cm);
\pgftransformrotate{180}
	\filldraw[red] (70:.4cm) circle (.05cm);
	\node at (85:.65cm) {\anchor};
\pgftransformrotate{180}
	\draw (-1.8,-1.5) -- (-1.8,.3) arc (180:0:.3cm) -- (-1.2,-1.5);
	\draw (1.8,1.5) -- (1.8,-.3) arc (0:-180:.3cm) -- (1.2,1.5);
	\draw[dotted] (0,0) circle (1cm);
\draw[blue, dashed, very thin]
(2.1,.4) -- (0,.4)
(2.1,-.4) -- (0,-.4)
(2.1,-.6) -- (1.5,-.6)
(2.1,.6) -- (-1.5,.6);
\end{tikzpicture}
$$

\item
Next, we rotate the input circles so as to let the anchor points sink to the bottom.
The direction of rotation depends on whether the anchor point is in the left half of right half of the input circle.
The boundary points rotate by the same amount (by our choice of $\mu_n$, this implies that they end up above the equator).
At the same time, we further shrink the input circles, so as make sure that the new critical $y$-values satisfy the conditions of a tangle in standard form
(Definition \ref{def: anchored planar tangle : underlying tangle in standard form}),
and that no unnecessary new maxima/minima are introduced.

$$
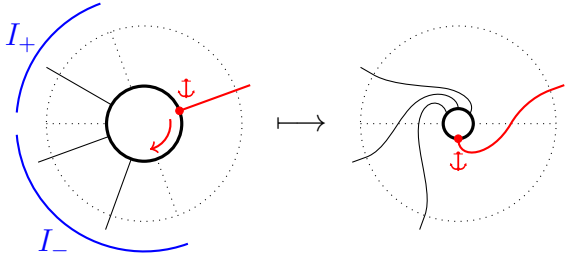
\begin{tikzpicture}[baseline = -.1cm]
	\draw[dotted] (110:1.3cm) -- (-70:1.3cm)(180:1.3cm) -- (0,0);
	\draw[dotted] (0,0) circle (1.3cm);
	\draw[very thick, unshaded] (0,0) circle (.5cm);
	\filldraw[red] (20:.5cm) circle (.05cm);
	\node at (40:.73cm) {\anchor};
	\draw[thick, red] (20:.5cm) -- (20:1.5cm);
	\draw (200:.5cm) -- (200:1.5cm);
	\draw (150:.5cm) -- (150:1.5cm);
	\draw (250:.5cm) -- (250:1.5cm);
	\draw[thick, blue] (-70:1.7cm) arc (-70:-175:1.7cm);
	\draw[thick, blue] (175:1.7cm) arc (175:110:1.7cm);
	\node at (-127:2cm) {\textcolor{blue}{$I_-$}};
	\node at (145:2cm) {\textcolor{blue}{$I_+$}};
\draw[red, thick, line join = round, ->] (10:.35) arc (10:-75:.35) -- ++(-159:.01);
\end{tikzpicture}
\,\,\,\longmapsto\,\,\,
\begin{tikzpicture}[baseline = -.1cm]
	\draw[dotted] (-1.3,0) -- (-.2,0) (.2,0) -- (1.3,0);
	\draw[very thick] (0,0) circle (.2cm);
	\filldraw[red] (270:.2cm) circle (.05cm);
	\node at (270:.5cm) {\anchor};
	\draw[thick, red] (20:1.5cm) -- (20:1.3cm) .. controls ++(200:.25cm) and ++(240-180:.2cm) .. (0:.7cm)  .. controls ++(240:.5cm) and ++(270:.3cm) .. (270:.2cm) ;
	\draw (250:1.5cm) -- (250:1.3cm) .. controls ++(200-130:.4cm) and ++(240-180-130-50:.4cm) .. (-20-130-50:.5cm)  .. controls ++(240-130-50:.2cm) and ++(270-130-40:.15cm) .. (140:.2cm) ;
	\draw (200:1.5cm) -- (200:1.3cm) .. controls ++(200-180:.25cm) and ++(240-10:.2cm) .. (-20-180+10:.72cm)  .. controls ++(240-180-5:.5cm) and ++(270-180:.2cm) .. (90:.2cm) ;
	\draw (150:1.5cm) -- (150:1.3cm) .. controls ++(200-230:.3cm) and ++(240-180-230-20:.2cm) .. (-20-230:.5cm)  .. controls ++(240-230-20:.2cm) and ++(270-230+20:.2cm) .. (40:.2cm);
	\draw[dotted] (0,0) circle (1.3cm);
\end{tikzpicture}
$$

No new local maxima/minima are created on the strings attached above the equator (attached in the blue interval $I_+$).
The strings attached below the equator (in the blue interval $I_-$) acquire exactly one new local maximum.
\end{enumerate}
\end{alg}

\subsection{The ribbon braid group}
\label{sec:RibbonBraidGroup}

The ribbon braid group $\cR\cB_n=\cB_n\ltimes\mathbb Z^n$ is the semi-direct product of the braid group $\cB_n$ with the free abelian group of rank $n$
with respect to the action that permutes the coordinates.

\begin{defn}
The ribbon braid group is the group $\cR\cB_n$ with generators $\varepsilon_1,\dots, \varepsilon_{n-1}$, $\vartheta_1,\dots, \vartheta_n$ and relations:
\begin{itemize}
\item 
$\varepsilon_i \varepsilon_{i+1} \varepsilon_i = \varepsilon_{i+1} \varepsilon_i \varepsilon_{i+1}$ for all $i=1,\dots, n-1$,
\item
$\varepsilon_i \varepsilon_j = \varepsilon_j \varepsilon_i $ if $|i-j|>1$,
\item
$\varepsilon_i \vartheta_i=\vartheta_{i+1}\varepsilon_i,\quad \varepsilon_i \vartheta_{i+1}=\vartheta_i\varepsilon_i$,
\item
$\varepsilon_i \vartheta_j=\vartheta_j\varepsilon_i$ for $j\not\in\{i,i+1\}$.
\end{itemize}\smallskip
\end{defn}

\noindent
The elements of $\cR\cB_n$ are best visualized as ribbons that braid among each other, and twirl on themselves:
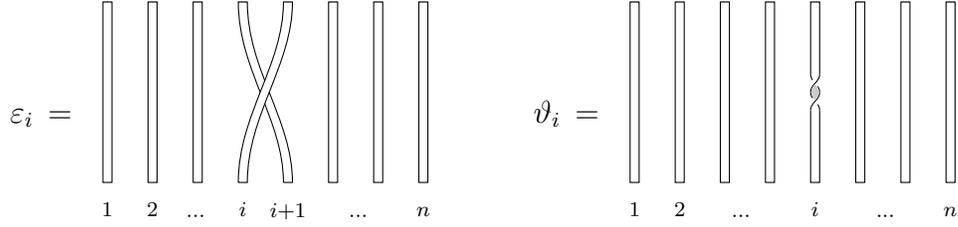
\[
\varepsilon_i\,=\,\,
\begin{tikzpicture}[baseline=.8cm, scale=1.2]
\draw (-.5,0) rectangle ++(.1,2);
\draw (0,0) rectangle ++(.1,2);
\draw (.5,0) rectangle ++(.1,2);
\filldraw[fill=white] (1.5,0) -- (1.6,0) .. controls ++(90:.7) and ++(-90:.75) .. (1.1,2) -- (1,2) .. controls ++(-90:.75) and ++(90:.7) .. (1.5,0);
\filldraw[fill=white] (1,0) -- (1.1,0) .. controls ++(90:.75) and ++(-90:.7) .. (1.6,2) -- (1.5,2) .. controls ++(-90:.7) and ++(90:.75) .. (1,0);
\draw (2,0) rectangle ++(.1,2);
\draw (2.5,0) rectangle ++(.1,2);
\draw (3,0) rectangle ++(.1,2);
\node[above] at (-.5 + .05,-.5) {$\scriptstyle 1$}; 
\node[above] at (.05,-.5) {$\scriptstyle 2$}; 
\node[above] at (.5 + .05,-.5-.05) {$\scriptstyle \cdots$}; 
\node[above] at (1 + .05,-.5) {$\scriptstyle i$}; 
\node[above] at (1.5 + .05,-.5-.03) {$\scriptstyle i+1$}; 
\node[above] at (2.3 + .05,-.5-.05) {$\scriptstyle \cdots$}; 
\node[above] at (3 + .05,-.5) {$\scriptstyle n$}; 
\end{tikzpicture}
\qquad\quad
\vartheta_i\,=\,\,
\begin{tikzpicture}[baseline=.8cm, scale=1.2]
\draw (-1,0) rectangle ++(.1,2);
\draw (-.5,0) rectangle ++(.1,2);
\draw (0,0) rectangle ++(.1,2);
\draw (.5,0) rectangle ++(.1,2);
\draw (1,.8) -- (1,0) -- (1.1,0) -- (1.1,.8) (1,1.2) -- (1,2) -- (1.1,2) -- (1.1,1.2);
\draw (1.1,.8) to[in=-90,out=90] (1,1)(1.1,1) to[in=-90,out=90] (1,1.2);
\draw[line width = 3, white] (1,.8) to[in=-90,out=90] (1.1,1)(1,1) to[in=-90,out=90] (1.1,1.2);
\draw (1,.8) to[in=-90,out=90] (1.1,1)(1,1) to[in=-90,out=90] (1.1,1.2);
\draw (1.5,0) rectangle ++(.1,2);
\draw (2,0) rectangle ++(.1,2);
\draw (2.5,0) rectangle ++(.1,2);
\node[above] at (-1 + .05,-.5) {$\scriptstyle 1$}; 
\node[above] at (-0.5 + .05,-.5) {$\scriptstyle 2$}; 
\node[above] at (.2 + .05,-.5-.05) {$\scriptstyle \cdots$}; 
\node[above] at (1 + .05,-.5) {$\scriptstyle i$}; 
\node[above] at (1.8 + .05,-.5-.05) {$\scriptstyle \cdots$}; 
\node[above] at (2.5 + .05,-.5) {$\scriptstyle n$}; 
\fill[gray!40] (1.05,.91) -- ++(.01,.013) arc (-45:45:.11) -- ++(-.01,.013) -- ++(-.01,-.013) arc (180-45:180+45:.11) -- ++(.01,-.013);
\end{tikzpicture}\vspace{-.2cm}
\]
We compose from bottom to top so that, for example,
$
\varepsilon_1\vartheta_1=\,
\tikz[baseline=-.2cm]{
\filldraw[fill=white] (1,-.6) -- (1.1,-.6) .. controls ++(90:.4) and ++(-90:.33) .. (.6,0) -- (.6,.1) (.6,.5) -- (.6,.6) -- (.5,.6) -- (.5,.5) (.5,.1) -- (.5,0) .. controls ++(-90:.4) and ++(90:.33) .. (1,-.6);
\draw (.6,.8-.7) to[in=-90,out=90] (.5,1-.7)(.6,1-.7) to[in=-90,out=90] (.5,1.2-.7);
\draw[line width = 3, white] (.5,.8-.7) to[in=-90,out=90] (.6,1-.7)(.5,1-.7) to[in=-90,out=90] (.6,1.2-.7);
\draw (.5,.8-.7) to[in=-90,out=90] (.6,1-.7)(.5,1-.7) to[in=-90,out=90] (.6,1.2-.7);
\filldraw[fill=white, xshift=31.5, rotate=180] (0,0) -- (0,-.6) -- (.1,-.6) -- (.1,0) .. controls ++(90:.33) and ++(-90:.4) .. (.6,.6) -- (.5,.6) .. controls ++(-90:.33) and ++(90:.4) .. (0,0);
\fill[gray!40] (.55,.21) -- ++(.01,.013) arc (-45:45:.11) -- ++(-.01,.013) -- ++(-.01,-.013) arc (180-45:180+45:.11) -- ++(.01,-.013);
}
$\,\,.\smallskip

Let $\mathfrak S_n$ be the symmetric group, and let $\sigma\mapsto\pi_\sigma:\cR\cB_n\to \mathfrak S_n$ be the obvious homomorphism.
Given objects $a_1,\ldots,a_n$ in a braided pivotal category, an element
$\sigma\in\cR\cB_n$ induces a map
\begin{equation} \label{action of ribbon braid on iterated tensor product}
P(\sigma):a_1\otimes\ldots\otimes a_n\to a_{\pi_\sigma^{-1}(1)}\otimes\ldots\otimes a_{\pi_\sigma^{-1}(n)}
\end{equation}
by the rules
$P(\varepsilon_i):=\id_{a_1\otimes\ldots\otimes a_{i-1}}\otimes \beta_{a_i,a_{i+1}}\otimes\id_{a_{i+2}\otimes\ldots\otimes a_n}$,
$P(\vartheta_i):=\id_{a_1\otimes\ldots\otimes a_{i-1}}\otimes \theta_{a_i}\otimes\id_{a_{i+1}\otimes\ldots\otimes a_n}$,
and $P(\sigma\tau)=P(\sigma)\circ P(\tau)$,
where $\beta$ is the braiding in $\cC$ and the twist $\theta$ is defined in (\ref{def:theta1}).

Let $T$ be a planar tangle whose input circles have disjoint $y$-projections.
Then there is an action of $\cR\cB_n$ on the set of systems of anchor lines on $T$.
We assume without loss of generality that $T$ has no strands (as the anchor lines don't interact with the strands).
Let $\sigma\in\cR\cB_n$ be a ribbon braid, and let $A$ be a systems of anchor lines.
To compute $\sigma\cdot A$, we place the ribbon braid vertically under the tangle and connect the $i$th ribbon to the $i$th input circle in order of increasing $y$-projection.
Then we let the input circles flow down following the ribbons.
The anchor lines follow by continuity while the circles braid and twist around each other.
Here is an example of this process
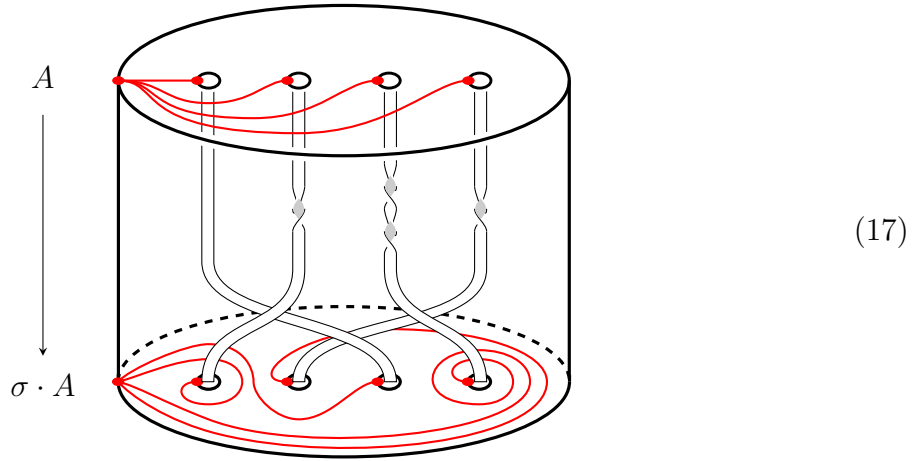
\begin{equation}\label{eq: Here is an example of this process}
\begin{matrix}
\begin{tikzpicture}[baseline=-.2cm]
\node (A1) at (-4,4.05) {$A$};
\node (A2) at (-4,-.05) {$\sigma\cdot A$};
\draw[-stealth, shorten <=6, shorten >=3] (A1) -- (A2);
\pgftransformrotate{-90}
\pgftransformyscale{1.5}
	\coordinate (a) at (0,0);
	\coordinate (b) at ($ (a) + (0,1.2)$);
	\coordinate (c) at ($ (a) + (0,.4)$);
	\coordinate (d) at ($ (a) + (0,-.4)$);
	\coordinate (e) at ($ (a) + (0,-1.2)$);
	\coordinate (f) at ($ (a) + (0,-2)$);
	\draw[very thick] (0,-2) arc (-90:90:1 and 2);
	\draw[very thick, dashed] (0,-2) arc (-90:-270:1 and 2);
	\draw[very thick] (0,-2) -- (-.1,-2)(0,2) -- (-.1,2);
	\ncircle{unshaded}{(b)}{.1}{270}{}
	\ncircle{unshaded}{(c)}{.1}{270}{}
	\ncircle{unshaded}{(d)}{.1}{270}{}
	\ncircle{unshaded}{(e)}{.1}{270}{}
	\draw[thick, red] ($ (b) + (0,-.1) $)  .. controls ++(270:.2cm) and ++(270:.3cm) ..  ($ (b) + (-.3,0) $) .. controls ++(90:.4cm) and ++(90:.4cm) .. ($ (b) + (.3,0) $) .. controls ++(270:.6cm) and ++(270:.5cm) ..  ($ (b) + (-.42,0) $) .. controls ++(90:.25cm) and ++(180:.25cm) ..  ($ (b) + (0,.45) $) .. controls ++(0:.7cm) and ++(60:1cm) .. ($ (e) + (.5,0) $) .. controls ++(240:.5cm) and ++(70:.2cm) .. (f);
	\draw[thick, red] ($ (c) + (0,-.1) $)  .. controls ++(270:.2cm) and ++(90:.3cm) ..  ($ (d) + (.5,0) $) .. controls ++(270:.2cm) and ++(0:.2cm) .. ($ (e) + (0,.4) $) .. controls ++(180:.2cm) and ++(90:.3cm) .. ($ (e) + (-.5,0) $) .. controls ++(270:.2cm) and ++(150:.2cm) .. (f);
	\draw[thick, red] ($ (e) + (0,-.1) $)  .. controls ++(270:.2cm) and ++(270:.3cm) ..  ($ (e) + (.3,0) $) .. controls ++(90:.4cm) and ++(90:.4cm) .. ($ (e) + (-.3,0) $) .. controls ++(270:.3cm) and ++(110:.2cm) .. (f);
	\draw[thick, red] ($ (d) + (0,-.1) $)  .. controls ++(270:.4cm) and ++(270:.8cm) ..  ($ (c) + (-.7,0) $) .. controls ++(90:.8cm) and ++(180:.4cm) .. ($ (b) + (0,.6) $) .. controls ++(0:.8cm) and ++(70:1.2cm) .. ($ (e) + (.7,.1) $) .. controls ++(250:.5cm) and ++(30:.2cm) .. (f) ;
	\filldraw[red] (f) circle (.05cm);
\pgftransformyscale{.666}
\pgftransformrotate{90}
\filldraw[fill=white] (1.9-.16,2) -- (1.9-.16,1.5) .. controls ++(-90:.8) and ++(90:.92) .. (-.68,0) -- (-.52,0) .. controls ++(90:.8) and ++(-90:.92) .. (1.9,1.5) -- (1.9,2) (1.9,2.6) --  (1.9,4) -- (1.9-.16,4) -- (1.9-.16,2.6);
\draw (1.9-.16,2) .. controls ++(90:.1) and ++(-90:.1) .. +(.16,.3) +(0,.3) .. controls ++(90:.1) and ++(-90:.1) .. +(.16,.6);
\draw[line width = 4, white] (1.9,2) .. controls ++(90:.1) and ++(-90:.1) .. +(-.16,.3) +(0,.3) .. controls ++(90:.1) and ++(-90:.1) .. +(-.16,.6);
\draw (1.9,2) .. controls ++(90:.1) and ++(-90:.1) .. +(-.16,.3) +(0,.3) .. controls ++(90:.1) and ++(-90:.1) .. +(-.16,.6);
\fill[gray!40] (1.9,2) ++ (-.08,.16) -- ++(.042,.056) arc (-40:40:.13) -- ++(-.042,.056) -- ++(-.042,-.056) arc (180-40:180+40:.13) -- ++(.042,-.056);
\filldraw[fill=white] (.7-.16,1.7) -- (.7-.16,1.5) .. controls ++(-90:.88) and ++(90:.8) .. (1.9-.16,0) -- (1.9,0) .. controls ++(90:.88) and ++(-90:.8) .. (.7,1.5) -- (.7,1.7) (.7,2.9) --  (.7,4) -- (.7-.16,4) -- (.7-.16,2.9);
\draw (.7-.16,1.7) .. controls ++(90:.1) and ++(-90:.1) .. +(.16,.3) +(0,.3) .. controls ++(90:.1) and ++(-90:.1) .. +(.16,.6);
\draw[line width = 4, white] (.7,1.7) .. controls ++(90:.1) and ++(-90:.1) .. +(-.16,.3) +(0,.3) .. controls ++(90:.1) and ++(-90:.1) .. +(-.16,.6);
\draw (.7,1.7) .. controls ++(90:.1) and ++(-90:.1) .. +(-.16,.3) +(0,.3) .. controls ++(90:.1) and ++(-90:.1) .. +(-.16,.6);
\fill[gray!40] (.7,1.7) ++ (-.08,.16) -- ++(.042,.056) arc (-40:40:.13) -- ++(-.042,.056) -- ++(-.042,-.056) arc (180-40:180+40:.13) -- ++(.042,-.056);
\draw (.7-.16,2.3) .. controls ++(90:.1) and ++(-90:.1) .. +(.16,.3) +(0,.3) .. controls ++(90:.1) and ++(-90:.1) .. +(.16,.6);
\draw[line width = 4, white] (.7,2.3) .. controls ++(90:.1) and ++(-90:.1) .. +(-.16,.3) +(0,.3) .. controls ++(90:.1) and ++(-90:.1) .. +(-.16,.6);
\draw (.7,2.3) .. controls ++(90:.1) and ++(-90:.1) .. +(-.16,.3) +(0,.3) .. controls ++(90:.1) and ++(-90:.1) .. +(-.16,.6);
\fill[gray!40] (.7,2.3) ++ (-.08,.16) -- ++(.042,.056) arc (-40:40:.13) -- ++(-.042,.056) -- ++(-.042,-.056) arc (180-40:180+40:.13) -- ++(.042,-.056);
\filldraw[fill=white] (-1.73-.16,1.6) .. controls ++(-90:.92) and ++(90:.8) .. (.7-.16,0) -- (.7,0) .. controls ++(90:.92) and ++(-90:.8) .. (-1.73,1.6) --  (-1.73,4) -- (-1.73-.16,4) -- (-1.73-.16,1.6);
\filldraw[fill=white] (-.52,2) -- (-.52,1.5) .. controls ++(-90:.9) and ++(90:.8) .. (-1.73,0) -- (-1.73-.16,0) .. controls ++(90:.9) and ++(-90:.8) .. (-.68,1.5) -- (-.68,2) (-.68,2.6) --  (-.68,4) -- (-.68+.16,4) -- (-.52,2.6);
\draw (-.68+.16,2) .. controls ++(90:.1) and ++(-90:.1) .. +(-.16,.3) +(0,.3) .. controls ++(90:.1) and ++(-90:.1) .. +(-.16,.6);
\draw[line width = 4, white](-.68,2) .. controls ++(90:.1) and ++(-90:.1) .. +(.16,.3) +(0,.3) .. controls ++(90:.1) and ++(-90:.1) .. +(.16,.6);
\draw (-.68,2) .. controls ++(90:.1) and ++(-90:.1) .. +(.16,.3) +(0,.3) .. controls ++(90:.1) and ++(-90:.1) .. +(.16,.6);
\fill[gray!40] (-.68+.16,2) ++ (-.08,.16) -- ++(.042,.056) arc (-40:40:.13) -- ++(-.042,.056) -- ++(-.042,-.056) arc (180-40:180+40:.13) -- ++(.042,-.056);
\pgftransformrotate{-90}
\pgftransformyscale{1.5}
\pgftransformxshift{-114}
	\coordinate (a) at (0,0);
	\coordinate (b) at ($ (a) + (0,1.205)$);
	\coordinate (c) at ($ (a) + (0,.4)$);
	\coordinate (d) at ($ (a) + (0,-.4)$);
	\coordinate (e) at ($ (a) + (0,-1.2)$);
	\coordinate (f) at ($ (a) + (0,-2)$);
	\draw[line width=5, white] (a) ellipse (1 and 2);
	\draw[very thick] (a) ellipse (1 and 2);
	\draw[very thick] (0,-2) -- (3.91,-2)(0,2) -- (3.91,2);
	\ncircle{double, draw=white, double=black, double distance=1.1, unshaded}{(b)}{.1}{270}{}
	\ncircle{double, draw=white, double=black, double distance=1.1, unshaded}{(c)}{.1}{270}{}
	\ncircle{double, draw=white, double=black, double distance=1.1, unshaded}{(d)}{.1}{270}{}
	\ncircle{double, draw=white, double=black, double distance=1.1, unshaded}{(e)}{.1}{270}{}
	\draw[thick, red] ($ (b) + (0,-.1) $)  .. controls ++(270:.3cm) and ++(90:.8cm) ..  ($ (d) + (.7,0) $) .. controls ++(270:1.5cm) and ++(90:.4cm) .. (f);
	\draw[thick, red] ($ (c) + (0,-.1) $)  .. controls ++(270:.3cm) and ++(90:.6cm) ..  ($ (e) + (.5,.4) $) .. controls ++(270:1cm) and ++(90:.4cm) .. (f);
	\draw[thick, red] ($ (d) + (0,-.1) $)  .. controls ++(270:.2cm) and ++(90:.4cm) ..  ($ (e) + (.3,0) $) .. controls ++(270:.4cm) and ++(90:.4cm) .. (f);
	\draw[thick, red] ($ (e) + (0,-.1) $) -- (f) ;
	\filldraw[red] (f) circle (.05cm);
\end{tikzpicture}
\end{matrix}
\end{equation}
with $\sigma=\varepsilon_2^{-1}\varepsilon_3^{-1}\varepsilon_1\,\vartheta_2\,\vartheta_3^{-2}\vartheta_4^{-1}$.

Let $A_{\text{st}}:=\begin{tikzpicture}[baseline=-.4cm, scale=.5]
	\coordinate (a) at (0,0);
	\coordinate (b) at ($ (a) + (0,1.2)$);
	\coordinate (c) at ($ (a) + (0,.4)$);
	\coordinate (d) at ($ (a) + (0,-.4)$);
	\coordinate (e) at ($ (a) + (0,-1.2)$);
	\coordinate (f) at ($ (a) + (0,-2)$);
	\draw[very thick] (a) ellipse (1 and 2);
	\ncircle{unshaded}{(b)}{.2}{270}{}
	\ncircle{unshaded}{(c)}{.2}{270}{}
	\ncircle{unshaded}{(d)}{.2}{270}{}
	\ncircle{unshaded}{(e)}{.2}{270}{}
	\draw[thick, red] ($ (b) + (0,-.2) $)  .. controls ++(270:.3cm) and ++(90:.8cm) ..  ($ (d) + (.7,0) $) .. controls ++(270:1.5cm) and ++(90:.4cm) .. (f);
	\draw[thick, red] ($ (c) + (0,-.2) $)  .. controls ++(270:.3cm) and ++(90:.6cm) ..  ($ (e) + (.5,.4) $) .. controls ++(270:1cm) and ++(90:.4cm) .. (f);
	\draw[thick, red] ($ (d) + (0,-.2) $)  .. controls ++(270:.2cm) and ++(90:.4cm) ..  ($ (e) + (.3,0) $) .. controls ++(270:.4cm) and ++(90:.4cm) .. (f);
	\draw[thick, red] ($ (e) + (0,-.2) $) -- (f) ;
	\filldraw[red] (f) circle (.05cm);
\end{tikzpicture}
$\,
denote the `standard' system of anchor lines.
We may identify the elements of $\cR\cB_n$ with systems of anchor lines via the map $\sigma\mapsto \sigma\cdot A_{\text{st}}$.
Using that correspondence, we can then transport the operad structure from anchored tangles to ribbon braids.
We illustrate the resulting operad structure on $(\cR\cB_n)_{n\ge0}$ with an example:
\begin{equation}\label{eq: pic: operadic product on braids}
\tikz[baseline=.5cm, scale=1.1]{
\filldraw[fill=white] (.5,0) rectangle (.6,1.2);
\filldraw[fill=white] (1,0) -- (1.1,0) .. controls ++(90:.44) and ++(-90:.5) .. (1.6,1.2) -- (1.5,1.2) .. controls ++(-90:.44) and ++(90:.5) .. (1,0);
\filldraw[fill=white] (1.5,0) -- (1.6,0) .. controls ++(90:.5) and ++(-90:.44) .. (1.1,1.2) -- (1,1.2) .. controls ++(-90:.5) and ++(90:.44) .. (1.5,0);
}
\,\,\,\,
\circ_3
\,\,\,
\tikz[baseline=-.1cm, scale=1.1]{
\filldraw[fill=white] (1,-.6) -- (1.1,-.6) .. controls ++(90:.4) and ++(-90:.33) .. (.6,0) -- (.6,.1) (.6,.5) -- (.6,.6) -- (.5,.6) -- (.5,.5) (.5,.1) -- (.5,0) .. controls ++(-90:.4) and ++(90:.33) .. (1,-.6);
\draw (.6,.8-.7) to[in=-90,out=90] (.5,1-.7)(.6,1-.7) to[in=-90,out=90] (.5,1.2-.7);
\draw[line width = 3, white] (.5,.8-.7) to[in=-90,out=90] (.6,1-.7)(.5,1-.7) to[in=-90,out=90] (.6,1.2-.7);
\draw (.5,.8-.7) to[in=-90,out=90] (.6,1-.7)(.5,1-.7) to[in=-90,out=90] (.6,1.2-.7);
\filldraw[fill=white, xshift=31.5, rotate=180] (0,0) -- (0,-.6) -- (.1,-.6) -- (.1,0) .. controls ++(90:.33) and ++(-90:.4) .. (.6,.6) -- (.5,.6) .. controls ++(-90:.33) and ++(90:.4) .. (0,0);
\fill[gray!40] (.55,.21) -- ++(.01,.013) arc (-45:45:.11) -- ++(-.01,.013) -- ++(-.01,-.013) arc (180-45:180+45:.11) -- ++(.01,-.013);
}
\,\,\,\,\,=\,\,\,\,\,
\tikz[baseline=-.6cm, scale=1.1]{
\filldraw[fill=white] (-.5,-1.5) rectangle (-.4,.6);
\filldraw[fill=white]  (1.1,-.6) .. controls ++(90:.4) and ++(-90:.33) .. (.6,0) -- (.6,.1) (.6,.5) -- (.6,.6) -- (.5,.6) -- (.5,.5) (.5,.1) -- (.5,0) .. controls ++(-90:.4) and ++(90:.33) .. (1,-.6);
\draw (.6,.8-.7) to[in=-90,out=90] (.5,1-.7)(.6,1-.7) to[in=-90,out=90] (.5,1.2-.7);
\draw[line width = 3, white] (.5,.8-.7) to[in=-90,out=90] (.6,1-.7)(.5,1-.7) to[in=-90,out=90] (.6,1.2-.7);
\draw (.5,.8-.7) to[in=-90,out=90] (.6,1-.7)(.5,1-.7) to[in=-90,out=90] (.6,1.2-.7);
\fill[fill=white, xshift=31.3, rotate=180] (0,0) -- (0,-.6) -- (.1,-.6) -- (.1,0) .. controls ++(90:.33) and ++(-90:.4) .. (.6,.6) -- (.5,.6) .. controls ++(-90:.33) and ++(90:.4) .. (0,0);
\draw[xshift=31.3, rotate=180] (0,0) -- (0,-.6) -- (.1,-.6) -- (.1,0) .. controls ++(90:.33) and ++(-90:.4) .. (.6,.6) (.5,.6) .. controls ++(-90:.33) and ++(90:.4) .. (0,0);
\fill[gray!40] (.55,.21) -- ++(.01,.013) arc (-45:45:.11) -- ++(-.01,.013) -- ++(-.01,-.013) arc (180-45:180+45:.11) -- ++(.01,-.013);
\filldraw[fill=white] (.5,-1.5) -- (.6,-1.5) .. controls ++(90:.44) and ++(-90:.5) .. (1.1,-.6) (1,-.6) .. controls ++(-90:.44) and ++(90:.5) .. (.5,-1.5);
\filldraw[fill=white] (0,-1.5) -- (.1,-1.5) .. controls ++(90:.44) and ++(-90:.5) .. (.6,-.6) (.5,-.6) .. controls ++(-90:.44) and ++(90:.5) .. (0,-1.5);
\filldraw[fill=white] (1.1,-1.5) -- (1,-1.5) .. controls ++(90:.44) and ++(-90:.5) .. (0,-.6) -- (0,.6) -- (.1,.6) -- (.1,-.6) .. controls ++(-90:.44) and ++(90:.5) .. (1.1,-1.5);
}
\end{equation}
The operadic composition of ribbon braids is compatible with their action \eqref{action of ribbon braid on iterated tensor product}
on tensor products of objects of $\cC$ in the sense that
\begin{equation}\label{P(st)=P(s)oP(t)}
P(\sigma\circ_i \tau) \,=\, P(\sigma)\circ_i P(\tau).
\end{equation}
Here, the $\circ_i$ in the left hand side is as in \eqref{eq: pic: operadic product on braids}, and the $\circ_i$ in the right hand side is the one mentioned in Remark~\ref{rem:operadic composition of morphisms}.

For the reader's convenience, we finish this section by including the pictures of the systems of anchor lines which correspond to the generators $\varepsilon_i$ and $\vartheta_i$ of the ribbon braid group:
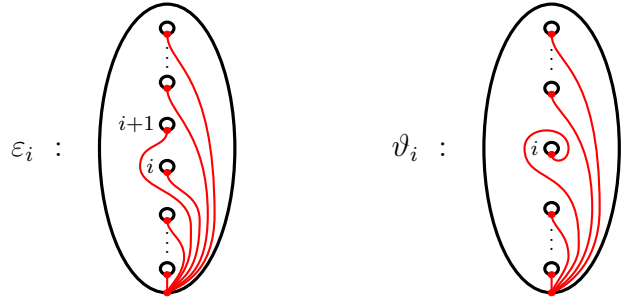
\begin{equation*} 
\varepsilon_i
\,\,:\,\,\,\,\,
\begin{tikzpicture}[baseline=-.1cm, yscale=.8, xscale=.9]
	\coordinate (a) at (0,0);
	\coordinate (b) at ($ (a) + (0,2)$);
	\coordinate (c) at ($ (a) + (0,1.1)$);
	\coordinate (d) at ($ (a) + (0,.4)$);
	\coordinate (e) at ($ (a) + (0,-.3)$);
	\coordinate (f) at ($ (a) + (0,-1.1)$);
	\coordinate (g) at ($ (a) + (0,-2)$);
	\coordinate (h) at ($ (a) + (0,-2.4)$);
	\draw[very thick] (a) ellipse (1 and 2.4);
	\ncircle{unshaded}{(b)}{.1}{270}{}
	\ncircle{unshaded}{(c)}{.1}{270}{}
	\ncircle{unshaded}{(d)}{.1}{270}{}
	\ncircle{unshaded}{(e)}{.1}{270}{}
	\ncircle{unshaded}{(f)}{.1}{270}{}
	\ncircle{unshaded}{(g)}{.1}{270}{}
	\node[scale=.9] at (0,-1.45) {\scriptsize{$\vdots$}};
	\node[scale=.9] at (0,1.63) {\scriptsize{$\vdots$}};
	\node at ($ (d) + (-.45,0) $) {\scriptsize{$i{+}1$}};
	\node at ($ (e) + (-.25,0) $) {\scriptsize{$i$}};
	\draw[thick, red] ($ (b) + (0,-.1) $)  .. controls ++(270:.5cm) and ++(90:2.4cm) ..  ($ (f) + (.7,0) $) .. controls ++(270:1cm) and ++(30:.2cm) .. (h);
	\draw[thick, red] ($ (c) + (0,-.1) $)  .. controls ++(270:.4cm) and ++(90:1.4cm) ..  ($ (f) + (.6,0) $) .. controls ++(270:1cm) and ++(40:.3cm) .. (h);
	\draw[thick, red] ($ (e) + (0,-.1) $)  .. controls ++(270:.3cm) and ++(90:.5cm) ..  ($ (f) + (.48,-.1) $) .. controls ++(270:.9cm) and ++(50:.3cm) .. (h);
	\draw[thick, red] ($ (d) + (0,-.1) $)  .. controls ++(270:.3cm) and ++(90:.4cm) ..  ($ (e) + (-.4,0) $) .. controls ++(270:.4cm) and ++(90:.6cm) .. ($ (f) + (.35,-.2) $) .. controls ++(270:.4cm) and ++(50:.6cm) .. (h);
	\draw[thick, red] ($ (f) + (0,-.1) $)  .. controls ++(270:.2cm) and ++(90:.3cm) ..  ($ (g) + (.25,.2) $) .. controls ++(270:.3cm) and ++(80:.1cm) .. (h);
	\draw[thick, red] ($ (g) + (0,-.1) $) -- (h) ;
	\filldraw[red] (h) circle (.05cm);
\end{tikzpicture}
\qquad
\quad\qquad
\vartheta_i 
\,\,:\,\,\,\,\,
\begin{tikzpicture}[baseline=-.1cm, yscale=.8, xscale=.9]
	\coordinate (a) at (0,0);
	\coordinate (b) at ($ (a) + (0,2)$);
	\coordinate (c) at ($ (a) + (0,1)$);
	\coordinate (d) at ($ (a) + (0,.0)$);
	\coordinate (e) at ($ (a) + (0,-1)$);
	\coordinate (f) at ($ (a) + (0,-2)$);
	\coordinate (g) at ($ (a) + (0,-2.4)$);
	\draw[very thick] (a) ellipse (1 and 2.4);
	\ncircle{unshaded}{(b)}{.1}{270}{}
	\ncircle{unshaded}{(c)}{.1}{270}{}
	\ncircle{unshaded}{(d)}{.1}{270}{}
	\ncircle{unshaded}{(e)}{.1}{270}{}
	\ncircle{unshaded}{(f)}{.1}{270}{}
	\node[scale=.9] at (0,-1.4) {\scriptsize{$\vdots$}};
	\node[scale=.9] at (0,1.6) {\scriptsize{$\vdots$}};
	\node at ($ (d) + (-.25,0) $) {\scriptsize{$i$}};
	\draw[thick, red] ($ (b) + (0,-.1) $)  .. controls ++(270:.5cm) and ++(90:2.4cm) ..  ($ (e) + (.71,0) $) .. controls ++(270:1cm) and ++(25:.16cm) .. (g);
	\draw[thick, red] ($ (c) + (0,-.1) $)  .. controls ++(270:.3cm) and ++(90:1.4cm) ..  ($ (e) + (.6,0) $) .. controls ++(270:1cm) and ++(40:.2cm) .. (g);
	\draw[thick, red] (-90:.1cm) .. controls ++(270:.15cm) and ++(270:.25cm) .. (0:.25cm) .. controls ++(90:.4cm) and ++(90:.4cm) .. (180:.4cm) .. controls ++(270:.6cm) and ++(90:1.2cm) .. ($ (e) + (.4,-.5) $) .. controls ++(270:.2cm) and ++(50:.7cm) .. (g);
	\draw[thick, red] ($ (e) + (0,-.1) $)  .. controls ++(270:.2cm) and ++(90:.4cm) ..  ($ (f) + (.25,.2) $) .. controls ++(270:.2cm) and ++(80:.1cm) .. (g);
	\draw[thick, red] ($ (f) + (0,-.1) $) -- (g) ;
	\filldraw[red] (g) circle (.05cm);
\end{tikzpicture}
\end{equation*}

\subsection{Assigning maps to generic tangles}
\label{sec:Assigning maps}

Given objects $\cP[n]\in \cC$ and morphisms $\eta, \alpha_i, \bar\alpha_i, \varpi_{i,j}$ as in Theorem \ref{thm:ConstructAPA}, we now present an algorithm 
which assigns to a generic anchored planar tangle $T$ of type $(k_1,\ldots,k_r;k_0)$ a morphism $Z(T):\cP[k_1]\otimes\ldots\otimes\cP[k_r]\to\cP[k_0]$ in $\cC$.

\begin{alg}
\label{alg:AssignMap}
Let $T=(T,X,Q,A)$ be a generic anchored planar tangle.
\begin{enumerate}[label=(\arabic*)]

\item
First, apply Algorithm \ref{alg:StandardForm} so as to bring the underlying planar tangle of $T$ into standard form.
The resulting tangle is determined canonically up to isotopy within the space of anchored planar tangles whose underlying tangle is in standard form.

\item
At this point, it becomes technically convenient to replace all circles by rectangles.
We do this while taking care not to change the minimal and maximal $y$-coordinates of the input boxes, and not to change the $y$-coordinates of the local maxima/minima of the strings.
We call the resulting rectangular anchored planar tangle $T'=(T',X',Q',A')$ (this is no longer an anchored planar tangle in the sense of Definition \ref{defn:AnchoredPlanarTangle}).
It is well defined up to isotopy within the space rectangular anchored planar tangles whose underlying rectangular tangle is in standard form (obvious analog of Definition \ref{def: anchored planar tangle : underlying tangle in standard form}).

\item
Endow $T'$ with a new system $A_{\text{st}}$ of \emph{standard anchor lines}, and call the result $T_{\text{st}}=(T',X',Q',A_{\text{st}})$.
Here, standard anchor lines are anchor lines which travel from the anchor points on the input circles horizontally to the right towards the external boundary, 
then down along that boundary, and finally come together at the very bottom
(see Figure~\ref{fig:StandardAnchorLines} for an example of this procedure).
Finally, let $\sigma_T\in \cR\cB_r$
be the ribbon braid group element uniquely determined by the equation
\[
\sigma_T \cdot A_{\text{st}} \,=\, A'.
\]
Here, the dot denotes the action of ribbon braids on anchor lines described in Section~\ref{sec:RibbonBraidGroup}.

\begin{figure}[!ht]
$$
\begin{tikzpicture}[scale=.5, baseline=-.1cm]
	\clip (-6.15,-6.15) rectangle (6.15,6.15);
	\roundNbox{unshaded}{(0,0)}{6}{0}{0}{};
	\draw (-3.5,6)--(-3.5,4);
	\draw (-3,4) -- (-3,5) arc (180:0:.5cm) -- (-2,4);
	\draw (-1.5,4) -- (-1.5,4.5) arc (180:0:.5cm) --(-.5,4) .. controls ++(270:1cm) and ++(90:1cm) .. (0.5,2);
	\draw (-2,-3.5)--(-2,-3) arc (0:180:1cm) -- (-4,-3) .. controls ++(270:5cm) and ++(270:5cm) .. (5,0) arc (0:180:.5cm) -- (4,0); 
	\draw (1.5,2) -- (1.5,6);
	\draw (3.3,0) -- (3.3,2) arc (0:180:.5cm);
	\draw (-2.5, -3.5)--(-2.5,-3) arc(0:180:.5cm) .. controls ++(270:3cm) and ++(270:3cm) .. (2,-2.5) arc (0:180:1cm)  -- (0,-3.5);	
	\draw (-5,6) .. controls ++(270:7cm) and ++(90:7cm) .. (-1,-3.5);
	\roundNbox{unshaded}{(-2.5,4)}{.4}{1}{1}{};
	\roundNbox{unshaded}{(-1,-3.5)}{.4}{1.5}{1.5}{};
	\roundNbox{unshaded}{(3.3,-.4)}{.4}{1}{1}{};
	\roundNbox{unshaded}{(1.5,1.6)}{.4}{1}{1}{};
	\draw[thick, red] (1.5,1.2) .. controls ++(270:1cm) and ++(-90:2cm) ..
	                          (-.8,1.9) .. controls ++(90:1cm) and ++(-90:1cm) .. 
	                          (-.7,3.5) .. controls ++(90:1cm) and ++(0:1cm) ..
	                          (-2.4,5)  .. controls ++(180:3.2cm) and ++(90:1cm) .. (-5,1.5) --
	                          (-5,-1.5) .. controls ++(270:4.5cm) and ++(120:1.5cm) .. (0,-6);
	\draw[thick, red] (-1,-3.9) .. controls ++(270:1cm) and ++(90:1cm) .. (0,-6);
	\draw[thick, red] (3.3,-.8) .. controls ++(270:4cm) and ++(60:1cm) .. (0,-6);
	\draw[thick, red] (-2.5,3.6)  .. controls ++(270:5cm) and ++(90:4cm) .. (1.5,-3) .. controls ++(270:2cm) and ++(70:1cm) .. (0,-6);
	\filldraw[red] (-2.5,3.6) circle (.1cm);
	\filldraw[red] (-1,-3.9) circle (.1cm);
	\filldraw[red] (3.3,-.8) circle (.1cm);
	\filldraw[red] (1.5,1.2) circle (.1cm);
	\filldraw[red] (0,-6) circle (.1cm);
\end{tikzpicture}
\longmapsto
\begin{tikzpicture}[scale=.5, baseline=-.1cm]
	\clip (-6.15,-6.15) rectangle (6.15,6.15);
	\roundNbox{unshaded}{(0,0)}{6}{0}{0}{};
	\draw (-3.5,6)--(-3.5,4);
	\draw (-3,4) -- (-3,5) arc (180:0:.5cm) -- (-2,4);
	\draw (-1.5,4) -- (-1.5,4.5) arc (180:0:.5cm) --(-.5,4) .. controls ++(270:1cm) and ++(90:1cm) .. (0.5,2);
	\draw (-2,-3.5)--(-2,-3) arc (0:180:1cm) -- (-4,-3) .. controls ++(270:5cm) and ++(270:5cm) .. (5,0) arc (0:180:.5cm) -- (4,0); 
	\draw (1.5,2) -- (1.5,6);
	\draw (3.3,0) -- (3.3,2) arc (0:180:.5cm);
	\draw (-2.5, -3.5)--(-2.5,-3) arc(0:180:.5cm) .. controls ++(270:3cm) and ++(270:3cm) .. (2,-2.5) arc (0:180:1cm)  -- (0,-3.5);	
	\draw (-5,6) .. controls ++(270:7cm) and ++(90:7cm) .. (-1,-3.5);
	\roundNbox{unshaded}{(-2.5,4)}{.4}{1}{1}{};
	\roundNbox{unshaded}{(-1,-3.5)}{.4}{1.5}{1.5}{};
	\roundNbox{unshaded}{(3.3,-.4)}{.4}{1}{1}{};
	\roundNbox{unshaded}{(1.5,1.6)}{.4}{1}{1}{};
	\draw[thick, red] (-2.5,3.6) arc (180:270:.5cm) -- (5.1+.2,3.1) arc (90:0:.5cm) -- (5.6+.2,-5.3) arc(0:-90:.5) -- +(-4,0) .. controls ++(180:.5cm) and ++(90:.2cm) .. (0,-6);
	\draw[thick, red] (1.5,1.2) arc (180:270:.4cm) --(5+.2,.8) arc (90:0:.5cm) -- (5.5+.2,-5.3)        arc(0:-90:.4) -- +(-4,0).. controls ++(180:.5cm) and ++(90:.3cm) .. (0,-6);
	\draw[thick, red] (3.3,-.8) arc (180:270:.5cm) --(4.9+.2,-1.3) arc (90:0:.5cm) -- (5.4+.2,-5.3)   arc(0:-90:.3) -- +(-4,0) .. controls ++(180:.5cm) and ++(90:.4cm) .. (0,-6);
	\draw[thick, red] (-1,-3.9) arc (180:270:.5cm) --(4.8+.2,-4.4) arc (90:0:.5cm) -- (5.3+.2,-5.3)  arc(0:-90:.2) -- +(-4,0) .. controls ++(180:.5cm) and ++(90:.5cm) .. (0,-6);
	\filldraw[red] (-2.5,3.6) circle (.1cm);
	\filldraw[red] (-1,-3.9) circle (.1cm);
	\filldraw[red] (3.3,-.8) circle (.1cm);
	\filldraw[red] (1.5,1.2) circle (.1cm);
	\filldraw[red] (0,-6) circle (.1cm);
\end{tikzpicture}
$$
\caption{Endowing a tangle in standard form with standard anchor lines}\label{fig:StandardAnchorLines}
\end{figure}
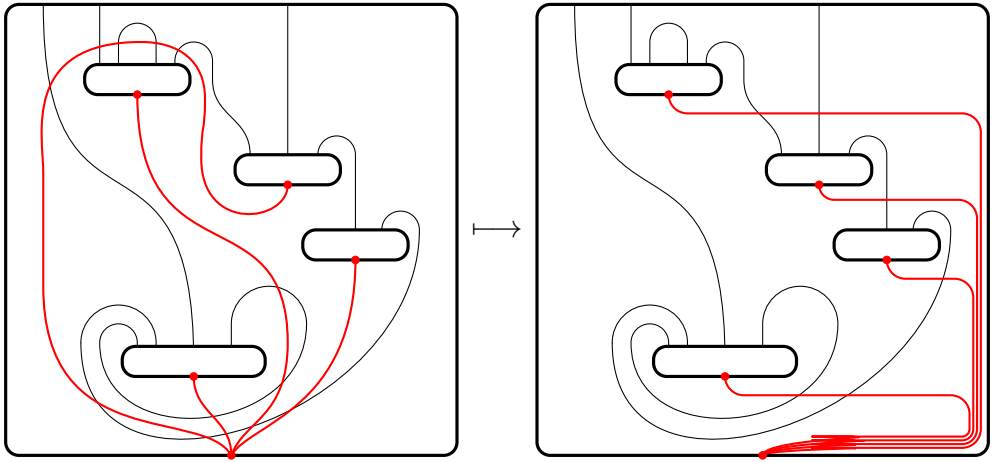

\item
We now write $T_{\text{st}}$ as a composite of generating tangles, in a way similar to the algorithm described in \cite{math.QA/9909027,MR2812459}.
Let $C$ be the union of all the input rectangles and all the critical points of strings.
Consider a nested sequence of closed rectangles
\[
R_0\subset R_1\subset\ldots\subset R_N
\]
with the following properties:
$R_N$ is the outer rectangle of $T_{\text{st}}$;
the $y$-projection of $R_0$ lies below that of $C$;
the $x$-projection of $R_0$ contains that of $C$;
the strands don't intersect the vertical sides of the rectangles;
$R_i$ contains $R_{i-1}$ in its interior;
finally and most importantly, the $y$-projection of $R_i$ minus that of $R_{i-1}$ contains exactly one component of the $y$-projection of $C$.
We illustrate this with an example:
$$
T_{\text{st}}\,\,=\,\,\,
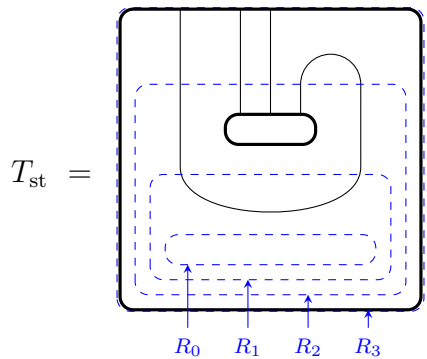
\begin{tikzpicture}[baseline =-.3cm]
	\roundNbox{}{(0,0)}{2}{0}{0}{}
	\draw (-.4,.6) -- (-.4,2);
	\draw (0,.6) -- (0,2);
	\draw (.4,.6) -- (.4,1) arc (180:0:.4cm) -- (1.2,-.1) .. controls ++(270:.8cm) and ++(270:.8cm) .. (-1.2,-.1) -- (-1.2,2);
	\roundNbox{unshaded}{(0,.4)}{.2}{.4}{.4}{}
	\roundNbox{draw=blue, dashed, thin}{(0,-.4)}{1.4}{.4}{.4}{}
	\roundNbox{draw=blue, dashed, thin}{(0,-.9)}{.7}{.9}{.9}{}
	\roundNbox{draw=blue, dashed, thin}{(0,-1.2)}{.2}{1.2}{1.2}{}
	\roundNbox{draw=blue, dashed, thin}{(0,0)}{2.02}{.02}{.02}{}
\node[blue] (a) at (-1.1,-2.5) {$\scriptstyle R_0$};
\node[blue] (b) at (-.3,-2.5) {$\scriptstyle R_1$};
\node[blue] (c) at (.5,-2.5) {$\scriptstyle R_2$};
\node[blue] (d) at (1.3,-2.5) {$\scriptstyle R_3$};
\draw[blue, -stealth] (a) -- +(0,1.1);
\draw[blue, -stealth] (b) -- +(0,.9);
\draw[blue, -stealth] (c) -- +(0,.7);
\draw[blue, -stealth] (d) -- +(0,.5);
\end{tikzpicture}
$$
Let $T_0:=R_0$, and let $T_i:=R_i\setminus \mathring R_{i-1}$ for $i\ge 1$.
By construction, each $T_i$ is isotopic to a generating tangle (Definition \ref{defn:GeneratingTangles}), and we have
$
T_{\text{st}} = T_N \circ_{\scriptscriptstyle 1} \ldots \circ_{\scriptscriptstyle 1} T_2 \circ_{\scriptscriptstyle 1} T_1  \circ_{\scriptscriptstyle 1} T_0$
(see Figure~\ref{fig:ReadRotation} for an example).
Letting $\mathcal T_0=\eta$, and letting $\mathcal T_i=\alpha_j$ if $T_i=a_j$, $\mathcal T_i=\bar\alpha_j$ if $T_i=\bar a_j$, and
$\mathcal T_i=\varpi_{j,k}$ if $T_i=p_{j,k}$ ($\eta$, $\alpha_j$, $\bar\alpha_j$, $\varpi_{j,k}$ as in Theorem~\ref{thm:ConstructAPA}), we set
\begin{equation}\label{eq: Z(T_st):= ... }
Z(T_{\text{st}})\,:=\,\mathcal T_N \circ_{\scriptscriptstyle 1} \ldots \circ_{\scriptscriptstyle 1} \mathcal T_2 \circ_{\scriptscriptstyle 1} \mathcal T_1  \circ_{\scriptscriptstyle 1} \mathcal T_0.
\end{equation}
Note that the operation $\circ_1$ is associative, and that the prescription \eqref{eq: Z(T_st):= ... } is therefore well defined.

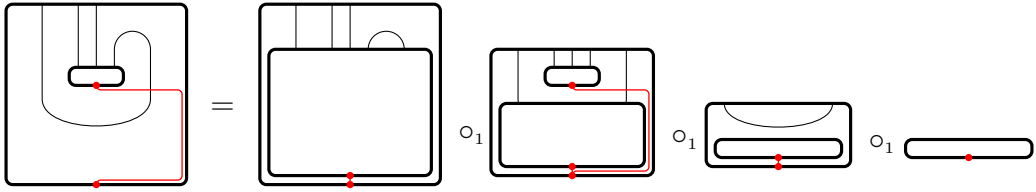
\begin{figure}[!ht]
$$
\begin{tikzpicture}[scale=.6, baseline =-.3cm]
	\roundNbox{rounded corners=2.5}{(0,0)}{2}{0}{0}{}
	\draw (-.4,.6) -- (-.4,2);
	\draw (0,.6) -- (0,2);
	\draw (.4,.6) -- (.4,1) arc (180:0:.4cm) -- (1.2,-.1) .. controls ++(270:.8cm) and ++(270:.8cm) .. (-1.2,-.1) -- (-1.2,2);
	\roundNbox{unshaded, rounded corners=2.5}{(0,.4)}{.2}{.4}{.4}{}
\node at (2.8,-.3) {$=$};
\fill[red] (0,.2) circle (.08);
\fill[red] (0,-2) circle (.08);
\draw[red, line width=.5] (0,-2) arc (180:90:.1) -- ++(1.7,0) arc (-90:0:.1) -- ++(0,1.8)  arc (0:90:.1) -- ++(-1.7,0) arc (-90:-180:.1);
\pgftransformxshift{160}
	\roundNbox{rounded corners=2.5}{(0,0)}{2}{0}{0}{}
	\draw (-.4,.6) -- (-.4,2);
	\draw (0,.6) -- (0,2);
	\draw (.4,.6) -- (.4,1) arc (180:0:.4cm) -- (1.2,-.1) .. controls ++(270:.8cm) and ++(270:.8cm) .. (-1.2,-.1) -- (-1.2,2);
	\roundNbox{unshaded, rounded corners=2.5}{(0,-.4)}{1.4}{.4}{.4}{}
\node at (2.7,-.9) {$\circ_{\scriptscriptstyle 1}$};
\fill[red] (0,-1.8) circle (.08);
\fill[red] (0,-2) circle (.08);
\draw[red, line width=.5] (0,-2) -- (0,-1.8);
\pgftransformxshift{140}
	\roundNbox{unshaded, rounded corners=2.5}{(0,-.4)}{1.4}{.4}{.4}{}
	\draw (-.4,.6) -- (-.4,1);
	\draw (0,.6) -- (0,1);
	\draw (.4,.6) -- (.4,1) (1.2,1) -- (1.2,-.1) .. controls ++(270:.8cm) and ++(270:.8cm) .. (-1.2,-.1) -- (-1.2,1);
	\roundNbox{unshaded, rounded corners=2.5}{(0,.4)}{.2}{.4}{.4}{}
	\roundNbox{unshaded, rounded corners=2.5}{(0,-.9)}{.7}{.9}{.9}{}
\node at (2.5,-1) {$\circ_{\scriptscriptstyle 1}$};
\fill[red] (0,-1.8) circle (.08);
\fill[red] (0,-1.6) circle (.08);
\fill[red] (0,.2) circle (.08);
\draw[red, line width=.5] (0,-1.6) -- (0,-1.8);
\draw[red, line width=.5] (0,-1.8) arc (180:90:.1) -- ++(1.5,0) arc (-90:0:.1) -- ++(0,1.6)  arc (0:90:.1) -- ++(-1.5,0) arc (-90:-180:.1);
\pgftransformxshift{130}
	\roundNbox{unshaded, rounded corners=2.5}{(0,-.9)}{.7}{.9}{.9}{}
	\draw (1.2,-.2) .. controls ++(270:.7cm) and ++(270:.7cm) .. (-1.2,-.2);
	\roundNbox{unshaded, rounded corners=2.5}{(0,-1.2)}{.2}{1.2}{1.2}{}
\node at (2.3,-1.1) {$\circ_{\scriptscriptstyle 1}$};
\fill[red] (0,-1.4) circle (.08);
\fill[red] (0,-1.6) circle (.08);
\draw[red, line width=.5] (0,-1.4) -- (0,-1.6);
\pgftransformxshift{120}
	\roundNbox{unshaded, rounded corners=2.5}{(0,-1.2)}{.2}{1.2}{1.2}{}
\fill[red] (0,-1.4) circle (.08);
\end{tikzpicture}
$$
\caption{Decomposing the tangle $T_{\text{st}}$ into generating tangles}
\label{fig:ReadRotation}
\end{figure}

\item
Finally, we let 
\begin{equation}\label{eq: Z(T) = Z(T_st) o P(s_T)}
Z(T) := Z(T_{\text{st}}) \circ P(\sigma_T),
\end{equation}
where
$P(\sigma_T):\cP[k_1]\otimes\ldots\otimes\cP[k_r] \to \cP[k_{\pi_T^{-1}(1)}]\otimes\ldots\otimes\cP[k_{\pi_T^{-1}(r)}]$
is the action (described in Section~\ref{sec:RibbonBraidGroup}) of the ribbon braid $\sigma_T$ on a tensor product of objects of $\cC$, 
and $\pi_T$ is the image of $\sigma_T$ in the symmetric group.

\end{enumerate}
\end{alg}

We illustrate the above algorithm by computing the value of certain simple tangles:

\begin{ex}
\label{ex:Zero tangle}
$Z\Big(\begin{tikzpicture}[baseline = -.1cm, scale=.8]
	\draw[very thick] (0,0) circle (.4cm);
	\filldraw[red] (0,-.4) circle (.05cm);
\end{tikzpicture}\Big)=\eta$.
\end{ex}

\begin{ex}
\label{ex:Identity}
We have
$
Z\bigg(
\begin{tikzpicture}[baseline=-.1cm, scale=.8]
	\ncircle{unshaded}{(0,0)}{.7}{270}{}
	\ncircle{unshaded}{(0,0)}{.25}{270}{}
	\draw (90:.25cm) -- (90:.7cm);
	\draw[thick, red] (270:.25cm) -- (270:.7cm);
	\node at (109:.477cm) {\scriptsize{$n$}};
\end{tikzpicture}
\bigg)
=\id_{\cP[n]}
$.
\end{ex}
\begin{proof}
Following Algorithm \ref{alg:AssignMap}, $Z(\id_n):=\varpi_{0,n}\circ_1 \eta$, which equals $\id_{\cP[n]}$ by \ref{reln:UnitMap}.
\end{proof}

\begin{ex}
\label{ex:computation alpha's}
We have $
Z
\left(\,
\begin{tikzpicture}[baseline = -.3cm, scale=.9]
	\draw (-.6,-.4) -- (-.6,1-.4);
	\draw (.2,-.4) -- (.2,.1) arc (0:180:.2) -- (-.2,-.4);
	\draw (.6,-.4) -- (.6,1-.4);
	\draw[thick, red] (0,-.6) -- (0,-1);
	\roundNbox{}{(0,-.2)}{.8}{.2}{.2}{}
	\roundNbox{unshaded}{(0,-.4)}{.2}{.6}{.6}{}
	\node at (-.8,.8-.5) {\scriptsize{$i$}};
	\node at (.8,.8-.5) {\scriptsize{$j$}};
\end{tikzpicture}
\,\right)
=\alpha_i$,
and 
$
Z
\left(\,
\begin{tikzpicture}[baseline = -.3cm, scale=.9]
	\draw (-.6,-.2) -- (-.6,1-.4);
	\draw (.2,1-.4) -- (.2,.7-.4) arc (0:-180:.2) -- (-.2,1-.4);
	\draw (.6,-.2) -- (.6,1-.4);
	\draw[thick, red] (0,-.6) -- (0,-1);
	\roundNbox{}{(0,-.2)}{.8}{.2}{.2}{}
	\roundNbox{unshaded}{(0,-.4)}{.2}{.6}{.6}{}
	\node at (-.8,.8-.5) {\scriptsize{$i$}};
	\node at (.8,.8-.5) {\scriptsize{$j$}};
\end{tikzpicture}
\,\right)
=\bar \alpha_i$.
\end{ex}
\begin{proof}
The first tangle evaluates to $\alpha_i\circ_1\varpi_{0,i+j+2}\circ_1\eta$,
which simplifies to $\alpha_i \circ_1 \id_{\cP[i+j+2]}=\alpha_i$ by \ref{reln:UnitMap}.
By the same argument, the second tangle evaluates to $\bar \alpha_i\circ_1\varpi_{0,i+j}\circ_1\eta =\bar\alpha_i$.
\end{proof}

\begin{ex}
\label{ex:SwapAnchorDependence}
We have $
Z
\left(
\begin{tikzpicture}[baseline = -.1cm, scale=.9]
	\draw (-.6,-.2) -- (-.6,1);
	\draw (0,1) -- (0,.6);
	\draw (.6,-.2) -- (.6,1);
	\draw[thick, red] (0,-.6) -- (0,-1);
	\draw[thick, red] (0,.2) arc (180:270:.2cm) -- (.6,0) arc (90:-90:.4cm) -- (.2,-.8) arc (90:180:.2cm);
	\roundNbox{}{(0,0)}{1}{.2}{.2}{}
	\roundNbox{unshaded}{(0,-.4)}{.2}{.6}{.6}{}
	\roundNbox{unshaded}{(0,.4)}{.2}{.2}{.2}{}
	\node at (-.8,.8) {\scriptsize{$i$}};
	\node at (-.2,.8) {\scriptsize{$j$}};
	\node at (.8,.8) {\scriptsize{$k$}};
\end{tikzpicture}
\right)
=\varpi_{i,j}$,
and 
$
Z\left(
\begin{tikzpicture}[baseline = -.1cm, scale=.9]
	\draw (-.6,-.2) -- (-.6,1);
	\draw (0,1) -- (0,.6);
	\draw (.6,-.2) -- (.6,1);
	\draw[thick, red] (0,-.6) -- (0,-1);
	\draw[thick, red] (0,.2) arc (0:-90:.2cm) -- (-.6,0) arc (90:270:.4cm) -- (-.2,-.8) arc (90:0:.2cm);
	\roundNbox{}{(0,0)}{1}{.2}{.2}{}
	\roundNbox{unshaded}{(0,-.4)}{.2}{.6}{.6}{}
	\roundNbox{unshaded}{(0,.4)}{.2}{.2}{.2}{}
	\node at (-.8,.8) {\scriptsize{$i$}};
	\node at (-.2,.8) {\scriptsize{$j$}};
	\node at (.8,.8) {\scriptsize{$k$}};
\end{tikzpicture}
\right)
=
\varpi_{i,j}\,
\circ \beta_{\cP[j],\cP[i+k]}
$.
\end{ex}
\begin{proof}
$Z(p_{i,j}):=\varpi_{i,j}\circ_1\varpi_{0,i+k}\circ_1\eta$ simplifies to $\varpi_{i,j}\circ_1 \id_{\cP[i+k]}=\varpi_{i,j}$.
The second tangle differs from the first one by its anchor lines. 
By (\ref{eq: Z(T) = Z(T_st) o P(s_T)}), it evaluates to $Z(p_{i,j})\circ \beta_{\cP[j],\cP[i+k]}$.
\end{proof}

\begin{ex}
\label{ex:2PiRotation}
We have $
Z\left(
\begin{tikzpicture}[baseline=-.1cm, scale=.8]
	\draw (0,0) -- (90:1cm);
	\draw[thick, red] (-90:.3cm) .. controls ++(270:.3cm) and ++(270:.5cm) .. (0:.5cm) .. controls ++(90:.8cm) and ++(90:.8cm) .. (180:.7cm) .. controls ++(270:.6cm) and ++(90:.4cm) .. (270:1cm);
	\draw[very thick] (0,0) circle (1cm);
	\draw[unshaded, very thick] (0,0) circle (.3cm);
	\node at (100:.8cm) {\scriptsize{$n$}};
\end{tikzpicture}
\right)
=\theta_{\cP[n]}
\text{ and }\,
Z\left(
\begin{tikzpicture}[baseline=-.1cm, xscale=-.8, yscale=.8]
	\draw (0,0) -- (90:1cm);
	\draw[thick, red] (-90:.3cm) .. controls ++(270:.3cm) and ++(270:.5cm) .. (0:.5cm) .. controls ++(90:.8cm) and ++(90:.8cm) .. (180:.7cm) .. controls ++(270:.6cm) and ++(90:.4cm) .. (270:1cm);
	\draw[very thick] (0,0) circle (1cm);
	\draw[unshaded, very thick] (0,0) circle (.3cm);
	\node at (100:.8cm) {\scriptsize{$n$}};
\end{tikzpicture}
\right)
=\theta_{\cP[n]}^{-1}.
$
\end{ex}
\begin{proof}
These tangles are identical to the one in Example \ref{ex:Identity}, except for the position of their anchor lines.
By (\ref{eq: Z(T) = Z(T_st) o P(s_T)}), they evaluate 
to $Z(\id_n)\circ\, \theta_{\cP[n]}$ and $Z(\id_n)\circ\, \theta_{\cP[n]}^{-1}$, respectively.
\end{proof}

\begin{ex}\label{ex: full twist tangle}
We have $
Z\left(
\begin{tikzpicture}[baseline=-.1cm, xscale=.8, yscale=-.8]
	\draw [thick, red] (0,0) -- (90:1cm);
	\draw[] (-90:.3cm) .. controls ++(270:.3cm) and ++(270:.5cm) .. (0:.5cm) .. controls ++(90:.8cm) and ++(90:.8cm) .. (180:.7cm) .. controls ++(270:.6cm) and ++(90:.4cm) .. (270:1cm);
	\draw[very thick] (0,0) circle (1cm);
	\draw[unshaded, very thick] (0,0) circle (.3cm);
	\node at (-80:.8cm) {\scriptsize{$n$}};
\end{tikzpicture}
\right)
=\theta_{\cP[n]}
\text{ and }\,
Z\left(
\begin{tikzpicture}[baseline=-.1cm, scale=-.8]
	\draw[thick, red] (0,0) -- (90:1cm);
	\draw[] (-90:.3cm) .. controls ++(270:.3cm) and ++(270:.5cm) .. (0:.5cm) .. controls ++(90:.8cm) and ++(90:.8cm) .. (180:.7cm) .. controls ++(270:.6cm) and ++(90:.4cm) .. (270:1cm);
	\draw[very thick] (0,0) circle (1cm);
	\draw[unshaded, very thick] (0,0) circle (.3cm);
	\node at (-80:.8cm) {\scriptsize{$n$}};
\end{tikzpicture}
\right)
=\theta_{\cP[n]}^{-1}
$.
\end{ex}
\begin{proof}
Following Algorithm \ref{alg:AssignMap},
the first tangle evaluates to
$\alpha_{n}\circ_1 \cdots \circ_1 \alpha_{2n-1} \circ_1 \varpi_{n,n}\circ_1 \bar\alpha_{n-1}\circ_1 \cdots \circ_1 \bar\alpha_0 \circ_1 \eta$,
which is equal to $\theta_{\cP[n]}$ by \ref{reln:RotationThetaMaps}.
The second tangle is the inverse up to isotopy of the first one.
By Propositions \ref{prop:Gluing} and \ref{prop:IsotopyInvariance} below, it evaluates to the inverse morhpism $\theta_{\cP[n]}^{-1}$.
\end{proof}

\begin{ex}\label{ex: double twirl}
We have
$
Z\left(
\begin{tikzpicture}[baseline=-.1cm, xscale=-.8, yscale=.8]
	\draw[] (90:.3cm) .. controls ++(90:.3cm) and ++(90:.5cm) .. (180:.5cm) .. controls ++(270:.8cm) and ++(270:.8cm) .. (0:.7cm) .. controls ++(90:.6cm) and ++(270:.4cm) .. (90:1cm);
	\draw[thick, red] (-90:.3cm) .. controls ++(270:.3cm) and ++(270:.5cm) .. (0:.5cm) .. controls ++(90:.8cm) and ++(90:.8cm) .. (180:.7cm) .. controls ++(270:.6cm) and ++(90:.4cm) .. (270:1cm);
	\draw[very thick] (0,0) circle (1cm);
	\draw[unshaded, very thick] (0,0) circle (.3cm);
	\node at (100:.8cm) {\scriptsize{$n$}};
\end{tikzpicture}
\right)
=
\id_{\cP[n]}.
$
\end{ex}
\begin{proof}
This follows from the prescription \eqref{eq: Z(T) = Z(T_st) o P(s_T)}, and from the first part of Example~\ref{ex: full twist tangle}
(the part which does not rely on Propositions \ref{prop:Gluing} and \ref{prop:IsotopyInvariance}).
\end{proof}

\subsection{Proof of Theorem \ref{thm:ConstructAPA}}
\label{sec:WellDefined}

As a first step towards Theorem \ref{thm:ConstructAPA},
we verify that the assignment $T\mapsto Z(T)$ is compatible with composition of tangles:

\begin{prop}
\label{prop:Gluing}
Let $T$ and $S$ be generic anchored planar tangles with underlying planar tangles in standard form.
Assume that the composition $T\circ_a S$ is allowed,
and let $Z(T)$, $Z(S)$, $Z(T\circ_a S)$ be as provided by Algorithm \ref{alg:AssignMap}.
Then we have $Z(T\circ_a S) = Z(T)\circ_a Z(S)$.
\end{prop}

\begin{proof}
We first check the special case when $S$ is a generating tangle (Definition \ref{defn:GeneratingTangles}), $T=p_{i,j}$, and $a=2$.
If $S=u$, we check using \ref{reln:UnitMap}:
\[
Z(T\circ_2 S)=Z(p_{i,0}\circ_2 u)=Z(\id_n)=\id_{\cP[n]}=\varpi_{i,0}\circ_2 \eta=Z(T)\circ_2 Z(S)
\]
If $S=a_k$, we check using \ref{reln:CapQuadraticMaps}:
\[
Z(T\circ_2 S)=Z(p_{i,j}\circ_2 a_k)=Z(a_{i+k} \circ_1 p_{i,j+2})=\alpha_{i+k} \circ_1 \varpi_{i,j+2}=\varpi_{i,j}\circ_2 \alpha_k=Z(T)\circ_2 Z(S).
\]
If $S=\bar a_k$, we check using \ref{reln:CupQuadraticMaps}:
\[
Z(T\circ_2 S)=Z(p_{i,j}\circ_2 \bar a_k)=Z(\bar a_{i+k} \circ_1 p_{i,j-2})=\bar \alpha_{i+k} \circ_1 \varpi_{i,j-2}=\varpi_{i,j}\circ_2 \bar \alpha_k=Z(T)\circ_2 Z(S).
\]
Finally, if $S=p_{k,l}$, we check using \ref{reln:EasyQuadraticMaps}:
\[
Z(T\circ_2 S)=Z(p_{i,j}\circ_2 p_{k,l})=Z(p_{i+k,l}\circ_1 p_{i,j-l})=\varpi_{i+k,l}\circ_1 \varpi_{i,j-l}=\varpi_{i,j}\circ_2 \varpi_{k,l}=Z(T)\circ_2 Z(S).
\]

Let us now assume that $S$ is a generating tangle (Definition \ref{defn:GeneratingTangles}) and that $T$ has standard anchor lines.
Decompose $T$ into generating tangles $T = T_N \circ_1 \ldots \circ_1 T_0$ as in Figure~\ref{fig:ReadRotation}.
By (\ref{eq: Z(T_st):= ... }) and by the computations performed in Examples \ref{ex:Zero tangle}--\ref{ex:SwapAnchorDependence}, we get that
$Z(T) \,=\, Z(T_N) \circ_1 \ldots \circ_1 Z(T_0)$.
We then have
\[
T\circ_a S
= [[T_N \circ_1 \ldots \circ_1  T_k] \circ_1  [T_{k-1}\circ_1\ldots \circ_1 T_0]] \circ_a S
= [[T_N \circ_1 \ldots \circ_1  T_k] \circ_2 S] \circ_1  [T_{k-1} \circ_1 \ldots \circ_1 T_0]
\]
for some index $k$ with $T_k=p_{i,j}$.
It follows that
\[
\begin{split}
Z(T\circ_a S)
&= Z([[T_N \circ_1 \ldots \circ_1  T_k] \circ_2 S] \circ_1  [T_{k-1} \circ_1 \ldots \circ_1 T_0])\\
&= Z(T_N \circ_1 \ldots \circ_1 T_{k+1} \circ_1  [T_k \circ_2 S] \circ_1  T_{k-1} \circ_1 \ldots \circ_1 T_0)\\
&= Z(T_N \circ_1 \ldots \circ_1 T_{k+1} \circ_1  T_k' \circ_1 T_k'' \circ_1  T_{k-1} \circ_1 \ldots \circ_1 T_0)\\
&= Z(T_N) \circ_1 \ldots \circ_1 Z(T_{k+1}) \circ_1  Z(T_k') \circ_1 Z(T_k'') \circ_1  Z(T_{k-1}) \circ_1 \ldots \circ_1 Z(T_0)\\
&= Z(T_N) \circ_1 \ldots \circ_1 Z(T_{k+1}) \circ_1  Z(T_k' \circ_1 T_k'') \circ_1  Z(T_{k-1}) \circ_1 \ldots \circ_1 Z(T_0)\\
&= Z(T_N) \circ_1 \ldots \circ_1 Z(T_{k+1}) \circ_1  Z(T_k \circ_2 S) \circ_1  Z(T_{k-1}) \circ_1 \ldots \circ_1 Z(T_0)\\
&= Z(T_N) \circ_1 \ldots \circ_1 Z(T_{k+1}) \circ_1  [Z(T_k) \circ_2 Z(S)] \circ_1  Z(T_{k-1}) \circ_1 \ldots \circ_1 Z(T_0)\\
&= [[Z(T_N) \circ_1 \ldots \circ_1  Z(T_k)] \circ_2 Z(S)] \circ_1  [Z(T_{k-1}) \circ_1 \ldots \circ_1 Z(T_0)]\\
&= [[Z(T_N) \circ_1 \ldots \circ_1  Z(T_k)] \circ_1  [Z(T_{k-1})\circ_1\ldots \circ_1 Z(T_0)]] \circ_a Z(S)\\
&=Z(T)\circ_a Z(S),
\end{split}
\]
where $T_k' \circ_1 T_k''$ is the standard way (as in Figure~\ref{fig:ReadRotation}) of decomposing $T_k \circ_2 S$ into generating tangles
(except when $S=u$, in which case the terms involving $T_k'$ and $T_k''$ should be omitted),
and we have used the previous computations to pass from the 6th to the 7th line.

We now take $T$ and $S$ to be arbitrary planar tangles (in standard form), equipped with standard anchor lines.
Write $S = S_N \circ_1 \ldots \circ_1 S_0$ as products of generating tangles,
so that $Z(S) = Z(S_N) \circ_1 \ldots \circ_1 Z(S_0)$.
By repeated use of the previous step, we get
\begin{align*}
Z(T \circ_a S)
&=Z(T \circ_a [S_N\circ_1\ldots \circ_1 S_2\circ_1 S_1 \circ_1 S_0])
\\&=Z([[[\cdots[T \circ_a S_N]\circ_a\ldots ]\circ_a S_2]\circ_a S_1]\circ_a S_0)
\\&=Z([[\cdots[T \circ_a S_N]\circ_a\ldots ]\circ_a S_2]\circ_a S_1)\circ_a Z(S_0)
\\&=[Z([\cdots[T \circ_a S_N]\circ_a\ldots ]\circ_a S_2)\circ_a Z(S_1)]\circ_a Z(S_0)
\\&=\ldots
\\&=[[[\cdots[Z(T )\circ_a Z(S_N)]\circ_a \cdots ]\circ_aZ(S_2)]\circ_a Z(S_1)] \circ_a Z (S_0)
\\&=Z(T) \circ_a [Z(S_N)\circ_1 \cdots \circ_1 Z (S_2)\circ_1 Z (S_1) \circ_1 Z (S_0)]
\\&=Z(T) \circ_a Z(S)
\end{align*}
as desired.

Finally, we treat the general case.
Let $T$ and $S$ to be anchored planar tangles with underlying planar tangles in standard form, and
let $\sigma_S$, $\sigma_T$, $\sigma_{T\circ_a S}$ be the ribbon braid elements which enter in the definition \eqref{eq: Z(T) = Z(T_st) o P(s_T)} of $Z$, so that
$Z(T) = Z(T_{\text{st}})\circ P(\sigma_T)$, $Z(S) = Z(S_{\text{st}})\circ P(\sigma_S)$, and $Z(T\circ_a S) = Z((T\circ_a S)_{\text{st}})\circ P(\sigma_{T\circ_a S})$.
Letting $\pi_T$ be the image of $\sigma_T$ in the symmetric group,
we have $(T\circ_a S)_{\text{st}}  = T_{\text{st}} \circ _{\pi_T(a)}S_{\text{st}}$.
It follows from our previous computations that
\[
Z((T\circ_a S)_{\text{st}}) = Z(T_{\text{st}} \circ _{\pi_T(a)}S_{\text{st}}) = Z(T_{\text{st}}) \circ_{\pi_T(a)} Z(S_{\text{st}}).
\]
We also have
\[
P(\sigma_{T\circ_a S}) = P(\sigma_T \circ_a \sigma_S) = P(\sigma_T) \circ_a P(\sigma_S),
\]
where the first equality holds by the definition of the operadic composition of ribbon braids (Section~\ref{sec:RibbonBraidGroup}), and the second equality is equation \eqref{P(st)=P(s)oP(t)}.
Putting all this together, we get
\begin{align*}
Z(T\circ_a S)
&=
Z((T\circ_a S)_{\text{st}})\circ P(\sigma_{T\circ_a S})
\\&=
[Z(T_{\text{st}}) \circ_{\pi_T(a)} Z(S_{\text{st}})] \circ [P(\sigma_T) \circ_a P(\sigma_S)]
\\&=
[Z(T_{\text{st}})\circ P(\sigma_T)] \circ_a [Z(S_{\text{st}})\circ P(\sigma_S)]
\\&=
Z(T) \circ_a Z(S).\qedhere
\end{align*}
\end{proof}

We now prove that the assignment $T\mapsto Z(T)$ is invariant under all isotopies of anchored planar tangles:

\begin{prop}
\label{prop:IsotopyInvariance}
Suppose $T$ and $S$ are generic anchored planar tangles which are isotopic. 
Then $Z(T)=Z(S)$.
\end{prop}
\begin{proof}
In Section \ref{sec:StandardForm}, we saw that any two generic anchored planar tangles which are isotopic can be connected by a finite sequence of the moves \ref{M:1} -- \ref{M:6}.
It is therefore enough to show that $Z(T)=Z(S)$ whenever $T$ and $S$ are related by one of these moves.
We proceed case by case:

\begin{enumerate}[label=(M\arabic*)]
\item
\label{rel:MorseMoveExchangeInputs}
$T$ is obtained from $S$ by exchanging the relative order of the $y$-coordinates of the centers of two input circles.
Then, up to permuting $T$ and $S$, the decompositions of $T_{\text{st}}$ and $S_{\text{st}}$ into generating tangles (as in Figure~\ref{fig:ReadRotation}) look like
\[
\begin{split}
T_{\text{st}} &= T_N \circ_1\ldots\circ_1T_{s+2}\circ_1
p_{i+j+k,\ell}\circ_1 p_{i,j}
\circ_1T_{s-1}\circ_1\ldots\circ_1 T_0
\\
S_{\text{st}} &= T_N \circ_1\ldots\circ_1T_{s+2}\circ_1
p_{i,j}\circ_1 p_{i+k,\ell}
\circ_1T_{s-1}\circ_1\ldots\circ_1 T_0
\end{split}
\]
(see part (1) of Remark \ref{rem: explanation of 7th and 8th rels}).
We also have $\sigma_S=\varepsilon_a\sigma_T$ for some appropriate index $a$.
It follows that
\begin{align*}Z(T) 
&= Z(T_{\text{st}}) \circ P(\sigma_T)\\
&=[Z(T_N) \circ_1\ldots\circ_1Z(p_{i+j+k,\ell})\circ_1 Z(p_{i,j})\circ_1\ldots\circ_1 Z(T_0)] \circ P(\sigma_T)\\
&= [Z(T_N) \circ_1\ldots\circ_1\varpi_{i+j+k,\ell}\circ_1 \varpi_{i,j}\circ_1\ldots\circ_1 Z(T_0)] \circ P(\sigma_T)\\
&= [Z(T_N) \circ_1\ldots\circ_1[[\varpi_{i,j}\circ_1 \varpi_{i+k,\ell}]\circ(\id \otimes \beta_{\cP[j],\cP[\ell]})]\circ_1\ldots\circ_1 Z(T_0)] \circ P(\sigma_T)\\
&= [Z(T_N) \circ_1\ldots\circ_1\varpi_{i,j}\circ_1 \varpi_{i+k,\ell}\circ_1\ldots\circ_1 Z(T_0)] \circ P(\varepsilon_a)\circ P(\sigma_T)\\
&=[Z(T_N) \circ_1\ldots\circ_1Z(p_{i,j})\circ_1 Z(p_{i+k,\ell})\circ_1\ldots\circ_1 Z(T_0)] \circ P(\varepsilon_a\sigma_T)\\
&=Z(S_{\text{st}})\circ  P(\sigma_S)=Z(S),
\end{align*}
where we have used \ref{reln:HardQuadraticMaps} in the fourth equality.

\item
\label{rel:MorseMoveExchangeInputAndCap}
$T$ is obtained from $S$ by exchanging the relative order of the $y$-coordinates of the center of and input circle and of a critical point of a string.
Then, up to permuting $T$ and $S$, the decompositions of $T_{\text{st}}$ and $S_{\text{st}}$ into generating tangles are of the from
\[
\begin{split}
T_{\text{st}} &= T_N \circ_1\ldots\circ_1T_{s+2}\circ_1
T_{s+1}\circ_1 T_s
\circ_1T_{s-1}\circ_1\ldots\circ_1 T_0
\\
S_{\text{st}} &= T_N \circ_1\ldots\circ_1T_{s+2}\circ_1
S_{s+1}\circ_1 S_s
\circ_1T_{s-1}\circ_1\ldots\circ_1 T_0
\end{split}
\]
with 
\begin{equation}\label{eq: lots of cases of T and S}
\begin{matrix} T_{s+1}=a_i \\ T_s=p_{j,k} \\ S_{s+1}=p_{j-2,k} \\ S_s=a_i\\ \text{for }  i+1<j \end{matrix}
\,\,\,\,\,\,\text{or}\,\,\,\,
\begin{matrix} T_{s+1}=a_i \\ T_s=p_{j,k} \\ S_{s+1}=p_{j,k} \\ S_s=a_{i-k} \\ \text{for } i+1> j+k \end{matrix}
\,\,\,\,\text{or}\,\,\,\,\,\,\,
\begin{matrix} T_{s+1}=\bar a_i \\ T_s=p_{j,k} \\ S_{s+1}=p_{j+2,k} \\ S_s=\bar a_i \\ \text{for } i\leq j \end{matrix}
\,\,\,\,\,\,\,\text{or}\,\,\,\,
\begin{matrix} T_{s+1}=\bar a_i \\ T_s=p_{j,k} \\ S_{s+1}=p_{j,k} \\ S_s=\bar a_{i-k} \\ \text{for } i\geq j+k. \end{matrix}
\end{equation}
By \ref{reln:CapQuadraticMaps} and \ref{reln:CupQuadraticMaps} we note that, in every case, the equation
\[
Z(T_{s+1})\circ_1 Z(T_s)=Z(S_{s+1})\circ_1 Z(S_s)
\]
holds.
We also have $\sigma_S = \sigma_T$. It follows that
\begin{align*}
Z(T) 
= 
Z(T_{\text{st}}) \circ P(\sigma_T)
&=[Z(T_N) \circ_1\ldots\circ_1 Z(T_{s+1})\circ_1 Z(T_s) \circ_1\ldots\circ_1 Z(T_0)]\circ P(\sigma_T)\\
&=[Z(T_N) \circ_1\ldots\circ_1 Z(S_{s+1})\circ_1 Z(S_s) \circ_1\ldots\circ_1 Z(T_0)]\circ P(\sigma_S)\\
&= 
Z(S_{\text{st}}) \circ P(\sigma_S)
= 
Z(S).
\end{align*}

\item
\label{rel:MorseMoveExchangeCaps}
$T$ is obtained from $S$ by exchanging the relative order of the $y$-coordinates of two critical points of strings. 
This case is entirely similar to the previous one, with 
\[
\begin{matrix} T_{s+1}=a_i \\ T_s=a_j \\ S_{s+1}=a_{j-2} \\ S_s=a_i\\ \text{for }  i+1<j \end{matrix}
\,\,\,\,\,\,\text{or}\,\,\,\,
\begin{matrix} T_{s+1}=\bar a_i \\ T_s=\bar a_j \\ S_{s+1}=\bar a_{j+2} \\ S_s=\bar a_i \\ \text{for } i\le j \end{matrix}
\,\,\,\,\text{or}\,\,\,\,\,\,\,
\begin{matrix} T_{s+1}=a_i \\ T_s=\bar a_j \\ S_{s+1}=\bar a_{j-2} \\ S_s=a_i \\ \text{for } i< j-1 \end{matrix}
\,\,\,\,\,\,\,\text{or}\,\,\,\,
\begin{matrix} T_{s+1}=a_i \\ T_s=\bar a_j \\ S_{s+1}=\bar a_j \\ S_s=a_{i-2} \\ \text{for } i> j+1. \end{matrix}
\]
instead of \eqref{eq: lots of cases of T and S}. The relevant relations are
\ref{reln:CapMaps},
\ref{reln:CupMaps}, and
\ref{reln:CapCupMaps}.

\item
\label{rel:MorseCancelation}
$T$ is obtained from $S$ by a Morse cancellation:
$\tikz[scale=.6, baseline=-4]{\draw (-.1,-.5) to[out=90,in=-90] (-.15,-.1) to[out=90,in=115, looseness=2] (0,0) to[out=-65,in=-90, looseness=2] (.15,.1) to[out=90,in=-90] (.1,.5);} \to \tikz[scale=.6, baseline=-4]{\draw (0,-.5) -- (0,.5);}$\,\,.
The argument to show $Z(T)=Z(S)$ uses \ref{reln:CapCupMaps}, and is similar to the previous ones.

\item
\label{rel:MorseMoveStrands}
$T$ is obtained from $S$ by moving the attaching point of a strand past the equator of some input circle. 
The invariance of $Z$ under that move is a direct consequence of the corresponding statement for the moves \ref{M:2} and \ref{M:4}.

\item
\label{rel:MorseMoveAnchor}
$T$ is obtained from $S$ by swinging an anchor point from slightly to the right of the north pole to slightly to the left of the north pole of an input circle. 
By applying Algorithm \ref{alg:StandardForm} (which is step (1) of Algorithm \ref{alg:AssignMap}) to $T$ and to $S$,
we obtain tangles $T'$ and $S'$ that differ from each other by a $2\pi$-rotation around some input circle:
$$
T'\,=\,\begin{tikzpicture}[baseline=-.1cm, xscale=-1]
	\draw [thick, red] (90:.7cm) .. controls ++(270:.3cm) and ++(90:.4cm) .. (.4,0) .. controls ++(270:.4cm) and ++(270:.2cm) .. (0,-.3);
\pgftransformrotate{180}
	\draw (90:.7cm) .. controls ++(270:.3cm) and ++(90:.4cm) .. (.4,0) .. controls ++(270:.4cm) and ++(270:.2cm) .. (0,-.3);
\pgftransformrotate{180}
	\ncircle{unshaded}{(0,0)}{.3}{270}{}
	\node[right] at (-90:.7cm) {$\scriptstyle n$};
\end{tikzpicture}
\,\,\,\,\longleftarrow\,\,\,\,
\begin{tikzpicture}[baseline=-.1cm]
	\draw (-90:.8cm) -- (0,0);
	\draw [thick, red] (90:.7cm) -- (0,0);
	\ncircle{unshaded}{(0,0)}{.3}{90}{}
	\node[right] at (-90:.7cm) {$\scriptstyle n$};
\end{tikzpicture}
\,\,\,\,\longrightarrow\,\,\,\,
\begin{tikzpicture}[baseline=-.1cm]
	\draw [thick, red] (90:.7cm) .. controls ++(270:.3cm) and ++(90:.4cm) .. (.4,0) .. controls ++(270:.4cm) and ++(270:.2cm) .. (0,-.3);
\pgftransformrotate{180}
	\draw (90:.7cm) .. controls ++(270:.3cm) and ++(90:.4cm) .. (.4,0) .. controls ++(270:.4cm) and ++(270:.2cm) .. (0,-.3);
\pgftransformrotate{180}
	\ncircle{unshaded}{(0,0)}{.3}{270}{}
	\node[right] at (-90:.7cm) {$\scriptstyle n$};
\end{tikzpicture}\,=\,S'
$$
By the already established moves \ref{M:2} -- \ref{M:4} and by Proposition~\ref{prop:Gluing}, checking the equality
$
Z\bigg(\begin{tikzpicture}[baseline=-.1cm, xscale=-1]
	\draw [thick, red] (90:.7cm) .. controls ++(270:.3cm) and ++(90:.4cm) .. (.4,0) .. controls ++(270:.4cm) and ++(270:.2cm) .. (0,-.3);
\pgftransformrotate{180}
	\draw (90:.7cm) .. controls ++(270:.3cm) and ++(90:.4cm) .. (.4,0) .. controls ++(270:.4cm) and ++(270:.2cm) .. (0,-.3);
\pgftransformrotate{180}
	\ncircle{unshaded}{(0,0)}{.3}{270}{}
	\node[right] at (-90:.7cm) {$\scriptstyle n$};
\end{tikzpicture}\bigg)
\,=\,
Z\bigg(\begin{tikzpicture}[baseline=-.1cm]
	\draw [thick, red] (90:.7cm) .. controls ++(270:.3cm) and ++(90:.4cm) .. (.4,0) .. controls ++(270:.4cm) and ++(270:.2cm) .. (0,-.3);
\pgftransformrotate{180}
	\draw (90:.7cm) .. controls ++(270:.3cm) and ++(90:.4cm) .. (.4,0) .. controls ++(270:.4cm) and ++(270:.2cm) .. (0,-.3);
\pgftransformrotate{180}
	\ncircle{unshaded}{(0,0)}{.3}{270}{}
	\node[right] at (-90:.7cm) {$\scriptstyle n$};
\end{tikzpicture}\bigg)
$
is equivalent to checking that\vspace{-.5cm}
$$
Z\left(
\begin{tikzpicture}[baseline=-.1cm, xscale=-.8, yscale=.8]
	\ncircle{unshaded}{(0,0)}{1}{-90}{}
	\draw[] (90:.3cm) .. controls ++(90:.3cm) and ++(90:.5cm) .. (180:.5cm) .. controls ++(270:.8cm) and ++(270:.8cm) .. (0:.7cm) .. controls ++(90:.6cm) and ++(270:.4cm) .. (90:1cm);
	\draw[thick, red] (-90:.3cm) .. controls ++(270:.3cm) and ++(270:.5cm) .. (0:.5cm) .. controls ++(90:.8cm) and ++(90:.8cm) .. (180:.7cm) .. controls ++(270:.6cm) and ++(90:.4cm) .. (270:1cm);
	\ncircle{unshaded}{(0,0)}{.3}{-90}{}
	\node at (100:.8cm) {\scriptsize{$n$}};
\end{tikzpicture}
\right)
=
Z\left(
\begin{tikzpicture}[baseline=-.1cm]
	\ncircle{unshaded}{(0,0)}{.7}{-90}{}
	\ncircle{unshaded}{(0,0)}{.25}{-90}{}
	\draw (90:.25cm) -- (90:.7cm);
	\node at (105:.45cm) {\scriptsize{$n$}};
	\draw[thick, red] (0,-.7) -- (0,-.25);
\end{tikzpicture}
\right)
$$
holds.
The latter follows from the computations performed in Examples~\ref{ex:Identity} and~\ref{ex: double twirl}.

\end{enumerate}
This concludes the proof of isotopy invariance of the assignment $T\mapsto Z(T)$.
\end{proof}

\begin{proof}[Proof of Theorem \ref{thm:ConstructAPA}]
Given an anchored planar tangle $T$, pick a generic anchored planar tangle $\tilde T$ which is isotopic to $T$
and set
\[
Z(T):=Z(\tilde T),
\]
where the meaning of the right hand side is provided by Algorithm~\ref{alg:AssignMap}.
We note that, by Proposition~\ref{prop:IsotopyInvariance}, the value $Z(\tilde T)$ is independent of the choice of perturbation $\tilde T$ of $T$,
and that the map $T\mapsto Z(T)$ is well-defined on isotopy classes of tangles.

It remains to show that $((\cP[n])_{n\ge 0}, Z)$ satisfies the axioms of an anchored planar algebra.
The action of the identity tangle was computed in Example \ref{ex:Identity}.
The gluing axiom was proved in Proposition \ref{prop:Gluing}.
Finally, the two anchor dependence axioms were established in Examples~\ref{ex:SwapAnchorDependence} and \ref{ex:2PiRotation}, respectively.
\end{proof}



\section{Anchored planar algebras from module tensor categories}
\label{sec:APAfromMTC}

In this section we will use Theorem~\ref{thm:ConstructAPA} in combination with the results of our previous paper \cite{1509.02937}
to construct an anchored planar algebra $\cP$ from a pair $(\cM, m)$ consisting of a module tensor category $\cM$ and a symmetrically self-dual object $m\in\cM$.
The anchored planar algebra only depends on the sub-module tensor category generated by $m$,
so we may assume without loss of generality that $m$ generates $\cM$ as a $\cC$-module tensor category.

\subsection{Diagram cheat sheet from \texorpdfstring{\cite{1509.02937}}{HPT15}}
\label{sec:TubeRelations}

Let $\cC$ be a braided pivotal category, and let $\cM$ be a module tensor category whose action functor $\Phi:\cC\to \cM$ admits a right adjoint $\Tr_\cC:\cM\to \cC$.
Following \cite{1509.02937}, we have a diagrammatic calculus of strings on cylinders for the categorified trace $\Tr_\cC$,
where the tubes are allowed to branch and braid.
We recall the main results of our earlier paper, which we will use to construct anchored planar algebras
(the objects of $\cM$ which appear below are not assumed to be self-dual, and so the strings should really be oriented).

\input{Chapters/TubeRelations.tex}

\subsection{Constructing the anchored planar algebra}
\label{sec:Constructing}

Let $\cC$ be a braided pivotal category.
For the remainder of this section, we fix a module tensor category  $\cM$ and a symmetrically self-dual object $m\in \cM$. 
We assume that $\Phi:= F\circ \Phi^{\scriptscriptstyle \cZ}:\cC\to\cM$ has a right adjoint, denoted
\[
\Tr_\cC:\cM\to\cC.
\]
Recall the morphisms $\bar\ev_m:m\otimes m \to 1$ and $\bar\coev_m:1 \to m \otimes m$ from \eqref{eq: ev bar and coev bar}.

\begin{thm}
\label{thm: construct P from M and m}
Let $\cC$ be a braided pivotal category.
Let $\cM$ and $m$ be as above, and let
\begin{equation}\label{eq:   cP[n]:=Tr_cC(m^otimes n)}
\cP[n]:=\Tr_\cC(m^{\otimes n}).
\end{equation}
Then the maps $\eta$, $\alpha_i$, $\bar\alpha_i$, $\varpi_{i,j}$ defined below satisfy the assumptions of Theorem~\ref{thm:ConstructAPA},
and thus endow $\cP=(\cP[n])_{n\ge 0}$ with the structure of an anchored planar algebra in~$\cC$:
\begin{enumerate}[label={}]
\item
$\eta:=
i:1\to\Tr_\cC(1_\cM)\,\,:
\quad 
\begin{tikzpicture}[baseline=0cm]
	\coordinate (a1) at (0,0);
	\coordinate (b1) at (0,.4);
	\draw[thick] (a1) -- (b1);
	\draw[thick] ($ (a1) + (.6,0) $) -- ($ (b1) + (.6,0) $);
	\draw[thick] ($ (b1) + (.3,0) $) ellipse (.3 and .1);
	\draw[thick] (a1) arc (-180:0:.3cm);
\end{tikzpicture}
$\vspace{-.3cm}

\item
$\alpha_i :=   \Tr_\cC(\id_{m^{\otimes i}} \otimes\bar\ev_m \otimes \id_{m^{\otimes n-i}}) :  \Tr_\cC(m^{\otimes n+2}) \to \Tr_\cC(m^{\otimes n})\,\,:\,\,\,
\begin{tikzpicture}[baseline=.6cm]

	\coordinate (a1) at (0,.1);
	\coordinate (b1) at (0,1.6);
	\coordinate (c) at (.35,.6);
	
	\draw[thick] (a1) -- (b1);
	\draw[thick] ($ (a1) + (1,0) $) -- ($ (b1) + (1,0) $);
	\halfDottedEllipse{(a1)}{.5}{.2}
	\draw[thick] ($ (b1) + (.5,0) $) ellipse (.5cm and .2cm);

	\draw[thick, orange] ($ (a1) + (.15,-.13) $) -- ($ (b1) + (.15,-.13) $);	
	\draw[thick, blue] ($ (a1) + (.35,-.18) $) -- (c) arc (180:0:.15cm) -- ($ (a1) + (.65,-.18) $);	
	\draw[thick, DarkGreen] ($ (a1) + (.85,-.13) $) -- ($ (b1) + (.85,-.13) $);	
\node[orange] at (-.5,.3) {$m^{\otimes i}$};
\node[blue, scale=.97] at (.5,1) {$m$};
\node[DarkGreen] at (1.75,.3) {$m^{\otimes n-i}$};
\end{tikzpicture}
$

\item
$\bar\alpha_i := \Tr_\cC(\id_{m^{\otimes i}} \otimes\bar\coev_m \otimes \id_{m^{\otimes n-i}}) : \Tr_\cC(m^{\otimes n}) \to \Tr_\cC(m^{\otimes n+2})\,\,:\,\,\,
\begin{tikzpicture}[baseline=.6cm]

	\coordinate (a1) at (0,.1);
	\coordinate (b1) at (0,1.6);
	\coordinate (c) at (.35,.9);
	
	\draw[thick] (a1) -- (b1);
	\draw[thick] ($ (a1) + (1,0) $) -- ($ (b1) + (1,0) $);
	\halfDottedEllipse{(a1)}{.5}{.2}
	\draw[thick] ($ (b1) + (.5,0) $) ellipse (.5cm and .2cm);

	\draw[thick, orange] ($ (a1) + (.15,-.13) $) -- ($ (b1) + (.15,-.13) $);	
	\draw[thick, blue] ($ (b1) + (.35,-.18) $) -- (c) arc (-180:0:.15cm) -- ($ (b1) + (.65,-.18) $);	
	\draw[thick, DarkGreen] ($ (a1) + (.85,-.13) $) -- ($ (b1) + (.85,-.13) $);	
\node[orange] at (-.5,.3) {$m^{\otimes i}$};
\node[blue, scale=.97] at (.5,.5) {$m$};
\node[DarkGreen] at (1.75,.3) {$m^{\otimes n-i}$};
\end{tikzpicture}
$

\item
\[
\begin{split}
\varpi_{i,j} := \tau^-_{m^{\otimes n-i},m^{\otimes i+j}} \circ \mu_{m^{\otimes n},m^{\otimes j}} &\circ (\tau^+_{m^{\otimes i},m^{\otimes n-i}} \otimes 1_{\Tr_\cC(m^{\otimes j})}):\\
\Tr_\cC(&m^{\otimes n}) \otimes \Tr_\cC(m^{\otimes j}) \to \Tr_\cC(m^{\otimes n+j})\,\,\,:\end{split}\,
\begin{tikzpicture}[baseline=.7cm]
	\pgfmathsetmacro{\voffset}{.08};
	\pgfmathsetmacro{\hoffset}{.15};
	\pgfmathsetmacro{\hoffsetTop}{.12};
	\coordinate (a1) at (-1,-1);
	\coordinate (a2) at ($ (a1) + (1.4,0)$);
	\coordinate (a3) at ($ (a1) + (2.8,0)$);
	\coordinate (b1) at ($ (a1) + (0,1)$);
	\coordinate (b2) at ($ (b1) + (1.4,0) $);
	\coordinate (b3) at ($ (b2) + (1.4,0) $);
	\coordinate (c1) at ($ (b1) + (.7,1.5)$);
	\coordinate (c2) at ($ (b2) + (.7,1.5)$);
	\coordinate (d1) at ($ (c1) + (0,1)$);
	\coordinate (d2) at ($ (c2) + (0,1)$);
	\pairOfPants{(b2)}{}
	\draw[thick] (c2) -- ($ (c2) + (0,1) $);
	\draw[thick] ($ (c2) + (.6,0) $) -- ($ (c2) + (.6,1) $);
	\draw[thick] ($ (d2) + (.3,0) $) ellipse (.3cm and .1cm);
	\draw[thick] (a2) -- ($ (a2) + (0,1) $);
	\draw[thick] ($ (a2) + (.6,0) $) -- ($ (a2) + (.6,1) $);
	\halfDottedEllipse{(a2)}{.3}{.1}	
	\draw[thick] (a3) -- ($ (a3) + (0,1) $);
	\draw[thick] ($ (a3) + (.6,0) $) -- ($ (a3) + (.6,1) $);
	\halfDottedEllipse{(a3)}{.3}{.1}	
%
	\draw[thick, orange] ($ (a2) + (\hoffset,0) + (0,-.1)$) .. controls ++(90:.4cm) and ++(270:.4cm) .. ($ (a2) + 3*(\hoffset,0) + (0,-\voffset) + (0,1) $);		
	\draw[thick, blue] ($ (a2) + 3*(\hoffset,0) + (0,-\voffset) $) .. controls ++(90:.2cm) and ++(225:.1cm) .. ($ (a2) + 4*(\hoffset,0) + (0,-\voffset) + (0,.45)$);
	\draw[thick, blue] ($ (a2) + (\hoffset,1) + (0,-\voffset) $) .. controls ++(270:.2cm) and ++(45:.1cm) .. ($ (a2) + (0,1) + (0,-\voffset) + (0,-.45)$);
%
	\draw[thick, orange] ($ (c2) + 2*(\hoffset,0) + (0,-.1)$) .. controls ++(90:.4cm) and ++(270:.4cm) .. ($ (c2) + (\hoffset,0) + (0,-\voffset) + (0,1) $);
	\draw[thick, DarkGreen] ($ (c2) + 3*(\hoffset,0) + (0,-\voffset) $) .. controls ++(90:.4cm) and ++(270:.4cm) .. ($ (c2) + 2*(\hoffset,0) + (0,.9) $);		
	\draw[thick, blue] ($ (c2) + (\hoffset,0) + (0,-\voffset) $) .. controls ++(90:.2cm) and ++(-45:.1cm) .. ($ (c2) + (0,-\voffset) + (0,.45)$);
	\draw[thick, blue] ($ (c2) + 3*(\hoffset,0) + (0,1) + (0,-\voffset) $) .. controls ++(270:.2cm) and ++(135:.1cm) .. ($ (c2) + 4*(\hoffset,0) + (0,-\voffset) + (0,-.45) + (0,1)$);
%
	\draw[thick, blue] ($ (b2) + (\hoffset,0) + (0,-\voffset)$) .. controls ++(90:.8cm) and ++(270:.8cm) .. ($ (c2) + (\hoffset,0) + (0,-\voffset) $);	
	\draw[thick, orange] ($ (b2) + 3*(\hoffset,0) + (0,-\voffset)$) .. controls ++(90:.8cm) and ++(270:.8cm) .. ($ (c2) + 2*(\hoffset,0) + (0,-.1) $);	
	\draw[thick, DarkGreen] ($ (a3) + 2*(\hoffset,0) + (0,-.1)$) -- ($ (b3) + 2*(\hoffset,0) + (0,-.1)$) .. controls ++(90:.8cm) and ++(270:.8cm) .. ($ (c2) + 3*(\hoffset,0) + (0,-\voffset) $);	
\node[orange] at (.3,-1.3) {$m^{\otimes i}$};
\node[blue] at (1.3,-1.3) {$m^{\otimes n-i}$};
\node[DarkGreen] at (2.95,-.8) {$m^{\otimes j}$};
\end{tikzpicture}
\]
\end{enumerate}
\end{thm}

\begin{proof}
We show that the maps $\eta$, $\alpha_i$, $\bar\alpha_i$, $\varpi_{i,j}$ satisfy the relations \ref{reln:UnitMap} -- \ref{reln:RotationThetaMaps} in Theorem~\ref{thm:ConstructAPA}.

The relation \ref{reln:UnitMap} holds by \ref{rel:CreateUnit}.
The three relations \ref{reln:CapMaps} -- \ref{reln:CapCupMaps} hold because the maps $\id_{m^{\otimes i}} \otimes\,\bar\ev_m \otimes \id_{m^{\otimes n-i}}$
and $\id_{m^{\otimes i}} \otimes\,\bar\coev_m \otimes \id_{m^{\otimes n-i}}$ satisfy corresponding relations in $\cM$, and $\Tr_\cC$ is a functor.
The relations \ref{reln:CapQuadraticMaps} -- \ref{reln:CupQuadraticMaps} holds by \ref{rel:MoveTensorThroughMultiplication} and \ref{The traciator is natural}.

Using \ref{rel:TraciatorComposition}, the relation \ref{reln:EasyQuadraticMaps} can be shown 
to be equivalent to the first equation in \ref{rel:MultiplicationAssociative} with
$x=m^{\otimes m+i}$, $y=m^{\otimes j}$, $z=m^{\otimes \ell}$, $w=m^{\otimes k}$, upon precomposing both sides by
$\tau_{m^{\otimes i},m^{\otimes m}}\otimes\id_{\Tr_\cC(m^{\otimes j+\ell})\otimes\Tr_\cC(m^{\otimes k})}$ and postcomposing both sides by $\tau^-_{m^{\otimes m},m^{\otimes i+j+k+\ell}}$.

Again using \ref{rel:TraciatorComposition}, the relation \ref{reln:HardQuadraticMaps} can be shown 
to be equivalent to the second equation in \ref{rel:MultiplicationAssociative} with
$x=m^{\otimes m+i}$, $y=m^{\otimes k}$, $z=m^{\otimes j}$, $w=m^{\otimes \ell}$, upon precomposing both sides by
$\tau_{m^{\otimes i+k},m^{\otimes m}}\otimes\id_{\Tr_\cC(m^{\otimes j})\otimes\Tr_\cC(m^{\otimes \ell})}$ and postcomposing both sides by $\tau^-_{m^{\otimes m},m^{\otimes i+j+k+\ell}}$.

Finally, to show \ref{reln:RotationThetaMaps}, we note that the left hand side of that relation can be represented diagrammatically as
$$
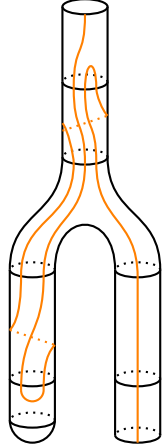
\begin{tikzpicture}[baseline=2.65cm]

	\coordinate (d1) at (0,0);
	\coordinate (a1) at (0,1);
	\coordinate (a2) at (1.4,-.2);
	\coordinate (b1) at (0,2);
	\coordinate (b2) at (1.4,2);
	\coordinate (c1) at (.7,3.5);
	\coordinate (e1) at (.7,4.5);
	\coordinate (f1) at (.7,5.5);
	\coordinate (z) at (.15,.4);
	
	\draw[thick] (c1) -- (e1);
	\draw[thick] ($ (c1) + (.6,0) $) -- ($ (e1) + (.6,0) $);
	\halfDottedEllipse{(e1)}{.3}{.1}
	\draw[thick] (f1) -- (e1);
	\draw[thick] ($ (f1) + (.6,0) $) -- ($ (e1) + (.6,0) $);
	\draw[thick] ($ (f1) + (.3,0) $) ellipse (.3cm and .1cm);

	\pairOfPants{(b1)}{}

	\draw[thick] (a2) -- (b2);
	\draw[thick] ($ (a2) + (.6,0) $) -- ($ (b2) + (.6,0) $);
	\halfDottedEllipse{(a2)}{.3}{.1}
	\halfDottedEllipse{($ (a2) + (0,.75) $)}{.3}{.1}
		
	\draw[thick] (d1) -- (b1);
	\draw[thick] ($ (d1) + (.6,0) $) -- ($ (b1) + (.6,0) $);
	\halfDottedEllipse{(d1)}{.3}{.1}
	\halfDottedEllipse{($ (d1) + (0,.55) $)}{.3}{.1}

	\draw[thick] (d1) arc (-180:0:.3cm);		

	\draw[thick, orange] ($ (a2) + (.3,-.1) $) -- ($ (b2) + (.3,-.1) $) .. controls ++(90:.8cm) and ++(270:.8cm) .. ($ (c1) + (.45,-.08) $);

	\draw[thick, orange] (z) arc (-180:0:.15cm)  .. controls ++(90:.2cm) and ++(225:.2cm) .. ($ (z) + (.45,.6) $);		
	\draw[thick, orange] ($ (c1) + (.15,-.08) $) .. controls ++(270:.8cm) and ++(90:.8cm) .. ($ (b1) + (.15,-.08) $) .. controls ++(270:.2cm) and ++(45:.2cm) .. ($ (z) + (-.15,.8) $);
	\draw[thick, orange, dotted] ($ (z) + (-.15,.8) $) -- ($ (z) + (.45,.6) $);	
	\draw[thick, orange] (z) .. controls ++(90:.8cm) and ++(270:.8cm) .. ($ (b1) + (.45,-.08) $) .. controls ++(90:.8cm) and ++(270:.8cm) .. ($ (c1) + (.3,-.1) $);

	\draw[thick, orange] ($ (e1) + (.45,-.08) $) .. controls ++(270:.1cm) and ++(135:.1cm) .. ($ (e1) + (.6,-.45) $);
	\draw[thick, orange] ($ (c1) + (.15,-.08) $) .. controls ++(90:.1cm) and ++(-45:.1cm) .. ($ (c1) + (0,.45) $);
	\draw[thick, orange, dotted] ($ (c1) + (0,.35) $) -- ($ (e1) + (.6,-.45) $);
	\draw[thick, orange] ($ (e1) + (.15,-.08) $) .. controls ++(270:.6cm) and ++(90:.6cm) .. ($ (c1) + (.3,-.1) $);
	\draw[thick, orange] ($ (e1) + (.3,-.1) $) .. controls ++(270:.6cm) and ++(90:.6cm) .. ($ (c1) + (.45,-.08) $);
	
	\draw[thick, orange] ($ (e1) + (.3,-.1) $) .. controls ++(90:.4cm) and ++(90:.4cm) .. ($ (e1) + (.45,-.08) $);
	\draw[thick, orange] ($ (f1) + (.3,-.1) $) .. controls ++(270:.6cm) and ++(90:.6cm) .. ($ (e1) + (.15,-.08) $);
	
\end{tikzpicture}
$$
with the orange strand standing for $m^{\otimes n}$.
By applying in order \ref{rel:StringOverCap}, \ref{rel:MoveTensorThroughMultiplication}, \ref{rel:CreateUnit}, \ref{rel:DualsAndTraciator}, and \ref{rel:ThetaAndTraciator},
we see that this map is equal to $\theta_{\Tr_\cC(m^{\otimes n})}$, as desired.
\end{proof}

We denote by $\Lambda(\cM,m)$
the anchored planar algebra associated to the pointed module tensor category $(\cM,m)$ by means of the construction described in Theorem~\ref{thm: construct P from M and m}.

\subsection{Functoriality}

In the previous section, given a pivotal module tensor category $\cM$ and a symmetrically self-dual object $m\in\cM$,
we constructed an anchored planar algebra $\cP=\Lambda(\cM,m)$.
Our next goal is to upgrade this to a functor
\[
\Lambda:\Mod_* \to \APA
\]
between the category of pointed pivotal module tensor categories over $\cC$ and the category of anchored planar algebras in $\cC$.
We recall that all module tensor categories are assumed pivotal, even if do not mention it explicitly.

Let $(\cM_1, m_1)$ and $(\cM_2, m_2)$ be pointed module tensor categories,
and let $\cP_1$ and $\cP_2$ be the associated anchored planar algebras.
Given a morphism $G:(\cM_1, m_1)\to(\cM_2, m_2)$ in $\Mod_*$ (Definition~\ref{def: pointed}), the wish to describe the corresponding morphism of anchored planar algebras
\begin{equation*}
\Lambda(G)\,:\,\,\cP_1=\Lambda(\cM_1, m_1)\,\longrightarrow\, \cP_2=\Lambda(\cM_2, m_2).
\end{equation*}

Recall from Definition~\ref{defn:ModuleTensorCategoryFunctor} that a morphism of module tensor categories is a pair $G=(G,\gamma)$, with $\gamma:\Phi_2\Rightarrow G\Phi_1$.
Let $\Tr_\cC^i : \cM_i\to \cC$, $i=1,2$, be the categorified traces associated to $\cM_1$ and $\cM_2$.
\begin{defn}\label{def: mate of zeta}
For each $x\in \cM_1$, we define $\zeta_x : \Tr_\cC^1(x) \to \Tr_\cC^2(G(x))$ to be the mate of
$$
\Phi_2(\Tr_\cC^1(x)) \xrightarrow{\gamma_{\Tr_\cC^1(x)}} G(\Phi_1(\Tr_\cC^1(x))) \xrightarrow{G(\varepsilon_x^1)} G(x)
$$
under the adjunction $\Phi_2 \dashv \Tr_\cC^2$.
\end{defn}

\begin{defn}\label{def: morphism of APA associated to G:M-->M'}
The map $\Lambda(G):\cP_1\to\cP_2$ associated to $G:(\cM_1, m_1)\to(\cM_2, m_2)$ is the sequence of morphisms
\begin{equation}\label{eq: this is a morphism of APA}
\big(\Lambda(G)[n]:\cP_1[n]\to\cP_2[n]\big)_{n\ge 0}
\end{equation}
given by\vspace{-.2cm}
\[
\Lambda(G)[n]\,:\,\cP_1[n] 
= \Tr_\cC^1(m_1^{\otimes n}) 
\xrightarrow{\,\,\,\,\textstyle\zeta_{m_1^{\otimes n}}\,\,\,} 
\Tr_\cC^2(G(m_1^{\otimes n})) 
\xrightarrow{\cong} 
\Tr_\cC^2(m_2^{\otimes n}) 
= 
\cP_2[n].
\smallskip
\]
\end{defn}

Our next goal is to show that this is a morphism of anchored planar algebras.
With this purpose in mind, we first list some properties of the natural transformation $\zeta$:

\begin{lem}
\label{lem:AttachingMapCompatible}
The following diagram commutes:
\[
\xymatrix{
\Phi_2(\Tr_\cC^1(x))
\ar[rr]^{\gamma_{\Tr^1_\cC(x)}}
\ar[d]^{\Phi_2(\zeta_{x})}
&&
G(\Phi_1(\Tr_\cC^1(x)))
\ar[d]^{G(\varepsilon^1_{x})}
\\
\Phi_2(\Tr_\cC^2(G(x)))
\ar[rr]^{\varepsilon^2_{G(x)}}
&&
G(x)
}
\]
\end{lem}
\begin{proof}
The morphism $\zeta_x$ is the mate of $G(\varepsilon^1_{x}) \circ \gamma_{\Tr^1_\cC(x)}$.
By \cite[Lem.\,4.3]{1509.02937}, it is also the mate of $\varepsilon^2_{G(x)} \circ \Phi_2(\zeta_{x})$.
The two composites therefore agree.
\end{proof}

\begin{lem}
\label{lem:MultiplicationCompatible}
For all $x,y\in \cM_1$, the following diagram commutes:
\[
\xymatrix{
\Tr_\cC^1(x)\otimes \Tr_\cC^1(y) 
\ar[rrr]^{\mu^1_{x,y}}
\ar[d]^{\zeta_x\otimes \zeta_y}
&&&
\Tr_\cC^1(x\otimes y)
\ar[d]^{\zeta_{x\otimes y}}
\\
\Tr_\cC^2(G(x))\otimes \Tr_\cC^2(G(y)) 
\ar[rr]^(.53){\mu^2_{G(x),G(y)}}
&&
\Tr_\cC^2(G(x)\otimes G(y))
\ar[r]^(.53){\cong}
&
\Tr_\cC^2(G(x\otimes y))
}
\]
\end{lem}
\begin{proof}
The mate of $\zeta_{x\otimes y}\circ\mu^1_{x,y}$ is $G(\varepsilon^1_{x\otimes y}) \,\circ \,\gamma_{\Tr^1_\cC(x\otimes y)}  \circ \Phi_2(\mu^1_{x,y})$,
and the mate of $\mu^2_{G(x),G(y)}\circ\zeta_x\otimes \zeta_y$ is $\varepsilon^2_{G(x)\otimes G(y)} \circ \Phi_2(\mu_{G(x),G(y)}^2 \circ (\zeta_x \otimes \zeta_y))$.
It is therefore enough to show that 
\[
G(\varepsilon^1_{x\otimes y}) \,\circ \,\gamma_{\Tr^1_\cC(x\otimes y)}  \circ \Phi_2(\mu^1_{x,y})
\,=\,
\varepsilon^2_{G(x)\otimes G(y)} \circ \Phi_2(\mu_{G(x),G(y)}^2 \circ (\zeta_x \otimes \zeta_y))
\]
(note that we have suppressed the isomorphism $\nu_{x,y}:G(x)\otimes G(y) \to G(x\otimes y)$).
This is verified as follows:
\begin{align*}
G(\varepsilon^1_{x\otimes y}) \,\circ \,&\gamma_{\Tr^1_\cC(x\otimes y)}  \circ \Phi_2(\mu^1_{x,y})
\\&=
G(\varepsilon^1_{x\otimes y}) \circ G(\Phi_1(\mu_{x,y}^1)) \circ \gamma_{\Tr^1(x)\otimes \Tr^1(y)}
&&\text{(naturality of $\gamma$)}
\\&=
G(\varepsilon^1_{x\otimes y} \circ \Phi_1(\mu_{x,y}^1)) \circ \gamma_{\Tr^1(x)\otimes \Tr^1(y)}
&&
\\&=
G(\varepsilon^1_{x}\otimes \varepsilon^1_{y}) \circ \gamma_{\Tr^1(x)\otimes \Tr^1(y)}
&&\text{(\cite[Lem.\,4.6]{1509.02937})}
\\&=
(G(\varepsilon^1_{x})\otimes G(\varepsilon^1_{y})) \circ (\gamma_{\Tr^1(x)}\otimes \gamma_{\Tr^1(y)})
&&
\\&=
(G(\varepsilon^1_{x})\circ \gamma_{\Tr^1(x)})
\otimes
(G(\varepsilon^1_{y}) \circ \gamma_{\Tr^1(y)})
&&
\\&=
(\varepsilon^2_{G(x)} \circ \Phi_2(\zeta_{x}))
\otimes
(\varepsilon^2_{G(y)} \circ \Phi_2(\zeta_{y}))
&&(\text{Lemma \ref{lem:AttachingMapCompatible}})
\\&=
(\varepsilon^2_{G(x)} \otimes \varepsilon^2_{G(y)})
\circ
\Phi_2(\zeta_{x} \otimes \zeta_{y})
&&
\\&=
\varepsilon^2_{G(x)\otimes G(y)} \circ \Phi_2(\mu_{G(x),G(y)}^2) \circ \Phi_2(\zeta_x \otimes \zeta_y)
&&\text{(\cite[Lem.\,4.6]{1509.02937})}
\\&=
\varepsilon^2_{G(x)\otimes G(y)} \circ \Phi_2(\mu_{G(x),G(y)}^2 \circ (\zeta_x \otimes \zeta_y))
&&
\qedhere
\end{align*}
\end{proof}

Let us write $\tau^i_{x,y}:\Tr^i_\cC(x\otimes y)\to \Tr^i_\cC(y\otimes x)$, $i=1,2$, for the traciator associatied to $\Tr^i:\cM_i\to \cC$.
Recall from Definition \ref{defn:ModuleTensorCategoryFunctor} that for every $c\in\cC$ and $x\in\cM_1$, we have
\begin{equation}\label{eq: from defn:ModuleTensorCategoryFunctor}
G(e_{\Phi_1(c),x})\circ(\gamma_c\otimes \id_{G(x)})=(\id_{G(x)}\otimes \gamma_c)\circ e_{\Phi_2(c), G(x)}.
\end{equation}

\begin{lem}
\label{lem:TraciatorsCompatible}
The following diagram is commutative:
\[
\xymatrix{
\Tr_\cC^1(x\otimes y) 
\ar[d]^{\zeta_{x\otimes y}}
\ar[rr]^{\tau_{x,y}^1}
&&
\Tr_\cC^1(y\otimes x)
\ar[d]^{\zeta_{y\otimes x}}
\\
\Tr_\cC^2(G(x\otimes y))
\ar[d]^{\cong}
&&
\Tr_\cC^2(G(x\otimes y))
\ar[d]^{\cong}
\\
\Tr_\cC^2(G(x)\otimes G(y))
\ar[rr]^{\tau_{G(x),G(y)}^2}
&&
\Tr_\cC^2(G(y)\otimes G(x))
}
\]
\end{lem}
\begin{proof}
We verify $\zeta_{y\otimes x}\circ\tau_{x,y}^1=\tau_{G(x),G(y)}^2\circ\zeta_{x\otimes y}$ upon taking mates.
Let $c:=\Tr^1_\cC(x\otimes y)$ and $d:=\Tr^1_\cC(y\otimes x)$.
The mate of $\zeta_{y\otimes x}\circ\tau_{x,y}^1$ is:
\begin{align*}
G(\varepsilon^1_{y\otimes x}) \circ \gamma_{d} \circ \Phi_2(\tau^1_{x,y}) 
&= G(\varepsilon^1_{y\otimes x})  \circ G\Phi_1(\tau_{x,y}^1) \circ \gamma_{c} 
\\&= G(\varepsilon^1_{y\otimes x} \circ \Phi_1(\tau_{x,y}^1)) \circ \gamma_{c} 
= G(\text{mate of $\tau_{x,y}^1$}) \circ \gamma_{c}.
\end{align*}
By definition (see \ref{rel:Traciator} or \cite[Def.\,4.13]{1509.02937}), the mate of $\tau^1_{x,y}$ is
\begin{equation}\label{equation}
(\id_{y\otimes x} \otimes\, \widetilde{\ev}_y) \circ (\id_y \otimes \varepsilon^1_{x\otimes y} \otimes \id_{y^*}) \circ (e_{\Phi_1(c),y} \otimes \id_{y^*}) \circ (\id_{\Phi_1(c)} \otimes \coev_y)
\end{equation}
where $\widetilde{\ev}_y=\ev_{y^*}\circ (\varphi_y\otimes \id_{y^*})$ and $\varphi_y: y\to y^{**}$ is the pivotal structure.
Let $\dot x:=G(x)$ and $\dot y:=G(y)$.
Applying $G$ to \eqref{equation} and precomposing with $\gamma_{\Tr^1_\cC(x\otimes y)}$, we obtain:
\begin{align*}
&\,\,\,\,\,\,G(\text{mate of\, $\tau_{x,y}^1$}) \circ \gamma_c 
\\&=
(\id_{\dot y\otimes \dot x} \otimes \widetilde{\ev}_{\dot y}) \circ (\id_{\dot y} \otimes G(\varepsilon^1_{x\otimes y}) \otimes \id_{\dot y^*}) \circ (G(e_{\Phi_1(c),y}) \otimes \id_{\dot y^*}) \circ (\id_{G\Phi_1(c)} \otimes \coev_{\dot y}) \circ\gamma_{c} 
\\&=
(\id_{\dot y\otimes \dot x} \otimes \widetilde{\ev}_{\dot y}) \circ (\id_{\dot y} \otimes G(\varepsilon^1_{x\otimes y}) \otimes \id_{\dot y^*}) \circ 
([ G(e_{\Phi_1(c),y})\circ (\gamma_{c} \otimes \id_{\dot y} )]\otimes \id_{\dot y^*})
\circ (\id_{\Phi_2(c)}\otimes \coev_{\dot y})
\\&=
(\id_{\dot y\otimes \dot x} \otimes \widetilde{\ev}_{\dot y}) \circ (\id_{\dot y} \otimes G(\varepsilon^1_{x\otimes y}) \otimes \id_{\dot y^*}) \circ 
([(\id_{\dot y} \otimes \gamma_{c})\circ e_{\Phi_2(c),\dot y}]\otimes \id_{\dot y^*})
\circ (\id_{\Phi_2(c)}\otimes \coev_{\dot y})
\\&=
(\id_{\dot y\otimes \dot x} \otimes \widetilde{\ev}_{\dot y}) 
\circ 
(\id_{\dot y} \otimes [G(\varepsilon^1_{x\otimes y})\circ   \gamma_{c}] \otimes \id_{\dot y^*}) 
\circ 
(e_{\Phi_2(c),\dot y} \otimes \id_{\dot y^*})
\circ 
(\id_{\Phi_2(c)}\otimes \coev_{\dot y})
\\&=
(\id_{\dot y\otimes \dot x} \otimes \widetilde{\ev}_{\dot y}) 
\circ 
(\id_{\dot y} \otimes [\varepsilon^2_{\dot x\otimes \dot y} \circ \Phi_2(\zeta_{x\otimes y})] \otimes \id_{\dot y^*}) 
\circ 
(e_{\Phi_2(c),\dot y} \otimes \id_{\dot y^*})
\circ 
(\id_{\Phi_2(c)}\otimes \coev_{\dot y})
\\&=
(\id_{\dot y\otimes \dot x} \otimes \widetilde{\ev}_{\dot y}) 
\circ 
(\id_{\dot y} \otimes \varepsilon^2_{\dot x\otimes \dot y} \otimes \id_{\dot y^*}) 
\circ 
(e_{\Phi_2(d),\dot y} \otimes \id_{\dot y^*}) \circ (\id_{\Phi_2(d)} \otimes \coev_{\dot y}) 
\circ 
\Phi_2(\zeta_{x\otimes y})
\\&=\,\,
\text{mate of $(\tau_{\dot x,\dot y}^2\circ\zeta_{x\otimes y})$.}
\end{align*}
The first equality holds because $G$ is a pivotal functor, the third one holds by equation \eqref{eq: from defn:ModuleTensorCategoryFunctor},
the fifth one holds by Lemma \ref{lem:AttachingMapCompatible}, and the sixth one holds by the naturality of $e_{\Phi_2(-),\dot y}$ (see \cite[Def.\,3.5]{1509.02937}).
\end{proof}

\begin{lem}\label{lem: all about i}
The following diagram commutes:
\[
\xymatrix{
\,1_\cC\,\ar[d]^{i^2}\ar[r]^(.45){i^1}&\Tr_\cC^1(1_{\cM_1})\ar[d]^{\zeta_1}\\
\Tr_\cC^2(1_{\cM_2})\ar[r]^(.45){\cong}&\Tr_\cC^1(G(1_{\cM_1}))
}
\]
\end{lem}

\begin{proof}
We check the equality $\zeta_1\circ i^1=i^2$ upon taking mates.
Recall that $i^2$ is the unit of the adjunction $\Phi_2\dashv\Tr_\cC^2$, evaluated on the object $1_{\cM_2}$.
Its mate is the identity on $1_{\cM_2}$.
We now compute the mate of $\zeta_1\circ i^1$.
To lessen the notational confusion between indices and unit objects, we write $\underline 1$ for $1_{\cM_1}$.
We have:
\[
\begin{split}
\text{mate of $(\zeta_{\underline 1}\circ i^1)$}
=\,\,&\varepsilon^2_{G(\underline 1)}\circ \Phi_2(\zeta_{\underline 1}\circ i^1)
\\=\,\,&\varepsilon^2_{G(\underline 1)}\circ \Phi_2(\zeta_{\underline 1})\circ\Phi_2(i^1)
\\=\,\,&G(\varepsilon^1_{\underline 1})\circ\gamma_{\Tr^2_\cC(\underline 1)}\circ\Phi_2(i^1)
\\=\,\,&G(\varepsilon^1_{\underline 1})\circ G\Phi_1(i^1)\circ\gamma_{1_\cC}
\\=\,\,&G(\varepsilon^1_{\underline 1}\circ \Phi_1(i^1))=G(\id_{\underline 1})=\id_{1_{\cM_2}}
\end{split}
\]
where the third equality holds by Lemma \ref{lem:AttachingMapCompatible}, the fourth equality is the naturality of $\gamma$,
and the sixth one is \cite[(12)]{1509.02937}.
Towards the end of the calculation, we have suppressed the unit coherences of $\Phi_i$ and of $G$.
\end{proof}

We are now in position to show that the sequence $(\Lambda(G)[n]:\cP_1[n]\to\cP_2[n])_{n\ge 0}$ defines a morphism of anchored planar algebras.

\begin{prop}
The map $\Lambda(G) : \cP_1 \to \cP_2$ given in Definition~\ref{def: morphism of APA associated to G:M-->M'} is a morphism of anchored planar algebras.
\end{prop}
\begin{proof}
In view of Theorem \ref{thm:ConstructAPA},
a morphism of anchored planar algebras $\cP_1\to\cP_2$ is equivalent to a sequence of morphisms $(\cP_1[n]\to\cP_2[n])_{n\ge 0}$ making the following diagrams commute:
\begin{equation*}
\xymatrix{
1\ar@{=}[d]\ar[r]^(.4){\eta}&\cP_1[0]\ar[d]
\\1\ar[r]^(.4){\eta}&\cP_2[0]
}
\quad
\xymatrix{
\cP_1[n+2]\ar[d]
\ar[r]^(.6){\alpha_i}&\cP_1[n]\ar[d]
\\\cP_2[n+2]\ar[r]^(.6){\alpha_i}&\cP_2[n]
}
\quad
\xymatrix{
\cP_1[n]\ar[d]
\ar[r]^(.4){\bar\alpha_i}&\cP_1[n+2]\ar[d]
\\\cP_2[n]\ar[r]^(.4){\bar\alpha_i}&\cP_2[n+2]
}
\quad
\xymatrix{
\cP_1[n]\otimes\cP_1[j]\ar[d]
\ar[r]^(.55){\varpi_{i,j}}&\cP_1[n+j]\ar[d]
\\\cP_2[n]\otimes\cP_2[j]\ar[r]^(.55){\varpi_{i,j}}&\cP_2[n+j]
}
\end{equation*}
We show that the morphisms \eqref{eq: this is a morphism of APA} satisfy the above conditions.

The commutativity of the first diagram is the content of Lemma \ref{lem: all about i}.
The commutativity of the middle diagrams
\[
\begin{matrix}\xymatrix{
\cP_1[n+2]\ar[d]_{\Lambda(G)[n+2]\,=\,\zeta_{m_1^{\otimes n+2}}}
\ar[r]^(.6){\alpha_i}&\cP_1[n]\ar[d]^{\zeta_{m_1^{\otimes n}}}
\\\cP_2[n+2]\ar[r]^(.6){\alpha_i}&\cP_2[n]
}\end{matrix}
\quad\,\,\,\text{and}\,\,\,\quad
\begin{matrix}\xymatrix{
\cP_1[n]\ar[d]_{\zeta_{m_1^{\otimes n}}}
\ar[r]^(.4){\bar\alpha_i}&\cP_1[n+2]\ar[d]^{\zeta_{m_1^{\otimes n+2}}}
\\\cP_2[n]\ar[r]^(.4){\bar\alpha_i}&\cP_2[n+2]
}\end{matrix}\qquad
\]
follows from the naturality of $\zeta$, and the relations $G(\bar{\ev}_{m_1}) = \bar{\ev}_{m_2}$ and $G(\bar{\coev}_{m_1}) = \bar{\coev}_{m_2}$, which we explain below.
We only show the first equation, as the other one is similar.
For clarity, we reintroduce the coherence $\nu$, which we had often suppressed.
The computation goes as follows:
\[
\begin{split}
G(\bar{\ev}_{m_1}) \circ \nu_{m_1,m_1}
&= G(\ev_{m_1} \circ (\psi_1\otimes \id_{m_1}) ) \circ \nu_{m_1,m_1}
\\&= G(\ev_{m_1}) \circ G(\psi_1\otimes \id_{m_1}) \circ \nu_{m_1,m_1}
\\&= G(\ev_{m_1}) \circ \nu_{m_1^*,m_1}\circ (G(\psi_1) \otimes\id_{m_2})
\\&= \ev_{G(m_1)} \circ ([\delta_{m_1} \circ G(\psi_1)] \otimes\id_{m_2})
= \ev_{m_2} \circ (\psi_2\otimes \id_{m_2}) 
= \bar{\ev}_{m_2},
\end{split}
\]
where the fourth equality uses the identity $G(\ev_{m_1}) \circ \nu_{m_1^*,m_1} = \ev_{G(m_1)} \circ (\delta_{m_1} \otimes \id)$
(which holds by the definition of $\delta$; see Definition \ref{defn:ModuleTensorCategoryFunctor}),
and the fifth equality follows from the relation $\delta_{m_1} \circ G(\psi_1) = \psi_2$ (see Definition~\ref{def: pointed}).

Finally, the commutativity of the diagram
\[
\xymatrix{
\cP_1[n]\otimes\cP_1[j]\ar[d]_{\Lambda(G)[n]\otimes\Lambda(G)[j]\,\,}
\ar[r]^(.55){\varpi_{i,j}}&\cP_1[n+j]\ar[d]^{\,\Lambda(G)[n+j]}
\\\cP_2[n]\otimes\cP_2[j]\ar[r]^(.55){\varpi_{i,j}}&\cP_2[n+j]
}
\]
follows from Lemmas \ref{lem:MultiplicationCompatible} and \ref{lem:TraciatorsCompatible}:
\[
\xymatrix{
\Tr_\cC^1(m_1^{\otimes n})\otimes\Tr_\cC^1(m_1^{\otimes j})\ar[d]|{\Lambda(G)[n]\otimes\Lambda(G)[j]}
\ar@{}[dr]|{\textstyle \ref{lem:TraciatorsCompatible}}
\ar[r]^(.5){\tau^+\otimes\id}
&
\ar@{}[dr]|{\textcolor{white}{---}\textstyle \ref{lem:MultiplicationCompatible}}
\Tr_\cC^1(m_1^{\otimes n})\otimes\Tr_\cC^1(m_1^{\otimes j})\ar[d]|{\Lambda(G)[n]\otimes\Lambda(G)[j]}
\ar[r]^(.6){\mu}
&
\ar@{}[dr]|{\textstyle \ref{lem:TraciatorsCompatible}}
\Tr_\cC^1(m_1^{\otimes n+j})\ar[d]|{\,\Lambda(G)[n+j]}\ar[r]^{\tau^-}
&
\Tr_\cC^1(m_1^{\otimes n+j})\ar[d]|{\,\Lambda(G)[n+j]}
\\
\Tr_\cC^2(m_2^{\otimes n})\otimes\Tr_\cC^2(m_2^{\otimes j})\ar[r]^(.5){\tau^+\otimes\id}
&
\Tr_\cC^2(m_2^{\otimes n})\otimes\Tr_\cC^2(m_2^{\otimes j})\ar[r]^(.6){\mu}
&
\Tr_\cC^2(m_2^{\otimes n+j})\ar[r]^{\tau^-}&\Tr_\cC^2(m_2^{\otimes n+j})
}
\]
\end{proof}

In the previous proposition, we have done all the hard work for:
\begin{thm}\label{thm: Here's Lambda!}
Let $\cC$ be a braided pivotal category. Then the map
\[
\Lambda:\Mod_* \to \APA
\]
given by Theorem~\ref{thm: construct P from M and m} at the level of objects 
and by Definition \ref{def: morphism of APA associated to G:M-->M'} at the level of morphisms 
is a functor
from the (2-)category of pointed pivotal module tensor categories over $\cC$ to the category of anchored planar algebras in $\cC$.
\end{thm}

\begin{proof}
Given two functors
\[
(\cM_1,m_1)\xrightarrow{(G_1,\gamma^1)} (\cM_2,m_2)\xrightarrow{(G_2,\gamma^2)}(\cM_3,m_3),
\]
with composite
$(G_3,\gamma^3):=(G_2,\gamma^2)\circ(G_1,\gamma^1)$, we still need to show that 
$\Lambda(G_3)=\Lambda(G_2)\circ\Lambda(G_1)$.
Thus, for every $n\in\mathbb N$, we need to show $\Lambda(G_3)[n]=\Lambda(G_2)[n]\circ\Lambda(G_1)[n]$, i.e.,
\begin{equation}\label{eq: GG=G}
\zeta^3_{m_1^{\otimes n}}=\zeta^2_{m_2^{\otimes n}}\circ\zeta^1_{m_1^{\otimes n}}.
\end{equation}
Note that, by definition, $\gamma^3_c=G_2(\gamma^1_c)\circ\gamma^2_c$.
We verify \eqref{eq: GG=G} upon taking mates:
\[
\begin{split}
\qquad\qquad\quad
\text{mate of $\zeta^3_{m_1^{\otimes n}}$}
  &=G_3(\varepsilon^1_{m_1^{\otimes n}})\circ\gamma^3_{\Tr_\cC^1(m_1^{\otimes n})}
\\&=G_2G_1(\varepsilon^1_{m_1^{\otimes n}})\circ G_2(\gamma^1_{\Tr_\cC^1(m_1^{\otimes n})})\circ\gamma^2_{\Tr_\cC^1(m_1^{\otimes n})}
\\&=G_2(G_1(\varepsilon^1_{m_1^{\otimes n}}) \circ \gamma^1_{\Tr_\cC^1(m_1^{\otimes n})})\circ\gamma^2_{\Tr_\cC^1(m_1^{\otimes n})}
\\&=G_2(\varepsilon^2_{G_1(m_1^{\otimes n})}\circ\Phi_2(\zeta^1_{m_1^{\otimes n}}))\circ\gamma^2_{\Tr_\cC^1(m_1^{\otimes n})}\qquad\quad(\text{Lemma \ref{lem:AttachingMapCompatible}})
\\&=G_2(\varepsilon^2_{m_2^{\otimes n}})\circ G_2\Phi_2(\zeta^1_{m_1^{\otimes n}})\circ\gamma^2_{\Tr_\cC^1(m_1^{\otimes n})}
\\&=G_2(\varepsilon^2_{m_2^{\otimes n}})\circ \gamma^2_{\Tr_\cC^2(m_2^{\otimes n})}\circ \Phi_3(\zeta^1_{m_1^{\otimes n}})\qquad\qquad\,\;(\text{naturality of $\gamma^2$})
\\&=\text{mate of $(\zeta^2_{m_2^{\otimes n}}\circ\zeta^1_{m_1^{\otimes n}})$}
\end{split}
\]

We also need to check that $\Lambda(\id_{(\cM,m)})=\id_{\Lambda(\cM,m)}$.
By definition, the map $\Lambda(\cM,m)[n]\to \Lambda(\cM,m)[n]$ associated to $(G,\gamma):=\id_{(\cM,m)}$
is the mate of $G(\varepsilon_{m^{\otimes n}})\circ \gamma_{\Tr_\cC(m^{\otimes n})}=\varepsilon_{m^{\otimes n}}$ (Definition~\ref{def: mate of zeta}).
But $\varepsilon$ is the counit of the adjunction, and so its mate is an identity morphism.
\end{proof}

We finish the section by noting that, since $\APA$ is a 1-category, the functor $\Lambda$ descends to a functor $\tau_{\le 1}(\Mod_*) \to \APA$.
We will not distinguish between these two meanings of $\Lambda$.



\section{Module tensor categories from anchored planar algebras}
\label{sec:MTCfromAPA}

In the previous section, we saw how to construct an anchored planar algebra $\cP$ from a pointed module tensor category $(\cM,\Phi^{\scriptscriptstyle \cZ},m)$.
Our next goal is to provide a construction that goes the other way.

Given braided pivotal category $\cC$ and an anchored planar algebra $\cP$ in $\cC$, we will construct a pivotal tensor category $\cM$,
a braided pivotal functor $\Phi^{\scriptscriptstyle \cZ}:\cC\to\cZ(\cM)$, and the symmetrically self-dual object $m$ that generates $\cM$ as a module tensor category (Definition~\ref{def: pointed}).
In Section \ref{sec:Equivalence}, we will show that these two constructions are each other's inverses.

\subsection{Reconstructing the category  \texorpdfstring{$\cM$}{M}}
\label{sec:DefinitionOfM}

In this section, given an anchored planar algebra, we construct a category $\cM$ with a distinguished object $m\in\cM$.

Our construction makes heavy use of the graphical calculus for morphisms in $\cC$.
Given an anchored planar algebra $(\cP,Z)$, an anchored planar tangle $T$ yields a morphism $Z(T)$ in $\cC$.
We represent this by a coupon labeled by the tangle $T$.
For example, the multiplication tangle
 $T=\begin{tikzpicture}[xscale=.5, yscale=.5, baseline = -.1cm]
	\draw[thick, red] (-.3,.5) -- (-1.15,0);
	\draw[thick, red] (-.3,-.5) -- (-1.15,0);
	\draw[very thick] (0,0) ellipse (1.15 and 1.25);
	\draw (0,-1.25) -- (0,1.25);
	\filldraw[very thick, unshaded] (0,.5) circle (.3cm);
	\filldraw[very thick, unshaded] (0,-.5) circle (.3cm);
	\node[scale=.6] at (.2+.05,-1) {\scriptsize{$n_3$}};
	\node[scale=.6] at (.2+.05,0) {\scriptsize{$n_2$}};
	\node[scale=.6] at (.2+.05,1) {\scriptsize{$n_1$}};
	\fill[red] (-.3,-.5) circle (.07) (-.3,.5) circle (.07) (-1.15,0) circle (.07);
\end{tikzpicture}$\,
yields a map\vspace{-.2cm}
 $Z(T):\cP[n_1+n_2]\otimes \cP[n_2+n_3] \to \cP[n_1+n_3]$, denoted
$$
Z(T)\,\,:\,\,\,\,\,\,\,\,\,\,\,
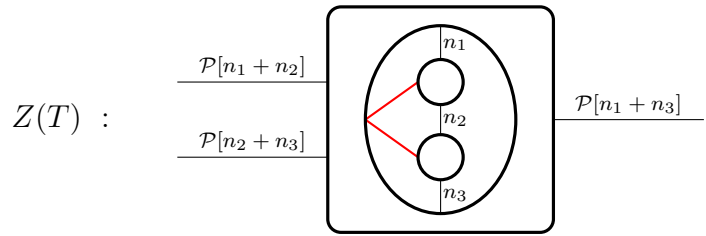
\begin{tikzpicture}[baseline = -.1cm]
	\draw (-3.5,-.5) -- (-1.5,-.5);
	\draw (-3.5,.5) -- (-1.5,.5);
	\draw (3.5,0) -- (1.5,0);
	\draw[rounded corners=5pt, very thick, unshaded] (-1.5,-1.5) rectangle (1.5,1.5);
	\draw[thick, red] (-.3,.5) -- (-1,0);
	\draw[thick, red] (-.3,-.5) -- (-1,0);
	\draw[very thick] (0,0) ellipse (1 and 1.25);
	\draw (0,-1.25) -- (0,1.25);
	\filldraw[very thick, unshaded] (0,.5) circle (.3cm);
	\filldraw[very thick, unshaded] (0,-.5) circle (.3cm);
	\node at (.2,-1) {\scriptsize{$n_3$}};
	\node at (.2,0) {\scriptsize{$n_2$}};
	\node at (.2,1) {\scriptsize{$n_1$}};
	\node at (-2.5,.7) {\scriptsize{$\cP[n_1+n_2]$}};
	\node at (-2.5,-.3) {\scriptsize{$\cP[n_2+n_3]$}};
	\node at (2.5,.2) {\scriptsize{$\cP[n_1+n_3]$}};
\end{tikzpicture}
$$
The above diagram reads from left to right;
later on, we will use some more complicated diagrams with strands coming in and out from multiple directions.
We adopt the convention that the strands that come in from the left and from the bottom are input strands, and that the strands that come out from
the top and from the right are output strands.

As a way towards constructing $\cM$, we first construct a certain full subcategory $\cM_0$ of $\cM$.
The latter is obtained from $\cM_0$ by formally adding finite direct sums, 
and passing to the idempotent completion. The objects of $\cM_0$ are formal symbols
$
\text{``}\Phi(c)\otimes m^{\otimes n}\text{''}
$,
where $c$ is an object of $\cC$, and $n\ge0$.
The hom spaces are given by
\begin{equation}\label{eq: that's the composition in M_0}
\cM_0\big(\text{``}\Phi(c)\otimes m^{\otimes n_1}\text{''},\text{``}\Phi(d)\otimes m^{\otimes n_2}\text{''}\big) \,:=\, \cC(c, d\otimes \cP[n_2 + n_1]).
\end{equation}
We represent an element $f$ of the above hom space as follows:
$$
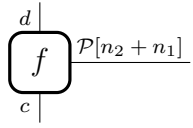
\begin{tikzpicture}[baseline=-.1cm]
	\draw (0,.8) -- (0,-.8);
	\draw (0,0) -- (2,0);
	\roundNbox{unshaded}{(0,0)}{.4}{0}{0}{$f$};
	\node at (-.2,.6) {\scriptsize{$d$}};
	\node at (1.2,.2) {\scriptsize{$\cP[n_2+n_1]$}};
	\node at (-.2,-.6) {\scriptsize{$c$}};
\end{tikzpicture}
$$
We then let
\begin{equation}\label{Phi M eq1}
m:=\text{``}\Phi(1_\cC)\otimes m^{\otimes 1}\text{''}\,\in\,\cM_0\subset \cM.
\end{equation}

The composition in $\cM_0$ of two morphisms $f:\text{``}\Phi(a)\otimes m^{\otimes n_1}\text{''}\to\text{``}\Phi(b)\otimes m^{\otimes n_2}\text{''}$ and $g:\text{``}\Phi(b)\otimes m^{\otimes n_2}\text{''}\to\text{``}\Phi(c)\otimes m^{\otimes n_3}\text{''}$
is given by
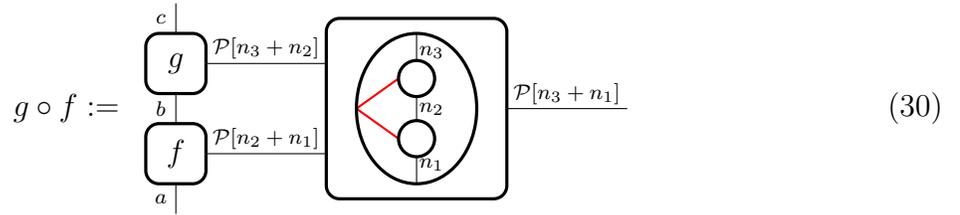
\begin{equation}\label{eq: that's how one defines composition}
g\circ f :=\,\,\,
\begin{tikzpicture}[baseline=-.1cm]
	\draw (0,1.4) -- (0,-1.4);
	\draw (.4,.6) -- (2,.6);
	\draw (.4,-.6) -- (2,-.6);
	\draw (4.4,0) -- (6,0);
	\roundNbox{unshaded}{(0,.6)}{.4}{0}{0}{$g$};
	\roundNbox{unshaded}{(0,-.6)}{.4}{0}{0}{$f$};
	\node at (-.2,1.2) {\scriptsize{$c$}};
	\node at (1.2,.8) {\scriptsize{$\cP[n_3+n_2]$}};
	\node at (-.2,0) {\scriptsize{$b$}};
	\node at (1.2,-.4) {\scriptsize{$\cP[n_2+n_1]$}};
	\node at (-.2,-1.2) {\scriptsize{$a$}};
	\multiplication{(3.2,0)}{1.2}{n_1}{n_2}{n_3}
	\node at (5.2,.2) {\scriptsize{$\cP[n_3+n_1]$}};
\end{tikzpicture}
\end{equation}
and the identity morphism on an object $\text{``}\Phi(c)\otimes m^{\otimes n}\text{''}$ is given by
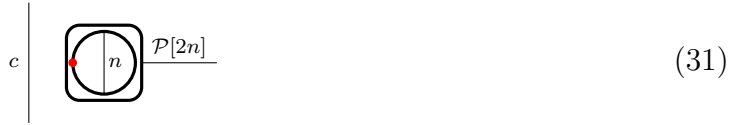
\begin{equation}\label{eq: ...and that's how one defines identities}
\begin{tikzpicture}[baseline=-.1cm]
	\draw (0,-.8) -- (0,.8);
	\draw (1.5,0) -- (2.5,0);
	\node at (-.2,0) {\scriptsize{$c$}};
	\identityMap{(1,0)}{.5}{n\,\,}
	\node at (2,.2) {\scriptsize{$\cP[2n]$}};
\end{tikzpicture}
\end{equation}
The following computations shows that the composition of morphisms is associative:
\begin{align*}
(h\circ g) \circ f\,\, &=\,\,\,
\begin{tikzpicture}[baseline=-.7cm]
	\draw (0,1.4) -- (0,-2.6);
	\draw (.4,.6) -- (2,.6);
	\draw (.4,-.6) -- (2,-.6);
	\draw (4.4,0) -- (6,0);
	\draw (.4,-1.8) -- (6,-1.8);
	\draw (9,-.9) -- (10.6,-.9);
	\roundNbox{unshaded}{(0,.6)}{.4}{0}{0}{$h$};
	\roundNbox{unshaded}{(0,-.6)}{.4}{0}{0}{$g$};
	\roundNbox{unshaded}{(0,-1.8)}{.4}{0}{0}{$f$};
	\node at (-.2,1.2) {\scriptsize{$d$}};
	\node at (1.2,.8) {\scriptsize{$\cP[n_4+n_3]$}};
	\node at (-.2,0) {\scriptsize{$c$}};
	\node at (1.2,-.4) {\scriptsize{$\cP[n_3+n_2]$}};
	\node at (-.2,-1.2) {\scriptsize{$b$}};
	\node at (3.2,-1.6) {\scriptsize{$\cP[n_2+n_1]$}};
	\node at (-.2,-2.4) {\scriptsize{$a$}};
	\multiplication{(3.2,0)}{1.2}{n_2}{n_3}{n_4}
	\node at (5.2,.2) {\scriptsize{$\cP[n_4+n_2]$}};
	\multiplication{(7.5,-.9)}{1.5}{n_1}{n_2}{n_4}
	\node at (9.8,-.7) {\scriptsize{$\cP[n_4+n_1]$}};
\end{tikzpicture}
\displaybreak[1]
\\&=\,\,\,
\begin{tikzpicture}[baseline=-.7cm]
	\draw (0,1.4) -- (0,-2.6);
	\draw (.4,.6) -- (2,.6);
	\draw (.4,-.6) -- (2,-.6);
	\draw (.4,-1.8) -- (2,-1.8);
	\draw (5.6,-.6) -- (7.2,-.6);
	\roundNbox{unshaded}{(0,.6)}{.4}{0}{0}{$h$};
	\roundNbox{unshaded}{(0,-.6)}{.4}{0}{0}{$g$};
	\roundNbox{unshaded}{(0,-1.8)}{.4}{0}{0}{$f$};
	\node at (-.2,1.2) {\scriptsize{$d$}};
	\node at (1.2,.8) {\scriptsize{$\cP[n_4+n_3]$}};
	\node at (-.2,0) {\scriptsize{$c$}};
	\node at (1.2,-.4) {\scriptsize{$\cP[n_3+n_2]$}};
	\node at (-.2,-1.2) {\scriptsize{$b$}};
	\node at (1.2,-1.6) {\scriptsize{$\cP[n_2+n_1]$}};
	\node at (-.2,-2.4) {\scriptsize{$a$}};
	\draw[rounded corners=5pt, very thick, unshaded] (2,-2.4) rectangle (5.6,1.2);
	\draw[thick, red] (3.5,.3) -- (2.6,-.6);
	\draw[thick, red] (3.5,-.6) -- (2.6,-.6);
	\draw[thick, red] (3.5,-1.5) -- (2.6,-.6);
	\draw[very thick] (3.8,-.6) ellipse (1.2 and 1.5);
	\draw (3.8,-2.1) -- (3.8,.9);
	\filldraw[very thick, unshaded] (3.8,.3) circle (.3cm);
	\filldraw[very thick, unshaded] (3.8,-.6) circle (.3cm);
	\filldraw[very thick, unshaded] (3.8,-1.5) circle (.3cm);
	\node at (4,.7) {\scriptsize{$n_4$}};
	\node at (4,-.15) {\scriptsize{$n_3$}};
	\node at (4,-1.05) {\scriptsize{$n_2$}};
	\node at (4,-1.9) {\scriptsize{$n_1$}};
	\node at (6.4,-.4) {\scriptsize{$\cP[n_4+n_1]$}};
\end{tikzpicture}
\displaybreak[1]
\\&=\,\,\,
\begin{tikzpicture}[baseline=.7cm, yscale=-1]
	\draw (0,1.4) -- (0,-2.6);
	\draw (.4,.6) -- (2,.6);
	\draw (.4,-.6) -- (2,-.6);
	\draw (4.4,0) -- (6,0);
	\draw (.4,-1.8) -- (6,-1.8);
	\draw (9,-.9) -- (10.6,-.9);
	\roundNbox{unshaded}{(0,.6)}{.4}{0}{0}{$f$};
	\roundNbox{unshaded}{(0,-.6)}{.4}{0}{0}{$g$};
	\roundNbox{unshaded}{(0,-1.8)}{.4}{0}{0}{$h$};
	\node at (-.2,1.2) {\scriptsize{$a$}};
	\node at (1.2,.4) {\scriptsize{$\cP[n_2+n_1]$}};
	\node at (-.2,0) {\scriptsize{$b$}};
	\node at (1.2,-.8) {\scriptsize{$\cP[n_3+n_2]$}};
	\node at (-.2,-1.2) {\scriptsize{$c$}};
	\node at (3.2,-2) {\scriptsize{$\cP[n_4+n_3]$}};
	\node at (-.2,-2.4) {\scriptsize{$d$}};
	\multiplication{(3.2,0)}{1.2}{n_3}{n_2}{n_1}
	\node at (5.2,-.2) {\scriptsize{$\cP[n_3+n_1]$}};
	\multiplication{(7.5,-.9)}{1.5}{n_4}{n_3}{n_1}
	\node at (9.8,-1.1) {\scriptsize{$\cP[n_4+n_1]$}};
\end{tikzpicture}
\displaybreak[1]
\\&=\,\,
h\circ (g\circ f)
\end{align*}
and unital:
\begin{align*}
\id \circ f\,\,
&=\,\,\,
\begin{tikzpicture}[baseline=-.6cm, yscale=-1]
	\draw (0,2.4) -- (0,-1);
	\draw (.4,1.4) -- (4,1.4);
	\draw (5.2,.7) -- (6.8,.7);
	\roundNbox{unshaded}{(0,1.4)}{.4}{0}{0}{$f$};
	\draw (1.5,0) -- (2.8,0);
	\node at (-.2,0) {\scriptsize{$b$}};
	\identityMap{(1,0)}{.5}{n_2}
	\node at (2,-.2) {\scriptsize{$\cP[2n_2]$}};
	\draw[rounded corners=5pt, dashed] (-.4,-.7) rectangle (2.5,.7);
	\node at (-.2,2) {\scriptsize{$a$}};
	\node at (1.7,1.6) {\scriptsize{$\cP[n_2+n_1]$}};
	\multiplication{(4,.7)}{1.2}{n_3}{n_2}{n_1}
	\node at (6,.5) {\scriptsize{$\cP[n_2+n_1]$}};
\end{tikzpicture}
\displaybreak[1]
\\&=\,\,\,
\begin{tikzpicture}[baseline=-.1cm]
	\draw (0,.8) -- (0,-.8);
	\draw (0,0) -- (2,0);
	\roundNbox{unshaded}{(0,0)}{.4}{0}{0}{$f$};
	\node at (-.2,.6) {\scriptsize{$b$}};
	\node at (1.2,.2) {\scriptsize{$\cP[n_2+n_1]$}};
	\node at (-.2,-.6) {\scriptsize{$a$}};
	\identityTangle{(2.8,0)}{.8}{n_1}{n_2}
	\draw (3.6,0) -- (5.2,0);
	\node at (4.4,.2) {\scriptsize{$\cP[n_2+n_1]$}};
\end{tikzpicture}
\,\,=\,\,f
\displaybreak[1]
\\&=\,\,\,
\begin{tikzpicture}[baseline=.6cm]
	\draw (0,2.4) -- (0,-1);
	\draw (.4,1.4) -- (4,1.4);
	\draw (5.2,.7) -- (6.8,.7);
	\roundNbox{unshaded}{(0,1.4)}{.4}{0}{0}{$f$};
	\draw (1.5,0) -- (2.8,0);
	\node at (-.2,0) {\scriptsize{$a$}};
	\identityMap{(1,0)}{.5}{n_1}
	\node at (2,.2) {\scriptsize{$\cP[2n_1]$}};
	\draw[rounded corners=5pt, dashed] (-.4,-.7) rectangle (2.5,.7);
	\node at (-.2,2) {\scriptsize{$b$}};
	\node at (1.7,1.6) {\scriptsize{$\cP[n_2+n_1]$}};
	\multiplication{(4,.7)}{1.2}{n_1}{n_1}{n_2}
	\node at (6,.9) {\scriptsize{$\cP[n_2+n_1]$}};
\end{tikzpicture}
\displaybreak[1]
\\&=\,\,\,f\circ \id.
\end{align*}

\subsection{The adjoint pair \texorpdfstring{$\Phi \dashv \Tr_\cC$}{Phi and Tr_C}}

Our next goal is to construct a pair of adjoint functors
\begin{equation}\label{eq: the adjunction : statement}
\Phi : \cC\,\,\raisebox{-.08cm}{$\stackrel{\textstyle\leftarrow}\to$}\,\, \cM:\Tr_\cC.
\end{equation}
The functor $\Phi$ lands in the full subcategory $\cM_0$ of $\cM$.
It is given by
\begin{equation}\vspace{.6cm}
\label{Phi M eq2}
\begin{split}
\qquad\Phi(c):=\text{``}\Phi(c)\otimes m^{\otimes 0}\text{''}\quad\qquad\text{and}\,\,&
\\
\Phi(f:a\to b) \,:= \,\,\,
\begin{tikzpicture}[baseline=-.1cm]
\useasboundingbox (-.4,-.2) rectangle (2.2,.9);
	\draw (0,-.8) -- (0,.8);
	\draw (1.4,0) -- (2.2,0);
	\node at (-.2,.6) {\scriptsize{$b$}};
	\node at (-.2,-.6) {\scriptsize{$a$}};
	\emptyMap{(1,0)}{.4}
	\node at (1.8,.2) {\scriptsize{$\cP[0]$}};
	\roundNbox{unshaded}{(0,0)}{.4}{0}{0}{$f$}
\end{tikzpicture}
\,\,\in\,\, \cM\big(\Phi(a), \Phi(b)\big) =&\,\, \cC(a, b\otimes \cP[0]).
\end{split}
\end{equation}
The functor $\Tr_\cC$ is first defined on $\cM_0$, and then formally extended to $\cM$
using the fact that $\cC$ admits direct sums and is idempotent complete.
It is given on objects by
\begin{equation}\label{eq: that's how you define Tr!}
\Tr_\cC(\text{``}\Phi(c)\otimes m^{\otimes n}\text{''}) \,:=\, c\otimes \cP[n],
\end{equation}
and on morphisms by
\begin{equation}\label{eq: That's Tr_cC(f) }
\Tr_\cC(f)\,:=\,\,\,
\begin{tikzpicture}[baseline=.6cm]
	\draw (0,2.4) -- (0,-.6);
	\draw (.4,1.4) -- (4,1.4);
	\draw (4,.7) -- (4.8,.7);
	\draw (1.2,-.6) arc (180:90:.8cm);
	\roundNbox{unshaded}{(0,1.4)}{.4}{0}{0}{$f$};
	\node at (-.2,0) {\scriptsize{$a$}};
	\node at (-.2,2.2) {\scriptsize{$b$}};
	\node at (1.2,1.6) {\scriptsize{$\cP[p+n]$}};
	\node at (4.4,.9) {\scriptsize{$\cP[p]$}};
	\node at (1,0) {\scriptsize{$\cP[n]$}};
	\multiplication{(3,.8)}{1}{0}{n}{p}
\end{tikzpicture}
\,\,\,\in\,\,\, \cC(a\otimes \cP[n], b\otimes \cP[p]).
\end{equation}
for $f\in \cM_0(\text{``}\Phi(a)\otimes m^{\otimes n}\text{''} , \text{``}\Phi(b) \otimes m^{p}\text{''})= \cC(a, b \otimes \cP[p+n])$.
To see that this defines a functor, we check:
\begin{align*}
\Tr_\cC(g)\circ \Tr_\cC(f)\,\,
&=\,\,\,
\begin{tikzpicture}[baseline=-.3cm]
	\draw (0,1.4) -- (0,-2);
	\draw (1.4,-2) arc (180:90:.6);
	\draw (.4,.6) -- (5.4,.6);
	\draw (.4,-.6) -- (5.4,-.6);
	\draw (7.4,0) -- (8.2,0);
	\roundNbox{unshaded}{(0,-.6)}{.4}{0}{0}{$f$};
	\roundNbox{unshaded}{(0,.6)}{.4}{0}{0}{$g$};
	\node at (1.2,-.4) {\scriptsize{$\cP[p+n]$}};
	\node at (4.7,-.4) {\scriptsize{$\cP[p]$}};
	\node at (1.2,.8) {\scriptsize{$\cP[q+p]$}};
	\node at (7.8,.2) {\scriptsize{$\cP[q]$}};
	\node at (1.2,-1.4) {\scriptsize{$\cP[n]$}};
	\node at (-.2,-1.4) {\scriptsize{$a$}};
	\node at (-.2,0) {\scriptsize{$b$}};
	\node at (-.2,1.2) {\scriptsize{$c$}};
	\multiplication{(3,-1)}{1}{0}{n}{p}
	\multiplication{(6.4,0)}{1}{0}{p}{q}
\end{tikzpicture}
\displaybreak[1]
\\&=\,\,\,
\begin{tikzpicture}[baseline=-.1cm]
	\draw (0,1.4) -- (0,-2.2);
	\draw (.4,.6) -- (2,.6);
	\draw (.4,-.6) -- (4.8,-.6);
	\draw (1.4,-2.2) arc (180:90:.6cm);
	\node at (1.2,-1.6) {\scriptsize{$\cP[n]$}};	
	\roundNbox{unshaded}{(0,.6)}{.4}{0}{0}{$g$};
	\roundNbox{unshaded}{(0,-.6)}{.4}{0}{0}{$f$};
	\node at (-.2,1.2) {\scriptsize{$c$}};
	\node at (1.2,.8) {\scriptsize{$\cP[q+p]$}};
	\node at (-.2,0) {\scriptsize{$b$}};
	\node at (1.2,-.4) {\scriptsize{$\cP[p+n]$}};
	\node at (-.2,-1.2) {\scriptsize{$a$}};
	\node at (4.4,-.4) {\scriptsize{$\cP[q]$}};
	\draw[rounded corners=5pt, very thick, unshaded] (2,-2.2) rectangle (4,1);
	\draw (3,-1.25) -- (3,.8);
	\draw[thick, red] (2.2,-.6) -- (2.75,-.6) ;
	\draw[thick, red] (2.2,-.6) -- (2.75,.2) ;
	\draw[thick, red] (2.2,-.6) -- (2.75,-1.4) ;
	\draw[very thick] (3,-.6) ellipse (.8cm and 1.4cm);
	\draw[unshaded, very thick] (3,-.6) circle (.25cm);
	\draw[unshaded, very thick] (3,.2) circle (.25cm);
	\draw[unshaded, very thick] (3,-1.4) circle (.25cm);
	\node at (3.15,-1) {\scriptsize{$n$}};
	\node at (3.15,-.2) {\scriptsize{$p$}};
	\node at (3.15,.6) {\scriptsize{$q$}};
%
\end{tikzpicture}
\displaybreak[1]
\\&=\,\,\,
\begin{tikzpicture}[baseline=-.1cm]
	\draw (0,1.4) -- (0,-1.8);
	\draw (.4,.6) -- (2,.6);
	\draw (.4,-.6) -- (2,-.6);
	\draw (4,0) -- (5.2,0);
	\roundNbox{unshaded}{(0,.6)}{.4}{0}{0}{$g$};
	\roundNbox{unshaded}{(0,-.6)}{.4}{0}{0}{$f$};
	\node at (-.2,1.2) {\scriptsize{$c$}};
	\node at (1.2,.8) {\scriptsize{$\cP[q+p]$}};
	\node at (-.2,0) {\scriptsize{$b$}};
	\node at (1.2,-.4) {\scriptsize{$\cP[p+n]$}};
	\node at (-.2,-1.2) {\scriptsize{$a$}};
	\multiplication{(3,0)}{1}{n}{p}{q}
	\node at (4.6,.2) {\scriptsize{$\cP[q+n]$}};
	\draw (7.2,-.6) -- (8, -.6);
	\node at (7.6,-.4) {\scriptsize{$\cP[q]$}};
	\draw (4.6,-1.8) arc (180:90:.6cm);
	\node at (4.4,-1.2) {\scriptsize{$\cP[n]$}};
	\multiplication{(6.2,-.6)}{1}{0}{n}{q}
\end{tikzpicture}
\displaybreak[1]
\\&=\,\,\,
\Tr_\cC(g\circ f)
\end{align*}
and
$$
\Tr_\cC(\id_{\text{``}\Phi(a)\otimes m^{\otimes n}\text{''}})
\,=\,\,
\begin{tikzpicture}[baseline=-.6cm]
	\draw (0,1) -- (0,-1.8);
	\draw (2,-1.8) arc (180:90:.6);
	\draw (4.6,-.6) -- (5.4,-.6);
	\draw (1.2,0) -- (2.6,0);
	\node at (-.2,0) {\scriptsize{$a$}};
	\identityMap{(.8,0)}{.4}{n\,\;}
	\node at (1.7,.2) {\scriptsize{$\cP[2n]$}};
	\draw[rounded corners=5pt, dashed] (-.4,-.6) rectangle (2.2,.6);
	\node at (1.8,-1.2) {\scriptsize{$\cP[n]$}};
	\node at (5,-.4) {\scriptsize{$\cP[n]$}};
	\multiplication{(3.6,-.6)}{1}{0}{n}{n}
\end{tikzpicture}
=
\begin{tikzpicture}[baseline=-.1cm]
	\draw (-.6,-.6) -- (-.6,.6);
	\draw (0,-.6) arc (180:90:.6cm);
	\draw (1.8,0) -- (2.6,0);
	\node at (2.2, .2) {\scriptsize{$\cP[n]$}};
	\node at (0,-.9) {\scriptsize{$\cP[n]$}};
	\node at (-.6,-.9) {\scriptsize{$a$}};
	\draw[rounded corners=5pt, very thick, unshaded] (.6,-.6) rectangle (1.8,.6);
	\draw (1.2,0) -- (1.2,.5);
	\draw[thick, red] (1.2,0) -- (.7,0) ;
	\draw[very thick] (1.2,0) circle (.5cm);
	\draw[unshaded, very thick] (1.2,0) circle (.15cm);
	\node at (1.35,.25) {\scriptsize{$n$}};
\end{tikzpicture}
\,=\,
\id_{a\otimes \cP[n]}.
$$

\begin{lem}\label{lem: exhibits Tr as the right adjoint}
The identity map 
$$
\cM(\Phi(a), \text{\rm``}\Phi(c)\otimes m^{\otimes n}\text{\rm''}) 
 =
 \cC(a, c\otimes \cP[n])
\,\longrightarrow\,
\cC(a, c \otimes \cP[n])
= 
\cC(a, \Tr_\cC(\text{\rm``}\Phi(c)\otimes m^{\otimes n}\text{\rm''})) 
$$
is natural in $a\in \cC$ and in $\text{\rm``}\Phi(c)\otimes m^{\otimes n}\text{\rm''}\in \cM$.
It exhibits $\Tr_\cC:\cM_0\to\cC$ as the right adjoint of $\Phi:\cC\to \cM_0$.
\end{lem}
\begin{proof}
First, given $f\in \cC(a, b)$, we must show that the following diagram commutes:
\[
\xymatrix{
\cM_0(\Phi(b), \text{``}\Phi(c)\otimes m^{\otimes n}\text{''}) = \cC(b, c\otimes \cP[n])
\ar@<-8ex>[d]^{-\,\circ\,\Phi(f)}
\ar[rr]^(.65){\id}
&&
\cC(b, c\otimes \cP[n])
\ar[d]^{-\;\!\circ\;\! f}
\\
\cM_0(\Phi(a), \text{``}\Phi(c)\otimes m^{\otimes n}\text{''}) = \cC(a, c\otimes \cP[n])
\ar[rr]^(.65){\id}
&&
\cC(a, c\otimes \cP[n])
}
\]
Starting with $g$ in the upper left corner, we compute:
$$
g\circ \Phi(f) 
\,=\,\,
\begin{tikzpicture}[baseline=-.1cm]
	\draw (0,1.4) -- (0,-1.4);
	\draw (.4,.6) -- (2,.6);
	\draw (1,-.6) -- (2,-.6);
	\draw (4,0) -- (4.6,0);
	\roundNbox{unshaded}{(0,.6)}{.4}{0}{0}{$g$};
	\roundNbox{unshaded}{(0,-.6)}{.4}{0}{0}{$f$};
	\node at (-.2,1.2) {\scriptsize{$c$}};
	\node at (-.2,0) {\scriptsize{$b$}};
	\node at (-.2,-1.2) {\scriptsize{$a$}};
	\emptyMap{(1,-.6)}{.4}
	\multiplication{(3,0)}{1}{0}{0}{n}
\end{tikzpicture}
\,\,=\,\,
\begin{tikzpicture}[baseline=-.1cm]
	\draw (-.4,-.8) -- (-.4,.8);
	\draw (0,0) -- (2.2,0);
	\node at (-.6,.6) {\scriptsize{$c$}};
	\node at (-.6,-.6) {\scriptsize{$a$}};
	\roundNbox{unshaded}{(-.4,0)}{.4}{.2}{.2}{$g\circ f$}
	\draw[rounded corners=5pt, very thick, unshaded] (.6,-.6) rectangle (1.8,.6);
	\draw (1.2,0) -- (1.2,.5);
	\draw[thick, red] (1.2,0) -- (.7,0) ;
	\draw[very thick] (1.2,0) circle (.5cm);
	\draw[unshaded, very thick] (1.2,0) circle (.15cm);
	\node at (1.35,.25) {\scriptsize{$n$}};
\end{tikzpicture}
\,\,=\,
g\circ f. 
$$

Second, given $g\in \cM_0(\text{``}\Phi(c)\otimes m^{\otimes n}\text{''} , \text{``}\Phi(d)\otimes m^{\otimes p}\text{''})$, we must show that the following diagram commutes:
\[
\xymatrix{
\cM_0(\Phi(a), \text{``}\Phi(c)\otimes m^{\otimes n}\text{''}) :=\cC(a, c\otimes \cP[n])
\ar@<-8ex>[d]^{g\;\!\circ\;\! -}
\ar[rr]^(.65){\id}
&&
\cC(a, c\otimes \cP[n])
\ar[d]^{\Tr_\cC(g)\,\circ\, -}
\\
\cM_0(\Phi(a),  \text{``}\Phi(d)\otimes m^{\otimes p}\text{''}) := \cC(a, d\otimes \cP[p])
\ar[rr]^(.65){\id}
&&
\cC(a, d\otimes \cP[p])
}
\]
Starting with $f$ in the upper left corner, we compute:
$$
g\circ f 
\,=\,\,
\begin{tikzpicture}[baseline=-.1cm]
	\draw (0,1.4) -- (0,-1.4);
	\draw (.4,.6) -- (2,.6);
	\draw (.4,-.6) -- (2,-.6);
	\draw (4,0) -- (4.8,0);
	\roundNbox{unshaded}{(0,.6)}{.4}{0}{0}{$g$};
	\roundNbox{unshaded}{(0,-.6)}{.4}{0}{0}{$f$};
	\node at (-.2,1.2) {\scriptsize{$d$}};
	\node at (1.1,.8) {\scriptsize{$\cP[p+n]$}};
	\node at (-.2,0) {\scriptsize{$c$}};
	\node at (1,-.4) {\scriptsize{$\cP[n]$}};
	\node at (-.2,-1.2) {\scriptsize{$a$}};
	\multiplication{(3,0)}{1}{0}{n}{p}
	\node at (4.4,.2) {\scriptsize{$\cP[p]$}};
\end{tikzpicture}
\,\,\,=\,\,
\begin{tikzpicture}[baseline=-.7cm]
	\draw (0,1.4) -- (0,-2.6);
	\draw (.4,.6) -- (2,.6);
	\draw (1.4,-1.2) arc (180:90:.6cm);
	\draw (.4,-1.8) -- (.8,-1.8) arc (-90:0:.6cm);
	\draw (4,0) -- (4.8,0);
	\roundNbox{unshaded}{(0,.6)}{.4}{0}{0}{$g$};
	\roundNbox{unshaded}{(0,-1.8)}{.4}{0}{0}{$f$};
	\node at (-.2,1.2) {\scriptsize{$d$}};
	\node at (1.1,.8) {\scriptsize{$\cP[p+n]$}};
	\node at (-.2,-.6) {\scriptsize{$c$}};
	\node at (1.2,-.6) {\scriptsize{$\cP[n]$}};
	\node at (-.2,-2.4) {\scriptsize{$a$}};
	\multiplication{(3,0)}{1}{0}{n}{p}
	\node at (4.4,.2) {\scriptsize{$\cP[p]$}};
	\draw[dashed] (-.4,-1.2) -- (4.8,-1.2);
\end{tikzpicture}
\,=\,
\Tr_\cC(g) \circ f.
$$
\end{proof}

By the above lemma, we get an adjunction
\begin{equation}\label{eq: the adjunction : pre}
\Phi : \cC\,\,\raisebox{-.08cm}{$\stackrel{\textstyle\leftarrow}\to$}\,\, \cM_0:\Tr_\cC.
\end{equation}
The adjunction \eqref{eq: the adjunction : statement} follows formally.
It is obtained from \eqref{eq: the adjunction : pre} by applying the 2-functor
\begin{equation}\label{eq: add direct sums and idempotent complete}
\{\text{linear categories}\}\,\to\, \{\text{additive idempodent complete linear categories}\},
\end{equation}
which adjoins finite direct sums and idempotent completes.

\subsection{The tensor structure on $\cM$}

We now endow $\cM$ with a monoidal structure.
To do so, we first define such a structure on $\cM_0$ and then extend it formally to $\cM$ by applying of the 2-functor \eqref{eq: add direct sums and idempotent complete}.
At the level of objects, the monoidal structure is given by
\begin{equation}\label{Phi M eq3}
\text{``}\Phi(c)\otimes m^{\otimes n_1}\text{''} \otimes \text{``}\Phi(d)\otimes m^{\otimes n_2}\text{''} := \text{``}\Phi(c\otimes d)\otimes m^{\otimes n_2+n_1}\text{''}\qquad\,\,\, 1_{\cM_0} := \Phi(1_\cC).
\end{equation}
The tensor product of morphisms
$f_i\in \cM_0(\text{``}\Phi(a_i)\otimes m^{\otimes n_i}\text{''},\text{``}\Phi(b_i)\otimes m^{\otimes p_i}\text{''})$, $i=1,2$,
is given by:
\begin{equation}
\label{eq:TensorProductInM}
f_1\otimes f_2 \,:=\,\,\,
\begin{tikzpicture}[baseline=-.6cm]
	\draw (0,0) -- (2,0);
	\draw (-1,-1) -- (2,-1);
	\draw (4.2,-.5) -- (6.8,-.5);
	\draw (-1,.8) -- (-1,-1.8);
	\draw[super thick, white] (0,.8) -- (0,-1.8);
	\draw (0,.8) -- (0,-1.8);
	\roundNbox{unshaded}{(-1,-1)}{.4}{0}{0}{$f_1$}
	\roundNbox{unshaded}{(0,0)}{.4}{0}{0}{$f_2$}
	\tensor{(3.2,-.5)}{1.2}{n_1}{p_1}{n_2}{p_2}
	\node at (1.2,.2) {\scriptsize{$\cP[p_2{+}n_2]$}};
	\node at (1.2,-.8) {\scriptsize{$\cP[p_1{+}n_1]$}};
	\node at (5.6,-.3) {\scriptsize{$\cP[\sum p_i {+} \sum n_i]$}};
	\node at (-1.2,-1.6) {\scriptsize{$a_1$}};
	\node at (-.2,-1.6) {\scriptsize{$a_2$}};
	\node at (-1.2,.6) {\scriptsize{$b_1$}};
	\node at (-.2,.6) {\scriptsize{$b_2$}};
\end{tikzpicture}
\end{equation}
Note that tensoring a morphism $f\in \cM_0(\text{``}\Phi(a)\otimes m^{\otimes n}\text{''} , \text{``}\Phi(b)\otimes m^{\otimes p}\text{''})$
with an identity morphism $\id_{\text{``}\Phi(c)\otimes m^{\otimes q}\text{''}}$ is computed as follows:
\begin{align*}
f\otimes \id &=
\begin{tikzpicture}[baseline=.4cm]
	\draw (2,-.8) -- (2,1.8);
	\draw (2.4,0) -- (5,0);
	\draw (4,1) -- (5,1);
	\draw[super thick, white] (2.7,-.8) -- (2.7,1.8);
	\draw (2.7,-.8) -- (2.7,1.8);
	\node at (1.8,-.6) {\scriptsize{$a$}};
	\node at (1.8,.6) {\scriptsize{$b$}};
	\node at (2.5,1) {\scriptsize{$c$}};
	\roundNbox{unshaded}{(2,0)}{.4}{0}{0}{$f$}
	\identityMap{(3.5,1)}{.5}{q\,\,}
	\node at (3.5,.2) {\scriptsize{$\cP[p+n]$}};
	\node at (4.5,1.2) {\scriptsize{$\cP[2q]$}};
	\tensor{(6.2,.5)}{1.2}{n}{p}{q}{q}
	\draw (7.4,.5) -- (9.4,.5);
	\node at (8.4,.7) {\scriptsize{$\cP[p+2q+n]$}};
\end{tikzpicture}
=
\begin{tikzpicture}[baseline=-.1cm]
	\draw (0,-1.2) -- (0,1.2);
	\draw (.4,0) -- (6,0);
	\draw[super thick, white] (.6,-1.2) -- (.6,1.2);
	\draw (.6,-1.2) -- (.6,1.2);
	\node at (-.2,.6) {\scriptsize{$b$}};
	\node at (-.2,-.6) {\scriptsize{$a$}};
	\node at (.4,.6) {\scriptsize{$c$}};
	\tensorRightId{(3,0)}{1}{n}{p}{q}
	\node at (1.3,.2) {\scriptsize{$\cP[p+n]$}};
	\node at (5,.2) {\scriptsize{$\cP[p+2q+n]$}};
	\roundNbox{unshaded}{(0,0)}{.4}{0}{0}{$f$}
\end{tikzpicture}
\\ 
\id \otimes f &=
\begin{tikzpicture}[baseline=.4cm]
	\draw (3.8,1) -- (5,1);
	\draw (2.4,0) -- (5,0);
	\draw[super thick, white] (3.4,-.8) -- (3.4,1.8);
	\draw (3.4,-.8) -- (3.4,1.8);
	\draw (1.5,-.8) -- (1.5,1.8);
	\node at (3.2,.4) {\scriptsize{$a$}};
	\node at (3.2,1.6) {\scriptsize{$b$}};
	\node at (1.7,1) {\scriptsize{$c$}};
	\roundNbox{unshaded}{(3.4,1)}{.4}{0}{0}{$f$}
	\identityMap{(2.3,0)}{.5}{q\,\,}
	\node at (4.4,.2) {\scriptsize{$\cP[2q]$}};
	\node at (4.4,1.2) {\scriptsize{$\cP[p+n]$}};
	\tensor{(6.2,.5)}{1.2}{q}{q}{n}{p}
	\draw (7.4,.5) -- (9.4,.5);
	\node at (8.4,.7) {\scriptsize{$\cP[p+2q+n]$}};
\end{tikzpicture}
=\!\!
\begin{tikzpicture}[baseline=-.1cm]
	\draw (-.6,-1.2) -- (-.6,1.2);
	\draw (0,-1.2) -- (0,1.2);
	\draw (.4,0) -- (5.6,0);
	\node at (-.2,.6) {\scriptsize{$b$}};
	\node at (-.2,-.6) {\scriptsize{$a$}};
	\node at (-.8,0) {\scriptsize{$c$}};
	\tensorLeftId{(2.6,0)}{1}{q}{n}{p}
	\node at (1,.2) {\scriptsize{$\cP[p+n]$}};
	\node at (4.6,.2) {\scriptsize{$\cP[p+2q+n]$}};
	\roundNbox{unshaded}{(0,0)}{.4}{0}{0}{$f$}
\end{tikzpicture}
\end{align*}

We now show that the above operation defines a functor $\otimes:\cM_0\times\cM_0\to \cM_0$.
Given morphisms $f_i\in \cM_0(\text{``}\Phi(a_i)\otimes m^{\otimes n_i}\text{''},\text{``}\Phi(b_i)\otimes m^{\otimes p_i}\text{''})$ and
$g_i\in \cM_0(\text{``}\Phi(b_i)\otimes m^{\otimes p_i}\text{''},\text{``}\Phi(b_i)\otimes m^{\otimes q_i}\text{''})$, for $i=1,2$,
we compute:
\begin{align*}
(g_1\circ f_1)&\otimes (g_2\circ f_2)
\displaybreak[1]
\\&=\,\,\,
\begin{tikzpicture}[baseline=-1.85cm]
	\draw (0,0) -- (2,0);
	\draw (0,-1) -- (2,-1);
	\draw (-1,-2.5) -- (2,-2.5);
	\draw (-1,-3.5) -- (2,-3.5);
	\draw (4.2,-.5) -- (6,-.5);
	\draw (4.2,-3) -- (6,-3);
	\draw (-1,.8) -- (-1,-4.3);
	\draw[super thick, white] (0,.8) -- (0,-4.3);
	\draw (0,.8) -- (0,-4.3);
	\roundNbox{unshaded}{(0,-1)}{.4}{0}{0}{$f_2$}
	\roundNbox{unshaded}{(0,0)}{.4}{0}{0}{$g_2$}
	\multiplication{(3.2,-.5)}{1.2}{n_2}{p_2}{q_2}
	\node at (1.2,.2) {\scriptsize{$\cP[q_2{+}p_2]$}};
	\node at (1.2,-.8) {\scriptsize{$\cP[p_2{+}n_2]$}};
	\node at (5.2,-.3) {\scriptsize{$\cP[q_2{+}n_2]$}};
	\roundNbox{unshaded}{(-1,-2.5)}{.4}{0}{0}{$g_1$}
	\roundNbox{unshaded}{(-1,-3.5)}{.4}{0}{0}{$f_1$}
	\multiplication{(3.2,-3)}{1.2}{n_1}{p_1}{q_1}
	\node at (1.2,-2.3) {\scriptsize{$\cP[q_1{+}p_1]$}};
	\node at (1.2,-3.3) {\scriptsize{$\cP[p_1{+}n_1]$}};
	\node at (5.2,-2.8) {\scriptsize{$\cP[q_1{+}n_1]$}};
	\tensor{(7.5,-1.75)}{1.5}{n_1}{q_1}{n_2}{q_2}
	\draw (9,-1.75) -- (11.4,-1.75);
	\node at (10.2,-1.55) {\scriptsize{$\cP[\sum q_i {+} \sum n_i]$}};
\end{tikzpicture}
\displaybreak[1]
\\&=\,\,\,
\begin{tikzpicture}[baseline=-1.6cm]
	\draw (0,0) -- (2,0);
	\draw (0,-1) -- (2,-1);
	\draw (-1,-2) -- (2,-2);
	\draw (-1,-3) -- (2,-3);
	\draw (-1,.8) -- (-1,-3.8);
	\draw[super thick, white] (0,.8) -- (0,-3.8);
	\draw (0,.8) -- (0,-3.8);
	\roundNbox{unshaded}{(0,0)}{.4}{0}{0}{$g_2$}
	\roundNbox{unshaded}{(0,-1)}{.4}{0}{0}{$f_2$}
	\node at (1.2,.2) {\scriptsize{$\cP[q_2{+}p_2]$}};
	\node at (1.2,-.8) {\scriptsize{$\cP[p_2{+}n_2]$}};
	\roundNbox{unshaded}{(-1,-2)}{.4}{0}{0}{$g_1$}
	\roundNbox{unshaded}{(-1,-3)}{.4}{0}{0}{$f_1$}
	\node at (1.2,-1.8) {\scriptsize{$\cP[q_1{+}p_1]$}};
	\node at (1.2,-2.8) {\scriptsize{$\cP[p_1{+}n_1]$}};
	\draw[rounded corners=5pt, very thick, unshaded] (2,-3.5) rectangle (6,.5);
	\draw[thick, red] (2.95,-1) -- (2.2,-1.5);
	\draw[thick, red] (2.95,-2) -- (2.2,-1.5);
	\draw[thick, red] (4.45,-2) .. controls ++(180:.4cm) and ++(0:.6cm) .. (3.25,-.5) .. controls ++(180:.4cm) and ++(45:.2cm) .. (2.2,-1.5);
	\draw[thick, red] (4.45,-1) .. controls ++(180:.4cm) and ++(0:.6cm) .. (3.25,-.2) .. controls ++(180:.4cm) and ++(45:.2cm) .. (2.2,-1.5);
	\draw[very thick] (4,-1.5) circle (1.8cm);
	\draw (3.25,-3.15) -- (3.25,.15);
	\draw (4.75,-3.15) -- (4.75,.15);
	\filldraw[very thick, unshaded] (4.75,-2) circle (.3cm);
	\filldraw[very thick, unshaded] (4.75,-1) circle (.3cm);
	\filldraw[very thick, unshaded] (3.25,-2) circle (.3cm);
	\filldraw[very thick, unshaded] (3.25,-1) circle (.3cm);
	\node at (3.45,-2.6) {\scriptsize{$n_1$}};
	\node at (4.95,-2.6) {\scriptsize{$n_2$}};
	\node at (3.45,-1.5) {\scriptsize{$p_1$}};
	\node at (4.95,-1.5) {\scriptsize{$p_2$}};
	\node at (3.45,-.4) {\scriptsize{$q_1$}};
	\node at (4.95,-.4) {\scriptsize{$q_2$}};
	\draw (6,-1.5) -- (8.4,-1.5);
	\node at (7.2,-1.3) {\scriptsize{$\cP[\sum q_i {+} \sum n_i]$}};
\end{tikzpicture}
\displaybreak[1]
\\&=\,\,\,
\begin{tikzpicture}[baseline=-1.6cm]
	\draw (0,0) -- (2,0);
	\draw (0,-2) -- (2,-2);
	\draw (-1,-1) -- (2,-1);
	\draw (-1,-3) -- (2,-3);
	\draw (-1,.8) -- (-1,-3.8);
	\draw[super thick, white] (0,.8) -- (0,-3.8);
	\draw (0,.8) -- (0,-3.8);
	\roundNbox{unshaded}{(-1,-1)}{.4}{0}{0}{$g_1$}
	\roundNbox{unshaded}{(0,0)}{.4}{0}{0}{$g_2$}
	\node at (1.2,.2) {\scriptsize{$\cP[q_2{+}p_2]$}};
	\node at (1.2,-.8) {\scriptsize{$\cP[q_1{+}p_1]$}};
	\roundNbox{unshaded}{(-1,-3)}{.4}{0}{0}{$f_1$}
	\roundNbox{unshaded}{(0,-2)}{.4}{0}{0}{$f_2$}
	\node at (1.2,-1.8) {\scriptsize{$\cP[p_2{+}n_2]$}};
	\node at (1.2,-2.8) {\scriptsize{$\cP[p_1{+}n_1]$}};
	\draw[rounded corners=5pt, very thick, unshaded] (2,-3.5) rectangle (6,.5);
	\draw[thick, red] (3.2,-.75) -- (2.2,-1.5);
	\draw[thick, red] (3.2,-2.25) -- (2.2,-1.5);
	\draw[thick, red] (4.2,-2.25) .. controls ++(180:.3cm) and ++(0:1.5cm) .. (2.2,-1.5);
	\draw[thick, red] (4.2,-.75) .. controls ++(180:.4cm) and ++(0:.4cm) .. (3.5,-.2) .. controls ++(180:.4cm) and ++(45:.2cm) .. (2.2,-1.5);
	\draw[very thick] (4,-1.5) circle (1.8cm);
	\draw (3.5,-3.25) -- (3.5,.25);
	\draw (4.5,-3.25) -- (4.5,.25);
	\filldraw[very thick, unshaded] (4.5,-2.25) circle (.3cm);
	\filldraw[very thick, unshaded] (4.5,-.75) circle (.3cm);
	\filldraw[very thick, unshaded] (3.5,-2.25) circle (.3cm);
	\filldraw[very thick, unshaded] (3.5,-.75) circle (.3cm);
	\node at (3.7,-2.8) {\scriptsize{$n_1$}};
	\node at (4.7,-2.8) {\scriptsize{$n_2$}};
	\node at (3.7,-1.5) {\scriptsize{$p_1$}};
	\node at (4.7,-1.5) {\scriptsize{$p_2$}};
	\node at (3.7,-.2) {\scriptsize{$q_1$}};
	\node at (4.7,-.2) {\scriptsize{$q_2$}};
	\draw (6,-1.5) -- (8.4,-1.5);
	\node at (7.2,-1.3) {\scriptsize{$\cP[\sum q_i {+} \sum n_i]$}};
\end{tikzpicture}
\displaybreak[1]
\\&=\,\,\,
\begin{tikzpicture}[baseline=-1.85cm]
	\draw (0,0) -- (2,0);
	\draw (0,-2.5) -- (2,-2.5);
	\draw (-1,-1) -- (2,-1);
	\draw (-1,-3.5) -- (2,-3.5);
	\draw (4.2,-.5) -- (6.8,-.5);
	\draw (4.2,-3) -- (6.8,-3);
	\draw (-1,.8) -- (-1,-4.3);
	\draw[super thick, white] (0,.8) -- (0,-4.3);
	\draw (0,.8) -- (0,-4.3);
	\roundNbox{unshaded}{(-1,-1)}{.4}{0}{0}{$g_1$}
	\roundNbox{unshaded}{(0,0)}{.4}{0}{0}{$g_2$}
	\tensor{(3.2,-.5)}{1.2}{p_1}{q_1}{p_2}{q_2}
	\node at (1.2,.2) {\scriptsize{$\cP[q_2{+}p_2]$}};
	\node at (1.2,-.8) {\scriptsize{$\cP[q_1{+}p_1]$}};
	\node at (5.6,-.3) {\scriptsize{$\cP[\sum q_i{+}\sum p_i]$}};
	\roundNbox{unshaded}{(-1,-3.5)}{.4}{0}{0}{$f_1$}
	\roundNbox{unshaded}{(0,-2.5)}{.4}{0}{0}{$f_2$}
	\tensor{(3.2,-3)}{1.2}{n_1}{p_1}{n_2}{q_2}
	\node at (1.2,-2.3) {\scriptsize{$\cP[p_2{+}n_2]$}};
	\node at (1.2,-3.3) {\scriptsize{$\cP[p_1{+}n_1]$}};
	\node at (5.6,-2.8) {\scriptsize{$\cP[\sum p_i {+} \sum n_i]$}};
	\multiplication{(8.3,-1.75)}{1.5}{q_i}{p_i}{n_i}
	\draw (9.8,-1.75) -- (12.2,-1.75);
	\node at (11,-1.55) {\scriptsize{$\cP[\sum q_i {+} \sum n_i]$}};
\end{tikzpicture}
\displaybreak[1]
\\&=\,\,(g_1\otimes g_2)\circ (f_1\otimes f_2)
\end{align*}
where we used the anchor dependence axiom (Definition \ref{def: anchored planar algebra}) in the third equality.
Finally, we verify that $\id_{\text{``}\Phi(a)\otimes m^{\otimes n}\text{''}}\otimes \id_{\text{``}\Phi(b)\otimes m^{\otimes p}\text{''}}=\id_{\text{``}\Phi(a)\otimes m^{\otimes n}\text{''} \otimes \text{``}\Phi(b)\otimes m^{\otimes p}\text{''}}$ holds:
$$
\begin{tikzpicture}[baseline=.4cm]
	\draw (.2,-.8) -- (.2,1.8);
	\draw (1.5,0) -- (5,0);
	\draw (4,1) -- (5,1);
	\draw[super thick, white] (2.7,-.8) -- (2.7,1.8);
	\draw (2.7,-.8) -- (2.7,1.8);
	\node at (0,0) {\scriptsize{$a$}};
	\node at (2.5,1) {\scriptsize{$b$}};
	\identityMap{(1,0)}{.5}{n\,\,}
	\identityMap{(3.5,1)}{.5}{p\,\,}
	\node at (2,.2) {\scriptsize{$\cP[2n]$}};
	\node at (4.5,1.2) {\scriptsize{$\cP[2p]$}};
	\tensor{(6.2,.5)}{1.2}{n}{n}{p}{p}
	\draw (7.4,.5) -- (9.4,.5);
	\node at (8.4,.7) {\scriptsize{$\cP[n+2p+n]$}};
\end{tikzpicture}
=
\begin{tikzpicture}[baseline=-.1cm]
	\draw (.2,-1.2) -- (.2,1.2);
	\node at (-.2,0) {\scriptsize{$a\otimes b$}};
	\draw (2.1,0) -- (3.7,0);
	\identityMap{(1.3,0)}{.8}{\,\,\,\,\,n{+}p}
	\node at (2.9,.2) {\scriptsize{$\cP[2(n+p)]$}};
\end{tikzpicture}
$$

The associator and unitor isomorphisms of $\cM_0$ are inherited from those of $\cC$.
The asociator
$\alpha :
\big(\text{``}\Phi(a)\otimes m^{\otimes n}\text{''}\otimes \text{``}\Phi(b)\otimes m^{\otimes p}\text{''}\big) \otimes \text{``}\Phi(c)\otimes m^{\otimes q}\text{''}
\to 
\text{``}\Phi(a)\otimes m^{\otimes n}\text{''}\otimes \big(\text{``}\Phi(b)\otimes m^{\otimes p}\text{''} \otimes \text{``}\Phi(c)\otimes m^{\otimes q}\text{''}\big)
$
is given by
$$
\begin{tikzpicture}[baseline=-.1cm]
	\draw (0,-.8) -- (0,.8);
	\draw (1.8,0) -- (2.8,0);
	\node at (-.8,.6) {\scriptsize{$a\otimes (b\otimes c)$}};
	\node at (-.8,-.6) {\scriptsize{$(a\otimes b)\otimes c$}};
	\identityMap{(1.3,0)}{.5}{r\,\,}
	\node at (2.3,.2) {\scriptsize{$\cP[2r]$}};
	\roundNbox{unshaded}{(0,0)}{.4}{.1}{.1}{$\alpha_{a,b,c}$}
\end{tikzpicture}\qquad\,\,\,\, r=n+p+q,
$$
and the unitors
$\lambda : \Phi(1_\cC)\otimes \text{``}\Phi(a)\otimes m^{\otimes n}\text{''} \to  \text{``}\Phi(a)\otimes m^{\otimes n}\text{''}$
and
$\rho : \text{``}\Phi(a)\otimes m^{\otimes n}\text{''}\otimes \Phi(1_\cC) \to  \text{``}\Phi(a)\otimes m^{\otimes n}\text{''}$
are given by
\[
\begin{tikzpicture}[baseline=-.1cm]
	\draw (0,-.8) -- (0,.8);
	\draw (1.7,0) -- (2.7,0);
	\node at (-.2,.6) {\scriptsize{$a$}};
	\node at (-.5,-.6) {\scriptsize{$1_\cC\otimes a$}};
	\identityMap{(1.2,0)}{.5}{n\,\,}
	\node at (2.2,.2) {\scriptsize{$\cP[2n]$}};
	\roundNbox{unshaded}{(0,0)}{.4}{0}{0}{$\lambda_a$}
\end{tikzpicture}
\qquad\text{and}\qquad
\begin{tikzpicture}[baseline=-.1cm]
	\draw (0,-.8) -- (0,.8);
	\draw (1.7,0) -- (2.7,0);
	\node at (-.2,.6) {\scriptsize{$a$}};
	\node at (-.5,-.6) {\scriptsize{$a\otimes 1_\cC$}};
	\identityMap{(1.2,0)}{.5}{n\,\,}
	\node at (2.2,.2) {\scriptsize{$\cP[2n]$}};
	\roundNbox{unshaded}{(0,0)}{.4}{0}{0}{$\rho_a$}
\end{tikzpicture}\,.
\]
The associator is natural in its three inputs:
\begin{align*}
\alpha \circ \big((f_1\otimes &f_2)\otimes f_3\big)
\displaybreak[1]
\\&=\,\,
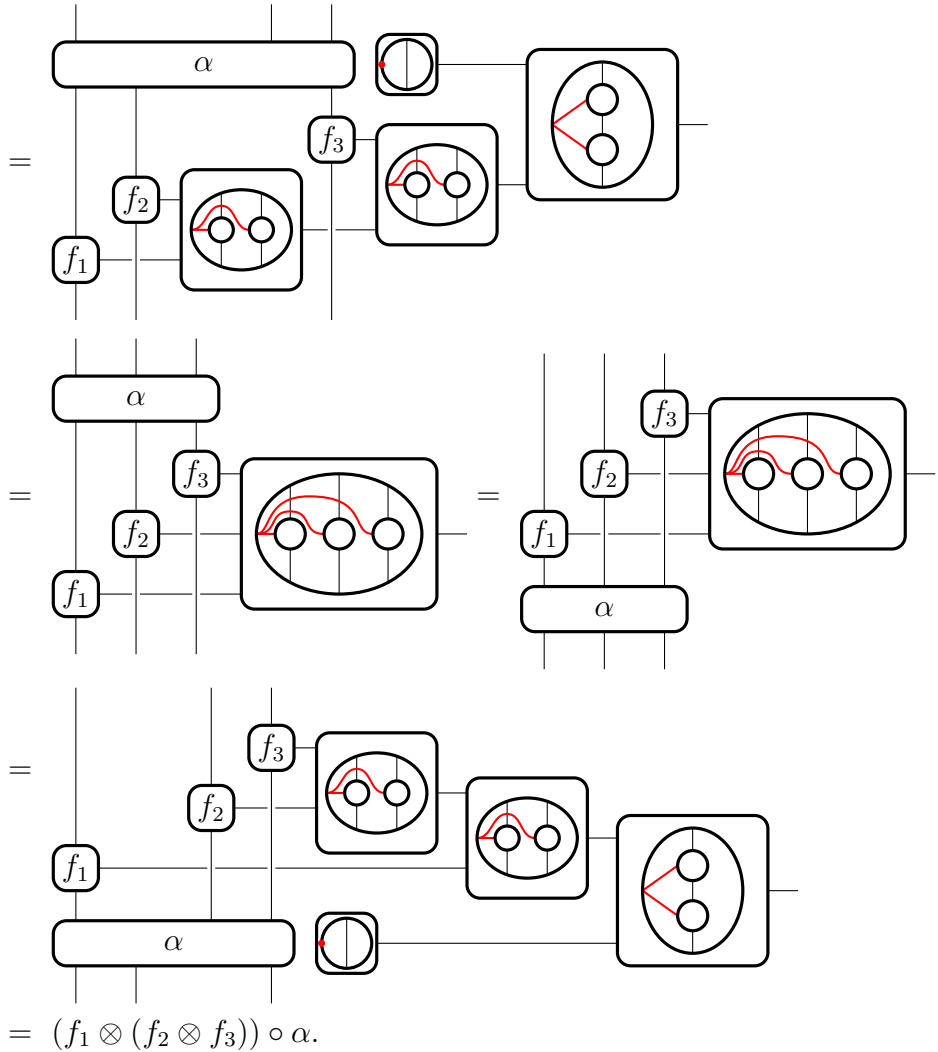
\begin{tikzpicture}[baseline = .4cm]
	\draw (-.8,-.8) -- (1.4,-.8);
	\draw (0,0) -- (1.4,0);
	\draw (2.6,.8) -- (4,.8);
	\draw (4,.2) -- (5.2,.2);
	\draw (1.4,-.4) -- (4,-.4);
	\draw (3.6,1.8) -- (5.6,1.8);
	\draw (7.2,1) -- (7.6,1);
	\draw (-.8,2.6) -- (-.8,-1.6);
	\draw (1.8,2.6) -- (1.8,1.8);
	\draw[super thick, white] (0,1.8) -- (0,-1.6);
	\draw (0,1.8) -- (0,-1.6);
	\draw[super thick, white] (2.6,2.6) -- (2.6,-1.6);
	\draw (2.6,2.6) -- (2.6,-1.6);
	\roundNbox{unshaded}{(-.8,-.8)}{.3}{0}{0}{$f_1$}
	\roundNbox{unshaded}{(0,0)}{.3}{0}{0}{$f_2$}
	\roundNbox{unshaded}{(2.6,.8)}{.3}{0}{0}{$f_3$}
	\roundNbox{unshaded}{(0,1.8)}{.3}{.8}{2.6}{$\alpha$}
	\tensor{(1.4,-.4)}{.8}{}{}{}{}
	\tensor{(4,.2)}{.8}{}{}{}{}
	\identityMap{(3.6,1.8)}{.4}{}
	\multiplication{(6.2,1)}{1}{}{}{}
\end{tikzpicture}
\displaybreak[1]
\\&=\,\,
\begin{tikzpicture}[baseline = .4cm]
	\draw (-.8,-.8) -- (1.4,-.8);
	\draw (0,0) -- (1.4,0);
	\draw (.8,.8) -- (1.4,.8);
	\draw (4,0) -- (4.4,0);
	\draw (-.8,2.6) -- (-.8,-1.6);
	\draw[super thick, white] (0,2.6) -- (0,-1.6);
	\draw (0,2.6) -- (0,-1.6);
	\draw[super thick, white] (.8,2.6) -- (.8,-1.6);
	\draw (.8,2.6) -- (.8,-1.6);
	\roundNbox{unshaded}{(-.8,-.8)}{.3}{0}{0}{$f_1$}
	\roundNbox{unshaded}{(0,0)}{.3}{0}{0}{$f_2$}
	\roundNbox{unshaded}{(.8,.8)}{.3}{0}{0}{$f_3$}
	\roundNbox{unshaded}{(0,1.8)}{.3}{.8}{.8}{$\alpha$}
	\coordinate (aa) at (2.7,0);
	\coordinate (bb) at (1.6,0);
	\draw[rounded corners=5pt, very thick, unshaded] (1.4,-1) rectangle (4,1);
	\draw (2.05,-.65) -- (2.05,.65);
	\draw (2.7,-.8) -- (2.7,.8);
	\draw (3.35,-.65) -- (3.35,.65);
	\draw[thick, red] (bb) -- (1.9,0);
	\draw[thick, red] (bb) .. controls ++(0:.3cm) and ++(180:.3cm) .. (2.05,.3) .. controls ++(0:.3cm) and ++(180:.3cm) .. (2.5,0);
	\draw[thick, red] (bb) .. controls ++(0:.2cm) and ++(180:.6cm) .. (2.3,.5) .. controls ++(0:.8cm) and ++(180:.3cm) .. (3.15,0);
	\draw[very thick] (aa) ellipse ( {1.1} and {.8});
	\filldraw[very thick, unshaded] (2.05,0) circle (.2cm);
	\filldraw[very thick, unshaded] (aa) circle (.2cm);
	\filldraw[very thick, unshaded] (3.35,0) circle (.2cm);
\end{tikzpicture}
=\,\,
\begin{tikzpicture}[baseline = -.4cm]
	\draw (-.8,-.8) -- (1.4,-.8);
	\draw (0,0) -- (1.4,0);
	\draw (.8,.8) -- (1.4,.8);
	\draw (4,0) -- (4.4,0);
	\draw (-.8,-2.6) -- (-.8,1.6);
	\draw[super thick, white] (0,-2.6) -- (0,1.6);
	\draw (0,-2.6) -- (0,1.6);
	\draw[super thick, white] (.8,-2.6) -- (.8,1.6);
	\draw (.8,-2.6) -- (.8,1.6);
	\roundNbox{unshaded}{(-.8,-.8)}{.3}{0}{0}{$f_1$}
	\roundNbox{unshaded}{(0,0)}{.3}{0}{0}{$f_2$}
	\roundNbox{unshaded}{(.8,.8)}{.3}{0}{0}{$f_3$}
	\roundNbox{unshaded}{(0,-1.8)}{.3}{.8}{.8}{$\alpha$}
	\coordinate (aa) at (2.7,0);
	\coordinate (bb) at (1.6,0);
	\draw[rounded corners=5pt, very thick, unshaded] (1.4,-1) rectangle (4,1);
	\draw (2.05,-.65) -- (2.05,.65);
	\draw (2.7,-.8) -- (2.7,.8);
	\draw (3.35,-.65) -- (3.35,.65);
	\draw[thick, red] (bb) -- (1.9,0);
	\draw[thick, red] (bb) .. controls ++(0:.3cm) and ++(180:.3cm) .. (2.05,.3) .. controls ++(0:.3cm) and ++(180:.3cm) .. (2.5,0);
	\draw[thick, red] (bb) .. controls ++(0:.2cm) and ++(180:.6cm) .. (2.3,.5) .. controls ++(0:.8cm) and ++(180:.3cm) .. (3.15,0);
	\draw[very thick] (aa) ellipse ( {1.1} and {.8});
	\filldraw[very thick, unshaded] (2.05,0) circle (.2cm);
	\filldraw[very thick, unshaded] (aa) circle (.2cm);
	\filldraw[very thick, unshaded] (3.35,0) circle (.2cm);
\end{tikzpicture}
\displaybreak[1]
\\&=\,\,
\begin{tikzpicture}[baseline = .4cm]
	\draw (-.8,-.8) -- (5.2,-.8);
	\draw (1,0) -- (3.2,0);
	\draw (1.8,.8) -- (3.2,.8);
	\draw (3.2,.2) -- (5.2,.2);
	\draw (5.2,-.4) -- (7.4,-.4);
	\draw (2.8,-1.8) -- (7.4,-1.8);
	\draw (8.4,-1.1) -- (8.8,-1.1);
	\draw (-.8,-2.6) -- (-.8,1.6);
	\draw (0,-2.6) -- (0,-1.8);
	\draw[super thick, white] (1,-1.8) -- (1,1.6);
	\draw (1,-1.8) -- (1,1.6);
	\draw[super thick, white] (1.8,-2.6) -- (1.8,1.6);
	\draw (1.8,-2.6) -- (1.8,1.6);
	\roundNbox{unshaded}{(-.8,-.8)}{.3}{0}{0}{$f_1$}
	\roundNbox{unshaded}{(1,0)}{.3}{0}{0}{$f_2$}
	\roundNbox{unshaded}{(1.8,.8)}{.3}{0}{0}{$f_3$}
	\roundNbox{unshaded}{(0,-1.8)}{.3}{.8}{1.8}{$\alpha$}
	\tensor{(3.2,.2)}{.8}{}{}{}{}
	\tensor{(5.2,-.4)}{.8}{}{}{}{}
	\identityMap{(2.8,-1.8)}{.4}{}
	\multiplication{(7.4,-1.1)}{1}{}{}{}
\end{tikzpicture}
\displaybreak[1]
\\&=\,\,
(f_1\otimes (f_2\otimes f_3))\circ \alpha.
\end{align*}
Similarly, the two unitors $\lambda$ and $\rho$ are natural (we only show this for $\lambda$):
\begin{align*}
\lambda\circ (\id \otimes f) 
&\,=\,
\begin{tikzpicture}[baseline=.5cm]
	\draw (-.3,1.5) -- (-.3,2.4);
	\draw[dotted] (-.6,-1.2) -- (-.6,1.2);
	\draw (0,-1.2) -- (0,1.2);
	\draw (.4,0) -- (4.8,0);
	\draw (1.1,1.5) -- (4.8,1.5);
	\node at (-.5,2.2) {\scriptsize{$b$}};
	\node at (-.2,.6) {\scriptsize{$b$}};
	\node at (-.2,-.6) {\scriptsize{$a$}};
	\node at (-.8,0) {\scriptsize{$1_\cC$}};
	\tensorLeftId{(2.6,0)}{.8}{0}{n}{p}
	\identityMap{(1.1,1.5)}{.5}{p\,\,}
	\node at (1,.2) {\scriptsize{$\cP[p+n]$}};
	\node at (4,.2) {\scriptsize{$\cP[p+n]$}};
	\node at (3,1.7) {\scriptsize{$\cP[2q]$}};
	\roundNbox{unshaded}{(0,0)}{.4}{0}{0}{$f$}
	\roundNbox{unshaded}{(-.3,1.5)}{.4}{.2}{.2}{$\lambda$}
	\multiplication{(5.8,.75)}{1.2}{n}{p}{p}
	\draw (7,.75) -- (8.2,.75);
	\node at (7.6,.95) {\scriptsize{$\cP[p+n]$}};
\end{tikzpicture}
\displaybreak[1]
\\&=\,
\begin{tikzpicture}[baseline=.5cm]
	\draw (0,.8) -- (0,-.8);
	\draw (0,0) -- (1.6,0);
	\roundNbox{unshaded}{(0,0)}{.4}{0}{0}{$f$};
	\node at (-.2,.6) {\scriptsize{$b$}};
	\node at (1,.2) {\scriptsize{$\cP[p+n]$}};
	\node at (-.2,-.6) {\scriptsize{$a$}};
	\roundNbox{unshaded}{(-.3,1.2)}{.4}{.2}{.2}{$\lambda$}
	\draw[dotted] (-.6,-.8) -- (-.6,.8);
	\draw (-.3,1.6) -- (-.3,2);
	\node at (-.5,1.8) {\scriptsize{$b$}};
	\node at (-.8,0) {\scriptsize{$1_\cC$}};
\end{tikzpicture}
\,=
\begin{tikzpicture}[baseline=.5cm]
	\draw (-.3,0) -- (-.3,-.8);
	\draw (.1,-.4) .. controls ++(270:1cm) and ++(180:.5cm) .. (1.6,.6);
	\draw (-.5,.4) -- (-.5,1.2);
	\draw (-.1,.4) -- (-.1,1.2);
	\draw (-.3,1.2) -- (-.3,2);
	\draw[dotted] (-.7,-.8) -- (-.7,0);
	\node at (-.5,1.8) {\scriptsize{$b$}};
	\roundNbox{unshaded}{(-.3,1.2)}{.4}{0}{0}{$f$};
	\node at (.5,.6) {\scriptsize{$\cP[p+n]$}};
	\node at (-.7,.6) {\scriptsize{$a$}};
	\roundNbox{unshaded}{(-.3,0)}{.4}{.2}{.2}{$\lambda$}
	\node at (-.9,-.6) {\scriptsize{$1_\cC$}};
	\node at (-.5,-.6) {\scriptsize{$a$}};
	\node at (.8,-.8) {\scriptsize{$\cP[p+n]$}};
\end{tikzpicture}
\,\,=
\begin{tikzpicture}[baseline=.5cm]
	\draw (0,0) -- (0,-.8);
	\draw (-.3,.4) -- (-.3,2);
	\draw (-.3,1.2) -- (1.6,1.2);
	\draw[dotted] (-.6,-.8) -- (-.6,0);
	\roundNbox{unshaded}{(-.3,1.2)}{.4}{0}{0}{$f$};
	\node at (-.2,-.6) {\scriptsize{$a$}};
	\node at (.8,1.4) {\scriptsize{$\cP[p+n]$}};
	\roundNbox{unshaded}{(-.3,0)}{.4}{.2}{.2}{$\lambda$}
	\node at (-.5,1.8) {\scriptsize{$b$}};
	\node at (-.5,.6) {\scriptsize{$a$}};
	\node at (-.8,-.6) {\scriptsize{$1_\cC$}};
\end{tikzpicture}
\displaybreak[1]
\\&=\,
\begin{tikzpicture}[baseline=.5cm]
	\draw (0,0) -- (0,-.8);
	\draw (-.3,.4) -- (-.3,2);
	\draw (-.3,1.2) -- (3,1.2);
	\draw (1.1,0) -- (3,0);
	\draw[dotted] (-.6,-.8) -- (-.6,0);
	\roundNbox{unshaded}{(-.3,1.2)}{.4}{0}{0}{$f$};
	\node at (-.2,-.6) {\scriptsize{$a$}};
	\node at (1.1,1.4) {\scriptsize{$\cP[p+n]$}};
	\roundNbox{unshaded}{(-.3,0)}{.4}{.2}{.2}{$\lambda$}
	\node at (-.5,1.8) {\scriptsize{$b$}};
	\node at (-.5,.6) {\scriptsize{$a$}};
	\node at (-.8,-.6) {\scriptsize{$1_\cC$}};
	\identityMap{(1.1,0)}{.5}{n\,\,}
	\multiplication{(3.6,.6)}{1.2}{n}{n}{p}
	\node at (2,.2) {\scriptsize{$\cP[2n]$}};
	\draw (4.8,.6) -- (6,.6);
	\node at (5.4,.8) {\scriptsize{$\cP[p+n]$}};
\end{tikzpicture}
\displaybreak[1]
\\&=\,\,
f\circ \lambda.
\end{align*}
Note that the fourth equality follows from the relation $\lambda_{a\otimes c}\circ \alpha_{1_\cC,a,c}=\lambda_a\otimes \id_c$,
which holds true in any tensor category.

Finally, we verify that the pentagon and triangle axioms are satisfied:
\begin{align*}
\hspace{-.2cm}
(\id\otimes \alpha&)\circ \alpha \circ (\alpha\otimes \id)
=
\begin{tikzpicture}[baseline=-.1cm]
	\draw (2,1) -- (2.7,1);
	\draw (2,0) -- (2.7,0);
	\draw (.4,-1) -- (2.7,-1);
	\draw (-1.1,-1.6) -- (-1.1,1.6);
	\draw (-.9,-1.6) -- (-.9,-1);
	\draw (-.7,-1) -- (-.7,0);
	\draw (-.5,-1.6) -- (-.5,0);
	\draw (.7,1) -- (.7,0);
	\draw (.5,1.6) -- (.5,0);
	\draw (.9,1.6) -- (.9,1);
	\draw[super thick, white] (1.1,-1.6) -- (1.1,1.6);
	\draw (1.1,-1.6) -- (1.1,1.6);
	\roundNbox{unshaded}{(.8,1)}{.3}{.2}{.2}{$\alpha$}
	\roundNbox{unshaded}{(0,0)}{.3}{1}{1}{$\alpha$}
	\roundNbox{unshaded}{(-.8,-1)}{.3}{.2}{.2}{$\alpha$}
	\identityMap{(2,1)}{.4}{}
	\identityMap{(2,0)}{.4}{}
	\identityMap{(.4,-1)}{.4}{}
	\tensorRightId{(2,-1)}{.4}{}{}{}
	\draw[rounded corners=5pt, very thick, unshaded] (2.7,-1.3) rectangle (4.3,1.3);
	\coordinate (a) at (2.85,0);
	\draw[thick, red] (a) -- (3.3,.7);
	\draw[thick, red] (a) -- (3.3,0);
	\draw[thick, red] (a) -- (3.3,-.7);
	\draw (4.3,0) -- (4.7,0);
	\draw[very thick] (3.5,0) ellipse (.65 and 1.15);
	\draw (3.5,-1.15) -- (3.5,1.15);
	\filldraw[very thick, unshaded] (3.5,.7) circle (.2cm);
	\filldraw[very thick, unshaded] (3.5,0) circle (.2cm);
	\filldraw[very thick, unshaded] (3.5,-.7) circle (.2cm);
\end{tikzpicture}
=
\begin{tikzpicture}[baseline=-.1cm]
	\draw (1.2,.6) -- (2,.6);
	\draw (1.2,-.6) -- (2,-.6);
	\draw (3.8,0) -- (3.4,0);
	\draw (-.4,-1.2) -- (-.4,1.2);
	\draw (-.3,-1.2) -- (-.3,.6);
	\draw (-.1,-1.2) -- (-.1,-.6);
	\draw (.1,1.2) -- (.1,.6);
	\draw (.3,1.2) -- (.3,-.6);
	\draw (.4,-1.2) -- (.4,1.2);
	\roundNbox{unshaded}{(0,.6)}{.3}{.3}{.3}{$\alpha$}
	\roundNbox{unshaded}{(0,-.6)}{.3}{.3}{.3}{$\alpha$}
	\identityMap{(1.2,.6)}{.4}{}
	\identityMap{(1.2,-.6)}{.4}{}
	\multiplication{(2.6,0)}{.8}{}{}{}
\end{tikzpicture}
=
\alpha\circ\alpha
\\
&(\id\otimes\lambda)\circ \alpha 
=
\begin{tikzpicture}[baseline=-.1cm]
	\draw (-.5,-1.2) -- (-.5,1.2);
	\draw (.4,.6) -- (.4,1.2);
	\draw[dotted] (.3,-.6) -- (.3,.6);
	\draw (.5,-1.2) -- (.5,.6);
	\draw[dotted] (-.3,-1.2) -- (-.3,-.6);
	\draw (1.3,.6) -- (4.5,.6);
	\draw (1.3,-.6) -- (4.5,-.6);
	\draw (4.3,0) -- (5.4,0);
	\roundNbox{unshaded}{(.4,.6)}{.3}{0}{0}{$\lambda$}
	\roundNbox{unshaded}{(0,-.6)}{.3}{.4}{.4}{$\alpha$}
	\identityMap{(1.3,.6)}{.4}{}
	\identityMap{(1.3,-.6)}{.4}{}
	\tensorLeftId{(2.6,.6)}{.6}{}{}{}
	\multiplication{(4.3,0)}{.8}{}{}{}
\end{tikzpicture}
=
\begin{tikzpicture}[baseline=-.1cm]
	\draw (0,0) -- (0,.8);
	\draw (-.1,-.8) -- (-.1,0);
	\draw[dotted] (.15,-.8) -- (.15,0);
	\draw (1,0) -- (3.3,0);
	\draw[super thick, white] (1.6,-.8) -- (1.6,.8);
	\draw (1.6,-.8) -- (1.6,.8);
	\roundNbox{unshaded}{(0,0)}{.4}{0}{0}{$\rho$}
	\identityMap{(1,0)}{.4}{}
	\tensorRightId{(2.4,0)}{.6}{}{}{}
\end{tikzpicture}
=
\rho\otimes \id.
\end{align*}
This finishes the proof that $\cM_0$, and thus $\cM$, is a tensor category.\bigskip

Note that, by \eqref{Phi M eq1}, \eqref{Phi M eq2} and \eqref{Phi M eq3}, there is a canonical isomorphism
\[
\text{``}\Phi(c)\otimes m^{\otimes n}\text{''}\,\cong\,\,\Phi(c)\otimes m^{\otimes n}.
\]
We may therefore blur the distinction between the formal symbol $\text{``}\Phi(c)\otimes m^{\otimes n}\text{''}$
and the tensor product $\Phi(c)\otimes m^{\otimes n}$.

We also note that $\Phi:\cC\to \cM_0$ (and therefore $\Phi:\cC\to \cM$) is visibly a tensor functor.

\subsection{The pivotal structure}\label{sec:The pivotal structure}

The goal of this section is to show that the category $\cM$ is pivotal, and that the object $m\in\cM$ is symmetrically self-dual.
Recall that the balanced structure on $\cC$ was defined in \eqref{def:theta1}.
We denote the twist and its inverse graphically by:
\[
\theta_a =\! 
\begin{tikzpicture}[baseline=-.1cm]
	\draw (0,-.8) -- (0,.8);
	\twist{(0,0)}
	\node at (-.2,-.6) {\scriptsize{$a$}};
	\node at (-.2,.6) {\scriptsize{$a$}};
\end{tikzpicture}
\qquad\quad
\text{and}
\qquad\quad
\theta_a^{-1}=\!
\begin{tikzpicture}[baseline=-.1cm]
	\draw (0,-.8) -- (0,.8);
	\twistInverse{(0,0)}
	\node at (-.2,-.6) {\scriptsize{$a$}};
	\node at (-.2,.6) {\scriptsize{$a$}};
\end{tikzpicture}
\,.
\]

The rigidity of $\cC$ implies the corresponding property for $\cM_0$ (which then implies that $\cM$ is rigid):

\begin{lem}
\label{lem:DualityInM0}
For any $c\in \cC$, $n\ge 0$,
the morphisms\vspace{-.2cm}
$$
\ev_{\Phi(c)\otimes m^{\otimes n}}\,\,:=\,\,\,
\begin{tikzpicture}[baseline=-.1cm]
	\draw (-.4,-.8) node[above, yshift=-12, xshift=1]{$\scriptstyle c^*$} -- (-.4,-.4) arc (180:0:.2cm) -- (0,-.8)node[above, yshift=-12]{$\scriptstyle c$};
	\draw (1.3,0) -- (2.3,0);
	\node[xshift=1] at (-.2,0) {\scriptsize{$\ev_c$}};
	\evaluationMap{(.8,0)}{.5}{n}
	\node at (1.8,.2) {\scriptsize{$\cP[2n]$}};
\end{tikzpicture}
\quad\,\,\,\text{and}\,\,\,\quad
\coev_{\Phi(c)\otimes m^{\otimes n}}\,\,:=\,\,\,
\begin{tikzpicture}[baseline=-.1cm]
	\draw (-.4,.8)node[above]{$\scriptstyle c$} -- (-.4,.4) arc (-180:0:.2cm) -- (0,.8)node[above, xshift=2]{$\scriptstyle c^*$};
	\draw (1.3,0) -- (2.3,0);
	\node at (-.2,0) {\scriptsize{$\coev_c$}};
	\coevaluationMap{(.8,0)}{.5}{n}
	\node at (1.8,.2) {\scriptsize{$\cP[2n]$}};
\end{tikzpicture}
$$
exhibit $\Phi(c^*)\otimes m^{\otimes n}$ as the dual object of $\Phi(c)\otimes m^{\otimes n}$ in $\cM_0$.
\end{lem}

\begin{proof}
We need to show the two zig-zag equations.
The first one is as follows:
$$
(\id\otimes \ev)\circ (\coev\otimes \id)
\,=
\begin{tikzpicture}[baseline=-.1cm]
	\draw (1.4,.6) -- (5,.6);
	\draw (0,-.6) -- (5,-.6);
	\draw[super thick, white] (.6,-1.2) -- (.6,.6);
	\draw (.6,-1.2) -- (.6,.6) .. controls ++(90:.6cm) and ++(90:.6cm) .. (-.6,.6) -- (-.6,-.6) arc (0:-180:.2cm) -- (-1,1.2);
	\draw (5.2,0) -- (5.6,0);
	\node at (-1.2,0) {\scriptsize{$c$}};
	\node at (-.4,0) {\scriptsize{$c^*$}};
	\node at (.8,0) {\scriptsize{$c$}};
	\evaluationMap{(1.2,.6)}{.4}{n}
	\coevaluationMap{(0,-.6)}{.4}{n}
	\tensorLeftIdEv{(2.4,.6)}{.4}{}{}
	\tensorRightIdCoev{(2.4,-.6)}{.4}{}{}
	\multiplication{(4.2,0)}{1}{n}{3n}{n}
\end{tikzpicture}
\,=
\begin{tikzpicture}[baseline=-.1cm]
	\draw (0,-.8) -- (0,.8);
	\draw (1,0) -- (1.8,0);
	\node at (-.2,0) {\scriptsize{$c$}};
	\identityMap{(.6,0)}{.4}{n\,\,\,}
	\node at (1.4,.2) {\scriptsize{$\cP[2n]$}};
\end{tikzpicture}
\,=\id.
$$
The other one is similar.
\end{proof}

From now on, we identify $(\Phi(c)\otimes m^{\otimes n})^*$ with $\Phi(c^*)\otimes m^{\otimes n}$ by the above lemma.
We now equip $\cM_0$ (and hence $\cM$) with a pivotal stucture:
\begin{gather}\notag
\varphi_{\Phi(a)\otimes m^{\otimes n}}\,:\, \Phi(a)\otimes m^{\otimes n} \to (\Phi(a)\otimes m^{\otimes n})^{**}=\Phi(a^{**})\otimes m^{\otimes n}
\\
\label{eq: def: piv structure}
\varphi_{\Phi(a)\otimes m^{\otimes n}}\,:= 
\begin{tikzpicture}[baseline=-.1cm]
	\draw (0,-1) -- (0,1);
	\twist{(0,-.4)}
	\loopIsoReverse{(0,.4)}
	\draw (1.4,0) -- (2.4,0);
	\node at (-.2,-.8) {\scriptsize{$a$}};
	\node at (-.3,.8) {\scriptsize{$a^{**}$}};
	\identityMap{(1,0)}{.4}{n\,\,\,}
	\node at (1.9,.2) {\scriptsize{$\cP[2n]$}};
\end{tikzpicture}
\,\,=
\begin{tikzpicture}[baseline=-.1cm]
	\draw (0,-1) -- (0,1);
	\roundNbox{unshaded}{(0,0)}{.4}{0}{0}{$\varphi_a$}
	\draw (1.4,0) -- (2.4,0);
	\node at (-.2,-.8) {\scriptsize{$a$}};
	\node at (-.3,.8) {\scriptsize{$a^{**}$}};
	\identityMap{(1,0)}{.4}{n\,\,\,}
	\node at (1.9,.2) {\scriptsize{$\cP[2n]$}};
\end{tikzpicture}\,.
\end{gather}

\begin{lem}
The isomorphism $\varphi$ is natural.
\end{lem}
\begin{proof}
Let $f: \Phi(a)\otimes m^n \to \Phi(b)\otimes m^p$ be a morphism in $\cM_0$.
We must show that
$$
\varphi \circ f 
=\,
\begin{tikzpicture}[baseline=-.5cm]
	\draw (0,-2.2) -- (0,1);
	\twist{(0,-.4)}
	\loopIsoReverse{(0,.4)}
	\draw (1.7,0) -- (2,0);
	\draw (.4,-1.4) -- (2,-1.4);
	\draw (4.4,-.7) -- (5,-.7);
	\node at (-.2,-2) {\scriptsize{$a$}};
	\node at (-.2,-.8) {\scriptsize{$b$}};
	\node at (-.3,.8) {\scriptsize{$b^{**}$}};
	\identityMap{(1.2,0)}{.5}{p\,\,}
	\roundNbox{unshaded}{(0,-1.4)}{.4}{0}{0}{$f$};
	\multiplication{(3.2,-.7)}{1.2}{n}{p}{p}
\end{tikzpicture}
\,\,=
\begin{tikzpicture}[baseline=-.5cm]
	\draw (0,-2.2) -- (0,1);
	\twist{(0,-.4)}
	\loopIsoReverse{(0,.4)}
	\draw (.4,-1.4) -- (1,-1.4);
	\node at (-.2,-2) {\scriptsize{$a$}};
	\node at (-.2,-.8) {\scriptsize{$b$}};
	\node at (-.3,.8) {\scriptsize{$b^{**}$}};
	\roundNbox{unshaded}{(0,-1.4)}{.4}{0}{0}{$f$};
\end{tikzpicture}
$$
is equal to $f^{**}\circ \varphi$, where 
$f^{**}:\Phi(a^{**})\otimes m^n \to \Phi(b^{**})\otimes m^p$ is the double dual of $f$: 
$$
f^{**} = (\ev \otimes \id)\circ (\id\otimes \ev\otimes \id\otimes \id)\circ (\id\otimes \id \otimes f\otimes \id\otimes \id)\circ (\id\otimes \id\otimes \coev\otimes \id)\circ (\id\otimes\coev). 
$$
We compute:
\begin{align*}
f^{**} \circ \varphi
\,&=
\begin{tikzpicture}[baseline=2.4cm]
	\draw (-4,-1) -- (-4,1);
	\twist{(-4,-.4)}
	\loopIsoReverse{(-4,.4)}
	\node at (-3.8,-.8) {\scriptsize{$a$}};
	\node at (-3.7,.8) {\scriptsize{$a^{**}$}};
	\evaluationMap{(-.7,5.4)}{.4}{n}
	\evaluationMap{(-2.1,4.4)}{.4}{p}
	\roundNbox{unshaded}{(-3.2,3.4)}{.4}{0}{0}{$f$}
	\coevaluationMap{(-.8,2.4)}{.4}{n}
	\coevaluationMap{(.7,1.4)}{.4}{p}
	\identityMap{(-2.8,0)}{.5}{n}
	\draw (-.3,5.4) -- (3.2,5.4);
	\draw (-1.7,4.4) -- (3.2,4.4);
	\draw (-2.8,3.4) -- (3.2,3.4);
	\draw (-.4,2.4) -- (3.2,2.4);
	\draw (1.1,1.4) -- (3.2,1.4);
	\draw (-2.3,0) -- (3.2,0);
	\draw[super thick, white] (0,2) -- (0,6);
	\draw (0,6) -- (0,2) .. controls ++(270:1cm) and ++(270:1cm) .. (-3.8,2) -- (-3.8,4.2) arc (180:0:.3cm) -- (-3.2,3.8);
	\draw (-3.2,3) -- (-3.2,2.6) .. controls ++(270:.6cm) and ++(270:.6cm) .. (-1.4,2.6);
	\draw[super thick, white] (-1.4,2.6) -- (-1.4,5);
	\draw (-1.4,2.6) -- (-1.4,5) .. controls ++(90:1cm) and ++(90:1cm) .. (-4,5) -- (-4,1);
	\node at (.3,5.8) {\scriptsize{$b^{**}$}};
	\node at (-3,4) {\scriptsize{$b$}};
	\node at (-3,2.8) {\scriptsize{$a$}};
	\tensorRightIdEv{(2.4,5.4)}{.4}{}{}{}
	\tensorLeftRightIdEv{(2.4,4.4)}{.4}
	\tensorLeftRightId{(2.4,3.4)}{.4}
	\tensorLeftRightIdCoev{(2.4,2.4)}{.4}
	\tensorLeftIdCoev{(2.4,1.4)}{.4}{}{}{}
	\draw (6.8,2.4) -- (7.4,2.4);
	\draw[rounded corners=5pt, very thick, unshaded] (3.2,-.6) rectangle (6.8,6);
	\coordinate (a) at (3.4,2.7);
	\draw[thick, red] (a) -- (4.75,5);
	\draw[thick, red] (a) -- (4.75,4.1);
	\draw[thick, red] (a) -- (4.75,3.2);
	\draw[thick, red] (a) -- (4.75,2.3);
	\draw[thick, red] (a) -- (4.75,1.4);
	\draw[thick, red] (a) -- (4.75,.5);
	\draw (5,5.7) -- (5,-.3);
	\draw[very thick] (5,2.7) ellipse (1.6 and 3);
	\filldraw[very thick, unshaded] (5,5) circle (.25cm);
	\filldraw[very thick, unshaded] (5,4.1) circle (.25cm);
	\filldraw[very thick, unshaded] (5,3.2) circle (.25cm);
	\filldraw[very thick, unshaded] (5,2.3) circle (.25cm);
	\filldraw[very thick, unshaded] (5,1.4) circle (.25cm);
	\filldraw[very thick, unshaded] (5,.5) circle (.25cm);
	\node at (5.2,5.45) {\scriptsize{$p$}};
	\node at (5.4,4.55) {\scriptsize{$2n{+}p$}};
	\node at (5.8,3.65) {\scriptsize{$n{+}2p{+}n{+}p$}};
	\node at (5.8,2.75) {\scriptsize{$n{+}p{+}2n{+}p$}};
	\node at (5.4,1.85) {\scriptsize{$n{+}2p$}};
	\node at (5.2,.95) {\scriptsize{$n$}};
	\node at (5.2,.05) {\scriptsize{$n$}};
\end{tikzpicture}
\displaybreak[1]\\&=
\begin{tikzpicture}[baseline=.8cm]
	\draw (-2,-1) -- (-2,.9);
	\twist{(-2,-.4)}
	\loopIsoReverse{(-2,.4)}
	\node at (-2.2,-.8) {\scriptsize{$a$}};
	\node at (-2.3,1.2) {\scriptsize{$a^{**}$}};
	\roundNbox{unshaded}{(-1,1)}{.4}{0}{0}{$f$}
	\draw (-.6,1) -- (4,1);
	\draw[super thick, white] (0,2.8) -- (0,.2);
	\draw (0,2.8) -- (0,.2) .. controls ++(270:1cm) and ++(270:1cm) .. (-1.6,.2) -- (-1.6,1.8) arc (180:0:.3cm) -- (-1,1.4);
	\draw (-1,.6) -- (-1,.2) arc (-180:0:.3cm);
	\draw[super thick, white] (-.4,.2) -- (-.4,1.8);
	\draw (-.4,.2) -- (-.4,1.8) .. controls ++(90:1cm) and ++(90:1cm) .. (-2,1.8) -- (-2,.9);
	\node at (.3,2.6) {\scriptsize{$b^{**}$}};
	\node at (-.8,1.6) {\scriptsize{$b$}};
	\node at (-.8,.4) {\scriptsize{$a$}};
	\draw[rounded corners=5pt, very thick, unshaded] (.6,-.4) rectangle (3.4,2.4);
	\draw[thick, red] (.8,1) -- (1.7,1);
	\draw[very thick] (2,1) circle (1.2cm);
	\draw[very thick] (2,1) circle (.3cm);
	\draw (2,1.3) .. controls ++(90:.4cm) and ++(90:.6cm) .. (1.4,1) .. controls ++(270:1cm) and ++(270:1cm) .. (2.9,1) .. controls ++(90:.8cm) and ++(270:.2cm) .. (2,2.2);
	\draw (2,.7) .. controls ++(270:.4cm) and ++(270:.6cm) .. (2.6,1) .. controls ++(90:1cm) and ++(90:1cm) .. (1.1,1) .. controls ++(270:.8cm) and ++(90:.2cm) .. (2,-.2);
	\node at (2.1,0) {\scriptsize{$n$}};
	\node at (1.9,.6) {\scriptsize{$n$}};
	\node at (2.1,1.4) {\scriptsize{$p$}};
	\node at (1.9,2) {\scriptsize{$p$}};
\end{tikzpicture}
\,\,=
\begin{tikzpicture}[baseline=.8cm]
	\draw (-2,-1) -- (-2,.9);
	\twist{(-2,-.4)}
	\loopIsoReverse{(-2,.4)}
	\node at (-2.2,-.8) {\scriptsize{$a$}};
	\node at (-2.3,1.2) {\scriptsize{$a^{**}$}};
	\roundNbox{unshaded}{(-1,1)}{.4}{0}{0}{$f$}
	\draw (-.6,1) -- (3.6,1);
	\draw[super thick, white] (0,2.8) -- (0,.2);
	\draw (0,2.8) -- (0,.2) .. controls ++(270:1cm) and ++(270:1cm) .. (-1.6,.2) -- (-1.6,1.8) arc (180:0:.3cm) -- (-1,1.4);
	\draw (-1,.6) -- (-1,.2) arc (-180:0:.3cm);
	\draw[super thick, white] (-.4,.2) -- (-.4,1.8);
	\draw (-.4,.2) -- (-.4,1.8) .. controls ++(90:1cm) and ++(90:1cm) .. (-2,1.8) -- (-2,.9);
	\node at (.3,2.6) {\scriptsize{$b^{**}$}};
	\node at (-.8,1.6) {\scriptsize{$b$}};
	\node at (-.8,.4) {\scriptsize{$a$}};
	\draw[rounded corners=5pt, very thick, unshaded] (.6,-.2) rectangle (3,2.2);
	\draw (1.8,0) -- (1.8,2);
	\draw[thick, red] (1.5,1) .. controls ++(180:.3cm) and ++(180:.5cm) .. (1.8,1.5) .. controls ++(0:.8cm) and ++(0:.8cm) .. (1.8,.4) .. controls ++(180:.6cm) and ++(0:.4cm) .. (.8,1);
	\draw[very thick] (1.8,1) circle (1cm);
	\draw[unshaded, very thick] (1.8,1) circle (.3cm);
	\node at (2,.2) {\scriptsize{$n$}};
	\node at (2,1.8) {\scriptsize{$p$}};
\end{tikzpicture}
\displaybreak[1]\\&=
\begin{tikzpicture}[baseline=.8cm]
	\draw (-2,-1) -- (-2,.9);
	\twist{(-2,-.4)}
	\loopIsoReverse{(-2,.4)}
	\node at (-2.2,-.8) {\scriptsize{$a$}};
	\node at (-2.3,1.2) {\scriptsize{$a^{**}$}};
	\roundNbox{unshaded}{(-1,1)}{.4}{0}{0}{$f$}
	\draw (-.6,1) -- (1.6,1);
	\draw[super thick, white] (0,2.8) -- (0,.2);
	\draw (0,2.8) -- (0,.2) .. controls ++(270:1cm) and ++(270:1cm) .. (-1.6,.2) -- (-1.6,1.8) arc (180:0:.3cm) -- (-1,1.4);
	\draw (-1,.6) -- (-1,.2) arc (-180:0:.3cm);
	\draw[super thick, white] (-.4,.2) -- (-.4,1.8);
	\draw (-.4,.2) -- (-.4,1.8) .. controls ++(90:1cm) and ++(90:1cm) .. (-2,1.8) -- (-2,.9);
	\node at (.3,2.6) {\scriptsize{$b^{**}$}};
	\node at (-.8,1.6) {\scriptsize{$b$}};
	\node at (-.8,.4) {\scriptsize{$a$}};
	\twistHorizontalInverse{(.8,1)}
\end{tikzpicture}
\,\,=\,\,\,
\begin{tikzpicture}[baseline=.3cm, yscale=-1]
	\draw (0,.4) -- (0,1.6);
	\twistInverse{(0,1)}
	\node at (1.3,-2.6) {\scriptsize{$b^{**}$}};
	\node at (-.2,-.6) {\scriptsize{$b$}};
	\node at (-.2,.6) {\scriptsize{$a$}};
	\node at (-.2,1.4) {\scriptsize{$a$}};
	\draw (0,-.4) -- (0,-.6) arc (-180:-90:.2cm);
	\draw (1,-1.6) -- (2.6,-1.6);
	\draw (.2,-.8) -- (1.4,-.8);
	\draw[super thick, white] (1.4,-.8) arc (90:0:.2cm) -- (1.6,-2.8);
	\draw (1.4,-.8) arc (90:0:.2cm) -- (1.6,-2.8);
	\draw[super thick, white] (.4,0) -- (.6,0) arc (90:0:.2cm) -- (.8,-1.4) arc (-180:-90:.2cm);
	\draw (.4,0) -- (.6,0) arc (90:0:.2cm) -- (.8,-1.4) arc (-180:-90:.2cm);
	\roundNbox{unshaded}{(0,0)}{.4}{0}{0}{$f$}
	\loopIso{(1.6,-2.2)}
	\twistHorizontal{(2.2,-1.6)}
\end{tikzpicture}
\!=\,\,
\begin{tikzpicture}[baseline=-.9cm]
	\draw (0,-2.2) -- (0,1);
	\twist{(0,-.4)}
	\loopIsoReverse{(0,.4)}
	\draw (.4,-1.4) -- (1,-1.4);
	\node at (-.2,-2) {\scriptsize{$a$}};
	\node at (-.2,-.8) {\scriptsize{$b$}};
	\node at (-.3,.8) {\scriptsize{$b^{**}$}};
	\roundNbox{unshaded}{(0,-1.4)}{.4}{0}{0}{$f$};
\end{tikzpicture}
\displaybreak[1]
\end{align*} 
where the last equality follows from the compatibility between the full braiding and the twist, valid in any balanced tensor category.
\end{proof}

\begin{lem}\label{lem: The isomorphism varphi is monoidal}
The isomorphism $\varphi$ is monoidal: $\varphi_a\otimes\varphi_b=\varphi_{a\otimes b}$.
\end{lem}
\begin{proof}
We compute:
$$
\begin{tikzpicture}[baseline=.4cm]
	\draw (1.4,0) -- (4,0);
	\draw (3.4,1) -- (4,1);
	\draw (0,-1) -- (0,2);
	\draw[super thick, white] (2,-1) -- (2,2);
	\draw (2,-1) -- (2,2);
	\roundNbox{unshaded}{(0,0)}{.4}{0}{0}{$\varphi_a$}	
	\roundNbox{unshaded}{(2,1)}{.4}{0}{0}{$\varphi_b$}	
	\node at (-.2,-.8) {\scriptsize{$a$}};
	\node at (-.3,1.8) {\scriptsize{$a^{**}$}};
	\node at (1.8,-.8) {\scriptsize{$b$}};
	\node at (1.7,1.8) {\scriptsize{$b^{**}$}};
	\identityMap{(1,0)}{.4}{n\,\;}
	\identityMap{(3,1)}{.4}{p\,\;}
	\tensor{(5.2,.5)}{1.2}{n}{n}{p}{p}
	\draw (6.4,.5) -- (7,.5);
\end{tikzpicture}
\,\,\,=\!
\begin{tikzpicture}[baseline=-.1cm]
	\draw (0,-1) -- (0,1);
	\roundNbox{unshaded}{(0,0)}{.4}{.1}{.1}{$\varphi_{a\otimes b}$}	
	\node at (-.4,-.8) {\scriptsize{$a\otimes b$}};
	\node at (-.6,.8) {\scriptsize{$(a\otimes b)^{**}$}};
	\draw (1.9,0) -- (2.5,0);
	\identityMap{(1.3,0)}{.6}{\;\scriptscriptstyle n{+}p}
\end{tikzpicture}\,.\vspace{-.3cm}
$$
\end{proof}

The object $m$ is visibly self-dual.
Its evaluation of coevaluations maps are given by:
\begin{equation}\label{eq: ev and coev for m}
\ev_m = 
\begin{tikzpicture}[baseline=-.1cm]
	\draw[thick, dotted] (-.4,-.8) -- (-.4,-.4) arc (180:0:.2cm) -- (0,-.8);
	\draw (1.3,0) -- (2.3,0);
	\node at (-.2,0) {\scriptsize{$1_\cC$}};
	\evaluationMap{(.8,0)}{.5}{ }
	\node at (1.8,.2) {\scriptsize{$\cP[2]$}};
\end{tikzpicture}
\qquad\text{and}\qquad
\coev_m=
\begin{tikzpicture}[baseline=-.1cm]
	\draw[thick, dotted] (-.4,.8) -- (-.4,.4) arc (-180:0:.2cm) -- (0,.8);
	\draw (1.3,0) -- (2.3,0);
	\node at (-.2,0) {\scriptsize{$1_\cC$}};
	\coevaluationMap{(.8,0)}{.5}{ }
	\node at (1.8,.2) {\scriptsize{$\cP[2]$}};
\end{tikzpicture}\,\,.
\end{equation}
We now show that $m$ is symmetrically self-dual (Section \ref{sec:InternalTrace}).
Upon identifying $m^*$ with $m$ by means of \eqref{eq: ev and coev for m}, we let $\psi:m\to m^*$ be the identity map.
In the same way, we identify $m^{**}$ with $m^*$ and with $m$.
The isomorphism $\varphi:m\to m^{**}$ described in \eqref{eq: def: piv structure} is then also the identity map, as $\varphi_1=1$ in $\cC$.
Finally, letting $\psi^*$ be as in \eqref{eq: psi** def}, we verify that the desired equation $\psi^*\circ\varphi=\psi$ holds true:
\begin{equation}\label{eq: It's symmetrically self-dual!}
\psi^*=\,\,\begin{tikzpicture}[baseline=.3cm]
	\draw (1,.4) -- (6.4,.4);
	\draw (2,-.6) -- (4,-.6);
	\draw (1,1.4) -- (4,1.4);
	\draw[thick, dotted] (-.4,-1) -- (-.4,1.6) arc (180:0:.2cm) -- (0,-.4) .. controls ++(270:.6cm) and ++(270:.6cm) .. (1.2,-.4) -- (1.2,2);
	\identityMap{(.6,.4)}{.4}{ }
	\evaluationMap{(.6,1.4)}{.4}{ }
	\coevaluationMap{(1.8,-.6)}{.4}{ }
	\tensorLeftIdCoev{(3,-.6)}{.4}{}{}
	\tensorRightIdEv{(3,1.4)}{.4}{}{}
	\tensorLeftRightId{(3,.4)}{.4}
	\coordinate (a) at (4,-1);
	\coordinate (b) at ($ (a) + (.2,1.4) $);
	\pgfmathsetmacro{\width}{.2}
	\pgfmathsetmacro{\height}{.6}
	\coordinate (c) at ($ (a) + (\width,\height) $);
	\draw[rounded corners=5pt, very thick, unshaded] (a) rectangle ($ (a) + (1.8,2.8) $);
	\draw[thick, red] ($ (a) + (.7,.7) $) -- (b);
	\draw[thick, red] ($ (a) + (.7,1.4) $) -- (b);
	\draw[thick, red] ($ (a) + (.7,2.1) $) -- (b);
	\draw[very thick] ($ (a) + (.9,1.4) $) ellipse (.7 and 1.3);
	\draw ($ (a) + (.9,.1) $) -- ($ (a) + (.9, 2.7) $);
	\filldraw[very thick, unshaded] ($ (a) + (.9,.7) $) circle (.2cm);
	\filldraw[very thick, unshaded] ($ (a) + (.9,1.4) $) circle (.2cm);
	\filldraw[very thick, unshaded] ($ (a) + (.9,2.1) $) circle (.2cm);
	\node at ($(a) + (1.02,.35)$) {\scriptsize{$1$}};
	\node at ($(a) + (1.05,1.05)$) {\scriptsize{$3$}};
	\node at ($(a) + (1.05,1.75)$) {\scriptsize{$3$}};
	\node at ($(a) + (1.02,2.48)$) {\scriptsize{$1$}};
	\node at (6.2,.6) {\scriptsize{$\cP[2]$}};
\end{tikzpicture}
=
\begin{tikzpicture}[baseline=-.1cm]
	\draw (0,0) -- (4,0);
	\identityMap{(0,0)}{.4}{ }
\pgftransformscale{.9}
\pgftransformxshift{5}
	\coordinate (a) at (2,0);
	\draw[rounded corners=5pt, very thick, unshaded] (.8,-1.2) rectangle (3.2,1.2);
	\draw[thick, red] (1,0) -- (1.7,0);
	\draw ($ (a) + (90:.3cm) $) .. controls ++(90:.4cm) and ++(90:.5cm) .. ($ (a) + (180:.6cm) $) .. controls ++(270:.5cm) and ++(90:.4cm) .. ($ (a) + (270:1cm) $);
	\draw ($ (a) + (270:.3cm) $) .. controls ++(270:.4cm) and ++(270:.5cm) .. ($ (a) + (0:.6cm) $) .. controls ++(90:.5cm) and ++(270:.4cm) .. ($ (a) + (90:1cm) $);
	\draw[very thick] (2,0) circle (1cm);
	\draw[very thick, unshaded] (2,0) circle (.3cm);
	\node at (3.65,.22) {\scriptsize{$\cP[2]$}};
\end{tikzpicture}
=
\begin{tikzpicture}[baseline=-.1cm]
	\draw (0,0) -- (1.2,0);
	\identityMap{(0,0)}{.4}{ }
	\node at (.8,.2) {\scriptsize{$\cP[2]$}};
\end{tikzpicture}\,.
\end{equation}

\subsection{The central functor}\label{sec:The central functor}
In this section, we equip the functor $\Phi:\cC\to \cM$ with the structure of a central functor, i.e., we provide a factorization
\[
\xymatrix@C=1cm
{
&\cZ(\cM)\ar[dr]^F\\
\cC\ar[ur]^{\Phi^\cZ}\ar[rr]^{\Phi}&&\cM.
}
\]
To do so, we associate to each object $c\in \cC$ a half-braiding $e_{\Phi(c)}$, so that
$(\Phi(c), e_{\Phi(c)})$ is an object of $\cZ(\cM)$.
Once again, it is enough to do this at the level of the category $\cM_0$.
For $c\in \cC$ and $\Phi(a)\otimes m^{\otimes n}\in \cM_0$, we define:
\begin{equation}\label{eq: That's the half-braiding!}
e_{\Phi(c),\Phi(a)\otimes m^{\otimes n}}
\,:=\,\Phi(\beta_{c,a})\otimes\id_{m^{\otimes n}}\,=
\begin{tikzpicture}[baseline=.4cm]
	\draw (.2,0)  .. controls ++(90:.2cm) and ++(270:.2cm) .. ($ (-.2,0) + (0,1) $);
	\draw[super thick, white] (-.2,0)  .. controls ++(90:.2cm) and ++(270:.2cm) ..($ (.2,0) + (0,1) $);
	\draw (-.2,0)  .. controls ++(90:.2cm) and ++(270:.2cm) ..($ (.2,0) + (0,1) $);
	\node at (-.35,.1) {\scriptsize{$c$}};
	\node at (-.35,.9) {\scriptsize{$a$}};
	\identityMap{(1,.5)}{.5}{n\,\,}
	\draw (1.5,.5) -- (2.3,.5);
	\node at (1.9,.7) {\scriptsize{$\cP[2n]$}};
\end{tikzpicture}\,.
\end{equation}

\begin{lem}
The isomorphisms $e_{\Phi(c)}$ are natural.
\end{lem}
\begin{proof}
Given a morphism $f: \Phi(a)\otimes m^{\otimes n} \to \Phi(b)\otimes m^{\otimes p}$ in $\cM_0$, we compute:
\begin{align*}
(f\otimes \id) \circ e_{\Phi(c),\Phi(a)\otimes m^{\otimes n}}
\,\,&=\,
\begin{tikzpicture}[baseline=1.3cm]
	\draw (1.2,.5) -- (3.4,.5);
	\draw (0,2.1) -- (3.4,2.1);
	\draw (-.4,1) -- (-.4,3.1);
	\draw (.4,0)  .. controls ++(90:.3cm) and ++(270:.3cm) .. ($ (-.4,0) + (0,1) $);
	\draw[super thick, white] (-.4,0)  .. controls ++(90:.3cm) and ++(270:.3cm) ..($ (.4,0) + (0,1) $);
	\draw (-.4,0)  .. controls ++(90:.3cm) and ++(270:.3cm) ..($ (.4,0) + (0,1) $);
	\draw[super thick, white] (.4,1) -- (.4,3.1);
	\draw (.4,1) -- (.4,3.1);
	\node at (-.6,1.2) {\scriptsize{$a$}};
	\node at (-.6,2.8) {\scriptsize{$b$}};
	\node at (.2,1.2) {\scriptsize{$c$}};
	\tensorRightIdZero{(1.8,2.1)}{.8}{n}{p}
	\identityMap{(1.2,.5)}{.5}{n\,\,}
	\roundNbox{unshaded}{(-.4,2.1)}{.4}{0}{0}{$f$}
	\multiplication{(4.6,1.3)}{1.2}{n}{n}{p}
	\draw (5.8,1.3) -- (6.4,1.3);
\end{tikzpicture}
\displaybreak[1]
\\&=\,
\begin{tikzpicture}[baseline=-.5cm]
	\draw (1.2,.5) -- (3.4,.5);
	\draw (.4,-1.1) -- (3.4,-1.1);
	\draw (-.4,0) -- (-.4,-2.1);
	\draw (.4,0)  .. controls ++(90:.3cm) and ++(270:.3cm) .. ($ (-.4,0) + (0,1) $);
	\draw[super thick, white] (-.4,0)  .. controls ++(90:.3cm) and ++(270:.3cm) ..($ (.4,0) + (0,1) $);
	\draw (-.4,0)  .. controls ++(90:.3cm) and ++(270:.3cm) ..($ (.4,0) + (0,1) $);
	\draw[super thick, white] (.4,0) -- (.4,-2.1);
	\draw (.4,0) -- (.4,-2.1);
	\node at (-.6,-.2) {\scriptsize{$c$}};
	\node at (.2,-1.8) {\scriptsize{$a$}};
	\node at (.2,-.2) {\scriptsize{$b$}};
	\tensorLeftIdZero{(2.1,-1.1)}{.8}{n}{p}
	\identityMap{(1.2,.5)}{.5}{p\,\,}
	\roundNbox{unshaded}{(.4,-1.1)}{.4}{0}{0}{$f$}
	\multiplication{(4.6,-.3)}{1.2}{n}{p}{p}
	\draw (5.8,-.3) -- (6.4,-.3);
\end{tikzpicture}
\displaybreak[1]
\\&=\,\,\,
e_{\Phi(c),\Phi(b)\otimes m^{\otimes p}}\circ (\id \otimes f).
\qedhere
\displaybreak[1]
\end{align*}
\end{proof}

\begin{lem}
The isomorphisms $e_{\Phi(a)}$ satisfy the hexagon axiom
$$
(\id_{\Phi(b)\otimes m^{\otimes n}} \otimes e_{\Phi(a),\Phi(c)\otimes m^{\otimes p}})
\circ
(e_{\Phi(a),\Phi(b)\otimes m^{\otimes n}}\otimes \id_{\Phi(c)\otimes m^{\otimes p}})
= 
e_{\Phi(a),\Phi(b\otimes c)\otimes m^{\otimes n+p}}.
$$
\end{lem}

\begin{proof}
We compute:
$$
\begin{tikzpicture}[baseline=.4cm]
	\draw (.2,-1)  .. controls ++(90:.2cm) and ++(270:.2cm) .. ($ (-.2,0) + (0,.2) $);
	\draw[super thick, white] (-.2,-1)  .. controls ++(90:.2cm) and ++(270:.2cm) ..($ (.2,0) + (0,.2) $);
	\draw (-.2,-1)  .. controls ++(90:.2cm) and ++(270:.2cm) ..($ (.2,0) + (0,.2) $);
	\draw (1.6,.8)  .. controls ++(90:.6cm) and ++(270:.6cm) .. ($ (.2,0) + (0,2) $);
	\draw[super thick, white] (.2,.8)  .. controls ++(90:.6cm) and ++(270:.6cm) ..($ (1.6,0) + (0,2) $);
	\draw (.2,.8)  .. controls ++(90:.6cm) and ++(270:.6cm) ..($ (1.6,0) + (0,2) $);
	\draw (2.2,1.5) -- (6.2,1.5);
	\draw (.8,-.5) -- (6.2,-.5);
	\draw (-.2,.2) -- (-.2,2);
	\draw (.2,.2) -- (.2,.8);
	\draw[super thick, white] (1.6,.8) -- (1.6,-1);
	\draw (1.6,.8) -- (1.6,-1);
	\node at (-.2,-1.2) {\scriptsize{$a$}};
	\node at (.2,-1.18) {\scriptsize{$b$}};
	\node at (1.6,-1.2) {\scriptsize{$c$}};
	\identityMap{(2.2,1.4)}{.4}{p\,\,}
	\identityMap{(.8,-.4)}{.4}{n\,\,}
	\tensorLeftId{(3.8,1.4)}{.8}{n}{p}{p}
	\tensorRightId{(3.8,-.4)}{.8}{n}{n}{p}
	\multiplication{(6.2,.5)}{1.2}{\;\scriptscriptstyle n{+}p}{\;\scriptscriptstyle n{+}p}{\;\scriptscriptstyle n{+}p}
	\draw (7.4,.5) -- (7.8,.5);
\end{tikzpicture}
=
\begin{tikzpicture}[baseline=.4cm]
	\draw (.2,0)  .. controls ++(90:.2cm) and ++(270:.2cm) .. ($ (-.2,0) + (0,1) $);
	\draw (0,0)  .. controls ++(90:.2cm) and ++(270:.2cm) .. ($ (-.4,0) + (0,1) $);
	\draw[super thick, white] (-.4,0)  .. controls ++(90:.2cm) and ++(270:.2cm) ..($ (.2,0) + (0,1) $);
	\draw (-.4,0)  .. controls ++(90:.2cm) and ++(270:.2cm) ..($ (.2,0) + (0,1) $);
	\node at (-.4,-.2) {\scriptsize{$a$}};
	\node at (0,-.18) {\scriptsize{$b$}};
	\node at (.2,-.2) {\scriptsize{$c$}};
	\draw (1.9,.5) -- (2.5,.5);
	\identityMap{(1.3,.5)}{.6}{\;\scriptscriptstyle n{+}p}
\end{tikzpicture}\,.
$$
\end{proof}

We can now define the functor $\Phi^{\scriptscriptstyle \cZ}:\cC \to \cZ(\cM)$.
It is given by $\Phi^{\scriptscriptstyle \cZ}(c) := (\Phi(c), e_{\Phi(c)})$ on objects, and
\[
\Phi^{\scriptscriptstyle \cZ}(f)\,:=\,\Phi(f) \,=\, 
\begin{tikzpicture}[baseline=-.1cm]
	\draw (0,-.8) -- (0,.8);
	\draw (1.4,0) -- (2.2,0);
	\node at (-.2,.62) {\scriptsize{$b$}};
	\node at (-.2,-.6) {\scriptsize{$a$}};
	\emptyMap{(1,0)}{.4}
	\node at (1.8,.2) {\scriptsize{$\cP[0]$}};
	\roundNbox{unshaded}{(0,0)}{.4}{0}{0}{$f$}
\end{tikzpicture}
\]
on morphisms. 
To see that $\Phi^{\scriptscriptstyle \cZ}(f):\Phi^{\scriptscriptstyle \cZ}(a)\to \Phi^{\scriptscriptstyle \cZ}(b)$ is a morphism in $\cZ(\cM)$, i.e. that the equation
\[
(\id\otimes f)\circ e_{\Phi(a),\Phi(c)\otimes m^{\otimes n}} = e_{\Phi(b),\Phi(c)\otimes m^{\otimes n}} \circ (f \otimes \id)
\]
holds for any $\Phi(c)\otimes m^{\otimes n}\in \cM_0$, we verify:
$$
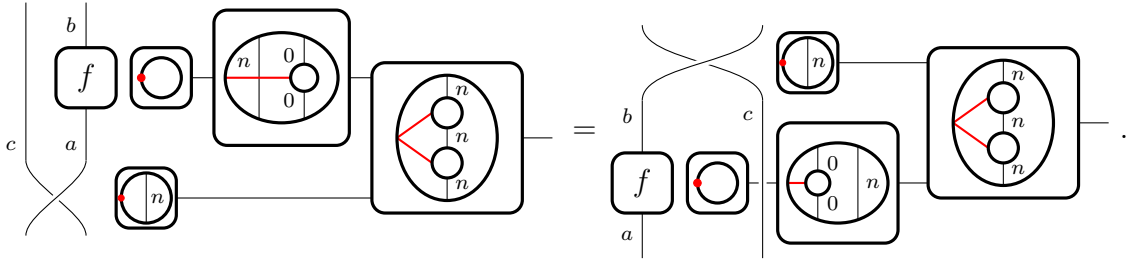
\begin{tikzpicture}[baseline=1.3cm]
	\draw (1.2,.5) -- (4.2,.5);
	\draw (1.8,2.1) -- (4.2,2.1);
	\draw (-.4,1) -- (-.4,3.1);
	\draw (.4,0)  .. controls ++(90:.3cm) and ++(270:.3cm) .. ($ (-.4,0) + (0,1) $);
	\draw[super thick, white] (-.4,0)  .. controls ++(90:.3cm) and ++(270:.3cm) ..($ (.4,0) + (0,1) $);
	\draw (-.4,0)  .. controls ++(90:.3cm) and ++(270:.3cm) ..($ (.4,0) + (0,1) $);
	\draw[super thick, white] (.4,1) -- (.4,3.1);
	\draw (.4,1) -- (.4,3.1);
	\node at (.2,1.2) {\scriptsize{$a$}};
	\node at (.2,2.8) {\scriptsize{$b$}};
	\node at (-.6,1.2) {\scriptsize{$c$}};
	\emptyMap{(1.4,2.1)}{.4}
	\tensorLeftId{(3,2.1)}{.9}{n}{0}{0}
	\identityMap{(1.2,.5)}{.4}{n\,\,}
	\roundNbox{unshaded}{(.4,2.1)}{.4}{0}{0}{$f$}
	\multiplication{(5.2,1.3)}{1}{n}{n}{n}
	\draw (6.2,1.3) -- (6.6,1.3);
\end{tikzpicture}
\,\,=\,
\begin{tikzpicture}[baseline=-.5cm]
	\draw (1.8,.5) -- (3.4,.5);
	\draw (.4,-1.1) -- (3.4,-1.1);
	\draw (-.4,0) -- (-.4,-2.1);
	\draw (1.2,0)  .. controls ++(90:.5cm) and ++(270:.5cm) .. ($ (-.4,0) + (0,1) $);
	\draw[super thick, white] (-.4,0)  .. controls ++(90:.5cm) and ++(270:.5cm) ..($ (1.2,0) + (0,1) $);
	\draw (-.4,0)  .. controls ++(90:.5cm) and ++(270:.5cm) ..($ (1.2,0) + (0,1) $);
	\draw[super thick, white] (1.2,0) -- (1.2,-2.1);
	\draw (1.2,0) -- (1.2,-2.1);
	\node at (1,-.2) {\scriptsize{$c$}};
	\node at (-.6,-1.8) {\scriptsize{$a$}};
	\node at (-.6,-.2) {\scriptsize{$b$}};
	\emptyMap{(.6,-1.1)}{.4}
	\tensorRightId{(2.2,-1.1)}{.8}{0}{0}{n}
	\identityMap{(1.8,.5)}{.4}{n\,\,}
	\roundNbox{unshaded}{(-.4,-1.1)}{.4}{0}{0}{$f$}
	\multiplication{(4.4,-.3)}{1}{n}{n}{n}
	\draw (5.4,-.3) -- (5.8,-.3);
\end{tikzpicture}\,\,.
$$
The equations $\Phi^{\scriptscriptstyle \cZ}(g\circ f) = \Phi^{\scriptscriptstyle \cZ}(g)\circ \Phi^{\scriptscriptstyle \cZ}(f)$ 
and $\Phi^{\scriptscriptstyle \cZ}(\id)=\id$ are trivial to check, and so we get a functor $\Phi^{\scriptscriptstyle \cZ}:\cC\to \cZ(\cM)$.

Our next task is to show that $\Phi^{\scriptscriptstyle \cZ}$ is a tensor functor.
Note that, by \eqref{Phi M eq3}, the functor $\Phi:\cC\to \cM_0$ is a strict tensor functor:
\[
\Phi(a\otimes b) = \Phi(a)\otimes \Phi(b),\quad\,\,\,\Phi(1) = 1.
\]
The half-braiding of $\Phi^{\scriptscriptstyle \cZ}(a\otimes b)$ is identical to that of $\Phi^{\scriptscriptstyle \cZ}(a)\otimes \Phi^{\scriptscriptstyle \cZ}(b)$:
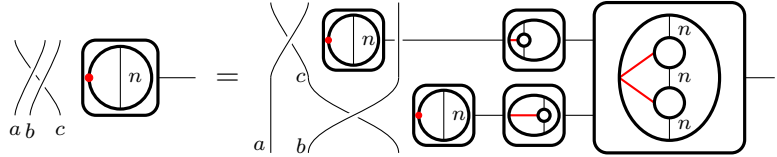
$$
\begin{tikzpicture}[baseline=.4cm]
	\draw (.2,0)  .. controls ++(90:.2cm) and ++(270:.2cm) .. ($ (-.4,0) + (0,1) $);
	\draw[super thick, white] (-.2,0)  .. controls ++(90:.2cm) and ++(270:.2cm) ..($ (.2,0) + (0,1) $);
	\draw (-.2,0)  .. controls ++(90:.2cm) and ++(270:.2cm) ..($ (.2,0) + (0,1) $);
	\draw[super thick, white] (-.4,0)  .. controls ++(90:.2cm) and ++(270:.2cm) ..($ (0,0) + (0,1) $);
	\draw (-.4,0)  .. controls ++(90:.2cm) and ++(270:.2cm) ..($ (0,0) + (0,1) $);
	\node at (-.4,-.2) {\scriptsize{$a$}};
	\node at (-.2,-.2) {\scriptsize{$b$}};
	\node at (.2,-.2) {\scriptsize{$c$}};
	\identityMap{(1,.5)}{.5}{n}
	\draw (1.5,.5) -- (2,.5);
\end{tikzpicture}
\,\,=\!
\begin{tikzpicture}[baseline=.9cm]
	\draw (.7,0)  .. controls ++(90:.4cm) and ++(270:.4cm) .. ($ (-.5,0) + (0,1) $);
	\draw[super thick, white] (-.5,0)  .. controls ++(90:.4cm) and ++(270:.4cm) ..($ (.7,0) + (0,1) $);
	\draw (-.5,0)  .. controls ++(90:.4cm) and ++(270:.4cm) ..($ (.7,0) + (0,1) $);
	\draw (-1,2)  .. controls ++(270:.2cm) and ++(90:.2cm) .. ($ (-.5,0) + (0,1) $);
	\draw[super thick, white] (-1,1)  .. controls ++(90:.2cm) and ++(270:.2cm) ..($ (-.5,0) + (0,2) $);
	\draw (-1,1)  .. controls ++(90:.2cm) and ++(270:.2cm) ..($ (-.5,0) + (0,2) $);
	\node at (-1.15,.1) {\scriptsize{$a$}};
	\node at (-.6,.1) {\scriptsize{$b$}};
	\node at (-.6,1) {\scriptsize{$c$}};
	\draw (0,1.5) -- (3.3,1.5);
	\draw (1.4,.5) -- (3.3,.5);
	\draw (-1,0) -- (-1,1);
	\draw[super thick, white] (.7,1) -- (.7,2);
	\draw (.7,1) -- (.7,2);
	\identityMap{(1.3,.5)}{.4}{n}
	\identityMap{(.1,1.5)}{.4}{n}
	\tensorLeftIdZero{(2.5,.5)}{.4}{}{}
	\tensorRightIdZero{(2.5,1.5)}{.4}{}{}
	\multiplication{(4.3,1)}{1}{n}{n}{n}
	\draw (5.3,1) -- (5.7,1);
\end{tikzpicture}
$$
and the half-braiding on $\Phi^{\scriptscriptstyle \cZ}(1)$ is trivial.
It follows that $\Phi^{\scriptscriptstyle \cZ}:\cC\to \cZ(\cM_0)$, and hence $\Phi^{\scriptscriptstyle \cZ}:\cC\to \cZ(\cM)$, are strict tensor functors.

\begin{lem}
The functor $\Phi^{\scriptscriptstyle \cZ}$ is braided.
\end{lem}
\begin{proof}
For all $a,b\in\cC$, we have
$
\Phi^{\scriptscriptstyle \cZ}(\beta_{a,b})
=\,
\begin{tikzpicture}[baseline=.4cm]
	\draw (.2,0)  .. controls ++(90:.2cm) and ++(270:.2cm) .. ($ (-.2,0) + (0,1) $);
	\draw[super thick, white] (-.2,0)  .. controls ++(90:.2cm) and ++(270:.2cm) ..($ (.2,0) + (0,1) $);
	\draw (-.2,0)  .. controls ++(90:.2cm) and ++(270:.2cm) ..($ (.2,0) + (0,1) $);
	\node at (-.35,.1) {\scriptsize{$a$}};
	\node at (-.35,.9) {\scriptsize{$b$}};
	\emptyMap{(.8,.5)}{.4}
	\draw (1.2,.5) -- (1.6,.5);
\end{tikzpicture}
\,=\,
e_{\Phi(a),\Phi(b)}
$.
\end{proof}

\begin{lem}
The functor $\Phi^{\scriptscriptstyle \cZ}$ is compatible with the twists: $\theta_{\Phi^{\scriptscriptstyle \cZ}(a)}=\Phi^{\scriptscriptstyle \cZ}(\theta_a)$.
\end{lem}
\begin{proof}
We have
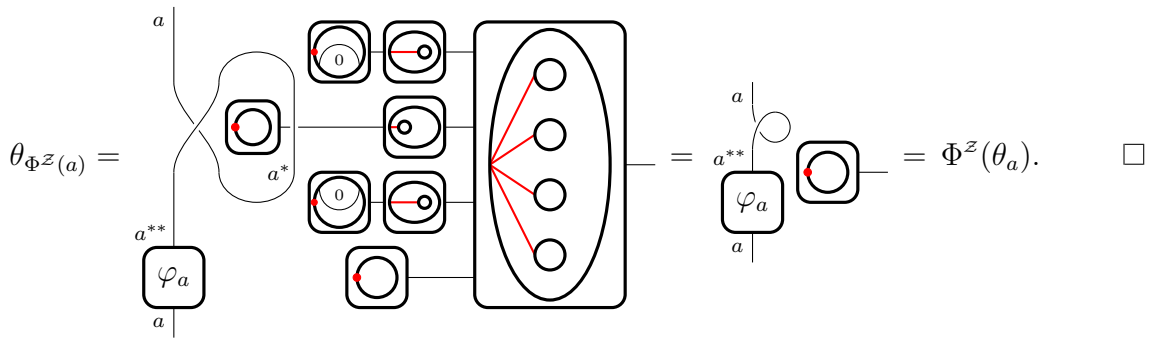
\begin{align*}
\theta_{\Phi^{\scriptscriptstyle \cZ}(a)} 
&=\!
\begin{tikzpicture}[baseline=-.5cm]
	\draw (1.5,-1) -- (3.4,-1);
	\draw (1.5,1) -- (3.4,1);
	\draw (2,-2) -- (3.4,-2);
	\draw (.5,0) -- (3.4,0);
	\draw[super thick, white] (-.6,1.6) -- (-.6,.6);
	\draw (-.6,1.6) -- (-.6,.6);
	\draw[super thick, white] (-.6,-2.8) -- (-.6,-.6);
	\draw (-.6,-2.8) -- (-.6,-.6);
	\draw (0,-.6)  .. controls ++(90:.4cm) and ++(270:.4cm) .. ($ (-.6,-.6) + (0,1.2) $);
	\draw[super thick, white] (-.6,-.6)  .. controls ++(90:.4cm) and ++(270:.4cm) ..($ (0,-.6) + (0,1.2) $);
	\draw (-.6,-.6)  .. controls ++(90:.4cm) and ++(270:.4cm) ..($ (0,-.6) + (0,1.2) $);
	\draw[super thick, white] (0,-.6) arc (-180:0:.5cm and .4cm) -- (1,.6) arc (0:180:.5cm and .4cm);
	\draw (0,-.6) arc (-180:0:.5cm and .4cm) -- (1,.6) arc (0:180:.5cm and .4cm);
	\node at (.8,-.6) {\scriptsize{$a^*$}};
	\node at (-.9,-1.4) {\scriptsize{$a^{**}$}};
	\node at (-.8,1.4) {\scriptsize{$a$}};
	\node at (-.8,-2.6) {\scriptsize{$a$}};
	\roundNbox{unshaded}{(-.6,-2)}{.4}{0}{0}{$\varphi_{a}$}
	\evaluationMap{(1.6,1)}{.4}{\scriptstyle 0}
	\coevaluationMap{(1.6,-1)}{.4}{\scriptstyle 0}
	\emptyMap{(2.1,-2)}{.4}
	\emptyMap{(.45,0)}{.35}
	\tensorRightIdZeroZero{(2.6,0)}{.4}
	\tensorLeftIdZeroZero{(2.6,1)}{.4}
	\tensorLeftIdZeroZero{(2.6,-1)}{.4}
	\draw[rounded corners=5pt, very thick, unshaded] (3.4,1.4) rectangle (5.4,-2.4);
	\coordinate (a) at (3.6,-.5);
	\draw[thick, red] (a) -- (4.2,-1.7);
	\draw[thick, red] (a) -- (4.2,-.9);
	\draw[thick, red] (a) -- (4.2,-.1);
	\draw[thick, red] (a) -- (4.2,.7);
	\draw (5.4,-.5) -- (5.8,-.5);
	\draw[very thick] (4.4,-.5) ellipse (.8 and 1.8);
	\filldraw[very thick, unshaded] (4.4,-1.7) circle (.2cm);
	\filldraw[very thick, unshaded] (4.4,-.9) circle (.2cm);
	\filldraw[very thick, unshaded] (4.4,-.1) circle (.2cm);
	\filldraw[very thick, unshaded] (4.4,.7) circle (.2cm);
\end{tikzpicture}
\,=
\begin{tikzpicture}[baseline=.5cm]
	\draw (0,1.6) -- (0,-.8);
	\loopIsoInverseReverse{(0,1)}
	\roundNbox{unshaded}{(0,0)}{.4}{0}{0}{$\varphi_a$}	
	\node at (-.2,-.6) {\scriptsize{$a$}};
	\node at (-.2,1.4) {\scriptsize{$a$}};
	\node at (-.3,.6) {\scriptsize{$a^{**}$}};
	\emptyMap{(1,.4)}{.4}
	\draw (1.4,.4) -- (1.8,.4);
\end{tikzpicture}
\,=\,
\Phi^{\scriptscriptstyle \cZ}(\theta_a).
\qedhere
\end{align*}
\end{proof}

This finishes the construction of the module tensor category $\cM$.
We sum up the results of this section in the following theorem:

\begin{thm}
\label{thm:ConstructCMG}
Given an anchored planar algebra $\cP$ in a braided pivotal category $\cC$, we can construct a pair $(\cM,m)$ where $\cM$ is a pivotal module tensor category, and $m\in \cM$ is a symmetrically self-dual object.
 Moreover, the functor $\Phi=F\circ \Phi^{\scriptscriptstyle \cZ}:\cC\to \cM$ admits a right adjoint $\Tr_\cC:\cM\to\cC$.
\end{thm}

We denote by $\Delta(\cP)$ the pointed module tensor category associated to the anchored planar algebra $\cP$ by means of the above construction.

\subsection{Functoriality}
\label{sec:APA_Functoriality}
In Theorem~\ref{thm:ConstructCMG}, given an anchored planar algebra,
we constructed a pointed module tensor category $\Delta(\cP)$.
Our next goal is to upgrade this to a functor
\[
\Delta : \APA \to \Mod_*
\]
between the category of anchored planar algebras in $\cC$ and that of pointed pivotal module tensor categories over $\cC$. 

Let $\cP_1$ and $\cP_2$ be anchored planar algebras, and let $\cM_1:=\Delta(\cP_1)$ and $\cM_2:=\Delta(\cP_2)$ be the associated pointed module tensor categories.
Given a morphism $H: \cP_1\to \cP_2$, we describe the associated functor 
$\Delta(H)=(\Delta(H),\gamma^{\Delta(H)}):\cM_1\to \cM_2$. 
It is enough to give $\Delta(H)$ on the full subcategory of $\cM_1$ with objects of the form $\Phi_1(c)\otimes m_1^{\otimes n}$.
We let
\begin{equation}\label{eq:def of Delta -- objects}
\Delta(H)(\Phi_1(c)\otimes m_1^{\otimes n}) := \Phi_2(c)\otimes m_2^{\otimes n}
\end{equation}
and $(\gamma^{\Delta(H)})_c:=\id: \Phi_2(c) \to G(\Phi_1(c))$.
The image of a morphism $f\in \cM_1(\Phi(a)\otimes m_1^{\otimes k}, \Phi(b)\otimes m_1^{\otimes n}) := \cC(a, b\otimes \cP[n+k])$ is then given by
$\Delta(H)(f):=(\id_b\otimes {H[n+k]})\circ f$:
\begin{equation}\label{eq:def of Delta}
\Delta(H)(f) =
\Delta(H)\left(
\begin{tikzpicture}[baseline=-.1cm]
	\draw (0,.8) -- (0,-.8);
	\draw (0,0) -- (2,0);
	\roundNbox{unshaded}{(0,0)}{.4}{0}{0}{$f$};
	\node at (-.2,.6) {\scriptsize{$b$}};
	\node at (1.2,.2) {\scriptsize{$\cP_1[n+k]$}};
	\node at (-.2,-.6) {\scriptsize{$a$}};
\end{tikzpicture}
\right)
:=\,
\begin{tikzpicture}[baseline=-.1cm]
	\draw (0,.8) -- (0,-.8);
	\draw (0,0) -- (3.8,0);
	\roundNbox{unshaded}{(2.1,0)}{.3}{0}{0}{$H$};
	\roundNbox{unshaded}{(0,0)}{.4}{0}{0}{$f$};
	\node at (-.2,.6) {\scriptsize{$b$}};
	\node at (1.1,.2) {\scriptsize{$\cP_1[n+k]$}};
	\node at (3.1,.2) {\scriptsize{$\cP_2[n+k]$}};
	\node at (-.2,-.6) {\scriptsize{$a$}};
\end{tikzpicture}
\,.
\end{equation}
To see that $\Delta(H)$ is a functor, we use that $H$ is a map of anchored planar algebras, and thus commutes with the action of tangles:
\begin{align*}
&\Delta(H)(g\circ f)
=\,
\begin{tikzpicture}[baseline=-.1cm, scale=.9]
	\draw (0,1.4) -- (0,-1.4);
	\draw (.4,.6) -- (1,.6);
	\draw (.4,-.6) -- (1,-.6);
	\draw (3.4,0) -- (4.8,0);
	\roundNbox{unshaded}{(4.2,0)}{.35}{0}{0}{$H$};
	\roundNbox{unshaded}{(0,.6)}{.4}{0}{0}{$g$};
	\roundNbox{unshaded}{(0,-.6)}{.4}{0}{0}{$f$};
	\multiplication{(2.2,0)}{1.2}{}{}{}
\end{tikzpicture}
\,=
\begin{tikzpicture}[baseline=-.1cm, scale=.9]
	\draw (0,1.4) -- (0,-1.4);
	\draw (.4,.6) -- (2,.6);
	\draw (.4,-.6) -- (2,-.6);
	\draw (4,0) -- (4.4,0);
	\roundNbox{unshaded}{(1,.6)}{.35}{0}{0}{$H$};
	\roundNbox{unshaded}{(1,-.6)}{.35}{0}{0}{$H$};
	\roundNbox{unshaded}{(0,.6)}{.4}{0}{0}{$g$};
	\roundNbox{unshaded}{(0,-.6)}{.4}{0}{0}{$f$};
	\multiplication{(2.8,0)}{1.2}{}{}{}
\end{tikzpicture}
=
\Delta(H)(g)\circ \Delta(H)(f)
\\
&\Delta(H)(\id_{\Phi_1(c)\otimes m_1^{\otimes n}}) 
=
\begin{tikzpicture}[baseline=-.1cm]
	\draw (.3,-.8) -- (.3,.8);
	\draw (1.5,0) -- (4.1,0);
	\identityMap{(1,0)}{.5}{n\,\,}
	\node at (.1,0) {\scriptsize{$c$}};
	\node at (2,.2) {\scriptsize{$\cP_1[2n]$}};
	\node at (3.6,.2) {\scriptsize{$\cP_2[2n]$}};
	\roundNbox{unshaded}{(2.8,0)}{.3}{0}{0}{$H$};
\end{tikzpicture}
=
\begin{tikzpicture}[baseline=-.1cm]
	\draw (.3,-.8) -- (.3,.8);
	\draw (1.5,0) -- (2.5,0);
	\identityMap{(1,0)}{.5}{n\,\,}
	\node at (.1,0) {\scriptsize{$c$}};
	\node at (2,.2) {\scriptsize{$\cP_2[2n]$}};
\end{tikzpicture}
=
\id_{\Phi_2(c)\otimes m_2^{\otimes n}}.
\end{align*}
The functor $\Delta(H)$ is a tensor functor because $\Delta(H)(1_{\cM_1}) = \Delta(H)(\Phi_1(1_\cC))=\Phi_2(1_\cC)=1_{\cM_2}$ and
\[
\begin{split}
\Delta(H)\big((\Phi_1(c)\otimes m_1^{\otimes n}) &\otimes (\Phi_1(d)\otimes m_1^{\otimes n'})\big) = \Delta(H)\big(\Phi_1(c\otimes d)\otimes m_1^{\otimes n+n'}\big)
\\&=\Phi_2(c\otimes d)\otimes m_2^{\otimes n+n'}
=(\Phi_2(c)\otimes m_2^{\otimes n}) \otimes (\Phi_2(d)\otimes m_2^{\otimes n'})
\\&=\Delta(H)(\Phi_1(c)\otimes m_1^{\otimes n}) \otimes \Delta(H)(\Phi_1(d)\otimes m_1^{\otimes n'}),
\end{split}
\]
and because:
\begin{align*}
\Delta(H)(f\otimes g)
&=
\begin{tikzpicture}[baseline=-.6cm]
	\draw (0,0) -- (2,0);
	\draw (-1,-1) -- (2,-1);
	\draw (3.4,-.5) -- (5,-.5);
	\draw (-1,.8) -- (-1,-1.8);
	\draw[super thick, white] (0,.8) -- (0,-1.8);
	\draw (0,.8) -- (0,-1.8);
	\roundNbox{unshaded}{(4.2,-.5)}{.3}{0}{0}{$H$}
	\roundNbox{unshaded}{(-1,-1)}{.4}{0}{0}{$f$}
	\roundNbox{unshaded}{(0,0)}{.4}{0}{0}{$g$}
	\tensor{(2.2,-.5)}{1.2}{}{}{}{}
\end{tikzpicture}
\displaybreak[1]
\\&=
\begin{tikzpicture}[baseline=-.6cm]
	\draw (0,0) -- (2,0);
	\draw (-1,-1) -- (2,-1);
	\draw (4.4,-.5) -- (5,-.5);
	\draw (-1,.8) -- (-1,-1.8);
	\draw[super thick, white] (0,.8) -- (0,-1.8);
	\draw (0,.8) -- (0,-1.8);
	\roundNbox{unshaded}{(1.2,0)}{.3}{0}{0}{$H$}
	\roundNbox{unshaded}{(1.2,-1)}{.3}{0}{0}{$H$}
	\roundNbox{unshaded}{(-1,-1)}{.4}{0}{0}{$f$}
	\roundNbox{unshaded}{(0,0)}{.4}{0}{0}{$g$}
	\tensor{(3.2,-.5)}{1.2}{}{}{}{}
\end{tikzpicture}
\displaybreak[1]
\\&=
\begin{tikzpicture}[baseline=-.6cm]
	\draw (.5,0) -- (2,0);
	\draw (-1,-1) -- (2,-1);
	\draw (4.4,-.5) -- (5,-.5);
	\draw (-1,.8) -- (-1,-1.8);
	\draw[super thick, white] (.5,.8) -- (.5,-1.8);
	\draw (.5,.8) -- (.5,-1.8);
	\roundNbox{unshaded}{(1.4,0)}{.3}{0}{0}{$H$}
	\roundNbox{unshaded}{(-.1,-1)}{.3}{0}{0}{$H$}
	\roundNbox{unshaded}{(-1,-1)}{.4}{0}{0}{$f$}
	\roundNbox{unshaded}{(.5,0)}{.4}{0}{0}{$g$}
	\tensor{(3.2,-.5)}{1.2}{}{}{}{}
\end{tikzpicture}
=
\Delta(H)(f)\otimes \Delta(H)(g).
\displaybreak[1]
\end{align*}
It is pivotal because
$$
\Delta(H)(\varphi_{\Phi_1(c)\otimes m_1^{\otimes n}})
\,=\!
\begin{tikzpicture}[baseline=-.1cm]
	\draw (0,-1) -- (0,1);
	\draw (1.4,0) -- (4.3,0);
	\roundNbox{unshaded}{(0,0)}{.4}{0}{0}{$\varphi_c$}
	\roundNbox{unshaded}{(2.8,0)}{.3}{0}{0}{$H$}
	\node at (-.2,-.8) {\scriptsize{$c$}};
	\node at (-.3,.8) {\scriptsize{$c^{**}$}};
	\identityMap{(1,0)}{.4}{n\,\,\,}
	\node at (1.9,.2) {\scriptsize{$\cP_1[2n]$}};
	\node at (3.7,.2) {\scriptsize{$\cP_2[2n]$}};
\end{tikzpicture}
\,=\!
\begin{tikzpicture}[baseline=-.1cm]
	\draw (0,-1) -- (0,1);
	\roundNbox{unshaded}{(0,0)}{.4}{0}{0}{$\varphi_c$}
	\draw (1.4,0) -- (2.4,0);
	\node at (-.2,-.8) {\scriptsize{$c$}};
	\node at (-.3,.8) {\scriptsize{$c^{**}$}};
	\identityMap{(1,0)}{.4}{n\,\,\,}
	\node at (1.9,.2) {\scriptsize{$\cP_2[2n]$}};
\end{tikzpicture}
=
\varphi_{\Phi_2(c)\otimes m_2^{\otimes n}}.
$$
$\Delta(H)$ is compatible with the half-braidings, and thus a functor of module tensor categories (Definition~\ref{defn:ModuleTensorCategoryFunctor}):
$$
\Delta(H)(e_{\Phi_1(c),\Phi_1(a)\otimes m_1^{\otimes n}})
=\!\!
\begin{tikzpicture}[baseline=.4cm]
	\draw (.2,0)  .. controls ++(90:.2cm) and ++(270:.2cm) .. ($ (-.2,0) + (0,1) $);
	\draw[super thick, white] (-.2,0)  .. controls ++(90:.2cm) and ++(270:.2cm) ..($ (.2,0) + (0,1) $);
	\draw (-.2,0)  .. controls ++(90:.2cm) and ++(270:.2cm) ..($ (.2,0) + (0,1) $);
	\node at (-.35,.1) {\scriptsize{$c$}};
	\node at (-.35,.9) {\scriptsize{$a$}};
	\identityMap{(1,.5)}{.5}{n\,\,}
	\draw (1.5,.5) -- (4.5,.5);
	\roundNbox{unshaded}{(3,.5)}{.3}{0}{0}{$H$}
	\node at (2.1,.7) {\scriptsize{$\cP_1[2n]$}};
	\node at (3.9,.7) {\scriptsize{$\cP_2[2n]$}};
\end{tikzpicture}
=\!\!
\begin{tikzpicture}[baseline=.4cm]
	\draw (.2,0)  .. controls ++(90:.2cm) and ++(270:.2cm) .. ($ (-.2,0) + (0,1) $);
	\draw[super thick, white] (-.2,0)  .. controls ++(90:.2cm) and ++(270:.2cm) ..($ (.2,0) + (0,1) $);
	\draw (-.2,0)  .. controls ++(90:.2cm) and ++(270:.2cm) ..($ (.2,0) + (0,1) $);
	\node at (-.35,.1) {\scriptsize{$c$}};
	\node at (-.35,.9) {\scriptsize{$a$}};
	\identityMap{(1,.5)}{.5}{n\,\,}
	\draw (1.5,.5) -- (2.4,.5);
	\node at (2,.7) {\scriptsize{$\cP_2[2n]$}};
\end{tikzpicture}
=
e_{\Phi_2(c),\Phi_2(a)\otimes m_2^{\otimes n}}
\,.
$$
Finally, it is a functor of pointed module tensor categories (Definition~\ref{def: pointed})
as it sends $m_1$ to $m_2$ and the condition $\psi_2=\delta_{m_1}\circ\Delta(H)(\psi_1)$ holds true (both sides are identity morphisms).

It remains to show that the assignement $H\mapsto\Delta(H)$ is a functor, i.e., that the relations $\Delta(H_1\circ H_2)=\Delta(H_1)\circ \Delta(H_2)$ and $\Delta(\id)=\id$ are satisfied.
Both are obvious from the definition \eqref{eq:def of Delta} of $\Delta$. We have shown:

\begin{thm}\label{thm: Here's Delta!}
Let $\cC$ be a braided pivotal category which admits finite direct sums and is idempotent complete. Then the map
\[
\Delta:\APA \to \Mod_*
\]
given by Theorem~\ref{thm:ConstructCMG} at the level of objects and by \eqref{eq:def of Delta -- objects} and \eqref{eq:def of Delta} at the level of morphisms is a functor
from the category of anchored planar algebras in $\cC$ to the (2-)category of pointed pivotal module tensor categories over $\cC$.
\end{thm}

Note that we may compose $\Delta$ with the projection $\Mod_*\xrightarrow{\,\scriptscriptstyle\simeq\,}\tau_{\le 1}(\Mod_*)$ to get a functor $\APA\to\tau_{\le 1}(\Mod_*)$.


\section{Equivalence of categories}\label{sec:Bijective}
\label{sec:Equivalence}

In this section, we complete the proof of Theorem \ref{thm:EquivalenceOfCategories2} and show that the functors
\[
\Lambda\,:\,\, \Mod_* \,\rightleftarrows\, \APA\,\,:\,\Delta
\]
witness an equivalence of (2-)categories.
In view of Lemma~\ref{lem: it's secretly a 1-category}, this is equivalent to
the functors $\Lambda: \tau_{\le1}(\Mod_*)\rightleftarrows \APA:\Delta$ witnessing an equivalence of ordinary categories.

The proof is in two steps.
In Sections \ref{sec:CMGtoAPAtoCMG} and \ref{sec:CMGtoAPAtoCMG_Naturality}
we construct a natural isomorphism $\Delta\circ\Lambda \Rightarrow \id$, and in Section \ref{sec:APAtoCMGtoAPA}
we show that $\Lambda\circ\Delta = \id$. 

\subsection{Module tensor categories to planar algebras and back}\label{sec:CMGtoAPAtoCMG}

Let $(\cM,m)$ be a pointed module tensor category, and let $(\cM',m')$ be its image under $\Delta\circ\Lambda$.
We wish to construct an equivalence
\[
\Psi = (\Psi, \gamma^\Psi) : (\cM',m') \to (\cM,m)
\]
(an isomorphism in $\tau_{\le1}(\Mod_*)$).

Let $\Phi : \cC\to \cM$ and $\Phi' :\cC\to \cM'$ be the action functors of $\cM$ and of $\cM'$.
Since $\cM$ admits direct sums and is idempotent complete, it is enough to define $\Psi$
on the full subcategory $\cM'_0\subset \cM'$ with objects $\Phi'(c)\otimes m'^{\otimes n}$.
At the level of objects, the functor $\Psi:\cM'\to\cM$ is given by:
\begin{equation}\label{eq: def Psi -- objects}
\Psi (\Phi'(c)\otimes m'^{\otimes n}):=\Phi(c)\otimes m^{\otimes n}.
\end{equation}
At the level of morphisms, the image of
\[
f\in \cM'(\Phi'(a)\otimes m'^{\otimes i}, \Phi'(b)\otimes m'^{\otimes j})=\cC(a, b\otimes \cP[j+i])=\cC(a, b\otimes \Tr_\cC(m^{\otimes j+i}))
\]
under $\Psi$ is given by:
\begin{equation}\label{eq:Psi of f}
\begin{split}
\Psi(f)&:=\,\,\,
\begin{tikzpicture}[baseline=0cm, scale=.8]
	\plane{(-.4,-.9)}{2.8}{1.8}
	\node at (-2.4,1.7) {\scriptsize{$a$}};
	\node at (1,1.7) {\scriptsize{$b$}};
	\draw[thick, red] (-2.8,1.5) -- (-1.2,1.5);
	\draw[thick, blue] (-1.2,1.5) -- (2,1.5);
	\coordinate (a) at (0,-.5);
	\draw ($ (a) + (.8,.4) $)  -- ($ (a) + (1.6,.4) $) ;
	\draw ($ (a) + (.8,.2) $)  arc (90:-90:.2cm) -- ($ (a) + (-.6,-.2) $) ;
	\node at ($ (a) + (1.1,.6) $)  {\scriptsize{$j$}};
	\node at ($ (a) + (1.2,0) $)  {\scriptsize{$i$}};
	\CMbox{box}{(a)}{.8}{.8}{.4}{$\varepsilon$}
	\curvedTubeNoString{($ (a) + (-.3,.4) $)}{1}{0}{$f$}
\end{tikzpicture}
\\&=\,
(\id_{\Phi(b)}\otimes \id_{m^{\otimes j}} \otimes\;\! \bar\ev_{m^{\otimes i}}) 
\circ 
(\id_{\Phi(b)}\otimes \varepsilon_{m^{\otimes i+j}} \otimes\id_{m^{\otimes i}})
\circ 
(\Phi(f)\otimes \id_{m^{\otimes i}}),
\end{split}
\end{equation}
where we have written $\cP$ for $\Lambda(\cM)$.
The action coherence isomorhpisms $(\gamma^\Psi)_c:\Phi(c)\to\Psi\Phi'(c)$ are all identity morphisms.

Note that, at this point, it is not obvious that $\Psi$ is a functor.
Nevertheless, it is easy to see that, if $\Psi$ is a functor, then it is actually an equivalence of categories:

\begin{lem}\label{lem: yes, it's an equivalence of categories}
If $\Psi:\cM'\to\cM$ is a functor, then it is an equivalence of categories.
\end{lem}

\begin{proof}
We first show that $\Psi|_{\cM'_0}$ is fully faithful.
This holds true because the image $\Psi(f)$ of a morphism $f$ is also its image under the following sequence of isomorphisms:
\begin{align*}
\cM'(\Phi'(a)\otimes m'^{\otimes k}, \Phi'(b)\otimes m'^{\otimes n})
&=
\cC(a, b\otimes \cP[n+k]) 
\\&=
\cC(a, b\otimes \Tr_\cC(m^{\otimes n+k}))
\\&\cong 
\cC(b^*\otimes a , \Tr_\cC(m^{\otimes n+k})) 
\\&\cong 
\cM(\Phi(b^*\otimes a), m^{\otimes n+k}) 
\\&\cong 
\cM(\Phi(b)^* \otimes \Phi(a), m^{\otimes n}\otimes m^{\otimes k})
\\&\cong 
\cM(\Phi(a)\otimes m^{\otimes k} , \Phi(b)\otimes m^{\otimes n}).
\end{align*}
This implies that $\Psi$ is also fully faithful.
Finally, the functor $\Psi$ is essentially surjective because $m$ generates $\cM$ (Definition~\ref{def: pointed}).
\end{proof}

We now show that $\Psi$ is a functor:

\begin{lem}
\label{lem:ThetaPreservesComposition}
$\Psi$ preserves composition of morphisms:\, $\Psi(g)\circ\Psi(f)=\Psi(g\circ f)$.
\end{lem}
\begin{proof}
It is enough to check this condition when $f$ and $g$ are morphisms in $\cM_0'$.
Let $f:\Phi'(a)\otimes m'^{\otimes i} \to \Phi'(b)\otimes m'^{\otimes j}$
and 
$g: \Phi'(b)\otimes m'^{\otimes j} \to \Phi'(c)\otimes m'^{\otimes k}$ be composable morphisms.
Then
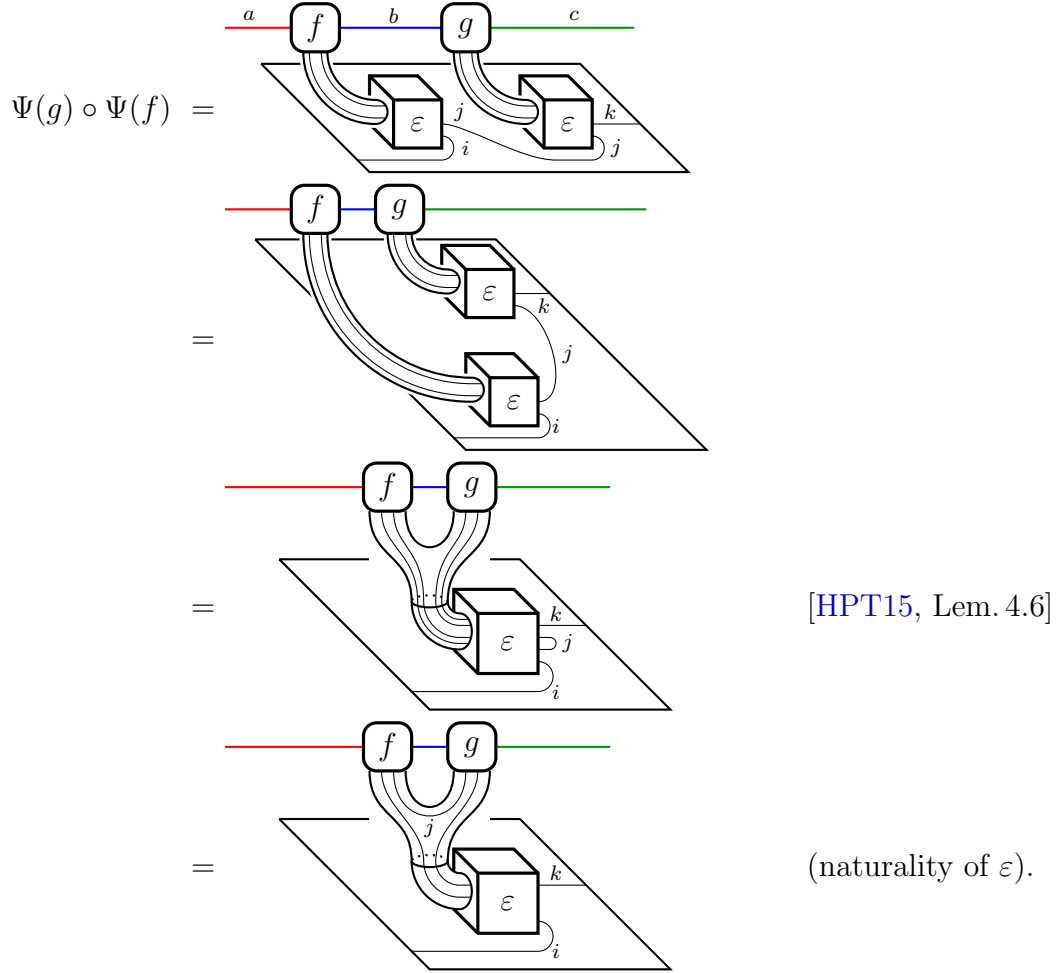
\begin{align*}
\Psi(g)\circ \Psi(f)\,\,
&=
\begin{tikzpicture}[baseline=0cm, scale=.8]
	\plane{(-.4,-.9)}{5.3}{1.8}
	\node at (-2.4,1.7) {\scriptsize{$a$}};
	\node at (0,1.7) {\scriptsize{$b$}};
	\node at (3,1.7) {\scriptsize{$c$}};
	\draw[thick, red] (-2.8,1.5) -- (-1.2,1.5);
	\draw[thick, blue] (-1.2,1.5) -- (1.6,1.5);
	\draw[thick, DarkGreen] (.8,1.5) -- (4,1.5);
	\coordinate (a) at (0,-.5);
	\draw ($ (a) + (.8,.2) $)  arc (90:-90:.2cm) -- ($ (a) + (-.6,-.2) $) ;
	\node at ($ (a) + (1.1,.6) $)  {\scriptsize{$j$}};
	\node at ($ (a) + (1.2,0) $)  {\scriptsize{$i$}};
	\CMbox{box}{(a)}{.8}{.8}{.4}{$\varepsilon$}
	\curvedTubeNoString{($ (a) + (-.3,.4) $)}{1}{0}{$f$}
	\coordinate (b) at (2.5,-.5);
	\draw ($ (b) + (.8,.4) $)  -- ($ (b) + (1.6,.4) $) ;
	\node at ($ (b) + (1.1,.6) $)  {\scriptsize{$k$}};
	\node at ($ (b) + (1.2,0) $)  {\scriptsize{$j$}};
	\CMbox{box}{(b)}{.8}{.8}{.4}{$\varepsilon$}
	\curvedTubeNoString{($ (b) + (-.3,.4) $)}{1}{0}{$g$}
	\draw ($ (a) + (.8,.4) $) .. controls ++(0:.4cm) and ++(180:.6cm) .. ($ (b) + (.2,-.2) $) -- ($ (b) + (.8,-.2) $) arc (-90:90:.2cm);
\end{tikzpicture}
\displaybreak[1]
\\&=
\begin{tikzpicture}[baseline=1.2cm, scale=.8]
	\plane{(0,-.2)}{4}{3.5}
	\draw[thick, red] (-4,3.8) -- (-2.4,3.8);
	\draw[thick, blue] (-2.4,3.8) -- (-1,3.8);
	\draw[thick, DarkGreen] (-1,3.8) -- (3,3.8);
	\coordinate (a) at (.4,.2);
	\CMbox{box}{(a)}{.8}{.8}{.4}{$\varepsilon$}
	\curvedTubeNoString{($ (a) + (-.3,.4) $)}{2.6}{0}{$f$}
	\coordinate (b) at (0,2);
	\CMbox{box}{(b)}{.8}{.8}{.4}{$\varepsilon$}
	\curvedTubeNoString{($ (b) + (-.3,.4) $)}{.8}{0}{$g$}
	\draw ($ (a) + (.8,.2) $) arc (90:-90:.2cm) -- ($ (a) + (-.6,-.2) $) ;
	\draw ($ (b) + (.8,.4) $)  -- ($ (b) + (1.4,.4) $) ;
	\draw ($ (a) + (.8,.4) $)  .. controls ++(0:.6cm) and ++(0:.6cm) .. ($ (b) + (.8,.2) $) ;
	\node at ($ (a) + (1.1,0) $)  {\scriptsize{$i$}};
	\node at ($ (a) + (1.3,1.2) $)  {\scriptsize{$j$}};
	\node at ($ (b) + (1.3,.2) $)  {\scriptsize{$k$}};
\end{tikzpicture}
\displaybreak[1]
\\&=
\begin{tikzpicture}[scale=.8, baseline = -.6cm]
	\plane{(1,-2.3)}{4}{2.5}
	\CMbox{box}{(1.8,-1.7)}{1}{1}{.4}{$\varepsilon$}
	\draw[thick, unshaded] (1.3,-.5) arc (180:270:.2cm) arc (90:-90: .2 cm and .3 cm) arc (270:180:.8cm);
	\fill[white] (0,0) rectangle (2,1);
	\invertedPairOfPants{(0,1)}{};
	\draw[thick, red] (-2.4,1.4) -- (.3,1.4);
	\draw[thick, blue] (.3,1.4) -- (1.7,1.4);
	\draw[thick, DarkGreen] (1.7,1.4) -- (4,1.4);
	\roundNbox{unshaded}{(.3,1.4)}{.4}{0}{0}{$f$}
	\roundNbox{unshaded}{(1.7,1.4)}{.4}{0}{0}{$g$}
	\draw (2.8,-1.5) arc (90:-90:.25cm) -- (.7,-2);
	\draw (2.8,-1.3) -- (3,-1.3) arc (-90:90:.1cm) -- (2.8,-1.1);
	\draw (2.8, -.9) -- (3.6,-.9);
	\node at (3.1,-2)  {\scriptsize{$i$}};
	\node at (3.3,-1.2)  {\scriptsize{$j$}};
	\node at (3.1,-.7)  {\scriptsize{$k$}};
	\draw (.2,1) .. controls ++(270:.6cm) and ++(90:.8cm) .. (.8,-.5) arc (180:270:.7cm) -- (1.65,-1.2);
	\draw (.4,1) .. controls ++(270:.6cm) and ++(90:.9cm) .. (.9,-.5) arc (180:270:.6cm) -- (1.7,-1.1);
	\draw (1.6,1) .. controls ++(270:.6cm) and ++(90:.9cm) .. (1.1,-.5) arc (180:270:.4cm) -- (1.7,-.9);
	\draw (1.8,1) .. controls ++(270:.6cm) and ++(90:.8cm) .. (1.2,-.5) arc (180:270:.3cm) -- (1.65,-.8);
\end{tikzpicture}
&&
\text{\cite[Lem.\,4.6]{1509.02937}}
\displaybreak[1]
\\&=
\begin{tikzpicture}[scale=.8, baseline = -.6cm]
	\plane{(1,-2.3)}{4}{2.5}
	\CMbox{box}{(1.8,-1.7)}{1}{1}{.4}{$\varepsilon$}
	\draw[thick, unshaded] (1.3,-.5) arc (180:270:.2cm) arc (90:-90: .2 cm and .3 cm) arc (270:180:.8cm);
	\fill[white] (0,0) rectangle (2,1);
	\invertedPairOfPants{(0,1)}{};
	\draw[thick, red] (-2.4,1.4) -- (.3,1.4);
	\draw[thick, blue] (.3,1.4) -- (1.7,1.4);
	\draw[thick, DarkGreen] (1.7,1.4) -- (4,1.4);
	\roundNbox{unshaded}{(.3,1.4)}{.4}{0}{0}{$f$}
	\roundNbox{unshaded}{(1.7,1.4)}{.4}{0}{0}{$g$}
	\draw (2.8,-1.5) arc (90:-90:.25cm) -- (.7,-2);
	\draw (2.8, -.9) -- (3.6,-.9);
	\node at (3.1,-2)  {\scriptsize{$i$}};
	\node at (1,.05)  {\scriptsize{$j$}};
	\node at (3.1,-.7)  {\scriptsize{$k$}};
	\draw (.2,1) .. controls ++(270:.6cm) and ++(90:.6cm) .. (.9,-.5) arc (180:270:.6cm) -- (1.7,-1.1);
	\draw (.4,1) .. controls ++(270:1cm) and ++(270:1cm) .. (1.6,1);
	\draw (1.8,1) .. controls ++(270:.6cm) and ++(90:.6cm) .. (1.1,-.5) arc (180:270:.4cm) -- (1.7,-.9);
\end{tikzpicture}
&&
\text{(naturality of $\varepsilon$).}
\displaybreak[1]
\end{align*}
We need to show that this is equal to the image of
\newcommand{\multiplicationforhere}[5]{
	\draw ($ #1 + 5/6*(0,#2) $) -- ($ #1 - 5/6*(0,#2) $);
	\draw[thick, red] ($ #1 + 1/3*(0,#2) - 1/5*(#2,0) $) -- ($ #1 - 4/5*(#2,0) $);
	\draw[thick, red] ($ #1 - 1/3*(0,#2) - 1/5*(#2,0) $) -- ($ #1 - 4/5*(#2,0) $);
	\draw[very thick] #1 ellipse ({4/5*#2} and {5/6*#2});
	\filldraw[very thick, unshaded] ($ #1 + 1/3*(0,#2) $) circle (1/5*#2);
	\filldraw[very thick, unshaded] ($ #1 - 1/3*(0,#2) $) circle (1/5*#2);
	\node at ($ #1 + (.15,.05) + 5/8*(0,#2)$) {\scriptsize{$#5$}};
	\node at ($ #1 + (.15,0) $) {\scriptsize{$#4$}};
	\node at ($ #1 + (.15,-.02) - 5/8*(0,#2)$) {\scriptsize{$#3$}};
}
\[
g\circ f \in
\cM'(\Phi'(a)\otimes m'^{\otimes i}, \Phi'(c)\otimes m'^{\otimes k})
=\cC(a,c\otimes \cP[i+k])
\]
under $\Psi$, where $\cP=\Lambda(\cM,m)$
and $g\circ f$ is as in \eqref{eq: that's how one defines composition}.
By chasing through the various definitions (Theorem~\ref{thm: construct P from M and m} and Algorithm~\ref{alg:AssignMap}) one verifies that the value of $\cP=(\cP,Z)$ on the multiplication tangle is given by
\[
Z\Bigg(\tikz[baseline=-2]{\multiplicationforhere{(0,0)}{1}{i}{j}{k}}\Bigg)=
\Tr_\cC(\id_{m^{\otimes k}}\otimes\bar\ev_{m^{\otimes j}}
\!\!\begin{tabular}{l}\\[1mm]
$\otimes\id_{m^{\otimes i}})
\circ
\mu_{m^{\otimes {k+j}},m^{\otimes {j+i}}}:$
\\[1mm]
\,\, $
\Tr_\cC(m^{\otimes k+j})\otimes \Tr_\cC(m^{\otimes j+i})\to \Tr_\cC(m^{\otimes k+i}).$
\end{tabular}
\]
It follows that $\Psi (g\circ f) = \Psi(g)\circ \Psi(f)$, as desired.
\end{proof}

\begin{lem}
\label{lem:PsiPreservesIdentity}
$\Psi$ preserves identity morphisms:\, $\Psi(\id_X)=\id_{\Psi(X)}$.
\end{lem}
\begin{proof}
It is enough to check this condition when $X$ is an object of $\cM_0'$.
The identity morphism $\id_{\Phi'(a)\otimes m'^{\otimes n}}$ is equal to the morphism in $\cC$ given by
$$
\begin{tikzpicture}[baseline=-.1cm]
	\draw[thick, orange] (0,-.8) -- (0,.8);
	\draw (1.5,0) -- (2.5,0);
	\node at (-.2,0) {\scriptsize{$a$}};
	\identityMap{(1,0)}{.5}{n}
	\node at (2,.2) {\scriptsize{$\cP[2n]$}};
\end{tikzpicture}
\,\in \cC(a,a\otimes \cP[2n]).
$$
By definition, $\Psi(\id_{\Phi'(a)\otimes m'^{\otimes n}})$ is then given by:
$$
\begin{tikzpicture}[baseline=0cm, scale=.8]
	\node at (-.2,1.4) {\scriptsize{$a$}};
	\draw[thick, orange] (-2,1.2) -- (2,1.2);
	\plane{(0,-.9)}{2.4}{1.8}
	\CMbox{box}{(0,-.5)}{.8}{.8}{.4}{$\varepsilon$}
	\straightCappedTube{(-.2,-.1)}{.1}{1}
	\draw (-.1,0) -- (-1.2,0) arc (270:90:.1cm) -- (-.1,.2);
	\draw (.8,-.1) -- (1.6,-.1);
	\draw (.8,-.3) arc (90:-90:.2cm) -- (-.2,-.7);
	\node at (1.1,.1) {\scriptsize{$n$}};
	\node at (1.2,-.5) {\scriptsize{$n$}};
\end{tikzpicture}
=
\begin{tikzpicture}[baseline=0cm, scale=.8]
	\node at (0,1.4) {\scriptsize{$a$}};
	\draw[thick, orange] (-2,1.2) -- (2.4,1.2);
	\plane{(0,-.9)}{2.8}{1.8}
	\CMbox{box}{(0,-.5)}{.8}{.8}{.4}{$\varepsilon$}
	\straightCappedTube{(-.2,-.1)}{.1}{1}
	\draw (1.2,-.1) -- (2,-.1);
	\draw (1.2,-.1) arc (90:270:.1cm);
	\draw (1.2,-.3) arc (90:-90:.2cm) -- (-.2,-.7);
	\node at (1.4,.1) {\scriptsize{$n$}};
	\node at (1.6,-.5) {\scriptsize{$n$}};
\end{tikzpicture}
=
\begin{tikzpicture}[baseline=0cm, scale=.8]
	\node at (-.8,1.4) {\scriptsize{$a$}};
	\draw[thick, orange] (-2.6,1.2) -- (1.4,1.2);
	\plane{(-.4,-.9)}{2}{1.8}
	\draw (-1.3,0) -- (.7,0);
	\node at (-.4,.2) {\scriptsize{$n$}};
\end{tikzpicture}
=\,\id_{\Phi(a)\otimes m^{\otimes n}}.
$$
\end{proof}

We now wish to endow $\Psi$ with the structure of a tensor functor.
It is enough to do this on $\Psi|_{\cM_0'}:\cM'_0\to\cM$.
Recall that, by definition,
$\left(\Phi'(a)\otimes m'^{\otimes k} \right) \otimes \left( \Phi'(b)\otimes m'^{\otimes n}\right) = \Phi'(a\otimes b) \otimes m'^{\otimes k+n}$,
and $1_{\cM'}=\Phi'(1_\cC)$.
So we must provide natural isomorphisms 
\[
\Psi\left(\Phi'(a)\otimes m'^{\otimes k}\right) \otimes \Psi\left(\Phi'(b)\otimes m'^{\otimes n}\right) \to \Psi\left(\Phi'(a\otimes b)\otimes m'^{\otimes k+n}\right),\quad\,
1_\cM\to\Psi(\Phi'(1_\cC)),
\]
subject to the usual associativity and unitality conditions.
We let the former be the composite
\begin{align*}
\big(\Phi(a)\otimes m^{\otimes k} \big) \otimes \big( \Phi(b)\otimes m^{\otimes n}\big)  
\to\,\,
&\Phi(a)\otimes \big(m^{\otimes k}  \otimes \Phi(b)\big)\otimes m^{\otimes n}  
\\\xrightarrow{ e_{\Phi(b),m^{\otimes k}}^{-1} }
\,\,&\Phi(a)\otimes \big( \Phi(b)  \otimes m^{\otimes k}\big)\otimes m^{\otimes n}  
\to
\Phi(a\otimes b) \otimes m^{\otimes k+n},
\end{align*}
and the latter be the canonical isomorphism $1_\cM\to\Phi(1_\cC)$.

\begin{prop}\label{prop: yes, it's a tensor functor}
The above isomorphisms equip $\Psi:\cM'\to \cM$ with the structure of a tensor functor.
\end{prop}

\begin{proof}
It is enough to check that $\Psi|_{\cM_0'}$ is a tensor functor.
Let $X=\Phi'(a)\otimes m'^{\otimes n}$, $Y=\Phi'(b)\otimes m'^{\otimes p}$, $Z=\Phi'(c)\otimes m'^{\otimes q}$ be objects.
The associativity condition
\[
\xymatrix{
\big(\Psi (X)\otimes \Psi(Y)\big) \otimes \Psi(Z)   \ar[d]_{\alpha}   \ar[r]
&\Psi(X\otimes Y) \otimes \Psi(Z)   \ar[r]
&\Psi\big((X\otimes Y) \otimes Z\big)   \ar[d]^{\alpha}
\\\Psi (X)\otimes \big(\Psi(Y) \otimes \Psi(Z)\big)   \ar[r]
&\Psi(X)\otimes \Psi\big(Y \otimes Z\big)   \ar[r]
&\Psi\big(X\otimes (Y \otimes Z)\big)}
\]
can be rewritten as:
{\scriptsize{
\[
\hspace{-.25cm}
\xymatrix@C=.5cm@R=.5cm{
\bigg(\!\!\left( \Phi(a)\otimes m^{\otimes n}\right)\otimes \left( \Phi(b)\otimes m^{\otimes p}\right)\!\!\bigg) \otimes \left( \Phi(c)\otimes m^{\otimes q}\right)
\ar[d]_{\alpha}
\ar[r]
&
\Phi\left((a\otimes b)\otimes m'^{\otimes n+p}\right)\otimes \left( \Phi(c)\otimes m^{\otimes q}\right)
\ar[r]
&
\Phi\big((a\otimes b)\otimes  c \big) \otimes m^{\otimes n+p+q} 
\ar[d]^{\alpha}
\\
\left( \Phi(a)\otimes m^{\otimes n}\right)\otimes \bigg(\!\!\left( \Phi(b)\otimes m^{\otimes p}\right) \otimes \left( \Phi(c)\otimes m^{\otimes q}\right)\!\!\bigg)
\ar[r]
&
\left( \Phi(a)\otimes m^{\otimes n}\right)\otimes \left( \Phi(b\otimes c)\otimes m^{\otimes p+q}\right)
\ar[r]
&
\Phi\big(a\otimes (b\otimes  c) \big) \otimes m^{\otimes n+p+q} 
}
\]
}}
\!\!This diagram is commutative because of the following equality of morphisms in $\cM$:
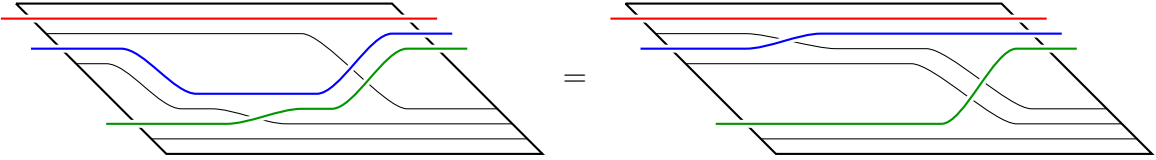
$$
\begin{tikzpicture}[baseline=.9cm]
	\plane{(0,0)}{5}{2}
	\draw[super thick, white] (-2.2,1.8) -- (3.6,1.8);
	\draw[thick, red] (-2.2,1.8) -- (3.6,1.8);
	\draw (-1.6,1.6) -- (1.8,1.6) .. controls ++(0:.3cm) and ++(180:.3cm) .. (3.2,.6) -- (4.4,.6);
	\draw (-1.2,1.2) -- (-.8,1.2) .. controls ++(0:.3cm) and ++(180:.3cm) .. (.2,.6) -- (.6,.6) .. controls ++(0:.3cm) and ++(180:.3cm) .. (1.6,.4) -- (4.6,.4);
	\draw[super thick, white] (-1.8,1.4) -- (-.6,1.4) .. controls ++(0:.3cm) and ++(180:.3cm) .. (.4,.8) -- (2,.8) .. controls ++(0:.3cm) and ++(180:.3cm) .. (3,1.6) -- (3.8,1.6); 
	\draw[thick, blue] (-1.8,1.4) -- (-.6,1.4) .. controls ++(0:.3cm) and ++(180:.3cm) .. (.4,.8) -- (2,.8) .. controls ++(0:.3cm) and ++(180:.3cm) .. (3,1.6) -- (3.8,1.6); 
	\draw (-.2,.2) -- (4.8,.2);
	\draw[super thick, white] (-.8,.4) -- (.8,.4) .. controls ++(0:.3cm) and ++(180:.3cm) .. (1.8,.6) -- (2.2,.6) .. controls ++(0:.3cm) and ++(180:.3cm) .. (3.2,1.4) -- (4,1.4); 
	\draw[thick, DarkGreen] (-.8,.4) -- (.8,.4) .. controls ++(0:.3cm) and ++(180:.3cm) .. (1.8,.6) -- (2.2,.6) .. controls ++(0:.3cm) and ++(180:.3cm) .. (3.2,1.4) -- (4,1.4); 
\end{tikzpicture}
\,\,=\,\,
\begin{tikzpicture}[baseline=.9cm]
	\plane{(0,0)}{5}{2}
	\draw[super thick, white] (-2.2,1.8) -- (3.6,1.8);
	\draw[thick, red] (-2.2,1.8) -- (3.6,1.8);
	\draw (-1.6,1.6) -- (-.4,1.6) .. controls ++(0:.3cm) and ++(180:.3cm) ..  (.8,1.4) -- (2,1.4) .. controls ++(0:.3cm) and ++(180:.3cm) .. (3.4,.6) -- (4.4,.6);
	\draw (-1.2,1.2) -- (1.8,1.2)  .. controls ++(0:.3cm) and ++(180:.3cm) .. (3.2,.4) -- (4.6,.4);
	\draw[super thick, white] (-1.8,1.4) -- (-.4,1.4) .. controls ++(0:.3cm) and ++(180:.3cm) .. (.6,1.6) -- (3.8,1.6); 
	\draw[thick, blue] (-1.8,1.4) -- (-.4,1.4) .. controls ++(0:.3cm) and ++(180:.3cm) .. (.6,1.6) -- (3.8,1.6); 
	\draw (-.2,.2) -- (4.8,.2);
	\draw[super thick, white] (-.8,.4) -- (2.2,.4) .. controls ++(0:.3cm) and ++(180:.3cm) .. (3.2,1.4) -- (4,1.4); 
	\draw[thick, DarkGreen] (-.8,.4) -- (2.2,.4) .. controls ++(0:.3cm) and ++(180:.3cm) .. (3.2,1.4) -- (4,1.4); 
\end{tikzpicture}
$$
The unitality conditions
\[
\begin{matrix} \xymatrix{
\Psi(1_{\cM'})\otimes\Psi(X)   \ar@{<-}[r]  \ar[d]
& 1_{\cM}\otimes\Psi(X)  \ar[d]
\\ \Psi\big(1_{\cM'}\otimes X\big)   \ar[r]
& \Psi(X)
}\end{matrix}
\quad\text{and}\quad
\begin{matrix}\xymatrix{
\Psi(X) \otimes 1_{\cM}  \ar[d]
& \Psi(X) \otimes \Psi(1_{\cM'})  \ar@{<-}[l]  \ar[d]
\\ \Psi(X) 
&\Psi\big(X \otimes 1_{\cM'}\big) \ar[l]
}\end{matrix}
\]
are left as an exercise to the reader.

It remains to show that the map $\Psi(X)\otimes\Psi(Y)\to \Psi(X\otimes Y)$ is natural in $X$ and $Y$.
Given two morphisms in $\cM'$ (which we may take with out loss of generality in $\cM_0'$)
\[
f:\Phi'(a)\otimes m'^{\otimes i} \to \Phi'(b)\otimes m'^{\otimes j},\qquad
g: \Phi'(c)\otimes m'^{\otimes k} \to \Phi'(d)\otimes m'^{\otimes \ell},
\]
we need to argue that the following diagram commutes:
\[
\xymatrix{
\Psi( \Phi'(a)\otimes m'^{\otimes i}) \otimes \Psi( \Phi'(c)\otimes m'^{\otimes k})
\ar[r]
\ar[d]|{\Psi(f)\otimes \Psi(g)}
&
\Psi\big(( \Phi'(a)\otimes m'^{\otimes i}) \otimes ( \Phi'(c)\otimes m'^{\otimes k}) \big)
\ar[d]^{\Psi(f\otimes g)}
\\
\Psi( \Phi'(b)\otimes m'^{\otimes j}) \otimes \Psi( \Phi'(d)\otimes m'^{\otimes \ell})
\ar[r]
&
\Psi\big(( \Phi'(b)\otimes m'^{\otimes j}) \otimes ( \Phi'(d)\otimes m'^{\otimes \ell}) \big)
}
\]
Equivalently:
\begin{equation}\label{fkwmdgo}
\begin{matrix}
\xymatrix{
( \Phi(a)\otimes m^{\otimes i}) \otimes ( \Phi(c)\otimes m^{\otimes k})
\ar[r]
\ar[d]|{\Psi(f)\otimes \Psi(g)}
&
\Phi(a\otimes b)\otimes m'^{\otimes i+k}
\ar[d]^{\Psi(f\otimes g)}
\\
( \Phi(b)\otimes m^{\otimes j}) \otimes ( \Phi(d)\otimes m^{\otimes \ell})
\ar[r]
&
\Phi(b\otimes d)\otimes m^{\otimes j+\ell}
}
\end{matrix}
\end{equation}

Recall the definition
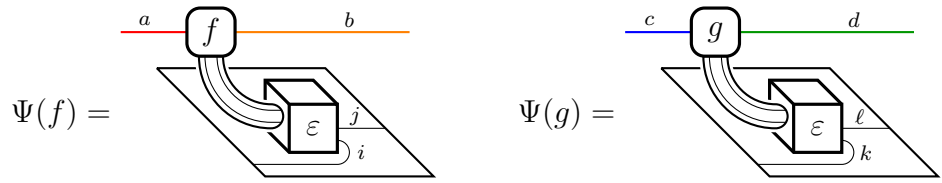
\[
\Psi(f)=
\begin{tikzpicture}[baseline=0cm, scale=.8]
	\plane{(-.4,-.9)}{2.8}{1.8}
	\node at (-2.4,1.7) {\scriptsize{$a$}};
	\node at (1,1.7) {\scriptsize{$b$}};
	\draw[thick, red] (-2.8,1.5) -- (-1.2,1.5);
	\draw[thick, orange] (-1.2,1.5) -- (2,1.5);
	\coordinate (a) at (0,-.5);
	\draw ($ (a) + (.8,.4) $)  -- ($ (a) + (1.6,.4) $) ;
	\draw ($ (a) + (.8,.2) $)  arc (90:-90:.2cm) -- ($ (a) + (-.6,-.2) $) ;
	\node at ($ (a) + (1.1,.6) $)  {\scriptsize{$j$}};
	\node at ($ (a) + (1.2,0) $)  {\scriptsize{$i$}};
	\CMbox{box}{(a)}{.8}{.8}{.4}{$\varepsilon$}
	\curvedTubeNoString{($ (a) + (-.3,.4) $)}{1}{0}{$f$}
\end{tikzpicture}
\qquad\,\,\,\,
\Psi(g)=
\begin{tikzpicture}[baseline=0cm, scale=.8]
	\plane{(-.4,-.9)}{2.8}{1.8}
	\node at (-2.4,1.7) {\scriptsize{$c$}};
	\node at (1,1.7) {\scriptsize{$d$}};
	\draw[thick, blue] (-2.8,1.5) -- (-1.2,1.5);
	\draw[thick, DarkGreen] (-1.2,1.5) -- (2,1.5);
	\coordinate (a) at (0,-.5);
	\draw ($ (a) + (.8,.4) $)  -- ($ (a) + (1.6,.4) $) ;
	\draw ($ (a) + (.8,.2) $)  arc (90:-90:.2cm) -- ($ (a) + (-.6,-.2) $) ;
	\node at ($ (a) + (1.1,.6) $)  {\scriptsize{$\ell$}};
	\node at ($ (a) + (1.2,0) $)  {\scriptsize{$k$}};
	\CMbox{box}{(a)}{.8}{.8}{.4}{$\varepsilon$}
	\curvedTubeNoString{($ (a) + (-.3,.4) $)}{1}{0}{$g$}
\end{tikzpicture}
\]
of $\Psi(f)$ and $\Psi(g)$.
The composite\;\! $\tikz{\draw[-latex, scale=.4, baseline=6] (1.5,1) -- (0,1) -- (0,0) -- (1.5,0);}$\;\!
in \eqref{fkwmdgo} is given by
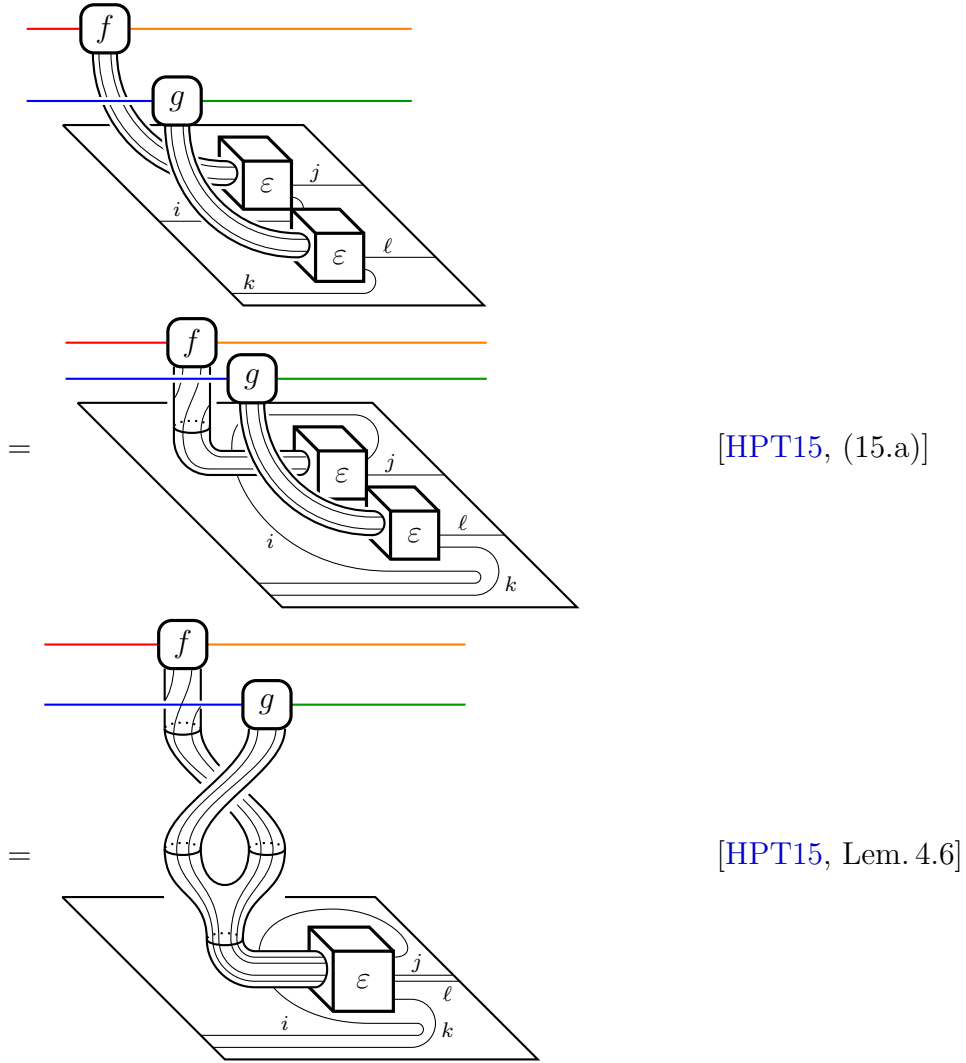
\begin{align}
&\phantom{=}
\begin{tikzpicture}[baseline=1.2cm, scale=.8]
	\plane{(0,-.6)}{4}{3}
	\coordinate (a) at (1.2,-.2);
	\coordinate (b) at (0,1);
	\draw[thick, red] ($ (b) + (-3.6,3) $) -- ($ (b) + (-2,3) $);
	\draw[thick, orange] ($ (b) + (2.8,3) $) -- ($ (b) + (-2,3) $);
	\draw ($ (b) + (.8,.2) $) arc (90:-90:.2cm) -- ($ (b) + (-1.4,-.2) $) ;
	\CMbox{box}{(b)}{.8}{.8}{.4}{$\varepsilon$}
	\curvedTubeNoString{($ (b) + (-.3,.4) $)}{2}{0}{$f$}
	\draw[super thick, white] ($ (a) + (-4.8,3) $) -- ($ (a) + (-2,3) $);
	\draw[super thick, white] ($ (a) + (1.6,3) $) -- ($ (a) + (-2,3) $);
	\draw[thick, blue] ($ (a) + (-4.8,3) $) -- ($ (a) + (-2,3) $);
	\draw[thick, DarkGreen] ($ (a) + (1.6,3) $) -- ($ (a) + (-2,3) $);
	\draw ($ (a) + (.8,.2) $) arc (90:-90:.2cm) -- ($ (a) + (-1.4,-.2) $) ;
	\CMbox{box}{(a)}{.8}{.8}{.4}{$\varepsilon$}
	\curvedTubeNoString{($ (a) + (-.3,.4) $)}{2}{0}{$g$}
	\draw ($ (a) + (.8,.4) $)  -- ($ (a) + (2,.4) $) ;
	\draw ($ (b) + (.8,.4) $)  -- ($ (b) + (2,.4) $) ;
	\node at ($ (b) + (-1.1,0) $)  {\scriptsize{$i$}};
	\node at ($ (a) + (-1.1,0) $)  {\scriptsize{$k$}};
	\node at ($ (b) + (1.2,.6) $)  {\scriptsize{$j$}};
	\node at ($ (a) + (1.2,.6) $)  {\scriptsize{$\ell$}};
\end{tikzpicture}
\notag
\\
&=\phantom{=}\begin{tikzpicture}[baseline=1.2cm, scale=.8]
	\plane{(-.6,-1)}{4.9}{3.4}
	\pgfmathsetmacro{\voffset}{.08};
	\pgfmathsetmacro{\hoffset}{.15};
	\coordinate (a) at (1.2,-.2);
	\coordinate (b) at (0,.8);
	\coordinate (tt) at ($ (b) + (-2.4,1.2) $);
	\filldraw[white] (-2.5,2.2) rectangle (-1.7,2.6);
	\draw ($ (b) + (.6,.6) $) arc (-90:90:.4cm) -- ($ (b) + (-1,1.4) $) .. controls ++(180:1cm) and ++(180:2.5cm) .. ($ (a) + (0,-.2) $) -- ($ (a) + (1.4,-.2) $) arc (90:-90:.1cm) -- ($ (a) + (-2.2,-.4) $);
	\CMbox{box}{(b)}{.8}{.8}{.4}{$\varepsilon$}
	\draw[thick, unshaded] ($ (b) + (-1.8,1.2) $) -- ($ (b) + (-1.8,1) $) arc (180:270:.2cm) -- ($ (b) + (-.3,.8) $) arc (90:-90: .15 cm and .2 cm) -- ($ (b) + (-1.8,.4) $) arc (270:180:.6cm) -- ($ (b) + (-2.4,1.2) $);
	\emptyCylinder{(tt)}{.3}{1}
	\draw ($ (tt) + 1*(\hoffset,0) + (0,-.1) $) .. controls ++(90:.4cm) and ++(270:.4cm) .. ($ (tt) + 3*(\hoffset,0) + (0,1) $);	
	\draw ($ (tt) + 1*(\hoffset,0) + (0,-.1) $) .. controls ++(270:.4cm) and ++(180:.3cm) .. ($ (tt) + (.6,-.7) $) -- ($ (tt) + (2.2,-.7) $);
	\draw ($ (tt) + 3*(\hoffset,0) + (0,-\voffset) $) .. controls ++(90:.2cm) and ++(225:.1cm) .. ($ (tt) + 4*(\hoffset,0) + (0,-\voffset) + (0,.45)$);
	\draw  ($ (tt) + 3*(\hoffset,0) + (0,-\voffset) $) .. controls ++(270:.2cm) and ++(180:.2cm) .. ($ (tt) + (.8,-.5) $) -- ($ (tt) + (2.2,-.5) $);
	\draw ($ (tt) + (\hoffset,1) $) .. controls ++(270:.2cm) and ++(45:.1cm) .. ($ (tt) + (0,1) + (0,-\voffset) + (0,-.45)$);
	\halfDottedEllipse{(tt)}{.3}{.1}
	\draw[thick, red] ($ (b) + (-4.2,2.6) $) -- ($ (b) + (-2,2.6) $);
	\draw[thick, orange] ($ (b) + (2.8,2.6) $) -- ($ (b) + (-2,2.6) $);
	\roundNbox{unshaded}{($ (b) + (-2.1,2.6) $)}{.4}{0}{0}{$f$}
	\draw[super thick, white] ($ (a) + (-4.8,3) $) -- ($ (a) + (-2,3) $);
	\draw[super thick, white] ($ (a) + (1.6,3) $) -- ($ (a) + (-2,3) $);
	\draw[thick, blue] ($ (a) + (-5.4,3) $) -- ($ (a) + (-2,3) $);
	\draw[thick, DarkGreen] ($ (a) + (1.6,3) $) -- ($ (a) + (-2,3) $);
	\draw ($ (a) + (.8,.2) $) -- ($ (a) + (1.4,.2) $) arc (90:-90:.4cm) -- ($ (a) + (-2,-.6) $) ;
	\CMbox{box}{(a)}{.8}{.8}{.4}{$\varepsilon$}
	\curvedTubeNoString{($ (a) + (-.3,.4) $)}{2}{0}{$g$}
	\draw ($ (a) + (.8,.4) $)  -- ($ (a) + (1.9,.4) $) ;
	\draw ($ (b) + (.8,.4) $)  -- ($ (b) + (2.1,.4) $) ;
	\node at ($ (a) + (-2,.3) $)  {\scriptsize{$i$}};
	\node at ($ (a) + (2,-.4) $)  {\scriptsize{$k$}};
	\node at ($ (b) + (1.2,.6) $)  {\scriptsize{$j$}};
	\node at ($ (a) + (1.2,.6) $)  {\scriptsize{$\ell$}};
\end{tikzpicture}
&&
\text{\cite[(15.a)]{1509.02937}}
\displaybreak[1]
\notag
\\&=
\begin{tikzpicture}[scale=.8, baseline = .6cm]
	\plane{(1,-2.5)}{5.2}{2.7}
	\pgfmathsetmacro{\voffset}{.08};
	\pgfmathsetmacro{\hoffset}{.15};
	\coordinate (tt) at (0,3);
	\draw (3.8,-.8) .. controls ++(0:.6cm) and ++(0:1cm) .. (2.8,0) .. controls ++(180:2cm) and ++(180:2cm) .. (3.4,-1.9) -- (4.2,-1.9) arc (90:-90:.1cm) -- (.6,-2.1);
	\CMbox{box}{(2.8,-1.7)}{1}{1}{.4}{$\varepsilon$}
	\draw[thick, unshaded] (1.3,-.5) arc (180:270:.2cm) -- (2.5,-.7) arc (90:-90: .2 cm and .3 cm) -- (1.5,-1.3) arc (270:180:.8cm);
	\fill[white] (0,0) rectangle (2,1);
	\draw ($ (tt) + 1*(\hoffset,0) + (0,-.1) $) .. controls ++(270:.6cm) and ++(90:.8cm) .. (1.6,1);
	\draw ($ (tt) + 3*(\hoffset,0) + (0,-\voffset) $) .. controls ++(270:.6cm) and ++(90:.7cm) .. (1.8,1);
	\braid{(0,1)}{.3}{2}
	\halfDottedEllipse{(0,1)}{.3}{.1}
	\halfDottedEllipse{(1.4,1)}{.3}{.1}
	\halfDottedEllipse{(0,3)}{.3}{.1}
	\invertedPairOfPants{(0,1)}{};
	\emptyCylinder{(0,3)}{.3}{1}
	\draw ($ (tt) + 1*(\hoffset,0) + (1.4,0) $) .. controls ++(270:.8cm) and ++(90:.8cm) .. (.2,1);
	\draw ($ (tt) + 3*(\hoffset,0) + (1.4,0) $) .. controls ++(270:.7cm) and ++(90:.8cm) .. (.4,1);
	\draw[thick, red] (-2,4.4) -- (.3,4.4);
	\draw[thick, orange] (5,4.4) -- (.3,4.4);
	\draw[super thick, white] (-2,3.4) -- (1.7,3.4);
	\draw[thick, blue] (-2,3.4) -- (1.7,3.4);
	\draw[thick, DarkGreen] (5,3.4) -- (1.7,3.4);
	\roundNbox{unshaded}{(.3,4.4)}{.4}{0}{0}{$f$}
	\roundNbox{unshaded}{(1.7,3.4)}{.4}{0}{0}{$g$}
	\draw (3.8,-1.5) -- (4.1,-1.5) arc (90:-90:.4cm) -- (.8,-2.3);
	\draw (3.8, -1.1) -- (4.8,-1.1);
	\draw (3.8, -1.2) -- (4.9,-1.2);
	\node at (2,-1.9)  {\scriptsize{$i$}};
	\node at (4.7,-2)  {\scriptsize{$k$}};
	\node at (4.7,-1.4)  {\scriptsize{$\ell$}};
	\node at (4.2,-.9)  {\scriptsize{$j$}};
	\draw (.2,1) .. controls ++(270:.6cm) and ++(90:.8cm) .. (.8,-.5) arc (180:270:.7cm) -- (2.65,-1.2);
	\draw (.4,1) .. controls ++(270:.6cm) and ++(90:.9cm) .. (.9,-.5) arc (180:270:.6cm) -- (2.7,-1.1);
	\draw (1.6,1) .. controls ++(270:.6cm) and ++(90:.9cm) .. (1.1,-.5) arc (180:270:.4cm) -- (2.7,-.9);
	\draw (1.8,1) .. controls ++(270:.6cm) and ++(90:.8cm) .. (1.2,-.5) arc (180:270:.3cm) -- (2.65,-.8) ;
	\draw ($ (tt) + 1*(\hoffset,0) + (0,-.1) $) .. controls ++(90:.4cm) and ++(270:.4cm) .. ($ (tt) + 3*(\hoffset,0) + (0,1) $);	
	\draw ($ (tt) + 3*(\hoffset,0) + (0,-\voffset) $) .. controls ++(90:.2cm) and ++(225:.1cm) .. ($ (tt) + 4*(\hoffset,0) + (0,-\voffset) + (0,.45)$);
	\draw ($ (tt) + (\hoffset,1) $) .. controls ++(270:.2cm) and ++(45:.1cm) .. ($ (tt) + (0,1) + (0,-\voffset) + (0,-.45)$);
\end{tikzpicture}
&&
\text{\cite[Lem.\,4.6]{1509.02937}}
\displaybreak[1]
\notag
\\&=\label{jyt,qgty7}
\begin{tikzpicture}[scale=.8, baseline = .6cm]
	\plane{(1,-2.3)}{5}{2.5}
	\pgfmathsetmacro{\voffset}{.08};
	\pgfmathsetmacro{\hoffset}{.15};
	\coordinate (tt) at (0,3);
	\CMbox{box}{(2.8,-1.7)}{1}{1}{.4}{$\varepsilon$}
	\draw[thick, unshaded] (1.3,-.5) arc (180:270:.2cm) -- (2.5,-.7) arc (90:-90: .2 cm and .3 cm) -- (1.5,-1.3) arc (270:180:.8cm);
	\fill[white] (0,0) rectangle (2,1);
	\draw ($ (tt) + 1*(\hoffset,0) + (0,-.1) $) .. controls ++(270:.6cm) and ++(90:.8cm) .. (1.6,1);
	\draw ($ (tt) + 3*(\hoffset,0) + (0,-\voffset) $) .. controls ++(270:.6cm) and ++(90:.7cm) .. (1.8,1);
	\braid{(0,1)}{.3}{2}
	\halfDottedEllipse{(0,1)}{.3}{.1}
	\halfDottedEllipse{(1.4,1)}{.3}{.1}
	\halfDottedEllipse{(0,3)}{.3}{.1}
	\invertedPairOfPants{(0,1)}{};
	\emptyCylinder{(0,3)}{.3}{1}
	\draw ($ (tt) + 1*(\hoffset,0) + (1.4,0) $) .. controls ++(270:.8cm) and ++(90:.8cm) .. (.2,1);
	\draw ($ (tt) + 3*(\hoffset,0) + (1.4,0) $) .. controls ++(270:.7cm) and ++(90:.8cm) .. (.4,1);
	\draw[thick, red] (-2,4.4) -- (.3,4.4);
	\draw[thick, orange] (5,4.4) -- (.3,4.4);
	\draw[super thick, white] (-2,3.4) -- (1.7,3.4);
	\draw[thick, blue] (-2,3.4) -- (1.7,3.4);
	\draw[thick, DarkGreen] (5,3.4) -- (1.7,3.4);
	\roundNbox{unshaded}{(.3,4.4)}{.4}{0}{0}{$f$}
	\roundNbox{unshaded}{(1.7,3.4)}{.4}{0}{0}{$g$}
	\draw (3.8,-1.5) arc (90:-90:.3cm) -- (.8,-2.1);
	\draw (3.8,-1.6) arc (90:-90:.2cm) -- (.7,-2);
	\draw (3.8, -1.1) -- (4.8,-1.1);
	\draw (3.8, -1.2) -- (4.9,-1.2);
	\node at (2,-1.8)  {\scriptsize{$i$}};
	\node at (4.2,-2)  {\scriptsize{$k$}};
	\node at (4.5,-1.4)  {\scriptsize{$\ell$}};
	\node at (4.2,-.9)  {\scriptsize{$j$}};
	\draw (.2,1) .. controls ++(270:.6cm) and ++(90:.8cm) .. (.8,-.5) arc (180:270:.7cm) -- (1.5,-1.2) .. controls ++(0:.1cm) and ++(180:.1cm) .. (2.4,-1.1) -- (2.7,-1.1);
	\draw (.4,1) .. controls ++(270:.6cm) and ++(90:.9cm) .. (.9,-.5) arc (180:270:.6cm) -- (1.5,-1.1) .. controls ++(0:.2cm) and ++(180:.2cm) .. (2.4,-.9) -- (2.7,-.9);
	\draw (1.6,1) .. controls ++(270:.6cm) and ++(90:.9cm) .. (1.1,-.5) arc (180:270:.4cm) -- (1.5,-.9) .. controls ++(0:.1cm) and ++(180:.1cm) .. (2.4,-.8) -- (2.65,-.8);
	\draw (1.8,1) .. controls ++(270:.6cm) and ++(90:.8cm) .. (1.2,-.5) arc (180:270:.3cm) -- (1.5,-.8) ;
	\draw (1.5,-.8) .. controls ++(0:.2cm) and ++(-135:.1cm) .. (1.8,-.7);
	\draw (2.1,-1.3) .. controls ++(45:.1cm) and ++(180:.3cm) .. (2.6,-1.2) -- (2.65,-1.2);
	\draw ($ (tt) + 1*(\hoffset,0) + (0,-.1) $) .. controls ++(90:.4cm) and ++(270:.4cm) .. ($ (tt) + 3*(\hoffset,0) + (0,1) $);	
	\draw ($ (tt) + 3*(\hoffset,0) + (0,-\voffset) $) .. controls ++(90:.2cm) and ++(225:.1cm) .. ($ (tt) + 4*(\hoffset,0) + (0,-\voffset) + (0,.45)$);
	\draw ($ (tt) + (\hoffset,1) $) .. controls ++(270:.2cm) and ++(45:.1cm) .. ($ (tt) + (0,1) + (0,-\voffset) + (0,-.45)$);
\end{tikzpicture}
&&
\text{\cite[(15.b)]{1509.02937}}
\displaybreak[1]
\end{align}
We need to show that this is equal to the image of
\newcommand{\tensorforhere}[6]{
	\draw ($ #1 + 16/23*(0,#2) - 1/3*(#2,0) $) -- ($ #1 - 16/23*(0,#2) - 1/3*(#2,0) $);
	\draw ($ #1 + 16/23*(0,#2) + 1/3*(#2,0) $) -- ($ #1 - 16/23*(0,#2) + 1/3*(#2,0) $);
	\draw[thick, red] ($ #1 - 1/3*(#2,0) - 1/5*(#2,0) $) -- ($ #1 - 5/6*(#2,0) $);
	\draw[thick, red] ($ #1 + 1/3*(#2,0) - 1/5*(#2,0) $) .. controls ++(180:.2cm) and ++(0:.2cm) .. ($ #1 - 1/3*(#2,0) + 2/5*(0,#2) $) .. controls ++(180:.2cm) and ++(0:.2cm) .. ($ #1 - 5/6*(#2,0) $);
	\draw[very thick] #1 ellipse ( {5/6*#2} and {78/100*#2});
	\filldraw[very thick, unshaded] ($ #1 + 1/3*(#2,0) $) circle (1/5*#2);
	\filldraw[very thick, unshaded] ($ #1 - 1/3*(#2,0) $) circle (1/5*#2);
	\node at ($ #1 + (.2,0) - 1/3*(#2,0) - .4*(0,#2) $) {\scriptsize{$#3$}};
	\node at ($ #1 + (.2,0) - 1/3*(#2,0) + .4*(0,#2) $) {\scriptsize{$#4$}};
	\node at ($ #1 + (.2,0) + 1/3*(#2,0) - .4*(0,#2) $) {\scriptsize{$#5$}};
	\node at ($ #1 + (.2,0) + 1/3*(#2,0) + .4*(0,#2) $) {\scriptsize{$#6$}};
}
\[
\begin{split}
f\otimes g &\,\in\,
\cM'\big((\Phi'(a)\otimes m'^{\otimes i}) \otimes (\Phi'(c)\otimes m'^{\otimes k}),
(\Phi'(b)\otimes m'^{\otimes j})\otimes (\Phi'(d)\otimes m'^{\otimes \ell})\big)
\\
&\,=\,
\cM'(\Phi'(a\otimes c)\otimes m'^{\otimes i+k},
\Phi'(b\otimes d)\otimes m'^{\otimes j+\ell})
\\&\,=\,\cC(a\otimes c,b\otimes d\otimes \cP[j+\ell+k+i])
\end{split}
\]
under $\Psi$.
Recall that $f\otimes g$ was defined in \eqref{eq:TensorProductInM}.
By the anchor dependence relation (Definition \ref{def: anchored planar algebra})
and the definition of $\varpi_{i,j}$ (Theorem~\ref{thm: construct P from M and m})
the value of $\cP=(\cP,Z)$ on the tensor product tangle is given by
\[
\begin{split}
Z\Bigg(\tikz[baseline=-17]{\tensorforhere{(3.2,-.5)}{1.2}{i}{j}{k}{\ell}}
\Bigg)
=\,&
Z\Bigg(\tikz[baseline=11, yscale=-1]{\tensorforhere{(3.2,-.5)}{1.2}{j}{i}{\ell}{k}}
\Bigg)\circ\beta_{\cP[\ell+k],\cP[j+i]}
\begin{tabular}{l}\\[1.5mm]
$=Z(p_{j,\ell+k})\circ\beta_{\cP[\ell+k],\cP[j+i]}$
\\[1.5mm]
$=\varpi_{j,\ell+k}\circ\beta_{\cP[\ell+k],\cP[j+i]}$
\end{tabular}
\\
=\tau^-_{m^{\otimes i},m^{\otimes j+\ell+k}} &\circ \mu_{m^{\otimes j+i},m^{\otimes \ell+k}}\circ (\tau^+_{m^{\otimes j},m^{\otimes i}} \otimes \id)\circ\beta_{\cP[\ell+k],\cP[j+i]}:
\\
&\hspace{2.3cm}\Tr_\cC(m^{\otimes \ell+k})\otimes \Tr_\cC(m^{\otimes j+i})\to \Tr_\cC(m^{\otimes j+\ell+k+i}).
\end{split}
\]
This composite is exactly the one which appears in \eqref{jyt,qgty7}, and so the latter is equal $\Psi(f\otimes g)$.
The commutativity of \eqref{fkwmdgo} follows.
\end{proof}

We also need to check that $\Psi$ is a pivotal functor (first item in Definition \ref{defn:ModuleTensorCategoryFunctor}):

\begin{lem}\label{lem: Psi is a pivotal functor}
The functor $\Psi:\cM'\to \cM$ is compatible with the pivotal structures:
\[
\Psi(\varphi_X)=\delta_{X^*}^{-1}\circ \delta_X^*\circ \varphi_{\Psi(X)}.
\]
\end{lem}

\begin{proof}
Once again, it is enough to check this compatibility on objects of the form $X=\Phi'(a)\otimes m'^{\otimes n}$.
Using Lemma \ref{lem:DualityInM0} to identify $(\Phi'(a)\otimes m'^{\otimes n})^*$ with $\Phi'(a^*)\otimes m'^{\otimes n}$, 
the desired equation becomes:
\begin{equation}\label{eq: ffwjjohj}
\Psi(\varphi_{\Phi'(a)\otimes m'^{\otimes n}})=\delta_{\Phi'(a^*)\otimes m'^{\otimes n}}^{-1}\circ \delta_{\Phi'(a)\otimes m'^{\otimes n}}^*\circ \varphi_{\Phi(a)\otimes m^{\otimes n}}.
\end{equation}
Applying the definition \eqref{eq:Psi of f} of $\Psi$ to the morphism $\varphi_{\Phi'(a)\otimes m'^{\otimes n}}$ (described in \eqref{eq: def: piv structure}) yields
the formula:
\[
\Psi(\varphi_{\Phi'(a)\otimes m'^{\otimes n}})=\Phi(\varphi_a)\otimes \id_{m^{\otimes n}}.
\]
Let $\psi:m\to m^*$ be the isomorphism which equips $m$ with the structure of a symmetrically self-dual object (paragraph before Definition \ref{def: pointed}),
and let $\dot\delta_a:\Phi(a^*)\to\Phi(a)^*$ be the natural isomorphism coming from the fact that $\Phi$ is a tensor functor.
Using the monoidal property of $\varphi$, the fact that $\Phi$ is a pivotal functor, and the axiom satisfied by $\psi$, we get:
\[
\varphi_{\Phi(a)\otimes m^{\otimes n}}=
\varphi_{\Phi(a)}\otimes {\varphi_{m}}^{\otimes n}=
\big[(\dot\delta_a^*)^{-1}\circ \dot\delta_{a^*}\circ\Phi(\varphi_a)\big]\otimes\big[(\psi^*)^{-1}\circ\psi\big]^{\otimes n}
\]
Finally, using the definition of $\delta$ (Definition \ref{defn:ModuleTensorCategoryFunctor}), we compute:
\[
\begin{split}
\delta_{\Phi'(a)\otimes m'^{\otimes n}}
&=(\Psi(\ev_{\Phi'(a)\otimes m'^{\otimes n}})\otimes \id)\circ(\id\otimes \coev_{\Psi(\Phi'(a)\otimes m'^{\otimes n})})\\
&=((\Phi(\ev_a)\otimes\bar\ev_{m^{\otimes n}})\otimes \id)\circ(\id\otimes \coev_{\Phi(a)\otimes m^{\otimes n}})\qquad\qquad\quad[\text{\eqref{eq:Psi of f} and Lem.\,\ref{lem:DualityInM0}}]
\\
&=\big((\Phi(\ev_a)\otimes \id)\circ(\id\otimes \coev_{\Phi(a)})\big)
\otimes
\big((\bar\ev_m\otimes \id)\circ(\id\otimes \coev_{m})\big)^{\otimes n}\\
&=\dot\delta_a\otimes \psi^{\otimes n}
\end{split}
\]
It follows that
$\delta_{\Phi'(a^*)\otimes m'^{\otimes n}}^{-1}=\dot\delta_{a^*}^{-1}\otimes (\psi^{-1})^{\otimes n}$
and
$\delta_{\Phi'(a)\otimes m'^{\otimes n}}^*=\dot\delta_a^*\otimes (\psi^*)^{\otimes n}$.
The right hand side of \eqref{eq: ffwjjohj} is therefore given by:
\[
\big[\dot\delta_{a^*}^{-1}\otimes (\psi^{-1})^{\otimes n}\big]\circ
\big[\dot\delta_a^*\otimes (\psi^*)^{\otimes n}\big]\circ
\big[
[(\dot\delta_a^*)^{-1}\circ \dot\delta_{a^*}\circ\Phi(\varphi_a)]\otimes[(\psi^*)^{-1}\circ\psi]^{\otimes n}
\big]=\Phi(\varphi_a)\otimes \id_{m^{\otimes n}}\!.
\]
This finishes the proof.
\end{proof}

Recall that $(\cM',m')$ is the image of $(\cM,m)$ under the functor $\Delta\circ\Lambda$.
By Lemmas \ref{lem: yes, it's an equivalence of categories}, \ref{lem:ThetaPreservesComposition}, \ref{lem:PsiPreservesIdentity}, \ref{lem: Psi is a pivotal functor} and Proposition \ref{prop: yes, it's a tensor functor},
we know that $\Psi: \cM' \to \cM$ is an equivalence of pivotal tensor categories.
To verify that $\Psi$ is an equivalence of \emph{pointed} pivotal tensor categories (Definition \ref{def: pointed}), we still need to check that $\Psi(m')=m$ and that
\begin{equation}\label{eq: psi=delta circ Psi(psi)}
\psi=\delta_{m'}\circ \Psi(\psi').
\end{equation}
Here, $\psi:m\to m^*$ and $\psi':m'\to (m')^*$ are the isomorphisms which make $m$ and $m'$ into symmetrically self-dual objects,
and $\delta_{m'}:\Psi((m')^*)\to\Psi(m')^*$ is as described in Definition~\ref{defn:ModuleTensorCategoryFunctor}.
The first condition $\Psi(m')=m$ is clear.
The second condition \eqref{eq: psi=delta circ Psi(psi)} reduces to $\psi=\delta_{m'}$
because $\psi'=\id_{m'}$ (see the end of Section \ref{sec:The pivotal structure}).
Here, we have identified $m'$ with its dual object by means of the pairing \eqref{eq: ev and coev for m}.

By definition, $\delta_{m'}$ is given by
\[
\delta_{m'}:\Psi((m')^*)=\Psi(m')=m\xrightarrow{\id_{m}\otimes\coev_{m}}m\otimes m\otimes m^* \xrightarrow{\Psi(\bar\ev_{m'})\otimes\id_{m^*}}m^*=\Psi(m')^*.
\]
Now, applying the definition \eqref{eq:Psi of f} of $\Psi$ to the morphism $\bar\ev_{m'}$ (which is equal to \eqref{eq: ev and coev for m} because $\psi'=\id_{m'}$), we learn that $\Psi(\bar\ev_{m'})=\bar\ev_m$.
Finally, going back to the definition \eqref{eq: ev bar and coev bar} of $\bar\ev_m$, we get $\psi=\delta_{m'}$, as desired.
We conclude that $\Psi$ is an equivalence of pointed pivotal tensor categories.

In order to show that $\Psi:(\cM',m')\to(\cM,m)$ is an isomorphism in $\Mod_*$,
it remain to show that it is a functor of module tensor categories:

\begin{lem}
$\Psi:\cM'\to\cM$ satisfies the compatibility \eqref{eqn:GammaAndHalfBriadings}
in the definition of functor between module tensor categories:
\begin{equation}\label{gwdkgjn;ho}
\begin{matrix}\xymatrix{
\Phi(c) \otimes \Psi(X) 
\ar[d]^{e_{\Phi(c), \Psi(X)}}
\ar[rr]^(.52){\cong}
&&
\Psi(\Phi'(c) \otimes X)
\ar[d]^{\Psi(e_{\Phi'(c),X})}
\\
\Psi(X) \otimes \Phi(c)
\ar[rr]^(.52){\cong}
&&
\Psi(X\otimes \Phi'(c))
}\end{matrix}
\end{equation}
\end{lem}
\begin{proof}
Without loss of generality, we take $X$ of the form $\Phi'(a)\otimes m'^{\otimes n}$.
Recall that $e_{\Phi'(c),\Phi'(a)\otimes m'^{\otimes n}}$ was defined in \eqref{eq: That's the half-braiding!}.
The composite\;\! $\tikz{\draw[-latex, scale=.38, baseline=6] (0,.9) --(1.5,.9) -- (1.5,0) -- (0,0);}$\;\! in \eqref{gwdkgjn;ho} is then given by:
$$
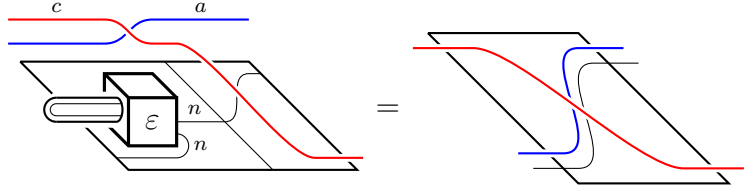
\begin{tikzpicture}[baseline=0cm, scale=.8]
	\plane{(0,-.9)}{3.8}{1.8}
	\draw (2.4,-.9) -- (.6,.9);	
	\CMbox{box}{(0,-.5)}{.8}{.8}{.4}{$\varepsilon$}
	\straightCappedTube{(-.2,-.1)}{.1}{1}
	\draw (-.1,0) -- (-1.2,0) arc (270:90:.1cm) -- (-.1,.2);
	\draw (.8,-.1) -- (1.6,-.1) .. controls ++(0:.4cm) and ++(180:.4cm) .. (2,.7) -- (2.2,.7);
	\draw (.8,-.3) arc (90:-90:.2cm) -- (-.2,-.7);
	\node at (1.1,.1) {\scriptsize{$n$}};
	\node at (1.2,-.5) {\scriptsize{$n$}};
	\node at (-1.2,1.8) {\scriptsize{$c$}};
	\node at (1.2,1.8) {\scriptsize{$a$}};
	\draw[thick, blue] (-2,1.2) -- (-.4,1.2) .. controls ++(0:.4cm) and ++(180:.4cm) .. (.4,1.6) -- (2,1.6);
	\draw[super thick, white] (-2,1.6) -- (-.4,1.6) .. controls ++(0:.4cm) and ++(180:.4cm) .. (.4,1.2) -- (.8,1.2) .. controls ++(0:.5cm) and ++(180:.5cm) .. (3.1,-.7) -- (3.9,-.7);
	\draw[thick, red] (-2,1.6) -- (-.4,1.6) .. controls ++(0:.4cm) and ++(180:.4cm) .. (.4,1.2) -- (.8,1.2) .. controls ++(0:.5cm) and ++(180:.5cm) .. (3.1,-.7) -- (3.9,-.7);
\end{tikzpicture}
=
\begin{tikzpicture}[baseline=.9cm]
	\plane{(0,0)}{2}{2}
	\draw[super thick, white] (-.8,.4) --  (-.2,.4) .. controls ++(0:.6cm) and ++(180:.6cm) .. (0,1.8) -- (.6,1.8);
	\draw[thick, blue] (-.8,.4) --  (-.2,.4) .. controls ++(0:.6cm) and ++(180:.6cm) .. (0,1.8) -- (.6,1.8);
	\draw (-.6,.2) --  (0,.2) .. controls ++(0:.6cm) and ++(180:.6cm) .. (.2,1.6) -- (.8,1.6);
	\draw[super thick, white] (-2.2,1.8) --  (-1.4,1.8) .. controls ++(0:.6cm) and ++(180:.6cm) .. (1.4,.2) -- (2.2,.2);
	\draw[thick, red] (-2.2,1.8) --  (-1.4,1.8) .. controls ++(0:.6cm) and ++(180:.6cm) .. (1.4,.2) -- (2.2,.2);
\end{tikzpicture}
$$
The latter is visibly equal to $e_{\Phi(c), \Psi(\Phi'(a)\otimes m'^{\otimes n})}=e_{\Phi(c), \Phi(a)\otimes m^{\otimes n}}$.\
\end{proof}

Recall that the functors $\Lambda:\Mod_*\to\APA$ and $\Delta:\APA \to \Mod_*$ were described in Theorems~\ref{thm: Here's Lambda!} and \ref{thm: Here's Delta!}, respectively.
We summarize the results of this section:

\begin{prop}\label{prop: That's Psi!}
For any pointed module tensor category $(\cM,m)\in\Mod_*$, there exists an isomorphism in $\Mod_*$
\[
\Psi_{(\cM,m)}:\Delta\circ\Lambda(\cM,m)\to(\cM,m).
\]
\end{prop}

Our next task is to investigate the question of whether these isomorphisms assemble to a natural transformation $\Psi:\Delta\circ\Lambda\Rightarrow \id_{\Mod_*}$.

\subsection{Naturality of \texorpdfstring{$\Psi$}{Psi}}\label{sec:CMGtoAPAtoCMG_Naturality}

Let $\cM_1=(\cM_1,m_1)$ and $\cM_2=(\cM_2,m_2)$ be pointed module tensor categories.
Let $G:\cM_1\to\cM_2$ be a functor of pointed module tensor categories (a morphism in $\Mod_*$), and let $G':\cM_1'\to\cM_2'$ be its image under $\Delta\circ\Lambda$.
Let also $\Psi_1:=\Psi_{\cM_1}:\cM_1'\to\cM_1$ and $\Psi_2:=\Psi_{\cM_2}:\cM_2'\to\cM_2$ be as in Proposition~\ref{prop: That's Psi!}.

We consider the following diagram:
\begin{equation}\label{eq: jagjrrejflwc}
\begin{matrix}\xymatrix{
\cM_1' 
\ar[rr]^{G'}
\ar[d]^{\Psi_1} 
&&
\cM_2'
\ar[d]^{\Psi_2} 
\\
\cM_1
\ar[rr]^{G}
&&
\cM_2.\!\!
}\end{matrix}
\end{equation}
Recall from Section~\ref{sec:InternalTrace} that $\tau_{\le 1}(\Mod_*)$ is obtained from $\Mod_*$ by identifying all $1$-morphisms which are connected by a (necessarily invertible) $2$-morphism.
The diagram \eqref{eq: jagjrrejflwc} does not commute in $\Mod_*$, but its image down in $\tau_{\le 1}(\Mod_*)$ does commute:
\begin{prop}
There exists an (invertible) natural transformation
\[
\kappa: \Psi_2\circ G' \Rightarrow G\circ \Psi_1
\]
of pointed module tensor category functors (a $2$-morphism in $\Mod_*$)
between the composites $\Psi_2\circ G'$ and $G\circ \Psi_1$.
\end{prop}

\begin{proof}
As usual, it is enough to define $\kappa$ on those objects of $\cM_1'$ of the form $\Phi_1'(c)\otimes m_1'^{\otimes n}$.
Let $\gamma_c:\Phi_2(c) \rightarrow G\Phi_1(c)$ be the action coherence isomorphism associated to the functor $G$.
By definition (\eqref{eq:def of Delta -- objects} and \eqref{eq: def Psi -- objects}), we have:
\begin{alignat*}{1}
\Psi_2\circ G'\big(\Phi_1'(c)\otimes m_1'^{\otimes n}\big)&=\Psi_2(\Phi_2'(c)\otimes m_2'^{\otimes n})\,=\,\Phi_2(c)\otimes m_2^{\otimes n}
\\
\text{and}\qquad\quad
G\circ \Psi_1\big(\Phi_1'(c)\otimes m_1'^{\otimes n}\big)\,&=\,G(\Phi_1(c)\otimes m_1^{\otimes n}) \;\!\cong\;\! G\Phi_1(c)\otimes m_2^{\otimes n}.\qquad\qquad
\end{alignat*}
Using the above identifications, we define the components of $\kappa$ as follows:
\begin{equation}\label{eq: That's kappa!}
\begin{split}
\kappa_{\Phi_1'(c)\otimes m_1'^{\otimes n}} :\,
\Psi_2G'(\Phi_1'(c)\otimes m_1'^{\otimes n})=&\;\Phi_2(c)\otimes m_2^{\otimes n}\\
\xrightarrow{\,\gamma_c \otimes \id_{m_2^{\otimes n}}\,}\;&
G\Phi_1(c)\otimes m_2^{\otimes n}\cong G\Psi_1(\Phi_1'(c)\otimes m_1'^{\otimes n}).
\end{split}
\end{equation}
Clearly, this satisfies the condition $\kappa_{m'_1}=\id_{m_2}$ of being pointed (Definition \ref{def: pointed}).
We must show that $\kappa$ is natural with respect to morphisms in $\cM_1'$,
monoidal (Definition \ref{Def: natural transformation for functors of module tensor categories}),
and compatible with the action coherence isomorphisms (Definition \ref{Def: natural transformation for functors of module tensor categories}, Eq.~\eqref{eq: def nat transf 2}).

We start by the naturality of $\kappa$. Given a morphism
$
f \in \cM_1'(\Phi_1'(a)\otimes m_1'^{\otimes n} , \Phi_1'(b)\otimes m_1'^{\otimes k}) = \cC(a, b\otimes \Tr_\cC^1(m_1^{\otimes k+n}))
$,
we need to show that the following diagram commutes:
\[
\xymatrix{
\Psi_2G'( \Phi_1'(a)\otimes m_1'^{n}) 
\ar[rrr]^{\kappa_{ \Phi_1'(a)\otimes m_1'^{n}}}
\ar[d]^{\Psi_2G'(f)}
&&&
G\Psi_1( \Phi_1'(a)\otimes m_1'^{n})
\ar[d]^{G\Psi_1(f)}
\\
\Psi_2G'( \Phi_1'(b)\otimes m_1'^{k})
\ar[rrr]^{\kappa_{ \Phi_1'(b)\otimes m_1'^{k}}}
&&&
G\Psi_1( \Phi_1'(b)\otimes m_1'^{k})
}
\]
Recall that $G'=\Delta\Lambda (G)$. The composite\;\! $\tikz{\draw[-latex, scale=.38, baseline=6] (0,.8) -- (0,0) -- (1.5,0);}$\;\! is given by:
\begin{align*}
\!&\hspace{.53cm} (\gamma_b \otimes \id_{m_2^{\otimes k}})\circ (\id_{\Phi_2(b)}\otimes \id_{m_2^{\otimes k}} \otimes\;\! \bar\ev_{m_2^{\otimes n}})\circ (\id_{\Phi_2(b)}\otimes \varepsilon^2_{m_2^{\otimes k+n}} \otimes \id_{m_2^n})\circ (\Phi_2(G'(f))\otimes \id_{m_2^{\otimes n}}) 
\put(-62,-15){\footnotesize [Def.~\eqref{eq:Psi of f} of $\Psi_2$]}
\\\!&=
(\gamma_b \otimes \id_{m_2^{\otimes k}} \otimes\;\! \bar\ev_{m_2^{\otimes n}})\circ (\id_{\Phi_2(b)}\otimes \varepsilon^2_{m_2^{\otimes k+n}} \otimes \id_{m_2^{\otimes n}})\circ (\Phi_2[(\id_b \otimes \Lambda{G}[k+n]) \circ f]\otimes \id_{m_2^{\otimes n}})
\put(-60,-15){\footnotesize [Def.~\eqref{eq:def of Delta} of $\Delta\Lambda G$]}
\\\!&=
(\gamma_b \otimes \id_{m_2^{\otimes k}} \otimes\;\! \bar\ev_{m_2^{\otimes n}})\circ (\id_{\Phi_2(b)}\otimes\, [\varepsilon^2_{m_2^{\otimes k+n}}\circ \Phi_2( \zeta_{m_1^{\otimes k+n}})] \otimes \id_{m_2^{\otimes n}})\circ (\id_{\Phi_2(b)}\otimes \Phi_2(f)\otimes \id_{m_2^{\otimes n}}) 
\put(-70,-15){\footnotesize [Def.~\ref{def: morphism of APA associated to G:M-->M'} of $\Lambda G$]}
\\\!&=
(\gamma_b \otimes \id_{m_2^{\otimes k}} \otimes\;\! \bar\ev_{m_2^{\otimes n}}){\circ} (\id_{\Phi_2(b)}{\otimes} [G(\varepsilon^1_{m_1^{\otimes k+n}}) \circ \gamma_{\Tr_\cC^1(m_1^{\otimes k+n})}] {\otimes} \id_{m_2^{\otimes n}}){\circ} (\id_{\Phi_2(b)}\otimes \Phi_2(f)\;\!{\otimes} \id_{m_2^{\otimes n}}) 
\put(-54,-15){\footnotesize [Lemma \ref{lem:AttachingMapCompatible}]}
\\\!&=
(\id_{G\Phi_1(b)} \otimes \id_{m_2^{\otimes k}} \otimes\;\! \bar\ev_{m_2^{\otimes n}}){\circ} (\id_{G\Phi_1(b)}\otimes\;\! G(\varepsilon^1_{m_1^{\otimes k+n}}) \;\!{\otimes} \id_{m_2^{\otimes n}}){\circ} ([\gamma_{b \otimes \Tr_\cC^1(m_1^{\otimes k+n})}{\circ} \Phi_2(f)] \otimes \id_{m_2^{\otimes n}}) 
\put(-65,-15){\footnotesize [$\gamma$ is monoidal]}
\\\!&=
(\id_{G\Phi_1(b)} \otimes \id_{m_2^{\otimes k}} \otimes\;\! \bar\ev_{m_2^{\otimes n}})\circ (\id_{G\Phi_1(b)}\otimes\;\! G(\varepsilon^1_{m_1^{\otimes k+n}}) \otimes \id_{m_2^{\otimes n}})\circ ([G\Phi_1(f) \circ \gamma_a] \otimes \id_{m_2^{\otimes n}}) 
\put(-30,-15){\footnotesize [$\gamma$ is natural]}
\\\!&=
(\id_{G\Phi_1(b)} \otimes \id_{m_2^{\otimes k}} \otimes\;\! \bar\ev_{m_2^{\otimes n}}){\circ} (\id_{G\Phi_1(b)}\otimes\;\! G(\varepsilon^1_{m_1^{\otimes k+n}}) \;\!{\otimes} \id_{m_2^{\otimes n}}){\circ} (G\Phi_1(f)\otimes \id_{m_2^{\otimes n}}){\circ}  (\gamma_a \;\!{\otimes} \id_{m_2^{\otimes n}}).
\end{align*}
Using the definition \eqref{eq:Psi of f} of $\Psi_1$ and the fact that $G(\bar\ev_{m_1}^{\otimes n})=\bar\ev_{m_2}^{\otimes n}$ (coming from the fact that $G$ is a pivotal functor), one recognizes the last line as the composite 
$\tikz[scale=.38]{\useasboundingbox (.02,.8) -- (1.52,.8) -- (1.52,0); \draw[-latex]  (0,.8) -- (1.5,.8) -- (1.5,-.1);}$\;\!\;\! in the above diagram.


To show that $\kappa$ is monoidal,
we argue that for every pair of objects $X,Y\in\cM_1'$, the following diagram is commutative (Definition \ref{Def: natural transformation for functors of module tensor categories}, Eq.~\eqref{eq: def nat transf 1}):
\begin{equation}\label{eq:KappaMonoidal}
\begin{matrix}
\xymatrix{
\Psi_2G'(X\otimes Y)
\ar[d]^{\kappa_{X\otimes Y}}
\ar[rr]^(.45){\cong}
&&
\Psi_2(G'(X)\otimes G'(Y))
\ar[rr]^{\cong}
&&
\Psi_2G'(X)\otimes \Psi_2G'(X)
\ar[d]^{\kappa_X\otimes \kappa_Y}
\\
G\Psi_1(X\otimes Y)
\ar[rr]^(.45){\cong}
&&
G(\Psi_1(X)\otimes \Psi_1(Y))
\ar[rr]^{\cong}
&&
G\Psi_1(X)\otimes G\Psi_1(Y).\!
}
\end{matrix}
\end{equation}
Without loss of generality, we take $X=\Phi_1'(a)\otimes m_1'^{\otimes n}$ and $Y=\Phi_1'(b)\otimes m_1'^{\otimes k}$.
We rewrite \eqref{eq:KappaMonoidal} as the boundary of the following pasting diagram:
\[
\!
\xymatrix{
\scriptstyle
\Phi_2(a\otimes b)\otimes m_2^{n+k}   
\ar[d]^{\scriptscriptstyle\gamma_{a\otimes b} \otimes \id}
\ar[rr]^(.45){\scriptscriptstyle\cong}
&&
\scriptstyle
\Phi_2(a)\otimes \Phi_2(b)\otimes m_2^{n} \otimes m_2^{k}
\ar[d]^{\scriptscriptstyle\gamma_a\otimes \gamma_b \otimes \id\otimes \id}
\ar[rr]^{\scriptscriptstyle \id\otimes e_{\Phi_2(b),m_2^{n}}\otimes \id}
&&
\scriptstyle
\Phi_2(a)\otimes m_2^{n}\otimes \Phi_2(b) \otimes m_2^{k}
\ar[ddd]|{\qquad\quad\scriptscriptstyle\gamma_a \otimes \id \otimes \gamma_b \otimes \id}
\\
\scriptstyle
G\Phi_1(a\otimes b)\otimes m_2^{n+k}
\ar[dd]_{\scriptscriptstyle\cong}
\ar[rr]^(.4){\scriptscriptstyle\cong}
&&
\scriptstyle
G\Phi_1(a)\otimes G\Phi_1(b)\otimes G(m_1^n) \otimes G(m_1^k)
&&
\\
\\
\scriptstyle
G(\Phi_1(a)\otimes \Phi_1(b)\otimes m_1^n\otimes m_1^k)
\ar[rr]^{\scriptscriptstyle G{\scriptstyle(}\id\otimes e_{\Phi_1(b),m_1^{n}}\otimes \id{\scriptstyle)}}
\ar@{}[uurr]|{\,\,G\Phi_1(a)\otimes G(\Phi_1(b) \otimes m_1^n)\otimes  G(m_1^k)\,\,}="a"
&&
\scriptstyle
G(\Phi_1(a)\otimes m_1^n\otimes \Phi_1(b)\otimes m_1^k)
\ar[rr]^{\scriptscriptstyle\cong}
&&
\scriptstyle
G\Phi_1(a)\otimes m_2^{n}\otimes G\Phi_1(b) \otimes m_2^{k}
\ar@{}[uull]|{\,\,G\Phi_1(a)\otimes G(m_1^n\otimes \Phi_1(b))\otimes  G(m_1^k)\,\,}="b"
\ar "a";"b" ^{\scriptscriptstyle\id\otimes G(e_{\Phi_1(b),m_1^n})\otimes \id}
\ar [-2,-2];"a"_(.55){\scriptscriptstyle\cong}
\ar "a";[0,-4]_(.55){\scriptscriptstyle\cong}
\ar "b";[0,-2]_(.55){\scriptscriptstyle\cong}
\ar "b";[0,0]^(.55){\scriptscriptstyle\cong}
}
\]
The upper left rectangle commutes because $\gamma$ is monoidal.
The irregular hexagon commutes because $\gamma$ is compatible with the half-braidings \eqref{eqn:GammaAndHalfBriadings}.
The other three regions are visibly commutative.
The commutativity of \eqref{eq:KappaMonoidal} follows.

The unitality condition\vspace{-.5mm}
$\kappa_{1_{\cM_1'}}\circ i_{\Psi_2\circ G'}=i_{G\circ \Psi_1}$ for $\kappa$ (Definition \ref{Def: natural transformation for functors of module tensor categories})
is a consequence of the corresponding unitality condition $\gamma_{1_\cC}=G(i_{\Psi_1})\circ i_G\circ i_{\Phi_2}^{-1}$ of $\gamma$ (Definition~\ref{defn:ModuleTensorCategoryFunctor}).\vspace{.5mm}


At last, we show that $\kappa$ is
compatible with the action coherence isomorphisms (Definition~\ref{Def: natural transformation for functors of module tensor categories}, Eq.~\eqref{eq: def nat transf 2}):
\begin{equation}\label{eq: gamma stuff}
\begin{matrix}
\xymatrix@C=1.3cm{
\Phi_2(c) 
\ar[rr]^(.45){\gamma^{\Psi_2}_{c}}
\ar[d]^{=}
&&
\Psi_2\Phi_2'(c)
\ar[rr]^(.45){\Psi_2(\gamma^{G'}_c)}
&&
\Psi_2G'(\Phi'_1(c))
\ar[d]^{\kappa_{\Phi_1'(c)}}
\\
\Phi_2(c)
\ar[rr]^(.45){\gamma^G_c}
&&
G\Phi_1(c)
\ar[rr]^(.45){G(\gamma^{\Psi_1}_c)}
&&
G\Psi_1(\Phi_1'(c)).\!
}\end{matrix}
\end{equation}
Here, $\gamma^F$ refers to the action coherence isomorphism associated to the functor $F$.
By definition, $\gamma^{G'}_c=\id$ (line below \eqref{eq:def of Delta -- objects}) and $\gamma^{\Psi_i}_c=\id$ (line below \eqref{eq:Psi of f}).
The diagram \eqref{eq: gamma stuff} therefore reduces to the equation $\gamma_c^G=\kappa_{\Phi_1'(c)}$, which holds by the definition \eqref{eq: That's kappa!} of~$\kappa$.
\end{proof}

Combining all the results of Sections \ref{sec:CMGtoAPAtoCMG} and \ref{sec:CMGtoAPAtoCMG_Naturality}, we have proven:

\begin{thm}
The composite
\[
\Delta\circ\Lambda:\tau_{\le 1}(\Mod_*)\to\APA\to\tau_{\le 1}(\Mod_*)
\]
is naturally isomorphic to the identity functor $\id_{\tau_{\le 1}(\Mod_*)}$.\hfill $\square$
\end{thm}

\subsection{Planar algebras to module tensor categories and back}\label{sec:APAtoCMGtoAPA}

In this section, we finish the proof that $\Lambda:\Mod_* \rightleftarrows \APA :\Delta$ form an equivalence of categories
by arguing that $\Lambda\circ\Delta$ is naturally isomorphic to $\id_{\APA}$.
If one takes \eqref{eq: that's how you define Tr!} as our definition of $\Tr_\cC$, then
we will prove that there is actually an equality $\Lambda\circ\Delta=\id_{\APA}$ on the nose.

Let $\cP$ be an anchored planar algebra and let $\cP'$ be its image under $\Lambda\circ\Delta$.
Let $(\cM, m):=\Delta(\cP)$ so that, following \eqref{eq:   cP[n]:=Tr_cC(m^otimes n)}, $\cP'[n]=\Tr_\cC(m^{\otimes n})$.
If one takes \eqref{eq: that's how you define Tr!} as our definition of $\Tr_\cC$ then, indeed, $\Tr_\cC(m^{\otimes n})=\cP[n]$,
and so we have $\cP[n]=\cP'[n]$.
We still need to check that the action of the generating tangles $u$, $a_i$, $\overline{a}_i$, $p_{ij}$ (Definition \ref{defn:GeneratingTangles}) given by $\cP$ and by $\cP'$ agree.
Specifically, we need to check:
\[
\eta=Z(u),\quad
\alpha_i=Z(a_i),\quad
\overline{\alpha}_i=Z(\overline{a}_i)\quad\,\text{and}\quad\,
\varpi_{i,j}=Z(p_{i,j}),
\]
where $Z(-)$ denotes the action of $\cP$,
and the maps $\eta$, $\alpha_i$, $\bar\alpha_i$, $\varpi_{ij}$ are as described in Theorem~\ref{thm: construct P from M and m}.
The value of any tangle being determined by the values of the generating tangles (Algorithm~\ref{alg:AssignMap}), the above conditions will then imply that $\cP$ and $\cP'$
are equal as anchored planar algebras.

Recall that a morphism
$
f\in \cM(\Phi(c)\otimes m^{\otimes n_1} , \Phi(d)\otimes m^{\otimes n_2}) := \cC(c, d\otimes \cP[n_2 + n_1])
$
is represented diagrammatically by 
\begin{equation}\label{nwrblsfgbowjr}
\begin{tikzpicture}[baseline=-.1cm]
	\draw (0,.8) -- (0,-.8);
	\draw (0,0) -- (2,0);
	\roundNbox{unshaded}{(0,0)}{.3}{0}{0}{$f$};
	\node at (-.2,.6) {\scriptsize{$d$}};
	\node at (1.1,.2) {\scriptsize{$\cP[n_2+n_1]$}};
	\node at (-.2,-.6) {\scriptsize{$c$}};
\end{tikzpicture}\,.
\end{equation}

\begin{lem}\label{LEM 1 of the end}
For all $i=1,\dots, n-1$
\begin{enumerate}[label=(\arabic*)]
\item
the map $\alpha_i: \cP[n+2]\to \cP[n]$ is equal to $Z(a_i):\cP[n+2]\to \cP[n]$, and
\item
the map $\overline{\alpha}_i : \cP[n]\to \cP[n+2]$ is equal to $Z(\overline{a}_i):\cP[n]\to \cP[n+2]$.
\end{enumerate}
\end{lem}

\begin{proof}
We only prove the first statement, the other one is similar.
Recall from Theorem \ref{thm: construct P from M and m} that $\alpha_i$ is the image of $\id_{i} \otimes\,\bar\ev_m \otimes \id_{n-i} : m^{\otimes n+2} \to m^{\otimes n}$
under $\Tr_\cC$.

By the definition \eqref{eq: ev and coev for m} of $\ev_m$, the fact that $\ev_m=\bar\ev_m$ (paragraph after \eqref{eq: ev and coev for m}),
and the definition \eqref{eq:TensorProductInM} of the tensor product in $\cM$, we have

\[
\id_{i} \otimes\,\bar\ev_m \otimes \id_{n-i}
=
\begin{tikzpicture}[baseline=-.1cm]
	\node at (1.2,.2) {\scriptsize{$\cP[2n+2]$}};
	\draw (0,0) -- (2,0);
\draw[thick, dotted] (-1.1,-.8)node[left, xshift=1, yshift=4]{$\scriptstyle 1_\cC$} -- (-1.1,.8);
	\roundNbox{unshaded}{(-.2,0)}{.6}{0}{0}{}
	\draw[very thick] (-.2,0) circle (.5cm);	
	\filldraw[red] (-.7,0) circle (.05cm);
	\draw (-.5,-.38) -- (-.5,.38);
	\draw (.1,-.38) -- (.1,.38);
	\draw (-.35,-.47) -- (-.35,-.2) arc (180:0:.15cm) -- (-.05,-.47);
\end{tikzpicture}\,.
\]
Therefore, by \eqref{eq: That's Tr_cC(f) }, $\alpha_i:=\Tr_\cC(\id_{i} \otimes\,\bar\ev_m \otimes \id_{n-i})$ satisfies:
\newcommand{\multiplicationforhere}[5]{
	\draw ($ #1 + 5/6*(0,#2) $) -- ($ #1 - 5/6*(0,#2) $);
	\draw[thick, red] ($ #1 + 1/3*(0,#2) - 1/5*(#2,0) $) -- ($ #1 - 4/5*(#2,0) $);
	\draw[thick, red] ($ #1 - 1/3*(0,#2) - 1/5*(#2,0) $) -- ($ #1 - 4/5*(#2,0) $);
	\draw[very thick] #1 ellipse ({4/5*#2} and {5/6*#2});
	\filldraw[very thick, unshaded] ($ #1 + 1/3*(0,#2) $) circle (1/5*#2);
	\filldraw[very thick, unshaded] ($ #1 - 1/3*(0,#2) $) circle (1/5*#2);
	\node at ($ #1 + (.15,.05) + 5/8*(0,#2)$) {\scriptsize{$#5$}};
	\node at ($ #1 + (.15,0) $) {\scriptsize{$#4$}};
	\node at ($ #1 + (.15,-.02) - 5/8*(0,#2)$) {\scriptsize{$#3$}};
}
\[
\alpha_i=
\begin{tikzpicture}[baseline=-.8cm]
	\node at (.6,-2.1) {\scriptsize{$\cP[n+2]$}};
	\node at (3.6,-.5) {\scriptsize{$\cP[n]$}};
	\draw (0,0) -- (2,0);
	\draw (2,-.7) -- (4,-.7);
	\draw (0,-2.2) -- (0,-1.8) arc (180:90:.4cm) -- (2,-1.4);
	\multiplication{(2,-.7)}{1.2}{0}{\hspace{.37cm}n{+}2}{n}
\draw[thick, dotted] (-1.1,-2.15)node[left, xshift=1, yshift=2]{$\scriptstyle 1_\cC$} -- (-1.1,.7);
	\roundNbox{unshaded}{(-.2,0)}{.6}{0}{0}{}
	\draw[very thick] (-.2,0) circle (.5cm);	
	\filldraw[red] (-.7,0) circle (.05cm);
	\draw (-.5,-.38) -- (-.5,.38);
	\draw (.1,-.38) -- (.1,.38);
	\draw (-.35,-.47) -- (-.35,-.2) arc (180:0:.15cm) -- (-.05,-.47);
\end{tikzpicture}\,
=
Z\Bigg(\tikz[baseline=-2]{\multiplicationforhere{(0,0)}{1}{0}{\hspace{.37cm}n{+}2}{n}}\Bigg)
\circ_1
Z\bigg(\!\;\!\tikz[baseline=-2]{
	\draw[very thick] (-.2,0) circle (.5cm);	
	\filldraw[red] (-.7,0) circle (.05cm);
	\draw (-.5,-.38) -- (-.5,.38);
	\draw (.1,-.38) -- (.1,.38);
	\draw (-.35,-.47) -- (-.35,-.2) arc (180:0:.15cm) -- (-.05,-.47);}
\!\;\!\bigg)
=
Z(a_i).
\]
\end{proof}

\begin{lem}\label{LEM 2 of the end}
The map $\eta:1_\cC\to \cP[0]$ is equal to $Z(u)$.
\end{lem}

\begin{proof}
By definition, $\eta$ is the unit of the adjunction $\Phi \dashv \Tr_\cC$ evaluated on $1_\cC\in\cC$ (Theorem~\ref{thm: construct P from M and m} and \ref{rel:UnitMap}).
Equivalently, it is the image of $\id_{\Phi(1_\cC)}$ under the bijection $\cM(\Phi(1_\cC),\Phi(1_\cC))\cong \cC(1_\cC,\Tr_\cC(\Phi(1_\cC)))$ provided by the adjunction.
The latter is described in Lemma \ref{lem: exhibits Tr as the right adjoint}. It is the identity map
\[
\cM(\Phi(1_\cC),\Phi(1_\cC)):=\cC(1_\cC,\cP[0]) \,\to\, \cC(1_\cC,\cP[0])=\cC(1_\cC,\Tr_\cC(\Phi(1_\cC))).
\]
Therefore $\eta=\id_{\Phi(1_\cC)}$.
Combining this with the definition \eqref{eq: ...and that's how one defines identities} of identities in $\cM$, it follows that $\eta=\id_{\Phi(1_\cC)}=Z(u)$.
\end{proof}

Before checking the relation $\varpi_{i,j}=Z(p_{i,j})$, we investigate the traciator \ref{rel:Traciator} and multiplication map \ref{rel:MultiplicationMap} in the category $\cM=\Delta(\cP)$.

\begin{lem}
\label{lem:ZPrimeOfRotation}
The traciator 
$$
\tau_{m^{\otimes i},m^{\otimes n-i}}\in \cC\big(\Tr_\cC(m^{\otimes i}\otimes m^{\otimes n-i}), \Tr_\cC(m^{\otimes n-i}\otimes m^{\otimes i})\big) =\cC\big(\cP[n],\cP[n]\big)
$$ 
is equal to $Z(t_i)$, where $t_i$ is the rotation tangle 
$$
t_i =
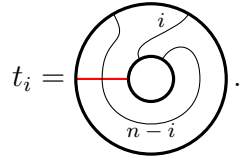
\begin{tikzpicture}[baseline=-.1cm]
	\draw (0,0) -- (120:.3cm) .. controls ++(120:.3cm) and ++(-120:.3cm) .. (60:1cm);
	\draw[thick, red] (180:.3cm) -- (180:1cm);
	\draw (0,0) -- (60:.3cm) .. controls ++(60:.3cm) and ++(90:.5cm) .. (0:.65cm) .. controls ++(270:.8cm) and ++(270:.8cm) .. (180:.65cm) .. controls ++(90:.6cm) and ++(-60:.5cm) .. (120:1cm);
	\draw[very thick] (0,0) circle (1cm);
	\draw[unshaded, very thick] (0,0) circle (.3cm);
	\node at (80:.8cm) {\scriptsize{$i$}};
	\node at (-90:.7cm) {\scriptsize{$n-i$}};
\end{tikzpicture}
\,.
$$
As similar result holds true for $\tau^{-1}_{m^{\otimes i},m^{\otimes n-i}}$ and the inverse rotation tangle.
\end{lem}
\begin{proof}
Set $j=n-i$.
Recall from \ref{rel:Traciator} that $\tau_{m^{\otimes i},m^{\otimes j}}$ is the mate of 
\[
	(\id_{m^{\otimes n}}\otimes\, \bar\ev_{m^{\otimes j}})
\circ	(\id_{m^{\otimes j}}\otimes\, \varepsilon_{m^{\otimes n}}\otimes \id_{m^{\otimes j}})
\circ	(e_{\Phi(\cP[n]),m^{\otimes j}}\otimes \id_{m^{\otimes j}})
\circ	(\id_{\Phi(\cP[n])}\otimes\,\bar\coev_{m^{\otimes j}})
\]
under the bijection $\cC(\cP[n],\cP[n])\cong\cM(\Phi(\cP[n]),m^{\otimes n})$.
The latter is the identity map, and may therefore be suppressed (Lemma \ref{lem: exhibits Tr as the right adjoint}).
Using the definition \eqref{eq: That's the half-braiding!} of the half-braiding, we can then compute:
$$
\tau_{m^{\otimes i},m^{\otimes j}}\,=\!\!
\begin{tikzpicture}[baseline=-.3cm]
	\draw (0,1) -- (3,1);
	\draw (2,0) -- (4,0);
	\draw (1,-1) -- (3,-1);
	\draw (0,-.8) -- (0,-.4) arc (180:90:.4cm) -- (2,0);
	\evaluationMap{(0,1)}{.4}{j}
	\tensorLeftIdEv{(1.2,1)}{.4}{}{}{}
	\coevaluationMap{(1.2,-1)}{.4}{j}
	\roundNbox{}{(0,0)}{.4}{0}{0}{};
	\tensorLeftRightIdCoev{(1.2,0)}{.4}
	\node at (.25,-1.8) {\scriptsize{$1_\cC$}};
	\node at (-1,-1.8) {\scriptsize{$\cP[n]$}};
	\draw (-.6,-2) -- (-.6,-1.6);
	\draw[dotted] (0,-2) -- (0,-1.6);
	\draw[dotted] (-.6,-.8) -- (-.6,1.6);
	\draw[dotted] (0,-1.6)  .. controls ++(90:.3cm) and ++(270:.3cm) .. (-.6,-.8);
	\draw[super thick, white] (-.6,-1.6)  .. controls ++(90:.3cm) and ++(270:.3cm) .. (0,-.8);
	\draw (-.6,-1.6)  .. controls ++(90:.3cm) and ++(270:.3cm) .. (0,-.8);
	\draw[rounded corners=5pt, very thick, unshaded] (2,-1.5) rectangle (4.4,1.5);
	\coordinate (a) at (2.15,0);
	\draw[thick, red] (a) -- (2.9,.8);
	\draw[thick, red] (a) -- (2.9,0);
	\draw[thick, red] (a) -- (2.9,-.8);
	\draw (4.4,0) -- (5.2,0);
	\node at (4.8,.2) {\scriptsize{$\cP[n]$}};
	\draw[very thick] (3.2,0) ellipse (1.05 and 1.35);
	\draw (3.1,-.7) -- (3.1,1.35);
	\filldraw[very thick, unshaded] (3.1,.8) circle (.2cm);
	\filldraw[very thick, unshaded] (3.1,0) circle (.2cm);
	\filldraw[very thick, unshaded] (3.1,-.8) circle (.2cm);
	\node at (3.3,1.2) {\scriptsize{$n$}};
	\node at (3.6,.4) {\scriptsize{$n+2j$}};
	\node at (3.4,-.4) {\scriptsize{$2j$}};
\end{tikzpicture}
\,=
Z\Bigg(
\begin{tikzpicture}[scale=.9, baseline=-.1cm]
	\draw (0,0) -- (120:.3cm) .. controls ++(120:.3cm) and ++(-120:.3cm) .. (60:1cm);
	\draw[thick, red] (180:.3cm) -- (180:1cm);
	\draw (0,0) -- (60:.3cm) .. controls ++(60:.3cm) and ++(90:.5cm) .. (0:.65cm) .. controls ++(270:.8cm) and ++(270:.8cm) .. (180:.65cm) .. controls ++(90:.6cm) and ++(-60:.5cm) .. (120:1cm);
	\draw[very thick] (0,0) circle (1cm);
	\draw[unshaded, very thick] (0,0) circle (.3cm);
	\node at (80:.8cm) {\scriptsize{$i$}};
	\node at (-90:.78cm) {\scriptsize{$j$}};
\end{tikzpicture}
\Bigg).
$$
\end{proof}

\begin{lem}
\label{lem:ZPrimeOfMultiplication}
\newcommand{\tensorforhere}[5]{
	\draw ($ #1 - 1/3*(#2,0) $) -- ($ #1 - #4*16/23*(0,#2) - 1/3*(#2,0) $);
	\draw ($ #1 + 1/3*(#2,0) $)  -- ($ #1 - #4*16/23*(0,#2) + 1/3*(#2,0) $);
	\draw[thick, red] ($ #1 - 1/3*(#2,0) - 1/5*(#2,0) $) -- ($ #1 - 5/6*(#2,0) $);
	\draw[thick, red] ($ #1 + 1/3*(#2,0) - 1/5*(#2,0) $) .. controls ++(180:.2cm) and ++(0:.2cm) .. ($ #1 - 1/3*(#2,0) + 2/5*(0,#2) $) .. controls ++(180:.2cm) and ++(0:.2cm) .. ($ #1 - 5/6*(#2,0) $);
	\draw[very thick] #1 ellipse ( {5/6*#2} and {78/100*#2});
	\filldraw[very thick, unshaded] ($ #1 + 1/3*(#2,0) $) circle (1/5*#2);
	\filldraw[very thick, unshaded] ($ #1 - 1/3*(#2,0) $) circle (1/5*#2);
	\node at ($ #1 + (.2,0) - 1/3*(#2,0) - #4*.4*(0,#2) $) {\scriptsize{$#3$}};
	\node at ($ #1 + (.2,0) + 1/3*(#2,0) - #4*.4*(0,#2) $) {\scriptsize{$#5$}};
}

The multiplication map 
\[
\mu_{m^{\otimes i},m^{\otimes j}} 
\in\,
\cC\big(\Tr_\cC(m^{\otimes i}) \otimes \Tr_\cC(m^{\otimes j}), \Tr_\cC(m^{\otimes i+j})\big) 
= \cC\big(\cP[i] \otimes \cP[j], \cP[i+j]\big)
\]
is equal to
$Z\Bigg(\tikz[baseline=10, yscale=-.9, xscale=.9]{\tensorforhere{(3.2,-.5)}{1.2}{i\;}{1}{j\,}}\Bigg)$.
\end{lem}

\begin{proof}
Recall from \ref{rel:MultiplicationMap} that $\mu_{m^{\otimes i},m^{\otimes j}}$ is the mate of $\varepsilon_{m^{\otimes i}}\otimes \varepsilon_{m^{\otimes j}}$
under the adjunction $\cC(\Tr_\cC(m^{\otimes i}) \otimes \Tr_\cC(m^{\otimes j}), \Tr_\cC(m^{\otimes i+j}))
\cong \cM(\Phi(\Tr_\cC(m^{\otimes i}) \otimes \Tr_\cC(m^{\otimes j})), m^{\otimes i+j})$.
As in the previous proofs, we suppress that bijection and just write $\mu_{m^{\otimes i},m^{\otimes j}}=\varepsilon_{m^{\otimes i}}\otimes \varepsilon_{m^{\otimes j}}$.

\newcommand{\tensorforhere}[5]{
	\draw ($ #1 - 1/3*(#2,0) $) -- ($ #1 - #4*16/23*(0,#2) - 1/3*(#2,0) $);
	\draw ($ #1 + 1/3*(#2,0) $)  -- ($ #1 - #4*16/23*(0,#2) + 1/3*(#2,0) $);
	\draw[thick, red] ($ #1 - 1/3*(#2,0) - 1/5*(#2,0) $) -- ($ #1 - 5/6*(#2,0) $);
	\draw[thick, red] ($ #1 + 1/3*(#2,0) - 1/5*(#2,0) $) .. controls ++(180:.2cm) and ++(0:.2cm) .. ($ #1 - 1/3*(#2,0) + 2/5*(0,#2) $) .. controls ++(180:.2cm) and ++(0:.2cm) .. ($ #1 - 5/6*(#2,0) $);
	\draw[very thick] #1 ellipse ( {5/6*#2} and {78/100*#2});
	\filldraw[very thick, unshaded] ($ #1 + 1/3*(#2,0) $) circle (1/5*#2);
	\filldraw[very thick, unshaded] ($ #1 - 1/3*(#2,0) $) circle (1/5*#2);
	\node at ($ #1 + (.2,0) - 1/3*(#2,0) - #4*.4*(0,#2) $) {\scriptsize{$#3$}};
	\node at ($ #1 + (.2,0) + 1/3*(#2,0) - #4*.4*(0,#2) $) {\scriptsize{$#5$}};
}

Under the adjunction $\cC(\cP[i],\cP[i])\cong\cM(\Phi(\cP[i]),m^{\otimes i})$, the counit $\varepsilon_{m^{\otimes i}}$ corresponds to the identity $\id_{\cP[i]}$.
In the notation~\eqref{nwrblsfgbowjr}, we represented it by a through-strand:
\begin{equation}\label{it's a through-strand}
\varepsilon_{m^{\otimes i}} \,= \tikz[baseline=15]{\draw (0,0)node[left, yshift=5, xshift=1]{$\scriptstyle \cP[i]$} -- ++(0,.8) arc(180:90:.4) -- ++(.8,0);}\,\,\,\in\,\cM\big(\Phi(\cP[i]),m^{\otimes i}\big).
\end{equation}
We now use the definition \eqref{eq:TensorProductInM} of the tensor product in $\cM$
to compute $\mu_{m^{\otimes i},m^{\otimes j}}=\varepsilon_{m^{\otimes i}}\otimes \varepsilon_{m^{\otimes j}}$:
$$
\begin{tikzpicture}[baseline=-.3cm]
	\draw (-1,-1.6) -- (-1,-1) arc (180:90:.4cm) -- (2,-.6);
	\draw[super thick, white] (0,-1.6) -- (0,.2) arc (180:90:.4cm) -- (2,.6);
	\draw (0,-1.6) -- (0,.2) arc (180:90:.4cm) -- (2,.6);
	\draw (3,0) -- (4.2,0);
	\roundNbox{}{(0,.6)}{.4}{0}{0}{};
	\roundNbox{}{(-1,-.6)}{.4}{0}{0}{};
	\tensor{(2,0)}{1}{0\,}{i\,}{0\,}{j\,};
	\node at (3.6,.2) {\scriptsize{$\cP[i+j]$}};
	\node at (-1.3,-1.4) {\scriptsize{$\cP[i]$}};
	\node at (-.3,-1.4) {\scriptsize{$\cP[j]$}};
\end{tikzpicture}
\,=
Z\Bigg(\tikz[baseline=-16, scale=.9]{\tensorforhere{(3.2,-.5)}{1.2}{i\;}{-1}{j\,}}
\Bigg)\circ\beta_{\cP[j],\cP[i]}^{-1}
=\,
Z\Bigg(\tikz[baseline=10, yscale=-.9, xscale=.9]{\tensorforhere{(3.2,-.5)}{1.2}{i\;}{1}{j\,}}
\Bigg).
$$
\end{proof}

As a consequence of the Lemmas \ref{lem:ZPrimeOfRotation} and \ref{lem:ZPrimeOfMultiplication}, we get:
\begin{lem}\label{LEM 3 of the end}
The relation $\varpi_{i,j}=Z(p_{i,j})$ holds.
\end{lem}
\begin{proof}
\newcommand{\tensorforhere}[5]{
	\draw ($ #1 - 1/3*(#2,0) $) -- ($ #1 - #4*16/23*(0,#2) - 1/3*(#2,0) $);
	\draw ($ #1 + 1/3*(#2,0) $)  -- ($ #1 - #4*16/23*(0,#2) + 1/3*(#2,0) $);
	\draw[thick, red] ($ #1 - 1/3*(#2,0) - 1/5*(#2,0) $) -- ($ #1 - 5/6*(#2,0) $);
	\draw[thick, red] ($ #1 + 1/3*(#2,0) - 1/5*(#2,0) $) .. controls ++(180:.2cm) and ++(0:.2cm) .. ($ #1 - 1/3*(#2,0) + 2/5*(0,#2) $) .. controls ++(180:.2cm) and ++(0:.2cm) .. ($ #1 - 5/6*(#2,0) $);
	\draw[very thick] #1 ellipse ( {5/6*#2} and {78/100*#2});
	\filldraw[very thick, unshaded] ($ #1 + 1/3*(#2,0) $) circle (1/5*#2);
	\filldraw[very thick, unshaded] ($ #1 - 1/3*(#2,0) $) circle (1/5*#2);
	\node at ($ #1 + (.2,0) - 1/3*(#2,0) - #4*.4*(0,#2) $) {\scriptsize{$#3$}};
	\node at ($ #1 + (.2,0) + 1/3*(#2,0) - #4*.4*(0,#2) $) {\scriptsize{$#5$}};
}
By the definition of $\varpi_{i,j}$ (Theorem~\ref{thm: construct P from M and m})
and by the Lemmas \ref{lem:ZPrimeOfRotation} and \ref{lem:ZPrimeOfMultiplication}, we have:
\[
\begin{split}
\varpi_{i,j}
&= \tau^{-1}_{m^{\otimes i+j},m^{\otimes n-i}} \circ \mu_{m^{\otimes n},m^{\otimes j}} \circ ( \tau_{m^{\otimes i},m^{\otimes n-i}} \otimes  \id_{\cP[j]})\\
&= \tau^{-1}_{m^{\otimes i+j},m^{\otimes n-i}} \circ_{\scriptscriptstyle 1} \mu_{m^{\otimes n},m^{\otimes j}} \circ_{\scriptscriptstyle 1} \tau_{m^{\otimes i},m^{\otimes n-i}}\\
&= Z\Bigg(
\begin{tikzpicture}[baseline=-.1cm, xscale=-.9, yscale=.9]
	\draw (0,0) -- (120:.3cm) .. controls ++(120:.3cm) and ++(-120:.3cm) .. (60:1cm);
	\draw[thick, red] (0:.3cm) -- (0:1cm);
	\draw (0,0) -- (60:.3cm) .. controls ++(60:.3cm) and ++(90:.5cm) .. (0:.65cm) .. controls ++(270:.8cm) and ++(270:.8cm) .. (180:.65cm) .. controls ++(90:.6cm) and ++(-60:.5cm) .. (120:1cm);
	\draw[very thick] (0,0) circle (1cm);
	\draw[unshaded, very thick] (0,0) circle (.3cm);
	\node at (85:.79cm) {\scriptsize{$i{+}j$}};
	\node at (-90:.7cm) {\scriptsize{$n-i$}};
\end{tikzpicture}
\Bigg)
\circ_1
Z\Bigg(
\tikz[baseline=11, yscale=-.97, xscale=.97]{\tensorforhere{(3.2,-.5)}{1.2}{n\;}{1}{j\,}}
\Bigg)
\circ_1
Z\Bigg(
\begin{tikzpicture}[baseline=-.1cm, scale=.9]
	\draw (0,0) -- (120:.3cm) .. controls ++(120:.3cm) and ++(-120:.3cm) .. (60:1cm);
	\draw[thick, red] (180:.3cm) -- (180:1cm);
	\draw (0,0) -- (60:.3cm) .. controls ++(60:.3cm) and ++(90:.5cm) .. (0:.65cm) .. controls ++(270:.8cm) and ++(270:.8cm) .. (180:.65cm) .. controls ++(90:.6cm) and ++(-60:.5cm) .. (120:1cm);
	\draw[very thick] (0,0) circle (1cm);
	\draw[unshaded, very thick] (0,0) circle (.3cm);
	\node at (80:.8cm) {\scriptsize{$i$}};
	\node at (-90:.7cm) {\scriptsize{$n-i$}};
\end{tikzpicture}
\Bigg)
\\
&= Z\Bigg(
\begin{tikzpicture}[baseline=-.1cm, xscale=-.9, yscale=.9]
	\draw (0,0) -- (120:.3cm) .. controls ++(120:.3cm) and ++(-120:.3cm) .. (60:1cm);
	\draw[thick, red] (0:.3cm) -- (0:1cm);
	\draw (0,0) -- (60:.3cm) .. controls ++(60:.3cm) and ++(90:.5cm) .. (0:.65cm) .. controls ++(270:.8cm) and ++(270:.8cm) .. (180:.65cm) .. controls ++(90:.6cm) and ++(-60:.5cm) .. (120:1cm);
	\draw[very thick] (0,0) circle (1cm);
	\draw[unshaded, very thick] (0,0) circle (.3cm);
	\node at (85:.79cm) {\scriptsize{$i{+}j$}};
	\node at (-90:.7cm) {\scriptsize{$n-i$}};
\end{tikzpicture}
\circ_1
\tikz[baseline=11, yscale=-.97, xscale=.97]{\tensorforhere{(3.2,-.5)}{1.2}{n\;}{1}{j\,}}
\circ_1
\begin{tikzpicture}[baseline=-.1cm, scale=.9]
	\draw (0,0) -- (120:.3cm) .. controls ++(120:.3cm) and ++(-120:.3cm) .. (60:1cm);
	\draw[thick, red] (180:.3cm) -- (180:1cm);
	\draw (0,0) -- (60:.3cm) .. controls ++(60:.3cm) and ++(90:.5cm) .. (0:.65cm) .. controls ++(270:.8cm) and ++(270:.8cm) .. (180:.65cm) .. controls ++(90:.6cm) and ++(-60:.5cm) .. (120:1cm);
	\draw[very thick] (0,0) circle (1cm);
	\draw[unshaded, very thick] (0,0) circle (.3cm);
	\node at (80:.8cm) {\scriptsize{$i$}};
	\node at (-90:.7cm) {\scriptsize{$n-i$}};
\end{tikzpicture}
\Bigg)=Z(p_{i,j}).
\end{split}
\]
Here, the last equality is best checked by drawing:
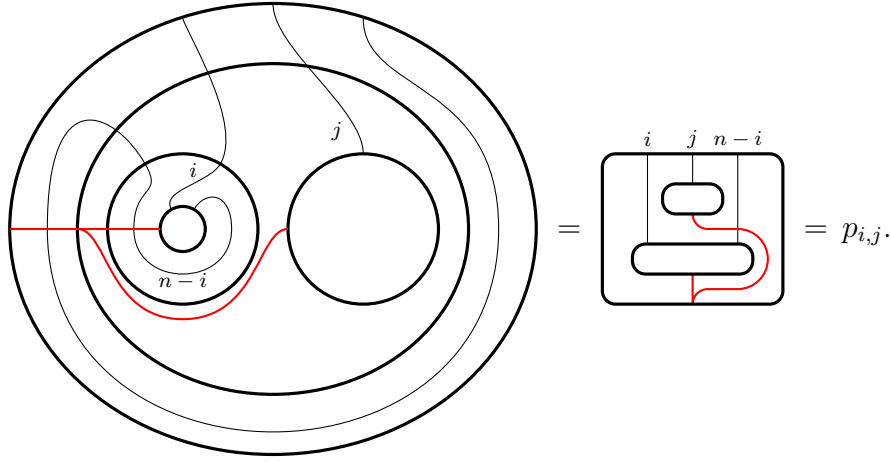
\[\vspace{-1.2cm}
\begin{tikzpicture}[baseline=-.1cm]
	\coordinate (a) at (-2.3,0);
	\coordinate (b) at (-1.4,0);
	\draw[very thick] (1.2,0) ellipse (3.5 and 3);
	\draw[very thick] (1.2,0) ellipse (2.6 and 2.2);
	\draw[very thick] (0,0) circle (1cm);
	\draw[very thick] (2.4,0) circle (1cm);
	\draw (0,0) -- (120:.3cm) .. controls ++(120:.3cm) and ++(-120:.3cm) .. (60:1cm);
	\draw (0,0) -- (60:.3cm) .. controls ++(60:.3cm) and ++(90:.5cm) .. (0:.65cm) .. controls ++(270:.8cm) and ++(270:.8cm) .. (180:.65cm) .. controls ++(90:.6cm) and ++(-60:.5cm) .. (120:1cm);
	\node at (80:.8cm) {\scriptsize{$i$}};
	\node at (-90:.7cm) {\scriptsize{$n-i$}};
%
	\node at (2.05,1.3) {\scriptsize{$j$}};
	\draw[unshaded, very thick] (0,0) circle (.3cm);	
%
	\draw (2.4,1)  .. controls ++(90:.6cm) and ++(270:.6cm) .. (1.2,3) ;	
	\draw (60:1cm) .. controls ++(60:.6cm) and ++(-70:.3cm)  .. (0,2.8);
	\draw (120:1cm) .. controls ++(120:.8cm) and ++(90:2cm) .. (-1.8,0) .. controls ++(270:3.6cm) and ++(270:3.6cm) .. (4.2,0) .. controls ++(90:2cm) and ++(270:1cm) .. (2.4,2.8);
	\draw[thick, red] (180:.3cm) -- (a);
	\draw[thick, red] (1.4,0) .. controls ++(180:.4cm) and ++(0:1cm) .. (0,-1.2) .. controls ++(180:1cm) and ++(0:.4cm) .. (b);
\end{tikzpicture}
\,\,=\,\,
\begin{tikzpicture}[baseline = -.1cm]
	\draw (-.6,-.2) -- (-.6,1);
	\node at (-.6,1.2) {\scriptsize{$i$}};
	\draw (0,1) -- (0,.6);
	\node at (0,1.2) {\scriptsize{$j$}};
	\draw (.6,-.2) -- (.6,1);
	\node at (.6,1.2) {\scriptsize{$n-i$}};
	\draw[thick, red] (0,-.6) -- (0,-1);
	\draw[thick, red] (0,.2) arc (180:270:.2cm) -- (.6,0) arc (90:-90:.4cm) -- (.2,-.8) arc (90:180:.2cm);
	\roundNbox{}{(0,0)}{1}{.2}{.2}{}
	\roundNbox{unshaded}{(0,-.4)}{.2}{.6}{.6}{}
	\roundNbox{unshaded}{(0,.4)}{.2}{.2}{.2}{}
\end{tikzpicture}
\,\,=\,p_{i,j}.
\]
\end{proof}
\vspace{.2cm}

So far, we have checked that the functor $\Lambda\circ\Delta:\APA\to\APA$ is the identity on objects.
We now check that it is the identity on morphisms.
Let $H: \cP_1\to \cP_2$ be a map of anchored planar algebras, and let $H': \cP_1\to \cP_2$ be its image under $\Lambda\circ\Delta$:

\begin{lem}
\label{lem:APAmapEquality}
For every $n$, the two maps $H[n]:\cP_1[n]\to \cP_2[n]$ and $H'[n]:\cP_1[n]\to \cP_2[n]$ are equal.
\end{lem}

\begin{proof}
Let $(\cM_i,m_i):=\Delta(\cP_i)$ and $G:=\Delta(H):\cM_1\to \cM_2$.
Recall from \eqref{eq: this is a morphism of APA} that 
$$
H'[n]=\Lambda(G)[n]=\zeta_{m_1^{\otimes n}}: \cP_1[n] = \Tr^1_\cC(m_1^{\otimes n}) \to \Tr_\cC^2(m_2^{\otimes n}) = \cP_2[n]
$$
is the mate of $G(\varepsilon^1_{m_1^{\otimes n}})\circ \gamma^G_{\Tr_\cC^1(m_1^{\otimes n})}=G(\varepsilon^1_{m^{\otimes n}_1})$,
where the last equality holds because $\gamma^G=\gamma^{\Delta(H)}=\id$ (see the line below \eqref{eq:def of Delta -- objects}).
As before, the adjunction 
$\cC(\cP_1[n],\cP_2[n]) \cong \cM_2($ $\Phi_2( \cP_1[n]), m_2^{\otimes n})$
being an identity map, we simply write $H'[n]=G(\varepsilon^1_{m^{\otimes n}_1})$.
Recall from \eqref{it's a through-strand} that
\[
\varepsilon^1_{m_1^{\otimes n}} =\! \tikz[baseline=15]{\draw (0,0)node[left, yshift=5, xshift=1]{$\scriptstyle \cP_1[n]$} -- ++(0,.8) arc(180:90:.4) -- ++(.8,0);}\,\,\,\in\,\cM_1\big(\Phi_1(\cP_1[n]),m_1^{\otimes n}\big).
\]
By \eqref{eq:def of Delta}, it follows that 
\[
H'[n]\;\!=\;\!\Delta(H)(\varepsilon^1_{m_1^{\otimes n}}) \;\!=\!
\begin{tikzpicture}[baseline=-.45cm]
	\draw (-.2,-1)node[left, yshift=5, xshift=1]{$\scriptstyle \cP_1[n]$} -- (-.2,-.4) arc (180:90:.4cm) -- (2.4,0);
	\roundNbox{unshaded}{(.85,0)}{.35}{.2}{.2}{$H[n]$};
	\node at (1.95,.22) {\scriptsize{$\cP_2[n]$}};
\end{tikzpicture}
=\;\! H[n]. \qedhere
\]
\end{proof}

Combining all the results of this section, have have proven:

\begin{thm}
The composite $$\Lambda\circ\Delta:\APA\to \Mod_*\to\APA$$ is the identity functor.\hfill $\square$
\end{thm}

This concludes the proof of Theorem \ref{thm:EquivalenceOfCategories2}.\hfill $\square$


\appendix 

\begin{landscape}
\section{An associativity type relation}
\label{appendix}
\thispagestyle{empty}

Let $\cC$ be a braided pivotal category and let $\Tr_\cC:\cM\to\cC$ be as in Section \ref{sec:InternalTrace}.
\begin{lem}\label{lem: Lemma A1}
For any $x,y,z,w\in\cM$, the following maps are equal in $\cC\big(\Tr_\cC(x\otimes y)\otimes \Tr_\cC(z)\otimes \Tr_\cC(w), \Tr_\cC(x\otimes z\otimes y\otimes w)\big)$:
\begin{align*}
\mu_{x\otimes z\otimes y,w}&\circ(\tau^-_{y,x\otimes z}\otimes \id_{\Tr_\cC(w)})\circ(\mu_{y\otimes x, z} \otimes \id_{\Tr_\cC(w)})\circ (\tau^+_{x,y}\otimes \id_{\Tr_\cC(z)}\otimes \id_{\Tr_\cC(w)})
\\
&
=\tau^-_{y\otimes w, x\otimes z}\circ \mu_{y\otimes w\otimes x,z} \circ (\tau^+_{x,y\otimes w}\otimes \id_{\Tr_\cC(z)}) \circ (\mu_{x\otimes y, w}\otimes \id_{\Tr_\cC(z)})\circ (\id_{\Tr_\cC(x\otimes y)}\otimes\beta_{\Tr_\cC(z),\Tr_\cC(w)}).
\end{align*}
\end{lem}
\begin{proof}
Consider the following five morphisms, named $(a)$--$(e)$ from left to right.
$$
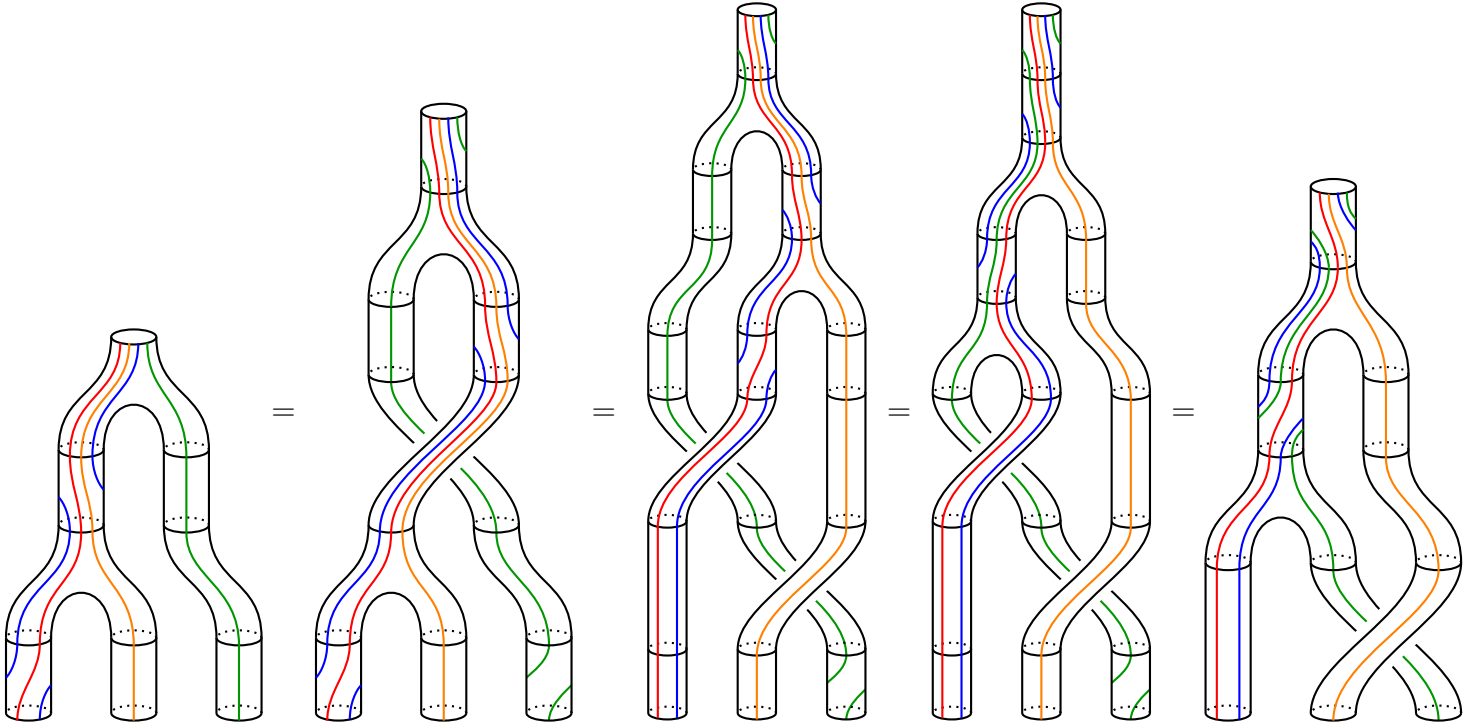
\begin{tikzpicture}[baseline=-1.1cm]	
	\pgfmathsetmacro{\voffset}{.08};
	\pgfmathsetmacro{\hoffset}{.15};
	\pgfmathsetmacro{\hoffsetTop}{.12};

	\coordinate (a1) at (-1,-5);
	\coordinate (a2) at ($ (a1) + (1.4,0)$);
	\coordinate (a3) at ($ (a1) + (2.8,0)$);
	\coordinate (b1) at ($ (a1) + (0,1)$);
	\coordinate (b2) at ($ (a2) + (0,1) $);
	\coordinate (b3) at ($ (a3) + (0,1) $);
	\coordinate (c1) at ($ (b1) + (.7,1.5)$);
	\coordinate (c2) at ($ (c1) + (1.4,0) $);
	\coordinate (d1) at ($ (c1) + (0,1)$);
	\coordinate (d2) at ($ (c2) + (0,1)$);
	\coordinate (e) at ($ (d1) + (.7,1.5)$);

	\bottomCylinder{(a1)}{.3}{1}
	\bottomCylinder{(a2)}{.3}{1}
	\bottomCylinder{(a3)}{.3}{1}
	\pairOfPants{(b1)}{}
	\LeftSlantCylinder{(b3)}{}
	\emptyCylinder{(c1)}{.3}{1}
	\emptyCylinder{(c2)}{.3}{1}
	\topPairOfPants{(d1)}{}
	
	\draw[thick, red] ($ (a1) + (\hoffset,0) + (0,-.1)$) .. controls ++(90:.4cm) and ++(270:.4cm) .. ($ (a1) + 3*(\hoffset,0) + (0,-\voffset) + (0,1) $);		
	\draw[thick, blue] ($ (a1) + 3*(\hoffset,0) + (0,-\voffset) $) .. controls ++(90:.2cm) and ++(225:.1cm) .. ($ (a1) + 4*(\hoffset,0) + (0,-\voffset) + (0,.45)$);
	\draw[thick, blue] ($ (a1) + (\hoffset,1) + (0,-\voffset) $) .. controls ++(270:.2cm) and ++(45:.1cm) .. ($ (a1) + (0,1) + (0,-\voffset) + (0,-.45)$);

	\draw[thick, DarkGreen] ($ (a3) + 2*(\hoffset,0) + (0,-.1) $) -- ($ (b3) + 2*(\hoffset,0) + (0,-.1) $) .. controls ++(90:.8cm) and ++(270:.6cm) .. ($ (c2) + 2*(\hoffset,0) + (0,-.1)$) -- ($ (d2) + 2*(\hoffset,0) + (0,-.1) $);
	\draw[thick, orange] ($ (a2) + 2*(\hoffset,0) + (0,-.1) $) -- ($ (b2) + 2*(\hoffset,0) + (0,-.1) $);

	\draw[thick, blue] ($ (b1) + (\hoffset,0) + (0,-\voffset) $) .. controls ++(90:.8cm) and ++(270:.7cm) .. ($ (c1) + (\hoffset,0) + (0,-\voffset)$);
	\draw[thick, red] ($ (b1) + 3*(\hoffset,0) + (0,-\voffset) $) .. controls ++(90:.8cm) and ++(270:.7cm) .. ($ (c1) + 2*(\hoffset,0) + (0,-.1)$);
	\draw[thick, orange] ($ (b2) + 2*(\hoffset,0) + (0,-.1) $) .. controls ++(90:.8cm) and ++(270:.7cm) .. ($ (c1) + 3*(\hoffset,0) + (0,-\voffset)$);

	\draw[thick, red] ($ (c1) + 2*(\hoffset,0) + (0,-.1)$) .. controls ++(90:.4cm) and ++(270:.4cm) .. ($ (c1) + (\hoffset,0) + (0,-\voffset) + (0,1) $);
	\draw[thick, orange] ($ (c1) + 3*(\hoffset,0) + (0,-\voffset) $) .. controls ++(90:.4cm) and ++(270:.4cm) .. ($ (c1) + 2*(\hoffset,0) + (0,.9) $);		
	\draw[thick, blue] ($ (c1) + (\hoffset,0) + (0,-\voffset) $) .. controls ++(90:.2cm) and ++(-45:.1cm) .. ($ (c1) + (0,-\voffset) + (0,.45)$);
	\draw[thick, blue] ($ (c1) + 3*(\hoffset,0) + (0,1) + (0,-\voffset) $) .. controls ++(270:.2cm) and ++(135:.1cm) .. ($ (c1) + 4*(\hoffset,0) + (0,-\voffset) + (0,-.45) + (0,1)$);
	
	\draw[thick, red] ($ (d1) + (\hoffset,0) + (0,-\voffset) $) .. controls ++(90:.8cm) and ++(270:.7cm) .. ($ (e) + 1*(\hoffsetTop,0) + (0,-\voffset)$);
	\draw[thick, orange] ($ (d1) + 2*(\hoffset,0) + (0,-.1) $) .. controls ++(90:.8cm) and ++(270:.7cm) .. ($ (e) + 2*(\hoffsetTop,0) + (0,-\voffset)$);
	\draw[thick, blue] ($ (d1) + 3*(\hoffset,0) + (0,-\voffset) $) .. controls ++(90:.8cm) and ++(270:.7cm) .. ($ (e) + 3*(\hoffsetTop,0) + (0,-\voffset)$);
	\draw[thick, DarkGreen] ($ (d2) + 2*(\hoffset,0) + (0,-.1) $) .. controls ++(90:.8cm) and ++(270:.7cm) .. ($ (e) + 4*(\hoffsetTop,0) + (0,-\voffset)$);

\end{tikzpicture}
=\,\,
\begin{tikzpicture}[baseline=-1.1cm]	
	\pgfmathsetmacro{\voffset}{.08};
	\pgfmathsetmacro{\hoffset}{.15};
	\pgfmathsetmacro{\hoffsetTop}{.12};

	\coordinate (a1) at (-1,-5);
	\coordinate (a2) at ($ (a1) + (1.4,0)$);
	\coordinate (a3) at ($ (a1) + (2.8,0)$);
	\coordinate (b1) at ($ (a1) + (0,1)$);
	\coordinate (b2) at ($ (a2) + (0,1) $);
	\coordinate (b3) at ($ (a3) + (0,1) $);
	\coordinate (c1) at ($ (b1) + (.7,1.5)$);
	\coordinate (c2) at ($ (c1) + (1.4,0) $);
	\coordinate (d1) at ($ (c1) + (0,2)$);
	\coordinate (d2) at ($ (c2) + (0,2)$);
	\coordinate (e1) at ($ (d1) + (0,1)$);
	\coordinate (e2) at ($ (d2) + (0,1)$);
	\coordinate (f) at ($ (e1) + (.7,1.5)$);

	\draw[thick, DarkGreen] ($ (a3) + 2*(\hoffset,0) + (0,-.1) $) .. controls ++(90:.2cm) and ++(-135:.1cm) .. ($ (a3) + 4*(\hoffset,0) + (0,-\voffset) + (0,.45)$);
	\draw[thick, DarkGreen] ($ (b3) + 2*(\hoffset,0) + (0,-.1) $) .. controls ++(270:.2cm) and ++(45:.1cm) .. ($ (b3) + (0,-\voffset) + (0,-.45) $);
	\draw[thick, DarkGreen] ($ (b3) + 2*(\hoffset,0) + (0,-.1) $) .. controls ++(90:.8cm) and ++(270:.6cm) .. ($ (c2) + 2*(\hoffset,0) + (0,-.1) $) .. controls ++(90:.8cm) and ++(270:.6cm) ..	($ (d1) + 2*(\hoffset,0) + (0,-.1) $) -- ($ (e1) + 2*(\hoffset,0) + (0,-.1) $);

	\bottomCylinder{(a1)}{.3}{1}
	\bottomCylinder{(a2)}{.3}{1}
	\bottomCylinder{(a3)}{.3}{1}
	\LeftSlantCylinder{(b3)}{}
	\pairOfPants{(b1)}{}
	\braid{(c1)}{.3}{2}
	\bottomCylinder{(d1)}{.3}{1}
	\bottomCylinder{(d2)}{.3}{1}
	\pairOfPants{(e1)}{}
	\topCylinder{(f)}{.3}{1}

	\draw[thick, red] ($ (a1) + (\hoffset,0) + (0,-.1)$) .. controls ++(90:.4cm) and ++(270:.4cm) .. ($ (a1) + 3*(\hoffset,0) + (0,-\voffset) + (0,1) $);		
	\draw[thick, blue] ($ (a1) + 3*(\hoffset,0) + (0,-\voffset) $) .. controls ++(90:.2cm) and ++(225:.1cm) .. ($ (a1) + 4*(\hoffset,0) + (0,-\voffset) + (0,.45)$);
	\draw[thick, blue] ($ (a1) + (\hoffset,1) + (0,-\voffset) $) .. controls ++(270:.2cm) and ++(45:.1cm) .. ($ (a1) + (0,1) + (0,-\voffset) + (0,-.45)$);

	\draw[thick, orange] ($ (a2) + 2*(\hoffset,0) + (0,-.1) $) -- ($ (b2) + 2*(\hoffset,0) + (0,-.1) $) .. controls ++(90:.8cm) and ++(270:.8cm) .. ($ (c1) + 3*(\hoffset,0) + (0,-\voffset)$) .. controls ++(90:.8cm) and ++(270:.6cm) .. ($ (d2) + 3*(\hoffset,0) + (0,-\voffset)$);
	\draw[thick, blue] ($ (b1) + (\hoffset,0) + (0,-\voffset) $) .. controls ++(90:.8cm) and ++(270:.7cm) .. ($ (c1) + (\hoffset,0) + (0,-\voffset)$) .. controls ++(90:.8cm) and ++(270:.6cm) ..  ($ (d2) + (\hoffset,0) + (0,-\voffset)$);
	\draw[thick, red] ($ (b1) + 3*(\hoffset,0) + (0,-\voffset) $) .. controls ++(90:.8cm) and ++(270:.7cm) .. ($ (c1) + 2*(\hoffset,0) + (0,-.1)$) .. controls ++(90:.8cm) and ++(270:.55cm) .. ($ (d2) + 2*(\hoffset,0) + (0,-.1)$);	
	
	\draw[thick, red] ($ (d2) + 2*(\hoffset,0) + (0,-.1)$) .. controls ++(90:.4cm) and ++(270:.4cm) .. ($ (d2) + (\hoffset,0) + (0,-\voffset) + (0,1) $);
	\draw[thick, orange] ($ (d2) + 3*(\hoffset,0) + (0,-\voffset) $) .. controls ++(90:.4cm) and ++(270:.4cm) .. ($ (d2) + 2*(\hoffset,0) + (0,.9) $);		
	\draw[thick, blue] ($ (d2) + (\hoffset,0) + (0,-\voffset) $) .. controls ++(90:.2cm) and ++(-45:.1cm) .. ($ (d2) + (0,-\voffset) + (0,.45)$);
	\draw[thick, blue] ($ (d2) + 3*(\hoffset,0) + (0,1) + (0,-\voffset) $) .. controls ++(270:.2cm) and ++(135:.1cm) .. ($ (d2) + 4*(\hoffset,0) + (0,-\voffset) + (0,-.45) + (0,1)$);

	\draw[thick, red] ($ (f) + 2*(\hoffsetTop,0) + (0,-.1)$) .. controls ++(90:.4cm) and ++(270:.4cm) .. ($ (f) + (\hoffsetTop,0) + (0,-\voffset) + (0,1) $);
	\draw[thick, orange] ($ (f) + 3*(\hoffsetTop,0) + (0,-\voffset) $) .. controls ++(90:.4cm) and ++(270:.4cm) .. ($ (f) + 2*(\hoffsetTop,0) + (0,.9) $);		
	\draw[thick, DarkGreen] ($ (f) + (\hoffsetTop,0) + (0,-\voffset) $) .. controls ++(90:.2cm) and ++(-45:.1cm) .. ($ (f) + (0,-\voffset) + (0,.45)$);
	\draw[thick, DarkGreen] ($ (f) + 4*(\hoffsetTop,0) + (0,1) + (0,-\voffset) $) .. controls ++(270:.2cm) and ++(135:.1cm) .. ($ (f) + 5*(\hoffsetTop,0) + (0,-\voffset) + (0,-.45) + (0,1)$);
	\draw[thick, blue] ($ (f) + 4*(\hoffsetTop,0) + (0,-\voffset) $) .. controls ++(90:.4cm) and ++(270:.4cm) .. ($ (f) + 3*(\hoffsetTop,0) + (0,-\voffset) + (0,1)$);

	\draw[thick, DarkGreen] ($ (e1) + 2*(\hoffset,0) + (0,-.1)$) .. controls ++(90:.8cm) and ++(270:.7cm) .. ($ (e1) + (\hoffsetTop,0) + (0,-\voffset) + (.7,1.5) $);	
	\draw[thick, red] ($ (e2) + (\hoffset,0) + (0,-\voffset)$) .. controls ++(90:.8cm) and ++(270:.7cm) .. ($ (e1) + 2*(\hoffsetTop,0) + (0,-\voffset) + (.7,1.5) $);	
	\draw[thick, orange] ($ (e2) + 2*(\hoffset,0) + (0,-.1)$) .. controls ++(90:.8cm) and ++(270:.7cm) .. ($ (e1) + 3*(\hoffsetTop,0) + (0,-\voffset) + (.7,1.5) $);	
	\draw[thick, blue] ($ (e2) + 3*(\hoffset,0) + (0,-\voffset)$) .. controls ++(90:.8cm) and ++(270:.7cm) .. ($ (e1) + 4*(\hoffsetTop,0) + (0,-\voffset) + (.7,1.5) $);
\end{tikzpicture}
\,\,=\,\,
\begin{tikzpicture}[baseline=-1.2cm, scale=.85]	
	\pgfmathsetmacro{\voffset}{.08};
	\pgfmathsetmacro{\hoffset}{.15};
	\pgfmathsetmacro{\hoffsetTop}{.12};

	\coordinate (a1) at (-1,-1);
	\coordinate (a2) at ($ (a1) + (1.4,0)$);
	\coordinate (a3) at ($ (a1) + (2.8,0)$);
	\coordinate (b1) at ($ (a1) + (0,1)$);
	\coordinate (b2) at ($ (b1) + (1.4,0) $);
	\coordinate (b3) at ($ (b2) + (1.4,0) $);
	\coordinate (c1) at ($ (b1) + (.7,1.5)$);
	\coordinate (c2) at ($ (b2) + (.7,1.5)$);
	\coordinate (d1) at ($ (c1) + (0,1)$);
	\coordinate (d2) at ($ (c2) + (0,1)$);
	\coordinate (e) at ($ (d1) + (.7,1.5)$);
	\coordinate (x1) at ($ (a1) + (0,-2)$);
	\coordinate (x2) at ($ (a2) + (0,-2)$);
	\coordinate (x3) at ($ (a3) + (0,-2)$);
	\coordinate (y1) at ($ (x1) + (0,-2)$);
	\coordinate (y2) at ($ (x2) + (0,-2)$);
	\coordinate (y3) at ($ (x3) + (0,-2)$);
	\coordinate (z1) at ($ (y1) + (0,-1)$);
	\coordinate (z2) at ($ (y2) + (0,-1)$);
	\coordinate (z3) at ($ (y3) + (0,-1)$);

	\draw[thick, DarkGreen] ($ (z3) + 2*(\hoffset,0) + (0,-.1) $) .. controls ++(90:.2cm) and ++(-135:.1cm) .. ($ (z3) + 4*(\hoffset,0) + (0,-\voffset) + (0,.45)$);
	\draw[thick, DarkGreen] ($ (y3) + 2*(\hoffset,0) + (0,-.1) $) .. controls ++(270:.2cm) and ++(45:.1cm) .. ($ (y3) + (0,-\voffset) + (0,-.45) $);
	\draw[thick, DarkGreen] ($ (y3) + 2*(\hoffset,0) + (0,-.1) $) .. controls ++(90:.8cm) and ++(270:.6cm) .. ($ (x2) + 2*(\hoffset,0) + (0,-.1) $) .. controls ++(90:.8cm) and ++(270:.6cm) ..	($ (a1) + 2*(\hoffset,0) + (0,-.1) $);

	\braid{(x1)}{.3}{2}
	\braid{(y2)}{.3}{2}
	\emptyCylinder{(y1)}{.3}{2}
	\emptyCylinder{(x3)}{.3}{2}
	\bottomCylinder{(z1)}{.3}{1}
	\bottomCylinder{(z2)}{.3}{1}
	\bottomCylinder{(z3)}{.3}{1}
	\halfDottedEllipse{(x1)}{.3}{.1}	
	\halfDottedEllipse{(x2)}{.3}{.1}	
	\halfDottedEllipse{(x3)}{.3}{.1}	
	\halfDottedEllipse{(y1)}{.3}{.1}	
	\halfDottedEllipse{(y2)}{.3}{.1}	
	\halfDottedEllipse{(y3)}{.3}{.1}	
	\bottomCylinder{(a1)}{.3}{1}
	\bottomCylinder{(a2)}{.3}{1}
	\bottomCylinder{(a3)}{.3}{1}
	\RightSlantCylinder{(b1)}{}
	\pairOfPants{(b2)}{}
	\emptyCylinder{(c1)}{.3}{1}
	\emptyCylinder{(c2)}{.3}{1}
	\pairOfPants{(d1)}{}
	\topCylinder{(e)}{.3}{1}

	\draw[thick, red] ($ (a2) + (\hoffset,0) + (0,-.1)$) .. controls ++(90:.4cm) and ++(270:.4cm) .. ($ (a2) + 3*(\hoffset,0) + (0,-\voffset) + (0,1) $);		
	\draw[thick, blue] ($ (a2) + 3*(\hoffset,0) + (0,-\voffset) $) .. controls ++(90:.2cm) and ++(225:.1cm) .. ($ (a2) + 4*(\hoffset,0) + (0,-\voffset) + (0,.45)$);
	\draw[thick, blue] ($ (a2) + (\hoffset,1) + (0,-\voffset) $) .. controls ++(270:.2cm) and ++(45:.1cm) .. ($ (a2) + (0,1) + (0,-\voffset) + (0,-.45)$);

	\draw[thick, red] ($ (c2) + 2*(\hoffset,0) + (0,-.1)$) .. controls ++(90:.4cm) and ++(270:.4cm) .. ($ (c2) + (\hoffset,0) + (0,-\voffset) + (0,1) $);
	\draw[thick, orange] ($ (c2) + 3*(\hoffset,0) + (0,-\voffset) $) .. controls ++(90:.4cm) and ++(270:.4cm) .. ($ (c2) + 2*(\hoffset,0) + (0,.9) $);		
	\draw[thick, blue] ($ (c2) + (\hoffset,0) + (0,-\voffset) $) .. controls ++(90:.2cm) and ++(-45:.1cm) .. ($ (c2) + (0,-\voffset) + (0,.45)$);
	\draw[thick, blue] ($ (c2) + 3*(\hoffset,0) + (0,1) + (0,-\voffset) $) .. controls ++(270:.2cm) and ++(135:.1cm) .. ($ (c2) + 4*(\hoffset,0) + (0,-\voffset) + (0,-.45) + (0,1)$);

	\draw[thick, red] ($ (e) + 2*(\hoffsetTop,0) + (0,-.1)$) .. controls ++(90:.4cm) and ++(270:.4cm) .. ($ (e) + (\hoffsetTop,0) + (0,-\voffset) + (0,1) $);
	\draw[thick, orange] ($ (e) + 3*(\hoffsetTop,0) + (0,-\voffset) $) .. controls ++(90:.4cm) and ++(270:.4cm) .. ($ (e) + 2*(\hoffsetTop,0) + (0,.9) $);		
	\draw[thick, DarkGreen] ($ (e) + (\hoffsetTop,0) + (0,-\voffset) $) .. controls ++(90:.2cm) and ++(-45:.1cm) .. ($ (e) + (0,-\voffset) + (0,.45)$);
	\draw[thick, DarkGreen] ($ (e) + 4*(\hoffsetTop,0) + (0,1) + (0,-\voffset) $) .. controls ++(270:.2cm) and ++(135:.1cm) .. ($ (e) + 5*(\hoffsetTop,0) + (0,-\voffset) + (0,-.45) + (0,1)$);
	\draw[thick, blue] ($ (e) + 4*(\hoffsetTop,0) + (0,-\voffset) $) .. controls ++(90:.4cm) and ++(270:.4cm) .. ($ (e) + 3*(\hoffsetTop,0) + (0,-\voffset) + (0,1)$);
		
	\draw[thick, DarkGreen] ($ (a1) + 2*(\hoffset,0) + (0,-.1)$) -- ($ (b1) + 2*(\hoffset,0) + (0,-.1)$) .. controls ++(90:.8cm) and ++(270:.8cm) .. ($ (c1) + 2*(\hoffset,0) + (0,-.1) $) -- ($ (d1) + 2*(\hoffset,0) + (0,-.1) $) ;	
	\draw[thick, blue] ($ (b2) + (\hoffset,0) + (0,-\voffset)$) .. controls ++(90:.8cm) and ++(270:.7cm) .. ($ (c2) + (\hoffset,0) + (0,-\voffset) $);	
	\draw[thick, red] ($ (b2) + 3*(\hoffset,0) + (0,-\voffset)$) .. controls ++(90:.8cm) and ++(270:.7cm) .. ($ (c2) + 2*(\hoffset,0) + (0,-.1) $);	
	\draw[thick, orange] ($ (a3) + 2*(\hoffset,0) + (0,-.1)$) -- ($ (b3) + 2*(\hoffset,0) + (0,-.1)$) .. controls ++(90:.8cm) and ++(270:.8cm) .. ($ (c2) + 3*(\hoffset,0) + (0,-\voffset) $);	

	\draw[thick, DarkGreen] ($ (d1) + 2*(\hoffset,0) + (0,-.1)$) .. controls ++(90:.8cm) and ++(270:.7cm) .. ($ (d1) + (\hoffsetTop,0) + (0,-\voffset) + (.7,1.5) $);	
	\draw[thick, red] ($ (d2) + (\hoffset,0) + (0,-\voffset)$) .. controls ++(90:.8cm) and ++(270:.7cm) .. ($ (d1) + 2*(\hoffsetTop,0) + (0,-\voffset) + (.7,1.5) $);	
	\draw[thick, orange] ($ (d2) + 2*(\hoffset,0) + (0,-.1)$) .. controls ++(90:.8cm) and ++(270:.7cm) .. ($ (d1) + 3*(\hoffsetTop,0) + (0,-\voffset) + (.7,1.5) $);	
	\draw[thick, blue] ($ (d2) + 3*(\hoffset,0) + (0,-\voffset)$) .. controls ++(90:.8cm) and ++(270:.7cm) .. ($ (d1) + 4*(\hoffsetTop,0) + (0,-\voffset) + (.7,1.5) $);
	
	\draw[thick, orange] ($ (z2) + 2*(\hoffset,0) + (0,-.1) $) -- ($ (y2) + 2*(\hoffset,0) + (0,-.1) $) .. controls ++(90:.8cm) and ++(270:.6cm) .. ($ (x3) + 2*(\hoffset,0) + (0,-.1)$) -- ($ (a3) + 2*(\hoffset,0) + (0,-.1)$);
	\draw[thick, red] ($ (z1) + (\hoffset,0) + (0,-\voffset) $) -- ($ (x1) + (\hoffset,0) + (0,-\voffset) $) .. controls ++(90:.8cm) and ++(270:.6cm) .. ($ (a2) + (\hoffset,0) + (0,-\voffset)$);
	\draw[thick, blue] ($ (z1) + 3*(\hoffset,0) + (0,-\voffset) $) -- ($ (x1) + 3*(\hoffset,0) + (0,-\voffset) $) .. controls ++(90:.8cm) and ++(270:.6cm) .. ($ (a2) + 3*(\hoffset,0) + (0,-\voffset)$);	
	
\end{tikzpicture}
=
\begin{tikzpicture}[baseline=-1.2cm, scale=.85]

	\pgfmathsetmacro{\voffset}{.08};
	\pgfmathsetmacro{\hoffset}{.15};
	\pgfmathsetmacro{\hoffsetTop}{.12};

	\coordinate (a1) at (-1,-1);
	\coordinate (a2) at ($ (a1) + (1.4,0)$);
	\coordinate (a3) at ($ (a2) + (1.4,0)$);
	\coordinate (b1) at ($ (a1) + (.7,1.5)$);
	\coordinate (b2) at ($ (b1) + (1.4,0)$);
	\coordinate (c1) at ($ (b1) + (0,1)$);
	\coordinate (c2) at ($ (c1) + (1.4,0)$);
	\coordinate (d) at ($ (c1) + (.7,1.5)$);	
	\coordinate (e) at ($ (d) + (0,1)$);

	\coordinate (x1) at ($ (a1) + (0,-2)$);
	\coordinate (x2) at ($ (a2) + (0,-2)$);
	\coordinate (x3) at ($ (a3) + (0,-2)$);
	\coordinate (y1) at ($ (x1) + (0,-2)$);
	\coordinate (y2) at ($ (x2) + (0,-2)$);
	\coordinate (y3) at ($ (x3) + (0,-2)$);
	\coordinate (z1) at ($ (y1) + (0,-1)$);
	\coordinate (z2) at ($ (y2) + (0,-1)$);
	\coordinate (z3) at ($ (y3) + (0,-1)$);

	\draw[thick, DarkGreen] ($ (z3) + 2*(\hoffset,0) + (0,-.1) $) .. controls ++(90:.2cm) and ++(-135:.1cm) .. ($ (z3) + 4*(\hoffset,0) + (0,-\voffset) + (0,.45)$);
	\draw[thick, DarkGreen] ($ (y3) + 2*(\hoffset,0) + (0,-.1) $) .. controls ++(270:.2cm) and ++(45:.1cm) .. ($ (y3) + (0,-\voffset) + (0,-.45) $);
	\draw[thick, DarkGreen] ($ (y3) + 2*(\hoffset,0) + (0,-.1) $) .. controls ++(90:.8cm) and ++(270:.6cm) .. ($ (x2) + 2*(\hoffset,0) + (0,-.1) $) .. controls ++(90:.8cm) and ++(270:.6cm) ..	($ (a1) + 2*(\hoffset,0) + (0,-.1) $);

	\braid{(x1)}{.3}{2}
	\braid{(y2)}{.3}{2}
	\emptyCylinder{(y1)}{.3}{2}
	\emptyCylinder{(x3)}{.3}{2}
	\bottomCylinder{(z1)}{.3}{1}
	\bottomCylinder{(z2)}{.3}{1}
	\bottomCylinder{(z3)}{.3}{1}
	\halfDottedEllipse{(x1)}{.3}{.1}	
	\halfDottedEllipse{(x2)}{.3}{.1}	
	\halfDottedEllipse{(x3)}{.3}{.1}	
	\halfDottedEllipse{(y1)}{.3}{.1}	
	\halfDottedEllipse{(y2)}{.3}{.1}	
	\halfDottedEllipse{(y3)}{.3}{.1}	
	\pairOfPants{(a1)}{}
	\LeftSlantCylinder{(a3)}{}
	\emptyCylinder{(b1)}{.3}{1}
	\emptyCylinder{(b2)}{.3}{1}
	\pairOfPants{(c1)}{}
	\emptyCylinder{(d)}{.3}{1}
	\halfDottedEllipse{(e)}{.3}{.1}		
	\topCylinder{(e)}{.3}{1}

	\draw[thick, DarkGreen] ($ (b1) + (\hoffset,-\voffset) $) .. controls ++(90:.4cm) and ++(270:.4cm) .. ($ (b1) + 2*(\hoffset,0) + (0,.9) $);
	\draw[thick, red] ($ (b1) + 2*(\hoffset,0) + (0,-.1) $) .. controls ++(90:.4cm) and ++(270:.4cm) .. ($ (b1) + 3*(\hoffset,0) + (0,.92) $);		
	\draw[thick, blue] ($ (b1) + 3*(\hoffset,0) + (0,-\voffset) $) .. controls ++(90:.2cm) and ++(225:.1cm) .. ($ (b1) + 4*(\hoffset,0) + (0,-\voffset) + (0,.45)$);
	\draw[thick, blue] ($ (b1) + (\hoffset,1) + (0,-\voffset) $) .. controls ++(270:.2cm) and ++(45:.1cm) .. ($ (b1) + (0,1) + (0,-\voffset) + (0,-.45)$);

	\draw[thick, DarkGreen] ($ (d) + 2*(\hoffsetTop,0) + (0,-.1)$) .. controls ++(90:.4cm) and ++(270:.4cm) .. ($ (d) + (\hoffsetTop,0) + (0,-\voffset) + (0,1) $);
	\draw[thick, red] ($ (d) + 3*(\hoffsetTop,0) + (0,-\voffset) $) .. controls ++(90:.4cm) and ++(270:.4cm) .. ($ (d) + 2*(\hoffsetTop,0) + (0,.9) $);		
	\draw[thick, blue] ($ (d) + (\hoffsetTop,0) + (0,-\voffset) $) .. controls ++(90:.2cm) and ++(-45:.1cm) .. ($ (d) + (0,-\voffset) + (0,.45)$);
	\draw[thick, blue] ($ (d) + 4*(\hoffsetTop,0) + (0,1) + (0,-\voffset) $) .. controls ++(270:.2cm) and ++(135:.1cm) .. ($ (d) + 5*(\hoffsetTop,0) + (0,-\voffset) + (0,-.45) + (0,1)$);
	\draw[thick, orange] ($ (d) + 4*(\hoffsetTop,0) + (0,-\voffset) $) .. controls ++(90:.4cm) and ++(270:.4cm) .. ($ (d) + 3*(\hoffsetTop,0) + (0,-\voffset) + (0,1)$);

	\draw[thick, red] ($ (e) + 2*(\hoffsetTop,0) + (0,-.1)$) .. controls ++(90:.4cm) and ++(270:.4cm) .. ($ (e) + (\hoffsetTop,0) + (0,-\voffset) + (0,1) $);
	\draw[thick, orange] ($ (e) + 3*(\hoffsetTop,0) + (0,-\voffset) $) .. controls ++(90:.4cm) and ++(270:.4cm) .. ($ (e) + 2*(\hoffsetTop,0) + (0,.9) $);		
	\draw[thick, DarkGreen] ($ (e) + (\hoffsetTop,0) + (0,-\voffset) $) .. controls ++(90:.2cm) and ++(-45:.1cm) .. ($ (e) + (0,-\voffset) + (0,.45)$);
	\draw[thick, DarkGreen] ($ (e) + 4*(\hoffsetTop,0) + (0,1) + (0,-\voffset) $) .. controls ++(270:.2cm) and ++(135:.1cm) .. ($ (e) + 5*(\hoffsetTop,0) + (0,-\voffset) + (0,-.45) + (0,1)$);
	\draw[thick, blue] ($ (e) + 4*(\hoffsetTop,0) + (0,-\voffset) $) .. controls ++(90:.4cm) and ++(270:.4cm) .. ($ (e) + 3*(\hoffsetTop,0) + (0,-\voffset) + (0,1)$);

	\draw[thick, DarkGreen] ($ (a1) + 2*(\hoffset,0) + (0,-.1)$) .. controls ++(90:.8cm) and ++(270:.8cm) .. ($ (b1) + (\hoffset,0) + (0,-\voffset) $);	
	\draw[thick, red] ($ (a2) + (\hoffset,0) + (0,-.1)$) .. controls ++(90:.8cm) and ++(270:.7cm) .. ($ (b1) + 2*(\hoffset,0) + (0,-\voffset) $);	
	\draw[thick, blue] ($ (a2) + 3*(\hoffset,0) + (0,-\voffset)$) .. controls ++(90:.8cm) and ++(270:.7cm) .. ($ (b1) + 3*(\hoffset,0) + (0,-\voffset) $);	
	\draw[thick, orange] ($ (a3) + 2*(\hoffset,0) + (0,-.1)$) .. controls ++(90:.8cm) and ++(270:.6cm) .. ($ (b2) + 2*(\hoffset,0) + (0,-.1) $) -- ($ (c2) + 2*(\hoffset,0) + (0,-.1) $);	
	
	\draw[thick, blue] ($ (c1) + (\hoffset,0) + (0,-\voffset)$) .. controls ++(90:.8cm) and ++(270:.7cm) .. ($ (d) + (\hoffsetTop,0) + (0,-\voffset) $);
	\draw[thick, DarkGreen] ($ (c1) + 2*(\hoffset,0) + (0,-.1)$) .. controls ++(90:.8cm) and ++(270:.7cm) .. ($ (d) + 2*(\hoffsetTop,0) + (0,-\voffset) $);	
	\draw[thick, red] ($ (c1) + 3*(\hoffset,0) + (0,-\voffset)$) .. controls ++(90:.8cm) and ++(270:.7cm) .. ($ (d) + 3*(\hoffsetTop,0) + (0,-\voffset) $);	
	\draw[thick, orange] ($ (c2) + 2*(\hoffset,0) + (0,-.1)$) .. controls ++(90:.8cm) and ++(270:.7cm) .. ($ (d) + 4*(\hoffsetTop,0) + (0,-\voffset) $);	

	\draw[thick, orange] ($ (z2) + 2*(\hoffset,0) + (0,-.1) $) -- ($ (y2) + 2*(\hoffset,0) + (0,-.1) $) .. controls ++(90:.8cm) and ++(270:.6cm) .. ($ (x3) + 2*(\hoffset,0) + (0,-.1)$) -- ($ (a3) + 2*(\hoffset,0) + (0,-.1)$);
	\draw[thick, red] ($ (z1) + (\hoffset,0) + (0,-\voffset) $) -- ($ (x1) + (\hoffset,0) + (0,-\voffset) $) .. controls ++(90:.8cm) and ++(270:.6cm) .. ($ (a2) + (\hoffset,0) + (0,-\voffset)$);
	\draw[thick, blue] ($ (z1) + 3*(\hoffset,0) + (0,-\voffset) $) -- ($ (x1) + 3*(\hoffset,0) + (0,-\voffset) $) .. controls ++(90:.8cm) and ++(270:.6cm) .. ($ (a2) + 3*(\hoffset,0) + (0,-\voffset)$);	

\end{tikzpicture}
=
\begin{tikzpicture}[baseline=-1.1cm]	
	\pgfmathsetmacro{\voffset}{.08};
	\pgfmathsetmacro{\hoffset}{.15};
	\pgfmathsetmacro{\hoffsetTop}{.12};

	\coordinate (a1) at (-1,-5);
	\coordinate (a2) at ($ (a1) + (1.4,0)$);
	\coordinate (a3) at ($ (a1) + (2.8,0)$);
	\coordinate (b1) at ($ (a1) + (0,2)$);
	\coordinate (b2) at ($ (a2) + (0,2) $);
	\coordinate (b3) at ($ (a3) + (0,2) $);
	\coordinate (c1) at ($ (b1) + (.7,1.5)$);
	\coordinate (c2) at ($ (c1) + (1.4,0) $);
	\coordinate (d1) at ($ (c1) + (0,1)$);
	\coordinate (d2) at ($ (c2) + (0,1)$);
	\coordinate (e) at ($ (d1) + (.7,1.5)$);

	\draw[thick, DarkGreen] ($ (a3) + 2*(\hoffset,0) + (0,-.1) $) .. controls ++(90:.8cm) and ++(270:.6cm) .. ($ (b2) + 2*(\hoffset,0) + (0,-.1) $) .. controls ++(90:.8cm) and ++(270:.8cm) .. ($ (c1) + 3*(\hoffset,0) + (0,-\voffset) $);

	\bottomCylinder{(a1)}{.3}{2}
	\braid{(a2)}{.3}{2}
	\halfDottedEllipse{(a2)}{.3}{.1}
	\halfDottedEllipse{(a3)}{.3}{.1}
	\pairOfPants{(b1)}{}
	\LeftSlantCylinder{(b3)}{}
	\emptyCylinder{(c1)}{.3}{1}
	\emptyCylinder{(c2)}{.3}{1}
	\pairOfPants{(d1)}{}
	\topCylinder{(e)}{.3}{1}
	
	\draw[thick, blue] ($ (d1) + 1*(\hoffset,0) + (0,-\voffset) $) .. controls ++(90:.8cm) and ++(270:.6cm) .. ($ (e) + 1*(\hoffsetTop,0) + (0,-\voffset)$);
	\draw[thick, DarkGreen] ($ (d1) + 2*(\hoffset,0) + (0,-.1) $) .. controls ++(90:.8cm) and ++(270:.6cm) .. ($ (e) + 2*(\hoffsetTop,0) + (0,-\voffset)$);
	\draw[thick, red] ($ (d1) + 3*(\hoffset,0) + (0,-\voffset) $) .. controls ++(90:.8cm) and ++(270:.6cm) .. ($ (e) + 3*(\hoffsetTop,0) + (0,-\voffset)$);
	\draw[thick, orange] ($ (d2) + 2*(\hoffset,0) + (0,-.1) $) .. controls ++(90:.8cm) and ++(270:.6cm) .. ($ (e) + 4*(\hoffsetTop,0) + (0,-\voffset)$);

	\draw[thick, red] ($ (c1) + (\hoffset,0) + (0,-\voffset)$) .. controls ++(90:.4cm) and ++(270:.4cm) .. ($ (c1) + 3*(\hoffset,0) + (0,-\voffset) + (0,1) $);
	\draw[thick, blue] ($ (c1) + 2*(\hoffset,0) + (0,-.1) $) .. controls ++(90:.2cm) and ++(-135:.1cm) .. ($ (c1) + 4*(\hoffset,0) + (0,-\voffset) + (0,.5)$);
	\draw[thick, blue] ($ (c1) + 1*(\hoffset,0) + (0,1) + (0,-\voffset) $) .. controls ++(270:.2cm) and ++(45:.1cm) .. ($ (c1) + (0,-\voffset) + (0,-.35) + (0,1)$);
	\draw[thick, DarkGreen] ($ (c1) + 3*(\hoffset,0) + (0,-\voffset) $) .. controls ++(90:.2cm) and ++(-135:.1cm) .. ($ (c1) + 4*(\hoffset,0) + (0,-\voffset) + (0,.35)$);
	\draw[thick, DarkGreen] ($ (c1) + 2*(\hoffset,0) + (0,1) + (0,-.1) $) .. controls ++(270:.2cm) and ++(45:.1cm) .. ($ (c1) + (0,-\voffset) + (0,-.5) + (0,1)$);	
	
	\draw[thick, orange] ($ (a2) + 2*(\hoffset,0) + (0,-.1) $) .. controls ++(90:.8cm) and ++(270:.6cm) .. ($ (b3) + 2*(\hoffset,0) + (0,-.1)$) .. controls ++(90:.8cm) and ++(270:.6cm) .. ($ (c2) + 2*(\hoffset,0) + (0,-.1)$) -- ($ (d2) + 2*(\hoffset,0) + (0,-.1)$);
	\draw[thick, red] ($ (a1) + (\hoffset,0) + (0,-\voffset) $) -- ($ (b1) + (\hoffset,0) + (0,-\voffset) $) .. controls ++(90:.8cm) and ++(270:.7cm) .. ($ (c1) + (\hoffset,0) + (0,-\voffset)$);
	\draw[thick, blue] ($ (a1) + 3*(\hoffset,0) + (0,-\voffset) $) -- ($ (b1) + 3*(\hoffset,0) + (0,-\voffset) $) .. controls ++(90:.8cm) and ++(270:.7cm) .. ($ (c1) + 2*(\hoffset,0) + (0,-.1)$);	

	\draw[thick, red] ($ (e) + 3*(\hoffsetTop,0) + (0,-\voffset)$) .. controls ++(90:.4cm) and ++(270:.4cm) .. ($ (e) + (\hoffsetTop,0) + (0,-\voffset) + (0,1) $);
	\draw[thick, orange] ($ (e) + 4*(\hoffsetTop,0) + (0,-\voffset) $) .. controls ++(90:.4cm) and ++(270:.4cm) .. ($ (e) + 2*(\hoffsetTop,0) + (0,.9) $);		
	\draw[thick, blue] ($ (e) + (\hoffsetTop,0) + (0,-\voffset) $) .. controls ++(90:.2cm) and ++(-45:.1cm) .. ($ (e) + (0,-\voffset) + (0,.35)$);
	\draw[thick, blue] ($ (e) + 3*(\hoffsetTop,0) + (0,1) + (0,-\voffset) $) .. controls ++(270:.2cm) and ++(135:.1cm) .. ($ (e) + 5*(\hoffsetTop,0) + (0,-\voffset) + (0,-.5) + (0,1)$);
	\draw[thick, DarkGreen] ($ (e) + 2*(\hoffsetTop,0) + (0,-\voffset) $) .. controls ++(90:.2cm) and ++(-45:.1cm) .. ($ (e) + (0,-\voffset) + (0,.5)$);
	\draw[thick, DarkGreen] ($ (e) + 4*(\hoffsetTop,0) + (0,1) + (0,-\voffset) $) .. controls ++(270:.2cm) and ++(135:.1cm) .. ($ (e) + 5*(\hoffsetTop,0) + (0,-\voffset) + (0,-.35) + (0,1)$);
	
\end{tikzpicture}
$$
We get the following equalities:
$(a)=(b)$
by relation \ref{rel:TwistMultiplicationAndTraciators},
$(b)=(c)$
by the braiding in $\cC$,
$(c)=(d)$
by the first relation in \ref{rel:MultiplicationAssociative}, and
$(d)=(e)$
by the relations \ref{rel:TraciatorComposition}, \ref{rel:TwistMultiplicationAndTraciators}, and the braiding in $\cC$.
\end{proof}
\end{landscape}



\section{Anchored planar tangles with coupons}
\label{sec:APAsWithCoupons}

Let $\cC$ be a braided tensor category, and let $\cM$ be a pivotal module tensor category over $\cC$.
The constructions in Sections \ref{sec:Constructing anchored planar algebras} and \ref{sec:APAfromMTC}
generalize to assign a morphism
\begin{equation}\label{gwblh0mg}
Z\left(
\begin{tikzpicture}[scale=.8, baseline =-.1cm]
	\coordinate (a) at (0,0);
	\coordinate (c) at (.6,-.6);         
	\coordinate (d) at (-.6,.6);         
	
	\ncircle{}{(a)}{1.6}{-115}{}
	\draw[thick, red] (c)+(235:.4) .. controls ++(230:.5cm) and ++(45:.4cm) .. (-115:1.6);
	\draw[thick, red] (d)+(235:.4) .. controls ++(250:.6cm) and ++(70:.4cm) .. (-115:1.6);
			
	\draw (60:1.6cm) arc (150:300:.4cm);
	\draw ($ (c) + (0,.4) $) arc (0:90:.8cm);
	\draw ($ (c) + (-.4,0) $) circle (.25cm);
	\draw ($ (d) + (0,.88) $) -- (d) -- ($ (d) + (-.88,0) $);
	\draw ($ (c) + (0,-.88) $) -- (c) -- ($ (c) + (.88,0) $);
	\ncircle{unshaded}{(d)}{.4}{235}{}
	\ncircle{unshaded}{(c)}{.4}{235}{}
\draw[<-] (-.6,1.2) -- +(90:.01);
\draw[->] (-1.25,.6) -- +(180:.01);
\draw[<-] (-.05,-.62) -- +(87:.01);
\draw[<-] (.365,.365) -- +(135:.01);
\draw[->] (.85,.91) -- +(139:.01);
\draw[<-] (.6,-1.28) -- +(90:.01);
\draw[<-] (1.28,-.6) -- +(180:.01);
\node[right] at (-.7,1.25) {$\scriptstyle a$};
\node[below] at (-1.21,.68) {$\scriptstyle a$};
\node[left] at (.05,-.55) {$\scriptstyle a$};
\node[left] at (.83,1.05) {$\scriptstyle c$};
\node at (.62,.42) {$\scriptstyle b$};
\node[right] at (.52,-1.2) {$\scriptstyle b$};
\node[above] at (1.23,-.65) {$\scriptstyle c$};
\useasboundingbox;
\node[scale=.7] at (-120:1.82) {\anchor};
\end{tikzpicture}
\right)
\,:\,
\begin{tabular}{l}
\\
$\Tr_\cC(a\otimes a^*\otimes b)\otimes \Tr_\cC(a^*\otimes a\otimes b^*\otimes c\otimes b)$
\\
\hspace{4.5cm}$\to\, \Tr_\cC(a\otimes a^*\otimes c\otimes c^*\otimes c\otimes b)$
\end{tabular}
\end{equation}
to every oriented anchored planar tangle $T$ with strands labeled by objects of $\cM$.
In this appendix, we further generalize the construction to assign a morphism $Z(T)$ to every anchored planar tangle $T$ with strands labeled by objects and coupons labeled by morphisms in~$\cM$.

\begin{defn}\label{def: labeled oriented anchored planar tangle with coupons}
Let $\cM$ be a pivotal tensor category. 
A \emph{labeled oriented anchored planar tangle with coupons} consists of the following data:
\begin{itemize}
\item A disc with holes $T=\mathbb D\setminus (\mathring D_1\cup \ldots \cup \mathring D_r)$, as in Definition \ref{def:planar tangle}.
\item Anchor points $Q=\{q_0,\ldots q_r\}$ and anchor lines as in Definition \ref{defn:AnchoredPlanarTangle}.
\item Disjoint closed discs $C_1,\ldots, C_s\subset \mathring T$ (the coupons) with marked points $c_i\in\partial C_i$.
The coupons are not ordered. The anchor lines are allowed to intersect the coupons.
\item A closed $1$-dimensional oriented submanifold $X\subset T\setminus (\mathring C_1\cup \ldots \cup \mathring C_s)$ (the strands).
The boundary of $X$ lies on $\partial T \cup \bigcup\partial C_i$ and it does not touch the $q_i$ and the $c_i$.
\item An object of $\cM$ for each connected component of $X$.
\item A morphism $f_i\in \cM(1, x_1\otimes \cdots \otimes x_k)$ for each coupon $C_i$, where the $x_1,\ldots,x_k\in\cM$ are as follows.
Let $m_j$ be the label of the $j$th strand that one encounters as one walks clockwise on $\partial C_i$, starting at $c_i$.
We set $x_j=m_j$ if that $j$th strand is oriented outwards, and $x_j=m_j^*$ otherwise.\vspace{-.2cm}
\end{itemize}
\end{defn}

We illustrate this notion with an example:\,\,
$
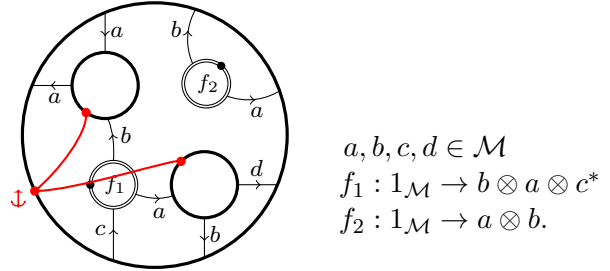
\begin{tikzpicture}[scale=1.1, baseline =.5cm]
	\coordinate (a) at (0,0);
	\coordinate (c) at (.6,-.6);         
	\coordinate (d) at (-.6,.6);         
	
	\ncircle{}{(a)}{1.6}{-155}{}

\node[circle, draw, double, inner sep=2.7] (x) at (-.5,-.6) {$\scriptstyle f_1\!$};
\draw (x) + (-23:.29) to[bend right=30] (.45,-.6);
\draw (x) + (-90:.29) -- (-.5,-1.5);
\draw (x) + (90:.29) to[bend right=25] (-.85,.6);
\fill (x) + (180:.28) circle (.05);
	\draw (72:1.6cm) arc (150:300:.73cm);
\node[fill=white, circle, draw, double, inner sep=2.7] (y) at (45:.88) {$\scriptstyle f_2\!$};
\fill (y) + (50:.28) circle (.05);
	\draw ($ (d) + (0,.88) $) -- (d) -- ($ (d) + (-.88,0) $);
	\draw ($ (c) + (0,-.88) $) -- (c) -- ($ (c) + (.88,0) $);
	\ncircle{unshaded}{(d)}{.4}{235}{}
	\ncircle{unshaded}{(c)}{.4}{135}{}
\draw[<-] (-.6,1.2) -- +(90:.01);
\draw[->] (-1.25,.6) -- +(180:.01);
\draw[<-] (.6,-1.28) -- +(90:.01);
\draw[<-] (1.28,-.6) -- +(180:.01);

\draw[->] (.4,1.25) --node[left, xshift=2, yshift=1, scale=1.1]{$\scriptstyle b$} +(85:.01);
\draw[->] (1.2,.425) --node[below, xshift=1, yshift=2, scale=1.1]{$\scriptstyle a$} +(0:.01);
\draw[->] (.05,-.77) --node[below, xshift=0, yshift=1, scale=1.1]{$\scriptstyle a$} +(0:.01);
\draw[->] (-.51,0) --node[right, xshift=-1.5, yshift=0, scale=1.1]{$\scriptstyle b$} +(100:.01);
\draw[->] (-.5,-1.15) --node[left, xshift=1.5, yshift=0, scale=1.1]{$\scriptstyle c$} +(90:.01);
\node[right, scale=1.1] at (-.68,1.25) {$\scriptstyle a$};
\node[below, scale=1.1] at (-1.21,.66) {$\scriptstyle a$};
\node[right, scale=1.1] at (.52,-1.2) {$\scriptstyle b$};
\node[above, scale=1.1] at (1.23,-.645) {$\scriptstyle d$};

	\draw[thick, red] (c)+(135:.4) .. controls ++(130:.45cm) and ++(0:.5cm) .. (-155:1.6);
	\draw[thick, red] (d)+(235:.4) .. controls ++(250:.6cm) and ++(40:.4cm) .. (-155:1.6);
\node[scale=.9] at (-155:1.82) {\anchor};
\node[right] at (1.9,-.6) {\begin{tabular}{l}
\small $\, a,b,c,d\in\cM$
\\[-.3mm]
\small $f_1:1_\cM\to b\otimes a\otimes c^*$
\\[-.3mm]
\small $f_2:1_\cM \to a\otimes b.$
\end{tabular}};
\end{tikzpicture}
$\bigskip\\
To an isotopy class of such tangles $T$, we wish to associate a morphism $Z(T)$.
For the above example, this will be a map
\[
Z(T):\Tr_\cC(a\otimes a^*\otimes b^*)\otimes\Tr_\cC(d\otimes b\otimes a^*)\to \Tr_\cC(a\otimes a^*\otimes b\otimes a\otimes d\otimes b\otimes c^*).
\]
We proceed as follows.
First, for each coupon $C$, labeled by some $f\in \cM(1, x_1\otimes \cdots \otimes x_k)$,
we associate the element
\[
Z(C):=\Tr_\cC(f)\circ i:1_\cC \to \Tr_\cC(x_1\otimes \cdots \otimes x_k).
\]
In the graphical calculus of Section~\ref{sec:TubeRelations}, this is represented by:
$$
C\,=
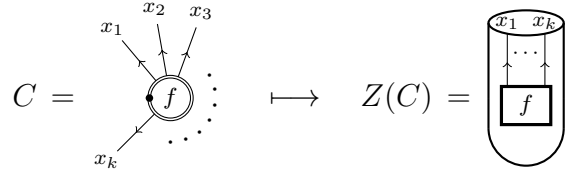
\begin{tikzpicture}[baseline=-.1cm]
\node[fill=white, circle, draw, double, inner sep=2.15, scale=1.1] (y) at (0,0) {$\scriptstyle f$};
\fill (y) + (180:.28) circle (.05);
	\draw[mid>] (-135:.3cm) -- (-135:1cm);
	\draw[mid>] (130:.3cm) -- (130:1cm);
	\draw[mid>] (100:.3cm) -- (100:1cm);
	\draw[mid>] (70:.3cm) -- (70:1cm);
	\node at (30:.6cm) {$\cdot$};
	\node at (10:.6cm) {$\cdot$};
	\node at (-10:.6cm) {$\cdot$};
	\node at (-30:.6cm) {$\cdot$};
	\node at (-50:.6cm) {$\cdot$};
	\node at (-70:.6cm) {$\cdot$};
	\node at (-90:.6cm) {$\cdot$};
	\node at (130:1.2cm) {\scriptsize{$x_1$}};
	\node at (100:1.2cm) {\scriptsize{$x_2$}};
	\node at (70:1.2cm) {\scriptsize{$x_3$}};
	\node at (-135:1.2cm) {\scriptsize{$x_k$}};
\end{tikzpicture}
\quad\longmapsto\quad
Z(C)\,=\,
\begin{tikzpicture}[baseline=.3cm]
	\coordinate (a1) at (0,0);
	\coordinate (b1) at (0,.4);
	\draw[thick] (a1) -- (b1);
	\draw[thick] ($ (a1) + (1,0) $) -- ($ (b1) + (1,0) $);
	\topCylinder{(b1)}{.5}{1}
	\draw[thick] (a1) arc (-180:0:.5cm);
	\draw[mid>] (.25,.5) -- (.25,1.25);
	\draw[mid>] (.75,.5) -- (.75,1.25);
	\node at (.52,1) {\scriptsize{$\cdots$}};
	\node at (.25,1.4) {\scriptsize{$x_1$}};
	\node at (.75,1.4) {\scriptsize{$x_k$}};
\node[fill=white, draw, very thick, inner ysep=3, inner xsep=6, scale=1.1] at (.5,.3) {$\scriptstyle f$};
\end{tikzpicture}
$$
Let $\check T$ be tangle obtained from $T$ by sliding the anchor lines off of coupons, removing the interiors of the coupons,
and adding new anchor lines that connect the marked points $c_i$ to anchor point $q_0$.
This is a tangle without coupons, and so it has an associated morphism $Z(\check T)$ as in \eqref{gwblh0mg}.
We define $Z(T)$ by:
\begin{equation}\label{sfwagiletjnad}
Z(T):=\big[\big[Z(\check T)\circ_{j_1}Z(C_1)\big]\circ_{j_2}Z(C_2)\big] \ldots \circ_{j_s}Z(C_s),
\end{equation}
where we write $j_i$ for the order of arrival at $q_0$ of the anchor line which connects it to $C_i$,
and the coupons have been renumbered so as to have $j_1>j_2>\ldots>j_s$.

\begin{thm} 
\label{thm:LabeledAnchoredPlanarTangles}
Let $T$ be a labeled oriented anchored planar tangle with coupons.
Then the morphism $Z(T)$ defined above only depends on the isotopy class of $T$.

Let $S$ and $T$ be labeled oriented anchored planar tangles with coupons such that the labels on the $i$th input circle of $T$ agree with those on the outer circle of $S$.
Then the familiar composition formula holds:
\[
Z(T \circ_i S) = Z(T) \circ_i Z(S).
\]
\end{thm}

\begin{proof}
In the construction of $\check T$, we had to move the anchor lines off the coupons, and add new anchor lines from $q_0$ to $c_i$.
We need to show that $Z(T)$ does not depend on these choices. 
The isotopy invariance of $Z(T)$ will then follow from the corresponding property of $Z(\check T)$ (proven in Sections \ref{sec:Constructing anchored planar algebras} and \ref{sec:APAfromMTC} modulo the fact that, in those sections, we were not working with orientations and labelings).

We first argue that $Z(T)$ does not depend on the choice of anchor lines from $q_0$ to $c_i$.
Fix a tangle $T$ whose anchor lines do not intersect the coupons, and let us assume that $\check T$ and $\check T'$ are
obtained from $T\setminus(\mathring C_1\cup\ldots\cup \mathring C_s)$ by adding the required new anchor lines in two different ways. 
We write $Z(T)$ and $Z(T)'$ for the quantity \eqref{sfwagiletjnad} computed using $\check T$ and $\check T'$, respectively.
Our goal is to show that $Z(T)=Z(T)'$.

The systems of anchor lines on $\check T$ and on $\check T'$ are related by a sequence of the following moves:

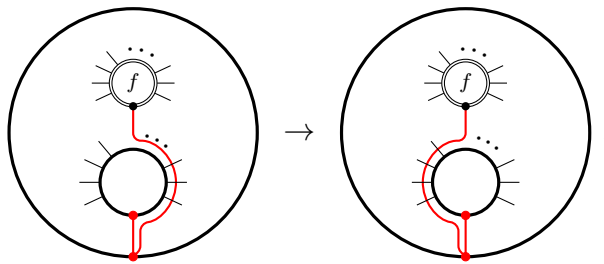
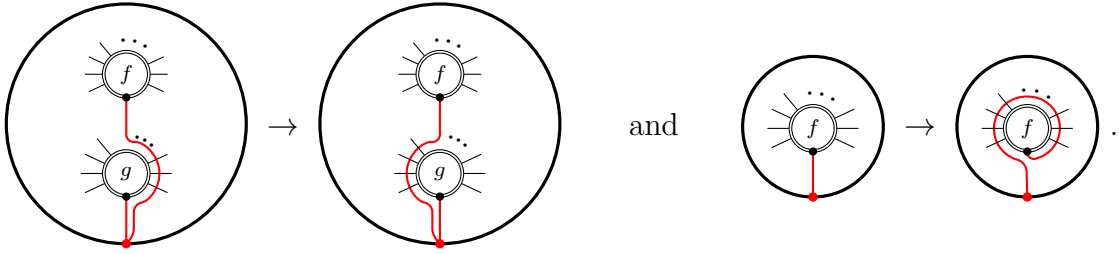
\[\begin{tikzpicture}[scale=1.1, baseline =-.1cm]
\coordinate (a) at (0,0);\coordinate (c) at (0,-.6);\ncircle{}{(a)}{1.5}{-90}{}\ncircle{unshaded}{(c)}{.4}{-90}{}\node[circle, draw, double, inner sep=3] (x) at (0,.6) {$\scriptstyle f$};\draw[thick, red, rounded corners=3] (x) + (-90:.29) -- (0,-.1) arc (90+2:-80-2:.5) -- +(0,-.37) -- (0,-1.48);\draw[thick, red] (c)+(-90:.4) -- (-90:1.5);\fill (x) + (-90:.28) circle (.05);\draw (x) + (-155:.29) -- + (-155:.5);\draw (x) + (-180:.29) -- + (-180:.5);\draw (x) + (155:.29) -- + (155:.5);\draw (x) + (130:.29) -- + (130:.5);\draw (x) + (-25:.29) -- + (-25:.5);\draw (x) + (0:.29) -- + (0:.5);\draw (x) + (25:.29) -- + (25:.5);\node at ($(x) + (55:.4)$) {$\cdot$};\node at ($(x) + (75:.4)$) {$\cdot$};\node at ($(x) + (95:.4)$) {$\cdot$};\draw (c) + (-155:.4) -- + (-155:.65);\draw (c) + (-180:.4) -- + (-180:.65);\draw (c) + (155:.4) -- + (155:.65);\draw (c) + (130:.4) -- + (130:.65);\draw (c) + (-25:.4) -- + (-25:.65);\draw (c) + (0:.4) -- + (0:.65);\draw (c) + (25:.4) -- + (25:.65);\node at ($(c) + (47:.58)$) {$\cdot$};\node at ($(c) + (60:.58)$) {$\cdot$};\node at ($(c) + (73:.58)$) {$\cdot$};
\end{tikzpicture}
\,\,\,\to\,\,\,\begin{tikzpicture}[scale=1.1, baseline =-.1cm]
\coordinate (a) at (0,0);\coordinate (c) at (0,-.6);\ncircle{}{(a)}{1.5}{-90}{}\ncircle{unshaded}{(c)}{.4}{-90}{}\node[circle, draw, double, inner sep=3] (x) at (0,.6) {$\scriptstyle f$};\draw[thick, red, rounded corners=3] (x) + (-90:.29) -- (0,-.1) arc (90-2:260+2:.5) -- +(0,-.37) -- (0,-1.48);\draw[thick, red] (c)+(-90:.4) -- (-90:1.5);\fill (x) + (-90:.28) circle (.05);\draw (x) + (-155:.29) -- + (-155:.5);\draw (x) + (-180:.29) -- + (-180:.5);\draw (x) + (155:.29) -- + (155:.5);\draw (x) + (130:.29) -- + (130:.5);\draw (x) + (-25:.29) -- + (-25:.5);\draw (x) + (0:.29) -- + (0:.5);\draw (x) + (25:.29) -- + (25:.5);\node at ($(x) + (55:.4)$) {$\cdot$};\node at ($(x) + (75:.4)$) {$\cdot$};\node at ($(x) + (95:.4)$) {$\cdot$};\draw (c) + (-155:.4) -- + (-155:.65);\draw (c) + (-180:.4) -- + (-180:.65);\draw (c) + (155:.4) -- + (155:.65);\draw (c) + (-25:.4) -- + (-25:.65);\draw (c) + (130:.4) -- + (130:.65);\draw (c) + (0:.4) -- + (0:.65);\draw (c) + (25:.4) -- + (25:.65);\node at ($(c) + (47:.55)$) {$\cdot$};\node at ($(c) + (60:.55)$) {$\cdot$};\node at ($(c) + (73:.55)$) {$\cdot$};
\end{tikzpicture}\,\,,
\]

\[\,\,\,\,\begin{tikzpicture}[scale=1.1, baseline =-.15cm]
\coordinate (a) at (0,0);\ncircle{}{(a)}{1.45}{-90}{}\node[circle, draw, double, inner sep=3.7] (c) at (0,-.6) {$\scriptstyle g$};\node[circle, draw, double, inner sep=3] (x) at (0,.6) {$\scriptstyle f$};\draw[thick, red, rounded corners=3] (x) + (-90:.29) -- (0,-.225) arc (90+4:180-256-4:.375) -- +(0,-.45+.05) -- (0,-1.49+.05);\draw[thick, red] (c)+(-90:.29) -- (-90:1.5-.05);\fill (x) + (-90:.28) circle (.05);\fill (c) + (-90:.28) circle (.05);\draw (x) + (-155:.29) -- + (-155:.5);\draw (x) + (-180:.29) -- + (-180:.5);\draw (x) + (155:.29) -- + (155:.5);\draw (x) + (130:.29) -- + (130:.5);\draw (x) + (-25:.29) -- + (-25:.5);\draw (x) + (0:.29) -- + (0:.5);\draw (x) + (25:.29) -- + (25:.5);\node at ($(x) + (55:.4)$) {$\cdot$};\node at ($(x) + (75:.4)$) {$\cdot$};\node at ($(x) + (95:.4)$) {$\cdot$};\draw (c) + (-155:.29) -- + (-155:.55);\draw (c) + (-180:.29) -- + (-180:.55);\draw (c) + (155:.29) -- + (155:.55);\draw (c) + (-25:.29) -- + (-25:.55);\draw (c) + (130:.29) -- + (130:.55);\draw (c) + (0:.29) -- + (0:.55);\draw (c) + (25:.29) -- + (25:.55);\node at ($(c) + (47:.45)$) {$\cdot$};\node at ($(c) + (60:.45)$) {$\cdot$};\node at ($(c) + (73:.45)$) {$\cdot$};
\end{tikzpicture}
\,\,\to\,\,\begin{tikzpicture}[scale=1.1, baseline =-.15cm]
\coordinate (a) at (0,0);\ncircle{}{(a)}{1.45}{-90}{}\node[circle, draw, double, inner sep=3.7] (c) at (0,-.6) {$\scriptstyle g$};\node[circle, draw, double, inner sep=3] (x) at (0,.6) {$\scriptstyle f$};\draw[thick, red, rounded corners=3] (x) + (-90:.29) -- (0,-.225) arc (90-4:256+4:.375) -- +(0,-.45+.05) -- (0,-1.49+.05);\draw[thick, red] (c)+(-90:.29) -- (-90:1.5-.05);\fill (x) + (-90:.28) circle (.05);\fill (c) + (-90:.28) circle (.05);\draw (x) + (-155:.29) -- + (-155:.5);\draw (x) + (-180:.29) -- + (-180:.5);\draw (x) + (155:.29) -- + (155:.5);\draw (x) + (130:.29) -- + (130:.5);\draw (x) + (-25:.29) -- + (-25:.5);\draw (x) + (0:.29) -- + (0:.5);\draw (x) + (25:.29) -- + (25:.5);\node at ($(x) + (55:.4)$) {$\cdot$};\node at ($(x) + (75:.4)$) {$\cdot$};\node at ($(x) + (95:.4)$) {$\cdot$};\draw (c) + (-155:.29) -- + (-155:.55);\draw (c) + (-180:.29) -- + (-180:.55);\draw (c) + (155:.29) -- + (155:.55);\draw (c) + (-25:.29) -- + (-25:.55);\draw (c) + (130:.29) -- + (130:.55);\draw (c) + (0:.29) -- + (0:.55);\draw (c) + (25:.29) -- + (25:.55);\node at ($(c) + (47:.45)$) {$\cdot$};\node at ($(c) + (60:.45)$) {$\cdot$};\node at ($(c) + (73:.45)$) {$\cdot$};
\end{tikzpicture}
\,\qquad \text{and}\qquad\begin{tikzpicture}[scale=1.1, baseline =-.75cm]
\node[circle, draw, double, inner sep=3] (c) at (0,-.6) {$\scriptstyle f$};\ncircle{}{(0,-.58)}{.85}{-90}{}\draw[thick, red] (c) + (-90:.29) -- (0,-1.49+.05);\fill (c) + (-90:.28) circle (.05);\draw (c) + (-155:.29) -- + (-155:.55);\draw (c) + (-180:.29) -- + (-180:.55);\draw (c) + (155:.29) -- + (155:.55);\draw (c) + (-25:.29) -- + (-25:.55);\draw (c) + (130:.29) -- + (130:.55);\draw (c) + (0:.29) -- + (0:.55);\draw (c) + (25:.29) -- + (25:.55);\node at ($(c) + (55:.42)$) {$\cdot$};\node at ($(c) + (75:.42)$) {$\cdot$};\node at ($(c) + (95:.42)$) {$\cdot$};
\end{tikzpicture}
\,\,\to\,\,
\begin{tikzpicture}[scale=1.1, baseline =-.75cm]\node[circle, draw, double, inner sep=3] (c) at (0,-.6) {$\scriptstyle f$};\ncircle{}{(0,-.58)}{.85}{-90}{}\draw[thick, red] (c) + (-90:.29) arc (-180:-90:.08)  arc (-80:90:.38) arc (90-4:256-5:.39) arc(80:0:.16) to[rounded corners=3]  (0,-1.49+.05);\fill (c) + (-90:.28) circle (.05);\draw (c) + (-155:.29) -- + (-155:.55);\draw (c) + (-180:.29) -- + (-180:.55);\draw (c) + (155:.29) -- + (155:.55);\draw (c) + (-25:.29) -- + (-25:.55);\draw (c) + (130:.29) -- + (130:.55);\draw (c) + (0:.29) -- + (0:.55);\draw (c) + (25:.29) -- + (25:.55);\node at ($(c) + (55:.45)$) {$\cdot$};\node at ($(c) + (75:.45)$) {$\cdot$};\node at ($(c) + (95:.45)$) {$\cdot$};
\end{tikzpicture}\,\,.
\]
(In the language of Section \ref{sec:RibbonBraidGroup}, these correspond to right-composition of a ribbon braid by one of the generators $\varepsilon_i$ or $\vartheta_i$.)
It is enough to prove $Z(T)=Z(T)'$ when $\check T$ differs from $\check T'$ by a single one of the above moves.
If $\check T$ and $\check T'$ differ by the first type of move, then we have:
\[
\begin{split}
Z(T)'
:\!\!
&=\big[\big[Z(\check T')\circ_{j'_1}\ldots\big]\circ_{j'_\ell}Z(C'_\ell)\big]\circ_{j'_{\ell+1}}\ldots\\
&=\big[\big[Z(\check T')\circ_{j_1}\ldots\big]\circ_{j_\ell-1}Z(C_\ell)\big]\circ_{j_{\ell+1}}\ldots\\
&=\big[\big[\big(Z(\check T)\circ_{j_\ell-1}\beta\big)\circ_{j_1}\ldots\big]\circ_{j_\ell-1}\big(Z(C_\ell)\otimes \id\big)\big]\circ_{j_{\ell+1}}\ldots\\
&=\big[\big[Z(\check T)\circ_{j_1}\ldots\big]\circ_{j_\ell-1}\big(\beta\circ(Z(C_\ell)\otimes \id)\big)\big]\circ_{j_{\ell+1}}\ldots\\
&=\big[\big[Z(\check T)\circ_{j_1}\ldots\big]\circ_{j_\ell-1}\big(\id\otimes Z(C_\ell)\big)\big]\circ_{j_{\ell+1}}\ldots\\
&=\big[\big[Z(\check T)\circ_{j_1}\ldots\big]\circ_{j_\ell}Z(C_\ell)\big]\circ_{j_{\ell+1}}\ldots = Z(T),
\end{split}
\]
where $C'_\ell=C_\ell$ is the coupon labelled by $f$.
If they differ by the second type of move, the computation is similar:
\[
\begin{split}
Z(T)'
&=\big[\big[Z(\check T')\circ_{j'_1}\ldots\big]\circ_{j'_{\ell-1}}Z(C'_{\ell-1})\big]\circ_{j'_{\ell}}Z(C'_{\ell})\big]\circ_{j'_{\ell+2}}\ldots\\
&=\big[\big[\big(Z(\check T)\circ_{j_\ell}\beta\big)\circ_{j_1}\ldots\big]\circ_{j_\ell}\big(Z(C'_{\ell})\otimes Z(C'_{\ell-1})\big)\big]\circ_{j_{\ell+2}}\ldots\\
&=\big[\big[Z(\check T)\circ_{j_1}\ldots\big]\circ_{j_\ell}\big(\beta\circ(Z(C_{\ell-1})\otimes Z(C_{\ell}))\big)\big]\circ_{j_{\ell+2}}\ldots\\
&=\big[\big[Z(\check T)\circ_{j_1}\ldots\big]\circ_{j_\ell}\big(Z(C_{\ell})\otimes Z(C_{\ell-1})\big)\big]\circ_{j_{\ell+2}}\ldots\\
&=\big[\big[Z(\check T)\circ_{j_1}\ldots\big]\circ_{j_{\ell-1}}Z(C_{\ell-1})\big]\circ_{j_{\ell}}Z(C_{\ell})\big]\circ_{j_{\ell+2}}\ldots = Z(T),
\end{split}
\]
where $C'_\ell=C_{\ell-1}$ is the coupon labelled by $f$ and $C'_{\ell-1}=C_\ell$ is the coupon labelled by $g$.
Finally, if $\check T$ and $\check T'$ differ by the third type of move, then:
\[
\begin{split}
Z(T)'
&=\big[\big[Z(\check T')\circ_{j'_1}\ldots\big]\circ_{j'_\ell}Z(C'_\ell)\big]\circ_{j'_{\ell+1}}\ldots\\
&=\big[\big[\big(Z(\check T)\circ_{j_\ell}\theta\big)\circ_{j_1}\ldots\big]\circ_{j_\ell}Z(C_\ell)\big]\circ_{j_{\ell+1}}\ldots\\
&=\big[\big[Z(\check T)\circ_{j_1}\ldots\big]\circ_{j_\ell}\big(\theta\circ Z(C_\ell)\big)\big]\circ_{j_{\ell+1}}\ldots\\
&=\big[\big[Z(\check T)\circ_{j_1}\ldots\big]\circ_{j_\ell}Z(C_\ell)\big]\circ_{j_{\ell+1}}\ldots = Z(T).
\end{split}
\]
In the above three computations, we have used in an essential way that the source of $Z(C'_\ell)$ (the unit object $1_\cC$) has trivial braiding and trivial twist.

We now show that $Z(T)$ does not depend on the way we slide the anchor lines off of the coupons.
Given the freedom that we have just established of choosing the anchor lines connecting $q_0$ and $c_i$,
any two ways of sliding the initial anchor lines off of the coupons differ by (a composition of) the following move:
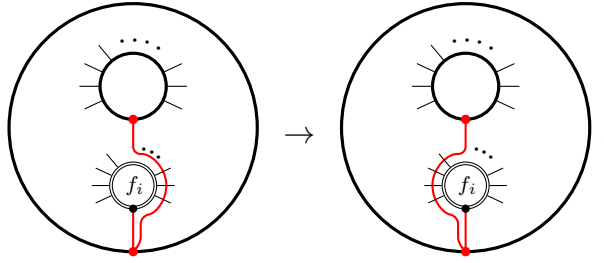
\[
\begin{tikzpicture}[scale=1.1, baseline =-.1cm]
\coordinate (a) at (0,.1);\coordinate (c) at (0,.6);\ncircle{}{(a)}{1.5}{-90}{}\ncircle{unshaded}{(c)}{.4}{-90}{}\node[circle, draw, double, inner sep=2.7] (x) at (0,-.6) {$\scriptstyle f_i\!$};\draw[thick, red, rounded corners=3] (c) + (-90:.4) -- (0,-.225) arc (90+4:180-256-4:.375) -- +(0,-.45+.1) -- (0,-1.49+.1);\draw[thick, red] (x)+(-90:.29) -- (-90:1.5-.05);\fill (x) + (-90:.28) circle (.05);\draw (x) + (-155:.29) -- + (-155:.5);\draw (x) + (-180:.29) -- + (-180:.5);\draw (x) + (155:.29) -- + (155:.5);\draw (x) + (130:.29) -- + (130:.5);\draw (x) + (-25:.29) -- + (-25:.5);\draw (x) + (0:.29) -- + (0:.5);\draw (x) + (25:.29) -- + (25:.5);\node at ($(c) + (50+6:.55)$) {$\cdot$};\node at ($(c) + (70+2:.55)$) {$\cdot$};\node at ($(c) + (90-2:.55)$) {$\cdot$};\node at ($(c) + (110-6:.55)$) {$\cdot$};\draw (c) + (-155:.4) -- + (-155:.65);\draw (c) + (-180:.4) -- + (-180:.65);\draw (c) + (155:.4) -- + (155:.65);\draw (c) + (130:.4) -- + (130:.65);\draw (c) + (-25:.4) -- + (-25:.65);\draw (c) + (0:.4) -- + (0:.65);\draw (c) + (25:.4) -- + (25:.65);\node at ($(x) + (47:.45)$) {$\cdot$};\node at ($(x) + (60:.45)$) {$\cdot$};\node at ($(x) + (73:.45)$) {$\cdot$};
\end{tikzpicture}
\,\,\,\to\,\,\,
\begin{tikzpicture}[scale=1.1, baseline =-.1cm]
\coordinate (a) at (0,.1);\coordinate (c) at (0,.6);\ncircle{}{(a)}{1.5}{-90}{}\ncircle{unshaded}{(c)}{.4}{-90}{}\node[circle, draw, double, inner sep=2.7] (x) at (0,-.6) {$\scriptstyle f_i\!$};\draw[thick, red, rounded corners=3] (c) + (-90:.4) -- (0,-.225) arc (90-4:256+4:.375) -- +(0,-.45+.1) -- (0,-1.49+.1);\draw[thick, red] (x)+(-90:.29) -- (-90:1.5-.05);\fill (x) + (-90:.28) circle (.05);\draw (x) + (-155:.29) -- + (-155:.5);\draw (x) + (-180:.29) -- + (-180:.5);\draw (x) + (155:.29) -- + (155:.5);\draw (x) + (130:.29) -- + (130:.5);\draw (x) + (-25:.29) -- + (-25:.5);\draw (x) + (0:.29) -- + (0:.5);\draw (x) + (25:.29) -- + (25:.5);\node at ($(c) + (50+6:.55)$) {$\cdot$};\node at ($(c) + (70+2:.55)$) {$\cdot$};\node at ($(c) + (90-2:.55)$) {$\cdot$};\node at ($(c) + (110-6:.55)$) {$\cdot$};\draw (c) + (-155:.4) -- + (-155:.65);\draw (c) + (-180:.4) -- + (-180:.65);\draw (c) + (155:.4) -- + (155:.65);\draw (c) + (130:.4) -- + (130:.65);\draw (c) + (-25:.4) -- + (-25:.65);\draw (c) + (0:.4) -- + (0:.65);\draw (c) + (25:.4) -- + (25:.65);\node at ($(x) + (47:.45)$) {$\cdot$};\node at ($(x) + (60:.45)$) {$\cdot$};\node at ($(x) + (73:.45)$) {$\cdot$};
\end{tikzpicture}\,\,.
\]
If $\check T$ and $\check T'$ are related by the above move, then we have:
\[
\begin{split}
Z(T)'
&=\big[\big[Z(\check T')\circ_{j'_1}\ldots\big]\circ_{j'_\ell}Z(C'_\ell)\big]\circ_{j'_{\ell+1}}\ldots\\
&=\big[\big[Z(\check T')\circ_{j_1}\ldots\big]\circ_{j_\ell+1}Z(C_\ell)\big]\circ_{j_{\ell+1}}\ldots\\
&=\big[\big[\big(Z(\check T)\circ_{j_\ell}\beta\big)\circ_{j_1}\ldots\big]\circ_{j_\ell}\big(\id\otimes Z(C_\ell)\big)\big]\circ_{j_{\ell+1}}\ldots\\
&=\big[\big[Z(\check T)\circ_{j_1}\ldots\big]\circ_{j_\ell}\big(\beta\circ(\id\otimes Z(C_\ell))\big)\big]\circ_{j_{\ell+1}}\ldots\\
&=\big[\big[Z(\check T)\circ_{j_1}\ldots\big]\circ_{j_\ell}\big(Z(C_\ell)\otimes \id\big)\big]\circ_{j_{\ell+1}}\ldots\\
&=\big[\big[Z(\check T)\circ_{j_1}\ldots\big]\circ_{j_\ell}Z(C_\ell)\big]\circ_{j_{\ell+1}}\ldots = Z(T),
\end{split}
\]
where $C'_\ell=C_\ell$ is the coupon labelled by $f$.
This finishes the proof that $Z(T)$ only depends on the isotopy class of $T$.

\makeatletter
\DeclareRobustCommand\widecheck[1]{{\mathpalette\@widecheck{#1}}}
\def\@widecheck#1#2{%
    \setbox\z@\hbox{\m@th$#1#2$}%
    \setbox\tw@\hbox{\m@th$#1%
       \widehat{%
          \vrule\@width\z@\@height\ht\z@
          \vrule\@height\z@\@width\wd\z@}$}%
    \dp\tw@-\ht\z@
    \@tempdima\ht\z@ \advance\@tempdima2\ht\tw@ \divide\@tempdima\thr@@
    \setbox\tw@\hbox{%
       \raise\@tempdima\hbox{\scalebox{1}[-1]{\lower\@tempdima\box
\tw@}}}%
    {\ooalign{\box\tw@ \cr \box\z@}}}
\makeatother

Let now $S$ and $T$ be tangles with coupons whose $i$th operadic composition is defined
(the labels on $\partial^iT$ agree with those on $\partial^0 S$).
By \eqref{sfwagiletjnad},
$Z(T)=[[Z(\check T)\circ_{j_1}Z(C_1)] \ldots \circ_{j_s}Z(C_s)]$
and
$Z(S)=[[Z(\check S)\circ_{k_1}Z(D_1)] \ldots \circ_{k_t}Z(D_t)]$,
where the $C$'s are the coupons of $T$ and the $D$'s are the coupons of $S$.
If choose the new anchor lines on $\check T$ so as to arrive at $q_0$ after those of $T$ (with respect to the usual clockwise order \eqref{aflfmwbjsnnd}),
then we compute:
\[
\begin{split}
Z(T) \circ_i Z(S)
&=[[Z(\check T)\circ_{j_1}Z(C_1)] \ldots \circ_{j_s}Z(C_s)]\circ_i Z(S)\\
&=[[(Z(\check T)\circ_i Z(S))\circ_{j'_1}Z(C_1)] \ldots \circ_{j'_s}Z(C_s)]\\
&=[[(Z(\check T)\circ_i [[Z(\check S)\circ_{k_1}Z(D_1)] \ldots \circ_{k_t}Z(D_t)])\circ_{j'_1}Z(C_1)] \ldots \circ_{j'_s}Z(C_s)]\\
&=[[(Z(\check T)\circ_i Z(\check S))\circ_{k'_1}Z(D_1)] \ldots \circ_{k'_t}Z(D_t)]\circ_{j'_1}Z(C_1)] \ldots \circ_{j'_s}Z(C_s)]\\
&=[[(Z(\check T)\circ_i Z(\check S))\circ_{j''_1}Z(C_1)] \ldots \circ_{j''_s}Z(C_s)]\circ_{k'_1}Z(D_1)] \ldots \circ_{k'_t}Z(D_t)]\\
&=[[Z(\,\,\widecheck{\!\!T\circ_i S\!\!}\,\,)\circ_{j''_1}Z(C_1)] \ldots \circ_{j''_s}Z(C_s)]\circ_{k'_1}Z(D_1)] \ldots \circ_{k'_t}Z(D_t)]\\
&=Z(T\circ_i S),
\end{split}
\]
where  $k'_n=k_n+i-1$, $j'_n=j_n+a-1$, $j''_n=j_n+b-1$,
$a$ is the number of inputs of $Z(S)$, and $b$ is the number of inputs of $\check Z(S)$.
\end{proof}


\section{The tube string calculus for the categorified trace}
\label{sec:TubeCalculus}
In this appendix, we show that the calculus of tubes with strings introduced in our earlier paper \cite{1509.02937} and reviewed in Section~\ref{sec:TubeRelations} is invariant under all $3$-dimensional isotopies.\newpage

We illustrate the non-triviality of the problem with an example.
Consider the following two `tube string diagrams':
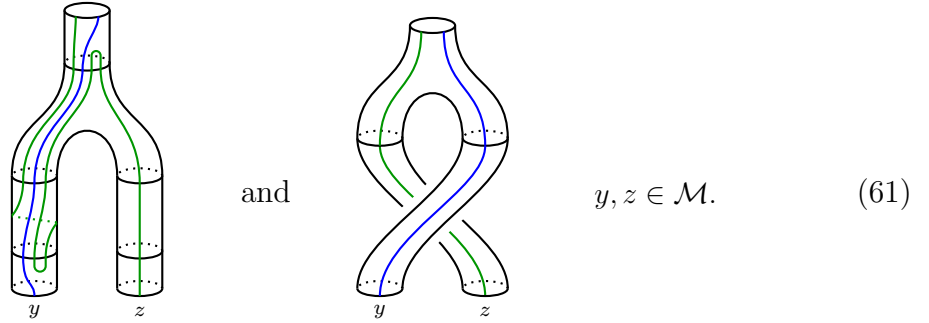
\begin{equation}\label{sbljsvcbflm}
\qquad\qquad\quad
\begin{tikzpicture}[baseline=1.15cm]
	\pgfmathsetmacro{\voffset}{.08};
	\pgfmathsetmacro{\hoffset}{.15};
	\pgfmathsetmacro{\hoffsetTop}{.12};
	\coordinate (a1) at (0,0);
	\coordinate (a2) at (1.4,0);
	\coordinate (b1) at (0,.5);
	\coordinate (b2) at (1.4,.5);
	\coordinate (c1) at (0,1.5);
	\coordinate (c2) at (1.4,1.5);
	\coordinate (d) at (.7,3);
	\coordinate (e) at (.7,3.7);
	\bottomCylinder{(a1)}{.3}{.5}
	\bottomCylinder{(a2)}{.3}{.5}
	\bottomCylinder{(b1)}{.3}{1}
	\bottomCylinder{(b2)}{.3}{1}
	\pairOfPants{(c1)}{}
	\topCylinder{(d)}{.3}{.7}
%
		\draw[thick, yString] ($ (a1) + (.3,-.1)$) .. controls ++(90:.2cm) and ++(270:.4cm) .. ($ (b1) + 1*(\hoffset,0) + (0,-\voffset) $);
		\draw[thick, zString] ($ (b1) + 2*(\hoffset,0) + (0,-.1)$) -- ($ (b1) + 2*(\hoffset,0) + (0,-.2)$) arc (-180:0:{.5*\hoffset}) -- ($ (b1) + 3*(\hoffset,0) + (0,-\voffset)$);	
%
	\draw[thick, yString] ($ (b1) + (\hoffset,0) + (0,-.1)$) .. controls ++(90:.4cm) and ++(270:.4cm) .. ($ (b1) + 2*(\hoffset,0) + (0,-\voffset) + (0,1) $);		
	\draw[thick, zString] ($ (b1) + 2*(\hoffset,0) + (0,-.1)$) .. controls ++(90:.4cm) and ++(270:.4cm) .. ($ (b1) + 3*(\hoffset,0) + (0,-\voffset) + (0,1) $);		
	\draw[thick, zString] ($ (b1) + 3*(\hoffset,0) + (0,-\voffset) $) .. controls ++(90:.2cm) and ++(225:.1cm) .. ($ (b1) + 4*(\hoffset,0) + (0,-\voffset) + (0,.45)$);
	\draw[thick, zString] ($ (b1) + (\hoffset,1) + (0,-\voffset) $) .. controls ++(270:.2cm) and ++(45:.1cm) .. ($ (b1) + (0,1) + (0,-\voffset) + (0,-.45)$);
%
	\draw[thick, zString] ($ (c1) + 1*(\hoffset,0) + (0,-\voffset) $) .. controls ++(90:.8cm) and ++(270:.6cm) .. ($ (d) + 1*(\hoffsetTop,0) + (0,-\voffset)$);
	\draw[thick, yString] ($ (c1) + 2*(\hoffset,0) + (0,-.1) $) .. controls ++(90:.8cm) and ++(270:.6cm) .. ($ (d) + 2*(\hoffsetTop,0) + (0,-\voffset)$);
	\draw[thick, zString] ($ (c1) + 3*(\hoffset,0) + (0,-\voffset) $) .. controls ++(90:.8cm) and ++(270:.6cm) .. ($ (d) + 3*(\hoffsetTop,0) + (0,-\voffset)$);
	\draw[thick, zString] ($ (c2) + 2*(\hoffset,0) + (0,-.1) $) .. controls ++(90:.8cm) and ++(270:.6cm) .. ($ (d) + 4*(\hoffsetTop,0) + (0,-\voffset)$);
%
		\draw[thick, zString] ($ (d) + 1*(\hoffsetTop,0) + (0,-.1)$) .. controls ++(90:.4cm) and ++(270:.2cm) .. ($ (e) + 1*(\hoffset,0) + (0,-\voffset) $);
		\draw[thick, yString] ($ (d) + 2*(\hoffsetTop,0) + (0,-.1)$) .. controls ++(90:.4cm) and ++(270:.2cm) .. ($ (e) + 3*(\hoffset,0) + (0,-\voffset) $);
		\draw[thick, zString] ($ (d) + 3*(\hoffsetTop,0) + (0,-.1)$) -- ($ (d) + 3*(\hoffsetTop,0) + (0,.1)$) arc (180:0:{.5*\hoffsetTop}) -- ($ (d) + 4*(\hoffsetTop,0) + (0,-\voffset)$);	
	\draw[thick, zString] ($ (a2) + (.3,-.1) $) -- ($ (c2) + (.3,-.1) $);
\draw[thick, zString, dotted] (.6,.88) -- +(-.6,.08);
	\node at ($ (a1) + (.3,-.3) $) {\scriptsize{$y$}};
	\node at ($ (a2) + (.3,-.3) $) {\scriptsize{$z$}};
\end{tikzpicture}
\quad\qquad\!\!\!\text{and}\qquad\!\!
\begin{tikzpicture}[baseline=1.15cm]
	\pgfmathsetmacro{\voffset}{.08};
	\pgfmathsetmacro{\hoffset}{.15};
	\pgfmathsetmacro{\hoffsetTop}{.12};
	\coordinate (a1) at (0,0);
	\coordinate (a2) at (1.4,0);
	\coordinate (b1) at (0,2);
	\coordinate (b2) at (1.4,2);
	\coordinate (c) at (.7,3.5);
	\draw[thick, zString] ($ (b1) + (.3,-.1) $) .. controls ++(270:.6cm) and ++(90:.8cm) .. ($ (a2) + (.3,-.1) $);		
	\braid{(a1)}{.3}{2}
	\topPairOfPants{(b1)}{}
	\halfDottedEllipse{(a1)}{.3}{.1}
	\halfDottedEllipse{(a2)}{.3}{.1}
	\draw[thick, yString] ($ (b2) + (.3,-.1) $) .. controls ++(270:.6cm) and ++(90:.8cm) .. ($ (a1) + (.3,-.1) $);	
	\draw[thick, yString] ($ (b2) + (.3,-.1) $) .. controls ++(90:.8cm) and ++(270:.8cm) .. ($ (c) + 3*(\hoffset,0) + (0, -\voffset) $);
	\draw[thick, zString] ($ (b1) + (.3,-.1) $) .. controls ++(90:.8cm) and ++(270:.8cm) .. ($ (c) + 1*(\hoffset,0) + (0, -\voffset) $);
	\node at ($ (a1) + (.3,-.3) $) {\scriptsize{$y$}};
	\node at ($ (a2) + (.3,-.3) $) {\scriptsize{$z$}};
\end{tikzpicture}\qquad \tikz[baseline=-.09cm]{\node[scale=.93]{$y,z\in\cM.$};}
\end{equation}
Here, as usual, $\cM$ is a module tensor category over some braided pivotal category $\cC$.
Following \cite{1509.02937}, we can associate morphisms $\Tr_\cC(y)\otimes \Tr_\cC(z)\to \Tr_\cC(z\otimes y)$ to the above two pictures.
These are:
\begin{equation*} 
\begin{split}
\Tr_\cC(\id_{z\otimes y}\otimes \ev_z)\circ \mu_{z\otimes y\otimes z^*, z}\circ\,&(\tau_{y\otimes z^*, z}\otimes \id_{\Tr_\cC(z)}) 
\circ (\Tr_\cC(\id_y\otimes \coev_{z^*})\otimes  \id_{\Tr_\cC(z)})\\
\text{and}&\qquad
\mu_{z,y}\circ\beta_{\Tr_\cC(y),\Tr_\cC(z)}.
\end{split}
\end{equation*}
Now, as it happens, the two pictures in \eqref{sbljsvcbflm} are isotopic, which raises the expectation that those two morphisms should be equal.
That is indeed true, but a direct proof using the identities in Section~\ref{sec:TubeRelations} as axioms is a rather non-trivial exercise\footnote{\emph{Hint:} the proof uses \ref{rel:StringOverCap}, \ref{rel:TwistMultiplicationAndTraciators} and \ref{rel:MultiplicationAssociative}.}.

The goal of this appendix is to show that whenever two string tube string diagrams are isotopic, 
the corresponding morphisms in $\cC$ are equal.
We prove this by establishing a dictionary between tube string diagrams and anchored planar tangles.

\begin{defn}\label{def:TubeStringDiagram}
A \emph{tube string diagram} consists of:
\begin{itemize}
\item 
An embedded 2-manifold $\Sigma\subset [-1,1]^3$ with no local maxima in its interior, which is isotopic to the following standard surface
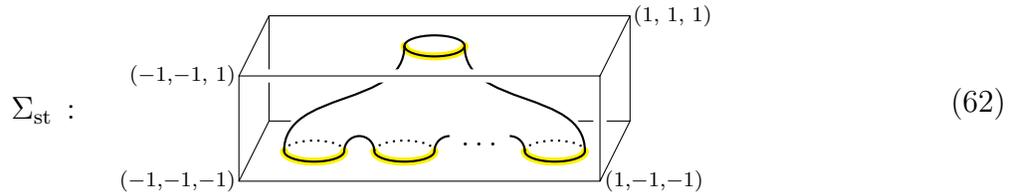
\begin{equation}\label{eq: standard surface in box}
\begin{matrix}
\tikz[scale=.2]{
\draw
(-1,2) rectangle ++(24,7);
\filldraw[white, line width=5]  (15,1) arc (90:0:1) arc (-180:0:2 and .7) to[in=-90, out=90] ++(-8,6) -- ++ (-4,0) to[in=90, out=-90] ++(-8,-6)
arc (-180:0:2 and .7) arc (180:0:1) arc (-180:0:2 and .7) arc (180:90:1) -- cycle;
\draw[yellow, line width=3, line cap=round]  (16,0) arc (-180:0:2 and .7)  ++(-8,7) arc (0:-180:2 and .7)  ++(-8,-7)
arc (-180:0:2 and .7) ++(2,0) arc (-180:0:2 and .7);
\draw[thick]  (15,1) arc (90:0:1) arc (-180:0:2 and .7) to[in=-90, out=90] ++(-8,7) arc (0:360:2 and .7) ++ (-4,0) to[in=90, out=-90] ++(-8,-7)
arc (-180:0:2 and .7) arc (180:0:1) arc (-180:0:2 and .7) arc (180:90:1) node[right, yshift=-3] {$\ldots$};
\draw[dotted, thick] (0,0) arc (180:0:2 and .7) ++(2,0) arc (180:0:2 and .7) ++(6,0) arc (180:0:2 and .7) ++(1,0);
\draw[white, line width=5](-3,-2) rectangle ++(24,7);
\draw
(-3,-2) rectangle ++(24,7) 
(-3,-2) -- +(2,4) ++(0,7) -- +(2,4) (-2,-2) ++(23,0) -- +(2,4) (-3,-2) ++(24,7) -- +(2,4) ;
\node at (-16,2.7) {$\Sigma_{\mathrm{st}}\,:$};
\node[left] at (-2.5,-2) {$\scriptstyle (-1,-1,-1)$};
\node[left] at (-2.5,5) {$\scriptstyle (-1,-1,\,\,1)$};
\node[right] at (20.6,-2) {$\scriptstyle (1,-1,-1)$};
\node[right] at (22.5,9) {$\scriptstyle (1,\,\,1,\,\,1)$};
}
\qquad\qquad
\end{matrix}
\end{equation}
via an isotopy that fixes the boundary.
Here, $\partial \Sigma=\partial \Sigma_{\mathrm{st}}$ consists of $k\ge 0$ `input circles' $\partial^1\Sigma,\ldots,\partial^k\Sigma$ (on the $z=-1$ plane) whose centers have zero $y$-coordinate,
and of one `output circle' $\partial^0\Sigma$ (on the $z=1$ plane) centered at $(0,0,1)$.

We write $\partial_{\text{vis}}\Sigma$ for the part of $\partial\Sigma$ whose $y$-coordinate is negative:
the `visible' part of the boundary, highlighted in yellow in the above picture.
\item 
Disjoint closed discs $C_1,\ldots,C_s\subset \mathring\Sigma$ (the coupons) with marked points $c_i\in\partial C_i$.
\item 
A closed $1$-dimensional oriented submanifold $X\subset \Sigma\setminus (\mathring C_1\cup \ldots \cup \mathring C_s)$ (the strands);
the boundary of $X$ lies in $\partial_{\text{vis}}\Sigma \cup \bigcup\partial C_i$ and does not touch the marked points $c_i$.
\item An object of $\cM$ for each connected component of $X$.
\item A morphism $f_i\in \cM(1, x_1\otimes \cdots \otimes x_k)$ for each coupon $C_i$, as in Definition~\ref{def: labeled oriented anchored planar tangle with coupons}.
\end{itemize} 
\end{defn}

We now explain how to associate an oriented anchored planar tangle with coupons, well defined up to isotopy, to a tube string diagram.
This assignment will be denoted
\begin{equation*} 
\Sigma\,\mapsto\, T(\Sigma).
\end{equation*}
We refer the reader to Figure~\ref{fig: Construct T(Sig)} for an example.

Given a tube string diagram $\Sigma\subset [0,1]^3$, let us write $D_i\subset \partial [0,1]^3$ for the disc bounded by $\partial^i \Sigma$ (so that $\partial^i\Sigma=\partial D_i$),
and let $\Sigma^+:=\Sigma\cup(D_1\cup\ldots\cup D_k)$.
Pick an orientation preserving diffeomorphism $\varphi:\Sigma^+\to\mathbb D:=\{z\in\mathbb C:|z|\le 1\}$ with the property that $\varphi(C_i)$ and $\varphi(D_i)$ are round disks, and let
\[
T:=\mathbb D\setminus (\varphi(\mathring D_1)\cup\ldots\cup \varphi(\mathring D_k)).
\]
We can then use the diffeomorphism $\varphi:\Sigma\to T$ to transfer the strands and coupons from $\Sigma$ to $T$.
Our next task is to equip $T$ with anchor lines.
We will first define the anchor lines on $\Sigma$, and then transfer them to $T$ using $\varphi$.

Let
$
\Gamma:=\partial [0,1]^3\setminus(\mathring D_0\cup\ldots\cup \mathring D_k)$,
and let $W$ be the 3-manifold bounded by $\Gamma\cup\Sigma$ (the solid cube minus the stuff inside $\Sigma$).
Then
\[
\Gamma\hookrightarrow W\hookleftarrow\Sigma
\]
is a trivial cobordism\footnote{A cobordism isomorphic to $\Sigma\times\{0\}\hookrightarrow \Sigma\times[0,1] \cup_{\partial\Sigma\times [0,1]} \partial\Sigma\hookleftarrow \Sigma\times\{1\}$.}. It therefore induces a diffeomorphism $\psi:\Gamma\to\Sigma$, well defined up to isotopy (because surface diffeomorphisms are isotopic if and only if they are homotopic \cite[\S1.4]{MR2850125}).

Let $q_i\in\partial^i\Sigma=\partial D_i$ be the points on the back of the tubes (the points with maximal $y$-coordinate), and
consider the following system of standard lines, $A_\Gamma\subset\Gamma$, which connect $q_0$ to the other points $q_i$.
The $i$th line starts at $q_i$, travels straight to $[-1,1] \times \{1\} \times \{-1\}$, then straight to $[-1,1] \times \{1\} \times \{1\}$, and then straight to $q_0$.
The anchor lines on $\Sigma$ are the image of those lines under the diffeomorphism $\psi:\Gamma\to\Sigma$
\[
A_\Sigma:=\psi(A_\Gamma)
\]
and are characterised by the fact that $A_\Gamma\cup A_\Sigma$ bound $k$ disjoint disks in $W$.
At last, we transfer the anchor lines $A_\Sigma$ from $\Sigma$ to $T$ by means of the diffeomorphism $\varphi:\Sigma\to T$.
This construction establishes a bijection between isotopy classes of tube string diagrams and isotopy classes of labeled oriented anchored planar tangle with coupons.

\begin{figure}[!ht]
$$
\begin{tikzpicture}[baseline=1.6cm]
	\pgfmathsetmacro{\voffset}{.08};
	\pgfmathsetmacro{\hoffset}{.15};
	\pgfmathsetmacro{\hoffsetTop}{.12};
	\coordinate (z1) at (0,-1);  
	\coordinate (z2) at (1.4,-1);  
	\coordinate (a1) at (0,0);  
	\coordinate (a2) at (1.4,0);  
	\coordinate (b1) at (0,2);  
	\coordinate (b2) at (1.4,2);  
	\coordinate (c) at (.7,3.5);  
	\coordinate (d) at (.7,4.5);  
	\coordinate (anchor1) at ($ (z1) + (.45,.08) $);
	\coordinate (anchor2) at ($ (z2) + (.45,.08) $);
	\coordinate (anchor0) at ($ (d) + (.45,.08) $);
	\draw[thick, zString] ($ (b1) + (.23,0) $) .. controls ++(270:.7cm) and ++(90:.8cm) .. ($ (a2) + (.3,-.1) $) -- ($ (z2) + (.3,-.1) $);
	\draw[thick, red, densely dotted]  (anchor2) -- ($ (a2) + (.45,-.15) $) .. controls ++(90:.9cm) and ++(270:.7cm) .. ($ (b1) + 2.8*(\hoffset,0) + (0,-.15*\hoffset) $) .. controls ++(90:.65cm) and ++(270:.85cm) .. ($ (c) + 2*(\hoffsetTop,0) + (0,-\voffset) $) .. controls ++(90:.6cm) and ++(270:.6cm) .. ($ (d) + 3*(\hoffsetTop,0) + (0, -\voffset) $);
	\braid{(a1)}{.3}{2}
	\pairOfPantsNoCircles{(b1)}{}
	\bottomCylinder{(z1)}{.3}{1}
	\bottomCylinder{(z2)}{.3}{1}
	\topCylinder{(c)}{.3}{1}
	\draw[thick, yString] ($ (b2) + 2.7*(\hoffset,0) + (0,-.5*\voffset) $) .. controls ++(270:.6cm) and ++(90:.6cm) .. ($ (a1) + (.3,-.1) $) -- ($ (z1) + (.3,-.1) $);	
	\draw[thick, yString] ($ (b2) + 2.7*(\hoffset,0) + (0,-.5*\voffset) $) .. controls ++(90:.8cm) and ++(270:.8cm) .. ($ (c) + 4*(\hoffsetTop,0) + (0, -\voffset) $) -- ($ (d) + 4*(\hoffsetTop,0) + (0, -\voffset) $);
	\draw[thick, zString] ($ (b1) + (.23,0) $) .. controls ++(90:.7cm) and ++(270:.8cm) .. ($ (c) + 1*(\hoffsetTop,0) + (0, -\voffset) $) .. controls ++(90:.6cm) and ++(270:.6cm) .. ($ (d) + 2*(\hoffsetTop,0) + (0, -\voffset) $);
	\nbox{unshaded}{(1.7,-.5)}{.27}{-.1}{-.1}{$f$}
	\node at ($ (z1) + (.3,-.3) $) {\scriptsize{$y$}};
	\node at ($ (z2) + (.3,-.3) $) {\scriptsize{$z$}};
	\node at ($ (d) + (.6,.3) $) {\scriptsize{$q_0$}};
	\pgfmathsetmacro{\xZero}{-1-.5};
	\pgfmathsetmacro{\xOne}{2.5+.2};
	\pgfmathsetmacro{\yZero}{-1.5};
	\pgfmathsetmacro{\yOne}{-.5};
	\pgfmathsetmacro{\zZero}{-.5};
	\pgfmathsetmacro{\zOne}{4};
	\pgfmathsetmacro{\xZeroBack}{-.5+.1};
	\pgfmathsetmacro{\xOneBack}{3+.68};
	\pgfmathsetmacro{\yOneTop}{5};
	\pgfmathsetmacro{\convex}{.3};
	\coordinate (000) at (\xZero,\yZero);
	\coordinate (100) at (\xOne,\yZero);
	\coordinate (001) at (\xZero,\zOne);
	\coordinate (101) at (\xOne,\zOne);
	\coordinate (010) at (\xZeroBack,\yOne);
	\coordinate (110) at (\xOneBack,\yOne);
	\coordinate (011) at (\xZeroBack,\yOneTop);
	\coordinate (111) at (\xOneBack,\yOneTop);
	\node at ($ (000) + (-.9,0) $) {\scriptsize{$(-1,-1,-1)$}};
	\node at ($ (100) + (.8,0) $) {\scriptsize{$(1,-1,-1)$}};
	\node at ($ (001) + (-.8,0) $) {\scriptsize{$(-1,-1,\,1)$}};
	\node at ($ (111) + (.6,0) $) {\scriptsize{$(1,\,1,\,1)$}};
	\draw[dashed] ($ .5*(000) + .5*(010) $) -- ($ .5*(000) + .5*(010) + (.95,0) $);
	\draw[dashed] ($ .5*(000) + .5*(010) + (1.65,0) $) -- ($ .5*(000) + .5*(010) + (2.35,0) $);
	\draw[dashed] ($ .5*(000) + .5*(010) + (3.05,0) $) -- ($ .5*(100) + .5*(110) $);
	\draw[dashed] ($ .5*(001) + .5*(011) $) -- ($ .5*(001) + .5*(011) + (1.6,0) $);
	\draw[dashed] ($ .5*(001) + .5*(011) + (2.33,0) $) -- ($ .5*(101) + .5*(111) $);
%
	\draw (010) -- (-.1,-.5);
	\draw (.7,-.5) -- (1.3,-.5);
	\draw (2.1,-.5) -- (110);
	\draw (010) -- (011);
	\draw (111) -- (011);
	\draw (111) -- (110);
	\filldraw[red] (anchor1) circle (.05);
	\filldraw[red] (anchor2) circle (.05);
	\filldraw[red] (anchor0) circle (.05);
	\draw[thick, red, densely dotted] ($ (z1) + (0,.08) + (.45,0) $) -- ($ (z1) + (0,.08) + (.45,0) +\convex*(1,.5)- \convex*(0,.08) - \convex*(.45,0)  $);
	\draw[thick, red] ($ (z1) + (1,.5) $) -- ($ (z1) + (0,.08) + (.45,0) +\convex*(1,.5)- \convex*(0,.08) - \convex*(.45,0)  $);
	\draw[thick, red] (1,-.5) -- (1,.5);
	\draw[thick, red] (1,1.45) -- (1,2.5);
	\draw[thick, red] (1,4.7) -- (1,5);
	\draw[thick, red] (1,5) -- (anchor0);
	\draw[thick, red, densely dotted] ($ (z2) + (0,.08) + (.45,0) $) -- ($ (z2) + (0,.105) + (.45,0) +\convex*(1,.5)- \convex*(0,.08) - \convex*(.45,0)  $);
	\draw[thick, red] (2.3,-.5) -- ($ (z2) + (0,.105) + (.45,0) +\convex*(1,.5)- \convex*(0,.08) - \convex*(.45,0)  $);
	\draw[thick, red] (2.3,-.5) -- (2.3,5);
	\draw[thick, red] (anchor0) -- (2.3,5);
%
	\draw[thick, red, densely dotted] ($ (z1) + 3*(\hoffset,0) + (0,\voffset) $) .. controls ++(90:.3cm) and ++(-45:.1cm) .. ($ (z1) + (0,\voffset) + (0,.45)$);
	\draw[thick, red] ($ (a1) + (\hoffset,0) + (0,-\voffset) $) .. controls ++(270:.3cm) and ++(45:.1cm) .. ($ (z1) + (0,\voffset) + (0,.45)$);
%
	\draw[thick, red]  ($ (a1) + 1*(\hoffset,0) + (0,-\voffset) $) .. controls ++(90:.65cm) and ++(270:.65cm) .. ($ (b2) + 1.4*(\hoffset,0) + (0,-.1*\hoffset) $) .. controls ++(90:.65cm) and ++(270:.85cm) .. ($ (c) + 3*(\hoffsetTop,0) + (0,-\voffset) $) .. controls ++(90:.6cm) and ++(-45:.2cm) .. ($ (d) + (0, -.3) $);
%
	\draw[thick, red, densely dotted]  ($ (d) + 1*(\hoffsetTop,0) + (0,-\voffset) $) .. controls ++(270:.1cm) and ++(45:.1cm) .. ($ (d) + (0, -.3) $);
	\draw[thick, red] ($ (d) + 1*(\hoffsetTop,0) + (0,-\voffset) $) .. controls ++(90:.1cm) and ++(-135:.05cm) .. (anchor0);
	\draw[thick, red] ($ (d) + 3*(\hoffsetTop,0) + (0,-\voffset) $) .. controls ++(90:.1cm) and ++(-135:.05cm) .. (anchor0);
%
	\draw[line width=3, white] (001) -- (101);\draw[line width=5, white] (101) -- (100);
	\draw (111) -- (101);
	\draw (001) -- (101);
	\draw (001) -- (011);
	\draw (000) -- (100);
	\draw (000) -- (001);
	\draw (000) -- (010);
	\draw (100) -- (101);
	\draw (100) -- (110);
\node[scale=.98] at (.05,3.1) {$\Sigma$};
\end{tikzpicture}
\!\!\!
\longmapsto
\!\!\!
\qquad
\begin{tikzpicture}[scale=1.3, baseline =-.06cm]
	\coordinate (a) at (-.05,.05);
	\coordinate (c) at (.55,.3);         
	\coordinate (d) at (-.55,-.4);         
	\ncircle{}{(a)}{1.48}{185}{}
\draw[DarkGreen, thick] (d) --node[right, black, pos=.86, xshift=-1.5] {\scriptsize{$z$}} +(-.01,1.84);
\draw[blue, thick] (c) --node[right, black, pos=.63, xshift=-1.5] {\scriptsize{$y$}} +(.005,1.1);
\node[circle, draw, double, inner sep=3.1, scale=.9, fill=white] (x) at (-.55,.68) {$f$};
\fill (x) + (180:.26) circle (.045);
	\ncircle{unshaded}{(d)}{.4}{185}{}
	\ncircle{unshaded}{(c)}{.4}{178}{}

	\draw[thick, red] (c)+(178:.4) .. controls ++(159:1.2cm) and ++(28:.6cm) .. ($(a)+(185:1.5)$);
	\draw[thick, red] (d)+(185:.4) .. controls ++(168:.35cm) and ++(-23:.3cm) .. ($(a)+(185:1.5)$);
\node[scale=.9] at ($(a)+(185:1.72)$) {\anchor};
\node[scale=.98] at ($(a)+(-90:1.95)$) {$T(\Sigma)$};
\end{tikzpicture}
$$
\caption{The construction of an oriented anchored planar tangle with coupons $T(\Sigma)$ from a tube string diagram $\Sigma$
(the green and blue strands are oriented, even though this is not indicated in the picture).
}\label{fig: Construct T(Sig)}
\end{figure}
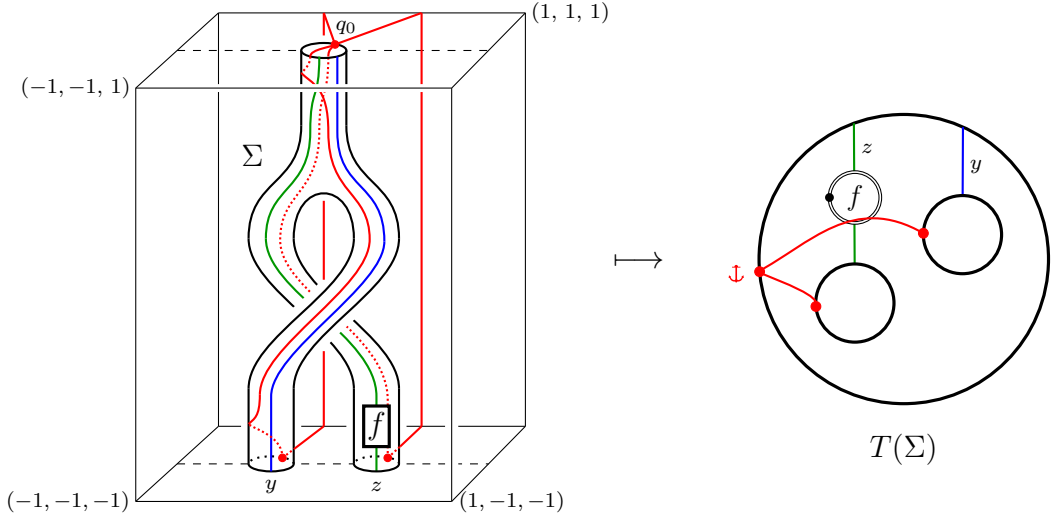

If $\Sigma$ and $\Sigma'$ are two tube string diagrams and if the labels along the $i$th input of $\Sigma$ agree with those along the output of $\Sigma'$, then we can glue them together
in the obvious way to form a new tube string diagram $\Sigma \circ_i \Sigma'$, well defined up to isotopy.

\begin{lem}\label{lem: kxjbksjsvhgvs}
The map $\Sigma\mapsto T(\Sigma)$ is compatible with the operadic composition of tangles and of tube string diagrams:\,
$T(\Sigma\circ_i\Sigma')=T(\Sigma)\circ_iT(\Sigma')$.
\end{lem}

\begin{proof}
Let $\Sigma$ and $\Sigma'$ be two tube string diagrams whose $i$th composition is well defined.
The operations $\Sigma\mapsto T(\Sigma)$ and $\circ_i$ being isotopy invariant, we may assume without loss of generality that the underlying surfaces of $\Sigma$ and $\Sigma'$ are in standard form \eqref{eq: standard surface in box}.
In that case, the construction of $T(\Sigma)$ from $\Sigma$ is very simple and we have:
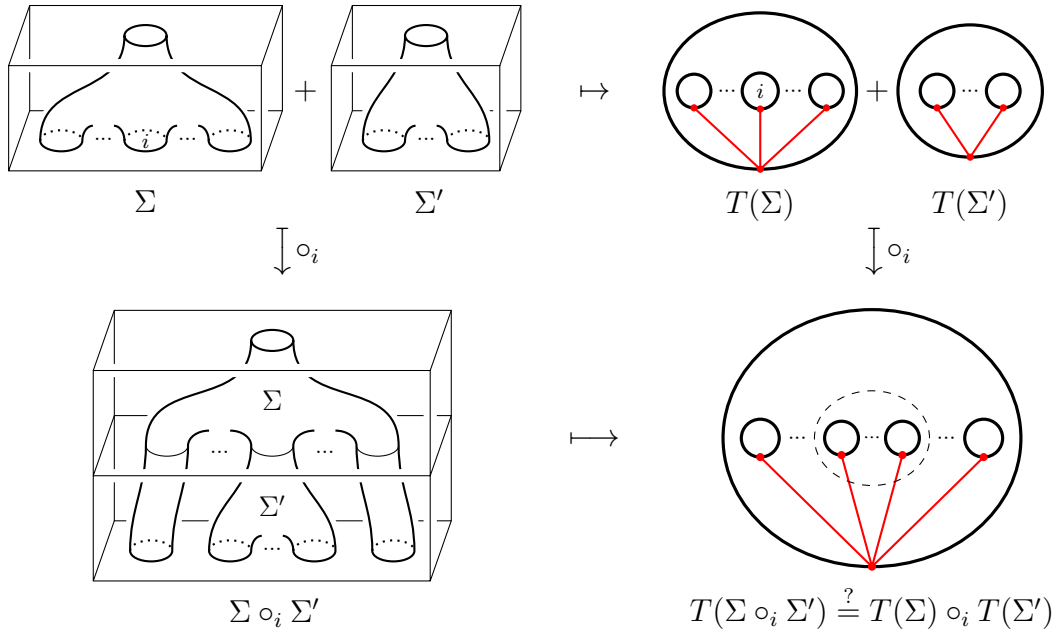
\[
\begin{matrix}
\tikz[scale=.2, xscale=.7, baseline=16]{
\draw
(-1,2) rectangle ++(24,7);
\filldraw[white, line width=5]  (15,1) arc (90:0:1) arc (-180:0:2 and .7) to[in=-90, out=90] ++(-8,6) -- ++ (-4,0) to[in=90, out=-90] ++(-8,-6)
arc (-180:0:2 and .7) arc (180:90:1) ++(1,0) arc (90:0:1) arc (-180:0:2 and .7) arc (180:90:1) -- cycle;
\draw[thick]  (15,1) arc (90:0:1) arc (-180:0:2 and .7) to[in=-90, out=90] ++(-8,7) arc (0:360:2 and .7) ++ (-4,0) to[in=90, out=-90] ++(-8,-7)
arc (-180:0:2 and .7) arc (180:90:1) node[yshift=-5, xshift=4] {$\scriptstyle \ldots$} ++(2,0) arc (90:0:1) arc (-180:0:2 and .7) arc (180:90:1) node[yshift=-5, xshift=4] {$\scriptstyle \ldots$};
\draw[dotted, thick] (0,0) arc (180:0:2 and .7) ++(4,0) arc (180:0:2 and .7) ++(4,0) arc (180:0:2 and .7) ++(1,0);
\draw[white, line width=5](-3,-2) rectangle ++(24,7);
\draw
(-3,-2) rectangle ++(24,7) 
(-3,-2) -- +(2,4) ++(0,7) -- +(2,4) (-2,-2) ++(23,0) -- +(2,4) (-3,-2) ++(24,7) -- +(2,4) ;
\node at (10,0) {$\scriptstyle i$};
\node at (10,-4) {$\Sigma$};
}
\;\!+
\tikz[scale=.2, xscale=.7, baseline=16]{
\draw
(-1,2) rectangle ++(16,7);
\filldraw[white, line width=5]  (7,1) arc (90:0:1) arc (-180:0:2 and .7) to[in=-90, out=90] ++(-4,7) -- ++ (-4,0) to[in=90, out=-90] ++(-4,-7)
arc (-180:0:2 and .7) arc (180:90:1) -- cycle;
\draw[thick]  (7,1) arc (90:0:1) arc (-180:0:2 and .7) to[in=-90, out=90] ++(-4,7) arc (0:360:2 and .7) ++ (-4,0) to[in=90, out=-90] ++(-4,-7)
arc (-180:0:2 and .7) arc (180:90:1) node[yshift=-5, xshift=4] {$\scriptstyle \ldots$};
\draw[dotted, thick] (0,0) arc (180:0:2 and .7) ++(4,0) arc (180:0:2 and .7);
\draw[white, line width=5](-3,-2) rectangle ++(16,7);
\draw
(-3,-2) rectangle ++(16,7) 
(-3,-2) -- +(2,4) ++(0,7) -- +(2,4) (-2,-2) ++(15,0) -- +(2,4) (-3,-2) ++(16,7) -- +(2,4) ;
\node at (6.3,-4) {$\Sigma'$};
}
&\,\,\mapsto\,&
\begin{tikzpicture}[baseline=-.1cm,scale=.8]
	\draw[thick, red] (1.1,-.28) -- (0,-1.3) (0,-.28) -- (0,-1.3) (-1.1,-.28) -- (0,-1.3);
	\draw[very thick] (0,0) circle (1.6 and 1.3);
	\draw[unshaded, very thick] (1.1,0) circle (.28cm);
	\draw[unshaded, very thick] (0,0) circle (.3cm);
	\draw[unshaded, very thick] (-1.1,0) circle (.28cm);
	\fill[red] (1.1,-.28) circle (.06) (0,-.3) circle (.06) (-1.1,-.28) circle (.06) (0,-1.3) circle (.06);
\node at (0,0) {$\scriptstyle i$};
\node at (-.55,0) {$\;\!\scriptstyle \cdots$};
\node at (.55,0) {$\;\!\scriptstyle \cdots$};
\node at (0,-1.9) {$T(\Sigma)$};
\end{tikzpicture}
+
\begin{tikzpicture}[baseline=-.1cm,scale=.8]
	\draw[thick, red] (.55,-.28) -- (0,-1.1) (-.55,-.28) -- (0,-1.1);
	\draw[very thick] (0,0) circle (1.2 and 1.1);
	\draw[unshaded, very thick] (.55,0) circle (.28cm);
	\draw[unshaded, very thick] (-.55,0) circle (.28cm);
	\fill[red] (.55,-.28) circle (.06) (-.55,-.28) circle (.06) (0,-1.1) circle (.06);
\node at (0,0) {$\;\!\scriptstyle \cdots$};
\node at (0,-1.9) {$T(\Sigma')$};
\end{tikzpicture}
\\
\qquad\tikz[baseline=-2.5]{
\useasboundingbox (-.2,-.75) rectangle (.2,.25);
 \node[rotate=-90]{$\longmapsto$};}\circ_i & & \qquad\tikz[baseline=-2.5]{
 \useasboundingbox (-.2,-.75) rectangle (.2,.25);
 \node[rotate=-90]{$\longmapsto$};}\circ_i
\\
\tikz[scale=.2, xscale=.7, baseline=0]{
\draw
(-5,-5) rectangle ++(32,14) (-5,2) -- +(32,0);
\filldraw[white, line width=5]  (16,1) arc (90:0:2 and 1) to[in=90, out=-90] ++(1.5,-7) arc (-180:0:2 and .7) to[in=-90, out=90] ++(-1.5,7)  to[in=-90, out=90] ++(-10,7) -- ++ (-4,0) to[in=90, out=-90] ++(-10,-7) to[in=90, out=-90] ++(-1.5,-7) arc (-180:0:2 and .7) to[in=-90, out=90] ++(1.5,7) arc (180:90:2 and 1) -- ++(2,0) arc (90:0:2 and 1) to[in=90, out=-90] ++(-4,-7) arc (-180:0:2 and .7) arc (180:90:1) -- ++(2,0) arc (90:0:1) arc (-180:0:2 and .7) to[in=-90, out=90] ++(-4,7) arc (180:90:2 and 1) -- cycle;
\draw[thick] (16,1) arc (90:0:2 and 1) to[in=90, out=-90] ++(1.5,-7) arc (-180:0:2 and .7) to[in=-90, out=90] ++(-1.5,7)  to[in=-90, out=90] ++(-10,7) arc (0:360:2 and .7) ++ (-4,0) to[in=90, out=-90] ++(-10,-7) to[in=90, out=-90] ++(-1.5,-7) arc (-180:0:2 and .7) to[in=-90, out=90] ++(1.5,7) arc (180:90:2 and 1) node[yshift=-8, xshift=4] {$\scriptstyle \ldots$} ++(2,0) arc (90:0:2 and 1) to[in=90, out=-90] ++(-4,-7) arc (-180:0:2 and .7) arc (180:90:1) node[yshift=-5, xshift=4] {$\scriptstyle \ldots$} ++(2,0) arc (90:0:1) arc (-180:0:2 and .7) to[in=-90, out=90] ++(-4,7) arc (180:90:2 and 1) node[yshift=-8, xshift=4] {$\scriptstyle \ldots$};
\draw (-2,0) arc (-180:0:2 and .7) ++(6,0) arc (-180:0:2 and .7) ++(6,0) arc (-180:0:2 and .7);
\draw[dotted, thick] (-3.5,-7) arc (180:0:2 and .7) ++(3.5,0) arc (180:0:2 and .7) ++(4,0) arc (180:0:2 and .7) ++(3.5,0) arc (180:0:2 and .7);
\draw[white, line width=5](-7,-9) rectangle ++(32,14) (-7,-2) -- +(32,0);
\draw (-7,-9) rectangle ++(32,14) (-7,-2) -- +(32,0)
(-7,-9) -- +(2,4) ++ (32,0) -- +(2,4) (-7,-2) -- +(2,4) ++(0,7) -- +(2,4) (-2,-2) ++(27,0) -- +(2,4) (-3,-2) ++(28,7) -- +(2,4) ;
\node[scale=.9] at (10,3) {$\Sigma$};
\node[scale=.9] at (10,-4) {$\Sigma'$};
\node at (10.2,-11.1) {$\Sigma\circ_i\Sigma'$};
}
\hspace{-.2cm}
&\,\,\,\,\longmapsto\,\,\,&
\begin{tikzpicture}[baseline=-.1cm,scale=.9]
	\draw[thick, red] (1.65,-.28) -- (0,-1.9) (.45,-.25) -- (0,-1.9) (-.45,-.25) -- (0,-1.9) (-1.65,-.28) -- (0,-1.9);
	\draw[dashed] (0,0) circle (.85 and .7);
	\draw[very thick] (0,0) circle (2.2 and 1.9);
	\draw[unshaded, very thick] (1.65,0) circle (.28cm);
	\draw[unshaded, very thick] (.45,0) circle (.25cm);
	\draw[unshaded, very thick] (-.45,0) circle (.25cm);
	\draw[unshaded, very thick] (-1.65,0) circle (.28cm);
	\fill[red] (1.65,-.28) circle (.06) (-.45,-.25) circle (.06) (.45,-.25) circle (.06) (-1.65,-.28) circle (.06) (0,-1.9) circle (.06);
\node at (-1.1,0) {$\;\!\scriptstyle \cdots$};
\node[scale=.9] at (0,0) {$\;\!\scriptstyle \cdots$};
\node at (1.1,0) {$\;\!\scriptstyle \cdots$};
\node at (0,-2.5) {$\!\!T(\Sigma\circ_i\Sigma')\stackrel{?}=T(\Sigma)\circ_iT(\Sigma')\!\!$};
\end{tikzpicture}
\hspace{-.5cm}
\end{matrix}
\]
Let $T$ be the manifold depicted on the bottom right (a disc with holes), along with the anchor lines as drawn.
The tangle $T(\Sigma\circ_i\Sigma')$ is the result of taking $T$ and 
transferring onto it the pattern of strands and coupons from $\Sigma$ and $\Sigma'$ using the diffeomorphism
$\Sigma\cup_i\Sigma'\to \Sigma\circ_i\Sigma' \to T$.
Similarly, the tangle $T(\Sigma)\circ_iT(\Sigma')$ is obtained by transferring that same pattern of strands and coupons via the diffeomorphisms
$\Sigma\cup_i\Sigma'\to T(\Sigma)\cup_i T(\Sigma') \to T$.
Those diffeomorphisms are clearly isotopic, so $T(\Sigma\circ_i\Sigma')=T(\Sigma)\circ_iT(\Sigma')$.
\end{proof}

Given a ribbon braid $\sigma\in\cR\cB_n$ and a tube string diagram $\Sigma$ with $n$ inputs, one can form a new tube string diagram $\Sigma\cdot\sigma$ by gluing $\sigma$ on the bottom of $\Sigma$ .
For example:
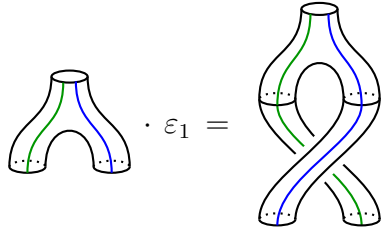
\[
\begin{tikzpicture}[baseline=-.35cm, scale=.8, xscale=-1]
	\pgfmathsetmacro{\voffset}{.08};
	\pgfmathsetmacro{\hoffset}{.15};

	\coordinate (a) at (-1,-1);
	\coordinate (a1) at ($ (a) + (1.4,0)$);
	\coordinate (a2) at ($ (a1) + (1.4,0)$);
	\coordinate (b) at ($ (a) + (.7,1.5)$);
	\coordinate (b2) at ($ (b) + (1.4,0)$);
	\coordinate (c) at ($ (b) + (.7,1.5)$);	

	\topPairOfPants{(a)}{}
	\draw[thick, yString] ($ (a) + 2*(\hoffset,0) + (0,-.1)$) .. controls ++(90:.8cm) and ++(270:.8cm) .. ($ (b) + 1.3*(\hoffset,0) + (0,-\voffset) $);	
	\draw[thick, zString] ($ (a1) + 2*(\hoffset,0) + (0,-.1)$) .. controls ++(90:.8cm) and ++(270:.8cm) .. ($ (b) + 2.7*(\hoffset,0) + (0,-.1) $);	
\end{tikzpicture}
\,\cdot\, \varepsilon_1\,=\;
\begin{tikzpicture}[scale=.8, baseline=1.15cm]
	\pgfmathsetmacro{\voffset}{.08};
	\pgfmathsetmacro{\hoffset}{.15};
	\pgfmathsetmacro{\hoffsetTop}{.12};
	\coordinate (a1) at (0,0);
	\coordinate (a2) at (1.4,0);
	\coordinate (b1) at (0,2);
	\coordinate (b2) at (1.4,2);
	\coordinate (c) at (.7,3.5);
	\draw[thick, zString] ($ (b1) + (.3,-.1) $) .. controls ++(270:.6cm) and ++(90:.8cm) .. ($ (a2) + (.3,-.1) $);		
	\braid{(a1)}{.3}{2}
	\topPairOfPants{(b1)}{}
	\halfDottedEllipse{(a1)}{.3}{.1}
	\halfDottedEllipse{(a2)}{.3}{.1}
	\draw[thick, yString] ($ (b2) + (.3,-.1) $) .. controls ++(270:.6cm) and ++(90:.8cm) .. ($ (a1) + (.3,-.1) $);	
	\draw[thick, yString] ($ (b2) + (.3,-.1) $) .. controls ++(90:.8cm) and ++(270:.8cm) .. ($ (c) + 3*(\hoffset,0) + (0, -\voffset) $);
	\draw[thick, zString] ($ (b1) + (.3,-.1) $) .. controls ++(90:.8cm) and ++(270:.8cm) .. ($ (c) + 1*(\hoffset,0) + (0, -\voffset) $);
\end{tikzpicture}
\]
(we write $\varepsilon_i$ and $\vartheta_i$ for the generators of $\cR\cB_n$, as in Section \ref{sec:RibbonBraidGroup}).
This defines a right\footnote{Here, we read braids from top to bottom (e.g., $\varepsilon_1\varepsilon_2=
\tikz[scale=.3, baseline=1.5]{
\draw (0,1) circle (.3 and .1);
\draw (1,1) circle (.3 and .1);
\draw (2,1) circle (.3 and .1);
\draw (0,0) +(.3,0) arc (0:-180:.3 and .1);
\draw (1,0) +(.3,0) arc (0:-180:.3 and .1);
\draw (2,0) +(.3,0) arc (0:-180:.3 and .1);
\draw[densely dotted] (0,0) +(.3,0) arc (0:180:.3 and .1);
\draw[densely dotted] (1,0) +(.3,0) arc (0:180:.3 and .1);
\draw[densely dotted] (2,0) +(.3,0) arc (0:180:.3 and .1);
\draw (2-.3,0) .. controls ++(90:.4) and ++(-90:.6) .. (0-.3,1);
\draw (2+.3,0) .. controls ++(90:.6) and ++(-90:.4) .. (0+.3,1);
\fill[white] (.54,.85) -- (1.15,.65) -- (.68,.37) -- (.05,.58)  -- cycle;
\fill[white] (2-.54,1-.85) -- (2-1.15,1-.65) -- (2-.68,1-.37) -- (2-.05,1-.58)  -- cycle;
\draw (0-.3,0) .. controls ++(85:.6) and ++(-90:.3) .. (1-.3,1);
\draw (0+.3,0) .. controls ++(90:.4) and ++(-95:.45) .. (1+.3,1);
\draw (1-.3,0) .. controls ++(85:.45) and ++(-90:.4) .. (2-.3,1);
\draw (1+.3,0) .. controls ++(90:.3) and ++(-95:.6) .. (2+.3,1);
}$). This is the opposite of the convention used in Section~\ref{sec:RibbonBraidGroup}.} action
 of the ribbon braid group on the set of isotopy classes of tube string diagrams with $n$ inputs.

For a fixed planar tangle $T$ with $n$ inputs, there is also a right action of $\cR\cB_n$ on the set of systems of anchor lines on $T$.
We now describe that action.
Let $T_{\mathrm{st}}$ be the `standard tangle', with input circles along the $y$-axis, and standard system of anchor lines $A_{\mathrm{st}}$, as in Section~\ref{sec:RibbonBraidGroup}. 
Given a system of anchor lines $A$ on $T$, pick a diffeomorphism $\varphi:T_{\mathrm{st}}\to T$ (disregarding strands and coupons) that sends $A_{\mathrm{st}}$ to $A$.
We define:
\[
A\cdot\sigma\,:=\,\varphi(\sigma\cdot A_{\mathrm{st}}),
\]
where $\sigma\cdot A_{\mathrm{st}}$ is as described in Section \ref{sec:RibbonBraidGroup}.
Equivalently, if $\psi_\sigma$ is the diffeomorphism of $T_{\mathrm{st}}$ defined in the picture \eqref{eq: Here is an example of this process}, so that $\sigma\cdot A_{\mathrm{st}}=\psi_\sigma(A_{\mathrm{st}})$,
then the right action of $\sigma$ on $A=\varphi(A_{\mathrm{st}})$ is given by:
\[
\varphi(A_{\mathrm{st}})\cdot\sigma:=\varphi\circ\psi_\sigma(A_{\mathrm{st}}).
\]
Note that, by taking $\varphi=\psi_{\sigma'}$ in the above formula for some $\sigma'\in\cR\cB_n$, one gets the relation
$(\sigma'\cdot A_{\mathrm{st}})\cdot \sigma=\sigma'\,\sigma \cdot A_{\mathrm{st}}$.

\begin{lem}\label{lem: kxjbksjsvhgvs   hmmm.... }
The map $\Sigma\mapsto T(\Sigma)$ is equivarient with respect to the above right actions of the ribbon braid group:\,
$T(\Sigma\cdot\sigma)=T(\Sigma)\cdot\sigma$.
\end{lem}

\begin{proof}
It is enough to prove the lemma for the generators $\varepsilon_i$ and $\vartheta_i$.
If $\sigma=\vartheta_i$, then $T(\Sigma\cdot\vartheta_i)$ differs from $T(\Sigma)$ be the fact that its strands make an extra twist around the $i$th input disc,
whereas $T(\Sigma)\cdot\vartheta_i$ differs from $T(\Sigma)$ be the fact that the $i$th anchor line makes an extra twist.
Those two anchored planar tangles are isotopic by means of a Dehn twist. 

We now treat the case $\sigma=\varepsilon_i$.
Let $\Sigma$ be a tube string diagram.
The operation $\Sigma\mapsto T(\Sigma)$ and the right actions of $\sigma$ being isotopy invariant, we may assume, as in Lemma~\ref{lem: kxjbksjsvhgvs}, that the underlying surface of $\Sigma$ is in standard form.
To compute $T(\Sigma\cdot\varepsilon_i)$, one starts with $T$, adds a braid between the $i$th and $(i+1)$st legs, and computes the anchor lines, as in Figure~\ref{fig: Construct T(Sig)}.
This yields the following picture:
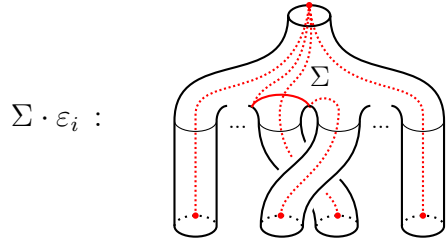
\[
\Sigma\cdot\varepsilon_i\,:\qquad
\tikz[scale=.2, xscale=.7, baseline=-2]{
\draw[thick]
(7.8,1) arc (90:0:1) -- ++(0,-7) arc (-180:0:2 and .7) -- ++(0,7)  to[in=-90, out=90] ++(-10.8,7) arc (0:360:2 and .7) ++ (-4,0) to[in=90, out=-90] ++(-10.8,-7) -- ++(0,-7) arc (-180:0:2 and .7) -- ++(0,7) arc (180:90:1) node[yshift=-8, xshift=4] {$\scriptstyle \ldots$} ++(2,0) arc (90:0:1) 
.. controls ++(-90:4) and ++(90:3.2) ..
 ++(5.4,-7) arc (-180:0:2 and .7) 
.. controls ++(90:4) and ++(-90:3.2) ..
 ++(-5.4,7) arc (180:90:.8 and  1);
\draw[red, line width=.8, densely dotted]
(-10.8,-7)++(0,.7) -- ++(0,6.7) .. controls ++(90:3) and ++(-100:7) .. (0,7.7)
(10.8,-7)++(0,.7) -- ++(0,6.7) .. controls ++(90:3) and ++(-80:7) .. (0,7.7)
(2.7,-7)++(0,.7) .. controls ++(90:2) and ++(-90:4) .. (-2.8,.4) .. controls ++(90:3) and ++(-80:6) .. (0,7.7);
\filldraw[white, line width=5] (4.8,0)
.. controls ++(-90:4) and ++(90:3.2) ..
 ++(-5.4,-7) arc (0:-180:2 and .7) 
.. controls ++(90:4) and ++(-90:3.2) .. ++(5.4,7) -- cycle;
\draw[red, line width=.8, densely dotted]
(-2.7,-7)++(0,.7) .. controls ++(90:2) and ++(-90:4) .. (2.8,.7) arc (0:168:1.4 and .8);
\draw[red, line width=.8] (0,1) arc (15:170:2.8 and 1) coordinate (d);
\draw[red, line width=.8, densely dotted] (d) .. controls ++(60:3) and ++(-94:6) .. (0,7.7);
\draw[thick]
(0,1) arc (90:0:.8 and  1)
.. controls ++(-90:3.2) and ++(90:4) ..
 ++(-5.4,-7) arc (-180:0:2 and .7) 
.. controls ++(90:3.2) and ++(-90:4) ..
 ++(5.4,7) arc (180:90:1) node[yshift=-8, xshift=4] {$\scriptstyle \ldots$};
\draw
(-12.8,0) arc (-180:0:2 and .7)
(8.8,0) arc (-180:0:2 and .7)
(.8,0) arc (-180:0:2 and .7)
(-.8,0) arc (0:-180:2 and .7);
\draw[dotted, thick]
(-12.8,-7) arc (180:0:2 and .7)
(8.8,-7) arc (180:0:2 and .7)
(.7,-7) arc (180:0:2 and .7)
(-.7,-7) arc (0:180:2 and .7);
	\fill[red] 	(-12.8,-7)+(2,.7) circle (.3 and .21)
			(8.8,-7)+(2,.7) circle (.3 and .21)
			(.7,-7)+(2,.7) circle (.3 and .21)
			(-.7,-7)+(-2,.7) circle (.3 and .21)
			(0,7.7) circle (.3 and .21);	
\node[scale=.9] at (1,3) {$\Sigma$};
}
\]
Now, by `shortening the legs', one can map the above picture, via a diffeomorphism, back to $\Sigma$.
Note that the anchor lines are now twisted:
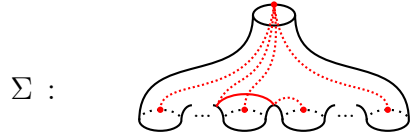
\[
\Sigma\,:\qquad
\tikz[scale=.2, xscale=.7, baseline=8]{
\useasboundingbox (-14,-1.5) rectangle (14,8.5);
\draw[thick]
(7.8,1) arc (90:0:1) arc (-180:0:2 and .7) to[in=-90, out=90] ++(-10.8,7) arc (0:360:2 and .7) ++ (-4,0) to[in=90, out=-90] ++(-10.8,-7)  arc (-180:0:2 and .7) arc (180:90:1) node[yshift=-4, xshift=4] {$\scriptstyle \ldots$} ++(2,0) arc (90:0:1) arc (-180:0:2 and .7)  arc (180:90:.8 and  1);
\draw[red, line width=.8, densely dotted]
(-10.8,-7)++(0,.7) ++(0,7) .. controls ++(90:2.5) and ++(-100:7) .. (0,7.7)
(10.8,-7)++(0,.7) ++(0,7) .. controls ++(90:2.5) and ++(-80:7) .. (0,7.7)
(-2.8,-7)++(0,.7) ++(0,7) .. controls ++(90:2.5) and ++(-80:6) .. (0,7.7);
\draw[red, line width=.8, densely dotted]
(2.7,-7)++(0,.7) ++(0,7) arc (0:160:1.4 and .8);
\draw[red, line width=.8] (0,1) arc (15:170:2.8 and 1) coordinate (d);
\draw[red, line width=.8, densely dotted] (d) .. controls ++(60:3) and ++(-94:6) .. (0,7.7);
\draw[thick]
(0,1) arc (90:0:.8 and  1)
arc (-180:0:2 and .7) 
arc (180:90:1) node[yshift=-4, xshift=4] {$\scriptstyle \ldots$};
\draw[dotted, thick]
(-12.8,0) arc (180:0:2 and .7)
(8.8,0) arc (180:0:2 and .7)
(.8,0) arc (180:0:2 and .7)
(-.8,0) arc (0:180:2 and .7);
	\fill[red] 	(-12.8,-7)+(2,7.7) circle (.3 and .21)
			(8.8,-7)+(2,7.7) circle (.3 and .21)
			(.7,-7)+(2.1,7.7) circle (.3 and .21)
			(-.7,-7)+(-2.1,7.7) circle (.3 and .21)
			(0,7.7) circle (.3 and .21);	
}
\]
One can then map these anchor lines onto the standard tangle $T_{\mathrm{st}}$, using the standard diffeomorphism $\varphi_{\mathrm{st}}:\Sigma\to T_{\mathrm{st}}$ which `flattens things out':
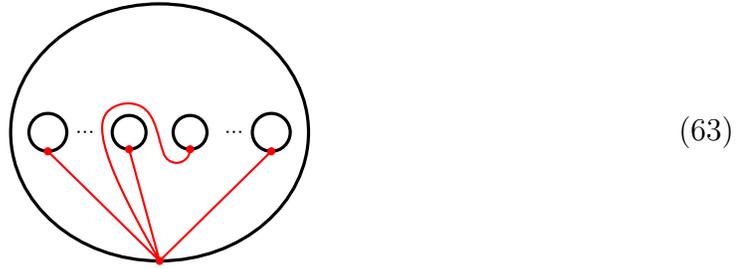
\begin{equation}\label{slgkmslgbsljf}
\quad\qquad
\begin{tikzpicture}[baseline=-.1cm,scale=.9]
	\draw[thick, red] (1.65,-.28) -- (0,-1.9) (-.45,-.25) -- (0,-1.9) (-1.65,-.28) -- (0,-1.9)
	(.45,-.25) arc(0:-165:.2) .. controls ++(110:.3cm) and ++(0:.4cm) .. ++ (-.52,.73) arc(90:120:.45) .. controls ++(210:.3cm) and ++(120:2cm) .. (0,-1.9);
	\draw[very thick] (0,0) circle (2.2 and 1.9);
	\draw[unshaded, very thick] (1.65,0) circle (.28cm);
	\draw[unshaded, very thick] (.45,0) circle (.25cm);
	\draw[unshaded, very thick] (-.45,0) circle (.25cm);
	\draw[unshaded, very thick] (-1.65,0) circle (.28cm);
	\fill[red] (1.65,-.28) circle (.06) (-.45,-.25) circle (.06) (.45,-.25) circle (.06) (-1.65,-.28) circle (.06) (0,-1.9) circle (.06);
\node at (-1.1,0) {$\;\!\scriptstyle \cdots$};
\node at (1.1,0) {$\;\!\scriptstyle \cdots$};
\end{tikzpicture}
\end{equation}
The tangle $T(\Sigma\cdot\varepsilon_i)$ is the above picture, along with the pattern of stands and coupons transferred from $\Sigma$ via the diffeomorphism $\varphi_{\mathrm{st}}$.

The anchor lines \eqref{slgkmslgbsljf} are equal to\,
\begin{tikzpicture}[baseline=-.1cm,scale=.5]
	\draw[thick, red] (1.65,-.28) -- (0,-1.9) (-.45,-.25) -- (0,-1.9) (-1.65,-.28) -- (0,-1.9)
	(.45,-.25) -- (0,-1.9);
	\draw[very thick] (0,0) circle (2.2 and 1.9);
	\draw[unshaded, very thick] (1.65,0) circle (.28cm);
	\draw[unshaded, very thick] (.45,0) circle (.25cm);
	\draw[unshaded, very thick] (-.45,0) circle (.25cm);
	\draw[unshaded, very thick] (-1.65,0) circle (.28cm);
	\fill[red] (1.65,-.28) circle (.06) (-.45,-.25) circle (.06) (.45,-.25) circle (.06) (-1.65,-.28) circle (.06) (0,-1.9) circle (.06);
\node at (-1.1,0) {$\;\!\scriptstyle \cdots$};
\node at (1.1,0) {$\;\!\scriptstyle \cdots$};
\end{tikzpicture}
${}\cdot\varepsilon_i$\,, and so $T(\Sigma\cdot\varepsilon_i)=T(\Sigma)\cdot\varepsilon_i$.
\end{proof}

Let $P(\sigma)$ is as in \eqref{action of ribbon braid on iterated tensor product}.

\begin{lem}\label{lem: kxjbksjsvhgvs   ouahahahaaaahahaaa!!!....!!.!!.... }
The assignment $T\mapsto Z(T)$ defined in \eqref{sfwagiletjnad} is compatible with the right action of the braid group:\, $Z(T\cdot\sigma)=Z(T)\circ P(\sigma)$.
\end{lem}
\begin{proof}
We first prove the lemma when $T$ has no coupons. Let $T$ be such a tangle.
Put $T$ in standard form (Section~\ref{sec:StandardForm}) and write its anchor lines as $A=\sigma_T\cdot A_{\mathrm{st}}$, as in Algorithm~\ref{alg:AssignMap}.
By \eqref{eq: Z(T) = Z(T_st) o P(s_T)}, we then have $Z(T)=Z(T,A_{\mathrm{st}})\circ P(\sigma_T)$, and so:
\[
\begin{split}
Z(T\cdot\sigma)
&= Z((T,\sigma_T\cdot A_{\mathrm{st}})\cdot\sigma)\\
&=Z(T,\sigma_T\,\sigma\cdot A_{\mathrm{st}})\\
&=Z(T,A_{\mathrm{st}})\circ P(\sigma_T\,\sigma)\\
&=Z(T,A_{\mathrm{st}})\circ P(\sigma_T)\circ P(\sigma)\\
&=Z(T)\circ P(\sigma).
\end{split}
\]
If $T$ has coupons, then we can write it as a composite $T=T'\circ_1T''$ of an annular tangle $T'$ (with coupons), and a tangle without coupons $T''$.
We then have:
\[
\begin{split}
Z(T\cdot\sigma)
&=Z((T'\circ_1T'')\cdot\sigma)\\
&=Z(T'\circ_1(T''\cdot\sigma))\\
&=Z(T')\circ Z(T''\cdot\sigma)\\
&=Z(T')\circ Z(T'')\circ P(\sigma)\\
&=Z(T'\circ_1T'')\circ P(\sigma)\\
&=Z(T)\circ P(\sigma).\qedhere
\end{split}
\]
\end{proof}


Given a tube string diagram labelled by objects and morphisms of $\cM$, we now define its evaluation in $\cC$.
Consider the following `elementary tube string diagrams'
(all the strands are oriented, even though this is not visible in our pictures)
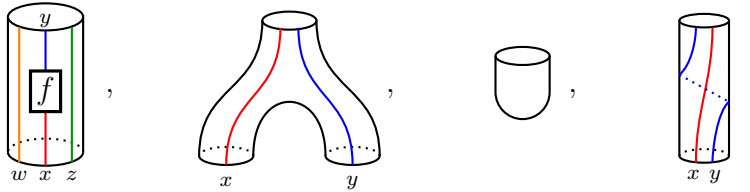
\begin{equation*} 
\begin{tikzpicture}[baseline=-.1cm, yscale=.9]
	\draw[thick] (-.5,-1) -- (-.5,1);
	\draw[thick] (.5,-1) -- (.5,1);
	\draw[thick] (0,1) ellipse (.5cm and .2cm);
	\halfDottedEllipse{(-.5,-1)}{.5}{.2}
	
	\draw[thick, wString] (-.35,-1.15) -- (-.35,.85);
	\draw[thick, xString] (0,-1.2) -- (0,0);
	\draw[thick, yString] (0,.8) -- (0,0);
	\draw[thick, zString] (.35,-1.15) -- (.35,.85);
	\nbox{unshaded}{(0,-.1)}{.3}{-.1}{-.1}{$f$}
	\node at (-.35,-1.35) {$\scriptstyle w$};
	\node at (0,-1.35) {$\scriptstyle x$};
	\node at (0,.95) {$\scriptstyle y$};
	\node at (.35,-1.35) {$\scriptstyle z$};
\end{tikzpicture}\;\;,\;\;\;
\quad\,\,\,\,
\begin{tikzpicture}[baseline=-.35cm, scale=1.2]
	\topPairOfPants{(-1,-1)}{}
	\draw[thick, xString] (-.7,-1.1) .. controls ++(90:.8cm) and ++(270:.8cm) .. (-.1,.42);		
	\draw[thick, yString] (.7,-1.1) .. controls ++(90:.8cm) and ++(270:.8cm) .. (.1,.42);		
	\node at (-.7,-1.28) {$\scriptstyle x$};
	\node at (.7,-1.3) {$\scriptstyle y$};
\end{tikzpicture}\,,\;
\quad\,\,\,\,\,\,\,\;\;
\begin{tikzpicture}[baseline=0cm, scale=1.2]
	\coordinate (a1) at (0,0);
	\coordinate (b1) at (0,.4);
	\draw[thick] (a1) -- (b1);
	\draw[thick] ($ (a1) + (.6,0) $) -- ($ (b1) + (.6,0) $);
	\draw[thick] ($ (b1) + (.3,0) $) ellipse (.3 and .1);
	\draw[thick] (a1) arc (-180:0:.3cm);
\end{tikzpicture}\;\;,
\quad\,\,\,\,\,\,\,\;\;\;
\begin{tikzpicture}[baseline=-.05cm, yscale=.9, xscale=1.1]
	\draw[thick] (-.3,-1) -- (-.3,1);
	\draw[thick] (.3,-1) -- (.3,1);
	\draw[thick] (0,1) ellipse (.3cm and .1cm);
	\halfDottedEllipse{(-.3,-1)}{.3}{.1}
	
	\draw[thick, xString] (-.1,-1.1) .. controls ++(90:.8cm) and ++(270:.8cm) .. (.1,.9);		
	\draw[thick, yString] (.1,-1.1) .. controls ++(90:.2cm) and ++(225:.2cm) .. (.3,-.2);		
	\draw[thick, yString] (-.1,.9) .. controls ++(270:.2cm) and ++(45:.2cm) .. (-.3,.2);
	\draw[thick, yString, dotted] (-.3,.2) -- (.3,-.2);	
	\node at (-.13,-1.28) {$\scriptstyle x$};
	\node at (.13,-1.3) {$\scriptstyle y$};
\end{tikzpicture}\;\;,
\end{equation*}
possibly modified by replacing each stand by many parallel strands (or no strands), changing the orientations of the strands, and replacing the coupon by a cup or cap.
(To conform to the definition, we should really have replaced the morphism $f:x\to y$ by its mate $\tilde f:1\to y\otimes x^*$, and equipped the coupon with a marked point on its left side, but we will ignore these details.)
Suppose that $\Sigma$ is a tube string diagram which is given to us as a composition of elementary tube string diagrams $\Sigma_i$ 
and a braid $\sigma\in\cB_n$:
\begin{equation}\label{eq: tube string diagram ready to be evaluated}
\Sigma\,=\,\big[\big[\big[\Sigma_1\circ_{j_1}\Sigma_2\big]\circ_{j_2}\Sigma_3\ldots \big]\circ_{j_{k-1}}\Sigma_k \big]\cdot \sigma.
\end{equation}
Then, following \cite{1509.02937}, we define
\begin{equation}\label{eq: definition of Z(tube string diagram)}
Z(\Sigma):=\big[\big[\big[Z(\Sigma_1)\circ_{j_1}Z(\Sigma_2)\big]\circ_{j_2}\ldots \big]\circ_{j_{k-1}}Z(\Sigma_k) \big]\circ P(\sigma),
\end{equation}
where the value of $Z(-)$ 
on the elementary tube string diagrams is given by
\[
Z\Bigg(\;\!\;\!\!\begin{tikzpicture}[baseline=-.1cm, xscale=1.1, scale=.6]

	\draw[thick] (-.5,-1) -- (-.5,1);
	\draw[thick] (.5,-1) -- (.5,1);
	\draw[thick] (0,1) ellipse (.5cm and .2cm);
	\halfDottedEllipse{(-.5,-1)}{.5}{.2}
	
	\draw[thick, wString] (-.35,-1.15) -- (-.35,.85);
	\draw[thick, xString] (0,-1.2) -- (0,0);
	\draw[thick, yString] (0,.8) -- (0,0);
	\draw[thick, zString] (.35,-1.15) -- (.35,.85);
	\nbox{unshaded}{(0,-.1)}{.3}{-.1}{-.1}{$\scriptstyle f$}
	\node at (-.35,-1.35) {$\scriptstyle w$};
	\node at (0,-1.35) {$\scriptstyle x$};
	\node at (0,.95) {$\scriptstyle y$};
	\node at (.35,-1.35) {$\scriptstyle z$};
\end{tikzpicture}\;\!\;\!\Bigg)\!:=\Tr_\cC(\id_w \otimes f \otimes \id_z),
\,\,\,\,
Z\Bigg(\,\begin{tikzpicture}[baseline=-.3cm, xscale=.95, scale=.8]
	\topPairOfPants{(-1,-1)}{}
	\draw[thick, xString] (-.7,-1.1) .. controls ++(90:.8cm) and ++(270:.8cm) .. (-.1,.42);		
	\draw[thick, yString] (.7,-1.1) .. controls ++(90:.8cm) and ++(270:.8cm) .. (.1,.42);		
	\node at (-.7,-1.28) {$\scriptstyle x$};
	\node at (.7,-1.3) {$\scriptstyle y$};
\end{tikzpicture}\,\Bigg)\!:=\mu_{x,y},
\,\,\,\,
Z\Big(\,\begin{tikzpicture}[baseline=0cm, scale=.8]
	\coordinate (a1) at (0,0);
	\coordinate (b1) at (0,.4);
	\draw[thick] (a1) -- (b1);
	\draw[thick] ($ (a1) + (.6,0) $) -- ($ (b1) + (.6,0) $);
	\draw[thick] ($ (b1) + (.3,0) $) ellipse (.3 and .1);
	\draw[thick] (a1) arc (-180:0:.3cm);
\end{tikzpicture}\,\Big)\!:=i,
\,\,\,\,
Z\Bigg(\,
\begin{tikzpicture}[baseline=-.1cm, scale=.8, yscale=.8]

	\draw[thick] (-.3,-1) -- (-.3,1);
	\draw[thick] (.3,-1) -- (.3,1);
	\draw[thick] (0,1) ellipse (.3cm and .1cm);
	\halfDottedEllipse{(-.3,-1)}{.3}{.1}
	
	\draw[thick, xString] (-.1,-1.1) .. controls ++(90:.8cm) and ++(270:.8cm) .. (.1,.9);		
	\draw[thick, yString] (.1,-1.1) .. controls ++(90:.2cm) and ++(225:.2cm) .. (.3,-.2);		
	\draw[thick, yString] (-.1,.9) .. controls ++(270:.2cm) and ++(45:.2cm) .. (-.3,.2);
	\draw[thick, yString, dotted] (-.3,.2) -- (.3,-.2);	
	\node at (-.13,-1.28) {$\scriptstyle x$};
	\node at (.13,-1.3) {$\scriptstyle y$};
\end{tikzpicture}\,\Bigg)\!:=\tau_{x,y},
\]
and $P(\sigma)$ is as in \eqref{action of ribbon braid on iterated tensor product}.
We can now state the main theorem of this section:

\begin{thm}
\label{thm:TubeStringCalculusWellDefined}
Let $\Sigma$ and $\Sigma'$ be two tube string diagrams which are presented
as composites of elementary tube string diagrams and a braid (as in~\eqref{eq: tube string diagram ready to be evaluated}),
and let $Z(\Sigma)$ and $Z(\Sigma')$ be as in \eqref{eq: definition of Z(tube string diagram)}.
Then
\[
\text{$\Sigma$ and $\Sigma'$ are isotopic}\quad\Longrightarrow\quad Z(\Sigma)=Z(\Sigma').
\]
\end{thm}

\begin{proof}
By Theorem~\ref{thm:LabeledAnchoredPlanarTangles}, it is enough to argue that for every tube string diagram as in~\eqref{eq: tube string diagram ready to be evaluated},
we have 
\begin{equation}\label{sldgbnskjqngb}
Z(\Sigma)=Z(T(\Sigma))
\end{equation}
(where the two sides are defined in \eqref{eq: definition of Z(tube string diagram)} and \eqref{sfwagiletjnad}, respectively).
Indeed, assuming that fact, if $\Sigma$ and $\Sigma'$ are isotopic tube string diagrams, then 
$Z(\Sigma)=Z(T(\Sigma)) = Z(T(\Sigma')) = Z(\Sigma')$,
where the middle equality holds by Theorem~\ref{thm:LabeledAnchoredPlanarTangles}.



If $\Sigma$ is equal to one of the following elementary tube string diagrams (potentially modified by replacing strands with some number of parallel strands, and changing orientations)
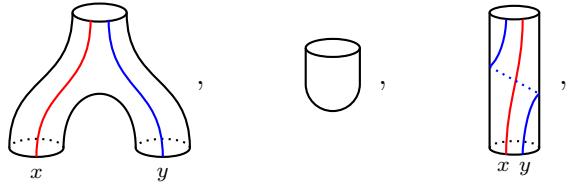
\[
\begin{tikzpicture}[baseline=-.35cm, scale=1.2]
	\topPairOfPants{(-1,-1)}{}
	\draw[thick, xString] (-.7,-1.1) .. controls ++(90:.8cm) and ++(270:.8cm) .. (-.1,.42);		
	\draw[thick, yString] (.7,-1.1) .. controls ++(90:.8cm) and ++(270:.8cm) .. (.1,.42);		
	\node at (-.7,-1.28) {$\scriptstyle x$};
	\node at (.7,-1.3) {$\scriptstyle y$};
\end{tikzpicture}\,,\;
\quad\,\,\,\,\,\,\,\;\;
\begin{tikzpicture}[baseline=0cm, scale=1.2]
	\coordinate (a1) at (0,0);
	\coordinate (b1) at (0,.4);
	\draw[thick] (a1) -- (b1);
	\draw[thick] ($ (a1) + (.6,0) $) -- ($ (b1) + (.6,0) $);
	\draw[thick] ($ (b1) + (.3,0) $) ellipse (.3 and .1);
	\draw[thick] (a1) arc (-180:0:.3cm);
\end{tikzpicture}\;\;,
\quad\,\,\,\,\,\,\,\;\;\;
\begin{tikzpicture}[baseline=-.05cm, yscale=.9, xscale=1.1]
	\draw[thick] (-.3,-1) -- (-.3,1);
	\draw[thick] (.3,-1) -- (.3,1);
	\draw[thick] (0,1) ellipse (.3cm and .1cm);
	\halfDottedEllipse{(-.3,-1)}{.3}{.1}
	
	\draw[thick, xString] (-.1,-1.1) .. controls ++(90:.8cm) and ++(270:.8cm) .. (.1,.9);		
	\draw[thick, yString] (.1,-1.1) .. controls ++(90:.2cm) and ++(225:.2cm) .. (.3,-.2);		
	\draw[thick, yString] (-.1,.9) .. controls ++(270:.2cm) and ++(45:.2cm) .. (-.3,.2);
	\draw[thick, yString, dotted] (-.3,.2) -- (.3,-.2);	
	\node at (-.13,-1.28) {$\scriptstyle x$};
	\node at (.13,-1.3) {$\scriptstyle y$};
\end{tikzpicture}\;\;,
\]
then $T(\Sigma)$ is given by:
\[\quad
\begin{tikzpicture}[baseline=-.1cm,scale=.8]
	\draw (.4,.3) -- +(0,.6)node[below, xshift=3.5, yshift=-1] {\scriptsize{$y$}} (-.4,.3) -- +(0,.6)node[below, xshift=3.5, yshift=-1] {\scriptsize{$x$}};
	\draw[thick, red] (.4,-.28) .. controls ++(-95:.3cm) and ++(70:.3cm) .. (0,-1) (-.4,-.28) .. controls ++(-85:.3cm) and ++(110:.3cm) .. (0,-1);
	\draw[very thick] (0,0) circle (1cm);
	\draw[unshaded, very thick] (.4,0) circle (.28cm);
	\draw[unshaded, very thick] (-.4,0) circle (.28cm);
	\fill[red] (.4,-.28) circle (.06) (-.4,-.28) circle (.06) (0,-1cm) circle (.06);
\end{tikzpicture}\;\,,
\qquad
\begin{tikzpicture}[baseline = -.1cm, scale=.8]
	\draw[very thick] (0,0) circle (.4cm);
	\filldraw[red] (0,-.4) circle (.05cm);
\end{tikzpicture}\;\;,
\qquad
\begin{tikzpicture}[baseline=-.1cm,scale=.8]
	\draw (0,0) -- (120:.3cm) .. controls ++(120:.3cm) and ++(-120:.3cm) .. (60:1cm);
	\draw[thick, red] (0,-.3cm) -- (0,-1cm);
	\draw (0,0) -- (60:.3cm) .. controls ++(60:.3cm) and ++(90:.5cm) .. (0:.65cm) .. controls ++(270:.8cm) and ++(270:.8cm) .. (180:.65cm) .. controls ++(90:.6cm) and ++(-60:.5cm) .. (120:1cm);
	\draw[very thick] (0,0) circle (1cm);
	\draw[unshaded, very thick] (0,0) circle (.3cm);
	\node at (80:.8cm) {\scriptsize{$x$}};
	\node at (-110:.8cm) {\scriptsize{$y$}};
	\fill[red] (0,-.3cm) circle (.06)  (0,-1cm) circle (.06);
\end{tikzpicture}\;\,.
\]
By the analog of Lemmas \ref{lem:ZPrimeOfMultiplication}, \ref{LEM 2 of the end}, and \ref{lem:ZPrimeOfRotation} for labeled oriented tangles, we then have
\begin{equation}
\label{eqn:APAofGenerators}
Z\Bigg(\;\begin{tikzpicture}[baseline=-.1cm,scale=.8]
	\draw (.4,.3) -- +(0,.6)node[below, xshift=3.5, yshift=-1] {\scriptsize{$y$}} (-.4,.3) -- +(0,.6)node[below, xshift=3.5, yshift=-1] {\scriptsize{$x$}};
	\draw[thick, red] (.4,-.28) .. controls ++(-95:.3cm) and ++(70:.3cm) .. (0,-1) (-.4,-.28) .. controls ++(-85:.3cm) and ++(110:.3cm) .. (0,-1);
	\draw[very thick] (0,0) circle (1cm);
	\draw[unshaded, very thick] (.4,0) circle (.28cm);
	\draw[unshaded, very thick] (-.4,0) circle (.28cm);
	\fill[red] (.4,-.28) circle (.06) (-.4,-.28) circle (.06) (0,-1cm) circle (.06);
\end{tikzpicture}\;\Bigg) 
= \mu_{x,y},
\qquad
Z\Big(\begin{tikzpicture}[baseline = -.1cm, scale=.8]
	\draw[very thick] (0,0) circle (.4cm);
	\filldraw[red] (0,-.4) circle (.05cm);
\end{tikzpicture}\Big)
= i,
\quad \mbox{ and } \quad
Z\Big(\begin{tikzpicture}[baseline=-.1cm,scale=.8]
	\draw (0,0) -- (120:.3cm) .. controls ++(120:.3cm) and ++(-120:.3cm) .. (60:1cm);
	\draw[thick, red] (0,-.3cm) -- (0,-1cm);
	\draw (0,0) -- (60:.3cm) .. controls ++(60:.3cm) and ++(90:.5cm) .. (0:.65cm) .. controls ++(270:.8cm) and ++(270:.8cm) .. (180:.65cm) .. controls ++(90:.6cm) and ++(-60:.5cm) .. (120:1cm);
	\draw[very thick] (0,0) circle (1cm);
	\draw[unshaded, very thick] (0,0) circle (.3cm);
	\node at (80:.8cm) {\scriptsize{$x$}};
	\node at (-110:.8cm) {\scriptsize{$y$}};
	\fill[red] (0,-.3cm) circle (.06)  (0,-1cm) circle (.06);
\end{tikzpicture}\Big)
= \tau_{x,y},
\end{equation}
and so equation \eqref{sldgbnskjqngb} holds for those tangles.
The same argument goes through when the strands labelled $x$ and $y$ are replaced by an arbitrary numbers of parallel strands.

If
\begin{equation}\label{lhgfsbmkgdf}
\Sigma \,=\, 
\begin{tikzpicture}[baseline=-.2cm, yscale=.9]
	\draw[thick] (-.5,-1) -- (-.5,1);
	\draw[thick] (.5,-1) -- (.5,1);
	\draw[thick] (0,1) ellipse (.5cm and .2cm);
	\halfDottedEllipse{(-.5,-1)}{.5}{.2}
	
	\draw[thick, wString] (-.35,-1.15) -- (-.35,.85);
	\draw[thick, xString] (0,-1.2) -- (0,0);
	\draw[thick, yString] (0,.8) -- (0,0);
	\draw[thick, zString] (.35,-1.15) -- (.35,.85);
	\nbox{unshaded}{(0,-.1)}{.3}{-.1}{-.1}{$f$}
	\node at (-.35,-1.35) {$\scriptstyle w$};
	\node at (0,-1.35) {$\scriptstyle x$};
	\node at (0,.95) {$\scriptstyle y$};
	\node at (.35,-1.35) {$\scriptstyle z$};
\end{tikzpicture}\;\,,
\qquad\text{then}\qquad
T(\Sigma) = 
\begin{tikzpicture}[baseline = -.1cm, scale=.9]
	\draw[orange] (-.6,-.2) -- (-.6,1);
	\draw[red] (.6,-.2) -- (.6,.6) arc (0:180:.2cm);
	\draw[blue] (-.15,1) -- (-.15,.6);
	\draw[DarkGreen] (1,-.2) -- (1,1);
	\draw[thick, red] (.2,-.6) -- (.2,-1);
	\roundNbox{}{(0,0)}{1}{.2}{.6}{}
	\roundNbox{unshaded}{(0,-.4)}{.2}{.6}{1}{}
\node[scale=.95] at (.04,.395) {$\scriptstyle \tilde f$};
\draw [rounded corners=5, draw, double] (.038,.4) +(-.35,-.23) rectangle +(.35,.21);
\fill (.038,.4) +(0,-.23) circle (.045);
	\node at (-.6,1.2) {\scriptsize{$w$}};
	\node at (.75,.4) {\scriptsize{$x$}};
	\node at (-.15,1.2) {\scriptsize{$y$}};
	\node at (1,1.2) {\scriptsize{$z$}};
\end{tikzpicture}\;\,,
\end{equation}
where $\tilde f:1\to y\otimes x^*$ is the mate of $f:x\to y$.
Applying the definition \eqref{sfwagiletjnad} of $Z(-)$ to the right hand side of \eqref{lhgfsbmkgdf}, we get:
\[
Z\big(T(\Sigma)\big) \,=\,
Z\Bigg(\;\begin{tikzpicture}[baseline = -.1cm, scale=.9]
	\draw[orange] (-.6,-.2) -- (-.6,1);
	\draw[red] (.6,-.2) -- (.6,.6) arc (0:180:.2cm);
	\draw[blue] (-.15,1) -- (-.15,.6);
	\draw[DarkGreen] (1,-.2) -- (1,1);
	\draw[thick, red] (.1,-.6) -- (.1,-1);
	\draw[thick, red] (.03,.2) arc (180:270:.2cm) -- (1,0) arc (90:-90:.4cm) -- (.3,-.8) arc (90:180:.2cm);
	\roundNbox{}{(0,0)}{1}{.2}{.6}{}
	\roundNbox{unshaded}{(0,-.4)}{.2}{.6}{1}{}
\draw [rounded corners=5, draw, very thick] (.038,.4) +(-.35,-.21) rectangle +(.35,.21);
	\node at (-.6,1.2) {\scriptsize{$w$}};
	\node at (.75,.4) {\scriptsize{$x$}};
	\node at (-.15,1.2) {\scriptsize{$y$}};
	\node at (1,1.2) {\scriptsize{$z$}};
\end{tikzpicture}\;\Bigg)\circ_2 \Big(\Tr_\cC(\tilde f)\circ i\;\!\Big).
\] 
By the second part of Theorem \ref{thm:LabeledAnchoredPlanarTangles}, 
this is then equal to:
\[
Z\Bigg(\;\begin{tikzpicture}[baseline = -.2cm, scale=.9]
	\draw[orange] (-.6,-.2) -- (-.6,1);
	\draw[red] (.6,-.2) -- (.6,.65) arc (0:180:.2cm) -- (.2,.6);
	\draw[blue] (-.15,1) -- (-.15,.6);
	\draw[DarkGreen] (1,-.2) -- (1,1);
	\draw[thick, red] (.1,-.6) -- (.1,-1);
	\roundNbox{}{(0,0)}{1}{.2}{.6}{}
	\roundNbox{unshaded}{(0,-.4)}{.2}{.6}{1}{}
\draw [rounded corners=5, draw, very thick] (.038,.4) +(-.35,-.21) rectangle +(.35,.21);
	\node at (-.6,1.2) {\scriptsize{$w$}};
	\node at (-.15,1.2) {\scriptsize{$y$}};
	\node at (1,1.2) {\scriptsize{$z$}};
\draw[very thick, rounded corners=5, fill=white] (-1.1,-.9) rectangle (1.5,.65);	
	\node at (.64,.42) {\scriptsize{$x$}};
	\node at (.25,.46) {\scriptsize{$x^*$}};
\end{tikzpicture}
\;\Bigg)
\circ_1
Z\Bigg(\;
\begin{tikzpicture}[baseline = -.18cm, scale=.9]
	\draw[orange] (-.6,-.2) -- (-.6,.7);
	\draw[red] (.6,-.2) -- (.6,.7) (.2,.7) -- (.2,.55);
	\draw[blue] (-.15,.7) -- (-.15,.55);
	\draw[DarkGreen] (1,-.2) -- (1,.7);
	\draw[thick, red] (.1,-.6) -- (.1,-.9);
	\draw[thick, red] (.03,.15) -- (.03,.1) arc (180:270:.15cm) -- (1,-.05) arc (90:-90:.35cm) -- (.3,-.75) arc (90:180:.2 and .15);
	\roundNbox{unshaded}{(0,-.4)}{.2}{.6}{1}{}
\draw [rounded corners=5, draw, very thick] (.038,.35) +(-.35,-.21) rectangle +(.35,.21);
	\node at (-.6,1.2-.3) {\scriptsize{$w$}};
	\node at (-.15,1.2-.32) {\scriptsize{$y$}};
	\node at (1,1.2-.3) {\scriptsize{$z$}};
\draw[very thick, rounded corners=5] (-1.1,-.9) rectangle (1.5,.7);	
	\node at (.64,1.2-.3) {\scriptsize{$x$}};
	\node at (.28,1.2-.26) {\scriptsize{$x^*$}};
\end{tikzpicture}\;\Bigg)\circ_2 \Big(\Tr_\cC(\tilde f)\circ i\;\!\Big).
\]
By the analog of Examples \ref{ex:computation alpha's} and \ref{ex:SwapAnchorDependence} for labeled oriented tangles (recall that $\alpha_i$ and $\varpi_{i,j}$ are defined in Theorem~\ref{thm: construct P from M and m}),
we then get:
\[
Z\big(T(\Sigma)\big)\,\,=\,\,
Z\Bigg(
\begin{tikzpicture}[baseline=.4cm, xscale=1.2]
	\pgfmathsetmacro{\voffset}{.08};
	\pgfmathsetmacro{\hoffset}{.15};
	\pgfmathsetmacro{\hoffsetTop}{.1};
	\coordinate (a) at (0,-2);
	\coordinate (b) at (0,-1);
	\coordinate (c) at (.7,.5);
	\coordinate (d) at (.7,1.5);
	\coordinate (e) at (.7,2.5);
	\pairOfPants{(b)}{}
	\bottomCylinder{(a)}{.3}{1}
	\emptyCylinder{(c)}{.3}{1}
	\halfDottedEllipse{(d)}{.3}{.1}
	\topCylinder{(d)}{.3}{1}
	\draw[thick] (1.4,-1) -- (1.4,-1.7) arc (-180:0:.3cm) -- (2,-1);
	\draw[thick, zString]  ($ (b) + 2*(\hoffset,0) + (0,-.1)$) .. controls ++(90:.8cm) and ++(270:.8cm) .. ($ (c) + 2*(\hoffsetTop,0) + (0,-\voffset) $);		
	\draw[thick, xString]  ($ (b) + 1*(\hoffset,0) + (0,-\voffset)$) .. controls ++(90:.8cm) and ++(270:.8cm) .. ($ (c) + 1*(\hoffsetTop,0) + (0,-\voffset) $);
	\draw[thick, wString]  ($ (b) + 3*(\hoffset,0) + (0,-\voffset)$) .. controls ++(90:.8cm) and ++(270:.8cm) .. ($ (c) + 3*(\hoffsetTop,0) + (0,-.1) $);
	\draw[thick, yString]  (1.6,-1.5) -- (1.6,-1.08) .. controls ++(90:.8cm) and ++(270:.8cm) .. ($ (c) + 4*(\hoffsetTop,0) + (0,-\voffset) $);
	\draw[thick, xString]  (1.8,-1.5) -- (1.8,-1.08) .. controls ++(90:.8cm) and ++(270:.8cm) .. ($ (c) + 5*(\hoffsetTop,0) + (0,-\voffset) $);
	\draw[thick, wString] ($ (a) + (\hoffset,0) + (0,-\voffset)$) .. controls ++(90:.4cm) and ++(270:.4cm) .. ($ (a) + 3*(\hoffset,0) + (0,-\voffset) + (0,1) $);
	\draw[thick, xString] ($ (a) + 2*(\hoffset,0) + (0,-.1) $) .. controls ++(90:.2cm) and ++(-135:.1cm) .. ($ (a) + 4*(\hoffset,0) + (0,-\voffset) + (0,.5)$);
	\draw[thick, xString] ($ (a) + 1*(\hoffset,0) + (0,1) + (0,-\voffset) $) .. controls ++(270:.2cm) and ++(45:.1cm) .. ($ (a) + (0,-\voffset) + (0,-.35) + (0,1)$);
	\draw[thick, zString] ($ (a) + 3*(\hoffset,0) + (0,-\voffset) $) .. controls ++(90:.2cm) and ++(-135:.1cm) .. ($ (a) + 4*(\hoffset,0) + (0,-\voffset) + (0,.35)$);
	\draw[thick, zString] ($ (a) + 2*(\hoffset,0) + (0,1) + (0,-.1) $) .. controls ++(270:.2cm) and ++(45:.1cm) .. ($ (a) + (0,-\voffset) + (0,-.5) + (0,1)$);
	\nbox{unshaded}{(1.7,-1.5)}{.28}{-.1}{-.1}{$\scriptstyle \tilde{f}$}
	\draw[thick, wString] ($ (c) + 3*(\hoffsetTop,0) + (0,-.1)$) .. controls ++(90:.4cm) and ++(270:.4cm) .. ($ (d) + (\hoffsetTop,0) + (0,-\voffset) $);
	\draw[thick, yString] ($ (c) + 4*(\hoffsetTop,0) + (0,-\voffset) $) .. controls ++(90:.4cm) and ++(270:.4cm) .. ($ (d) + 2*(\hoffsetTop,0) + (0,-\voffset) $);		
	\draw[thick, xString] ($ (c) + 5*(\hoffsetTop,0) + (0,-\voffset) $) .. controls ++(90:.4cm) and ++(270:.4cm) .. ($ (d) + 3*(\hoffsetTop,0) + (0,-.1) $);		
	\draw[thick, xString] ($ (c) + (\hoffsetTop,0) + (0,-\voffset) $) .. controls ++(90:.2cm) and ++(-45:.1cm) .. ($ (c) + (0,-\voffset) + (0,.35)$);
	\draw[thick, xString] ($ (c) + 4*(\hoffsetTop,0) + (0,1) + (0,-\voffset) $) .. controls ++(270:.2cm) and ++(135:.1cm) .. ($ (d) + 6*(\hoffsetTop,0) + (0,-\voffset) + (0,-.5) $);
	\draw[thick, DarkGreen] ($ (c) + 2*(\hoffsetTop,0) + (0,-\voffset) $) .. controls ++(90:.2cm) and ++(-45:.1cm) .. ($ (c) + (0,-\voffset) + (0,.5)$);
	\draw[thick, DarkGreen] ($ (c) + 5*(\hoffsetTop,0) + (0,1) + (0,-\voffset) $) .. controls ++(270:.2cm) and ++(135:.1cm) .. ($ (d) + 6*(\hoffsetTop,0) + (0,-\voffset) + (0,-.35) $);
	\draw[thick, wString] ($ (d) + 1*(\hoffsetTop,0) + (0,-\voffset)$) .. controls ++(90:.4cm) and ++(270:.4cm) .. ($ (e) + (\hoffset,0) + (0,-\voffset) $);
	\draw[thick, yString] ($ (d) + 2*(\hoffsetTop,0) + (0,-\voffset)$) .. controls ++(90:.4cm) and ++(270:.4cm) .. ($ (e) + 2*(\hoffset,0) + (0,-.1) $);
	\draw[thick, xString] ($ (d) + 3*(\hoffsetTop,0) + (0,-.1)$) -- ($ (d) + 3*(\hoffsetTop,0) + (0,.2) $) arc (180:0:.05cm) -- ($ (d) + 4*(\hoffsetTop,0) + (0,-\voffset) $);
	\draw[thick, zString] ($ (d) + 5*(\hoffsetTop,0) + (0,-\voffset)$) .. controls ++(90:.4cm) and ++(270:.4cm) .. ($ (e) + 3*(\hoffset,0) + (0,-\voffset) $);
	\node at (.1,-2.24) {$\scriptstyle w$};
	\node at (.3,-2.24) {$\scriptstyle x$};
	\node at (.5,-2.24) {$\scriptstyle z$};
	\node at (1.47,-.8) {$\scriptstyle y$};
\end{tikzpicture}\;\;\Bigg).
\]
It is now an easy exercise, using the identities in Section~\ref{sec:TubeRelations}, to check that the right hand side simplifies to $\Tr_\cC(\id_w \otimes f \otimes \id_z)$.
This finishes the proof of equation \eqref{sldgbnskjqngb} for the tube string diagram in~\eqref{lhgfsbmkgdf}.
The variants of \eqref{lhgfsbmkgdf} where the coupon is replaced by a cup or a cap can be treated by similar methods, and are left to the reader as an exercise.

So far, we have shown that equation \eqref{sldgbnskjqngb} holds true for all the elementary tube string diagrams.
We now consider the general case.
If $\Sigma = [[\Sigma_1\circ_{j_1}\Sigma_2]\circ_{j_2}\Sigma_3\ldots ]\cdot \sigma$ is 
a composite of elementary tube string diagrams $\Sigma_i$ and a braid $\sigma$,
then we have:
\[
\begin{split}
Z(T(\Sigma))
&=Z\big(T([[\Sigma_1\circ_{j_1}\Sigma_2]\circ_{j_2}\Sigma_3\ldots ]\cdot \sigma)\big)\\
&=Z\big(T([[\Sigma_1\circ_{j_1}\Sigma_2]\circ_{j_2}\Sigma_3\ldots )\cdot \sigma\big)\\
&=Z\big(T([[\Sigma_1\circ_{j_1}\Sigma_2]\circ_{j_2}\Sigma_3\ldots )\big)\circ P(\sigma)\\
&=Z\big([[T(\Sigma_1)\circ_{j_1}T(\Sigma_2)]\circ_{j_2}T(\Sigma_3)\ldots\big)\circ P(\sigma)\\
&=[[Z(T(\Sigma_1))\circ_{j_1}Z(T(\Sigma_2))]\circ_{j_2}Z(T(\Sigma_3))\ldots]\circ P(\sigma)\\
&=[[Z(\Sigma_1)\circ_{j_1}Z(\Sigma_2)]\circ_{j_2}Z(\Sigma_3)\ldots  ]\circ P(\sigma)\\
&=Z(\Sigma),
\end{split}
\]
where the second equality holds by Lemma~\ref{lem: kxjbksjsvhgvs   hmmm.... },
the third equality holds by Lemma~\ref{lem: kxjbksjsvhgvs   ouahahahaaaahahaaa!!!....!!.!!.... },
the fourth equality holds by Lemma~\ref{lem: kxjbksjsvhgvs},
the fifth equality holds by the second half of Theorem \ref{thm:LabeledAnchoredPlanarTangles},
the sixth equality holds because the $\Sigma_i$ are elementary tube string diagrams,
and the last equality is the definition~\eqref{eq: definition of Z(tube string diagram)} of $Z(\Sigma)$.
\end{proof}

\bibliographystyle{amsalpha}
{\footnotesize{
\bibliography{../../../../Documents/research/penneys/bibliography}
}}
\end{document}